\documentclass[10pt]{article}%
\usepackage{lmodern}
\usepackage{asymptote}
\usepackage{tikz}
\usepackage{pgf}
\usetikzlibrary{arrows}
\usetikzlibrary{calc}
\usetikzlibrary{decorations.pathmorphing}
\usepackage[protrusion=true,expansion=true]{microtype}
\usepackage{xfrac}

\usepackage{faktor}
\usepackage{soul}
\setstcolor{red}
\setulcolor{red}

\definecolor{light-blue}{rgb}{0.53,.8,98}
\definecolor{verde}{RGB}{50,180,50}

\usepackage{rotating} % needed for \begin{sidewaysfigure}\begin{sidewaystable}

\usepackage{mathrsfs}
\DeclareMathAlphabet{\mathpzc}{OT1}{pzc}{m}{it}

\usepackage{romannum}

\usepackage{color}
\usepackage{amsmath}
\usepackage{graphicx}
\usepackage{graphics}
\usepackage{float}
\usepackage{amsfonts}
\usepackage{amssymb}%
\usepackage{latexsym}
\usepackage{psfrag}
\usepackage{subcaption}
\usepackage{accents}

\newtheorem{theorem}{Theorem}[section]
\newtheorem{corollary}[theorem]{Corollary}
\newtheorem{definition}[theorem]{Definition}
\newenvironment{proof}[1][Proof]{\noindent \emph{#1.} }
{\hfill \ \rule{0.5em}{0.5em}}
\newtheorem{lemma}[theorem]{Lemma}
\newtheorem{proposition}[theorem]{Proposition}

\newtheorem{assumption}[theorem]{Assumption}

\numberwithin{equation}{section}
\numberwithin{table}{section}
\numberwithin{figure}{section}

\usepackage[a4paper]{geometry}
\newtheorem{remark}[theorem]{Remark}
\newtheorem{example}[theorem]{Example}

\newcommand{\R}{\mathbb{R}}

\newcommand{\Z}{\mathbb{Z}}
\newcommand{\cA}{{\cal A}}

\newcommand{\cG}{{\cal G}}
\newcommand{\cE}{{\cal E}}

\newcommand{\cH}{{\cal H}}

\newcommand{\cL}{{\cal L}}
\newcommand{\cP}{{\cal P}}

\newcommand{\cS}{{\cal S}}
\newcommand{\cT}{{\cal T}}

\newcommand{\cD}{{\cal D}}

\newcommand{\cW}{{\cal W}}
\newcommand{\cX}{{\cal X}}
\newcommand{\cY}{{\cal Y}}

\newcommand{\cZ}{{\cal Z}}
\newcommand{\C}{\mathbb{C}}

    \newcommand\quotient[2]{
        \mathchoice
            {% \displaystyle
                \text{\raise1ex\hbox{$#1$}\Big/\lower1ex\hbox{$#2$}}%
            }
            {% \textstyle
                #1\,/\,#2
            }
            {% \scriptstyle
                #1\,/\,#2
            }
            {% \scriptscriptstyle  
                #1\,/\,#2
            }
    }

\newcommand{\re}{{\rm e}}
\newcommand{\ri}{{\rm i}}

\newcommand{\beq}{\begin{equation}}
\newcommand{\eeq}{\end{equation}}
\newcommand{\beqs}{\begin{equation*}}
\newcommand{\eeqs}{\end{equation*}}
\newcommand{\bit}{\begin{itemize}}
\newcommand{\eit}{\end{itemize}}
\newcommand{\ben}{\begin{enumerate}}
\newcommand{\een}{\end{enumerate}}
\newcommand{\bal}{\begin{align}}
\newcommand{\eal}{\end{align}}
\newcommand{\bals}{\begin{align*}}
\newcommand{\eals}{\end{align*}}
\newcommand{\bse}{\begin{subequations}}
\newcommand{\ese}{\end{subequations}}
\newcommand{\bpr}{\begin{proposition}}
\newcommand{\epr}{\end{proposition}}
\newcommand{\bre}{\begin{remark}}
\newcommand{\ere}{\end{remark}}
\newcommand{\bpf}{\begin{proof}}
\newcommand{\epf}{\end{proof}}
\newcommand{\ble}{\begin{lemma}}
\newcommand{\ele}{\end{lemma}}
\newcommand{\bco}{\begin{corollary}}
\newcommand{\eco}{\end{corollary}}
\newcommand{\bex}{\begin{example}}
\newcommand{\eex}{\end{example}}
\newcommand{\bth}{\begin{theorem}}
\newcommand{\enth}{\end{theorem}}
\usepackage{enumitem}
\newcommand{\Rea}{\mathbb{R}}

\newcommand{\diam}{\mathop{{\rm diam}}}

\newcommand{\pdiff}[2]{\frac{\partial #1}{\partial #2}}

\def\XXint#1#2#3{{\setbox0=\hbox{$#1{#2#3}{\int}$}
     \vcenter{\hbox{$#2#3$}}\kern-.5\wd0}}

\usepackage{hyperref}
\counterwithin{figure}{section}
\definecolor{myblue}{rgb}{0,0,0.6}
\hypersetup{colorlinks=true,linkcolor=myblue,citecolor=myblue,filecolor=myblue,urlcolor=myblue}
\newcommand*{\N}[1]{\left\|#1\right\|}

\newcommand{\blueOld}[1]{#1}
\newcommand{\blue}[1]{#1}%{\color{blue}{#1}}}
\allowdisplaybreaks[4]
\newcommand{\tfa}{\text{ for all }}
\newcommand{\tfor}{\text{ for }}

\newcommand{\tin}{\text{ in }}
\newcommand{\ton}{\text{ on }}
\newcommand{\tas}{\text{ as }}
\newcommand{\tand}{\text{ and }}

\newcommand{\mapF}{\mathscr{F}}

\newcommand{\X}{\mathcal{X}}

\newcommand{\vertiii}[1]{{\left\vert\kern-0.25ex\left\vert\kern-0.25ex\left\vert #1
    \right\vert\kern-0.25ex\right\vert\kern-0.25ex\right\vert}}

\definecolor{jwcol}{RGB}{27, 137, 18}  %{rgb}{1,0.88,0.21} changed color for visibility (david)

\definecolor{dalcol}{rgb}{0.8,0,0}

\definecolor{escol}{rgb}{0,0,0.8}
\definecolor{estcol}{rgb}{0,0.5,0}
\definecolor{esnewcol}{rgb}{0,0.5,0}

\definecolor{macol}{HTML}{136739}

\newcommand{\dist}{{\rm dist}}
\DeclareMathOperator{\supp}{\rm supp}
\newcommand{\comp}{{\rm comp}}

\newcommand{\abs}[1]{{\left\lvert{#1}\right\rvert}}
\newcommand{\norm}[1]{{\left\lVert{#1}\right\rVert}}

\newcommand{\Op}{{\rm Op}}

\DeclareMathOperator{\Id}{I}

\newcommand{\tr}{{\rm tr}}

\newcommand{\e}{\epsilon}

\newcommand{\RPMLo}{R_{\rm PML, -}}

\newcommand{\WF}{\operatorname{WF}}

\newcommand{\Rtr}{R_{\tr} }

\newcommand{\ad}{\operatorname{ad}}

\makeatletter
\newcommand{\settheoremtag}[1]{% \settheoremtag{<tag>}
  \let\oldthetheorem\thetheorem% Store \thetheorem
  \renewcommand{\thetheorem}{#1}% Redefine it to a fixed value
  \g@addto@macro\endtheorem{% At \end{theorem}, ...
    \addtocounter{theorem}{-1}% ...restore theorem counter value and...
    \global\let\thetheorem\oldthetheorem}% ...restore \thetheorem
  }
\makeatother

\usetikzlibrary{positioning}
\usetikzlibrary{decorations.markings}
\usepackage{lscape}

\definecolor{jeffColor}{RGB}{102, 0, 204}

\newcommand{\mc}[1]{\mathcal{#1}}

\newcommand{\Zspace}[1]{\cZ_k^{#1}}
\newcommand{\Zspaced}[1]{\cZ_{k, d}^{#1}}
\newcommand{\Zspacen}[1]{\cZ_{k, n}^{#1}}
\newcommand{\ZcupHspace}[1]{\mc{Y}_k^{#1}}
\newcommand{\ZcapHspace}[1]{\mc{W}_k^{#1}}

\newcommand{\Dspace}[1]{\mc{D}_k^{#1}}
\newcommand{\Dspaced}[1]{\mc{D}_{k,d}^{#1}}
\newcommand{\Dspacen}[1]{\mc{D}_{k,n}^{#1}}
\newcommand{\Hspace}[1]{\mathcal{H}_k^{#1}}
\newcommand{\specialC}{C_\dagger}
\newcommand{\matrixB}{B}
\newcommand{\matrixZ}{Z}

\newcommand{\psis}{\psi^{\sharp}}

\newcommand{\domainnumber}{M}
\newcommand{\length}{L}
\newcommand{\Lf}{\mathcal{L}_{\rm b}}

\newcommand{\RPs}{R_{k}^*}
\newcommand{\RPd}{R_{k}^\sharp}

\newcommand{\Rrr}{r}

\let\div\relax
\DeclareMathOperator{\div}{div}

\newcommand{\cGa}{c_{\rm Ga}}
\newcommand{\CGa}{C_{\rm Ga}}
\newcommand{\x}{\times}
\usepackage{array}

\usepackage{tabu}
\newcolumntype{P}[1]{>{\centering\arraybackslash}p{#1}}
\newcolumntype{M}[1]{>{\arraybackslash}m{#1}}

\newcommand{\newell}{m}

\newcommand{\Hdiag}{\cH}
\newcommand{\Hmin}{\cH^{\rm min}}
\newcommand{\Int}{{\rm I}}
\newcommand{\Pml}{{\rm P}}
\newcommand{\piminus}{\pi_{\Int,-}}
\newcommand{\piplus}{\pi_{\Int,+}}
\newcommand{\pipml}{\pi_{\Pml}}
\newcommand{\Oursubset}[1]{\Omega_{#1}'}
\newcommand{\oldT}{W}
\newcommand{\transferIntro}{\mathscr{T}}
\newcommand{\transfer}{T^\star}

\newcommand{\cavity}{\mathcal{K}}
\newcommand{\visible}{\mathcal{V}}
\newcommand{\invisible}{\mathcal{I}}
\newcommand{\pml}{\mathcal{P}}
\newcommand{\chibigger}{\chi}
\newcommand{\Ell}{\operatorname{Ell}}

\newcommand{\Cfem}{C}

\usepackage{nicematrix}

\title{
Non-uniform finite-element meshes defined by ray dynamics for Helmholtz problems
}

\author{
Martin~Averseng\thanks{Laboratoire Angevin de Recherche Mathématique, Universit\'e d'Angers, 49045 Angers, France, \tt martin.averseng@univ-angers.fr },
\,\,
Jeffrey~Galkowski\thanks{Department of Mathematics, University College London, London, WC1H 0AY, UK,   \tt J.Galkowski@ucl.ac.uk}
\,, Euan~A.~Spence\thanks{Department of Mathematical Sciences, University of Bath, Bath, BA2 7AY, UK, \tt E.A.Spence@bath.ac.uk }
}

\date{\today}

\begin{document}
\pagenumbering{arabic}

\maketitle

\begin{abstract}
The $h$-version of the finite-element method ($h$-FEM) applied to the high-frequency Helmholtz equation has been a classic topic in numerical analysis since the 1990s. It is now rigorously understood that (using piecewise polynomials of degree $p$ on a mesh of a maximal width $h$) the conditions ``$(hk)^p \rho$ sufficiently small'' and ``$(hk)^{2p} \rho$ sufficiently small'' guarantee, respectively, $k$-uniform quasioptimality (QO) and bounded relative error (BRE), where $\rho$ is the norm of the solution operator with $\rho\sim k$ for non-trapping problems. Empirically, these conditions are observed to be optimal in the context of $h$-FEM with a uniform mesh. This paper demonstrates that QO and BRE can be achieved using certain non-uniform meshes that violate the conditions above on $h$ and involve coarser meshes away from trapping and in the perfectly matched layer (PML). The main theorem details how varying the meshwidth in one region affects errors both in that region and elsewhere. One notable consequence is that, for any scattering problem (trapping or nontrapping), in the PML one only needs $hk$ to be sufficiently small; i.e. there is no pollution in the PML.

The motivating idea for the analysis is that the Helmholtz data-to-solution map behaves differently depending on the locations of both the measurement and data, in particular, on the properties of billiards trajectories (i.e.~rays) through these sets. Because of this, it is natural that the approximation requirements for finite-element spaces in a subset should depend on the properties of billiard rays through that set. Inserting this behaviour into the latest duality arguments for the FEM applied to the high-frequency Helmholtz equation allows us to retain detailed information about the influence of \emph{both} the mesh structure \emph{and} the behaviour of the true solution on local errors in FEM.
%\begin{AMS}
%\noi\textbf{AMS subject classifications:} 35J05, 65N15, 65N30, 78A45
%\end{AMS}
\end{abstract}

%\section*{To do}
%
%\ben
%\item Euan -- small edits in propagation
%\een

\setcounter{tocdepth}{1}
\tableofcontents

\section{Introduction}\label{sec:intro}

\subsection{The main result in its simplest form}\label{sec:1.1}

\paragraph{The scattering problem and its finite-element approximation using a PML.}
We study computing approximations to the solution of sound-soft or sound-hard scattering problems using the finite-element method with non-uniform meshes. We consider scattering by a \blueOld{bounded} open obstacle $\Omega_-\blueOld{\subset} \mathbb{R}^d$ with smooth boundary and connected complement, $\Omega_+:=\mathbb{R}^d\setminus \overline{\Omega}_-$: given $f\in L^2_{\rm comp}(\overline{\Omega_+})$, find $u\in H^1_{\rm loc}(\overline{\Omega_+})$ such that
\begin{equation}
\label{e:edp}
-k^{-2}\operatorname{div}( A\nabla u)-n u=f\text{ in }\Omega_+,\qquad (Bu)|_{\partial\Omega_+}=0,\qquad (k^{-1}\partial_r-i)u=o_{r\to \infty}(r^{\frac{1-d}{2}}),
\end{equation}
where $A$ is a smooth, symmetric, positive-definite matrix with real coefficients \blueOld{that equals the identity outside a compact set, 
$n$ is a smooth, strictly positive, real-valued function that equals one outside a compact set}, 
%$n\in C^\infty(\overline{\Omega_+};\mathbb{R}_+)$, 
%$\supp (A-I)\cup \supp (n-1)\Subset \overline{\Omega}_+$, 
and $Bu=u$ in the sound-soft case and $Bu=\partial_{\nu}u$, with $\nu$ the normal to $\partial\Omega_+$ in the sound-hard case. \blueOld{Under these assumptions, the solution to \eqref{e:edp} exists and is unique; see \cite[\S4.2]{GaSp:25b} and the references therein.}

We approximate the Sommerfeld radiation condition using a radial perfectly matched layer (PML):~let $\Omega_{\tr}\blueOld{\subset} \mathbb{R}^d$ be open, \blueOld{bounded, connected}, and contain  
\blueOld{a closed ball containing}
%closed convex hull of 
$\Omega_-\cup \supp (A-I)\cup \supp (n-1)$. We truncate the problem~\eqref{e:edp} to the computational domain $\Omega:=\Omega_+\cap \Omega_{\tr}$ and apply the finite-element method to the problem:~given $f\in L^2(\Omega)$, find $u\in H^1(\Omega)$ such that
\begin{equation}
\label{e:PML1}
P_ku:=-k^{-2}\operatorname{div}(A_\theta\nabla u)+k^{-2} b_\theta\cdot\nabla u-n_\theta u=f\text{ in }\Omega,\qquad (Bu)|_{\partial\Omega_+}=0,\qquad u|_{\partial\Omega_{\tr}}=0,
\end{equation} 
where $A_\theta$, $b_\theta$, and $n_\theta$ are defined in \S\ref{sec:PML} (and $A_\theta$, $b_\theta$, and $n_\theta$ are respectively $A$, $0$, and $n$ in the non-PML region). 
\blueOld{
For all but finitely-many values of $k$
the solution to \eqref{e:PML1} exists and is unique.\footnote{\blueOld{To see this, observe that $P_k$ is an analytic family of Fredholm operators  on $\mathbb{C}\setminus 0$; since $P_k$ satisfies a G\aa rding inequality, it is invertible for some $k_0$. Therefore, the Analytic Fredholm Theorem (see e.g.~\cite[Theorem C.8]{DyZw:19}) implies that $P_k$ is invertible at all but a discrete set of $k$'s. Under appropriate assumptions on the scaling function,~\cite[Theorem 1.2]{GLS2} then shows that $P_k$ is invertible for $k\gg 1$ and hence the set of real $k$'s such that $P_k$ is not invertible is finite.}}}
\blueOld{When the data $f$ is away from the PML region, the difference between the solutions of \eqref{e:edp} and \eqref{e:PML1} (measured away from the PML region) is exponentially small in $k$ \cite{GLS2}.}
 Let $a_k(\cdot,\cdot)$ be the sesquilinear form associated with~\eqref{e:PML1}.
\begin{definition}
Given  a subspace $\blueOld{V\subset} H_0^1(\Omega)$ in the sound-soft case or $V\subset H^1(\Omega)\cap H_0^1(\Omega_{\blueOld{\tr}})$ \blueOld{in the sound-hard case}, a \emph{finite-element/Galerkin solution of~\eqref{e:PML1}} is an element $u_h\in V$ such that 
\begin{equation}
\label{e:galerkinDef}
a_k( u_h,w_h) =\langle f,w_h\rangle\quad \tfa w_h\in V.
\end{equation}
\end{definition}

Let 
$$
\rho=\rho(k):=\sup\Big\{ \|u\|_{L^2(\Omega)}\,:\, u\text{ solves~\eqref{e:PML1} with }\, \|f\|_{L^2(\Omega)}=1\Big\}.
$$
Recall that, with the normalisation used in~\eqref{e:edp}, for all $k_0>0$ there exists $c>0$ such that  $\rho(k)\geq ck$ for $k>k_0$. By~\cite[Theorem 1.6]{GLS2}, for a radial PML (defined in \S\ref{sec:PML}), there exist $C, k_1>0$ such that for $k>k_1$ and $\chi\equiv 1$ on the convex hull of $\Omega$, 
$$
\rho\leq C\sup\Big\{ \|\chi u\|_{L^2(\Omega^+)}\,:\, u\text{ solves~\eqref{e:edp} with }\, \|\chi f\|_{L^2(\Omega^+)}=1\Big\};
$$
i.e. the PML solution operator is controlled by the scattering solution operator. 

\paragraph{State-of-the-art analysis of the $h$-FEM.}
The $h$-version of the finite-element method (FEM) considers the Galerkin solution to~\eqref{e:PML1} with $V$ given by the space of piecewise polynomials of a fixed degree, $p$, on a mesh with maximum width $h$.  The accuracy of the solution is then increased by decreasing $h$. 

Many authors have studied  $k$-explicit conditions on the meshwidth guaranteeing that the finite-element solution exists and has controlled error. The best existing result is the following: if $(hk)^{2p}\rho$ is sufficiently small, then for $m\in\{0,\dots, p\}$, 
\begin{equation}
\label{e:preasymptotic1}
\|u-u_h\|_{H_k^{1-m}(\Omega)}\leq C\Big((hk)^m+\rho (hk)^p\Big)
\inf_{w_h\in V_{\mc{T}_k}^p}\|u-w_h\|_{H_k^1(\Omega)}.
\end{equation}
This estimate was proved for general Helmholtz problems  and general $p\in \mathbb{Z}^+$ in \cite{GS3} (with earlier work in \cite{FeWu:09, MeSa:10, FeWu:11, MeSa:11,Wu:14,DuWu:15,BaChGo:17,LiWu:19, ChNi:20, Pe:20,ChGaNiTo:22, LSW2}) and is empirically sharp when the mesh considered has uniform width $h$.
The bound~\eqref{e:preasymptotic1} implies that if $\rho (hk)^p$ is bounded then the FE solution is quasi-optimal (QO) in the sense that
$\|u-u_h\|_{H^1_k(\Omega)}$ is, up to a constant, the best-approximation error. Since $\rho\gtrsim k$, the requirement $\rho (hk)^p\lesssim 1$ implies that $hk\lesssim \rho^{-1/p}\ll k^{-1/p}$-- this fact that $hk$ must decrease with $k$ is the \emph{pollution effect}~\cite{BaSa:00}.

% \cite[Theorem 6.1]{Wu:14},  \cite[Theorem 5.1]{DuWu:15}, \cite[Theorem 2]{BaChGo:17}, \cite[Theorem 4.1]{LSW2}, \cite[Theorem 4.4]{LiWu:19}, 
% \cite[Theorem 2.39]{Pe:20}, and 
% \cite[Theorem 7.2]{ChGaNiTo:22}

Using standard piecewise-polynomial approximation results in the right-hand side of \eqref{e:preasymptotic1}, one obtains 
%from~\eqref{e:preasymptotic1} 
that 
\begin{equation}
\label{e:preasymptotic2}
\|u-u_h\|_{H_k^{1-m}(\Omega)}\leq C\Big((hk)^m+\rho (hk)^p\Big)(hk)^p\|u\|_{H_k^{p+1}(\Omega)}.
\end{equation}
If the data is $k$-oscillatory, then so is the solution (by elliptic regularity; see \cite[Page 9]{GS3}), with $\|u\|_{H_k^{p+1}(\Omega)}\leq C\|u\|_{H^1_k(\Omega)}$. In this case,~\eqref{e:preasymptotic2} implies that the Galerkin solution has bounded relative error (BRE) if $(hk)^{2p} \rho$ is sufficiently small. We highlight that 
this threshold for BRE was famously identified for 1-d problems in the work of Ihlenburg and Babu\v{s}ka \cite{IhBa:95a, IhBa:97} (see \cite[Page 350, penultimate displayed equation]{IhBa:97}, \cite[Equation 4.7.41]{Ih:98}).  

To date, all $k$-explicit a priori analyses of the $h$-FEM consider uniform meshes. 
%\blueOld{and }
The goal of this paper is to study non-uniform meshes, designed by considering the ray dynamics in $\Omega_+$, and give local -- as opposed to global -- criteria on the meshwidths. In particular, we show that there exist meshes that obtain QO/BRE while severely violating the mesh thresholds above, and thus involve many fewer degrees of freedom (see Table~\ref{tab:regimes} below).

\paragraph{Subsets of $\Omega$ defined by ray dynamics.}
We define \blueOld{(generalised)} billiard trajectories to be geodesics for the metric $g^{-1}=A/n$ in $\Omega_+$ continued by reflection with respect to $g$ at \blueOld{transversal} intersections with the boundary 
of $\Omega_+$ -- when $A=I$ and $n=1$, these are straight line paths continued using the Snell--Descartes law at the boundary. 
\blueOld{When these geodesics hit $\partial\Omega$ tangentially they are continued in a more subtle manner (see Definition~\ref{d:PGBB} below). Roughly speaking, when the boundary is geodesically strictly concave, the trajectories continue as geodesics in $\Omega_+$ (as if the boundary was not encountered at all), and when it is geodesically convex they are continued as geodesics in the induced metric on the boundary until they encounter a point of geodesic concavity at which point they detach from the boundary. 
}
%(In fact, we use a somewhat more complicated notion, the \emph{generalised broken bicharacteristic} -- see Definition~\ref{d:PGBB} below.) 

Next, we define the \emph{cavity} $\cavity\subset \overline{\Omega}_+$ as the set of points $x\in \overline{\Omega}_+$ such that there is a  billiard trajectory passing over $x$ that remains in a compact set for all positive and negative times. \blueOld{Observe that $\cavity\subset \text{convex hull}(\Omega_-\cup \supp (A-I)\cup \supp (n-1))$. In fact, $\mathcal{K}$ is usually much smaller than this; for more discussion on the size of $\cavity$, see \S\ref{s:special}.}

We define the \emph{visible set} $\visible\subset \overline{\Omega}_+$ as those points $x\in \overline{\Omega}_+$ such that there is a  billiard trajectory passing over $x$ that remains in a compact set for all positive times or all negative times \blueOld{(so that $\cavity \subset \visible$)}. Finally, we define the \emph{invisible set} $\invisible:=\overline{\Omega}_+\setminus \blueOld{\visible}$ (the adjectives visible and invisible are relative to the cavity). Let $\Omega_{\pml}\subset \Omega$ be an open neighbourhood of $\partial\Omega_{\tr}$ that is strictly contained in the PML. Next, let $\Omega_\cavity$, $\Omega_\visible$ and $\Omega_\invisible$ be open neighbourhoods of the intersections with $\overline{\Omega}$ of, respectively, $\cavity$, $\visible\setminus (\cavity\cup\Omega_\pml)$, and $\invisible\setminus \Omega_\pml$ in the subspace topology of $\overline{\Omega}$ such that $\Omega_{\cavity}\cap \partial\Omega_{\tr}=\Omega_{\visible}\cap\partial \Omega_{\tr}=\Omega_{\invisible}\cap\partial\Omega_{\tr}=\emptyset.$
\blueOld{An example of the domains $\Omega_\cavity,\Omega_\visible,\Omega_\invisible,$ and $\Omega_\pml$ is given in Figure \ref{f:twoMirrors}.

It is typically difficult to determine $\cavity$ and $\visible$ exactly because they involve large-time dynamics of the (generalised) billiard trajectories; e.g., even in the example of Figure \ref{f:twoMirrors}, while $\cavity$ can be easily determined exactly, it is more difficult to do the same for $\visible$. 
To determine $\cavity$ and $\visible$ in practical situations requires ray tracing; this is discussed more in Remark \ref{r:findKV}. Note, however, that the question of choosing the regions $\Omega_\cavity$ and $\Omega_\visible$ (which only need to include $\cavity$ and $\visible$, respectively) is slightly simpler, and is discussed more at the end of this subsection.
}

\begin{figure}
\begin{center}
\begin{tikzpicture}
\begin{scope}[scale=.9]
\fill[fill=light-blue,opacity=.6] (-2.5,-1.1)rectangle(2.5,1.1);
\draw[fill=orange ](0,0)circle(3);
\begin{scope}
\clip (-2.5,-1.1)rectangle(2.5,1.1);
\fill[fill=light-blue ](0,0)circle(2.25);
\end{scope}
\node at(0,2.55){$\Omega_{\pml}$};
\fill[fill=light-blue] (0,0)circle(2.25);
\fill[fill=verde,domain=-157:-23,opacity=.6] plot({2.35*cos(\x)},{2.35*sin(\x)})--cycle;
\fill[fill=verde,domain=157:23,opacity=.6] plot({2.35*cos(\x)},{2.35*sin(\x)})--cycle;
\begin{scope}
\clip (0,0) circle (2.35);
\fill[fill=light-blue,opacity=.6] (-2.5,-1.1)rectangle(2.5,1.1);
\end{scope}
\fill[fill=pink,opacity=.6] (-1,-1.1)rectangle (1,1.1);
  % Draw the first rectangle with rounded corners
    \draw[rounded corners=0.2cm, fill=gray] (-1.6,-1) rectangle (-.5,1);

    % Draw the second rectangle with rounded corners
    \draw[rounded corners=0.2cm, fill=gray](1.6,-1) rectangle (.5,1);
    \node at(0,0){$\Omega_\cavity$};
\node at(0,-1.5){$\Omega_\visible$};
\node at(0,1.5){$\Omega_\visible$};
\node at(2,0){$\Omega_\invisible$};
\node at(-2,0){$\Omega_\invisible$};
    \end{scope}
\end{tikzpicture}
\end{center}
\caption{\blueOld{An example of overlapping} domains $\Omega_\cavity,\Omega_\visible,\Omega_\invisible,$ and $\Omega_\pml$, when $\Omega_-$ consists of two (rounded) aligned rectangles, \blueOld{$A=I$, and $n=1$.} \blueOld{In this case, one can see that any trapped billiard trajectory is a horizontal trajectory bouncing between the two rectangles and hence $\cavity$ is the union of such trajectories. Furthermore, any point in $\visible$ must have a trajectory passing through it that reaches the convex hull of $\Omega_-$ intersected with $\Omega_+$ and hence is contained in $\Omega_{\visible}$ as shown.}}
\label{f:twoMirrors}
\end{figure}

\paragraph{The finite-element space.}
Given \blueOld{$p\geq 1$, and  a $C^{p+1}$ simplicial, affine-conforming mesh,} $\mathcal{T}$ of $\Omega$, \blueOld{(see Definitions~\ref{d:conformingMesh} to \ref{d:AffineConforming})} we define $h_\cavity , h_\visible,h_\invisible,h_\pml>0$, 
 to be upper bounds for the diameter of any mesh element that intersects $\Omega_\cavity$, $\Omega_\visible$, $\Omega_\invisible$, and $\Omega_\pml$ respectively, and let $h:= \max\{ h_\cavity , h_\visible,h_\invisible,h_\pml\}$. 
 Since $\Omega$ is $C^\infty$, some elements of the mesh need to be curved; however, our results can, in principle, be combined with those of \cite{ChSp:24} to prove results about simplicial meshes. 
% Let $\gamma(\mathcal{T})$ denote the shape-regularity constant of the mesh $\mathcal{T}$
% \blueOld{
% \begin{gather*} 
% \gamma(\mathcal{T}):=\inf\Big\{ \frac{h_\blueOld{T}}{r_\blueOld{T}}\,:\, \blueOld{T}\in\mathcal{T}\Big\},\qquad h_\blueOld{T}:=\diam(\blueOld{T}),\\
% r_\blueOld{T}:=\sup\{ r>0\,:\, \exists x\in \blueOld{T},\, \blueOld{T}\text{ is star-shaped with respect to the ball } B(x,r)\}. 
% \end{gather*}
% }(see also \cite[Equation (4.4.16)]{BrSc:08}). 
We define the following measure of local uniformity of the mesh at scale $\e>0$:
$$
U(\mathcal{T},\e):= \sup_{x\in \Omega}\underset{\substack{T_1,T_2\in\mathcal{T}\\T_1\cap B(x,\e)\neq \emptyset\\T_2\cap B(x,\e)=\emptyset}}{\sup}\frac{\diam(T_1)}{\diam(T_2)}.
$$
%We say that a family of meshes $(\cT_k)_{k >0}$ is \emph{wavelength-scale quasiuniform with constant $\gamma_0>0$} if the mesh is shape regular with $\gamma(\cT_k) \geq \gamma_0$ and $U(\cT,(1+k)^{-1}) \leq \gamma_0^{-1}$ for all $k >0$.

\blueOld{For a $C^{p+1}$ simplicial, affine-conforming mesh $\mc{T}$ and $p\in\{ 1,2,\dots\}$, we denote by $V_{\mathcal{T}}^p\subset H^1_0(\Omega)$ (or $H^1_0(\Omega_{\tr})\cap H^1(\Omega)$ in the sound-hard case) the space of mapped polynomials of degree $p$ on the mesh $\mathcal{T}$ (defined in~\eqref{e:polySpace} below). }

\paragraph{\blueOld{Growth of the solution operator.}}

\begin{assumption}
\label{a:polyBoundIntro}
The set $\blueOld{\mathrm{J}}\subset\blueOld{(0,\infty)}$, $\Omega_-$, $A_\theta, b_\theta$, and $n_\theta$ are such that \blueOld{for all $k_0>0$
$$
\sup_{k\in [k_0,\infty)\setminus \blueOld{\mathrm{J}}}\frac{\log \rho(k)}{\log k}<\infty.
$$
i.e., $\rho(k)$ is polynomially bounded outside $\blueOld{\mathrm{J}}$.
} %there are $C>0$, $N>0$ such that $\rho(k)\leq C k^N$ for $k\in \blueOld{[k_0,\infty)}\setminus \blueOld{\mathrm{J}}$. 
\end{assumption}

\blueOld{
%The norm of the solution operator, $\rho$, can grow at most exponentially in $k$ \cite{Bu:98}, and 
For all but the strongest forms of trapping, $\rho$ is polynomially bounded \cite{Ik:83, NoZw:09, WuZw:11, ChWu:13, ChSpGiSm:20}, and so $\blueOld{\mathrm{J}}$ in Assumption \ref{a:polyBoundIntro} can be taken to be $\emptyset$. 
When the strongest form of trapping occurs, $\rho$ grows exponentially in $k$ \cite{Ra:71, BeChGrLaLi:11}, and this is the fastest-possible growth for smooth obstacles by \cite{Bu:98}. 
However, if a set of frequencies of arbitrarily-small measure is excluded, then the growth is polynomial; more specifically,}
by~\cite{LSW1} and~\cite[Theorem 1.6]{GLS2}, for any $\delta>0$, Assumption~\ref{a:polyBoundIntro} holds for a radial PML, any $(\Omega_-,A,n)$, and some $\blueOld{\mathrm{J}}_\delta$ with $|\blueOld{\mathrm{J}}_\delta|<\delta$, \blueOld{where $|\cdot|$ denotes the Lebesgue measure.}
 %Under weaker forms of trapping, $\blueOld{\mathrm{J}}=\emptyset$.
\blueOld{We require such a polynomial-growth estimate to control super-algebraically small errors in the proof of the main result.} %coming from propagation estimates and pseudolocality results. % (see \S\ref{s:sketch} below).}

\paragraph{The main result in its simplest form.}

Define
%\mathcal{C}:=\begin{pmatrix} \rho&\sqrt{k \rho}&0&0\\
%\sqrt{k\rho}&k&k&1\\
%0&k&k&1\\
%0&1&1&1\end{pmatrix},\quad \Hdiag:=\begin{pmatrix} h_\cavity  &0&0&0\\
%0&h_\visible &0&0\\
%0&0&h_\invisible&0\\
%0&0&0&h_P\end{pmatrix},
\begin{equation}
\begin{gathered}\label{e:matrices}
\mathcal{C}:=\begin{pmatrix} \rho&\sqrt{k \rho}&0&0\\
\sqrt{k\rho}&k&k&0\\
0&k&k&0\\
0&0&0&1
\end{pmatrix},
\quad 
%\Hdiag:=\begin{pmatrix} h_\cavity  &0&0&0\\
%0&h_\visible &0&0\\
%0&0&h_\invisible&0\\
%0&0&0&h_P\end{pmatrix},
\cH:=\begin{pmatrix} h_\cavity  &0&0&0\\
0&h_\visible &0&0\\
0&0&h_\invisible&0\\
0&0&0&h_\pml\end{pmatrix},
%\quad
%\cH_\Int:=\begin{pmatrix} h_\cavity  &0&0\\
%0&h_\visible &0\\
%0&0&h_\invisible\\
%\end{pmatrix},
\quad\mathscr{F}:=\begin{pmatrix}
1 &1&1& 1 \\
1 &1&1& 1\\
1 &1&1& 1\\
1 &1&1& 1
\end{pmatrix},
\\
\transferIntro:=\begin{pmatrix} 1&(h_\visible k)^{2p}\sqrt{k\rho}&(h_\visible k)^{2p}\sqrt{k \rho}( h_\invisible k)^{2p}k&
0\\
(h_\cavity  k)^{2p}\sqrt{k\rho}&1&(h_\invisible k)^{2p}k&0\\
(h_\cavity  k)^{2p}\sqrt{k\rho}(h_\visible k)^{2p}k&(h_\visible k)^{2p}k&1&0\\
0
&0&0&1\end{pmatrix}.
\end{gathered}
\end{equation}
Conceptually, $\mathcal{C}$ is the \blueOld{matrix consisting of norms} of the localised data-to-solution map (with $\mathcal{C}$ standing for ``communication"); \blueOld{i.e., the $ij$th entry of $\mathcal{C}$ is a bound on the norm of the solution measured in $\Omega_i$ with data in $\Omega_j$}.
\blueOld{Furthermore},
$\mathscr{T}$ \blueOld{captures} the propagation of Galerkin errors between subdomains according to the graph in 
Figure \ref{f:graph2} (with a simplified version -- Figure \ref{f:graph1} -- given in the sketch of the proof in \S\ref{s:sketch}). 

We work in $k$-weighted Sobolev spaces defined for $U \subset \R^d$ by
%For all $k > 0$ and $n\geq 0$, and given $U \subset \R^d$, let $H^n_k(U)$ (abbreviated $H^n_k$ when $U = \Omega$) be the completion of $C^\infty(U)$ with respect to the ``semiclassical norm'' (where $k^{-1}$ is the semiclassical parameter)
\beq\label{e:weightedNorm}
\|u\|^2_{H^n_k(U)} := \sum_{|\alpha|\leq n} k^{-2\abs{\alpha}} \|\partial^\alpha u\|^2_{L^2(U)}, \quad n\in \mathbb{N} \blueOld{:=\{0,1,2,\ldots\}},
\eeq
and let $H_k^{-n}(U)$ be the normed dual of $H_k^n(U)$. % (we emphasise that $H^{-n}$ usually denotes the dual of $H_0^n = \overline{C^\infty_c(\Omega)}^{H^n_k}$).

%We work in $k$-weighted Sobolev spaces define for a smooth open set $\Omega$ by
%$$
%\|u\|_{H_k^\ell(\Omega)}^2:=\big\langle (-k^{-2}\Delta+1)^\ell u,u\big\rangle_{L^2(\Omega)},\qquad\ell\in\mathbb{R}
%$$
%where $\Delta$ denotes the Dirichlet Laplacian on $\Omega$. 

The following is a particular case of our main result (Theorem \ref{t:theRealDeal} below).
\begin{theorem}
\label{t:simple}
Let $k_0,N, \blueOld{\Upsilon}>0$, $p\in \mathbb{N}\setminus\{0\}$, 
$\blueOld{\mathrm{J}}\subset \blueOld{(0,\infty)}$ be such that Assumption~\ref{a:polyBoundIntro} holds,
and let $\Oursubset{\star}$ be compactly contained in  $\Omega_{\star}$ with respect to the subspace topology of $\overline{\Omega}$, $\star\in\{ \cavity,\visible,\invisible,\pml\}$. 
% Let $k_0,N, \blueOld{\Upsilon}>0$, $p\in \mathbb{N}\setminus\{0\}$, $\blueOld{\mathrm{J}}\subset \blueOld{(0,\infty)}$ such that Assumption~\ref{a:polyBoundIntro} holds, 
% and
% let $\Oursubset{\star}$ be compactly contained in  $\Omega_{\star}$ with respect to the subspace topology of $\overline{\Omega}$, $\star\in\{ \cavity,\visible,\invisible,\pml\}$. 

There exist $c,C>0$ such that for all $k\in (k_0,\infty)\setminus \blueOld{\mathrm{J}}$, \blueOld{all $C^{p+1}$ simplicial, affine conforming meshes $\mathcal{T}$ with constant $\Upsilon$ satisfying $U(\mathcal{T},k^{-1})\leq \Upsilon$, and}
%all meshes $\mathcal{T}$ 
%for all families of meshes $(\cT_k)_{k>0}$ that are 
%wavelength-scale quasiuniform with constant $\gamma_0$ and satisfy 
%satisfying  $\gamma(\mathcal{T})\geq \gamma_0$,  $U(\mathcal{T},k^{-1})\leq \gamma_0^{-1}$, and
\begin{equation}
\label{e:meshConditions}
(h_\cavity k)^{2p}\rho(k)+(h_\visible k)^{2p}k+(h_\invisible k)^{2p}k+(h_\pml k)^{2p}\leq c,
\end{equation}
%all $k\in  (k_0,\infty)\setminus \blueOld{\mathrm{J}}$, 
and all $w_{h,\star} \in V_{\mathcal{T}}^p$, with $\star\in\{ \cavity,\visible,\invisible,\pml\}$, 
the Galerkin solution, $u_h\in V_{\mathcal{T}}^p$, to~\eqref{e:PML1} exists, is unique, and satisfies, for $m\in\{0,1,\ldots,p\}$,
 \renewcommand{\jot}{1pt}
\begin{align}\label{e:simple}
\left(
\begin{aligned}
&\|u-u_h\|_{H_k^{1-\newell}(\Omega_\cavity')}\\
&\|u-u_h\|_{H_k^{1-\newell}(\Omega_\visible')}\\
&\|u-u_h\|_{H_k^{1-\newell}(\Omega_\invisible')}\\
&\|u-u_h\|_{H_k^{1-\newell}(\Omega_\pml')}
\end{aligned}
\right)
&
\leq C
\Big[(\cH k)^{\newell}+
\transferIntro\mathcal{C} (\cH k)^p
+ k^{-N}(hk)^m \mathscr{F}
\Big]
\left(\begin{aligned}
&\|u-w_{h,\cavity}\|_{H_k^{1}(\Omega_\cavity)}\\
&\|u-w_{h,\visible}\|_{H_k^{1}(\Omega_\visible)}\\
&\|u-w_{h,\invisible}\|_{H_k^{1}(\Omega_\invisible)}\\
&\|u-w_{h,\pml}\|_{H_k^{1}(\Omega_\pml)}
\end{aligned}
\right),
\end{align}
\renewcommand{\jot}{3pt}
where the inequality in \eqref{e:simple} is understood component-wise.
\end{theorem}
\bre 
\blueOld{
\S\ref{s:sketch} gives an interpretation of the matrices appearing in the right-hand side of \eqref{e:simple} in terms of the propagation of errors.}
% From the estimate~\eqref{e:simple} and the interpretation of $\mathscr{T}$ as the propagation of Galerkin errors, the matrix $\mathcal{C}(\cH k)^p$ should be viewed as mapping best approximation errors to Galerkin errors. This appears more concretely in the proof of Theorem~\ref{t:simple} and we discuss this interpretation in \S\ref{s:interpret}. 
\ere

To compare with the estimate~\eqref{e:preasymptotic2} on relative error, we state the following corollary of Theorem~\ref{t:simple} which follows from standard piecewise-polynomial approximation estimates.
\begin{corollary}
\label{c:relError}
Let $k_0,N, \blueOld{\Upsilon}>0$, $p\in \mathbb{N}\setminus\{0\}$, $\blueOld{\mathrm{J}}\subset \blueOld{(0,\infty)}$ such that Assumption~\ref{a:polyBoundIntro} holds, 
and
let $\Oursubset{\star}$ be compactly contained in  $\Omega_{\star}$ with respect to the subspace topology of $\overline{\Omega}$, $\star\in\{ \cavity,\visible,\invisible,\pml\}$. 

There exist $c,C>0$ such that for all $k\in (k_0,\infty)\setminus \blueOld{\mathrm{J}}$, \blueOld{all $C^{p+1}$ simplicial, affine conforming meshes $\mathcal{T}$ with constant $\Upsilon$ satisfying $U(\mathcal{T},k^{-1})\leq \Upsilon$, and}
satisfy~\eqref{e:meshConditions},
%satisfying  $\gamma(\mathcal{T})\geq \gamma_0$,  $U(\mathcal{T},k^{-1})\leq \gamma_0^{-1}$, and 
the Galerkin solution, $u_h\in V_{\mathcal{T}}^p$, to~\eqref{e:PML1} exists, is unique, and satisfies, for $0\leq m\leq p$,%for $0\leq \newell \leq p$,
 \renewcommand{\jot}{1pt}
\begin{align}\label{e:relSimple}
\left(\begin{aligned}
%\begin{pmatrix} 
&\|u-u_h\|_{H_k^{1-\newell}(\Omega_\cavity')}\\
&\|u-u_h\|_{H_k^{1-\newell}(\Omega_\visible')}\\
&\|u-u_h\|_{H_k^{1-\newell}(\Omega_\invisible')}\\
&\|u-u_h\|_{H_k^{1-\newell}(\Omega_\pml')}
%\end{pmatrix}
\end{aligned}
\right)
&
\leq C
\Big[(\cH k)^{\newell}+
\transferIntro \mathcal{C}(\cH k)^p
+ k^{-N}(hk)^m \mathscr{F}
%+\mathscr{R}_{N,m}
\Big](\cH k)^{p}
\left(\begin{aligned}
%\begin{pmatrix} 
&\|u\|_{H_k^{p+1}(\Omega_\cavity)}\\
&\|u\|_{H_k^{p+1}(\Omega_\visible)}\\
&\|u\|_{H_k^{p+1}(\Omega_\invisible)}\\
&\|u\|_{H_k^{p+1}(\Omega_\pml)}
\end{aligned}
\right).
\end{align}
 \renewcommand{\jot}{3pt}
\end{corollary}
As with \eqref{e:preasymptotic2}, when the data, $f$, is $k$-oscillatory, so is the solution $u$, and in this case, $\|u\|_{H_k^{p+1}(U')}\leq C\|u\|_{H_k^1(U)}$ for $U'\Subset U$. 
%Hence, one can use~\eqref{e:relSimple} to find meshes with $(hk)^{2 p}\rho \gg 1$ that nevertheless have guaranteed control on the relative error (see Corollary~\ref{cor:RE}).

\paragraph{\blueOld{Discussion of the novelty of Theorem \ref{t:simple}}.}

To the best of the authors' knowledge, Theorem~\ref{t:simple} and its more sophisticated analogue Theorem~\ref{t:theRealDeal} are the first results concerning $k$-dependent, non-uniform finite-element meshes in the context of the Helmholtz equation. 
%\esdel{Theorem~t:simple should be compared with~e:preasymptotic1. Indeed,} 
For a uniform mesh ($\,h_{\cavity}=h_{\visible}=h_{\invisible}=h_{\pml}=h\,$)~\eqref{e:simple} implies the strongest previously-known bound~\eqref{e:preasymptotic1}. Indeed, for a uniform mesh with $(hk)^{2p}\rho$ sufficiently small, all the elements of the matrix $\transferIntro$ are bounded by a constant, and all the elements of $\mathscr{C}(\cH k)^p$ are bounded by $\rho(hk)^p $. 
However, Theorem~\ref{t:simple} provides much more information than~\eqref{e:preasymptotic1}:~it describes how the best approximation errors and local meshwidth in each region affect the Galerkin error in all other regions. 
%We describe 
Section~\ref{s:special} highlights some notable
%a few 
consequences of this description, with Section~\ref{s:numerical} illustrating these numerically.

Theorem~\ref{t:simple} is most interesting when $\rho(k)\gg k$, which is equivalent to the problem being trapping, i.e., $\cavity\neq \emptyset$ (see \cite{BoBuRa:10}, \cite[Theorem 7.1]{DyZw:19}). In particular, Theorem~\ref{t:simple} shows that in the trapping case there exist meshes with $(hk)^{p}\rho\gg 1$ whose finite-element solutions have guaranteed $k$-uniform quasioptimality (see Corollary~\ref{cor:QO}). 
Similarly, one can use~\eqref{e:relSimple} to find meshes with $(hk)^{2 p}\rho \gg 1$ that nevertheless have guaranteed control on the relative error (see Corollary~\ref{cor:RE}).
Even when $\cavity=\emptyset$, Theorem \ref{t:simple} gives new information including that one needs only a fixed number of points per wavelength in the PML.

% \blueOld{As described in above, all existing $k$-explicit analysis of the $h$-FEM giving sufficient conditions for the Galerkin solution to exist, and bounds on the corresponding Galerkin error, consider uniform meshes; thus the non-uniform mesh aspect of Theorem \ref{t:simple} is completely new.} 

\blueOld{Even in the case of uniform meshes, 
the local information about the error given by 
Theorem \ref{t:simple} (in the paragraph ``Estimates for uniform meshes" in \S\ref{s:special}) is new. 
Indeed, 
under the assumption that the Galerkin solution exists, 
the study of the local FEM error is classic and goes back to Nitsche and Schatz \cite{NiSc:74}. The $k$-explicit analogue of this theory was obtained in \cite{AvGaSp:24}, where error propagation in the Helmholtz FEM solution was demonstrated numerically \cite[Figure 3]{AvGaSp:24}. 
However, the results of \cite{NiSc:74} and \cite{AvGaSp:24} require apriori knowledge of the Galerkin error. 
More precisely, \cite[Theorem 1.2]{AvGaSp:24} proves the local estimate
\begin{equation} 
\label{e:localQOModulo}
\|u-u_h\|_{H_k^1(U)}\leq C\inf _{w_h\in V_{\mathcal{T}}^p}\|u-w_h\|_{H_k^1(V)}+C_N\|u-u_h\|_{H_k^{-N}(V)}
\end{equation}
where $U$ is compactly contained in $V$. 
%; i.e., the Galerkin solution is locally quasi-optimal modulo low frequencies. 
The estimate~\eqref{e:localQOModulo} provides local information on $u-u_h$ in $H^1_k$ %in a high-regularity norm 
given local information in a low-regularity norm. 
This low-regularity norm of the error contains -- but provides no understanding of -- all of the errors propagating from other parts of the domain. Theorem~\ref{t:simple} makes explicit this error propagation and provides local estimates on the Galerkin error given \emph{only} information about the true solution.
}

\bre[Improvements \blueOld{on Theorem \ref{t:simple}} in Theorem \ref{t:theRealDeal}]
Theorem \ref{t:theRealDeal} below is stronger than Theorem \ref{t:simple} in that it considers arbitrary covers of $\Omega$, and bounds the high ($\gg k$) and low ($\lesssim k$) frequencies of the Galerkin error separately.
Two situations in which a more complicated cover is advantageous are the following. 1) There are two or more cavities that are dynamically separated, i.e., for which there is no billiard trajectory whose closure insects both cavities. 2) One has a priori information about the data and/or solution and hence can obtain good control on the right-hand side of~\eqref{e:simple}. Even when $\cavity=\emptyset$, such information combined with Theorem \ref{t:simple} allows one to define meshes with a priori improved accuracy in some regions, without the need to choose a small meshwidth everywhere.
\ere

\paragraph{\blueOld{Discussion of degrees of freedom savings following from Theorem \ref{t:simple}.
}}

\blueOld{
To extract from Theorem \ref{t:simple} the conditions on $h_\cavity, h_\visible, h_\invisible,$ and $h_\pml$ sufficient to obtain, e.g., $k$-uniform quasi-optimality, requires a simple, but tedious, analysis of the matrix $\transferIntro\mathcal{C}$. The results of this analysis are given in the next subsection, \S\ref{s:special},
and these special cases of Theorem \ref{t:simple} are summarised in Table \ref{tab:regimes}.
}

\tabulinesep = 1.2mm
\begin{table}[htbp]
\small
\hspace{-.6cm}
	%\centering
	$\rotatebox[origin=c]{90}{fewer DoFs \qquad\quad more DoFs}%
	\left\updownarrow
	\begin{tabu}{|M{5.65cm}|M{2.85cm}|M{3.9cm}|M{1.4cm}|}
		\hline
		{\bf Mesh threshold} & {\bf Asymptotic DoFs} & {\bf Theoretical guarantee} & {\bf Name}\\
		\hline
		$(h_\cavity k)^p \rho + (h_\visible k)^p \rho + (h_\invisible k)^p \rho =c$ & $\textup{vol}(\Omega) k^d \rho^{\frac{d}{p}}$ & $k$-QO & U1 \\
		\hline
		$(h_\cavity k)^p \rho + (h_\visible k)^p \sqrt{k\rho} + (h_\invisible k)^p k=c$ & $\textup{vol}(\Omega_\cavity) k^d \rho^{\frac{d}{p}}$ & $k$-QO & QO\\
		\hline
		$(h_\cavity k)^p \sqrt{k\rho} + (h_\visible k)^p k+ (h_\invisible k)^p k =c$ & $\textup{vol}(\Omega_\cavity) k^{d+\frac{1}{2p}}\rho^{\frac{d}{2p}}$ & $k$-QO away from trapping & QO away\\
		\hline
		$(h_\cavity k)^{2p} \rho + (h_\visible k)^{2p} \rho + (h_\invisible k)^{2p} \rho =c$ & $\textup{vol}(\Omega) k^d \rho^{\frac{d}{2p}}$ & CRE & U2\\
		\hline
		$(h_\cavity k)^{2p} \rho + (h_\visible k)^{2p} \sqrt{k \rho} + (h_\invisible k)^{2p} k = c$ & $\textup{vol}(\Omega_\cavity) k^d \rho^{\frac{d}{2p}}$ & CRE&RE\\
		\hline
		$(h_\cavity k)^{2p} \rho + (h_\visible k)^p k + (h_\invisible k)^p k = c$ & $\textup{vol}(\Omega_\cavity) k^d \rho^{\frac{d}{2p}}$ & CRE away from trapping & RE away\\
		\hline
	\end{tabu}
	\right.%
	\rotatebox[origin=c]{90}{}$
	\normalsize
	\caption{Summary of the special cases of Theorem \ref{t:simple} discussed in this section, with $\cavity\neq \emptyset$. Note that in all cases we require $h_Pk= c$ which does not contribute to the asymptotic number of degrees of freedom (DoFs).  Here, $k$-QO stands for $k$-uniform quasioptimality, and CRE stands for controllably-small relative error. \blueOld{$k$-QO away from trapping is defined in~\eqref{eq:defQOaway} and CRE away from trapping is defined in~\eqref{eq:defREaway}. }}
	\label{tab:regimes}
\end{table}

%It is immediately clae
\blueOld{
Theorem \ref{t:simple}, and hence all its special cases in Table \ref{tab:regimes}, only require $h_\Pml k=c$, i.e., a fixed number of points per wavelength in the PML region. Even in non-trapping situations, where $\cavity=\visible=\emptyset$, this gives a substantial savings in number of degrees of freedom. 
Specifically, since the accuracy of PML truncation increases exponentially with the width of the PML (see e.g. \cite{GLS2} and the references therein),
relatively large PML widths, $w$, are desirable. Since the volume of the PML region grows like $w^d$, 
Theorem \ref{t:simple} implies that  only 
$\sim k^{-d}w^{d}$ degrees of freedom are required in the PML instead of $\sim k^{-d}\rho^{d/2p}w^d$. In particular, for the number of degrees of freedom in the PML region to be comparable to the number of degrees of freedom in the rest of domain $w\sim \rho^{1/2p}$; i.e., 
thanks to Theorem \ref{t:simple} 
taking a large PML width is now inexpensive (even for nontrapping problems).

%Even for non-trapping problems with domain of volume 1, increasing PML is nearly free. For pml region to include the same number of degrees of freedom as the rest of the domain this mains that $w\sim \rho^{1/2p}$ -- even in non-trapping problems this means taking a large PML width is not very costly.

From the mesh conditions in both  \eqref{e:meshConditions} and Table \ref{tab:regimes}, we see that the degrees of freedom savings in the non-PML regions depend crucially on the volume of $\Omega_\cavity$ and hence of $\cavity$ relative to $\Omega$.
%with this tied up with the question of how to choose $\Omega_\cavity$ and $\Omega_\visible$.
We now give some simple examples showing how this relative size is often small. 
First, as noted above, $\cavity$ is contained in the convex hull of the scatterer. 
However, in many cases $\cavity$ is contained in smaller sets. Indeed, when $A=I$ and $n=1$, then 
\begin{equation} \label{e:inclusion1}
\cavity\subset 
\Omega_+\cap 
\{ tx+(1-t)y\,:\, x,y\in \partial\Omega_-,\, t\in[0,1]\}.
\end{equation}
For example, if $\Omega_i$, $i=1,\dots, m$ are the connected components of $\Omega_-$, then 
\begin{equation}\label{e:pairs}
\cavity\subset \Omega_+\cap\bigcup_{i,j}\text{convex hull}(\Omega_i\cup \Omega_j)
\end{equation}
(See Figure~\ref{f:pairs} for an illustration of the set given in~\eqref{e:pairs} in an example)
The containment~\eqref{e:inclusion1} is often still a dramatic overestimate of the size $\cavity$. It does however show that if $\Omega_-$ is a union of many obstacles that are small relative to $\Omega$, then $\cavity$ will be small.
}

 \begin{figure}
    \centering
    \begin{tikzpicture}[
    obstacle/.style={fill=black!75, draw=black, line width=0.5pt},
    hullfill/.style={fill=pink, fill opacity=0.28, draw=none},
    hullline/.style={draw=pink!80!black, draw opacity=0.45, line width=0.4pt}
]

% Radius of each obstacle
\def\r{0.15}

% ------------------------------------------------------------
% Obstacle locations
% ------------------------------------------------------------
\coordinate (O1) at (-3.6,  2.2);
\coordinate (O2) at (-1.7,  2.9);
\coordinate (O3) at ( 3.5,  3.0);
\coordinate (O4) at (-4, -2);
\coordinate (O5) at ( 5, -1.0);

% ------------------------------------------------------------
% Shade convex hull of a pair of equal disks, minus the disks
% #1 = first center, #2 = second center, #3 = radius
% ------------------------------------------------------------
\newcommand{\shadepair}[3]{%
    \path[hullfill, even odd rule]
    let
        \p1 = ($ (#2)-(#1) $),
        \n1 = {atan2(\y1,\x1)}
    in
        ($ (#1)+(\n1+90:#3) $)
        --
        ($ (#2)+(\n1+90:#3) $)
        arc[start angle=\n1+90, end angle=\n1-90, radius=#3]
        --
        ($ (#1)+(\n1-90:#3) $)
        arc[start angle=\n1-90, end angle=\n1+90, radius=#3]
        -- cycle
        (#1) circle[radius=#3]
        (#2) circle[radius=#3];

    % optional outline of the hull
    \draw[hullline]
    let
        \p1 = ($ (#2)-(#1) $),
        \n1 = {atan2(\y1,\x1)}
    in
        ($ (#1)+(\n1+90:#3) $)
        --
        ($ (#2)+(\n1+90:#3) $)
        arc[start angle=\n1+90, end angle=\n1-90, radius=#3]
        --
        ($ (#1)+(\n1-90:#3) $)
        arc[start angle=\n1-90, end angle=\n1+90, radius=#3]
        -- cycle;
}

% ------------------------------------------------------------
% Draw ALL pairwise convex hull remainders
% ------------------------------------------------------------
\foreach \A in {O1,O2,O3,O4,O5}
\foreach \B in {O1,O2,O3,O4,O5}
{
{
    \shadepair{\A}{\B}{\r}
}}

% ------------------------------------------------------------
% Draw the obstacles on top
% ------------------------------------------------------------
\foreach \X in {O1,O2,O3,O4,O5}{
    \path[obstacle] (\X) circle[radius=\r];
}

\end{tikzpicture}
\caption{\label{f:pairs}
\blueOld{For the obstacle $\Omega_-$ given by the five black balls, $\cavity$ is contained in the pink region displayed above. If one imagines these balls in three dimensions, then the pink region occupies a much smaller relative portion of the volume than if the balls occupy two dimensions.}}
 \end{figure}

 \begin{figure}
     \centering
     \begin{tikzpicture}
     \begin{scope}{scale=2}
     \clip (-4,-4) rectangle(4.5,4.5);
     \node at (0,0){\includegraphics[width=\linewidth]{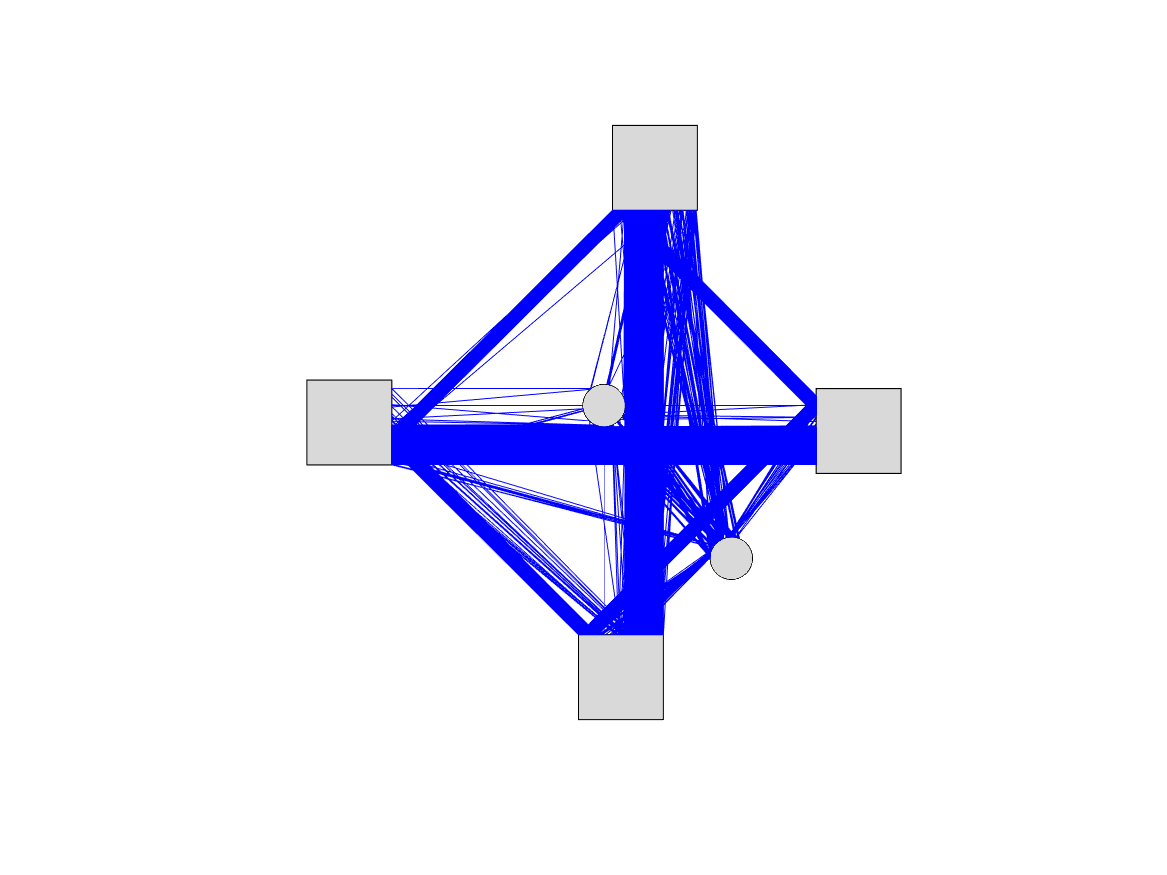}};
\end{scope}
     \end{tikzpicture}
     \caption{\blueOld{\label{f:trapped}For $\Omega_-$ given by the four rectangles and two circles shaded in grey, the figure shows a subset $\widetilde{\cavity}$ computed by ray tracing that contains $\cavity$. Observe that $\tilde{\cavity}$ has a volume that is significantly smaller than that of the region given in~\eqref{e:inclusion1}.}}
 \end{figure}

\blueOld{ 
Figure \ref{f:trapped} plots a set $\widetilde{\cavity}$ containing $\cavity$, obtained by ray tracing and demonstrates how $\cavity$ may be much smaller than the set given in~\eqref{e:pairs} as well as its complicated nature.

Figures \ref{f:meshesQO} 
and \ref{f:meshesRE} give examples of non-uniform meshes for the scatterer used in the numerical experiments in \S\ref{s:numerical} (the precise geometry is described in \S\ref{s:geometries}; we note here that the computational domain is a ball of radius $2.7$).
These meshes are in the QO and RE regimes in Table \ref{tab:regimes} with $p=2$; i.e., they give $k$-uniform quasi-optimality and $k$-uniform controlled relative error, respectively. 
The latter figure gives three different meshes, for increasing values of $k$, showing how the mesh becomes more non-uniform as $k$ increases. 

Finally, for the set up in 
Figures \ref{f:meshesQO} 
and \ref{f:meshesRE}, 
Table \ref{tab:DoFs} gives the number of degrees of freedom for the regimes U1, QO, U2, RE of Table \ref{tab:regimes} for different values of $k$. At the largest value of $k$ in the table ($k=138.84$), the number of degrees of freedom to achieve $k$-uniform quasioptimality is 
3.5 times smaller than that for a uniform mesh, and for controlled relative error the ratio is approximately 2.
For this geometric set up (with a relatively large cavity), the cavity occupies approximately one tenth of the computational domain, i.e., ${\rm vol}(\Omega_\cavity)/{\rm vol}(\Omega)\approx 1/10$, and so, in the limit $k\to\infty$, one expects the non-uniform meshes to have 10 times fewer degrees of freedom than the corresponding uniform meshes.
}

\begin{figure}
    \centering
        \begin{tikzpicture}
\clip (-5.5,-3.3) rectangle (2.5,3.5);
\node at (0,.125){\includegraphics[width=1.15\textwidth]{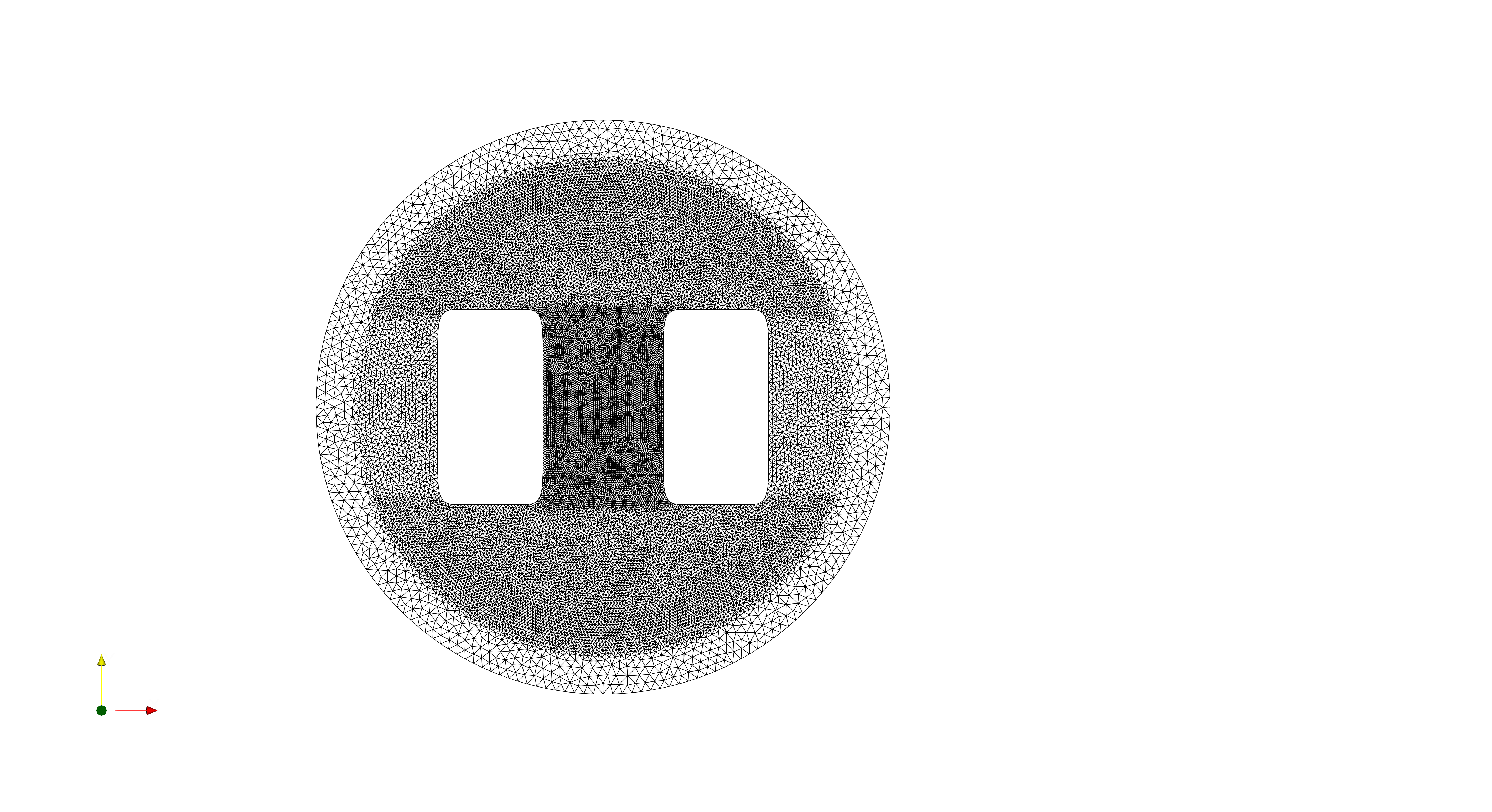}};
\end{tikzpicture}
       \caption{\blueOld{A mesh in the QO regime in Table \ref{tab:regimes} at  $k=55.5365$, with $p = 2$; i.e., satisfying the mesh thresholds guaranteeing $k$-uniform quasioptimality.}}
    \label{f:meshesQO}
\end{figure}

\begin{figure}
    \centering
    \begin{tikzpicture}
    \fill[white] (-9.5,-10.3) rectangle(7,4.5);  
    \node at(0,0){\begin{tikzpicture}[remember picture, overlay]
    \begin{scope}[xshift=-1.125cm]
%\draw[fill=red] (-9.5,-10.4) rectangle (12.5, 3.5);
    \begin{scope}[xshift=-5cm]
\clip (-4.5,-3.3) rectangle (3.5,3.5);
\node at (0,0){\includegraphics[width=.5\textwidth]{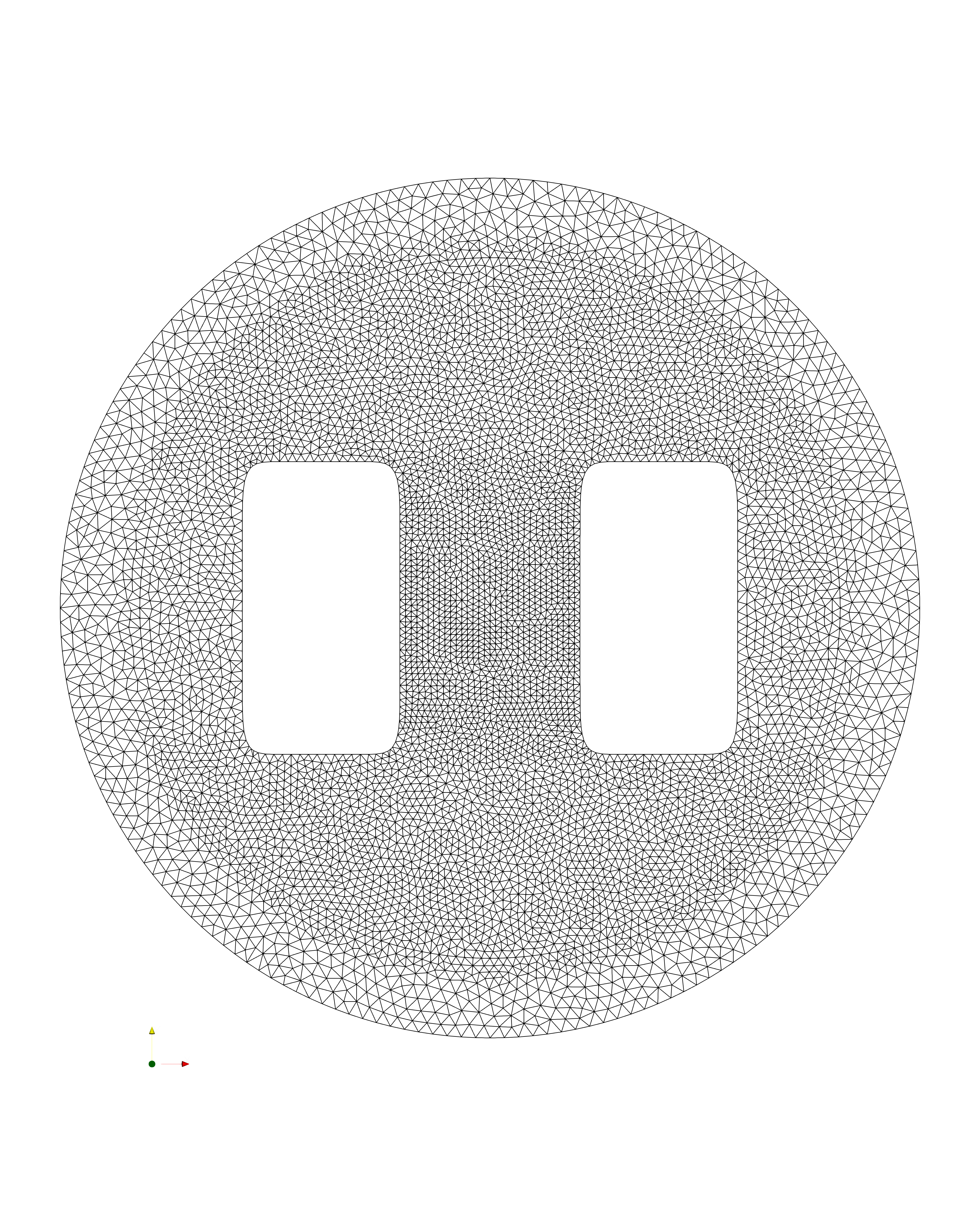}};
\end{scope}
    \node at (-5,3.75){$k=55.536$};
    \begin{scope}[xshift=5cm,yshift=-.16cm]
\clip  (-5.5,-3.3) rectangle (2,3.5);
\node at (0,0){\includegraphics[width=.95\textwidth]{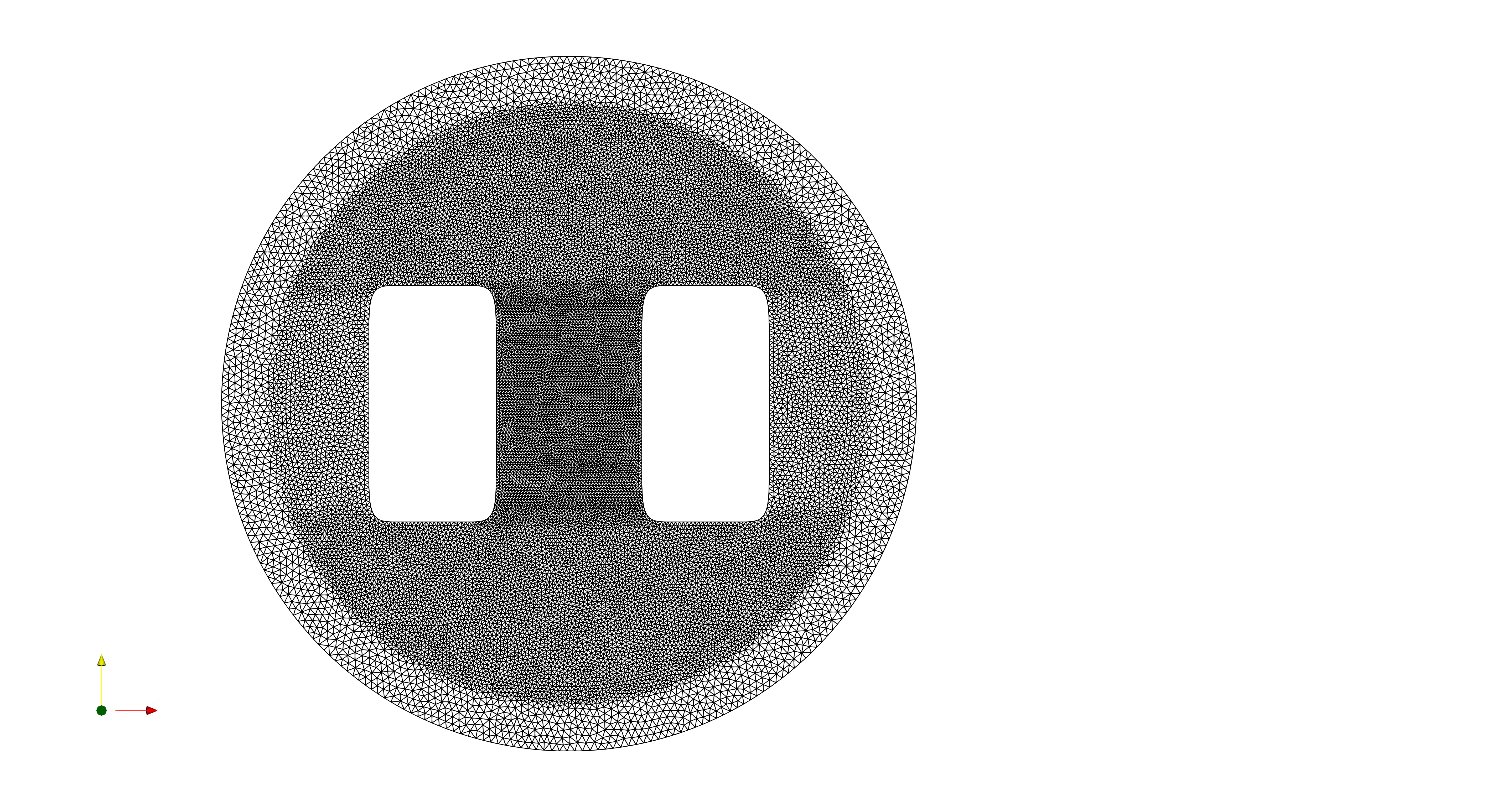}};
\end{scope}
\node at (3.3,3.75){$k=83.3041$};
    \begin{scope}[xshift=-.925cm,yshift=-7cm]
\clip  (-3.5,-3.4) rectangle (3.5,3.6);
\node at (0,0){\includegraphics[width=1.15\textwidth]{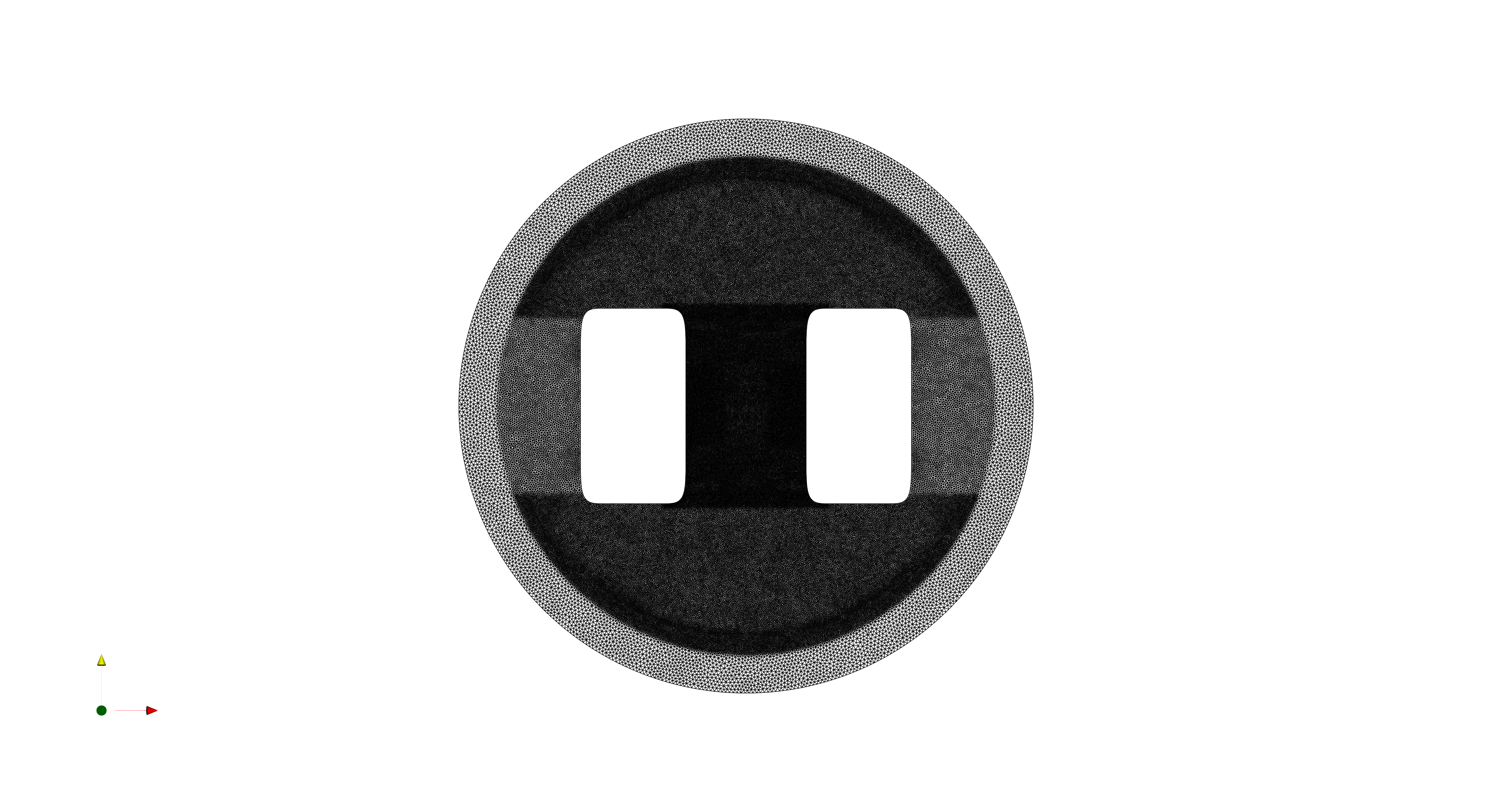}};
\end{scope}
\node at (-.9,-3.1){$k=138.84$};
\end{scope}
\end{tikzpicture}};
    \end{tikzpicture}

    \caption{\blueOld{Meshes in the RE regime of Table~\ref{tab:regimes} with polynomial order $p = 2$, with, ordered clockwise starting at the top left, wavenumbers $k = 55.536$, $k=83.3041$ and $k=138.84$; i.e., satisfying the mesh thresholds guaranteeing $k$-uniformly controlled relative error. The corresponding number of degrees of freedom are given in Table~\ref{tab:DoFs}.}}
    \label{f:meshesRE}
\end{figure}
\begin{table}
    \centering
    \begin{tabular}{|c|r|r|r|}
\hline
      \# DoFs   &$k=55.5365$&$k=83.3041$&$k=138.84$  \\
         \hline
       U1  & 695603&3691907&28091665\\
       \hline
       QO& 310121&1177749&7744981\\
       \hline
       U2&178541&662163&2658691\\
       \hline
       RE&106511&307573&1270585\\
       \hline
    \end{tabular}
    \caption{\label{tab:DoFs}\blueOld{The table shows the number of degrees of freedom used to in the regimes, U1, QO, U2, and RE of Table~\ref{tab:regimes} at three different values of $k$. The corresponding meshes for the QO and RE regimes are pictured in Figures~\ref{f:meshesQO} and~\ref{f:meshesRE} respectively. Recall that U1 is the uniform mesh guaranteeing $k$-uniform quasioptimality and the QO regime is the coarsest mesh guaranteed by Theorem~\ref{t:simple} to give $k$-uniform quasioptimality. Similarly, U2 is the uniform mesh guaranteeing $k$-uniform control on the relative error and RE is the regime guaranteed by Theorem~\ref{t:simple} to give $k$-uniformly controlled relative error.}}
\end{table}

\subsection{Special cases of Theorem \ref{t:simple}}\label{s:special}

We now apply Theorem \ref{t:simple} in several special cases, and derive consequences regarding quasi-optimality and bounded relative errors. 
For simplicity, we state these results for $\newell=0$, i.e., we bound the $H^1_k$ norm of the Galerkin error.
%The results of this section are summarised in Table \ref{tab:regimes}.

Define
$$
\mathscr{M}:= I+\transferIntro
\mathcal{C}(\cH k)^p ,\qquad\qquad \mathscr{M}_{\rm RE}:=\mathscr{M}(\cH k)^p,
$$
and set $\mathscr{M}_\Omega:= \mathscr{M}\begin{pmatrix}1&\dots&1\end{pmatrix}^T$, $\mathscr{M}_{\rm RE,\Omega}:=\mathscr{M}_{\rm RE}\begin{pmatrix}1&\dots&1\end{pmatrix}^T$.
With these definitions, the terms in square brackets on the right-hand sides of \eqref{e:simple} and \eqref{e:relSimple} become, respectively, 
$[
\mathscr{M}
+ k^{-N} \mathscr{F}
]
$ and 
$[
\mathscr{M}_{\rm RE}
+ k^{-N} \mathscr{F}(\cH k)^{p}]$
%\beqs
%\left(\begin{aligned}
%&\|u-u_h\|_{H_k^{1-\newell}(\Omega_\cavity')}\\
%&\|u-u_h\|_{H_k^{1-\newell}(\Omega_\visible')}\\
%&\|u-u_h\|_{H_k^{1-\newell}(\Omega_\invisible')}\\
%&\|u-u_h\|_{H_k^{1-\newell}(\Omega_\pml')}
%\end{aligned}
%\right)
%\leq C
%\Big[
%\mathscr{M}
%+ k^{-N} \mathscr{F}
%\Big]
%\left(\begin{aligned}
%&\|u-w_{h,\cavity}\|_{H_k^{1}(\Omega_\cavity)}\\
%&\|u-w_{h,\visible}\|_{H_k^{1}(\Omega_\visible)}\\
%&\|u-w_{h,\invisible}\|_{H_k^{1}(\Omega_\invisible)}\\
%&\|u-w_{h,\pml}\|_{H_k^{1}(\Omega_\pml)}
%\end{aligned}
%\right)
%\eeqs
%and
%\beqs
%\left(\begin{aligned}
%%\begin{pmatrix} 
%&\|u-u_h\|_{H_k^{1-\newell}(\Omega_\cavity')}\\
%&\|u-u_h\|_{H_k^{1-\newell}(\Omega_\visible')}\\
%&\|u-u_h\|_{H_k^{1-\newell}(\Omega_\invisible')}\\
%&\|u-u_h\|_{H_k^{1-\newell}(\Omega_\pml')}
%%\end{pmatrix}
%\end{aligned}
%\right)
%\leq C
%\Big[
%\mathscr{M}_{\rm RE}
%+ k^{-N} \mathscr{F}(\cH k)^{p}\Big]
%%+\mathscr{R}_{N,m}
%\left(\begin{aligned}
%%\begin{pmatrix} 
%&\|u\|_{H_k^{p+1}(\Omega_\cavity)}\\
%&\|u\|_{H_k^{p+1}(\Omega_\visible)}\\
%&\|u\|_{H_k^{p+1}(\Omega_\invisible)}\\
%&\|u\|_{H_k^{p+1}(\Omega_\pml)}
%\end{aligned}
%\right)
%\eeqs
and, in particular, imply that 
\beqs
\left(\begin{aligned}
&\|u-u_h\|_{H_k^{1}(\Omega_\cavity')}\\
&\|u-u_h\|_{H_k^{1}(\Omega_\visible')}\\
&\|u-u_h\|_{H_k^{1}(\Omega_\invisible')}\\
&\|u-u_h\|_{H_k^{1}(\Omega_\pml')}
\end{aligned}
\right)
C\leq 
\begin{cases}
\Big[ \mathscr{M}_\Omega + k^{-N}
\begin{pmatrix}1 & \ldots &1
\end{pmatrix}^T\Big]\displaystyle\inf_{w_h\in V_{\mc{T}_k}^p}\|u-w_h\|_{H_k^1(\Omega)}
%\eeqs
%and
%\beqs
%\left(\begin{aligned}
%%\begin{pmatrix} 
%&\|u-u_h\|_{H_k^{1}(\Omega_\cavity')}\\
%&\|u-u_h\|_{H_k^{1}(\Omega_\visible')}\\
%&\|u-u_h\|_{H_k^{1}(\Omega_\invisible')}\\
%&\|u-u_h\|_{H_k^{1}(\Omega_\pml')}
%%\end{pmatrix}
%\end{aligned}
%\right)
%\leq C
\\
\\
\Big[
\mathscr{M}_{\rm RE, \Omega}
+ k^{-N}(hk)^p 
\begin{pmatrix}1& \ldots &1
\end{pmatrix}^T\Big]
\N{u}_{H^{p+1}_k(\Omega)}.
\end{cases}
\eeqs

\subsubsection*{Bounds for the coarsest meshes allowed by Theorem \ref{t:simple}.}

\begin{corollary}[Bound on the quasi-optimality constant]
	\label{cor:QOCoarse}
	Under the same assumptions as Thoerem~\ref{t:simple},
	\begin{equation*}
%		\label{eq:Coarse_QOgainAway}
		 \mathscr{M}\leq C\begin{pmatrix}
			\sqrt{\rho} & \sqrt{\rho} & \sqrt{\rho}&0 \\
			\sqrt{k} & \sqrt{k} & \sqrt{k}&0\\
			\sqrt{k} & \sqrt{k} & \sqrt{k}&0\\
0&0&0&1
		\end{pmatrix},\qquad \mathscr{M}_{\Omega}\leq 
C\begin{pmatrix}
			\sqrt{\rho} \\ \sqrt{k} \\ \sqrt{k}\\1
		\end{pmatrix} .
	\end{equation*}
	\end{corollary}

\begin{corollary}[Bound on the relative error]
	\label{cor:REcoarse}
	Under the same assumptions as in Theorem \ref{t:simple}, for all $\varepsilon \leq c$, if $(\mathcal{T}_k)_{k>0}$ satisfies
	$$
	(h_\cavity k)^{2p}\rho(k)+(h_\visible k)^{2p}k+(h_\invisible k)^{2p}k+(h_\pml k)^{2p}\leq \varepsilon,
	$$ 
	then
	\begin{equation*}
			%\label{eq:Coarse_REgainAway}
			 \mathscr{M}_{\rm RE}\leq C\sqrt{\varepsilon}
			\begin{pmatrix}
				1 & \sqrt{\frac{\rho}{k}} & \sqrt{\frac{\rho}{k}}&0\\
				\sqrt{\frac{k}{\rho}} & 1 & 1&0\\
				\sqrt{\frac{k}{\rho}} & 1 & 1&0\\
				0&0&0&1
			\end{pmatrix},\qquad
		\mathscr{M}_{\rm RE,\Omega}\leq C\sqrt{\varepsilon} \begin{pmatrix}
				\sqrt{\rho/k} \\ 1 \\ 1\\1
			\end{pmatrix} .
	\end{equation*}
\end{corollary}

\paragraph{Estimates for uniform meshes.}

For quasi-uniform meshes $h_\cavity  = h_\visible = h_\invisible =h_\pml=: h$, recall from \S\ref{sec:1.1} that the known mesh conditions for ensuring 
$k$-uniform quasioptimality or a controllably-small relative error 
of the Galerkin solution are, respectively 
\[
 (hk)^{p} \rho(k) < c,\, \qquad\,(hk)^{2p} \rho(k) < c,
\]
for $c > 0$ sufficiently small, with the former regime known as the \emph{asymptotic regime}. % We These are known as the asymptotic and pre-asymptotic regimes, respectively. 
Here we show that these thresholds also ensure better error estimates for the Galerkin error away from trapping. 

\begin{corollary}[Asymptotic estimates]%Estimates under the asymptotic regime
	\label{cor:U1}
	Under the same assumptions as Theorem~\ref{t:simple}, if $\cT_k$ satisfies $h_\cavity  = h_\visible = h_\invisible = h_\pml= h$, with
	%\begin{equation*}
%		\label{eq:cond_Asymptotic}
		$(h k)^{p} \rho < c,$
	%\end{equation*}
	then
	\begin{equation*}
%		\label{eq:QOAsymptotic}
		\mathscr{M}\leq C
		\begin{pmatrix}
			1 & \sqrt{\frac{k}{\rho}} & 
			(\frac{k}{\rho})^{\frac{3}{2}}\frac{1}{\rho}& 0
			\\
			\sqrt{\frac{k}{\rho}} & 1 & \frac{k}{\rho} &0\\
			 (\frac{k}{\rho})^{\frac{3}{2}}\frac{1}{\rho}
			   &  \frac{k}{\rho}  & 1 &0\\
			   0&0&0&1
		\end{pmatrix}.
	\end{equation*}
\end{corollary}

\begin{corollary}[Preasymptotic estimates]% under the pre-asymptotic regime]
	\label{cor:U2}
	Under the same assumptions as Theorem~\ref{t:simple}, there exists $C>0$ such that for all $0<\varepsilon<c$ if $(\cT_k)_{k>0}$ satisfies $h_\cavity  = h_\visible = h_\invisible = h_\pml= h$, with
	%\begin{equation}
		%\label{eq:cond_PreAsymptotic}
		$(h k)^{2p} \rho < \varepsilon$, %\qquad \qquad h_Pk\leq \varepsilon,
	%\end{equation}
	 then
	%$\gamma(T) \geq \gamma_0$ and $\mathcal{U}(\mathcal{T},k^{-1}) \leq \gamma_0^{-1}$, the Galerkin solution satisfies
	\begin{equation}
		\label{eq:relErrPreAsymptotic}
%	\begin{pmatrix}
%		\norm{u - u_h}_{H^1_k(\Omega_\cavity' )} \\
%		\norm{u - u_h}_{H^1_k(\Omega_\visible' )} \\
%		\norm{u - u_h}_{H^1_k(\Omega_\invisible' )} \\
%		\norm{u - u_h}_{H^1_k(\Omega_\pml' )} 
%	\end{pmatrix}
	\mathscr{M}_{\rm RE} \leq C
\sqrt{\varepsilon} 
	\begin{pmatrix}
		1 & \sqrt{\frac{k}{\rho}} & \left(\frac{k}{\rho}\right)^{3/2}& 0\\
		\sqrt{\frac{k}{\rho}} & \frac{k}{\rho} + \frac{1}{\sqrt{\rho}} & \frac{k}{\rho} & 0\\
		\big(\frac{k}{\rho}\big)^{3/2}  &  \big(\frac{k}{\rho}\big)^2  & \frac{k}{\rho} + \frac{1}{\sqrt{\rho}} & 0\\
0		& 0& 0& \frac{1}{\sqrt{\rho}}
	\end{pmatrix} .
%	\begin{pmatrix}
%		\norm{u}_{H^{p+1}_k(\Omega_\cavity )} \\
%		\norm{u}_{H^{p+1}_k(\Omega_\visible )} \\
%		\norm{u}_{H^{p+1}_k(\Omega_\invisible )} \\
%		\norm{u}_{H^{p+1}_k(\Omega_\pml)} 
%	\end{pmatrix}.
	\end{equation}
\end{corollary}
%\begin{remark}
%	Eq. \eqref{eq:relErrPreAsymptotic} gives in particular that the relative error $\norm{u - u_h}_{H^1_k(\Omega)}/\norm{u}_{H^{p+1}_k(\Omega)}$. Moreover, it ensures that the local errors in $\Omega_\cavity,\Omega_\visible$ and $\Omega_\invisible$ are bounded by the local magnitude of the solution in the corresponding region, up to ``attenuated propagation terms". 
%\end{remark}

%%In the matrices on the right-hand sides of \eqref{eq:QOAsymptotic} and \eqref{eq:relErrPreAsymptotic}, the maximum entry in each row is 1. Therefore, if each norm on the right-hand side of either \eqref{eq:QOAsymptotic} or \eqref{eq:relErrPreAsymptotic} is bounded by a norm on $\Omega$, these bounds reduce to the standard asymptotic and preasymptotic bounds (involving global norms). However, 
%The fact that some matrix entries on the right-hand sides of \eqref{eq:QOAsymptotic} and \eqref{eq:relErrPreAsymptotic} are $\ll 1$ when $\rho \gg k$ means that, under trapping and with a uniform mesh,
%the error is smaller in some parts of the domain -- see the experiments in \S\ref{sec:expU1} and \S\ref{sec:expU2} for illustrations of this. 
%%\eqref{eq:QOAsymptotic} and \eqref{eq:relErrPreAsymptotic} are stronger bounds.

\subsubsection*{Weakest conditions guaranteeing $k$-uniform quasi-optimality and controllably-small relative error.}

We proceed by identifying the minimal thresholds under which Theorem \ref{t:simple} guarantees that (i) the Galerkin solution is quasi-optimal,  uniformly in $k$, and 
(ii) the relative error is controllably small.

\begin{corollary}[Threshold for $k$-uniform quasi-optimality]
	\label{cor:QO}
	Under the same assumptions as Theorem~\ref{t:simple}, 
%	Let $k_0 > 0$, $\gamma_0 > 0$ and $p \geq 1$. Let $\blueOld{\mathrm{J}} \subset \mathbb{R}_+$ be such that Assumption \ref{a:polyBoundIntro} holds. Then there are $c, C >0$ such that for all $k \in (k_0,\infty)\setminus \blueOld{\mathrm{J}}$ and all 
if $(\mathcal{T}_k)_{k>0}$ satisfies
	\begin{equation}
		\label{eq:cond_QO}
		(h_\cavity k)^p \rho + (h_\visible k)^p \sqrt{k \rho} + (h_\invisible k)^p k + (h_\pml k)^p < c,
	\end{equation}
	then
%	$\gamma(T) \geq \gamma_0$ and $\mathcal{U}(\mathcal{T},k^{-1}) \leq \gamma_0^{-1}$, the Galerkin solution satisfies
	\begin{equation}
		\label{eq:QO}
%		\begin{pmatrix}
%			\norm{u - u_h}_{H^1_k(\Omega_\cavity' )} \\
%			\norm{u - u_h}_{H^1_k(\Omega_\visible' )} \\
%			\norm{u - u_h}_{H^1_k(\Omega_\invisible' )} \\
%\norm{u - u_h}_{H^1_k(\Omega_\pml' )} 
%		\end{pmatrix} \leq C
\mathscr{M}\leq C		\begin{pmatrix}
			1 & 1 & \frac{1}{\sqrt{k \rho}}&0\\
			\sqrt{\frac{k}{\rho}} & 1 & 1 &0\\
			\frac{1}{\rho}\sqrt{\frac{k}{\rho}}
			&  \sqrt{\frac{k}{\rho}}  & 1 &0\\
			0 & 0 & 0&1
		\end{pmatrix} .
%		\begin{pmatrix}
%			\norm{u -w_h}_{H^{1}_k(\Omega_\cavity )} \\
%			\norm{u-w_h}_{H^{1}_k(\Omega_\visible' )} \\
%			\norm{u-w_h}_{H^{1}_k(\Omega_\invisible' )} \\
%						\norm{u - u_h}_{H^1_k(\Omega_\pml' )} 
%		\end{pmatrix}.
	\end{equation}
\end{corollary}

%	The estimate \eqref{eq:QO} tells us that the Galerkin error away from trapping can be decomposed into a local best approximation error and a pollution term coming from the cavity, but with an attenuation factor of $\sqrt{\frac{k}{\rho}}$.

%\begin{remark}
%	The estimate \eqref{eq:QO} tells us that the Galerkin error away from trapping can be decomposed into a local best approximation error and a pollution term coming from the cavity, but with an attenuation factor of $\sqrt{\frac{k}{\rho}}$. In particular, if the solution $u$ grows faster in the cavity (that is, if the right-hand side ``activates" the trapping), then the Galerkin error away is smaller than the global Galerkin error. 
%\end{remark}

\begin{corollary}[Threshold for bounded relative error]
	\label{cor:RE}
	Under the same assumptions as Theorem~\ref{t:simple}, 
there exists $C>0$ such that for all $0<\varepsilon<c$ 
	%	Let $k_0 > 0$, $\gamma_0 > 0$ and $p \geq 1$. Let $\blueOld{\mathrm{J}} \subset \mathbb{R}_+$ be such that Assumption \ref{a:polyBoundIntro} holds. Then there are $c, C >0$ such that for all $k \in (k_0,\infty)\setminus \blueOld{\mathrm{J}}$ and all 
if $(\mathcal{T}_k)_{k>0}$ satisfies
%	Let $k_0 > 0$, $\gamma_0 > 0$ and $p \geq 1$. Let $\blueOld{\mathrm{J}} \subset \mathbb{R}_+$ be such that Assumption \ref{a:polyBoundIntro} holds. Then there are $c, C >0$ such that for all $k \in (k_0,\infty)\setminus \blueOld{\mathrm{J}}$, $\varepsilon \leq c$ and all meshes $\mathcal{T}$ satisfying,
	\begin{equation}
		\label{eq:cond_RE}
		(h_\cavity k)^{2p} \rho + (h_\visible k)^{2p} \sqrt{\rho k} + (h_\invisible k)^{2p} k + (h_\pml k)^{2p}< \varepsilon,
	\end{equation}
	then 
%	$\gamma(T) \geq \gamma_0$ and $\mathcal{U}(\mathcal{T},k^{-1}) \leq \gamma_0^{-1}$, the Galerkin solution satisfies
	\begin{equation}
		\label{eq:RE}
%		\begin{pmatrix}
%			\norm{u - u_h}_{H^1_k(\Omega_\cavity )} \\
%			\norm{u - u_h}_{H^1_k(\Omega_\visible )} \\
%			\norm{u - u_h}_{H^1_k(\Omega_\invisible )} \\
%			\norm{u - u_h}_{H^1_k(\Omega_\pml )} 
		%\end{pmatrix}
		\mathscr{M}_{\rm RE}
		\leq C\sqrt{\varepsilon}
		\begin{pmatrix}
			1 & 1 & 1 &0\\
			\sqrt{\frac{k}{\rho}} & 	\sqrt{\frac{k}{\rho}} + \frac{1}{(\rho k)^{1/4}}& 1&0\\
			\frac{k}{\rho}&\frac{k}{\rho} & 1 &0\\
			0 &0&0 &1
		\end{pmatrix} .
%		\begin{pmatrix}
%			\norm{u}_{H^{p+1}_k(\Omega_\cavity )} \\
%			\norm{u}_{H^{p+1}_k(\Omega_\visible )} \\
%			\norm{u}_{H^{p+1}_k(\Omega_\invisible )} \\
%			\norm{u}_{H^{p+1}_k(\Omega_\pml )} 
%		\end{pmatrix}.
	\end{equation}
	In particular, $\norm{u - u_h}_{H^1_k(\Omega)} / \norm{u}_{H^{p+1}_k(\Omega)}$ is bounded. 
\end{corollary}

%\bre[Improvements away from trapping]
%In QO estimates.... see that error away from trapping is locally QO plus contribution from trapping multiplied by a factor decreasing with $k$.
%\ere

\bre
In Corollaries \ref{cor:REcoarse} and \ref{cor:U2} one can track how the matrix entries depend on $\varepsilon$, but we do not do this here for simplicity.
\ere

\paragraph{Condition guaranteeing quasi-optimality away from trapping.}

We finally give the weakest condition under which Theorem \ref{t:simple} ensures that the quantities
\begin{equation}
	\label{eq:defQOaway}
	\frac{\norm{u-u_h}_{H^1_k(\Omega_\visible')}}{\displaystyle\min_{w_h \in V_{\mathcal{T}}^p}\norm{u - w_h}_{H^1_k(\Omega)}}, \quad \frac{\norm{u-u_h}_{H^1_k(\Omega_\invisible')}}{\displaystyle\min_{w_h\in V_{\mathcal{T}}^p}\norm{u - w_h }_{H^1_k(\Omega)}}
\end{equation}
remain $k$-uniformly bounded. We refer to these quantities as the quasi-optimality constants ``away from trapping". These quantities should not be confused with ``local quasi-optimality" constants (which would be defined with $\Omega_\visible$ and $\Omega_\invisible$ instead of $\Omega$ in the denominators of \eqref{eq:defQOaway}).

\begin{corollary}[Threshold for $k$-uniform ``quasi-optimality away from trapping"]
	\label{cor:QOaway}
		Under the same assumptions as Theorem~\ref{t:simple}, 
if $(\mathcal{T}_k)_{k>0}$ satisfies
	\begin{equation}
		\label{eq:cond_QOaway}
		(h_\cavity k)^p \sqrt{k\rho} + (h_\visible k)^p k + (h_\invisible k)^p k + (h_\pml k)^p < c,
	\end{equation}
then 
	\begin{equation}
		\label{eq:QOaway}
\mathscr{M}\leq 		C\begin{pmatrix}
			\sqrt{\frac{\rho}{k}} & \sqrt{\frac{\rho}{k}} & \frac{1}{k^2}\sqrt{\frac{\rho}{k}} &0\\
			1 & 1 & 1 &0\\
			\frac{1}{k} &1 & 1&0\\
			0&0&0&1
		\end{pmatrix} 
\quad\tand\quad
\mathscr{M}_\Omega
 \leq C \begin{pmatrix}
		\sqrt{\frac{\rho}{k}}\\
			1\\
			1\\
			1
	\end{pmatrix} .
	\end{equation}
\end{corollary}
\begin{remark}
	It is natural to look for the analogous weakest condition guaranteeing a controllably small $k$-uniform bound on the ``relative error away from trapping", defined by
	\begin{equation}
		\label{eq:defREaway}
		\frac{\norm{u-u_h}_{H^1_k(\Omega_\visible')}}{\norm{u}_{H^{p+1}_k(\Omega)}}, \quad \frac{\norm{u-u_h}_{H^1_k(\Omega_\invisible')}}{\norm{u}_{H^{p+1}_k(\Omega)}}
	\end{equation}
 	(again not to be confused with ``local relative errors" which would involve $\norm{u}_{H^1_k(\Omega_\visible \cap \Omega)}$ and $\norm{u}_{H^1_k(\Omega_\invisible \cap \Omega)}$ in the denominators). However, these quantities are already bounded under the weakest possible condition in Theorem \ref{t:simple}, see Corollary \ref{cor:REcoarse}.
\end{remark}

\subsection{Discussion of the ideas behind Theorem~\ref{t:simple} and a sketch of the proof}

\subsubsection{
\blue{The two phenomena motivating} 
%The ideas behind 
Theorem~\ref{t:simple}}
%\red{make clear that high level stuff is here}

\blueOld{Recall that $P_k$ is the Helmholtz PML operator defined in \eqref{e:PML1}.}
The following two important phenomena motivate Theorem \ref{t:simple}. 
\begin{enumerate}
\item \emph{The solution operator $P_k^{-1}$ reflects the billiard dynamics in $\Omega$.} In particular, for $\chi_1,\chi_2\in C^\infty(\Omega)$ the operator $\chi_1P_k^{-1}\chi_2$ behaves differently depending on the locations of $\supp \chi_j$; e.g., $\|P_k^{-1}\|\gg k$ when $\Omega_{\cavity}\neq \emptyset$, but if both $\chi_1$ and $\chi_2$ are away from $\Omega_{\cavity}$ then $\|\chi_1P_k^{-1}\chi_2\|_{L^2\to L^2}\lesssim k$.  
\blueOld{(Recall that the $ij$th entry of the matrix $\mathcal{C}$ in \eqref{e:matrices} is a bound on $\|\chi_i P_k^{-1}\chi_j\|_{L^2\to L^2}$.)
}

\item \emph{The Galerkin error propagates.} The best possible situation would be local quasioptimality i.e., there exists $C>0$ such that the Galerkin solution $u_h$ satisfies, for every $U\subset\Omega$,
\begin{equation}
\label{e:localQO}
\|u-u_h\|_{H^1_k(U)}\leq C\inf_{w_h\in V_{\mathcal{T}}^p}\|u-w_h\|_{H^1_k(U)}.
\end{equation}
In this case, since approximation of oscillatory functions by piecewise polynomials is well understood (see \cite{G1} and the references therein), 
the properties of the data and behaviour of $P_k^{-1}$ would dictate the meshwidth in each region. 
Unfortunately~\eqref{e:localQO} cannot hold for general meshes. Indeed, suppose that \eqref{e:localQO} holds and let $\phi,\phi_1,\phi_2\in C^\infty(\Omega)$ be such that $\supp \phi\subset \{\phi_1\equiv 1\}$, $\phi\neq 0$, and $\phi_1+\phi_2\equiv 1$ on $\Omega$. Then,
\begin{equation}
\label{e:errorMoving}
\begin{aligned}
\phi(u-u_h)&=\sum_{j=1}^2\phi P_k^{-1}\phi_j P_k(u-u_h).
\end{aligned}
\end{equation}
We now consider a situation where $\mc{T}$ has arbitrarily small elements on $\supp \phi_1=:\Omega_1$ so that, by~\eqref{e:localQO}, 
$$
\|u-u_h\|_{H^1_k(\Omega_1)}\leq C\inf_{w_h\in V_{\mathcal{T}}^p}\|u-w_h\|_{H^1_k(\Omega_1)}\ll 1 .
$$
In particular,
$$
\|\phi(u-u_h)\|_{H^1_k}+\|P_k^{-1}\phi_1P_k(u-u_h)\|_{H^1_k}\ll 1
$$
(by continuity of $P_k$ and $P_k^{-1}$ and locality of $P_k$).
Then,~\eqref{e:errorMoving} implies that
$$
\|  \phi P_k^{-1}\phi_2P_k(u-u_h)\|_{H^1_k}\ll 1,
$$
which cannot be true unless the meshwidth is also sufficiently small on $\Omega_2$ or $\phi P_k^{-1}\phi_2\approx 0$. By Item 1, the latter is not the case whenever $\supp \phi$ and $\supp \phi_2$ are connected by a billiard trajectory. (For a striking illustration of this propagation of error, see~\cite[Figure 3]{AvGaSp:24}.)

This argument indicates, not only that the Galerkin error propagates, but that the norm of the operator $\phi P_k^{-1}\phi_2$ determines the strength of propagation from $\supp \phi_2$ to $\supp \phi$. 
%In fact, this argument by contradiction also indicates precisely how large the error propagation effect is; i.e. the norm of the operator $\phi P_k^{-1}\phi_2$ determines the strength of propagation from $\supp \phi_2$ to $\supp \phi$. 
\end{enumerate}
Item 1 motivates varying the meshwidth from one location to another, but Item 2 shows that, to be effective, this strategy must 
%indicates that this choice of mesh width needs to 
take into account the global behaviour of billiard trajectories. 
In particular, by Item 2, the error in the cavity is \emph{not} just dictated by the meshwidth in the cavity
%:Item 2 shows that a small mesh width inside the cavity is \emph{not} enough to guarantee an accurate solution there 
-- the meshwidth also needs to be sufficiently small away from the cavity to control the propagating error. 

\subsubsection{\blue{Error propagation and the ideas behind Theorem~\ref{t:simple} on a conceptual level}}

In the proofs of Theorems \ref{t:simple} and \ref{t:theRealDeal}, the propagation of Galerkin error \blue{described in the previous subsection} is captured via 
the graph in Figure~\ref{f:graph1} (and more complicated generalisations; see Figures \ref{f:twitch} and \ref{f:graph2}). 
\blue{This error repeatedly propagates through the graph, and under 
the mesh-thresholds in Theorem~\ref{t:simple} this accumulated error remains bounded.}
%local errors 
%repeatedly propagating through the graph 
%in Figure \ref{f:graph1} 
%do not become unbounded
% \blue{Observe that the weights of this graph involve the local meshwidths and the norms of the solution operator (coming from $\mathcal{C}$ \eqref{e:matrices}) -- one should think of this graph as describing the propagation of local errors throughout the domain (up to swapping the direction of the arrows caused by using a duality argument), 
% and the mesh-thresholds in Theorem~\ref{t:simple} then guarantee that these
% local errors repeatedly propagating through the graph 
% %in Figure \ref{f:graph1} 
% do not become unbounded.
% }
% In fact, errors do not propagate just once in the graph: once errors are made in a region and propagate to other regions, these new errors propagate again and so on. The Galerkin error is the result of the infinite repetition of this process. 
% One way to understand this graph is that it gives 
% (up to swapping the direction of the arrows caused by using a duality argument) the magnitude of the error in each domain caused by making a local error in $P_k(u-u_h)$ and propagating via the solution operator (with norms bounded via $\mathcal{C}$ \eqref{e:matrices}) this error into all other domains.
% One should then understand the mesh-thresholds in Theorem~\ref{t:simple} as guaranteeing that these
% local errors repeatedly propagating through the graph in Figure \ref{f:graph1} do not become unbounded.
%if one makes a local error, propagates, and repeats the error does not become unbounded. 
\blue{To understand how this repeated propagation arises:} 
%This repeated propagation appears because 
the (discrete) operator, $\widehat{P}_k$, coming from the FEM approximation to the Helmholtz operator $P_k$ is given by $\widehat{P}_k=P_k-E_k$, where $E_k$ encodes the errors incurred by approximating $P_k$. %\blue{(note that this error is local since both $P_k$ and $\widehat{P}_k$ are local)}. 
Since our goal is to show that $\widehat{P}_k^{-1}$ is close to $P_k^{-1}$, it is natural to write 
$$
\widehat{P}_k^{-1}= (P_k-E_k)^{-1}= (I-P_k^{-1}E_k)^{-1}P_k^{-1}=\sum_{j=0}^{\infty}(P_k^{-1}E_k)^jP_k^{-1}.
$$
Each application of $P_k^{-1}E_k$ then represents one step of making an error and propagating and the absence of accumulation means that the sum converges. \blue{
At the highest level, the proof of Theorem~\ref{t:simple} can be understood as showing that $P_k^{-1}E_k$ is controlled by the graph in Figure~\ref{f:graph1}, with one important component of this showing that $E_k$ is approximately local.}

% \S\ref{s:sketch} sketches the proof of Theorem \ref{t:simple} -- showing explicitly how the graph in Figure \ref{f:graph1} enters the analysis -- and \S\ref{s:interpret} 
% concretely links the key identity \eqref{e:startDuality} in the proof to the propagation of Galerkin error.

\begin{figure}[htbp]
\begin{center}
\begin{tikzpicture}[->,>=stealth,shorten >=1pt,auto,node distance=7cm,semithick]

\begin{scope}[xscale=2,yscale=1.5]
  % Define nodes in a line
  \node[draw, circle] (A) at (0, 0) {${\Omega_\cavity}$};
  \node[draw, circle] (B) at (2, 0) {${\Omega_\visible}$};
  \node[draw, circle] (C) at (4, 0) {${\Omega_\invisible}$};
  \node[draw, circle] (D) at (6, 0) {${\Omega_{\pml}}$};

  % Define edges and add labels
  \draw[<-] (A) to[bend left] node[midway, above] {$({h_\cavity }k )^{2p}\sqrt{k {\rho}}$} (B);
  \draw[<-] (B) to[bend left] node[midway, below] {$({h_\visible} k )^{2p}\sqrt{k {\rho}}$} (A);
  \draw[<-] (B) to[bend left] node[midway, above] {$({h_\visible} k )^{2p}k $} (C);
  \draw[<-] (C) to[bend left] node[midway, below] {$({h_\invisible}k )^{2p}k $} (B);
  \draw[<-] (D) to[in=0,out=-135] (4,-1.3)node[ below] {$({h_\pml}k )^{2p}$}to[in = -45,out=180] (B);
  \draw[<-] (B) to[in=180, out=45] (4,1.3)node[ above] {$({h_\visible}k )^{2p}$}to[in=135,out=0] (D);
    \draw[<-] (C) to[bend left] node[midway, below] {$({h_\invisible}k )^{2p}$} (D);
  \draw[<-] (D) to[bend left] node[midway, below] {$(h_{\pml}k )^{2p}$} (C);
%
  % Add self-loops with labels
  \draw[<-] (A) to[loop above] node[midway, above] {$({h_\cavity }k )^{2p}{\rho}$} (A);
  \draw[<-] (B) to[loop above] node[midway, above] {$({h_\visible} k )^{2p}k $} (B);
  \draw[<-] (C) to[loop above] node[midway, above] {$({h_{\invisible}}k )^{2p}k $} (C);
  \draw[<-] (D) to[loop above] node[midway, above] {$(h_{\pml}k )^{2p}$} (D);
  \end{scope}
\end{tikzpicture}
\end{center}
\caption{\label{f:graph1} The graph showing propagation of errors for the decomposition into $\Omega_\cavity$, $\Omega_\visible$, $\Omega_\invisible$, and $\Omega_{\pml}$ in the simplified setup of Section~\ref{s:sketch}. Note that this can be improved using the analysis in Section~\ref{sec:proof_realDeal}. The graph corresponding to Theorem~\ref{t:simple} is shown in Figure~\ref{f:graph2} and that for Theorem~\ref{t:theRealDeal} is shown in Figure~\ref{f:twitch}.}
\end{figure}

\subsubsection{
%\blue{The ideas of the proof of Theorem~\ref{t:simple} and a sketch}}
Sketch of the proof of Theorem~\ref{t:simple}}
\label{s:sketch}

%Sketch the proof
% \blue{The starting point of the proof is the duality argument of \cite{GS3}. 
% \S\ref{s:sketch} sketches the proof of Theorem \ref{t:simple} -- showing explicitly how the graph in Figure \ref{f:graph1} enters the analysis -- and \S\ref{s:interpret} 
% concretely links the key identity \eqref{e:startDuality} in the proof to the propagation of Galerkin error.

\blue{
The main idea of the proof is to take advantage of the fact that the solution operator, $P_k^{-1}$, behaves differently in different places-- both in frequency and in physical space. Specifically, 
the starting point of the proof of Theorem \ref{t:simple} is the latest analysis for $h$-FEM with a uniform mesh from \cite{GS3}. That analysis 
%is based on the ``elliptic projection'' argument from~\cite{GS3} that 
uses the fact that $P_k$ is coercive at high frequencies and hence its inverse is very well behaved there. The proofs of Theorem~\ref{t:simple} and Theorem~\ref{t:theRealDeal} then localise this argument further in physical space-- introducing a partition of the domain to take advantaged of better- and worse-behaved parts of the solution operator. We highlight that the arguments of \cite{GS3} use nonlocal operators; nevertheless, as we show here, these operators turn out to be almost local at high-frequency, we say that they are {\em pseudolocal}.  
%``pseudolocal'' 
Proving this pseudolocality, and handling errors from it, 
%this pseudolocality 
is one of the main technical challenges in implementing the sketch below. There are additional substantial challenges described in more detail in \S\ref{s:whyAmISOOOOOLONg}.}

\blue{Our goal in this subsection is to sketch the proof of the bound \eqref{e:simple} with $m=p$, and 
 show explicitly how the graph in Figure \ref{f:graph1} enters the analysis. For simplicity, along with only considering $m=p$, 
}
%For simplicity, we consider here the bound \eqref{e:simple} with $m=p$ and 
we ignore improvements that are possible in the overlaps between subdomains, in the PML region, and by splitting the frequencies of the Galerkin error into those $\gg k$ and $\lesssim k$. \blue{We highlight that this frequency splitting is essential in proving the bound \eqref{e:simple} with $m=0$ (i.e., in the energy norm) under the mesh thresholds \eqref{e:meshConditions}.} 

%\paragraph{Recap of the elliptic projection argument from \cite{GS3}.}

% The proofs of Theorem~\ref{t:simple} and Theorem~\ref{t:theRealDeal} are, at heart, localised versions of the elliptic projection-type argument introduced in~\cite{GS3}.

\paragraph{\blue{Recap of the preasymptotic error analysis for a uniform mesh.}}
We first recap 
the elliptic projection-type argument introduced in~\cite{GS3}
%argument %the argument in~\cite{GS3} 
and prove~\eqref{e:preasymptotic1} for $m=p$.
%\blue{The standard duality arguments used to analyse the $h$-FEM are based on the following.}
\blue{We start as in the standard duality argument (going back to Aubin \cite{Au:67} and Nitsche \cite{Ni:68}).
Denote $R_k:=P_k^{-1}$ and let $R_k^*$ be its $L^2$ adjoint. Then, by Galerkin orthogonality~\eqref{e:galerkinDef}, for all $\tilde{w}_h\in V_{\mc{T}_k}^p$, $v\in H_k^{p-1}$,
%\begin{equation}
%\label{e:basicEllipticProjection}
\begin{equation}
\langle u-u_h,v\rangle = \langle R_kP_k(u-u_h),v\rangle = \big\langle P_k(u-u_h) , R_k^*v \big\rangle= \big\langle P_k(u-u_h), R_k^* v -\tilde{w}_h\big\rangle.
\label{e:standard} 
\end{equation}
The most obvious choice of $\tilde{w}_h$ is the best approximation of $R_k^*v$, with this the approach taken in \cite{Au:67, Ni:68}, and (in the Helmholtz context) Schatz \cite{Sc:74}. This leads to a proof that the Galerkin solution exists and is unique when $(hk)^p\rho$ is sufficiently small (see, e.g., \cite[Lemma 6.9]{GaSp:25b}).

The key observation in~\cite{FeWu:09, FeWu:11} (albeit in the context of discontinuous Galerkin methods) is that one can make a better choice of $\tilde{w}_h$; namely $\tilde{w}_h:= \tilde{\Pi}_k R^*_kv$ where $\tilde{\Pi}_k^\sharp
:H_k^1\to V_{\mc{T}_k}^p$ is a projection defined by
$$
\big\langle (P_k+C)w_h, (I- \tilde{\Pi}_k)u\big\rangle =0\quad \tfa w_h\in V_{\mc{T}_k}^p
 $$
where $C$ is sufficiently large so that $P_k+C$ is coercive. Then, Céa's lemma implies that such a $\tilde{\Pi}_k$ exists and is quasioptimal. 
%this projection was called the \emph{elliptic projection} in \cite{FeWu:09,FeWu:11}. 
The point is that, with the choice $\tilde{w}_h:= \tilde{\Pi}_k R^*_kv$ in \eqref{e:standard},
one can replace $u-u_h$ by $u-w_h$ (for any $w_h\in \mathcal{V}_{\mathcal{T}_k}^p$) up to a lower-order term. More precisely,
\begin{equation}
\langle u-u_h,v\rangle = \big\langle P_k(u-w_h), (I-\tilde{\Pi}_k)R_k^* v\big\rangle - C \big\langle u-u_h, (I-\tilde{\Pi})R^* v \big\rangle.
\label{e:standard2} 
\end{equation}
For the Helmholtz $h$-FEM with a uniform mesh, this leads to a proof that the Galerkin solution exists and is unique when $(hk)^{p+1}\rho$ is sufficiently small.
}

The key insight in~\cite{GS3} \blue{is that one can make a still better choice of $\tilde{w}_h$ in~\eqref{e:standard} that leads to a proof that the Galerkin solution exists and is unique when  $(hk)^{2p}\rho$ is sufficiently small. The idea is that, since $P_k$ is coercive at high frequencies, there is} a self-adjoint smoothing operator $S_k$ so that $P_k^\sharp:= P_k+S_k$ is coercive (uniformly in $k$) and for all $N$ there exists $C>0$ such that for $k\geq k_0$ 
$$
\|S_k\|_{H_k^{-N}\to H_k^N}\leq C
$$
(see~\eqref{e:def_Sk} for the definition of the operator $S_k$).
Since $P_k^\sharp$ is coercive (\blue{again by C\'ea's lemma}) there is an \emph{elliptic projection} $\Pi_k^\sharp:H_k^1\to V_{\mc{T}_k}^p$ such that 
\beq\label{e:introPiSharp}
\big\langle P^\sharp_k w_h, (I- \Pi^\sharp_k)u\big\rangle =0 \quad\tfa w_h \in V_{\mc{T}_k}^p
\eeq
and there exists $C>0$ such that for all $k>k_0$
\beq\label{e:introPiSharpCea}
\|(\Id - \Pi_k^\sharp)v\|_{H^1_k} \leq C \inf_{w_h \in V_{\mc{T}_k}^p} \|v - w_h\|_{H^1_k}
\eeq
(i.e. $\Pi^\sharp_k$ is the adjoint Galerkin projection associated to $P^\sharp_k$). Moreover, by an Aubin--Nitsche-type duality argument
\beq
\label{e:lowNormPiSharp}
\|(\Id - \Pi_k^\sharp)v\|_{H^{-p+1}_k} \leq C(hk)^p\inf_{w_h \in V_{\mc{T}_k}^p} \|v - w_h\|_{H^1_k}.
\eeq

%\blue{We now start as in the standard duality argument (going back to Aubin \cite{Au:67} and Nitsche \cite{Ni:68}).}
%, and famously used in the Helmholtz context by Schatz \cite{Sc:74}).}
%Denote $R_k:=P_k^{-1}$ and let $R_k^*$ be its $L^2$ adjoint. 
\blue{Then, choosing in~\eqref{e:standard} $\tilde{w}_h=\Pi_k^\sharp R_k^*v$} and using~\eqref{e:introPiSharp}, we obtain 
 % it follows from and Galerkin orthogonality~\eqref{e:galerkinDef}
 that for all $w_h\in V_{\mc{T}_k}^p$, $v\in H_k^{p-1}$,
\begin{align}\nonumber
\langle u-u_h,v\rangle 
%&= \big\langle P_k(u-u_h) , R_k^*v \big\rangle\\
&= \big\langle P_k(u-u_h), (I-\Pi^\sharp_k)R_k^* v \big\rangle \\ \nonumber
&= \big\langle P^\sharp_k(u-u_h), (I-\Pi^\sharp_k) R_k^* v \big\rangle - 
\big\langle S_k(u-u_h), (I-\Pi^\sharp_k) R_k^* v \big\rangle \\
&= \big\langle P^\sharp_k(u-w_h), (I-\Pi^\sharp_k) R_k^*v\big\rangle - 
\big\langle S_k(u-u_h), (I-\Pi^\sharp_k) R_k^* v\big\rangle
\label{e:basicEllipticProjection}
\end{align}
\blue{(compare to \eqref{e:standard2}).}
By~\eqref{e:introPiSharpCea},~\eqref{e:lowNormPiSharp}, 
the mapping property $P_k^\sharp:H_k^1\to H_k^{-1}$ \blue{and crucially the smoothing property} $S_k:H_k^{-p+1}\to H_k^{p-1}$,
\begin{equation}
\label{e:preasymptoticSketch1}
\begin{gathered}
|\langle u-u_h,v\rangle|\leq C\eta_p\Big( \inf_{w_h\in V_{\mc{T}_k}^p}\|u-w_h\|_{H_k^1}+(hk)^p\|u-u_h\|_{H_k^{-p+1}}\Big)\|v\|_{H_k^{p-1}},\\
\text{ where }\quad \eta_p:=\sup_{0\neq v\in H_{k}^{p-1}}\inf_{w_h\in V_{\mc{T}_k}^p}\frac{ \|R_k^*v-w_h\|_{H_k^1}}{\|v\|_{H_k^{p-1}}}.
\end{gathered}
\end{equation}
 By duality,~\eqref{e:preasymptoticSketch1} implies
% $$
% \|u-u_h\|_{H_k^{-p+1}}\leq C\eta_p \Big(\inf_{w_h\in V_{\mc{T}_k}^p}\|u-w_h\|_{H_k^1}+(hk)^p\|u-u_h\|_{H_k^{-p+1}}\Big),
% $$
%so that
\beq \label{e:theFirstSystemSimple}
\begin{gathered}
 \|u-u_h\|_{H_k^{-p+1}}\leq C \Big( b\inf_{w_h\in V_{\mc{T}_k}^p}\|u-w_h\|_{H_k^1}+\omega\|u-u_h\|_{H_k^{-p+1}}\Big),\\
 b:=\eta_p,\qquad \omega:=(hk)^{p}\eta_p.
 \end{gathered}
\eeq 
By a frequency splitting argument similar to that in Lemma~\ref{l:bound_eta_microlocal} below and the fact that $\rho\geq ck$,
 $$
 \eta_p \leq C(hk)^p\big(1+\|R_k\|_{L^2\to L^2}\big)= C(hk)^p(1+\rho)\leq C(hk)^p\rho.
 $$
% (proved using a frequency splitting argument similar to that in Lemma~\ref{l:bound_eta_microlocal} below and the fact that $\rho\geq ck$), implies
Thus, from \eqref{e:theFirstSystemSimple}, when $(hk)^{2p}\rho$ is sufficiently small, $(1-C\omega)^{-1}$ exists and is positive, and then
 \begin{equation}
 \label{e:martinIsHappy}
 \|u-u_h\|_{H_k^{-p+1}}\leq C(1-C\omega)^{-1}b\inf_{w_h\in V_{\mc{T}_k}^p}\|u-w_h\|_{H_k^1}\leq C(hk)^p\rho\inf_{w_h\in V_{\mc{T}_k}^p}\|u-w_h\|_{H_k^1},
 \end{equation}
which is the preasymptotic estimate~\eqref{e:preasymptotic1} for $m=p$.

\paragraph{\blue{Sketch of the %localised elliptic-projection 
argument used to prove Theorems \ref{t:simple} and \ref{t:theRealDeal}}.}

We now sketch the localised version of the above argument, which is used to prove Theorems~\ref{t:simple} and \ref{t:theRealDeal}. In this sketch, we treat
$S_k$, 
\blue{$(P^\sharp_k)^{-1}$}, and 
$\Pi_k^\sharp$ as though they are local; i.e., for $\chi,\psi\in C^\infty(\overline{\Omega})$ with $\supp \chi \cap \supp \psi=\emptyset$, we neglect the terms
$$
\chi S_k\psi, \quad 
\blue{\chi(P_k^\sharp)^{-1}\psi,}
\,\text{ and }\,
\chi\Pi_k^\sharp\psi.
$$
In general these terms are nonzero, but Sections~\ref{sec:pseudolocS} to~\ref{sec:pseudoLocPi}, which contain the bulk of the technical work of this paper, show that they are $O(k^{-\infty})$  and smoothing \blue{(this is discussed further in \S\ref{s:whyAmISOOOOOLONg})}. 
% In general these terms are nonzero, but Sections~\ref{sec:pseudolocS} to~\ref{sec:pseudoLocPi} show that they are $O(k^{-\infty})$ and smoothing. \blueOld{These sections contain the bulk of the technical work of the paper.
% We define $S_k$ below as a functionof a self-adjoint elliptic $S_k$, such pseudolocality results in the abs
% pseudolocality of $S_k$ the main difficulty is in establishing a framework that applies in settings with boundaries. 
% For $\Pi^\sharp_k$,  
% }
% pseudolocality of $S$ -- sort of folk-lore but implementation in asetting amenable to boundary value problems is non-standard
% pseudolocality of elliptic projection -- completely new 
Using these properties, Section~\ref{sec:proof_realDeal} shows that these terms only contribute to the remainder term in Theorem~\ref{t:simple}.%Section~\
%ref{sec:proof_realDeal}.
% \blue{In fact, the proof of the bound \eqref{e:lowNormPiSharp} also uses that $(P^\sharp_k)^{-1}$ is pseudolocal, with this also established in Sections~\ref{sec:pseudolocS} to~\ref{sec:pseudoLocPi}.}

To localise the elliptic-projection argument, we introduce an open cover of $\Omega$,  $\{\Omega_j\}_{j=1}^\domainnumber$ and $\{\phi_j\}_{j=1}^\domainnumber\subset C^\infty(\overline{\Omega})$ a partition of unity subordinate to this cover.
(In Theorem~\ref{t:simple}, $\domainnumber=4$ and $(\Omega_1,\Omega_2,\Omega_3,\Omega_4):=(\Omega_\cavity,\Omega_\visible,\Omega_\invisible,\Omega_\pml)$.)
Next, let $\chi_j\in C^\infty(\overline{\Omega})$, $j=1,\dots,\domainnumber$, such that 
$$
\supp \chi_j \subset \Omega_j\cup\partial\Omega,\qquad \chi_j\equiv 1\text{ in a neighbourhood of }\supp \phi_j.
$$

\blue{The goal now is to use an elliptic-projection type argument to control $\chi_i (u-u_h)$ in terms of local best approximation errors and $(\chi_1(u-u_h),\dots, \chi_M(u-u_h))$, specifically, to obtain the system of inequalities~\eqref{e:theFirstSystem}.}
Arguing as in~\eqref{e:basicEllipticProjection}, we obtain that, for all $w_{h,j}\in \blueOld{V_{\mathcal{T}}}$, $j=1,\ldots,\domainnumber$, and $v\in H_k^{p-1}$, 
\begin{align}\nonumber
&\langle \chi_i(u-u_h) ,v\rangle \\ 
\nonumber 
&= \big\langle P_k(u-u_h) , R_k^* \chi_i v \big\rangle\\ \nonumber
&= \big\langle P^\sharp_k(u-u_h), (I-\Pi^\sharp_k) R_k^* \chi_iv \big\rangle - 
\big\langle S_k(u-u_h), (I-\Pi^\sharp_k) R_k^* \chi_iv \big\rangle \\ 
&=\sum_{j=1}^{\domainnumber} \bigg(\big\langle P^\sharp_k(u-w_{h,j}), (I-\Pi^\sharp_k) \phi_j R_k^* \chi_iv\big\rangle - 
\big\langle S_k(u-u_h), (I-\Pi^\sharp_k) \phi_j R_k^* \chi_i v\big\rangle \bigg) \label{e:startDuality}
\\ 
&=\sum_{j=1}^{\domainnumber} \bigg(\big\langle P^\sharp_k\chi_j(u-w_{h,j}), \chi_j(I-\Pi^\sharp_k) \phi_j R_k^* \chi_iv\big\rangle - 
\big\langle S_k\chi_j(u-u_h), \chi_j(I-\Pi^\sharp_k) \phi_j R_k^* \chi_i v\big\rangle \bigg),\nonumber
\end{align}
where we have neglected the nonlocal parts of $\Pi_k^\sharp$, $P^\sharp_k$, and $S_k$ in the last line.
The local Aubin--Nitsche--type argument in Lemma~\ref{l:localSharp} shows that (modulo remainder terms, \blue{and using pseudolocality of $(P^\sharp_k)^{-1}$})
\begin{equation}
\label{e:localLowNormPiSharp}
\|\chi_j (\Id - \Pi_k^\sharp)v\|_{H^{-p+1}_k} \leq C(h_jk)^p\inf_{w_h \in V_{\mc{T}_k}^p} \|v - w_h\|_{H^1_k},\quad \text{ where } h_j:=\max_{\substack{\blueOld{T}\in\mc{T}\\\blueOld{T}\cap \Omega_j\neq\emptyset}}h_\blueOld{T}
\end{equation}
(compare to~\eqref{e:lowNormPiSharp}).
By %~\eqref{e:introPiSharpCea},
\eqref{e:startDuality},  \eqref{e:localLowNormPiSharp} and the mapping properties $S_k:H_k^{-p+1}\to H_k^{p-1}$ and $P_k^\sharp: H_k^1\to H_k^{-1}$,
\begin{equation}
\label{e:preasymptoticSketch2}
\begin{gathered}
\begin{aligned}
|\langle \chi_i(&u-u_h),v\rangle|\\
&\leq C\sum_j\eta_p(j\to i) \Big(\inf_{w_{h,j}\in V_{\mc{T}_k}^p}\|\chi_j(u-w_{h,j})\|_{H_k^1}+(h_jk)^p\|\chi_j(u-u_h)\|_{H_k^{-p+1}}\Big)\|v\|_{H_k^{p-1}},
\end{aligned}\\
\text{ where } \quad\eta_p(j\to i):=\sup_{0\neq v\in H_{k}^{p-1}}\inf_{w_h\in V_{\mc{T}_k}^p}\frac{ \|\chi_jR_k^*\chi_iv-w_h\|_{H_k^1}}{\|v\|_{H_k^{p-1}}}
\end{gathered}
\end{equation}
(compare to \eqref{e:preasymptoticSketch1}). 
 By duality,~\eqref{e:preasymptoticSketch2} implies
 $$
 \|\chi_i(u-u_h)\|_{H_k^{-p+1}}\leq \sum_j C\eta_p(j\to i)\Big(\inf_{w_{h,j}\in V_{\mc{T}_k}^p}\|\chi_j(u-w_{h,j})\|_{H_k^1}+(h_jk)^p\|\chi_j(u-u_h)\|_{H_k^{-p+1}}\Big).
 $$
 We then use a frequency splitting argument (see Lemma~\ref{l:bound_eta_microlocal}) to obtain (neglecting remainder terms)
 $$
 \eta_p(j\to i) \leq C(h_jk)^p\|\chi_jR_k^*\chi_i\|_{L^2\to L^2}+C1_{\{\Omega_i\cap \Omega_j\neq \emptyset\}}(h_{ij}k)^p,\quad \text{ where } h_{ij}:=\min (h_i,h_j).
 $$
% and hence
%  \begin{align*}
% \|\chi_i(u-u_h)\|_{H_k^{-p+1}}&\leq C\sum_j\big[(h_jk)^p\|\chi_jR_k^*\chi_i\|_{L^2\to L^2}+1_{\{\Omega_i\cap \Omega_j\neq \emptyset\}}(h_{ij}k)^p\big]\\
% &\qquad\qquad\big(\inf_{w_h\in V_{\mc{T}_k}^p}\|\chi_j(u-w_h)\|_{H_k^1}+(h_jk)^p\|\chi_j(u-u_h)\|_{H_k^{-p+1}}\big).
% \end{align*}
 In particular, this yields the system of inequalities
 \begin{equation}
 \label{e:theFirstSystem}
\big(\| \chi_i(u-u_h)\|_{H_k^{-p+1}}\big)_{i=1}^\domainnumber\leq CB \Big(\inf_{w_{h,j}\in V_{\mc{T}_k}^p}\|\chi_j(u-w_{h,j})\|_{H_k^1}\Big)_{j=1}^\domainnumber+C\oldT\big(\|\chi_j(u-u_h)\|_{H_k^{-p+1}}\big)_{j=1}^\domainnumber,
 \end{equation}
 where 
 \begin{align*}
 B_{ij}&:=\eta_p(j\to i)\leq C(h_jk)^p\|\chi_jR_k^*\chi_i\|_{L^2\to L^2}+C1_{\{\Omega_i\cap \Omega_j\neq \emptyset\}}(h_{ij}k)^p\\
 \oldT_{ij}&:=(h_jk)^p\eta_p(j\to i)\leq C(h_jk)^{2p}\|\chi_jR_k^*\chi_i\|_{L^2\to L^2}+C(h_jk)^p1_{\{\Omega_i\cap \Omega_j\neq \emptyset\}}(h_{ij}k)^p
% \big(\underline{v}\big)_i&:=\chi_iv.
 \end{align*}
 (compare to \eqref{e:theFirstSystemSimple}). 
 Under the condition that
\begin{equation}
\label{e:conditionW}
\sum_{n=0}^\infty (C W)^n<\infty,
 \end{equation}
 $(I-C\oldT)^{-1}$ exists and has non-negative entries. Hence~\eqref{e:theFirstSystem} implies that
 $$
\|  \underline{u-u_h}\|_{H_k^{-p+1}}\leq C(I-C\oldT)^{-1}B \|  \underline{u-w_h}\|_{H_k^1}.
 $$
 (compare to~\eqref{e:martinIsHappy}).

% \begin{remark}
%     It is important in the proof of Theorems~\ref{t:simple} and~\ref{t:theRealDeal} that, like~\eqref{e:theFirstSystemSimple}, the estimate~\eqref{e:theFirstSystem} can be used to estimate $\|u-u_h\|_{H_k^{-p+1}}$ directly i.e. without the need to first estimate $\|u-u_h\|_{H_k^1}$. This is in contrast with other FEM error analyses, including the classic ``Schatz duality arguement", \cite{Sc:74,GS3}, that 
%         previous analyses of the error ``Schatz duality argument"  
% \end{remark}
 To understand when $\sum_n (CW)^n$ converges, consider $W$ as the weighted adjacency matrix of a directed graph with $\domainnumber$ nodes representing $\{\Omega\}_{j=1}^{\domainnumber}$. Observe that the entry in the $i^{\text{th}}$ row and $j^{\text{th}}$ column of $(CW)^{\ell}$ is given by $C^\ell$ times the sum of the weights over all paths of length $\ell$ from $j$ to $i$ in this graph. Hence, the sum converges if for any $i$ and $j$ the sum of the weights of all paths from $j$ to $i$ multiplied by $C^{\text{path length}}$ is finite. Using elementary graph analysis this condition can be reduced to the requirement that all the sum of such weights for all non-self intersecting loops is less than 1 (see Appendix~\ref{sec:loops}).

In the setting of Theorem~\ref{t:simple}, $\domainnumber=4$ and $(\Omega_1,\Omega_2,\Omega_3,\Omega_4):=(\Omega_\cavity,\Omega_\visible,\Omega_\invisible,\Omega_\pml)$. For $k\notin \mc{J}$, Section~\ref{sec:boundsCsol} obtains the bounds of Table~\ref{ta:resolve} on $\psi R_k^*\chi$ according to the support of $\psi$ and $\chi$.
\begin{table}[H]
\begin{center}
\renewcommand{\arraystretch}{1.5}
\begin{tabular}{|c|c|c|c|c|}
\hline
$\supp \psi\Big\backslash \supp \chi$&$\Omega_{\cavity}$&$\Omega_{\visible}$&$\Omega_{\invisible}$&$\Omega_\pml$\\
\hline
$\Omega_{\cavity}$&$\rho$&$\sqrt{k\rho}$&$O(k^{-\infty})$&$O(k^{-\infty})$\\
\hline
$\Omega_{\visible}$&$\sqrt{k\rho}$&$k$&$k$&$1$\\
\hline
$\Omega_{\invisible}$&$O(k^{-\infty})$&$k$&$k$&$1$\\
\hline
$\Omega_{\pml}$&$O(k^{-\infty})$&$1$&$1$&$1$\\
\hline
\end{tabular}
\renewcommand{\arraystretch}{1}
\caption{\label{ta:resolve}Bounds on $\|\psi R_k^*\chi\|_{L^2\to L^2}$ (up to $k$-independent constants) 
for $k\notin \mathrm{J}$ 
given in Section~\ref{sec:boundsCsol}.
\blue{These bounds are proved in Appendix \ref{app:help} following their proof 
on manifolds without boundary in~\cite{DaVa:12,DaVa:12a}.}}
\end{center}
\end{table}
As a result, the graph corresponding to $W$ is the one in Figure~\ref{f:graph1}, and the requirement that the sum of weights on all non-self intersecting loops be less than 1 reduces to~\eqref{e:meshConditions}. 

\subsubsection{\blue{The technical difficulties in the proof of Theorem \ref{t:simple}}}
\label{s:whyAmISOOOOOLONg}

\blue{
In the sketch above we omitted several important but technical parts of the proof of Theorem~\ref{t:simple}.
% \begin{enumerate}[label=(\roman*)]
%     \item\label{i:psedoSk} Proving pseudolocality of $S_k$.% -- boundaries
%     \item\label{i:psedoPiSharp}  Proving pseudolocality of $\Pi^\sharp$.% -- non-standard interactions between piecewise-polynomial spaces and microlocal tools. 
% \item\label{i:frequencySplit}  Analysing the high- and low-frequencies of the Galerkin error separately.
%     \item\label{i:PML} Treating the PML region separately. \item\label{i:trackPseudolocal}  Tracking carefully the pseudolocal errors. 
% \end{enumerate}
    
\begin{enumerate}[label=(\Roman*),ref=\Roman*]
    \item\label{i:psedoSk} Proving pseudolocality of $S_k$ and $(P_k^\sharp)^{-1}$.% -- boundaries
    \item\label{i:psedoPiSharp}  Proving pseudolocality of $\Pi_k^\sharp$.% -- non-standard interactions between piecewise-polynomial spaces and microlocal tools. 
    \item\label{i:frequencySplit}  Analysing the high- and low-frequencies of the Galerkin error separately.
    \item\label{i:PML} Treating the PML region separately.
%\item\label{i:trackPseudolocal}  Tracking carefully the pseudolocal errors. 
\end{enumerate}

\noindent 
Regarding~\ref{i:psedoSk}:~ The operator $S_k$ is defined using the functional calculus of the self-adjoint operator $\cP_k:=\Re P_k$. For the proof of Theorem~\ref{t:simple}, we need to show pseudolocality of $S_k := \psi(\cP_k)$ with respect to both spatial and frequency cutoffs (see Sections~\ref{s:spatialPseudo} and~\ref{s:frequencyPseudolocal} for more precise statements). This type of property belongs to the realm of microlocal analysis, whose tools are most comfortable without boundaries (and become substantially more technical with boundaries). In fact, when there are no boundaries, it is classical that both $S_k$ and $(P_k^\sharp)^{-1}$ are semiclassical pseudodifferential operators, and in particular, pseudolocal (see e.g. \cite[Chapter 14]{Zw:12}). 
%The operator $S_k$ is defined in terms of the functional calculus for the self-adjoint operator, $\cP_k:=\Re P_k$. 
%When there are no boundaries, it is classical that $S_k$ is pseudolocal (in fact, it is a semiclassical pseudodifferential operator, see e.g. \cite[Chapter 14]{Zw:12}). 
%Here we prove pseudolocality of $S_k$ based purely on commutator properties of $\cP_k$ with cutoffs, however, the presence of boundaries this requires 
To deal with boundaries, we prove pseudolocality of $S_k$ and $(P_k^\sharp)^{-1}$ (see Theorem \ref{t:pseudoLocGeneral} below) using only commutator properties of $\cP_k$ with cutoffs.
 The technicality associated with this is that the action of $\cP_k$ removes boundary conditions, which restricts the class of operators upon which repeated commutators with $\cP_k$ are well defined, and hence requires careful tracking of function spaces and constructing cutoff functions with specific properties at boundaries.
 %The presence of boundaries requires careful consideration of the domains of powers of the operator $\cP_k$ --
%Although this type of argument is ``folklore" in the microlocal analysis literature, we are not aware of a reference where they are implemented in a setting abstract enough to handle our Helmholtz operators.
%The proof of pseudolocality of $S_k$ is then based purely on commutator properties of $\cP_k$ with cutoffs.  Although this type of argument is ``folklore" in the microlocal analysis literature, we are not aware of a reference where they are implemented in a setting abstract enough to handle our Helmholtz operators.

\bigskip 

\noindent 
Regarding~\ref{i:psedoPiSharp}: 
In broad terms, 
pseudolocality of $\Pi^\sharp_k$ is delicate because 
it involves the interaction between discrete objects (piecewise-polynomial functions on a mesh) with microlocal objects (spatial/frequency cutoffs). 
More concretely, our proof of pseudolocality of $\Pi^\sharp_k$ builds on the result that the Galerkin solution is locally quasi-optimal modulo low frequencies (as discussed around \eqref{e:localQOModulo}) -- this was famously proved in \cite{NiSc:74}, with a $k$-explicit version proved more recently in \cite{AvGaSp:24}. 
However, the pseudolocality required here is significantly stronger than the results of~\cite{AvGaSp:24} in that the error is not only low frequency, but is also super-polynomially small in $k$ (see Theorem \ref{t:pseudoLocalPi} below). Furthermore, the proof is much more involved because $\Pi_k^\sharp$  % It takes the form
% \begin{equation} 
% \label{e:piSharpIsTheBomb}
% 	\|\chi (\Id - \Pi_k^\sharp) \psi u\|_{H^1_k} \leq C k^{-N} \|(\Id - \Pi_k^\sharp) \psi u\|_{H_k^{-p}},
% 	\end{equation}
%We highlight the important difference between~\eqref{e:piSharpIsTheBomb} and existing theory. 
%and (ii) $\Pi_k^\sharp$ 
is a projection related to the non-local operator $P^\sharp_k:=P_k + S_k$ (instead of simply $P_k$ in \cite{AvGaSp:24}); controlling the contribution from the non-local term $S_k$ requires new arguments involving the pseudolocality of $S_k$ and dyadic decomposition of the computational domain.

% gains in $k$
% $S_k$-- `frequency localisation' breaks locality of the operator and requires dyadic thing (interaction between ``microlocal'' and discrete
% implicit definition 
% The operator $\Pi_k^\sharp$ 
% discrete object

% For the Galerkin projection itself, 
% \cite{NiSc:74}
% adapted to the $k$-explicit setting in \cite{AvGaSp:24}.
% Here definition of $\Pi_k^\sharp$ involves $S_k$ -- makes it harder.
%As discussed in \S\ref{sec:1.1}, 
%\cite{NiSc:74} proved local quasi-optimality of the Galerkin solution, with \cite{AvGaSp:24}

\bigskip

\noindent 
Regarding~\ref{i:frequencySplit}:
The most direct application of the sketch above to obtain bounds on the error in higher $H_k^{1-m}$ norms results in suboptimal error bounds (and mesh thresholds) -- specifically, some instances of  $(h_{\star}k)^{2p}$ are replaced by $(h_{\star}k)^{p+m}$. 
The reason for this is that this direct application of the argument ignores the better behaviour of the Helmholtz solution operator at high-frequency. The natural idea is to split the Galerkin error into low- and high- frequency components and to analyse these separately using localised duality arguments. In practice, we first estimate the high frequencies of the error by $H_k^{1-p}$ norms and a best approximation error (Lemma \ref{l:highFreqUpgrade} below) using an iteration scheme that is effective because of the good behaviour of the solution operator at high frequency (see Remark \ref{r:iteration}). We then implement a (more involved) analogue of the sketch above to estimate the $H_k^{1-p}$ norms (Lemma \ref{l:lastDay}).
%Our analysis of these components itself involves the localised duality argument sketched above but with  with an iteration required to conclude.  

%to upgrade the estimate in the sketch from $H_k^{1-p}$ to $H_k^1$ requires splitting off the high-frequencies

\bigskip

\noindent 
Regarding~\ref{i:PML}:~We stated in the discussion below \eqref{e:matrices} that $\mathcal{C}$ in \eqref{e:matrices} should be understood as norms of the localised data-to-solution map. However, according to Table~\ref{ta:resolve} this matrix should then have ones in the last row and column instead of zeros. These zeros in $\mathcal{C}$ are possible because, just as the solution operator behaves better on high frequencies than on low frequencies, it also behaves better in the PML region. 
Moreover, the pseudolocality of the solution operator in the PML (see Theorem \ref{thm:PML})
% In the PML, it is important that the solution operator does not just have a smaller norm, but is also pseudolocal. This 
means that any communication between $\Omega_{\Pml}$ and $\Omega_{\star}$ happens in the overlap, $\Omega_{\star}\cap\Omega_{\Pml}$, where \emph{both} the mesh size $h_{T}\ll k^{-1}$ \emph{and} the solution operator is uniformly bounded in $k$. 
Arguing in a similar way as for the high frequencies (compare Lemmas \ref{l:highFreqUpgrade} and \ref{l:PMLUpgrade} below), these features lead to 
% Improvements in the PML region can be obtained in a way analogous way to those at high frequencies (compare Lemmas \ref{l:highFreqUpgrade} and \ref{l:PMLUpgrade} below, and see Remark \ref{r:iteration}). One important difference is that the solution operator is well behaved at all frequencies in the PML region and hence, up to small errors, 
the Galerkin error in the PML being locally quasioptimal and hence to the zeros in the last row and column of both $\mathcal{C}$ and $\transferIntro$.

%These improvements are required to obtain the zero entries in the last row and column of the matrix $\transferIntro$ \eqref{e:matrices}.
%We note that this improved behaviour of the solution operator in the PML region is proved for the first time 
%However, taking advantage of this improvement in the same way as for the improvement on high-frequencies does not  

}

\bigskip
\noindent \textsc{Acknowledgements:}  MA was supported by EPSRC grant EP/R005591/1, JG was supported by EPSRC grants EP/V001760/1 and EP/V051636/1, Leverhulme Research Project Grant RPG-2023-325, and ERC Synergy Grant PSINumScat - 101167139, and EAS was supported by EPSRC grant EP/R005591/1 and ERC Synergy Grant PSINumScat - 101167139.
\blue{We thank the referees for their comments, which helped improve the exposition of the paper.}

\section{Numerical experiments illustrating the main result}%Theorem \ref{t:simple}}
\label{s:numerical}

We illustrate Theorem \ref{t:simple} with numerical results in a selection of asymptotic regimes and in two different geometric settings, in which we solve the PDE \eqref{e:edp} with constant coefficients $A,n \equiv 1$.
\subsection{Experimental setup}

\subsubsection{Geometric setup}\label{s:geometries}
  The first geometric setting involves a scatterer with two parallel ``walls" obtained by placing two rectangular obstacles next to each other \blueOld{(as in Figure \ref{f:twoMirrors})}. The second geometric setting is similar, but has one of the two rectangles shifted slightly upwards to ``make way" for the wave to come inside the cavity; \blueOld{both of these settings are shown in Figure \ref{fig:showNumericalSetting}.} 

In the first setting (without shifting one of the two rectangles) the obstacle $\Omega_-$ is the union of two congruent rectangles of sides $L_1 = 0.7\sqrt{2}$ and $L_2=1.3\sqrt{2}$ with rounded corners so that they have a $C^\infty$ boundary (this is done using the technique from \cite{epstein2016smoothed}). The four vertices of the first (respectively second) rectangle are located at the coordinates $[-X_1/2 \pm \frac{L_1}{2},\pm \frac{L_2}{2}]$ (respectively, $[X_1/2 \pm \frac{L_1}{2},\pm \frac{L_2}{2}]$) where $X_1 = 3\frac{\sqrt{2}}{2}$. Therefore, the two rectangles have parallel sides and are separated by a gap in the $x$-axis equal to $L_{\rm gap} = X_1 - L_1 = 0.8\sqrt{2}$. \blueOld{The scatterer is surrounded by a PML layer with coefficients defined in \eqref{e:firstPML}, with the PML scaling function given by $f_\theta(r) = (r-\RPMLo)^3/(3(\Rtr - \RPMLo)^2)$ for $r > \RPMLo$, with $\RPMLo = 2.2$ and $R_\tr = 2.7$.} 

\blueOld{Set $\delta := \frac{\sqrt{2}}{8}$ Then, we take 
$$
\begin{aligned}
\Omega_\cavity &:= \overline{\Omega}_+ \cap \Big((-\tfrac{L_{\rm gap}}{2}-\delta,\tfrac{L_{\rm gap}}{2}+\delta) \times (-\tfrac{L_2}{2},\tfrac{L_2}{2})\Big)\\[1ex]
\blueOld{\Omega_\visible}& := \Big\{(x,y) \in \overline{\Omega}_+\,:\, \abs{y} > \tfrac{L_2}{2}-\delta\Big\} \blueOld{\cap \Big\{(x,y) \in \overline{\Omega_+} \,:\, x^2 + y^2 < 1.05\RPMLo+0.1\Big\}},\\[1ex]
\Omega_\invisible& := \Big\{(x,y) \in \overline{\Omega}_+ \,:\, \abs{x}> \tfrac{L_{\rm gap}}{2} +\delta  \Big\} \blueOld{\cap \Big\{(x,y) \in \overline{\Omega}_+ \,:\, x^2 + y^2 < 1.05\RPMLo+0.1\Big\}},\\[1ex]
\Omega_\pml &:= \{(x,y) \in \overline{\Omega}_+\,:\,x^2 + y^2 > 1.05 \RPMLo\}.
\end{aligned}
$$
%(e.g. smaller than, $\min(L_1,L_2/2)$)

For these geometries, and since the wave speed is constant, we can identify $\cavity$ and find a set containing $\visible$ ``by eye''. For more complicated geometries and wave speeds, one would need to identify these regions using ray tracing (see Figure~\ref{f:trapped}), \blueOld{as discussed in \S\ref{sec:1.1} after Table \ref{tab:regimes}.}}

%We use the radial PML with coefficients defined in \eqref{e:firstPML}, with the PML scaling function given by 
 %$f_\theta(r) = (r-R_{\pml})^3 /(3 (R_{\tr}-R_{\pml})^2)$ for $r>R_{\pml}$, with $R_{\pml} =2.2$ and $R_\tr =2.7$. 

\subsubsection{Discussion of $\rho(k)$ in the experimental setup}

For the wavenumbers 
\begin{equation*}
%	\label{e:badk}
	k_n := \frac{n\pi}{L_{\rm gap}},
\end{equation*}
one can show that there exists $c > 0$ such that  $\rho(k_n) \geq ck_n^2$ (e.g. by considering \eqref{e:edp} with the right hand side $f$ obtained by applying the operator $-k^{-2} \Delta - 1$ to $u(x,y) := \chi(x,y) \sin \Big[ k_n \left(x-\frac{L_{\rm gap}}{2}\right)\Big]$ where $\chi$ is a smooth compactly supported function which is identically $1$ in the set $[-\frac{L_{\rm gap}}{2},\frac{L_{\rm gap}}{2}] \times [-\varepsilon,\varepsilon]$ for some sufficiently small $\varepsilon>0$. The best known upper bound for all $k \in \R_+$ is $\rho(k) \leq C k^3$ for all $k \geq k_0$ \cite{ChSpGiSm:20}, but it is conjectured that $\rho(k) \leq Ck^2$, and we assume this from now on.

\subsubsection{Description of the sources}
For these two geometries, we consider $k$-dependent right-hand sides (source terms) $f_{\rm in}$ and $f_{\rm out}$ that are ``Gaussian beams" (or ``wave-packets") of width $k^{-1/2}$ both in the physical and frequency space. \blueOld{Namely, given $x_0,\xi_0 \in \R^2$ with $|\xi_0|=1$, we put 
$$f_{x_0,\xi_0}(x) := Ck^{1/4}\chi(x-x_0) e^{-k[(x-x_0) \cdot \xi_0^\perp]^2} e^{i k x\cdot \xi_0}$$
where $\xi_0^\perp$ is a unit vector orthogonal to $\xi_0$ and $\chi$ is the following bump function
$$\chi(x) := \begin{cases}
    e^{-\frac12|x|^2/(r^2 -|x|^2)} & \textup{for } |x| < r\\
    0 & \textup{otherwise.}
\end{cases}$$
with $r := 0.4$, and $C$ is such that $\|f_{x_0,\xi_0}\|_{L^2(\R^2)} = 1$.
} 
\blueOld{For the first geometry described in \S\ref{s:geometries}, we use a beam centered inside the cavity, and propagating in the $x$-direction, namely
$f_{\rm in} := f_{(0,0),(1,0)}$.
For the second geometry, we use a beam centered outside, and aimed at the cavity. More precisely, we choose
$f_{\rm out} := f_{x_0(k),\xi_0(k)}$,
where $\xi_0(k) = (\cos(\frac{3}{\sqrt{k}}),\sin(\frac{3}{\sqrt{k}}))$ and $x_0(k)$ is chosen such that the line $x_0(k) + t \xi_0(k)$ hits a point near the bottom left of the right wall 
}
%centered in physical space either at the origin $(0,0)$, i.e. {\em inside} the cavity $\cavity$, propagating in the $x$-direction, or {\em outside} the cavity, propagating in the direction of angle $\theta(k) = O(k^{-1/2})$ with respect to the $x$-axis. For the outside beam, the physical position $(x_0,y_0)$ is chosen so that the central ray of the beam hits the bottom of the right-hand ``wall" of the cavity 
(this ensures that the beam can coherently stay in the cavity as long as possible, with $O(\sqrt{k})$ reflections on the cavity's boundaries), see Figure \ref{fig:showNumericalSetting} (c). 
%The beams are normalized so that $\norm{f_{\rm in}}_{L^2(\R^2)}, \norm{f_{\rm out}}_{L^2(\R^2)}= 1$. 
The obstacle, the right-hand sides $f_{\rm in}$ and $f_{\rm out}$, and the corresponding solutions $u_{\rm in}$ and $u_{\rm out}$ are represented in Figure \ref{fig:showNumericalSetting}. 

\begin{figure}[h]
	\centering
	\begin{subfigure}{0.48\textwidth}
		\includegraphics[width=0.95\textwidth]{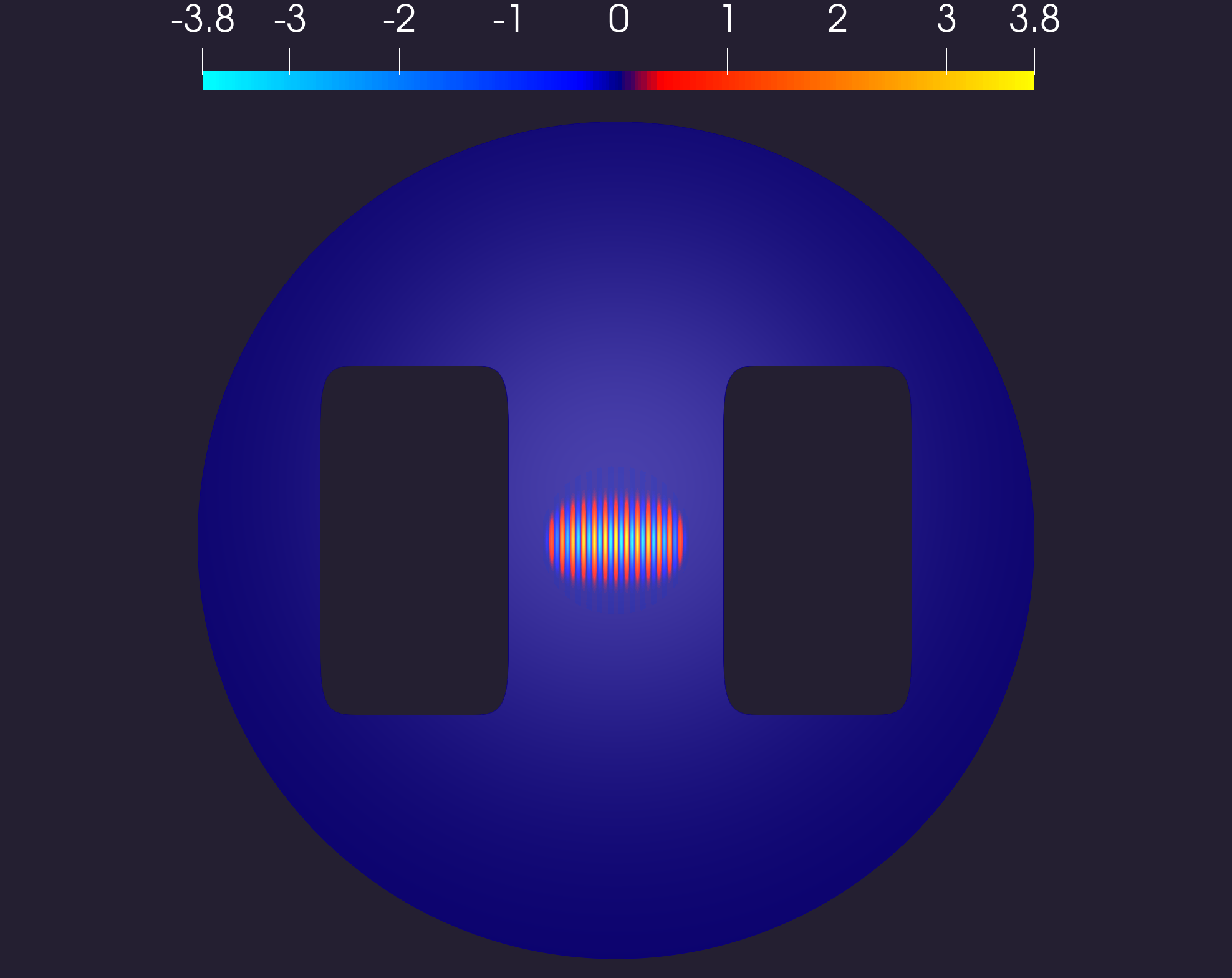}
		\caption{Data $f_{\rm in}$}
	\end{subfigure}
	\hfill
	\begin{subfigure}{0.48\textwidth}
		\includegraphics[width=0.95\textwidth]{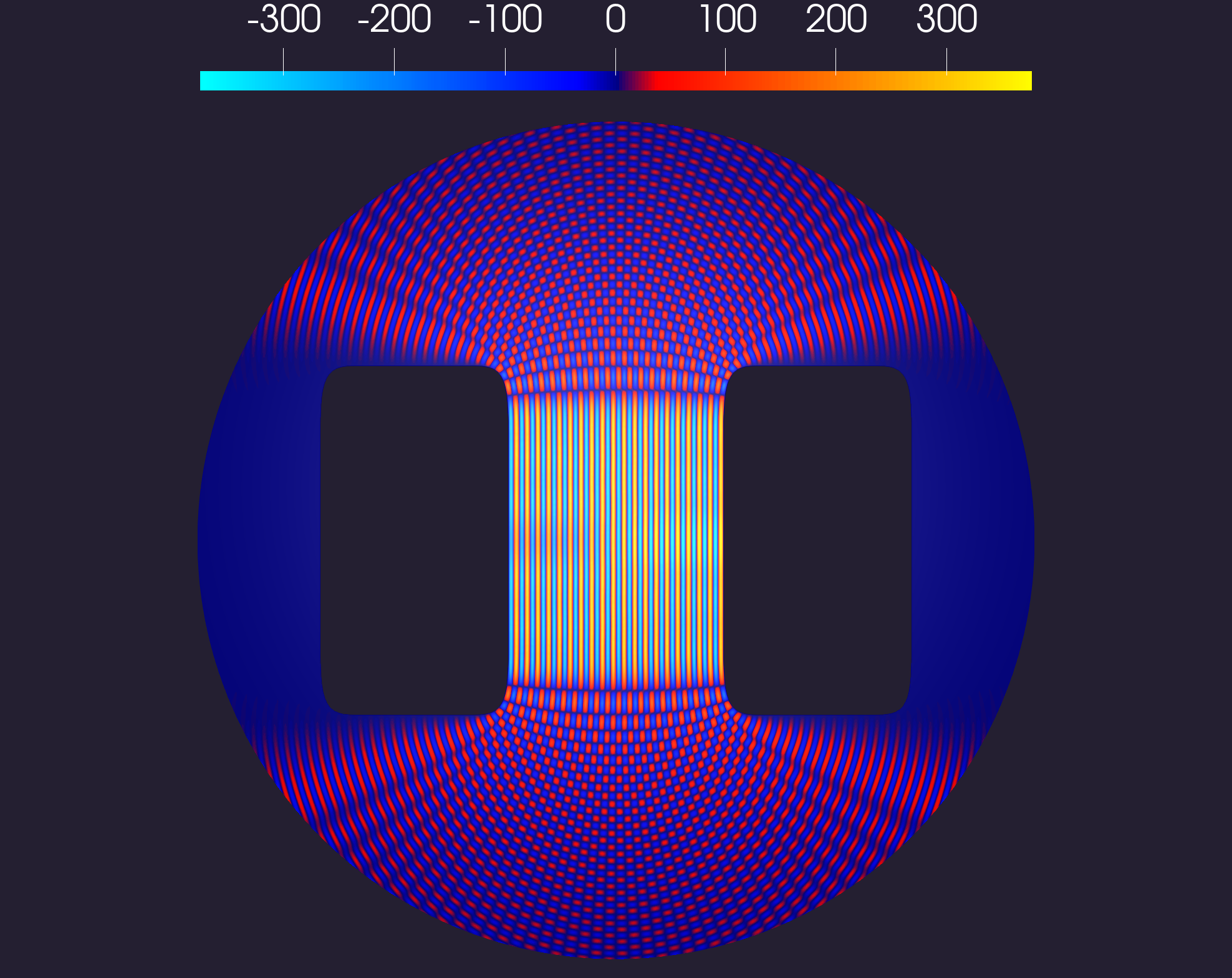}
		\caption{Solution $\textup{Re}(u_{\rm in})$}
	\end{subfigure}\\[1em]
	\begin{subfigure}{0.48\textwidth}
		\includegraphics[width=0.95\textwidth]{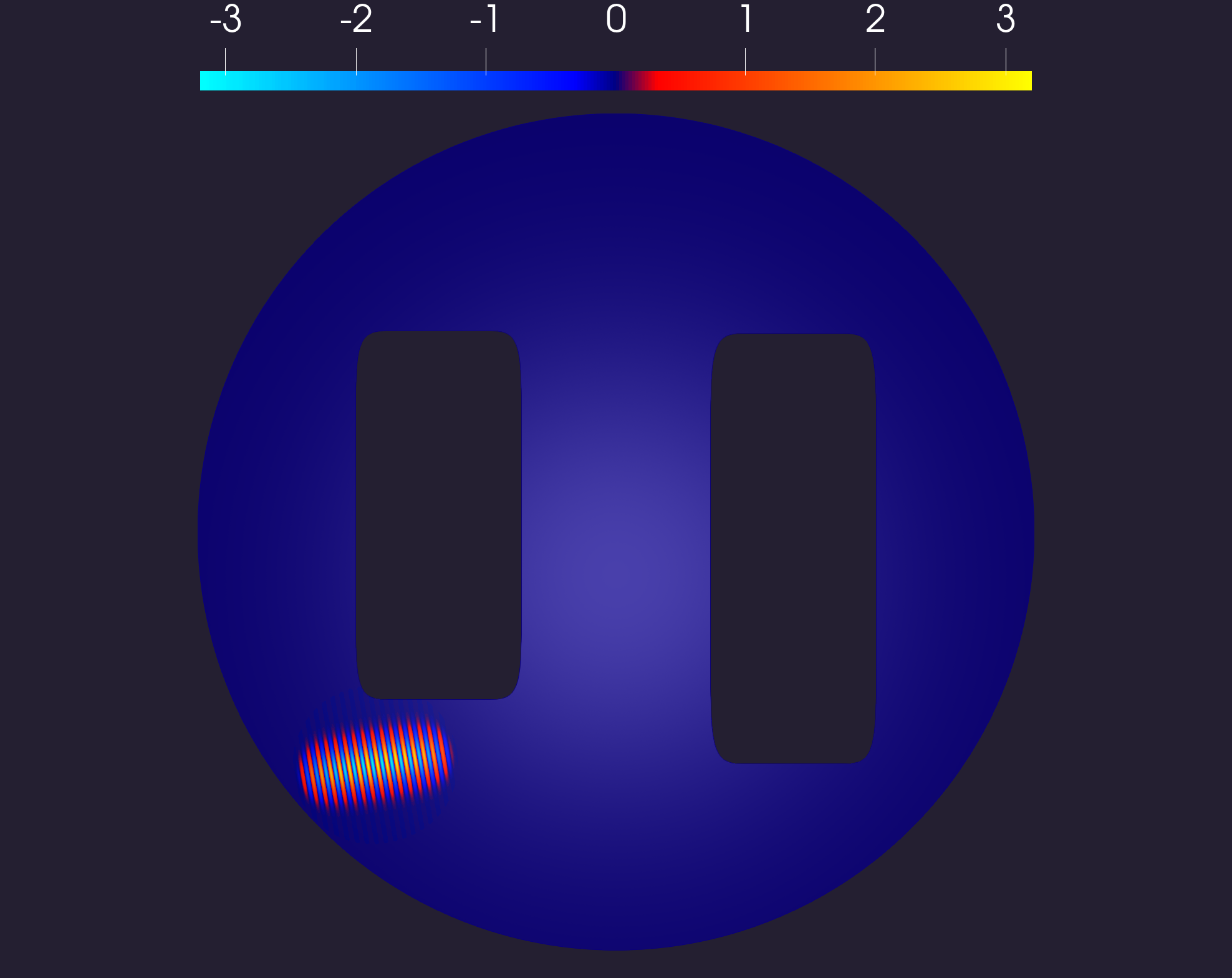}
		\caption{Data $f_{\rm out}$}
	\end{subfigure}
	\hfill
	\begin{subfigure}{0.48\textwidth}
		\includegraphics[width=0.95\textwidth]{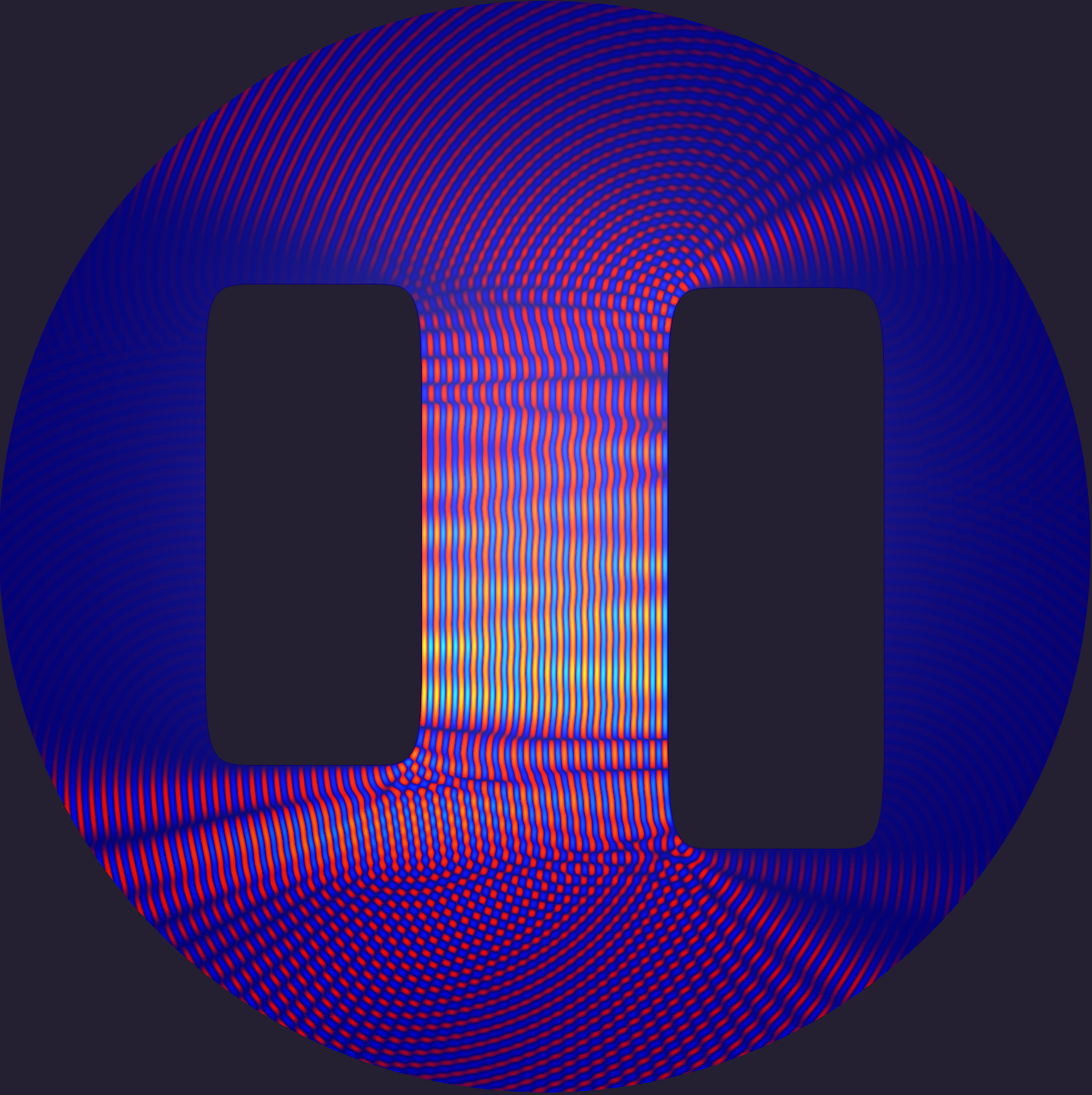}
		\caption{Solution $\textup{Re}(u_{\rm out})$}
	\end{subfigure}
	\caption{Top left: right-hand side $f_{\rm in}$. Top right: numerical approximation of the solution $u_{\rm in}$ with data $f_{\rm in}$. Bottom left: right-hand side $f_{\rm out}$. Bottom right: numerical approximation of the solution $u_{\rm out}$ with data $f_{\rm out}$. In these figures, $k = 40 \frac{\pi}{L_{\rm gap}} \approx 111$ and the functions $f$ and $u$ are truncated to a domain $\{x:|x|< R\} \cap \Omega_+$ where $R = 2.2$.}
		\label{fig:showNumericalSetting}
\end{figure}

\subsubsection{Reference solutions and their $k$-dependence}
Numerical approximations of the exact solutions $u_{\rm in}$ and $u_{\rm out}$ are computed using the FEM with piecewise polynomials of degree $p_{\rm ref}=4$. These numerical solutions are used as reference solutions to analyze the error in the FEM with $p=2$ throughout numerical experiments.  

By Theorem \ref{thm:DV}, there exists a constant $C > 0$ and, for each $N > 0$, a constant $C_N >0$ such that $\norm{u_{\rm in}}_{H^1_k(\Omega_\visible)} \leq C \sqrt{k \rho}$, $\norm{u_{\rm in}}_{H^1_k(\Omega_\invisible)} \leq C_N k^{-N}$, $\norm{u_{\rm out}}_{H^1_k(\Omega_\cavity)} \leq C \sqrt{k \rho}$ and $\norm{u_{\rm out}}_{H^1_k(\Omega_\visible \cup \Omega_\invisible)} \leq Ck$  for all $k \geq k_0$. This is illustrated in Figure \ref{fig:numericalResolvent}, where we observe the empirical rates $\norm{u_{\rm in}}_{H^1_k(\Omega_\cavity)} \approx C k^{1.7} \leq C k^2$, $\norm{u_{\rm in}}_{H^1_k(\Omega_\visible)} \approx C k^{1.2} \leq C k^{3/2}$, and  $\norm{u_{\rm out}}_{H^1_k(\Omega_\cavity)} \approx C k^{1.4} \leq C k^{3/2}$, $\norm{u_{\rm out}}_{H^1_k(\Omega_\visible)} \approx C k^{0.75} \leq C k$. 
The regimes that we consider are those of Table \ref{tab:regimes}.

\begin{figure}[htbp]
	\centering
	\begin{subfigure}{0.48\textwidth}
	\begin{tikzpicture}
		\draw node at (0,0){\includegraphics[width=0.95\textwidth]{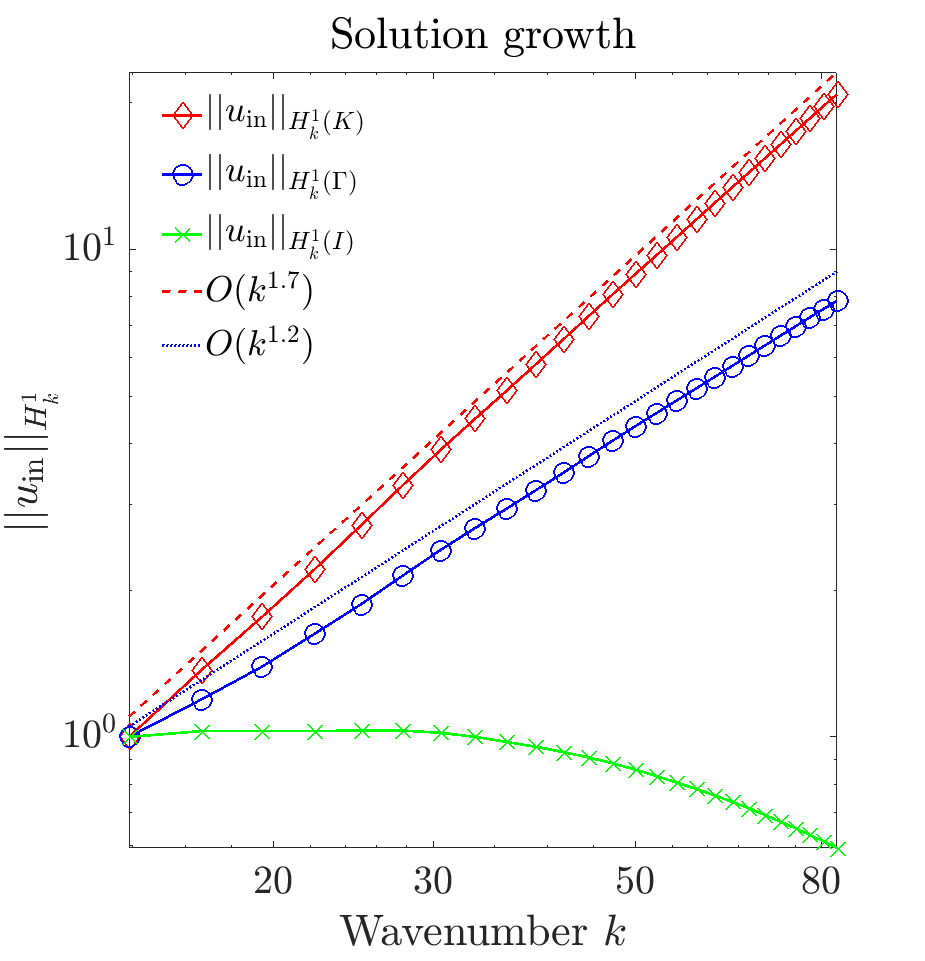}};
		\fill[color=white] (-1.9,.8) rectangle (-.5,3);
		\draw (-2.0,2.65) node[right]{\footnotesize{$\|u_{\rm in}\|_{_{H_k^1(\cavity)}}$}};
		\draw (-2.0,2.2) node[right]{\footnotesize{$\|u_{\rm in}\|_{_{H_k^1(\visible)}}$}};
		\draw (-2.0,1.75) node[right]{\footnotesize{$\|u_{\rm in}\|_{_{H_k^1(\invisible)}}$}};
		\draw (-2.0,1.35) node[right]{\footnotesize{$O(k^{1.7})$}};
		\draw (-2.0,.95) node[right]{\footnotesize{$O(k^{1.2})$}};
		\fill[color=white] (-2,-4)rectangle (2.1,-3.2);
		\draw node at (.125,-3.3){Wavenumber $k$};
		\fill[color=white] (-3.4,-1) rectangle (-2.9,4);
		\draw node[rotate=90] at (-3.2,.3){$\|u_{\rm in }\|_{H_k^1}$};
		\fill[color=white] (-2,3.1) rectangle (2,3.5);
		\draw node at (.125,3.3){Solution Growth};
		\end{tikzpicture}
	\end{subfigure}
	\hfill
	\begin{subfigure}{0.48\textwidth}
		\begin{tikzpicture}
		\draw node at (0,0){\includegraphics[width=0.95\textwidth]{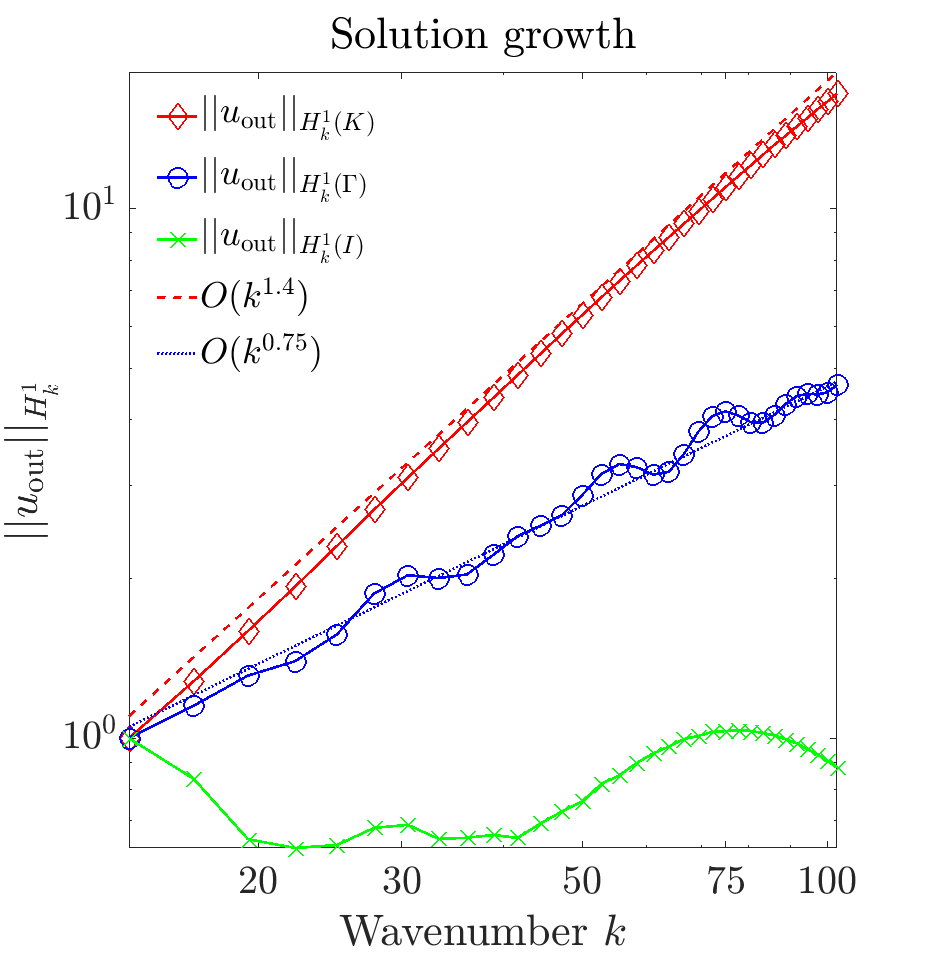}};
		\fill[color=white] (-2,.8) rectangle (-.5,3);
		\draw (-2.0,2.65) node[right]{\footnotesize{$\|u_{\rm out}\|_{_{H_k^1(\cavity)}}$}};
		\draw (-2.0,2.2) node[right]{\footnotesize{$\|u_{\rm out}\|_{_{H_k^1(\visible)}}$}};
		\draw (-2.0,1.75) node[right]{\footnotesize{$\|u_{\rm out}\|_{_{H_k^1(\invisible)}}$}};
		\draw (-2.0,1.35) node[right]{\footnotesize{$O(k^{1.4})$}};
		\draw (-2.0,.95) node[right]{\footnotesize{$O(k^{0.75})$}};
		\fill[color=white] (-2,-4)rectangle (2.1,-3.2);
		\draw node at (.125,-3.3){Wavenumber $k$};
		\fill[color=white] (-3.4,-1) rectangle (-2.9,4);
		\draw node[rotate=90] at (-3.2,.3){$\|u_{\rm out }\|_{H_k^1}$};
		\fill[color=white] (-2,3.1) rectangle (2,3.5);
		\draw node at (.125,3.3){Solution Growth};
		\end{tikzpicture}

	\end{subfigure}
	\caption{Left: growth of the solution $u_{\rm in}$. Right: growth of the solution $u_{\rm out}$. Solid red line (resp. blue, yellow): growth in the cavity (resp. the visible set, the invisible set).}
	\label{fig:numericalResolvent}
\end{figure}
\begin{table}[htbp]
	\centering
	\begin{tabular}{|c|c|c|}
		\hline
		Region & $\|u_{\rm in}\|_{H^1_k}$ & $\|u_{\rm out}\|_{H^1_k}$ 
		\\\hline&&\\[-0.8em]
		$\Omega_\cavity$ & $\approx k^{1.7}$ & $\approx k^{1.4}$ 
		\\\hline&&\\[-0.8em]
		$\Omega_\visible$ & $\approx k^{1.2}$ & $\approx k^{0.75}$\\\hline
	\end{tabular}
	\caption{Bounds on the $H^1_k$ norms of $u_{\rm in}$ and $u_{\rm out}$ inferred from Figure \ref{fig:numericalResolvent}.}
	\label{table:inferred}
\end{table}

\subsubsection{\blueOld{Meshing of the domain}}\label{s:meshing}

\blueOld{In the numerical experiments, the geometries described in \S\ref{s:geometries} are approximated by (non-uniform) simplicial meshes. Strictly speaking this setup is not covered by Theorem \ref{t:simple}, which requires (curved) meshes that fit the geometry exactly. However, the numerical experiments give quantitative agreement with the results of Theorem \ref{t:simple}, and we expect that the geometric error is small compared to the overall Galerkin error; see Remark \ref{r:geometricError} below.}

The non-uniform meshes used in the experiments are created using a feature of FreeFem++ allowing one to ``adapt" a mesh according to a custom  metric. For our purposes, we only require an isotropic metric, which is described by a scalar function $h: \Omega \to \R_+$ describing the local required mesh size. This function $h$ can be passed -- along with an initial, uniform mesh -- as an optional argument to the FreeFem++ ``adaptMesh" routine, which uses the BAMG algorithm \cite{hecht1998bamg}. We define $h \in C^\infty(\overline{\Omega})$ so that
\[\max_{x \in \Omega_\star} h(x)\leq h_\star,\]
where $\star\in \{\cavity, \visible, \invisible, \pml\}$, with $h_\star$ the corresponding mesh threshold. In some parts of $\Omega_{\star}$, the function $h$ can be significantly smaller than $h_\star$, for instance in intersections between two subdomains. 
 However, we enforce that $h(x) \equiv h_\star$ for all $x$ in a $k$-independent subset $\Omega_{\star}' \subset \Omega_{\star}$. 
Therefore, up to smooth transitions across regions, the metric is sharply described by $h_\star$.
In all the experiments, we take $h_\pml k$ to be constant. Since the solution in the PML region is not physically relevant, we do not display the errors in this region. 
 
%In all the experiments, we take $h_\pml k$ to be constant. Since the solution in the PML region is not physically relevant, we do not display the errors in this region. 

\subsection{Numerical results}
In the numerical results, we compute a few important quantities under a variety of mesh conditions. The \emph{local quasioptimality (QO) constants} for $u_{\rm in/\rm out}$ are given by
\[\norm{u_{\rm in/out}-u_h}_{H^1_k(\Omega_\star)}\Big/\norm{u_{\rm in/out} - w_h}_{H^1_k(\Omega_\star)},\qquad \star\in\{ \cavity,\visible,\invisible\},\]
where  $u_h$ is the Galerkin solution and $w_h$ is the best approximation of $u_{\rm in / out}$ in the finite-element space.
The \emph{local-global relative error} is the Galerkin error in the $H^1_k$ norm in these regions, normalized by the global $H^1_k$ norm of the solution is given by
\[\norm{u_{\rm in/out}-u_h}_{H^1_k(\Omega_\star)}\Big/\norm{u_{\rm in/out}}_{H^1_k(\Omega)},\qquad \star\in\{\cavity,\visible\},\]

\subsubsection{Regime Uniform 1 (U1)}\label{sec:expU1}

The first numerical experiment uses the uniform mesh guaranteeing $k$-uniform quasioptimality. We choose
\[(h_\cavity k)^p k^2 = (h_\visible k)^p k^2 = (h_\invisible k)^p k^2 =: (hk)^p k^2 = C\]
where $C$ is independent of $k$. Figure \ref{fig:U1QO} plots the local QO constants and Figure \ref{fig:U1} plots the local-global relative errors.

\begin{figure}[H]
	\centering
	\begin{tikzpicture}
	\draw node at (0,0){\includegraphics[width=0.45\textwidth]{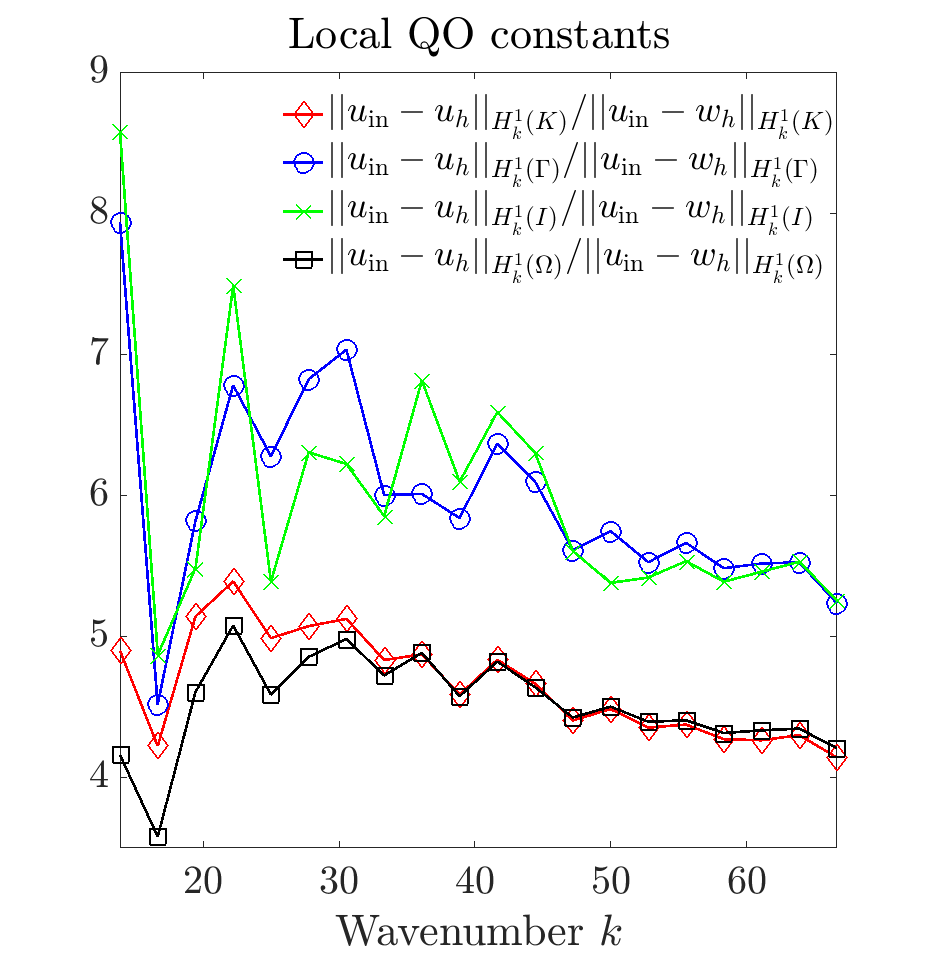}};
			\fill[color=white] (-1.025,1.25) rectangle (2.7,2.875);
		\draw (-1.125,2.6) node[right]{\tiny{$\|\!u_{\rm in}\!-\!u_h\!\|_{H_k^1(\cavity)}\!/\!\|\!u_{\rm in}\!-\!w_h\!\|_{H^1_{k}(\cavity)}$}};
		\draw (-1.125,2.25) node[right]{\tiny{$\|\!u_{\rm in}\!-\!u_h\!\|_{H_k^1(\visible)}\!/\!\|\!u_{\rm in}\!-\!w_h\!\|_{H^1_{k}(\visible)}$}};
		\draw (-1.125,1.9) node[right]{\tiny{$\|\!u_{\rm in}\!-\!u_h\!\|_{H_k^1(\invisible)}\!/\!\|\!u_{\rm in}\!-\!w_h\!\|_{H^1_{k}(\invisible)}$}};
		\draw (-1.125,1.5) node[right]{\tiny{$\|\!u_{\rm in}\!-\!u_h\!\|_{H_k^1(\Omega)}\!/\!\|\!u_{\rm in}\!-\!w_h\!\|_{H^1_{k}(\Omega)}$}};
%		\draw (-2.0,.95) node[right]{\footnotesize{$O(k^{1.2})$}};
		\fill[color=white] (-2,-4)rectangle (2.1,-3.1);
		\draw node at (.125,-3.3){Wavenumber $k$};
%		\fill[color=white] (-3.4,-1) rectangle (-2.9,4);
%		\draw node[rotate=90] at (-3.2,.3){$\|u_{\rm in }\|_{H_k^1}$};
		\fill[color=white] (-2,3.1) rectangle (2,3.5);
		\draw node at (.125,3.3){Local QO Constants};
	\end{tikzpicture}
	\begin{tikzpicture}
	\draw node at (0,0){\includegraphics[width=0.45\textwidth]{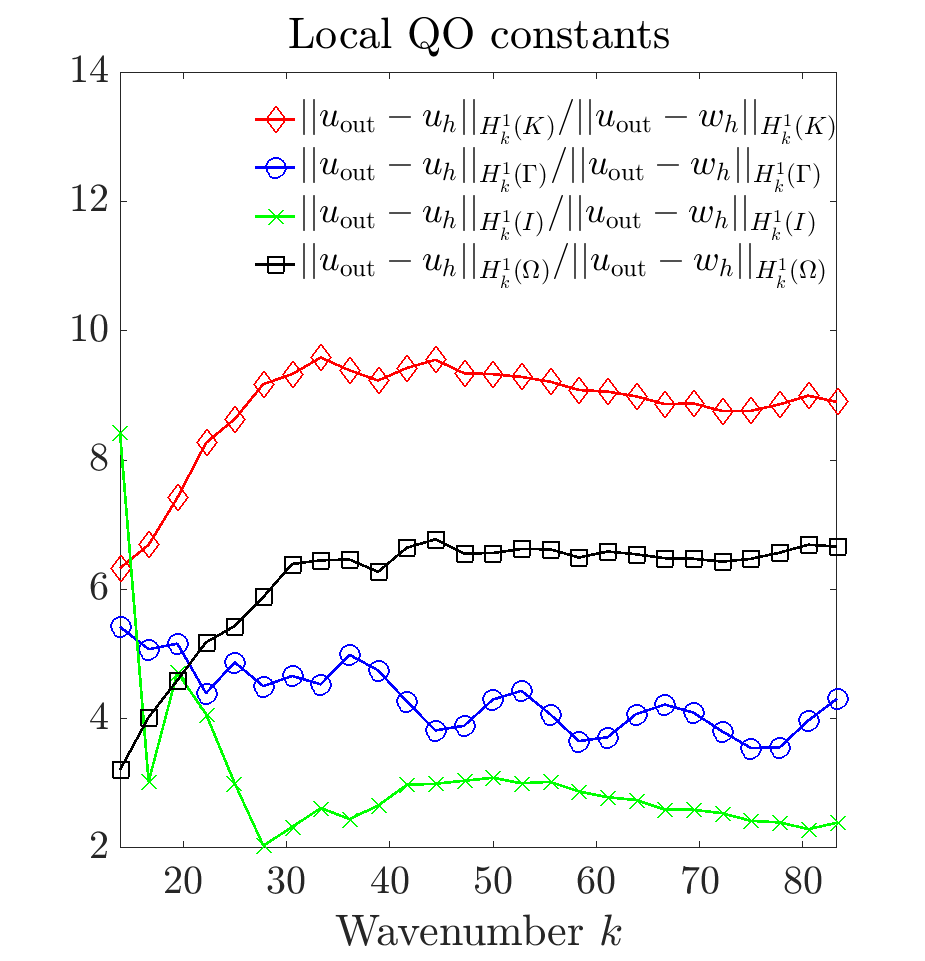}};
			\fill[color=white] (-1.2,1.25) rectangle (2.7,2.875);
		\draw (-1.35,2.6) node[right]{\tiny{$\|\!u_{\rm out}\!-\!u_h\!\|_{H_k^1(\cavity)}\!/\!\|\!u_{\rm out}\!-\!w_h\!\|_{H^1_{k}(\cavity)}$}};
		\draw (-1.35,2.25) node[right]{\tiny{$\|\!u_{\rm out}\!-\!u_h\!\|_{H_k^1(\visible)}\!/\!\|\!u_{\rm out}\!-\!w_h\!\|_{H^1_{k}(\visible)}$}};
		\draw (-1.35,1.9) node[right]{\tiny{$\|\!u_{\rm out}\!-\!u_h\!\|_{H_k^1(\invisible)}\!/\!\|\!u_{\rm out}\!-\!w_h\!\|_{H^1_{k}(\invisible)}$}};
		\draw (-1.35,1.5) node[right]{\tiny{$\|\!u_{\rm out}\!-\!u_h\!\|_{H_k^1(\Omega)}\!/\!\|\!u_{\rm out}\!-\!w_h\!\|_{H^1_{k}(\Omega)}$}};
%		\draw (-2.0,.95) node[right]{\footnotesize{$O(k^{1.2})$}};
		\fill[color=white] (-2,-4)rectangle (2.1,-3.1);
		\draw node at (.125,-3.3){Wavenumber $k$};
%		\fill[color=white] (-3.4,-1) rectangle (-2.9,4);
%		\draw node[rotate=90] at (-3.2,.3){$\|u_{\rm in }\|_{H_k^1}$};
		\fill[color=white] (-2,3.1) rectangle (2,3.5);
		\draw node at (.125,3.3){Local QO Constants};
	\end{tikzpicture}
	\caption{QO constants for $u_{\rm in}$ (left) and $u_{\rm out}$ (right) in regime U1. Black squares: global QO constant. Red diamonds: local QO constant in the cavity. Blue circles: local QO constant in the visible set. Green crosses: local QO constant in the invisible set.}
	\label{fig:U1QO}
\end{figure}

\begin{figure}[H]
	\centering
	\begin{tikzpicture}
	\draw node at (0,0){\includegraphics[width=0.45\textwidth]{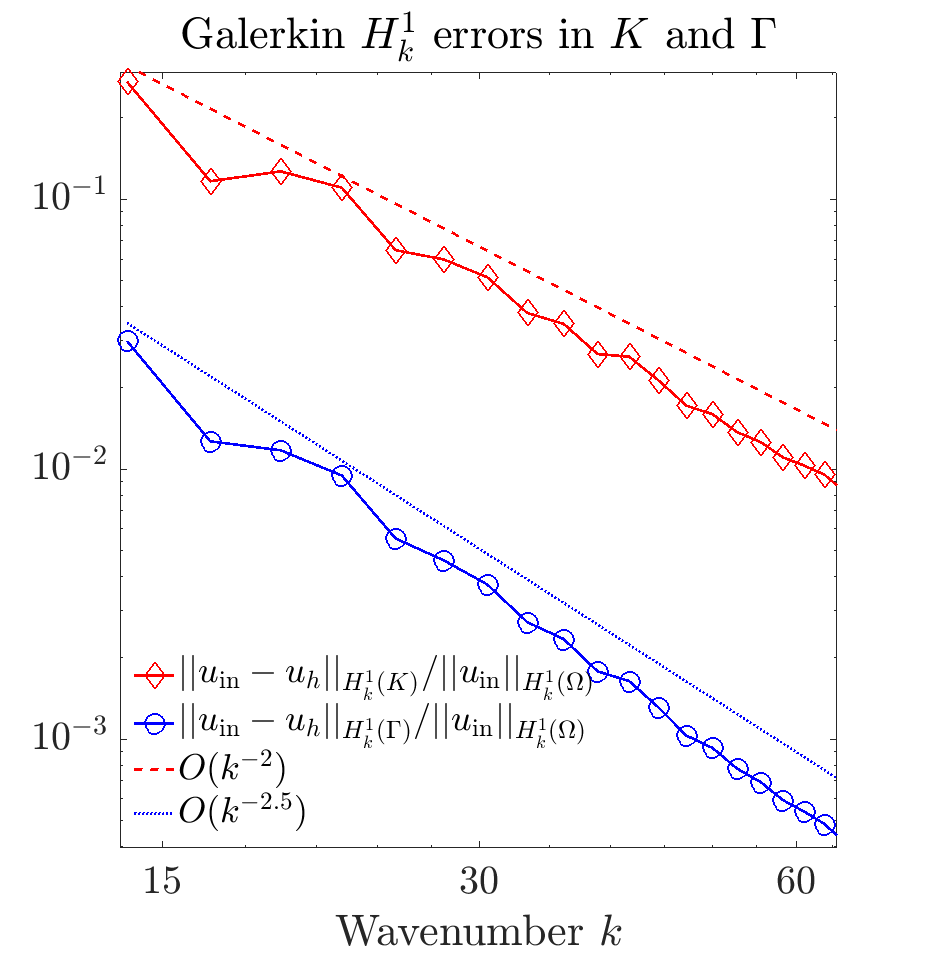}};
				\fill[color=white] (-2.1,-1.25) rectangle (.82,-2.6);
				\fill[color=white] (.82,-1.39) rectangle (.9,-2.6);
				\fill[color=white] (.9,-1.475) rectangle (1,-2.6);
		\draw (-2.25,-1.5) node[right]{\tiny{$\|\!u_{\rm in}\!-\!u_h\!\|_{H_k^1(\cavity)}\!/\!\|\!u_{\rm in}\!\|_{H^1_{k}(\Omega)}$}};
		\draw (-2.25,-1.85) node[right]{\tiny{$\|\!u_{\rm in}\!-\!u_h\!\|_{H_k^1(\visible)}\!/\!\|\!u_{\rm in}\!\|_{H^1_{k}(\Omega)}$}};
		\draw (-2.2,-2.1) node[right]{\tiny{$O(k^{-2})$}};
		\draw (-2.2,-2.4) node[right]{\tiny{$O(k^{-2.5})$}};
		\fill[color=white] (-2,-4)rectangle (2.1,-3.1);
		\draw node at (.125,-3.3){Wavenumber $k$};
%		\fill[color=white] (-3.4,-1) rectangle (-2.9,4);
%		\draw node[rotate=90] at (-3.2,.3){$\|u_{\rm in }\|_{H_k^1}$};
		\fill[color=white] (-2.6,3.05) rectangle (2.6,3.5);
		\draw node at (.125,3.3){Galerkin $H_k^1$ errors in $\cavity$ and $\visible$};
	\end{tikzpicture}
		\begin{tikzpicture}
	\draw node at (0,0){	\includegraphics[width=0.45\textwidth]{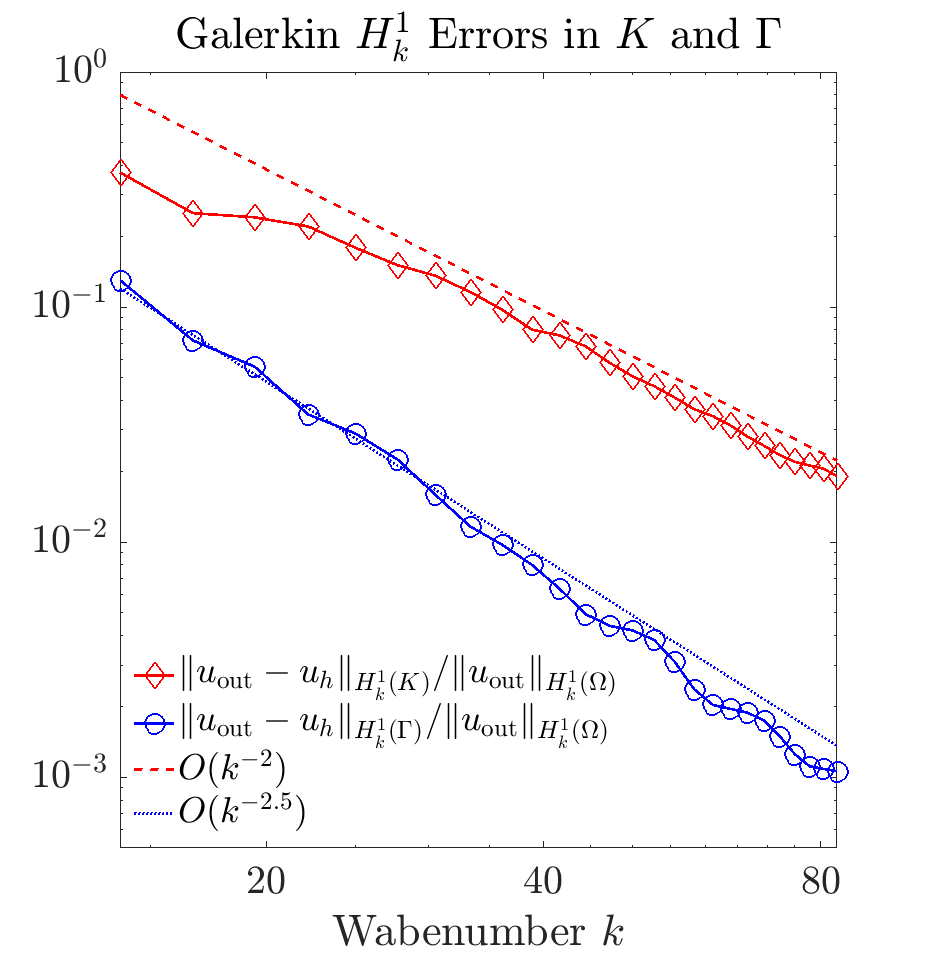}};
				\fill[color=white] (-2.1,-1.25) rectangle (1.3,-2.6);
		\draw (-2.25,-1.5) node[right]{\tiny{$\|\!u_{\rm out}\!-\!u_h\!\|_{H_k^1(\cavity)}\!/\!\|\!u_{\rm out}\!\|_{H^1_{k}(\Omega)}$}};
		\draw (-2.25,-1.85) node[right]{\tiny{$\|\!u_{\rm out}\!-\!u_h\!\|_{H_k^1(\visible)}\!/\!\|\!u_{\rm out}\!\|_{H^1_{k}(\Omega)}$}};
		\draw (-2.2,-2.1) node[right]{\tiny{$O(k^{-2})$}};
		\draw (-2.2,-2.4) node[right]{\tiny{$O(k^{-2.5})$}};
		\fill[color=white] (-2,-4)rectangle (2.1,-3.1);
		\draw node at (.125,-3.3){Wavenumber $k$};
%		\fill[color=white] (-3.4,-1) rectangle (-2.9,4);
%		\draw node[rotate=90] at (-3.2,.3){$\|u_{\rm in }\|_{H_k^1}$};
		\fill[color=white] (-2.6,3.05) rectangle (2.6,3.5);
		\draw node at (.125,3.3){Galerkin $H_k^1$ errors in $\cavity$ and $\visible$};
	\end{tikzpicture}

	\caption{
		Local Galerkin errors in the $H^1_k$ norm in $\cavity$ and $\visible$ for the approximation of $u_{\rm in}$ (left) and $u_{\rm out}$ (right) in regime U1. Red diamonds: Galerkin error in $\cavity$. Blue circles: Galerkin error in $\visible$. A priori bounds represented as red dashed lines (for the cavity) and blue dotted lines (away from cavity). 
	}
	\label{fig:U1}
\end{figure} 

It is well-known that in U1, the Galerkin solution is globally $k$-uniformly quasi-optimal (see Table \ref{tab:regimes}), and this also follows from Theorem~\ref{t:simple} (see Corollary \ref{cor:U1}, using that all matrix entries are $\lesssim 1$). This fact is illustrated by the solid black curves in Figure \ref{fig:QOQO}. By Corollary \ref{cor:U1}, the inferred rates in Table \ref{table:inferred}, and the fact that $\rho(k) \geq C k^{2}$ at the wavenumbers chosen  in the experiments, the following a priori bounds for $u_{\rm in}$ and $u_{\rm out}$ can be obtained:

\[\begin{split}
	\frac{\|u_{\rm in} - u_h\|_{H^1_k(\Omega_\cavity)}}{\|u_{\rm in}\|_{H^1_k(\Omega)}} &\lesssim (hk)^p \frac{\|u_{\rm in}\|_{H^1_k(\Omega_\cavity)}}{\|u_{\rm in}\|_{H^1_k(\Omega)}} + \sqrt{\frac{k}{\rho}} (hk)^p \frac{\|u_{\rm in}\|_{H^1_k(\Omega_\visible)}}{\|u_{\rm in}\|_{H^1_k(\Omega)}} + \left(\frac{k}{\rho}\right)^{\frac32} \frac{1}{\rho} (hk)^p\frac{\|u_{\rm in}\|_{H^1_k(\Omega_\invisible)}}{\|u_{\rm in}\|_{H^1_k(\Omega)}} \\
	& \lesssim k^{-2} ,	\\
	\frac{\|u_{\rm in} - u_h\|_{H^1_k(\Omega_\visible)}}{\|u_{\rm in}\|_{H^1_k(\Omega)}} &\lesssim (hk)^p \underbrace{\sqrt{\frac{k}{\rho}}}_{k^{-0.5}}\frac{\|u_{\rm in}\|_{H^1_k(\Omega_\cavity)}}{\|u_{\rm in}\|_{H^1_k(\Omega)}} + (hk)^p \underbrace{\frac{\|u_{\rm in}\|_{H^1_k(\Omega_\visible)}}{\|u_{\rm in}\|_{H^1_k(\Omega)}}}_{k^{-0.5}} + \left(\frac{k}{\rho}\right)^{\frac32} \frac{1}{\rho} (hk)^p\frac{\|u_{\rm in}\|_{H^1_k(\Omega_\invisible)}}{\|u_{\rm in}\|_{H^1_k(\Omega)}} \\
	& \lesssim k^{-\frac52} 	,\\
	\frac{\|u_{\rm out} - u_h\|_{H^1_k(\Omega_\cavity)}}{\|u_{\rm out}\|_{H^1_k(\Omega)}} &\lesssim (hk)^p \frac{\|u_{\rm out}\|_{H^1_k(\Omega_\cavity)}}{\|u_{\rm out}\|_{H^1_k(\Omega)}} + \sqrt{\frac{k}{\rho}} (hk)^p \frac{\|u_{\rm out}\|_{H^1_k(\Omega_\visible)}}{\|u_{\rm out}\|_{H^1_k(\Omega)}} + \left(\frac{k}{\rho}\right)^{\frac32} \frac{1}{\rho} (hk)^p\frac{\|u_{\rm out}\|_{H^1_k(\Omega_\invisible)}}{\|u_{\rm out}\|_{H^1_k(\Omega)}} \\
	& \lesssim k^{-2} ,	\\
	\frac{\|u_{\rm out} - u_h\|_{H^1_k(\Omega_\visible)}}{\|u_{\rm out}\|_{H^1_k(\Omega)}} &\lesssim (hk)^p \underbrace{\sqrt{\frac{k}{\rho}}}_{k^{-0.5}}\frac{\|u_{\rm out}\|_{H^1_k(\Omega_\cavity)}}{\|u_{\rm out}\|_{H^1_k(\Omega)}} + (hk)^p \underbrace{\frac{\|u_{\rm out}\|_{H^1_k(\Omega_\visible)}}{\|u_{\rm out}\|_{H^1_k(\Omega)}}}_{k^{-0.65}} + \left(\frac{k}{\rho}\right)^{\frac32} \frac{1}{\rho} (hk)^p\frac{\|u_{\rm out}\|_{H^1_k(\Omega_\invisible)}}{\|u_{\rm out}\|_{H^1_k(\Omega)}} \\
	& \lesssim k^{-\frac52} 	.
\end{split}\]
Figure~\ref{fig:U1} shows that, at least experimentally, these rates are sharp. Furthermore, since $u$ is $k$-oscillatory, the results of~\cite{G1} imply that the standard polynomial approximation bounds are locally sharp, i.e. the local best approximation errors satisfy
\begin{equation*}\|u - w_{h,\cavity}\|_{H^1_k(\Omega_\cavity)} \geq C(hk)^p \|u\|_{H^1_k(\Omega_\cavity)},\qquad \|u - w_{h,\visible}\|_{H^1_k(\Omega_\visible)}  \geq C (hk)^p \|u\|_{H^1_k(\Omega_\visible)}.
\end{equation*}
Corollary \ref{cor:U1} then implies that the local quasi-optimality constants in each region are  $k$-uniformly bounded as well, i.e.,
\[\frac{\|u - u_h\|_{H^1_k(\Omega_\cavity)}}{\|u - w_{h,\cavity}\|_{H^1_k(\Omega_\cavity)}} \lesssim 1, \quad \frac{\|u - u_h\|_{H^1_k(\Omega_\visible)}}{\|u - w_{h,\visible}\|_{H^1_k(\Omega_\visible)}} \lesssim 1.\]
This is consistent with the behavior observed in Figure \ref{fig:U1QO}.

\subsubsection{Regime Quasioptimality (QO)}

In QO, we choose
\[(h_\cavity k)^p k^2 + (h_\visible k)^p k^{\frac32} + (h_\invisible k)^p k = C,\]
where $C$ is independent of $k$. By Corollary \ref{cor:QO}, the Galerkin solution is again $k$-uniformly globally quasi-optimal, see Table \ref{tab:regimes}. Figure \ref{fig:QOQO} shows the local quasi-optimality constants in each regions for the problems involving $u_{\rm in}$ and $u_{\rm out}$. Figure \ref{fig:QOerrors} plots local--global relative errors.

\begin{figure}[H]
	\centering
	
		\begin{tikzpicture}
	\draw node at (0,0){\includegraphics[width=0.45\textwidth]{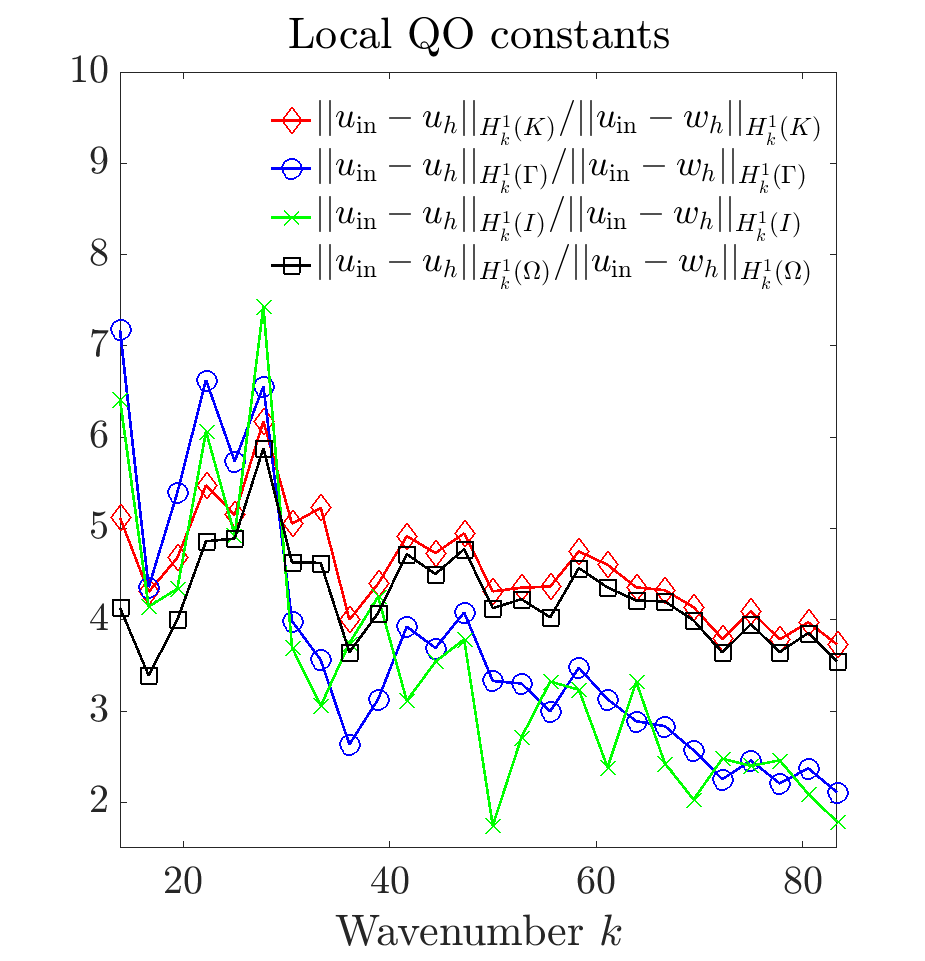}};
			\fill[color=white] (-1.1,1.25) rectangle (2.7,2.875);
		\draw (-1.175,2.6) node[right]{\tiny{$\|\!u_{\rm in}\!-\!u_h\!\|_{H_k^1(\cavity)}\!/\!\|\!u_{\rm in}\!-\!w_h\!\|_{H^1_{k}(\cavity)}$}};
		\draw (-1.175,2.25) node[right]{\tiny{$\|\!u_{\rm in}\!-\!u_h\!\|_{H_k^1(\visible)}\!/\!\|\!u_{\rm in}\!-\!w_h\!\|_{H^1_{k}(\visible)}$}};
		\draw (-1.175,1.9) node[right]{\tiny{$\|\!u_{\rm in}\!-\!u_h\!\|_{H_k^1(\invisible)}\!/\!\|\!u_{\rm in}\!-\!w_h\!\|_{H^1_{k}(\invisible)}$}};
		\draw (-1.175,1.5) node[right]{\tiny{$\|\!u_{\rm in}\!-\!u_h\!\|_{H_k^1(\Omega)}\!/\!\|\!u_{\rm in}\!-\!w_h\!\|_{H^1_{k}(\Omega)}$}};
%		\draw (-2.0,.95) node[right]{\footnotesize{$O(k^{1.2})$}};
		\fill[color=white] (-2,-4)rectangle (2.1,-3.1);
		\draw node at (.125,-3.3){Wavenumber $k$};
%		\fill[color=white] (-3.4,-1) rectangle (-2.9,4);
%		\draw node[rotate=90] at (-3.2,.3){$\|u_{\rm in }\|_{H_k^1}$};
		\fill[color=white] (-2,3.1) rectangle (2,3.5);
		\draw node at (.125,3.3){Local QO Constants};
	\end{tikzpicture}
		\begin{tikzpicture}
	\draw node at (0,0){\includegraphics[width=0.45\textwidth]{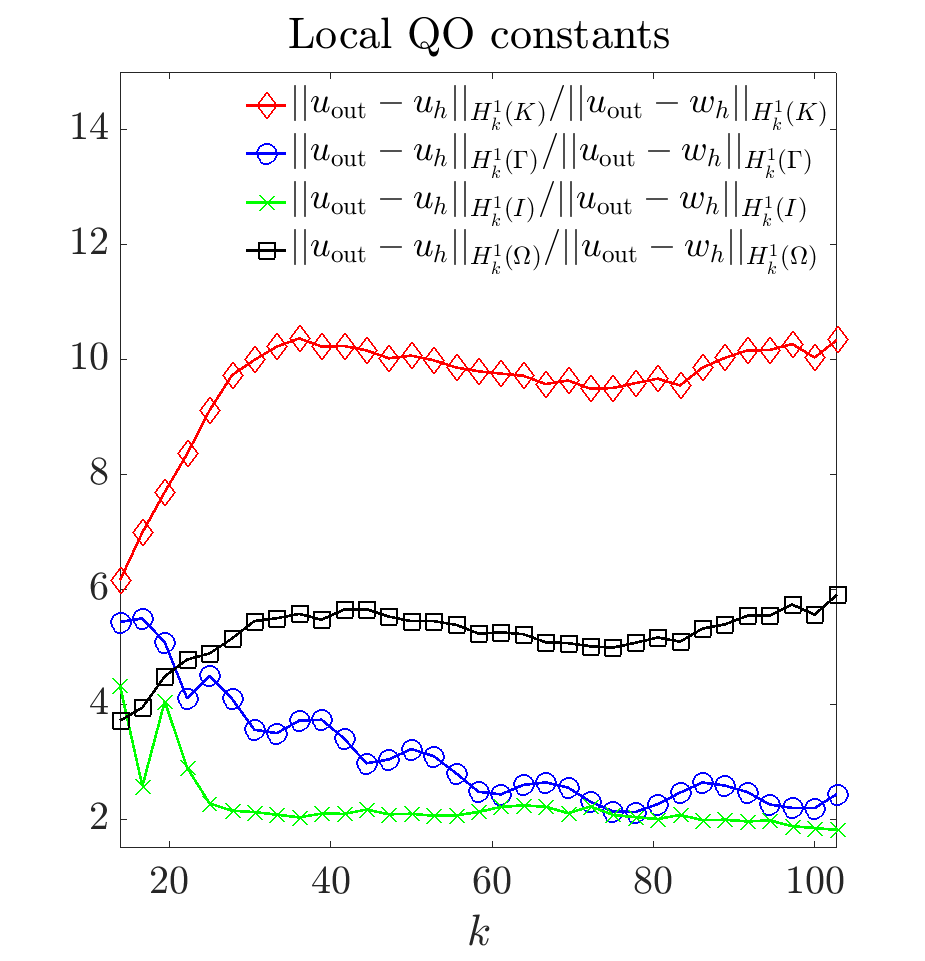}};
			\fill[color=white] (-1.3,1.25) rectangle (2.7,2.875);
		\draw (-1.4,2.675) node[right]{\tiny{$\|\!u_{\rm out}\!-\!u_h\!\|_{H_k^1(\cavity)}\!/\!\|\!u_{\rm out}\!-\!w_h\!\|_{H^1_{k}(\cavity)}$}};
		\draw (-1.4,2.325) node[right]{\tiny{$\|\!u_{\rm out}\!-\!u_h\!\|_{H_k^1(\visible)}\!/\!\|\!u_{\rm out}\!-\!w_h\!\|_{H^1_{k}(\visible)}$}};
		\draw (-1.4,1.975) node[right]{\tiny{$\|\!u_{\rm out}\!-\!u_h\!\|_{H_k^1(\invisible)}\!/\!\|\!u_{\rm out}\!-\!w_h\!\|_{H^1_{k}(\invisible)}$}};
		\draw (-1.4,1.575) node[right]{\tiny{$\|\!u_{\rm out}\!-\!u_h\!\|_{H_k^1(\Omega)}\!/\!\|\!u_{\rm out}\!-\!w_h\!\|_{H^1_{k}(\Omega)}$}};
%		\draw (-2.0,.95) node[right]{\footnotesize{$O(k^{1.2})$}};
		\fill[color=white] (-2,-4)rectangle (2.1,-3.1);
		\draw node at (.125,-3.3){Wavenumber $k$};
%		\fill[color=white] (-3.4,-1) rectangle (-2.9,4);
%		\draw node[rotate=90] at (-3.2,.3){$\|u_{\rm in }\|_{H_k^1}$};
		\fill[color=white] (-2,3.1) rectangle (2,3.5);
		\draw node at (.125,3.3){Local QO Constants};
	\end{tikzpicture}
	
	\caption{Local QO constants for $u_{\rm in}$ (left) and $u_{\rm out}$ (right) in the regime QO. Black squares: global QO constant. Red diamonds: local QO constant in the cavity. Blue circles: local QO constant in the visible set. Green crosses: local QO constant in the invisible set. }
	\label{fig:QOQO}
\end{figure}

\begin{figure}[h]
	\centering
	
		\begin{tikzpicture}
	\draw node at (0,0){\includegraphics[width=0.45\textwidth]{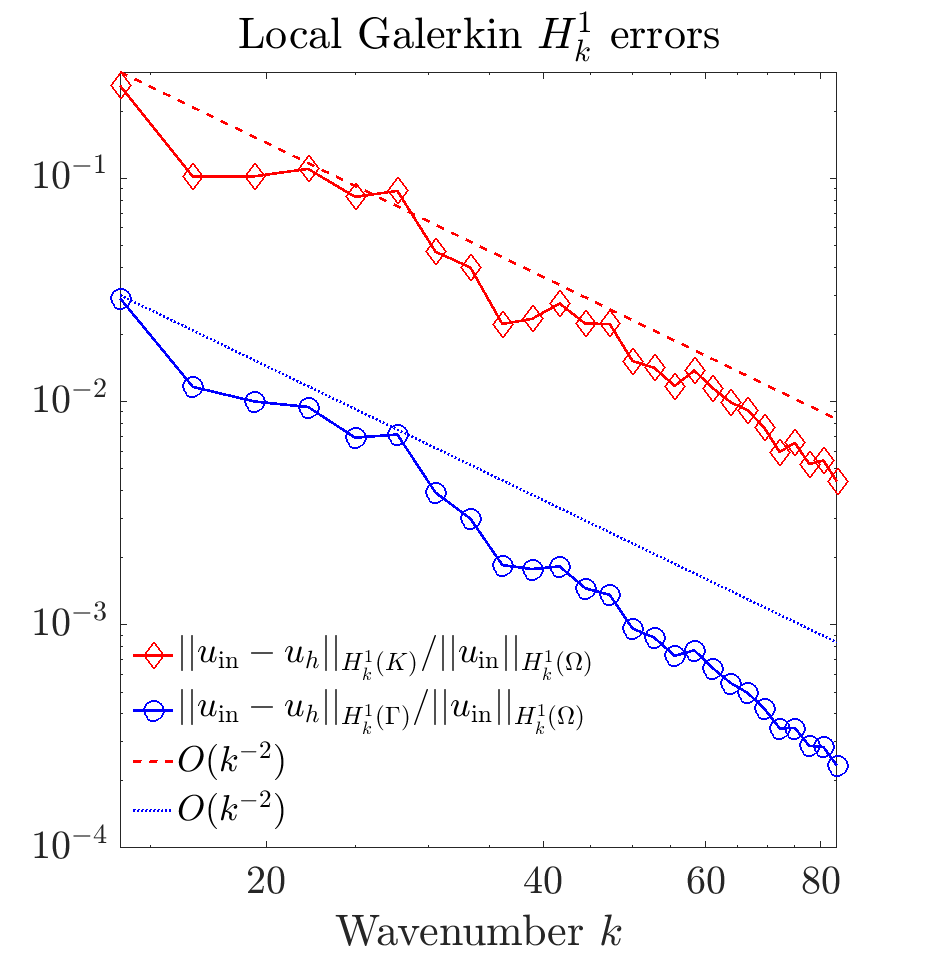}};
				\fill[color=white] (-2.1,-1.05) rectangle (1,-2.6);
		\draw (-2.2,-1.35) node[right]{\tiny{$\|\!u_{\rm in}\!-\!u_h\!\|_{H_k^1(\cavity)}\!/\!\|\!u_{\rm in}\!\|_{H^1_{k}(\Omega)}$}};
		\draw (-2.2,-1.75) node[right]{\tiny{$\|\!u_{\rm in}\!-\!u_h\!\|_{H_k^1(\visible)}\!/\!\|\!u_{\rm in}\!\|_{H^1_{k}(\Omega)}$}};
		\draw (-2.15,-2) node[right]{\tiny{$O(k^{-2})$}};
		\draw (-2.15,-2.4) node[right]{\tiny{$O(k^{-2})$}};
		\fill[color=white] (-2,-4)rectangle (2.1,-3.1);
		\draw node at (.125,-3.3){Wavenumber $k$};
%		\fill[color=white] (-3.4,-1) rectangle (-2.9,4);
%		\draw node[rotate=90] at (-3.2,.3){$\|u_{\rm in }\|_{H_k^1}$};
		\fill[color=white] (-2.6,3.05) rectangle (2.6,3.5);
		\draw node at (.125,3.3){Galerkin $H_k^1$ errors in $\cavity$ and $\visible$};
	\end{tikzpicture}
			\begin{tikzpicture}
	\draw node at (0,0){		\includegraphics[width=0.45\textwidth]{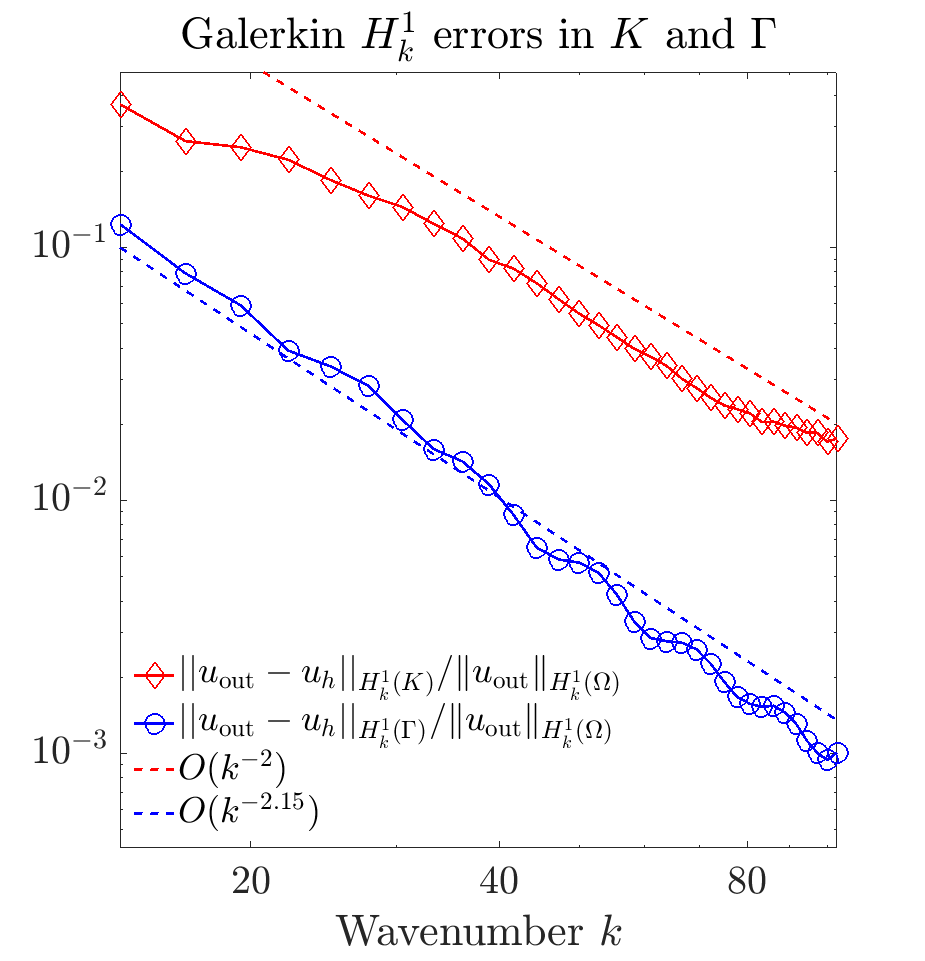}};
				\fill[color=white] (-2.1,-1.25) rectangle (1.3,-2.6);
		\draw (-2.2,-1.5) node[right]{\tiny{$\|\!u_{\rm out}\!-\!u_h\!\|_{H_k^1(\cavity)}\!/\!\|\!u_{\rm out}\!\|_{H^1_{k}(\Omega)}$}};
		\draw (-2.2,-1.85) node[right]{\tiny{$\|\!u_{\rm out}\!-\!u_h\!\|_{H_k^1(\visible)}\!/\!\|\!u_{\rm out}\!\|_{H^1_{k}(\Omega)}$}};
		\draw (-2.2,-2.1) node[right]{\tiny{$O(k^{-2})$}};
		\draw (-2.2,-2.4) node[right]{\tiny{$O(k^{-2.15})$}};
		\fill[color=white] (-2,-4)rectangle (2.1,-3.1);
		\draw node at (.125,-3.3){Wavenumber $k$};
%		\fill[color=white] (-3.4,-1) rectangle (-2.9,4);
%		\draw node[rotate=90] at (-3.2,.3){$\|u_{\rm in }\|_{H_k^1}$};
		\fill[color=white] (-2.6,3.05) rectangle (2.6,3.5);
		\draw node at (.125,3.3){Galerkin $H_k^1$ errors in $\cavity$ and $\visible$};
	\end{tikzpicture}

	\caption{Local Galerkin errors in the $H^1_k$ norm in $\cavity$ and $\visible$ for the approximation of $u_{\rm in}$ (left) and $u_{\rm out}$ (right) in the regime QO. Red diamonds: error in the cavity. Blue circles: error away from the cavity. A priori bounds represented as red dashed lines (for the cavity) and blue dotted lines (away from cavity).}
	\label{fig:QOerrors}
\end{figure}

% the Galerkin errors in the $H^1_k$ norm in $\cavity$ and $\visible$, normalized by the global $H^1_k$ norm of the solution $u_{\rm in/out}$ in $\Omega$. 
%
By Corollary \ref{cor:QO}, the inferred rates in Table \ref{table:inferred}, and the fact that $\rho(k) \geq C k^{2}$, the a priori bounds for $u_{\rm in}$ and $u_{\rm out}$ in each region are given by 
\[\begin{split}
	\frac{\|u_{\rm in} - u_h\|_{H^1_k(\Omega_\cavity)}}{\|u_{\rm in}\|_{H^1_k(\Omega)}} &\lesssim \underbrace{(h_\cavity k)^p}_{k^{-2}} \frac{\|u_{\rm in}\|_{H^1_k(\Omega_\cavity)}}{\|u_{\rm in}\|_{H^1_k(\Omega)}} + \underbrace{(h_\visible k)^p}_{k^{-3/2}} \underbrace{\frac{\|u_{\rm in}\|_{H^1_k(\Omega_\visible)}}{\|u_{\rm in}\|_{H^1_k(\Omega)}}}_{k^{-\frac12}} + (k\rho)^{-\frac{1}{2}} (h_\invisible k)^p\frac{\|u_{\rm in}\|_{H^1_k(\Omega_\invisible)}}{\|u_{\rm in}\|_{H^1_k(\Omega)}} \\
	& \lesssim k^{-2} ,	\\
	\frac{\|u_{\rm in} - u_h\|_{H^1_k(\Omega_\visible)}}{\|u_{\rm in}\|_{H^1_k(\Omega)}} &\lesssim \underbrace{(h_\cavity k)^p}_{k^{-2}} \frac{\|u_{\rm in}\|_{H^1_k(\Omega_\cavity)}}{\|u_{\rm in}\|_{H^1_k(\Omega)}} + \underbrace{(h_\visible k)^p}_{k^{-1}} \underbrace{\frac{\|u_{\rm in}\|_{H^1_k(\Omega_\visible)}}{\|u_{\rm in}\|_{H^1_k(\Omega)}}}_{k^{-0.5}} +  (h_\invisible k)^p\frac{\|u_{\rm in}\|_{H^1_k(\Omega_\invisible)}}{\|u_{\rm in}\|_{H^1_k(\Omega)}} 
	 \lesssim k^{-1.5} 	,
	 \end{split}\]\[
	 \begin{split}
	\frac{\|u_{\rm out} - u_h\|_{H^1_k(\Omega_\cavity)}}{\|u_{\rm out}\|_{H^1_k(\Omega)}} &\lesssim \underbrace{(h_\cavity k)^p}_{k^{-2}} \frac{\|u_{\rm out}\|_{H^1_k(\Omega_\cavity)}}{\|u_{\rm out}\|_{H^1_k(\Omega)}} + \underbrace{(h_\visible k)^p}_{k^{-\frac32}} \underbrace{\frac{\|u_{\rm out}\|_{H^1_k(\Omega_\visible)}}{\|u_{\rm out}\|_{H^1_k(\Omega)}}}_{k^{-0.65}} +   (k\rho)^{-\frac12}(h_\invisible k)^p\frac{\|u_{\rm out}\|_{H^1_k(\Omega_\invisible)}}{\|u_{\rm out}\|_{H^1_k(\Omega)}} \\
	& \lesssim k^{-2} ,	\\
	\frac{\|u_{\rm out} - u_h\|_{H^1_k(\Omega_\visible)}}{\|u_{\rm out}\|_{H^1_k(\Omega)}} &\lesssim \underbrace{(h_\cavity k)^p}_{k^{-2}} \underbrace{\sqrt{\frac{k}{\rho}}}_{k^{-0.5}}\frac{\|u_{\rm out}\|_{H^1_k(\Omega_\cavity)}}{\|u_{\rm out}\|_{H^1_k(\Omega)}} + \underbrace{(h_\visible k)^p}_{k^{-\frac32}} \underbrace{\frac{\|u_{\rm out}\|_{H^1_k(\Omega_\visible)}}{\|u_{\rm out}\|_{H^1_k(\Omega)}}}_{k^{-0.65}} + (h_\invisible k)^p\frac{\|u_{\rm out}\|_{H^1_k(\Omega_\invisible)}}{\|u_{\rm out}\|_{H^1_k(\Omega)}} \\
	& \lesssim k^{-2.15} 	.
\end{split}\]
Again, these rates are experimentally verified in Figure \ref{fig:QOerrors}.

%
%\begin{figure}[H]
%	\centering
%	%\includegraphics[width=0.45\textwidth]{Experiments/QO/beamIn/improvementErrorAway}
%	%\includegraphics[width=0.45\textwidth]{Experiments/QO/beamOut/improvementErrorAway}	
%	\caption{Empirical improvement of the error away compared to the error in trapping. Left: beam inside. Right: beam outside. In the beam inside, we expect no improvement in theory: we expect the errors away and in trapping to be balanced, given the behavior of $u$. Note the QO constant away which cannot continue decreasing so fast as it has to eventually be bounded from below by $1$. For the beam outside experiment, the numerical behavior is the expected one.}	
%\end{figure}

\subsubsection{Regime Quasioptimality away (QO away)}

In QO away, we choose
\begin{equation*}
\label{e:weakMeshQO}(h_\cavity k)^p k^{3/2} = (h_\visible k)^p k = (h_\invisible k)^p k =: (hk)^p k^2 = C,
\end{equation*}
where $C$ is independent of $k$. 
Theorem \ref{t:simple} no longer guarantees $k$-uniform quasi-optimality, but Corollary \ref{cor:QOaway} and the conjecture that $\rho(k) \leq C k^2$ imply the following bounds for the ``QO constants" (not to be confused with ``local QO constants" -- notice the global norm of the best approximation error in the denominator instead of the local norm in the local QO constants) 
\begin{equation}
	\label{eqs:placeDuRalliement}
	\frac{\|u - u_h\|_{H^1_k(\Omega_\cavity)}}{\|u - w_h\|_{H^1_k(\Omega)}} \lesssim \sqrt{k},\quad \frac{\|u - u_h\|_{H^1_k(\Omega_\visible)}}{\|u - w_h\|_{H^1_k(\Omega)}} \lesssim 1, \quad \frac{\|u - u_h\|_{H^1_k(\Omega_\invisible)}}{\|u - w_h\|_{H^1_k(\Omega)}} \lesssim 1,
\end{equation}
hence, the Galerkin solution remains $k$-uniform quasioptimal away from the cavity in the regime QO away (see also Table \ref{tab:regimes}). 

Figure \ref{fig:QOaway} shows the QO constants in each regions for the problems involving $u_{\rm in}$ and $u_{\rm out}$. The QO constant in the invisible set is orders of magnitude smaller than the other quantities, so it is not displayed. Figure \ref{fig:QOawayErrors} plots the local-global relative errors.

\begin{figure}[H]
	\centering
			\begin{tikzpicture}
	\draw node at (0,0){\includegraphics[width=0.45\textwidth]{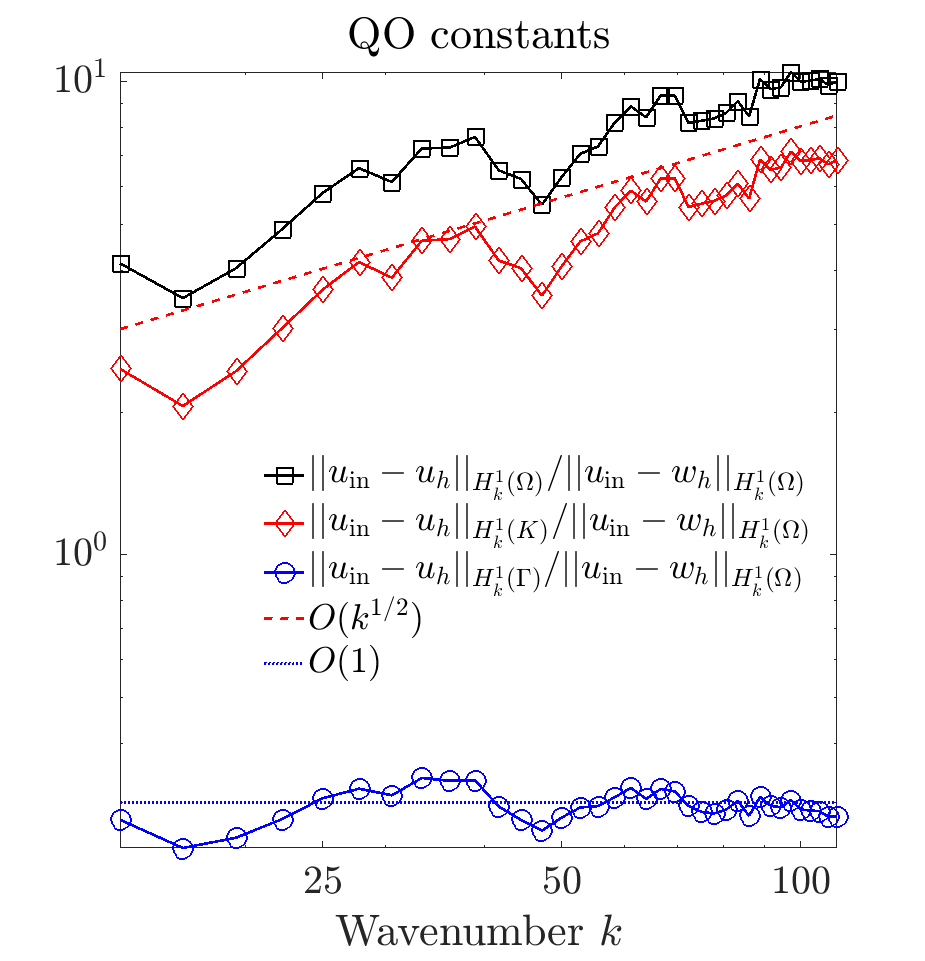}};
				\fill[color=white] (-1.15,.3) rectangle (2.6,-1.6);
				\draw (-1.25,-.025) node[right]{\tiny{$\|u_{\rm in}\!-\!u_h\!\|_{H_k^1(\Omega)}\!/\!\|u_{\rm in}\!-\!w_h\!\|_{H^1_{k}(\Omega)}$}};
		\draw (-1.25,-.375) node[right]{\tiny{$\|u_{\rm in}\!-\!u_h\!\|_{H_k^1(\cavity)}\!/\!\|u_{\rm in}\!-\!w_h\!\|_{H^1_{k}(\Omega)}$}};
		\draw (-1.25,-.725) node[right]{\tiny{$\|u_{\rm in}\!-\!u_h\!\|_{H_k^1(\visible)}\!/\!\|u_{\rm in}-\!w_h\!\|_{H^1_{k}(\Omega)}$}};
		\draw (-1.25,-.975) node[right]{\tiny{$O(k^{0.5})$}};
		\draw (-1.25,-1.325) node[right]{\tiny{$O(1)$}};
		\fill[color=white] (-2,-4)rectangle (2.1,-3.1);
		\draw node at (.125,-3.3){Wavenumber $k$};
%		\fill[color=white] (-3.4,-1) rectangle (-2.9,4);
%		\draw node[rotate=90] at (-3.2,.3){$\|u_{\rm in }\|_{H_k^1}$};
		\fill[color=white] (-2.6,3.05) rectangle (2.6,3.5);
		\draw node at (.125,3.3){QO Constants};
	\end{tikzpicture}	
				\begin{tikzpicture}
	\draw node at (0,0){\includegraphics[width=0.45\textwidth]{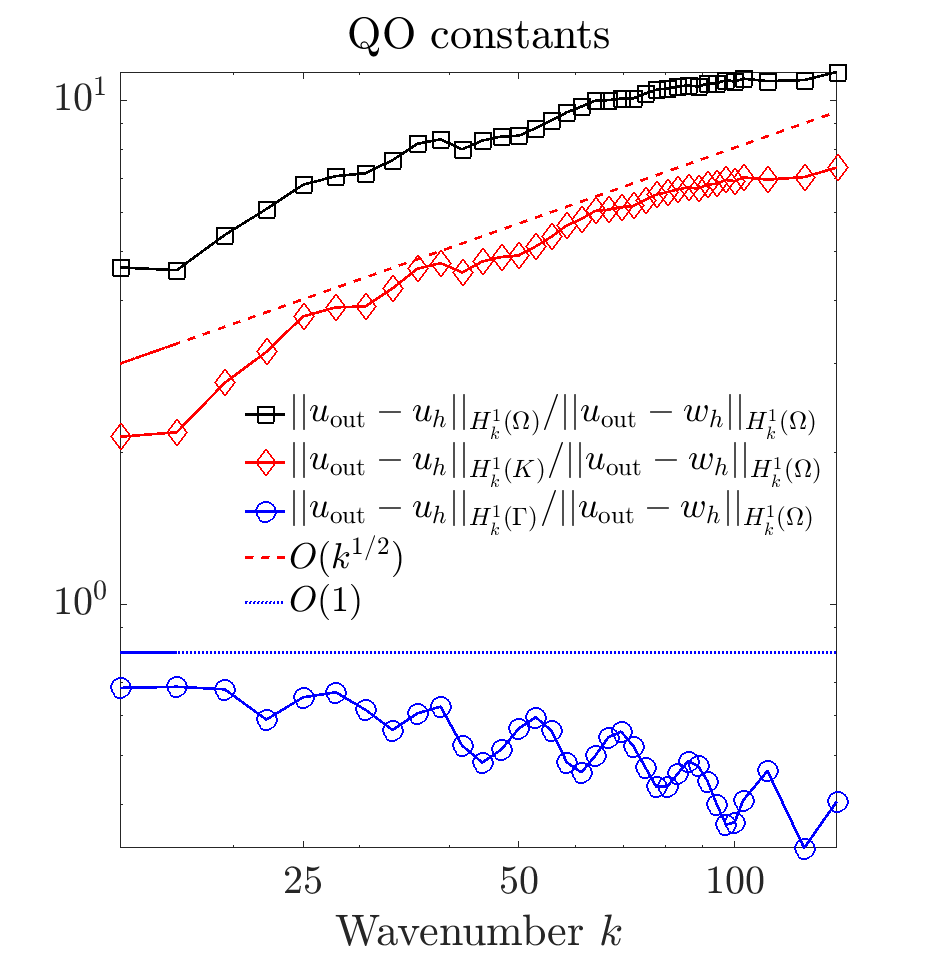}};
				\fill[color=white] (-1.25,.65) rectangle (2.6,-1.2);
				\draw (-1.425,.4) node[right]{\tiny{$\|u_{\rm out}\!-\!u_h\!\|_{H_k^1(\Omega)}\!/\!\|u_{\rm out}\!-\!w_h\!\|_{H^1_{k}(\Omega)}$}};
		\draw (-1.425,.05) node[right]{\tiny{$\|u_{\rm out}\!-\!u_h\!\|_{H_k^1(\cavity)}\!/\!\|u_{\rm out}\!-\!w_h\!\|_{H^1_{k}(\Omega)}$}};
		\draw (-1.425,-.3) node[right]{\tiny{$\|u_{\rm out}\!-\!u_h\!\|_{H_k^1(\visible)}\!/\!\|u_{\rm out}-\!w_h\!\|_{H^1_{k}(\Omega)}$}};
		\draw (-1.4,-.55) node[right]{\tiny{$O(k^{0.5})$}};
		\draw (-1.4,-.9) node[right]{\tiny{$O(1)$}};
		\fill[color=white] (-2,-4)rectangle (2.1,-3.1);
		\draw node at (.125,-3.3){Wavenumber $k$};
%		\fill[color=white] (-3.4,-1) rectangle (-2.9,4);
%		\draw node[rotate=90] at (-3.2,.3){$\|u_{\rm in }\|_{H_k^1}$};
		\fill[color=white] (-2.6,3.05) rectangle (2.6,3.5);
		\draw node at (.125,3.3){QO Constants};
	\end{tikzpicture}
	\caption{QO constants for $u_{\rm in}$ (left) and $u_{\rm out}$ (right) in the regime QO away. Black squares: global QO constant. Red diamonds: QO constant in $\cavity$. Blue circles: QO constant in $\Gamma$. A priori bounds represented as red dashed lines (for the cavity) and blue dotted lines (away from cavity).}
	\label{fig:QOaway}
\end{figure}
\begin{figure}[H]
	\centering
			\begin{tikzpicture}
	\draw node at (0,0){	\includegraphics[width=0.45\textwidth]{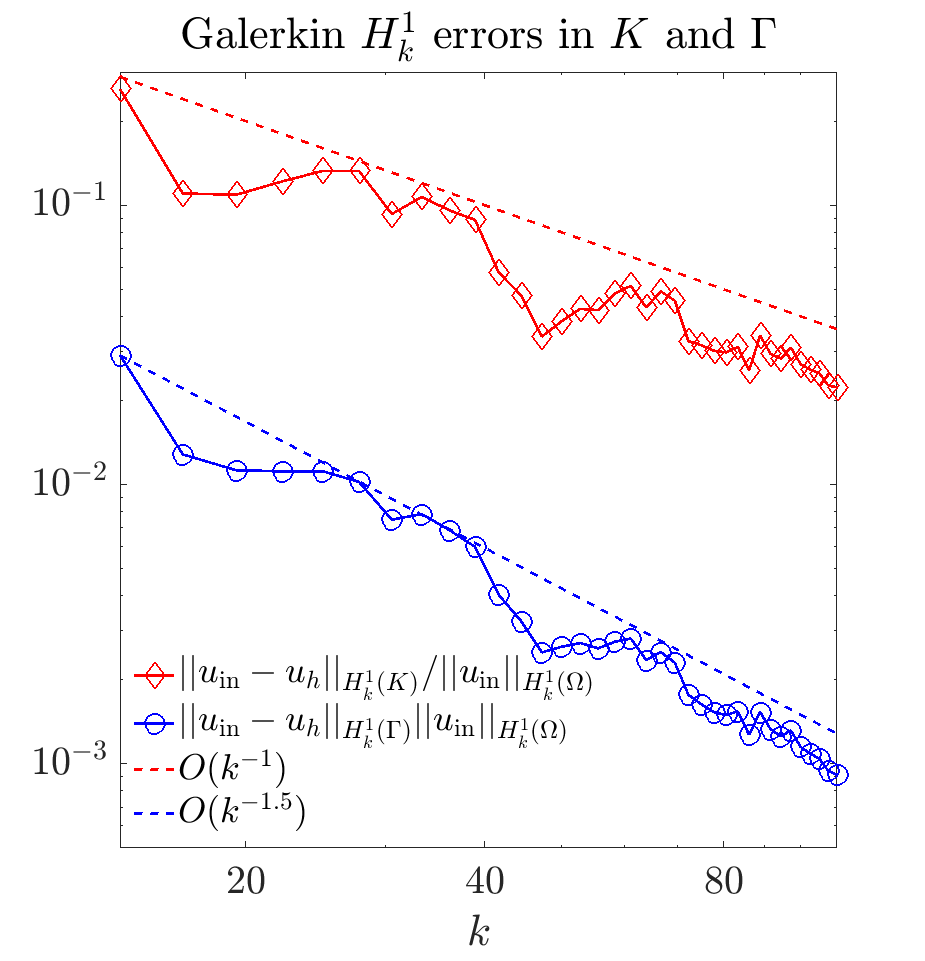}};
				\fill[color=white] (-2.1,-1.35) rectangle (1,-2.6);
				\fill[color=white] (-2.1,-1.25) rectangle (.5,-2.6);
		\draw (-2.2,-1.5) node[right]{\tiny{$\|\!u_{\rm in}\!-\!u_h\!\|_{H_k^1(\cavity)}\!/\!\|\!u_{\rm in}\!\|_{H^1_{k}(\Omega)}$}};
		\draw (-2.2,-1.85) node[right]{\tiny{$\|\!u_{\rm in}\!-\!u_h\!\|_{H_k^1(\visible)}\!/\!\|\!u_{\rm in}\!\|_{H^1_{k}(\Omega)}$}};
		\draw (-2.15,-2.1) node[right]{\tiny{$O(k^{-1})$}};
		\draw (-2.15,-2.4) node[right]{\tiny{$O(k^{-1.5})$}};
		\fill[color=white] (-2,-4)rectangle (2.1,-3.1);
		\draw node at (.125,-3.3){Wavenumber $k$};
%		\fill[color=white] (-3.4,-1) rectangle (-2.9,4);
%		\draw node[rotate=90] at (-3.2,.3){$\|u_{\rm in }\|_{H_k^1}$};
		\fill[color=white] (-2.6,3.05) rectangle (2.6,3.5);
		\draw node at (.125,3.3){Galerkin $H_k^1$ errors in $\cavity$ and $\visible$};
	\end{tikzpicture}
				\begin{tikzpicture}
	\draw node at (0,0){\includegraphics[width=0.45\textwidth]{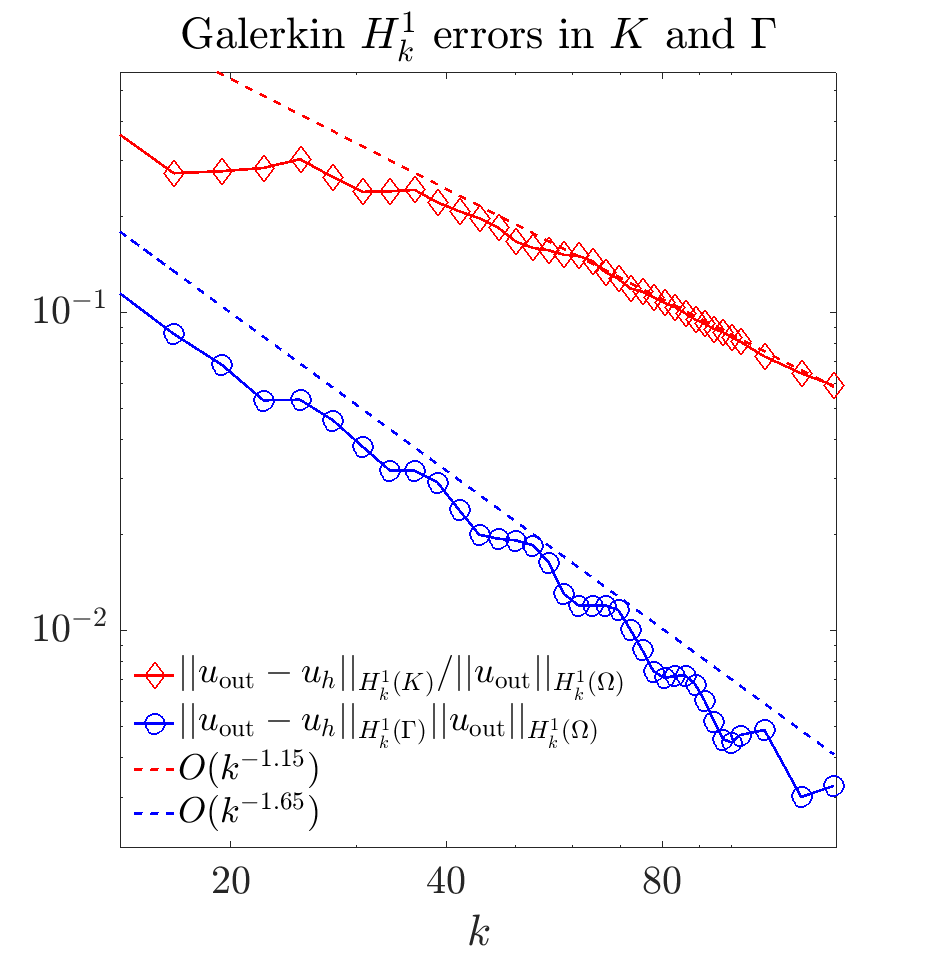}};
				\fill[color=white] (-2.1,-1.35) rectangle (1,-2.6);
				\fill[color=white] (-2.1,-1.25) rectangle (1.2,-2.6);
		\draw (-2.2,-1.5) node[right]{\tiny{$\|\!u_{\rm out}\!-\!u_h\!\|_{H_k^1(\cavity)}\!/\!\|\!u_{\rm out}\!\|_{H^1_{k}(\Omega)}$}};
		\draw (-2.2,-1.85) node[right]{\tiny{$\|\!u_{\rm out}\!-\!u_h\!\|_{H_k^1(\visible)}\!/\!\|\!u_{\rm out}\!\|_{H^1_{k}(\Omega)}$}};
		\draw (-2.15,-2.1) node[right]{\tiny{$O(k^{-1.15})$}};
		\draw (-2.15,-2.4) node[right]{\tiny{$O(k^{-1.65})$}};
		\fill[color=white] (-2,-4)rectangle (2.1,-3.1);
		\draw node at (.125,-3.3){Wavenumber $k$};
%		\fill[color=white] (-3.4,-1) rectangle (-2.9,4);
%		\draw node[rotate=90] at (-3.2,.3){$\|u_{\rm in }\|_{H_k^1}$};
		\fill[color=white] (-2.6,3.05) rectangle (2.6,3.5);
		\draw node at (.125,3.3){Galerkin $H_k^1$ errors in $\cavity$ and $\visible$};
	\end{tikzpicture}
	\caption{Local Galerkin errors in the $H^1_k$ norm in $\cavity$ and $\visible$ for the approximation of $u_{\rm in}$ (left) and $u_{\rm out}$ (right) in the regime QO away. Red diamonds: error in the cavity. Blue circles: error away from the cavity. The a priori bounds derived from Corollary \ref{cor:QO}, the lower bound $\rho(k) \geq C k^{2}$, and the inferred rates in Table \ref{table:inferred}, are represented as red dashed lines (for the cavity) and blue dotted lines (away from cavity).}
	\label{fig:QOawayErrors}
\end{figure}

The numerical results in Figure \ref{fig:QOaway} illustrate  that the bounds in \eqref{eqs:placeDuRalliement} are, at least experimentally, sharp. Furthermore, by Corollary \ref{cor:QOaway}, the inferred rates in Table \ref{table:inferred}, and the fact that $\rho(k) \geq C k^{2}$, the a priori bounds for $u_{\rm in}$ and $u_{\rm out}$ in $\cavity$ and $\visible$ are given by
\[\begin{split}
	\frac{\|u_{\rm in} - u_h\|_{H^1_k(\Omega_\cavity)}}{\|u_{\rm in}\|_{H^1_k(\Omega)}} &\lesssim \underbrace{(h_\cavity k)^p}_{k^{-2}} \underbrace{\sqrt{\frac{\rho}{k}}}_{k^{\frac12}} \frac{\|u_{\rm in}\|_{H^1_k(\Omega_\cavity)}}{\|u_{\rm in}\|_{H^1_k(\Omega)}} + \underbrace{\sqrt{\frac{\rho}{k}}}_{k^{\frac12}}\underbrace{(h_\visible k)^p}_{k^{-1}} \underbrace{\frac{\|u_{\rm in}\|_{H^1_k(\Omega_\visible)}}{\|u_{\rm in}\|_{H^1_k(\Omega)}}}_{k^{-\frac12}} + (k\rho)^{-\frac{1}{2}} (h_\invisible k)^p\frac{\|u_{\rm in}\|_{H^1_k(\Omega_\invisible)}}{\|u_{\rm in}\|_{H^1_k(\Omega)}} \\
	& \lesssim k^{-1},\\
	\frac{\|u_{\rm in} - u_h\|_{H^1_k(\Omega_\visible)}}{\|u_{\rm in}\|_{H^1_k(\Omega)}} &\lesssim \underbrace{(h_\cavity k)^p}_{k^{-2}} \underbrace{\sqrt{\frac{k}{\rho}}}_{k^{-0.5}}\frac{\|u_{\rm in}\|_{H^1_k(\Omega_\cavity)}}{\|u_{\rm in}\|_{H^1_k(\Omega)}} + \underbrace{(h_\visible k)^p}_{k^{-\frac32}} \underbrace{\frac{\|u_{\rm in}\|_{H^1_k(\Omega_\visible)}}{\|u_{\rm in}\|_{H^1_k(\Omega)}}}_{k^{-0.5}} +  (h_\invisible k)^p\frac{\|u_{\rm in}\|_{H^1_k(\Omega_\invisible)}}{\|u_{\rm in}\|_{H^1_k(\Omega)}} 
	 \lesssim k^{-2} 	,
\end{split}\]
and similarly, 
\[\frac{\|u_{\rm out} - u_h\|_{H^1_k(\Omega_\cavity)}}{\|u_{\rm out}\|_{H^1_k(\Omega)}} \lesssim k^{-1.15}, \quad \frac{\|u_{\rm out} - u_h\|_{H^1_k(\Omega_\visible)}}{\|u_{\rm out}\|_{H^1_k(\Omega)}} \lesssim k^{-1.65}. \]
These rates are experimentally verified in Figure \ref{fig:QOawayErrors}.

\subsubsection{Regime Uniform 2 (U2)}\label{sec:expU2}

In U2, we choose
\[(h_\cavity k)^{2p} k^2 = (h_\visible k)^{2p} k^2 = (h_\invisible k)^{2p} k^2 =: (hk)^p k^2 = C,\]
where $C$ is independent of $k$. Figure \ref{fig:U2QO} shows the QO constants in each regions for the problems involving $u_{\rm in}$ and $u_{\rm out}$. Figure \ref{fig:U2rel} plots local-global relative errors.
%the Galerkin errors in the $H^1_k$ norm in $\Omega$, $\cavity$ and $\visible$, normalized by the global $H^1_k$ norm of the solution $u_{\rm in/out}$ on $\Omega$.

\begin{figure}[htbp]
	\centering
				\begin{tikzpicture}
	\draw node at (0,0){\includegraphics[width=0.45\textwidth]{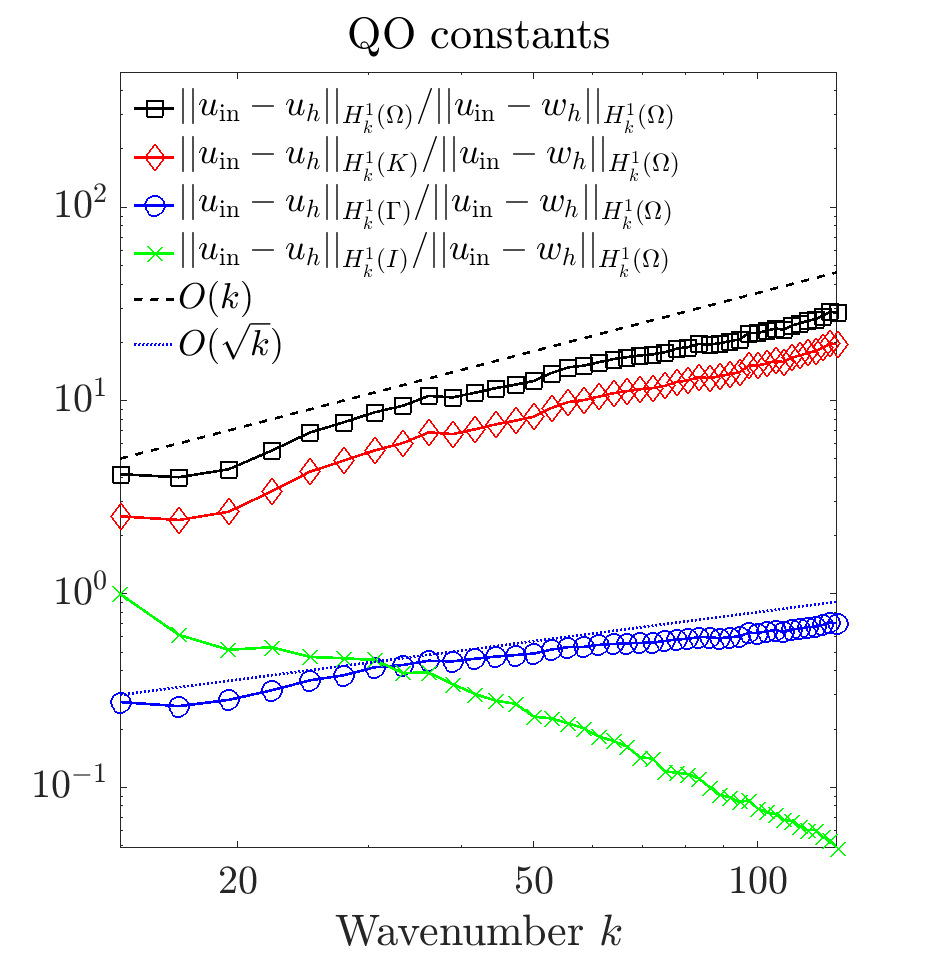}};
				\fill[color=white] (-2.1,.8) rectangle (0,2.9);
				\fill[color=white] (0,1.3) rectangle (1.8,2.9);
				\draw (-2.2,2.6) node[right]{\tiny{$\|u_{\rm in}\!-\!u_h\!\|_{H_k^1(\Omega)}\!/\!\|u_{\rm in}\!-\!w_h\!\|_{H^1_{k}(\Omega)}$}};
		\draw (-2.2,2.25) node[right]{\tiny{$\|u_{\rm in}\!-\!u_h\!\|_{H_k^1(\cavity)}\!/\!\|u_{\rm in}\!-\!w_h\!\|_{H^1_{k}(\Omega)}$}};
		\draw (-2.2,1.9) node[right]{\tiny{$\|u_{\rm in}\!-\!u_h\!\|_{H_k^1(\visible)}\!/\!\|u_{\rm in}-\!w_h\!\|_{H^1_{k}(\Omega)}$}};
		\draw (-2.2,1.55) node[right]{\tiny{$\|u_{\rm in}\!-\!u_h\!\|_{H_k^1(\invisible)}\!/\!\|u_{\rm in}-\!w_h\!\|_{H^1_{k}(\Omega)}$}};
		\draw (-2.2,1.3) node[right]{\tiny{$O(k)$}};
		\draw (-2.2,.95) node[right]{\tiny{$O(k^{0.5})$}};
		\fill[color=white] (-2,-4)rectangle (2.1,-3.1);
		\draw node at (.125,-3.3){Wavenumber $k$};
%		\fill[color=white] (-3.4,-1) rectangle (-2.9,4);
%		\draw node[rotate=90] at (-3.2,.3){$\|u_{\rm in }\|_{H_k^1}$};
		\fill[color=white] (-2.6,3.05) rectangle (2.6,3.5);
		\draw node at (.125,3.3){QO Constants};
	\end{tikzpicture}	
					\begin{tikzpicture}
	\draw node at (0,0){\includegraphics[width=0.45\textwidth]{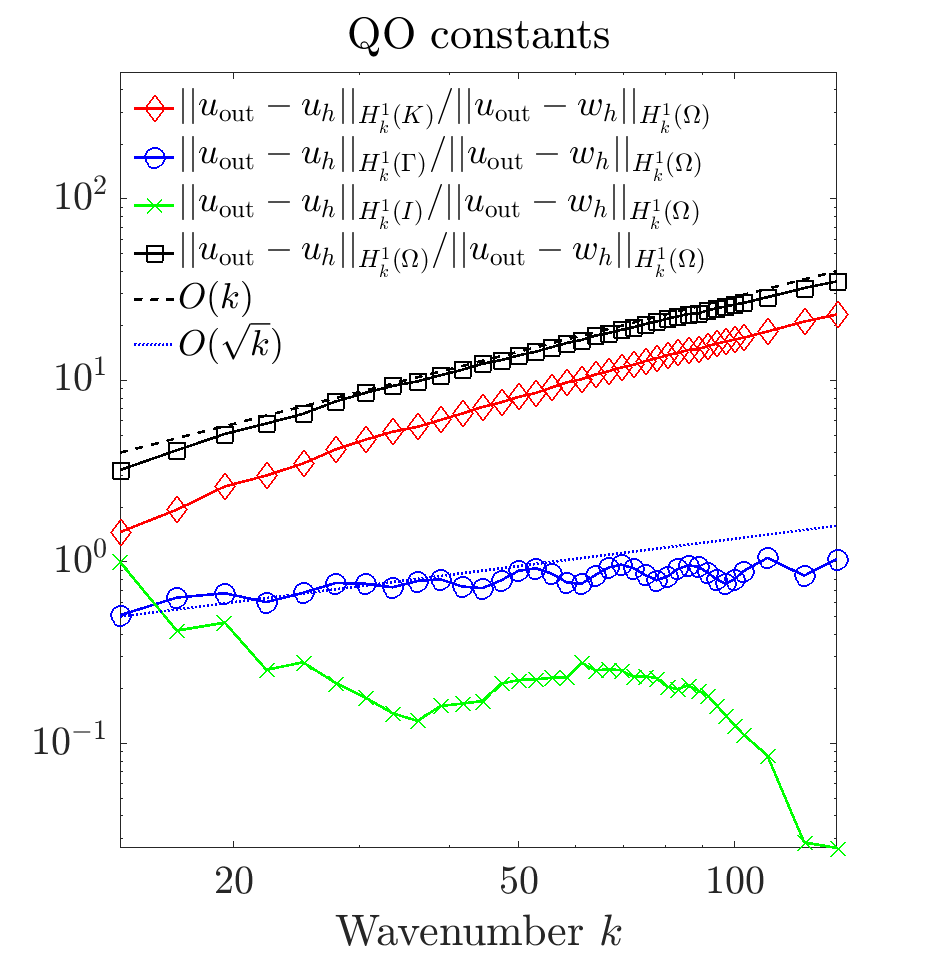}};
				\fill[color=white] (-2.1,.8) rectangle (0,2.9);
				\fill[color=white] (0,1.3) rectangle (1.8,2.9);
				\draw (-2.2,2.6) node[right]{\tiny{$\|u_{\rm out}\!-\!u_h\!\|_{H_k^1(\cavity)}\!/\!\|u_{\rm out}\!-\!w_h\!\|_{H^1_{k}(\Omega)}$}};
		\draw (-2.2,2.25) node[right]{\tiny{$\|u_{\rm out}\!-\!u_h\!\|_{H_k^1(\visible)}\!/\!\|u_{\rm out}-\!w_h\!\|_{H^1_{k}(\Omega)}$}};
		\draw (-2.2,1.9) node[right]{\tiny{$\|u_{\rm out}\!-\!u_h\!\|_{H_k^1(\invisible)}\!/\!\|u_{\rm out}-\!w_h\!\|_{H^1_{k}(\Omega)}$}};
		\draw (-2.2,1.55) node[right]{\tiny{$\|u_{\rm out}\!-\!u_h\!\|_{H_k^1(\Omega)}\!/\!\|u_{\rm out}\!-\!w_h\!\|_{H^1_{k}(\Omega)}$}};
		\draw (-2.2,1.3) node[right]{\tiny{$O(k)$}};
		\draw (-2.2,.95) node[right]{\tiny{$O(k^{0.5})$}};
		\fill[color=white] (-2,-4)rectangle (2.1,-3.1);
		\draw node at (.125,-3.3){Wavenumber $k$};
%		\fill[color=white] (-3.4,-1) rectangle (-2.9,4);
%		\draw node[rotate=90] at (-3.2,.3){$\|u_{\rm in }\|_{H_k^1}$};
		\fill[color=white] (-2.6,3.05) rectangle (2.6,3.5);
		\draw node at (.125,3.3){QO Constants};
	\end{tikzpicture}
	
	\caption{QO constants for $u_{\rm in}$ (left) and $u_{\rm out}$ (right) in the regime U2. Black squares: global QO constant. Red diamonds: QO constant in $\cavity$. Blue circles: QO constant in the visible set. Green crosses: QO constant in $\Gamma$. A priori bounds represented as red dashed lines (for the cavity) and blue dotted lines (away from cavity).}
	\label{fig:U2QO}
\end{figure}

The Galerkin solution is no longer $k$-uniformly quasi-optimal, but the relative error is bounded in terms of $C$, see Table \ref{tab:regimes}. The latter fact is illustrated by the black solid lines in Figure \ref{fig:U2rel}. Furthermore, Corollary \ref{cor:U2} and the conjecture that $\rho(k) = O(k^2)$ imply the a priori bound
\[\|u - u_h\|_{H^1_k(\Omega)} \lesssim k \|u - w_h\|_{H^1_k(\Omega)},\]
for $u = u_{\rm in}$ or $u = u_{\rm out}$, as well as the following bounds on the QO constants
\begin{gather*}
\|u - u_h\|_{H^1_k(\Omega_\cavity)} \lesssim k \|u - w_h\|_{H^1_k(\Omega)},\qquad
\|u - u_h\|_{H^1_k(\Omega_\visible)} \lesssim \sqrt{k} \|u - w_h\|_{H^1_k(\Omega)},\\
\|u - u_h\|_{H^1_k(\Omega_\invisible)} \lesssim \sqrt{k} \|u - w_h\|_{H^1_k(\Omega)}.\end{gather*}
These bounds are in line with the results shown in Figure \ref{fig:U2QO}. The inferred rates in Table \ref{table:inferred} additionally give the following a priori bounds
\[\begin{split}
	\frac{\|u_{\rm in} - u_h\|_{H^1_k(\Omega_\cavity)}}{\|u_{\rm in}\|_{H^1_k(\Omega)}} &\lesssim 1 + \sqrt{\frac{k}{\rho}} \frac{\|u_{\rm in}\|_{H^1_k(\Omega_\visible)}}{\|u_{\rm in}\|_{H^1_k(\Omega)}} 
	+ \left(\frac{k}{\rho}\right)^{3/2}\frac{\|u_{\rm in}\|_{H^1_k(\Omega_\invisible)}}{\|u_{\rm in}\|_{H^1_k(\Omega)}}  \lesssim 1,	\\
	\frac{\|u_{\rm in} - u_h\|_{H^1_k(\Omega_\visible)}}{\|u_{\rm in}\|_{H^1_k(\Omega)}} &\lesssim \underbrace{\sqrt{\frac{k}{\rho}}}_{k^{-\frac12}}\frac{\|u_{\rm in}\|_{H^1_k(\Omega_\cavity)}}{\|u_{\rm in}\|_{H^1_k(\Omega)}} +  \underbrace{\frac{\|u_{\rm in}\|_{H^1_k(\Omega_\visible)}}{\|u_{\rm in}\|_{H^1_k(\Omega)}}}_{k^{-0.5}} +  \underbrace{\frac{k}{\rho}}_{k^{-1}}\frac{\|u_{\rm in}\|_{H^1_k(\Omega_\invisible)}}{\|u_{\rm in}\|_{H^1_k(\Omega)}} 
	 \lesssim k^{-1/2} 	.
\end{split}\]
Similarly, 
\[\frac{\|u_{\rm out} - u_h\|_{H^1_k(\Omega_\cavity)}}{\|u_{\rm out}\|_{H^1_k(\Omega)}}
	 \lesssim 1 , \quad
	\frac{\|u_{\rm out} - u_h\|_{H^1_k(\Omega_\visible)}}{\|u_{\rm out}\|_{H^1_k(\Omega)}} \lesssim k^{-0.5}.\]
These rates are experimentally verified in Figure \ref{fig:U2rel}.
\begin{figure}[htbp]
	\centering
				\begin{tikzpicture}
	\draw node at (0,0){	\includegraphics[width=0.45\textwidth]{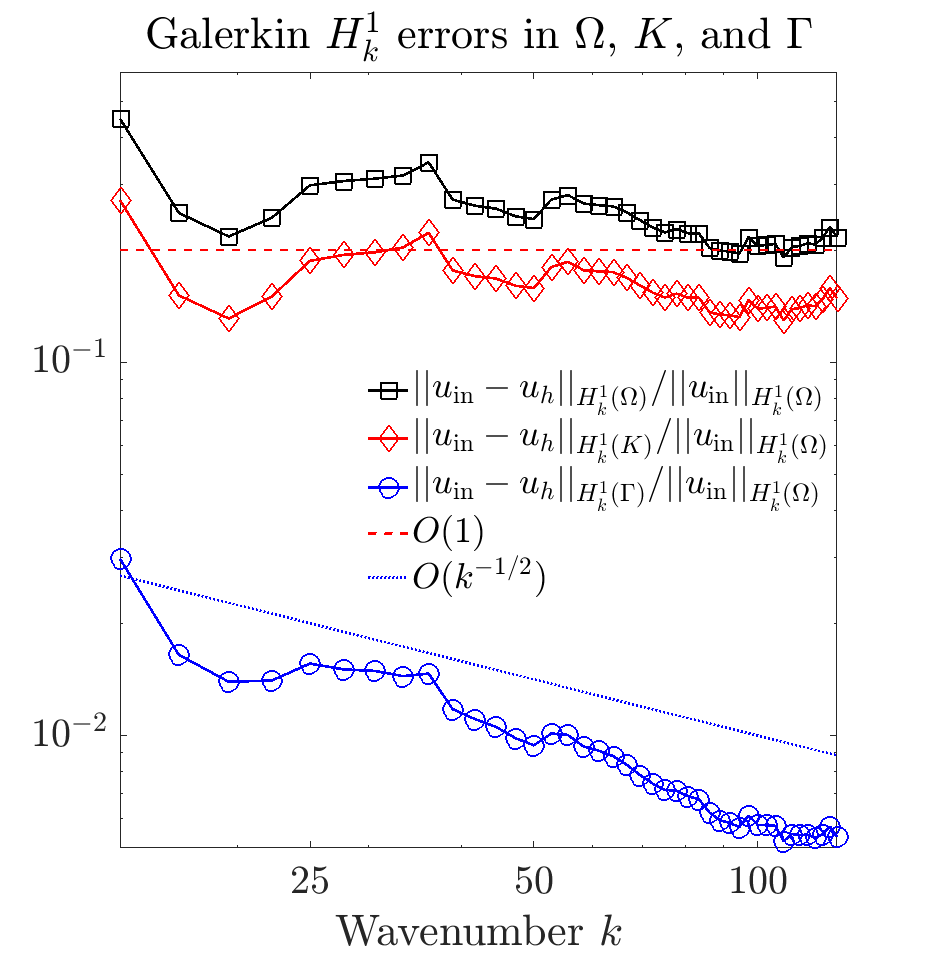}};
				\fill[color=white] (-.4,-1) rectangle (2.65,1);
				\draw (-.5,.55) node[right]{\tiny{$\|\!u_{\rm in}\!-\!u_h\!\|_{H_k^1(\Omega)}\!/\!\|\!u_{\rm in}\!\|_{H^1_{k}(\Omega)}$}};
		\draw (-.5,.2) node[right]{\tiny{$\|\!u_{\rm in}\!-\!u_h\!\|_{H_k^1(\cavity)}\!/\!\|\!u_{\rm in}\!\|_{H^1_{k}(\Omega)}$}};
		\draw (-.5,-.15) node[right]{\tiny{$\|\!u_{\rm in}\!-\!u_h\!\|_{H_k^1(\visible)}\!/\!\|\!u_{\rm in}\!\|_{H^1_{k}(\Omega)}$}};
		\draw (-.5,-.4) node[right]{\tiny{$O(1)$}};
		\draw (-.5,-.7) node[right]{\tiny{$O(k^{-0.5})$}};
		\fill[color=white] (-2,-4)rectangle (2.1,-3.1);
		\draw node at (.125,-3.3){Wavenumber $k$};
%		\fill[color=white] (-3.4,-1) rectangle (-2.9,4);
%		\draw node[rotate=90] at (-3.2,.3){$\|u_{\rm in }\|_{H_k^1}$};
		\fill[color=white] (-2.6,3.05) rectangle (2.6,3.5);
		\draw node at (.125,3.3){Galerkin $H_k^1$ errors in $\Omega$, $\cavity$, and $\visible$};
	\end{tikzpicture}
					\begin{tikzpicture}
	\draw node at (0,0){\includegraphics[width=0.45\textwidth]{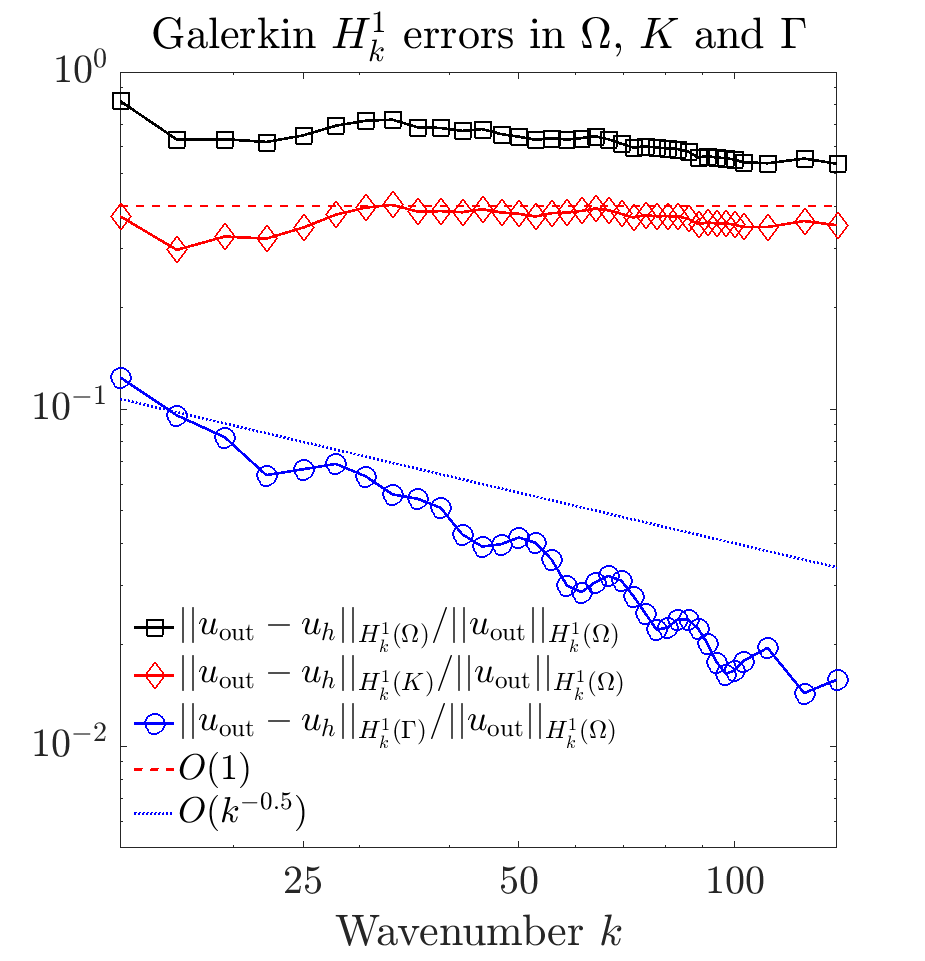}};
				\fill[color=white] (-2.1,-.9) rectangle (1.2,-2.6);
				%\fill[color=white] (-2.1,-1.25) rectangle (1.2,-2.6);
				\draw (-2.2,-1.15) node[right]{\tiny{$\|\!u_{\rm out}\!-\!u_h\!\|_{H_k^1(\Omega)}\!/\!\|\!u_{\rm out}\!\|_{H^1_{k}(\Omega)}$}};
		\draw (-2.2,-1.5) node[right]{\tiny{$\|\!u_{\rm out}\!-\!u_h\!\|_{H_k^1(\cavity)}\!/\!\|\!u_{\rm out}\!\|_{H^1_{k}(\Omega)}$}};
		\draw (-2.2,-1.85) node[right]{\tiny{$\|\!u_{\rm out}\!-\!u_h\!\|_{H_k^1(\visible)}\!/\!\|\!u_{\rm out}\!\|_{H^1_{k}(\Omega)}$}};
		\draw (-2.15,-2.1) node[right]{\tiny{$O(1)$}};
		\draw (-2.15,-2.4) node[right]{\tiny{$O(k^{-0.5})$}};
		\fill[color=white] (-2,-4)rectangle (2.1,-3.1);
		\draw node at (.125,-3.3){Wavenumber $k$};
%		\fill[color=white] (-3.4,-1) rectangle (-2.9,4);
%		\draw node[rotate=90] at (-3.2,.3){$\|u_{\rm in }\|_{H_k^1}$};
		\fill[color=white] (-2.6,3.05) rectangle (2.6,3.5);
		\draw node at (.125,3.3){Galerkin $H_k^1$ errors in $\Omega$, $\cavity$, and $\visible$};
	\end{tikzpicture}
	
	\caption{Local Galerkin errors in the $H^1_k$ norm in $\cavity$ and $\visible$ for the approximation of $u_{\rm in}$ (left) and $u_{\rm out}$ (right) in the regime U2. Black squares: global relative error. Red diamonds: error in the cavity. Blue circles: error away from the cavity. A priori bounds represented as red dashed lines (for the cavity) and blue dotted lines (away from cavity).}
	\label{fig:U2rel}
\end{figure}

%\begin{figure}[H]
%	\centering
%%	\includegraphics[width=0.45\textwidth]{Experiments/U2/beamIn/relativeErrors}
%%	\includegraphics[width=0.45\textwidth]{Experiments/U2/beamOut/relativeErrors}
%	\caption{Local relative errors. Left: source inside. Right: source outside. By Theorem \ref{t:simple}, the local relative error away is bounded by a constant times $k^{-1/2} \norm{u}_{K}/\norm{u}_{\Gamma} + k^{-1} + k^{-1} \norm{u}_{I}/ \norm{u}_\Gamma$. Numerically, the quantity $k^{-1/2} \norm{u}_{K}/\norm{u}_{\Gamma}$ is roughly constant (rather increasing than decreasing), so we don't expect theoretically the slight decrease in local relative error away that we observe numerically, for the beam outside experiment. }
%\end{figure}

\subsubsection{Regime Relative error (RE)}

In RE, we choose
\[(h_\cavity k)^{2p} k^2 = (h_\visible k)^{2p} k^{3/2} = (h_\invisible k)^{2p} k = C,\]
where $C$ is independent of $k$. Figure \ref{fig:REQO} shows the QO constants in each regions for the problems involving $u_{\rm in}$ and $u_{\rm out}$. Figure \ref{fig:RErel} plots local-global relative error.

By Corollary \ref{cor:RE}, the relative error is $k$-uniformly bounded, and this is illustrated by the black solid lines in Figure \ref{fig:RErel}. Furthermore, 
\[\begin{split}
	\frac{\|u_{\rm in} - u_h\|_{H^1_k(\Omega_\cavity)}}{\|u_{\rm in}\|_{H^1_k(\Omega)}} &\lesssim 1 + \frac{\|u_{\rm in}\|_{H^1_k(\Omega_\visible)}}{\|u_{\rm in}\|_{H^1_k(\Omega)}} 
	+ \frac{\|u_{\rm in}\|_{H^1_k(\Omega_\cavity)}}{\|u_{\rm in}\|_{H^1_k(\Omega)}}  \lesssim 1	,\\
	\frac{\|u_{\rm in} - u_h\|_{H^1_k(\Omega_\visible)}}{\|u_{\rm in}\|_{H^1_k(\Omega)}} &\lesssim \underbrace{\sqrt{\frac{k}{\rho}}}_{k^{-\frac12}}\frac{\|u_{\rm in}\|_{H^1_k(\Omega_\cavity)}}{\|u_{\rm in}\|_{H^1_k(\Omega)}} +  \underbrace{\frac{\|u_{\rm in}\|_{H^1_k(\Omega_\visible)}}{\|u_{\rm in}\|_{H^1_k(\Omega)}}}_{k^{-0.5}} + \frac{\|u_{\rm in}\|_{H^1_k(\Omega_\invisible)}}{\|u_{\rm in}\|_{H^1_k(\Omega)}}  \lesssim k^{-1/2}. 	
\end{split}\]
Similarly, 
\[\frac{\|u_{\rm out} - u_h\|_{H^1_k(\Omega_\cavity)}}{\|u_{\rm out}\|_{H^1_k(\Omega)}}
\lesssim 1 , \quad
\frac{\|u_{\rm out} - u_h\|_{H^1_k(\Omega_\visible)}}{\|u_{\rm out}\|_{H^1_k(\Omega)}} \lesssim k^{-\frac12}.\]
These rates are verified in Figure \ref{fig:RErel}.
\begin{figure}[H]
	\centering
					\begin{tikzpicture}
	\draw node at (0,0){\includegraphics[width=0.45\textwidth]{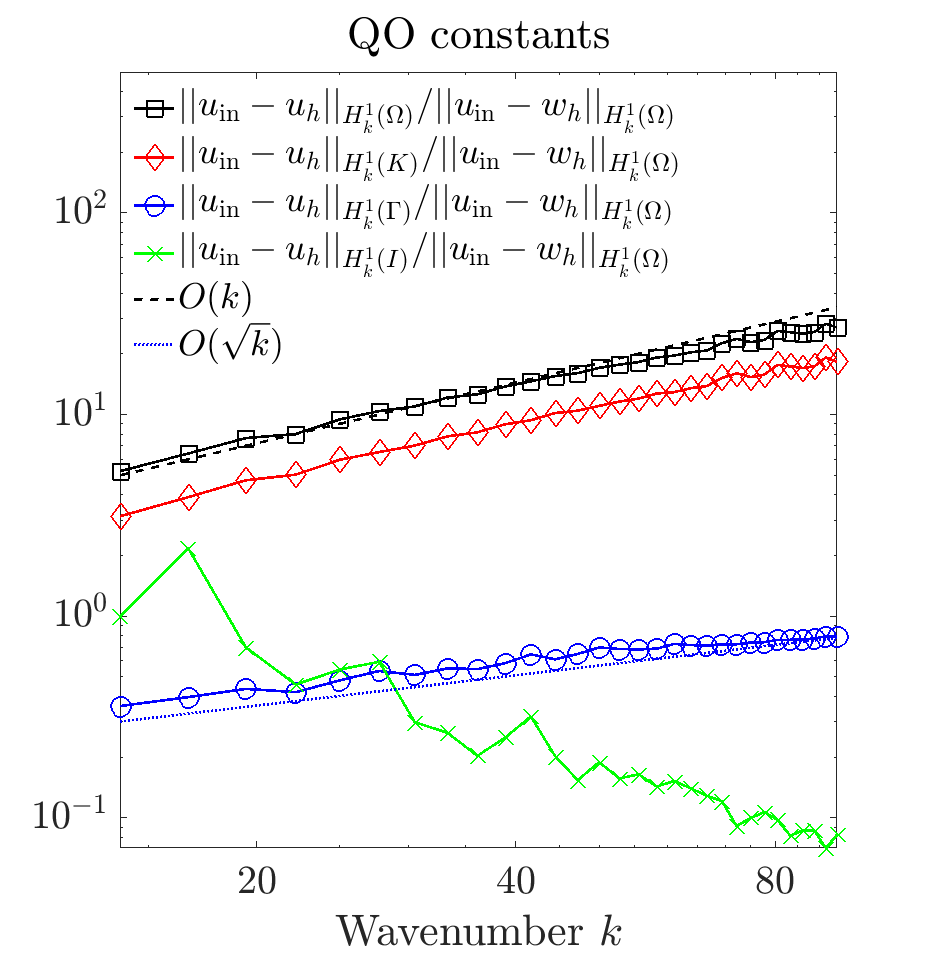}};
				\fill[color=white] (-2.1,.8) rectangle (0,2.9);
				\fill[color=white] (0,1.3) rectangle (1.8,2.9);
				\draw (-2.2,2.6) node[right]{\tiny{$\|u_{\rm in}\!-\!u_h\!\|_{H_k^1(\Omega)}\!/\!\|u_{\rm in}\!-\!w_h\!\|_{H^1_{k}(\Omega)}$}};
		\draw (-2.2,2.25) node[right]{\tiny{$\|u_{\rm in}\!-\!u_h\!\|_{H_k^1(\cavity)}\!/\!\|u_{\rm in}\!-\!w_h\!\|_{H^1_{k}(\Omega)}$}};
		\draw (-2.2,1.9) node[right]{\tiny{$\|u_{\rm in}\!-\!u_h\!\|_{H_k^1(\visible)}\!/\!\|u_{\rm in}-\!w_h\!\|_{H^1_{k}(\Omega)}$}};
		\draw (-2.2,1.55) node[right]{\tiny{$\|u_{\rm in}\!-\!u_h\!\|_{H_k^1(\invisible)}\!/\!\|u_{\rm in}-\!w_h\!\|_{H^1_{k}(\Omega)}$}};
		\draw (-2.2,1.3) node[right]{\tiny{$O(k)$}};
		\draw (-2.2,.95) node[right]{\tiny{$O(k^{0.5})$}};
		\fill[color=white] (-2,-4)rectangle (2.1,-3.1);
		\draw node at (.125,-3.3){Wavenumber $k$};
%		\fill[color=white] (-3.4,-1) rectangle (-2.9,4);
%		\draw node[rotate=90] at (-3.2,.3){$\|u_{\rm in }\|_{H_k^1}$};
		\fill[color=white] (-2.6,3.05) rectangle (2.6,3.5);
		\draw node at (.125,3.3){QO Constants};
	\end{tikzpicture}
					\begin{tikzpicture}
	\draw node at (0,0){\includegraphics[width=0.45\textwidth]{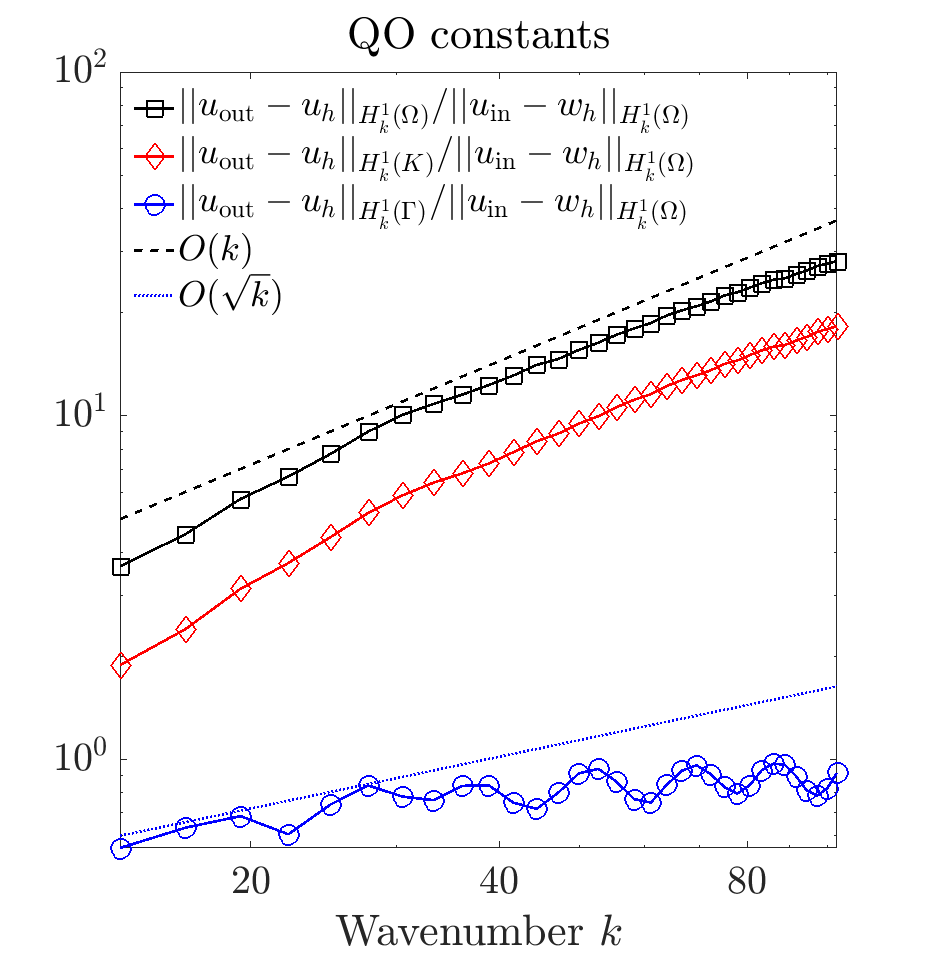}};
				\fill[color=white] (-2.1,.8) rectangle (0,2.9);
				\fill[color=white] (0,1.5) rectangle (1.8,2.9);
				\draw (-2.2,2.6) node[right]{\tiny{$\|u_{\rm out}\!-\!u_h\!\|_{H_k^1(\Omega)}\!/\!\|u_{\rm out}\!-\!w_h\!\|_{H^1_{k}(\Omega)}$}};
		\draw (-2.2,2.25) node[right]{\tiny{$\|u_{\rm out}\!-\!u_h\!\|_{H_k^1(\cavity)}\!/\!\|u_{\rm out}-\!w_h\!\|_{H^1_{k}(\Omega)}$}};
		\draw (-2.2,1.9) node[right]{\tiny{$\|u_{\rm out}\!-\!u_h\!\|_{H_k^1(\invisible)}\!/\!\|u_{\rm out}-\!w_h\!\|_{H^1_{k}(\Omega)}$}};
		%\draw (-2.2,1.55) node[right]{\tiny{$\|u_{\rm out}\!-\!u_h\!\|_{H_k^1(\Omega)}\!/\!\|u_{\rm out}\!-\!w_h\!\|_{H^1_{k}(\Omega)}$}};
		\draw (-2.2,1.65) node[right]{\tiny{$O(k)$}};
		\draw (-2.2,1.3) node[right]{\tiny{$O(k^{0.5})$}};
		\fill[color=white] (-2,-4)rectangle (2.1,-3.1);
		\draw node at (.125,-3.3){Wavenumber $k$};
%		\fill[color=white] (-3.4,-1) rectangle (-2.9,4);
%		\draw node[rotate=90] at (-3.2,.3){$\|u_{\rm in }\|_{H_k^1}$};
		\fill[color=white] (-2.6,3.05) rectangle (2.6,3.5);
		\draw node at (.125,3.3){QO Constants};
	\end{tikzpicture}
	\caption{QO constants for $u_{\rm in}$ (left) and $u_{\rm out}$ (right) in the regime RE. Black squares: global QO constant. Red diamonds: QO constant in $\cavity$. Blue circles: QO constant in the visible set. Green crosses: QO constant in $\Gamma$. A priori bounds represented as red dashed lines (for the cavity) and blue dotted lines (away from cavity).}
	\label{fig:REQO}
\end{figure}

\begin{figure}[H]
	\centering
					\begin{tikzpicture}
	\draw node at (0,0){\includegraphics[width=0.45\textwidth]{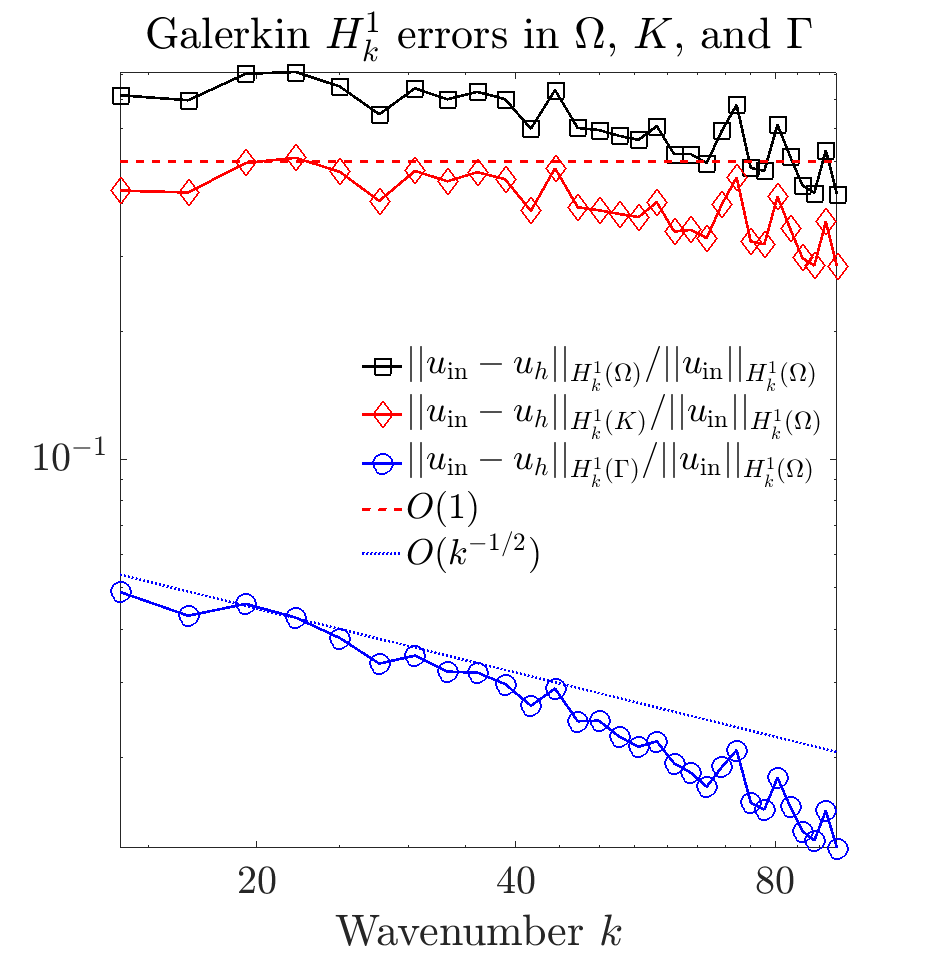}};
				\fill[color=white] (-.4,-1) rectangle (2.65,1);
				\draw (-.5,.75) node[right]{\tiny{$\|\!u_{\rm in}\!-\!u_h\!\|_{H_k^1(\Omega)}\!/\!\|\!u_{\rm in}\!\|_{H^1_{k}(\Omega)}$}};
		\draw (-.5,.4) node[right]{\tiny{$\|\!u_{\rm in}\!-\!u_h\!\|_{H_k^1(\cavity)}\!/\!\|\!u_{\rm in}\!\|_{H^1_{k}(\Omega)}$}};
		\draw (-.5,.05) node[right]{\tiny{$\|\!u_{\rm in}\!-\!u_h\!\|_{H_k^1(\visible)}\!/\!\|\!u_{\rm in}\!\|_{H^1_{k}(\Omega)}$}};
		\draw (-.5,-.2) node[right]{\tiny{$O(1)$}};
		\draw (-.5,-.5) node[right]{\tiny{$O(k^{-0.5})$}};
		\fill[color=white] (-2,-4)rectangle (2.1,-3.1);
		\draw node at (.125,-3.3){Wavenumber $k$};
%		\fill[color=white] (-3.4,-1) rectangle (-2.9,4);
%		\draw node[rotate=90] at (-3.2,.3){$\|u_{\rm in }\|_{H_k^1}$};
		\fill[color=white] (-2.6,3.05) rectangle (2.6,3.5);
		\draw node at (.125,3.3){Galerkin $H_k^1$ errors in $\Omega$, $\cavity$, and $\visible$};
	\end{tikzpicture}
					\begin{tikzpicture}
	\draw node at (0,0){\includegraphics[width=0.45\textwidth]{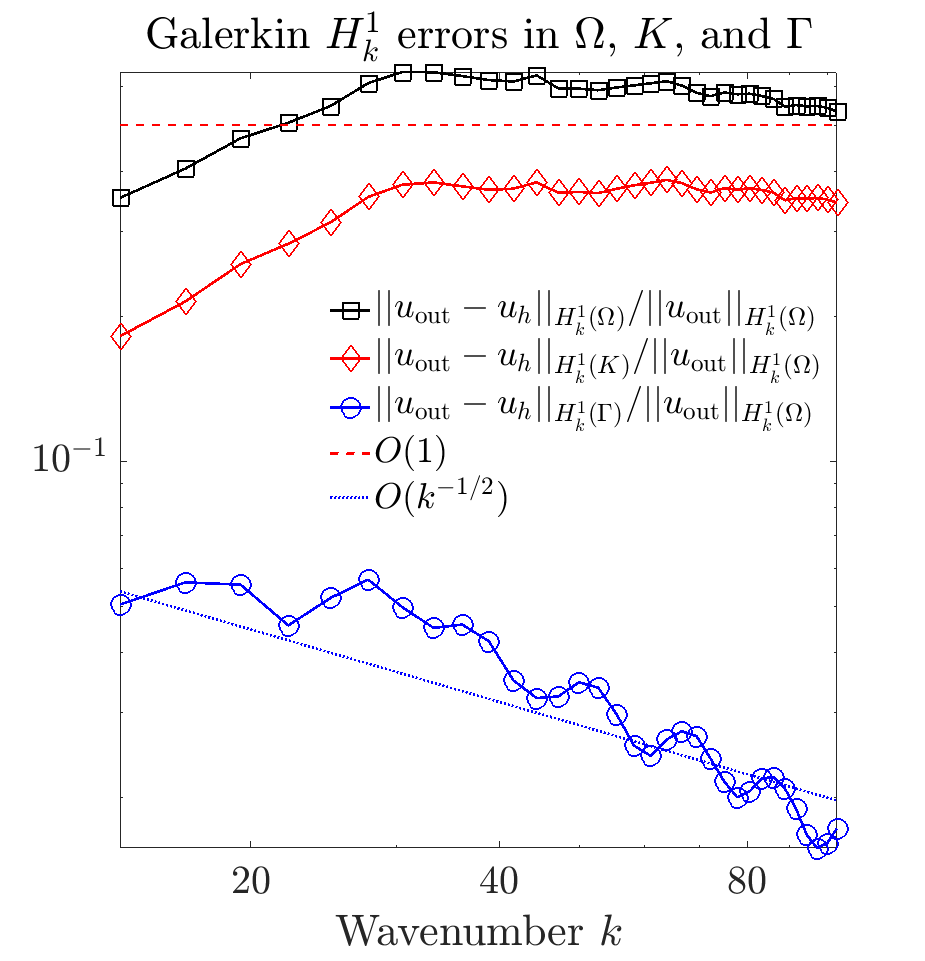}};
				\fill[color=white] (-.65,-.5) rectangle (2.6,1.6);
				%\fill[color=white] (-2.1,-1.25) rectangle (1.2,-2.6);
				\draw (-.8,1.15) node[right]{\tiny{$\|\!u_{\rm out}\!-\!u_h\!\|_{H_k^1(\Omega)}\!/\!\|\!u_{\rm out}\!\|_{H^1_{k}(\Omega)}$}};
		\draw (-.8,.8) node[right]{\tiny{$\|\!u_{\rm out}\!-\!u_h\!\|_{H_k^1(\cavity)}\!/\!\|\!u_{\rm out}\!\|_{H^1_{k}(\Omega)}$}};
		\draw (-.8,.45) node[right]{\tiny{$\|\!u_{\rm out}\!-\!u_h\!\|_{H_k^1(\visible)}\!/\!\|\!u_{\rm out}\!\|_{H^1_{k}(\Omega)}$}};
		\draw (-.8,.2) node[right]{\tiny{$O(1)$}};
		\draw (-.8,-.05) node[right]{\tiny{$O(k^{-0.5})$}};
		\fill[color=white] (-2,-4)rectangle (2.1,-3.1);
		\draw node at (.125,-3.3){Wavenumber $k$};
%		\fill[color=white] (-3.4,-1) rectangle (-2.9,4);
%		\draw node[rotate=90] at (-3.2,.3){$\|u_{\rm in }\|_{H_k^1}$};
		\fill[color=white] (-2.6,3.05) rectangle (2.6,3.5);
		\draw node at (.125,3.3){Galerkin $H_k^1$ errors in $\Omega$, $\cavity$, and $\visible$};
	\end{tikzpicture}
	
	\caption{Local Galerkin errors in the $H^1_k$ norm in $\cavity$ and $\visible$ for the approximation of $u_{\rm in}$ (left) and $u_{\rm out}$ (right) in the regime RE. Black squares: global relative error. Red diamonds: error in the cavity. Blue circles: error away from the cavity. A priori bounds represented as red dashed lines (for the cavity) and blue dotted lines (away from cavity).}
	\label{fig:RErel}
\end{figure}
%We verify that the estimates given by Theorem \ref{t:simple} are sharp. The local worsening factors predicted in theory ($O(\sqrt{\rho})$ in trapping and $O(\sqrt{k})$ away from trapping) are precisely observed in practice (in both problems). 
%
%\subsubsection*{Source inside}
%\begin{figure}[H]
%	\centering
%%	\includegraphics[width=0.45\textwidth]{Experiments/Coarse/beamIn/localWF2} 
%%	\includegraphics[width=0.45\textwidth]{Experiments/Coarse/beamOut/localWF2}
%	\caption{Local worsening factors for the ``source inside" (left) and "source outside" (right) experiments, defined for a subset $U$ as $\|u - u_h\|_{H^1_k(U)}/\|u - w_h\|_{H^1_k(\Omega)}$. }
%\end{figure}
%
%\subsubsection*{Source outside}
%\begin{figure}[H]
%	\centering
%	
%	\caption{Local worsening factors for the ``source outside" experiment.}
%\end{figure}

\subsubsection{Regime Relative error away (RE away)}

In RE away, we choose 
\[(h_\cavity k)^{2p} k^2 = (h_\visible k)^{2p} k = (h_\invisible k)^{2p} k = C,\]
where $C$ is independent of $k$. Figure \ref{fig:REawayQO} shows the QO constants in each regions for the problems involving $u_{\rm in}$ and $u_{\rm out}$. Figure \ref{fig:REawayRE} plots the local-global relative error.
% Galerkin errors in the $H^1_k$ norm in $\cavity$, $\Gamma$ and $\Omega$, normalized by the global $H^1_k$ norm of $u_{\rm in/out}$.

This is the coarsest regime for which Theorem \ref{t:simple} applies. By Corollary \ref{cor:QOCoarse} and the conjecture that $\rho(k) \leq C k^2$, one has the following a priori bounds on the local QO factors:
\begin{gather*}
\|u - u_h\|_{H^1_k(\Omega_\cavity)} \lesssim k \|u - w_h\|_{H^1_k(\Omega)},\qquad\|u - u_h\|_{H^1_k(\Omega_\visible)} \lesssim \sqrt{k} \|u - w_h\|_{H^1_k(\Omega)},\\ 
\|u - u_h\|_{H^1_k(\Omega_\invisible)} \lesssim \sqrt{k} \|u - w_h\|_{H^1_k(\Omega)}.
\end{gather*}
These bounds are experimentally verified in Figure \ref{fig:REawayQO}. By Corollary \ref{cor:REcoarse}, we also have the following a priori bounds on the local relative errors:
\[\begin{split}
	\frac{\|u_{\rm in} - u_h\|_{H^1_k(\Omega_\cavity)}}{\|u_{\rm in}\|_{H^1_k(\Omega)}} &\lesssim 1 + \underbrace{\sqrt{\frac{\rho}{k}}}_{k^{1/2}} \underbrace{\frac{\|u_{\rm in}\|_{H^1_k(\Omega_\cavity)}}{\|u_{\rm in}\|_{H^1_k(\Omega)}}}_{k^{-\frac12}} 
	+ \frac{\|u_{\rm in}\|_{H^1_k(\Omega_\invisible)}}{\|u_{\rm in}\|_{H^1_k(\Omega)}}   \lesssim 1	,\\
	\frac{\|u_{\rm in} - u_h\|_{H^1_k(\Omega_\visible)}}{\|u_{\rm in}\|_{H^1_k(\Omega)}} &\lesssim \underbrace{\sqrt{\frac{k}{\rho}}}_{k^{-\frac12}}\frac{\|u_{\rm in}\|_{H^1_k(\Omega_\cavity)}}{\|u_{\rm in}\|_{H^1_k(\Omega)}} +  \underbrace{\frac{\|u_{\rm in}\|_{H^1_k(\Omega_\visible)}}{\|u_{\rm in}\|_{H^1_k(\Omega)}}}_{k^{-\frac12}} +  \frac{\|u_{\rm in}\|_{H^1_k(\Omega_\invisible)}}{\|u_{\rm in}\|_{H^1_k(\Omega)}} , \lesssim k^{-1/2}.	
\end{split}\]
Similarly, 
\[	\frac{\|u_{\rm out} - u_h\|_{H^1_k(\Omega_\cavity)}}{\|u_{\rm out}\|_{H^1_k(\Omega)}} \lesssim 1 \quad \textup{and} \quad	\frac{\|u_{\rm out} - u_h\|_{H^1_k(\Omega_\visible)}}{\|u_{\rm out}\|_{H^1_k(\Omega)}} \lesssim k^{-1/2}.\]
These bounds are also verified in our numerical experiments, see Figure \ref{fig:REawayRE}.

\begin{figure}[htbp]
	\centering
				\begin{tikzpicture}
	\draw node at (0,0){\includegraphics[width=0.45\textwidth]{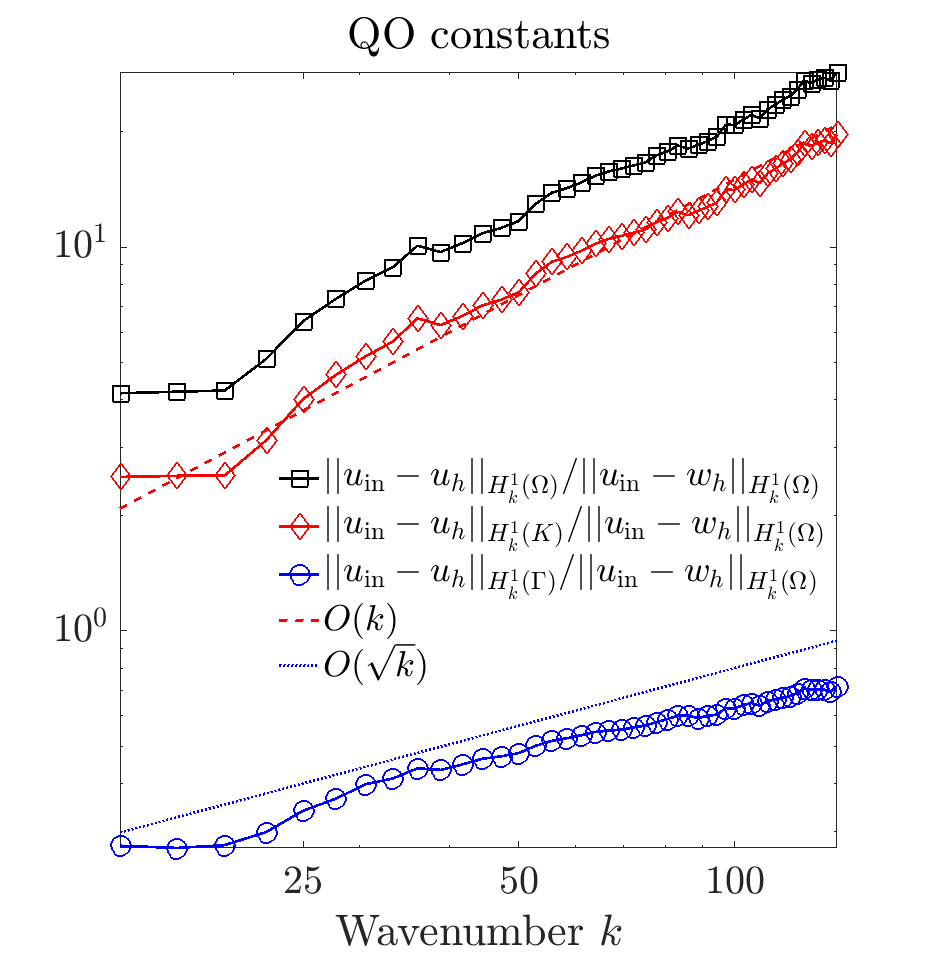}};
				\fill[color=white] (-1.05,.3) rectangle (1,-1.6);
				\fill[color=white] (1,.3) rectangle (2.7,-1);
				\draw (-1.15,-.05) node[right]{\tiny{$\|u_{\rm in}\!-\!u_h\!\|_{H_k^1(\Omega)}\!/\!\|u_{\rm in}\!-\!w_h\!\|_{H^1_{k}(\Omega)}$}};
		\draw (-1.15,-.4) node[right]{\tiny{$\|u_{\rm in}\!-\!u_h\!\|_{H_k^1(\cavity)}\!/\!\|u_{\rm in}\!-\!w_h\!\|_{H^1_{k}(\Omega)}$}};
		\draw (-1.15,-.75) node[right]{\tiny{$\|u_{\rm in}\!-\!u_h\!\|_{H_k^1(\visible)}\!/\!\|u_{\rm in}-\!w_h\!\|_{H^1_{k}(\Omega)}$}};
		\draw (-1.15,-1.0) node[right]{\tiny{$O(k)$}};
		\draw (-1.15,-1.35) node[right]{\tiny{$O(k^{0.5})$}};
		\fill[color=white] (-2,-4)rectangle (2.1,-3.1);
		\draw node at (.125,-3.3){Wavenumber $k$};
%		\fill[color=white] (-3.4,-1) rectangle (-2.9,4);
%		\draw node[rotate=90] at (-3.2,.3){$\|u_{\rm in }\|_{H_k^1}$};
		\fill[color=white] (-2.6,3.05) rectangle (2.6,3.5);
		\draw node at (.125,3.3){QO Constants};
	\end{tikzpicture}	
					\begin{tikzpicture}
	\draw node at (0,0){\includegraphics[width=0.45\textwidth]{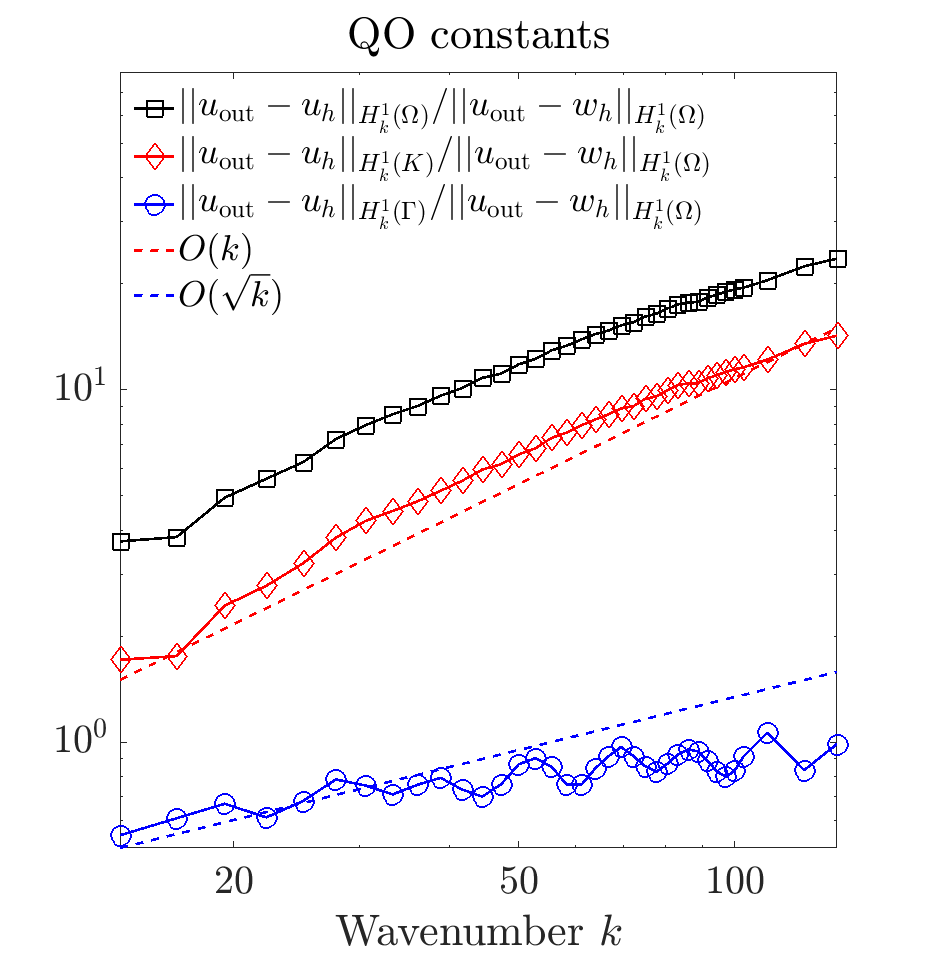}};
				\fill[color=white] (-2.1,.8) rectangle (0,2.9);
				\fill[color=white] (0,1.5) rectangle (1.8,2.9);
				\draw (-2.2,2.6) node[right]{\tiny{$\|u_{\rm out}\!-\!u_h\!\|_{H_k^1(\Omega)}\!/\!\|u_{\rm out}\!-\!w_h\!\|_{H^1_{k}(\Omega)}$}};
		\draw (-2.2,2.25) node[right]{\tiny{$\|u_{\rm out}\!-\!u_h\!\|_{H_k^1(\cavity)}\!/\!\|u_{\rm out}-\!w_h\!\|_{H^1_{k}(\Omega)}$}};
		\draw (-2.2,1.9) node[right]{\tiny{$\|u_{\rm out}\!-\!u_h\!\|_{H_k^1(\invisible)}\!/\!\|u_{\rm out}-\!w_h\!\|_{H^1_{k}(\Omega)}$}};
		%\draw (-2.2,1.55) node[right]{\tiny{$\|u_{\rm out}\!-\!u_h\!\|_{H_k^1(\Omega)}\!/\!\|u_{\rm out}\!-\!w_h\!\|_{H^1_{k}(\Omega)}$}};
		\draw (-2.2,1.65) node[right]{\tiny{$O(k)$}};
		\draw (-2.2,1.3) node[right]{\tiny{$O(k^{0.5})$}};
		\fill[color=white] (-2,-4)rectangle (2.1,-3.1);
		\draw node at (.125,-3.3){Wavenumber $k$};
%		\fill[color=white] (-3.4,-1) rectangle (-2.9,4);
%		\draw node[rotate=90] at (-3.2,.3){$\|u_{\rm in }\|_{H_k^1}$};
		\fill[color=white] (-2.6,3.05) rectangle (2.6,3.5);
		\draw node at (.125,3.3){QO Constants};
	\end{tikzpicture}
	
	\caption{QO constants for $u_{\rm in}$ (left) and $u_{\rm out}$ (right) in the regime RE away. Black squares: global QO constant. Red diamonds: QO constant in $\cavity$. Blue circles: QO constant in the visible set. A priori bounds represented as red dashed lines (for the cavity) and blue dotted lines (away from cavity).}
	\label{fig:REawayQO}
\end{figure}
\begin{figure}[htbp]
	\centering
						\begin{tikzpicture}
	\draw node at (0,0){\includegraphics[width=0.45\textwidth]{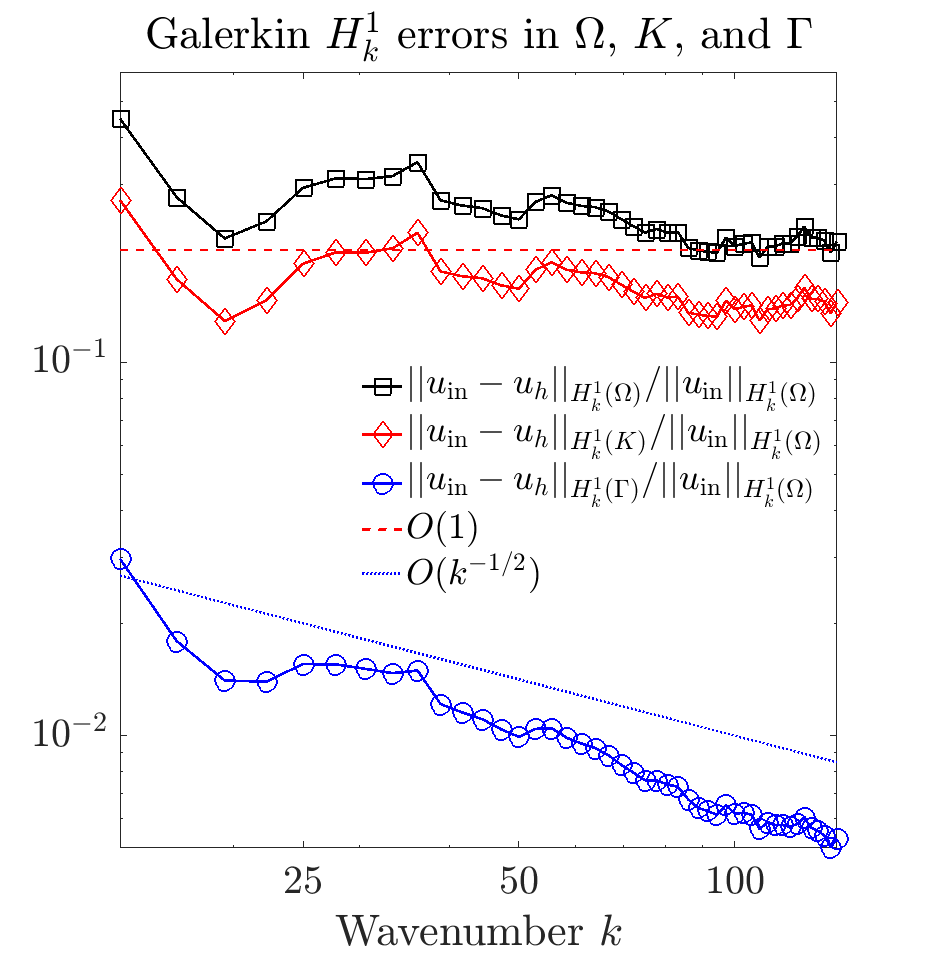}};
				\fill[color=white] (-.4,-1) rectangle (2.65,1);
				\draw (-.5,.6) node[right]{\tiny{$\|\!u_{\rm in}\!-\!u_h\!\|_{H_k^1(\Omega)}\!/\!\|\!u_{\rm in}\!\|_{H^1_{k}(\Omega)}$}};
		\draw (-.5,.25) node[right]{\tiny{$\|\!u_{\rm in}\!-\!u_h\!\|_{H_k^1(\cavity)}\!/\!\|\!u_{\rm in}\!\|_{H^1_{k}(\Omega)}$}};
		\draw (-.5,-.1) node[right]{\tiny{$\|\!u_{\rm in}\!-\!u_h\!\|_{H_k^1(\visible)}\!/\!\|\!u_{\rm in}\!\|_{H^1_{k}(\Omega)}$}};
		\draw (-.5,-.35) node[right]{\tiny{$O(1)$}};
		\draw (-.5,-.65) node[right]{\tiny{$O(k^{-0.5})$}};
		\fill[color=white] (-2,-4)rectangle (2.1,-3.1);
		\draw node at (.125,-3.3){Wavenumber $k$};
%		\fill[color=white] (-3.4,-1) rectangle (-2.9,4);
%		\draw node[rotate=90] at (-3.2,.3){$\|u_{\rm in }\|_{H_k^1}$};
		\fill[color=white] (-2.6,3.05) rectangle (2.6,3.5);
		\draw node at (.125,3.3){Galerkin $H_k^1$ errors in $\Omega$, $\cavity$, and $\visible$};
	\end{tikzpicture}
						\begin{tikzpicture}
	\draw node at (0,0){\includegraphics[width=0.45\textwidth]{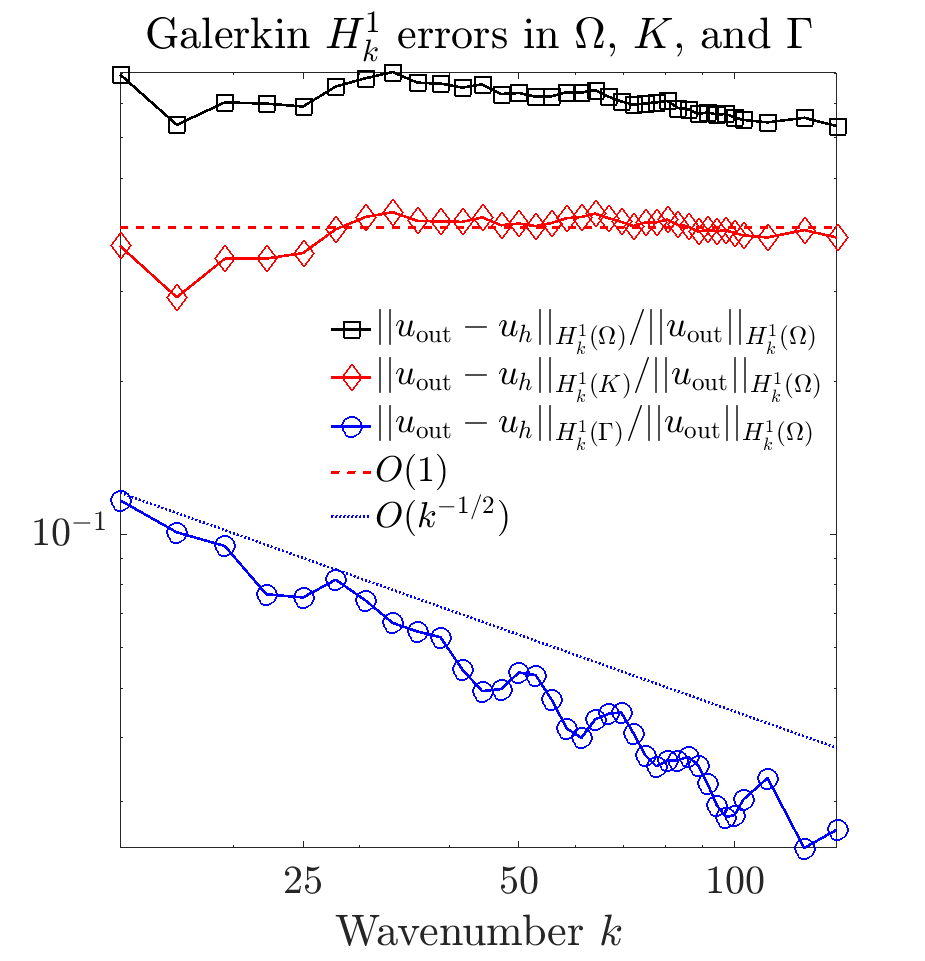}};
				\fill[color=white] (-.65,-.5) rectangle (2.6,1.6);
				%\fill[color=white] (-2.1,-1.25) rectangle (1.2,-2.6);
				\draw (-.8,1.05) node[right]{\tiny{$\|\!u_{\rm out}\!-\!u_h\!\|_{H_k^1(\Omega)}\!/\!\|\!u_{\rm out}\!\|_{H^1_{k}(\Omega)}$}};
		\draw (-.8,.7) node[right]{\tiny{$\|\!u_{\rm out}\!-\!u_h\!\|_{H_k^1(\cavity)}\!/\!\|\!u_{\rm out}\!\|_{H^1_{k}(\Omega)}$}};
		\draw (-.8,.3) node[right]{\tiny{$\|\!u_{\rm out}\!-\!u_h\!\|_{H_k^1(\visible)}\!/\!\|\!u_{\rm out}\!\|_{H^1_{k}(\Omega)}$}};
		\draw (-.775,.05) node[right]{\tiny{$O(1)$}};
		\draw (-.775,-.25) node[right]{\tiny{$O(k^{-0.5})$}};
		\fill[color=white] (-2,-4)rectangle (2.1,-3.1);
		\draw node at (.125,-3.3){Wavenumber $k$};
%		\fill[color=white] (-3.4,-1) rectangle (-2.9,4);
%		\draw node[rotate=90] at (-3.2,.3){$\|u_{\rm in }\|_{H_k^1}$};
		\fill[color=white] (-2.6,3.05) rectangle (2.6,3.5);
		\draw node at (.125,3.3){Galerkin $H_k^1$ errors in $\Omega$, $\cavity$, and $\visible$};
	\end{tikzpicture}
	\caption{Local Galerkin errors in the $H^1_k$ norm in $\cavity$ and $\visible$ for the approximation of $u_{\rm in}$ (left) and $u_{\rm out}$ (right) in the regime RE away. Black squares: global relative error. Red diamonds: error in the cavity. Blue circles: error away from the cavity. A priori bounds represented as red dashed lines (for the cavity) and blue dotted lines (away from cavity).}
	\label{fig:REawayRE}
\end{figure}

\subsection{An adaptive mesh-refinement algorithm}

The numerical experiments show that Theorem~\ref{t:simple} accurately captures the effect of local best approximation errors on the local Galerkin errors. 
It is therefore natural to use Theorem~\ref{t:simple} to inform an adaptive refinement algorithm. This will be investigated elsewhere, but we sketch the main steps here.
Given a set $U\subset \Omega$ on which one wants an accurate solution, implement the following.
\begin{enumerate}
\item Use ray tracing to (a) identify the regions $\cavity, \visible,$ and $\invisible$ and (b) estimate $\rho$ \blueOld{(see Remark \ref{r:findKV} below)}.
\item Compute a Galerkin solution, $u^0_h$ on a rough mesh and let $j=0$.
\item Compute $\|u^j_h\|_{\Omega_{\cavity}}$, $\|u^j_h\|_{\Omega_{\visible}}$,  $\|u^j_h\|_{\Omega_{\invisible}}$, and  $\|u^j_h\|_{\Omega_{\pml}}$.
\item Assuming that $\|u\|_{\Omega_j}\propto \|u_h\|_{\Omega_j}$, use the standard approximation property $\|u-w_h\|_{H_k^1}\leq C(hk)^p\|u\|_{H_k^{p+1}}$ %~\cite{BrennerScott} 
and associated lower bounds~\cite{G1} to obtain bounds on the vector of best approximation errors on the hand side of~\eqref{e:simple}.
\item Put the bounds from Step 3 into Theorem~\ref{t:simple} to give an estimator for the map 
\beqs
(h_\cavity, h_\visible, h_\invisible, h_\pml) \mapsto \Big( \| u-u_h\|_{H^1_k(\Omega_\cavity')}, \| u-u_h\|_{H^1_k(\Omega_\visible')},\| u-u_h\|_{H^1_k(\Omega_\invisible')},\| u-u_h\|_{H^1_k(\Omega_\pml')}\Big).
\eeqs
Use this map (e.g., via a penalised optimisation process) to determine what mesh refinement will be effective for reducing the error in $U$. 
\item Solve the problem on the new mesh to obtain $u_h^{j+1}$.
\item Set $j=j+1$ and repeat Steps 3-6 until the desired accuracy is achieved.
\end{enumerate}
The constants in Theorem~\ref{t:simple} are not given explicitly. 
However, we believe that replacing all the constants by one (or possibly adaptively tuning the constants) will produce an effective adaptive refinement for a fixed $k$ (large enough). 

\begin{remark}
\label{r:findKV}
To accomplish Step 1 (a), one can use a set of sample points and directions $x_i\in \Omega$, $i=1,\dots, N$ and $\xi_j\in S^{d-1}$, $j=1,\dots,M$ at a fine scale, $\delta$ and fix a maximal time $T_{\max}>0$ and then run a ray tracing algorithm from each sample point. \blueOld{Define} $t_{ij}:\min(T_{\max},T_{ij,\mc{E}})$, where $T_{ij,\mc{E}}$ is the time at which the ray from $(x_i,\xi_j)$ enters the PML region.  Define $I_\visible:=\{ i\,:\max_{j}t_{ij}=T_{\max}\}$. We set $\Omega_{\cavity}$ to be a neighbourhood of $\{x_i\,:\,i\in I_{\visible}\}$, and $\Omega_{\visible}$ a neighbourhood of the complement of $\Omega_{\cavity}$ intersected with the rays from $\{x_i\}_{i\in I_{\cavity}}$. 

 For Step 1(b) one can use the heuristic that 
 \begin{equation}
 \label{e:estim_rho_rays}
 \rho \lesssim kV^{-1}(k^{-(\blueOld{d-1})})\,,
 \end{equation}
 where $V^{-1}$ is the inverse of $t \mapsto V(t)$, and $V(t)$ is the volume of the set points in $\Omega\times S^{d-1}$ that do not enter the PML in time $t$. To estimate $V(t)$, we use the approximation
 $$
 V(t)\sim\widetilde{V}(t):=\delta^{2d-1} \#\{(x_{i},\xi_j)\,:\, t_{ij}\geq t   \}.
 $$
The heuristic \eqref{e:estim_rho_rays} is valid at least for some special cases, including some cases of the weakest form of trapping where trajectories escape exponentially fast,
and certain geometries that are warped products \cite{ChWu:13}. Furthermore, one place where $V(t)$ rigorously appears is in fractal upper bounds on the number of resonances near the real axis \cite{DyGa:17}. However, a precise characterisation of $\rho$ via billiard dynamics is a challenging open problem.

% It is unclear whether the heuristic \eqref{e:estim_rho_rays} accurately captures the behavior of $\ma{\rho}$ \st{$V(t)$} in general, but it is a natural attempt since it agrees in a variety of 
\end{remark}

\section{Assumptions and statement of the main result}
\label{sec:state_main_result}
We now gather some definitions and assumptions and state our main result, Theorem \ref{t:theRealDeal}.

\subsection{The Helmholtz PML operators}\label{sec:3.1}

Throughout this paper, $\Omega_{-} \subset \R^d$ (the obstacle) denotes a bounded open set with $C^\infty$ boundary and connected complement. Let $\Omega_{\tr}\Subset \mathbb{R}^d$ (the truncation domain) be a bounded open set with $\Omega_-\Subset \Omega_{\tr}$ and define $\Omega:=\Omega_{\tr}\setminus\overline{\Omega}_-$ (the computational domain). Let $\Gamma_{\tr}=\partial\Omega_{\tr}$ so that $\partial\Omega$ \blueOld{is the disjoint union of} $\partial\Omega_-$ \blueOld{and} $\Gamma_{\tr}$. For all $k > 0$ and $n\geq 0$, and given $U \subset \R^d$, let $H^n_k(U)$ (abbreviated $H^n_k$ when $U = \Omega$) be the completion of $C^\infty(U)$ with respect to the norm \eqref{e:weightedNorm},
%``semiclassical norm'' (where $k^{-1}$ is the semiclassical parameter)
%\[\|u\|^2_{H^n_k(U)} := \sum_{|\alpha|\leq n} k^{-2\abs{\alpha}} \|\partial^\alpha u\|^2_{L^2(U)},\]
and let $H_k^{-n}(U)$ be the normed dual of $H_k^n(U)$, with, as usual, $L^2(\Omega)$ identified through the $L^2$ pairing with a subspace of $H_k^{-n}$ for all $n \geq 0$. \footnote{We highlight that this is not the standard notation, as $H^{-n}$ usually denotes the dual of $H_0^n = \overline{C^\infty_c(\Omega)}^{H^n_k}$. \blueOld{One can show that the dual of $H^n(U)$ is given by $\overline{C_c^\infty(U)}^{H^{-n}(\mathbb{R}^n)}$~\cite[Theorem 3.30]{Mc:00}; i.e., it is the set of supported $H^{-n}$ distributions in $U$.} } 

\blue{Let $A$ be a smooth, symmetric, positive-definite matrix with real coefficients that equals the identity outside a compact set, 
and let $n$ be a smooth, strictly positive, real-valued function that equals one outside a compact set.} Let $a_k: H^1_k\times H^1_k \to \mathbb{C}$ be the sesquilinear form defined by 
\begin{equation}
\label{e:def_ak}
a_k(u,v) := \blue{\re^{\ri \omega}}\int_{\Omega} \Big(k^{-2}A_\theta(x) \nabla u(x) \cdot \overline{\nabla v(x)} 
+k^{-2}\langle b_\theta(x),\nabla u\rangle \overline{v(x)}
- n_\theta(x) u(x) \overline{v(x)}\Big)\,dx,
\end{equation}
where $A_\theta$, $b_\theta$, and $n_\theta$ are defined \blue{in terms of $A$ and $n$ in \S \ref{sec:PML}, and $\omega\in \Rea$.
The presence of the $\re^{\ri \omega}$ in \eqref{e:def_ak} does not change the Galerkin solution since 
\beqs
a_k(u-u_h, v_h) = 0 \quad \text{ if and only if }\quad \re^{-\ri \omega} a_k(u-u_h, v_h) = 0,
\eeqs
but is convenient for the analysis since it guarantees $a_k$ satisfies the standard G\aa rding inequality -- see \eqref{e:Garding1} below.}

To cover both Dirichlet (so-called ``sound-soft'') and Neumann (``sound-hard'') obstacles, we consider a subspace $\cZ_k \subset H^1_k$, which can be either given by $\cZ_k = \Zspaced{}$ or $\cZ_k = \Zspacen{}$, where
\begin{equation}
\label{e:defZk_concrete}
\Zspaced{}:=H_{0,k}^1(\Omega),\qquad\tand\qquad \Zspacen{}:=\overline{\{u\in C^\infty(\overline{\Omega})\,:\,\supp u\cap \Gamma_{\tr}=\emptyset\}}^{H_k^1},
\end{equation}
and let $\cZ_k^*$ be the normed dual of $\cZ_k$. The Helmholtz operator $P_k:\Zspace{} \to (\Zspace{})^*$ is then defined as the linear operator associated to $a_k$, i.e.
$$P_k : \cZ_k \to \cZ_k^*\,, \quad \langle P_k u,v\rangle := a_k(u,v) \quad \tfa u,v \in \cZ_k.$$
If $P_k$ is invertible, we denote by $R_k := (P_k)^{-1} : \cZ_k^* \to \cZ_k$ its inverse (also known as the resolvent), and let
\begin{equation*}
%\label{e:def_rho(k)}
\rho(k) := \|R_k\|_{L^2 \to L^2} = \sup_{f \in L^2(\Omega) \setminus \{0\}} \frac{\|P_k^{-1} f\|_{L^2(\Omega)}}{\|f\|_{L^2(\Omega)}}
\end{equation*}

Our main result holds for $k$ ranging in a subset $\R_+\setminus \blueOld{\mathrm{J}}$ of the positive real numbers on which $\rho(k)$ is polynomially bounded; i.e., we make the following assumption on $\blueOld{\mathrm{J}}$.
\begin{assumption}[Polynomial bound on the resolvent]
\label{a:polyBound}
\blueOld{The set $\blueOld{\mathrm{J}}\subset (0,\infty)$ is such that for all $k_0>0$,
$$
\sup_{k\in[k_0,\infty)}\frac{\log \rho(k)}{\log k}<\infty.
$$
i.e. $rho$ is polynomially bounded in $k$.
}
\end{assumption}

In this paper, we are then interested in the error in the finite-element approximation solution (see next paragraph) of the variational problem, for $k \in \R_+ \setminus \blueOld{\mathrm{J}}$:
\begin{equation}
\label{e:var_pb}
\text{find $u \in \cZ_k$ such that, for all $v \in \cZ_k$, $a_k(u,v) = F(v)$},
\end{equation}
where $F : \cZ_k \to \cZ_k^*$ is a continuous anti-linear form.

\subsection{Finite-element approximation}

We consider a Galerkin approximation of the variational problem \eqref{e:var_pb}. Following the practice in the local FEM error analysis literature (see in particular \cite[Assumptions A.1--A.3]{NiSc:74}, and also \cite{Wa:91,DeGuSc:11,Br:20}), and following closely \cite{AvGaSp:24}, we describe $\blueOld{V_{\mathcal{T}}}$ through a set of standard assumptions as follows. Throughout, a fixed positive integer $p$,  modelling the polynomial degree of the finite-element subspace, is chosen independently of $k$; hence this setting models \blueOld{an} ``$h$-version" of the FEM.

If $U \subset \Omega$ is an open set, define
%\[\Zspace{j}(U) := \big\{ u|_{U} \, \textup{ such that }\, u\in \Zspace{j} \big\}\]
\[C^{\infty}_<(U):= \big\{\chi \in C^\infty(\overline{\Omega}) \textup{ such that } \supp \chi \subset \overline{U} \tand \partial_<(\supp \chi,U) > 0\big\}\]
\[ \Zspace{1,<}(U) := \overline{\big\{ v \in \Zspace{} \, \textup{ s.t. }\, \supp v \subset \overline{U} \,\tand\, \partial_<(\supp v, U) > 0 \big\}},\]
where the closure is taken with respect to the $\Zspace{\blueOld{1}}$ norm, and for any subsets $\Omega_0 \subset \Omega_1 \subset \Omega$,
\begin{equation}
\label{e:def_partial}
\partial_< (\Omega_0,\Omega_1) := \textup{dist}(\,\partial \Omega_0 \setminus \partial \Omega\,,\,\,\partial \Omega_1 \setminus \partial \Omega\,).
\end{equation}

\begin{definition}
\label{d:mesh}
A {\em \blueOld{mesh}} $\mathcal{T}$ of $\Omega$ is a set of pairwise disjoint open subsets $\blueOld{T} \subset \Omega$ such that
\[\bigcup_{\blueOld{T} \in \mathcal{T}} \overline{\blueOld{T}} = \overline{\Omega}.\]
\end{definition}
We denote by $h_\blueOld{T}$ the diameter of $\blueOld{T} \in \cT$ and define the global meshwidth
$$\blueOld{h=h(\cT) :=\max_{\blueOld{T} \in \cT} h_\blueOld{T}}$$
\blueOld{A {\em finite-element space} $V_{\mathcal{T}}$ over a mesh $\mathcal{T}$ of $\Omega$ is a finite-dimensional subspace $V \subset \mathcal{H}$ such that for every $u \in V$ and $\blueOld{T} \in \mathcal{T}$, $u|_{\blueOld{T}} \in C^\infty(\overline{\blueOld{T}})$. If $V$ is a finite-element space and $U \subset \Omega$, define
\[V_k^<(U) := \Zspace{1,<}(U) \cap V.\] }

\blueOld{
\begin{definition}
    We define the uniformity constant of a mesh $\mathcal{T}$ at scale $r$, $\mathcal{U}(\mathcal{T},r)$, by 
    $$
    \mathcal{U}(\mathcal{T},r):=\sup \big\{ \tfrac{h_{\blueOld{T}_1}}{h_{\blueOld{T}_2}}\,:\, \blueOld{T}_1,\blueOld{T}_2\in\mathcal{T},\, d(\blueOld{T}_1,\blueOld{T}_2)<r\Big\}.
    $$
\end{definition}
}
% \begin{assumption}[Sub-wavelength grid]
% 	\label{ass:swlg}
% 	For all $k_0 > 0$, there exists a positive constant $\Cfem > 0$ such that for all $k \geq k_0$
% 	\[
% h \leq \Cfem k^{-1}.
% 	\]
% \end{assumption}
% \begin{assumption}[Wavelength-scale quasi-uniformity]
% 	\label{ass:qu}
% 	For all $R > 0$ and all $k_0 > 0$, there exists $C > 0$ such that for all $k \geq k_0$ and any elements $K,K'$ of $\cT_k$ such that $\dist(K,K') \leq Rk^{-1}$, 
% 	\[\frac{1}{\Cfem} \leq \frac{h_K}{h_{K'}} \leq \Cfem.\] 
% \end{assumption}

\blueOld{\begin{definition}[Order-$p$-approximation]
	\label{d:app}
    Let $C_0,k,\kappa>0$. We say that a finite-element space $V_{\mathcal{T}}$ over a mesh $\mathcal{T}$ has the $p$-approximation property at frequency $k$ with constants $(C_0,\kappa)$ if $V_{\mathcal{T}}\subset \mathcal{Z}_k$ and for all $j \in \{1,\ldots,p+1\}$, all  all $u \in \Zspace{1}\cap \Hspace{j}$, there exists $u_h \in V$ such that
	\[\sum_{\blueOld{T} \in \mathcal{T}}(h_\blueOld{T} k)^{2(1-j)} \|u - u_h\|_{H^{1}_k(\blueOld{T})}^{2} \leq C_0 \|u\|^2_{H_k^j}.\]
	Furthermore, given subsets $U_0 \subset U_1 \subset \Omega$ such that 
	$$\partial_<(U_0,U_1) \geq \kappa \max \big\{h_\blueOld{T} \,\,|\,\, \blueOld{T} \in \mathcal{T} \,\,\textup{ s.t. }\,\, \blueOld{T} \cap U_1 \neq \emptyset\big\},$$  
	if $\textup{supp}\,u \subset U_0 \cup \partial \Omega$, then $u_h$ can be chosen such that $\supp u_h \subset U_1 \cup \partial \Omega$.    
\end{definition}}

% \begin{assumption}[Approximation property]
% 	\label{ass:ap}
% 	There exists $\kappa > 0$ such that for every $k_0 > 0$, there exists $\Cfem > 0$ such that for all $j \in \{1,\ldots,p+1\}$, all $m \in \{0,\ldots,j\}$, all $k \geq k_0$ and all $u \in \Zspace{1}\cap \Hspace{j}$, there exists $u_h \in V_k$ such that
% 	\[\sum_{K \in \mathcal{T}_k}(h_K k)^{2(m-j)} \|u - u_h\|_{H^m_k(K)}^{2} \leq \Cfem \|u\|^2_{H_k^j}.\]
% 	Furthermore, given subsets $U_0 \subset U_1 \subset \Omega$ such that 
% 	$$\partial_<(U_0,U_1) \geq \kappa \max \big\{h_K \,\,|\,\, K \in \mathcal{T}_k \,\,\textup{ s.t. }\,\, K \cap U_1 \neq \emptyset\big\},$$  
% 	if $\textup{supp}\,u \subset U_0 \cup \partial \Omega$, then $u_h$ can be chosen such that $\supp u_h \subset U_1 \cup \partial \Omega$.
% \end{assumption}
\blueOld{
\begin{definition}[Super-approximation property]
\label{d:sap}
 Let $C_0,k,\kappa>0$. Then a finite-element space $V_{\mathcal{T}}$ over a mesh $\mathcal{T}$ has the \emph{super-approximation property at frequency $k$ with constants $(C_0,\kappa)$} if
	for any subsets $U_0 \subset U_1 \subset \Omega$ such that
	\[2>d:= \partial_<(U_0,U_1) \geq \kappa  \max \big\{h_\blueOld{T} \,\,|\,\, \blueOld{T} \in \mathcal{T}_k \,\,\textup{ s.t. }\,\, \blueOld{T} \cap U_1 \neq \emptyset\big\},\]
	and any $C_{\dagger}>0$ if $\chi \in C^{\infty}_<(U_0)$ is such that,  \beq\label{e:chiBound}\max_{|\alpha| = n} \|\partial^\alpha \chi \|_\infty \leq \frac{C_\dagger}{d^n}, \tfor n = 0,\ldots,p,
    \eeq
	then for any $u_h \in V$, there exists $v_h \in V_k^{<}(U_1)$ such that
	\[\|\chi^2 u_h - v_h\|_{H^1_k(\blueOld{T})} \leq C_0 C_\dagger \frac{h_\blueOld{T}}{d} \left[C_{\dagger}\left(1 + \frac{1}{kd}\right)\|u_h\|_{L^2(\blueOld{T})} + \|\chi u_h\|_{H^1_k(\blueOld{T})}\right]\,\quad \tfa \blueOld{T} \in \cT_k.\]
\end{definition}
% \begin{assumption}[Super-approximation property]
% 	\label{ass:sap}
% 	There exists $\kappa > 0$ such that for all $k_0 > 0$ and $C_{\dagger} > 0$, there exists $\Cfem > 0$ such that for all $k \geq k_0$ and any subsets $U_0 \subset U_1 \subset \Omega$ such that
% 	\[d:= \partial_<(U_0,U_1) \geq \kappa  \max \big\{h_K \,\,|\,\, K \in \mathcal{T}_k \,\,\textup{ s.t. }\,\, K \cap U_1 \neq \emptyset\big\},\]
% 	if $\chi \in C^{\infty}_<(U_0)$ is such that,  
% 	\[\max_{|\alpha| = n} \|\partial^\alpha \chi \|_\infty \leq \frac{C_\dagger}{d^n}, \tfor n = 0,\ldots,p,\]
% 	then for any $u_h \in V_k$, there exists $v_h \in V_k^{<}(U_1)$ such that
% 	\[\|\chi^2 u_h - v_h\|_{H^1_k(K)} \leq \Cfem \frac{h_K}{d} \left[\left(1 + \frac{1}{kd}\right)\|u_h\|_{L^2(K)} + \|\chi u_h\|_{H^1_k(K)}\right]\,\quad \tfa K \in \cT_k.\]
% \end{assumption}
}

\blueOld{
\begin{definition}[Inverse inequality]
	\label{d:ii}
    Let $C_0,k>0$. Then a finite element space $V_{\mathcal{T}}$ over a mesh $\mathcal{T}$ has \emph{inverse inequalities at frequency $k$ with constant $C_0$} if for all $\blueOld{T} \in \mathcal{T}$, all $u_h \in V$ and all $j \in \{0,1,\ldots,p\}$, 
	\begin{align*}
		\|u_h\|_{H^1_k(\blueOld{T})} \leq \frac{C_0}{h_\blueOld{T} k}\|u_h\|_{L^2(\blueOld{T})}\quad\tand\quad
		\|u_h\|_{L^2(\blueOld{T})} \leq \frac{C_0}{(h_\blueOld{T} k)^j} \|u_h\|_{H_k^{-j}(\blueOld{T})},
	\end{align*}
	where $\|u_h\|_{H_k^{-j}(\blueOld{T})} := \sup_{v \in C^\infty_c(\blueOld{T})} (\abs{\int_{\blueOld{T}} u_h v\,dx}/\|v\|_{H^j_k(\blueOld{T})})$.
\end{definition}
% \begin{assumption}[Inverse inequality on elements]
% 	\label{ass:ii}
% 	There exists $\Cfem$ such that for all $k > 0$, all $K \in \mathcal{T}_k$, all $u_h \in V_k$ and all $j \in \{0,1,\ldots,p\}$, 
% 	\begin{align*}
% 		\|u_h\|_{H^1_k(K)} \leq \frac{\Cfem}{h_K k}\|u_h\|_{L^2(K)}\quad\tand\quad
% 		\|u_h\|_{L^2(K)} \leq \frac{\Cfem}{(h_K k)^j} \|u_h\|_{H_k^{-j}(K)},
% 	\end{align*}
% 	where $\|u_h\|_{H_k^{-j}(K)} := \sup_{v \in C^\infty_c(K)} (\abs{\int_{K} u_h v\,dx}/\|v\|_{H^j_k(K)})$.
% \end{assumption}
}

\begin{definition}[Well-behaved finite-element space]
\label{d:wellbehaved}
\blueOld{
Let $C_0,k,\kappa>0$ and $p\geq 1$. We say that a finite-element space $V_{\mathcal{T}}$ over a mesh $\mathcal{T}$ is \emph{well-behaved of order $p$ at frequency $k$ with constants $(C_0,\kappa)$ if}}\\

\blueOld{\noindent
\begin{minipage}{0.45\textwidth}
\begin{equation}
\label{e:swlg}
h(\mathcal{T})\leq C_0k^{-1},
\end{equation}
\end{minipage}
\hfill
\begin{minipage}{0.45\textwidth}
\begin{equation}
\label{e:qu}\mathcal{U}(\mathcal{T},k^{-1})\leq C_0,
\end{equation}
\end{minipage}}\\

\noindent
\blueOld{
$V_{\mathcal{T}}$ has the $p$-approximation property at frequency $k$ (of Definition \ref{d:app}), the super-approximation property at frequency $k$ (of Definition \ref{d:sap}) with constants $(C_0,\kappa)$, and has inverse inequalities at frequency $k$ (of Definition \ref{d:ii}) with constant $C_0$.}
% We say that $(V_k)_{k > 0}$ is a {\em well-behaved finite-element of order $p$} if it satisfies Assumptions \ref{ass:swlg}-\ref{ass:ii} above.
\end{definition}

% \begin{definition}[Well-behaved finite-element of order $p$]
% \label{d:wellbehaved}
%  if it satisfies Assumptions \ref{ass:swlg}-\ref{ass:ii} above.
% \end{definition}

\begin{remark}\label{r:geometricError}
	Since $\Omega$ has a $C^\infty$ boundary \blueOld{well-behaved finite element spaces must have curved meshes}. However, 
	\begin{itemize}
		\item this type of assumptions is common in the high-frequency error analysis for the finite-element method for the Helmholtz equation, see e.g. \cite[\blueOld{Example 5.1 and Assumption 5.2}]{MeSa:10}, and
		\item the ``geometric error" incurred by using \blueOld{straight} elements instead of curved elements is studied in \cite{ChSp:24}, and shown to be smaller than the pollution error. 
	\end{itemize}
%	We leave more in-depths discussion of this issue to future work.
\end{remark} 

For each $k \in \R_+ \setminus \blueOld{\mathrm{J}}$, the Galerkin solution $u_h = u_h(k) \in \blueOld{V_{\mathcal{T}}}$ (where the subscript $h$ emphasizes the dependence of $u_h$ with respect to the meshwidth of the \blueOld{mesh}) is defined by 
\begin{equation*}
%\label{e:def_uh}
a_k(u-u_h,v_h) = 0 \quad \tfa v_h \in \blueOld{V_{\mathcal{T}}},
\end{equation*}
and our main result, Theorem \ref{t:theRealDeal}, describes the (micro-)local behaviour of the error $u-u_h$.

\subsection{Frequency splitting of the error}

We consider a splitting of the Galerkin error $u-u_h$ into ``low-frequencies'' and ``high-frequencies''. To define these notions, we introduce frequency cutoffs as follows.

The following G\aa rding inequality holds (see 
\S\ref{sec:verify} and \S\ref{sec:PML}): 
there exists $\omega\in \Rea$ (with $\omega=0$ for the most commonly-used PML) such that 
for all $k_0 > 0$ there are $\cGa, \CGa > 0$ such that for all $k \geq k_0$, 
\beq\label{e:Garding1}
\Re( a_k(u,u)) \geq \cGa\|u\|^2_{\cZ_k} - \CGa \|u\|^2_{L^2}.
\eeq
We deduce from this (see \S \ref{sec:pseudolocS}) that $\sigma(\cP_k) \subset [-\CGa,+\infty)$, and thus, for each $k_0$, there exists a function $\psis$ such that 
\begin{equation}
\label{e:defpsis}
\psis(x) \geq \frac{-x + \CGa}{2} \quad \tfa x \in \sigma(\cP_k).
\end{equation}

Let $\cP_k : \Zspace{} \to (\Zspace{})^*$ be defined by 
$$
\cP_k = \frac12 \big( P_k + P_k^{*}\big).
$$
We show in Section \ref{sec:assumptions} that $\cP_k$ is self-adjoint on $L^2(\Omega)$ with domain $\Zspace{2} = \Zspaced{2}$ in the Dirichlet case, and $\Zspace{2} = \Zspacen{2}$ in the Neumann case, where
\begin{equation}
\label{e:defZ2concrete}
\Zspaced{2}=H_0^1(\Omega)\cap H^2(\Omega),\qquad \Zspacen{2}=\{ u\in H^2(\Omega)\,:\,\partial_{\nu, A_\theta} u|_{\partial\Omega_-}=0,\,\, u|_{\Gamma_{\tr}}=0\}.
\end{equation}
Thus, if $f :\R \to \R$ is a bounded, continuous function, we may consider $f(\cP_k): L^2(\Omega) \to L^2(\Omega)$ defined by the functional calculus.

Low-frequency cutoffs will then be defined as $\Psi = \psi(\cP_k)$ where $\psi \in C_c^\infty(\R)$ is such that $\psi \equiv 1$ on the support of $\psis$, and $1 - \Psi$ will correspond to high-frequency cutoffs.

\subsection{Spatial splitting of the error}

In addition to considering the Galerkin error locally in frequency space, we also localize it spatially. We fix a neighbourhood $U_{\Pml}$ of $\Gamma_{\tr}$ in which Theorem~\ref{thm:PML} holds (that is, sufficiently ``deep'' in the PML region so that the resolvent $R_k^*$ in this region behaves like a pseudolocal, uniformly bounded operator with respect to $k$).  Let 
\begin{equation*}
%\label{e:coverOmega}
\Omega = \bigcup_{j = 1}^M \Omega_j
\end{equation*}
be an open cover of $\Omega$ by $M = M_\Int + M_\Pml$ subdomains. We assume that the ``interior'' domains $\Omega_1,\ldots,\Omega_{M_\Int}$ do not intersect the truncation boundary, while the ``PML'' domains $\Omega_{M_{\Int}+1},\ldots,\Omega_{M_{\Int}+M_\Pml}$ all lie inside the deep PML region, i.e.
\begin{equation}
\label{e:domainConditions}
\bigcup_{j = 1}^{M_\Int} \Omega_j \cap \Gamma_{\rm tr} = \emptyset\,, \qquad \bigcup_{j = M_\Int + 1}^{M_\Int + M_\Pml} \Omega_j \subset U_{\Pml}.
\end{equation}
For $i,j \in \{1,\ldots,M\}$, define 
\beq\label{e:hij}
h_i := \max_{\blueOld{T} \in \cT_k} \{\textup{diam}(\blueOld{T}) \,\mid\, \blueOld{T} \cap \Omega_i \neq \emptyset\} \quad \tand \quad h_{ij} := \min(h_i,h_j),
\eeq
the local mesh sizes on $\Omega_j$ and $\Omega_i \cap \Omega_j$.

\subsection{Matrix quantities}

In Theorem \ref{t:theRealDeal}, the description of the local error in subdomains is given in terms of matrices $\Hdiag$, $\Hmin$,  $\mathcal{C}$, $T$ and $B$ that we define now. 

%\begin{definition}
%	Given $k_0 > 0$ and $\mathfrak{c}$, we say that the cover $\{\Omega_j\}_{j = 1}^\domainnumber$ is $(k_0,\mathfrak{c})$-admissible if there exists an ``increasing" sequence $(\varphi_i^n)_{n = 0}^{3}$ of elements of $C^\infty_c(\R^d)$, with $\varphi_i^n$ supported in $\Omega_i$ and such that $\{\varphi_{i}^0\}_{i = 1}^\domainnumber$ is a partition of unity on $\Omega$, i.e.
%	\[\sum_{i = 1}^\domainnumber \varphi_i^0(x) = 1\quad \tfa x \in \Omega\]
%	\beq\label{e:distancemathfrak}
%	\partial_<( \supp {\varphi}_i^n,\supp (1 - \varphi_i^{n+1})) \geq \mathfrak{c} k_0^{-1}, \qquad i = 1,\ldots, N.
%	\eeq
%	%\[\supp (1 - \varphi_{i}^{n+1}) \cap \supp \varphi_{i}^n = \emptyset,\]
%\end{definition}
For every natural number $\ell$, define the following $M \times M$ matrices 
\begin{equation}
\label{e:defHdiagMin}
\Hdiag := \textup{diag}(h_1,\ldots,h_{\domainnumber}),\qquad
  \mathcal{H}^{\min}(\ell)
:=  (1_{\{(\ell,m)\,:\Omega_i \cap \Omega_j \neq \emptyset\}}(i,j)h_{ij}^\ell)_{1 \leq i,j \leq \domainnumber},
 \end{equation}
\blueOld{where $1_D$ denotes the indicator function of the set $D$.}
 
 Furthermore, let $\mathcal{C}$ be the $M \times M$ matrix defined by
\beq\label{e:matrixC}
\mathcal{C}_{ij} := \|1_{\Omega_j} R_k^* 1_{\Omega_i}\|_{L^2 \to L^2}= \|1_{\Omega_i} R_k 1_{\Omega_j}\|_{L^2 \to L^2}, \quad i,j=1,\ldots,\domainnumber.
\eeq
For an $\domainnumber \times \domainnumber$ matrix $A$ (either $\Hdiag$ or $\Hmin$), we write
 $$
 A=:\begin{pmatrix} A_{\Int,\Int}&A_{\Int, \Pml}\\A_{\Pml, \Int}&A_{\Pml,\Pml}\end{pmatrix},\quad A_{i,j}\in \mathbb{M}(\domainnumber_i\times \domainnumber_j).
 $$
%for every integer $\ell \geq 0$, 
where $M_1 := M_\Int$ and $M_2 := M_{\Pml}$ \blueOld{and $\mathbb{M}(n_1\times n_2)$ denotes the set of complex-valued matrices with $n_1$ rows and $n_2$ columns}. Let $B \in\mathbb{M}((2M_\Int + M_\Pml) \times M)$ be defined by 
\beq\label{e:matrixB}
B := \begin{pmatrix}
	\mathcal{C}_{\Int,\Int} (\Hdiag_{\Int, \Int}k)^p&0%\Hmin_{\Int,\Pml}(p)k^p 
	\\
	(\Hdiag_{\Int,\Int} k)^p&0%\Hmin_{\Int,\Pml}(p)k^p
	\\
	0%\Hmin_{\Pml,\Int}(p)k^p
	&(\Hdiag_{\Pml,\Pml} k)^p
\end{pmatrix}, \quad 
\eeq
and let $\oldT \in \mathbb{M}((2M_\Int + M_\Pml)\times(2M_\Int + M_\Pml))$ be defined by
\beq\label{e:matrixT}
\oldT := 
\begin{pmatrix}
\mathcal{C}_{\Int,\Int}(\Hdiag_{\Int, \Int} k)^{2p} & \mathcal{C}_{\Int,\Int}(\Hdiag_{\Int, \Int} k)^{2p}&\Hmin_{\Int,\Pml}(N)k^{N}  \\
\Hmin_{\Int,\Int}(2p)k^{2p} & \Hmin_{\Int,\Int}(N)k^{N}&\Hmin_{\Int,\Pml}(N)k^{N}\\
\Hmin_{\Pml,\Int}(N) k^{N}&\Hmin_{\Pml,\Int}(N)k^{N}&\Hmin_{\Pml,\Pml}(N)k^{N}
\end{pmatrix},
\eeq

\subsection{Simple-path matrix}
\label{sec:simple_path_mat}

To any square matrix $\oldT \in \mathbb{M}(N \times N)$, one can associate a matrix $V = V(\oldT)$ defined from the coefficients $\oldT$ in terms of simple paths on a graph. To define this, let $\mathcal{G} = \mathcal{G}(\oldT)$ be the (complete) directed graph, with node set $\mathcal{N} := \{1,\ldots,N\}$, and with edge set $\cE$ the set of ordered pairs $(i,j) \in \{1,\ldots,N\}^2$. A {\em path} $p$ in $\cG$ is a finite (and possibly empty) sequence of edges
$$p = (i_1,j_1) (i_2,j_2) \ldots (i_{\length - 1},j_{\length-1})(i_\length,j_\length)$$ 
satisfying the conditions $j_\ell = i_{\ell+1}$ for $1 \leq \ell  \leq \length-1$. Let $\mathbf{0}$ stand for the empty path. We 
write $\abs{p} := \length$ and denote by $p(\ell)$ the $\ell$-th node visited by $p$, i.e., $p(\ell):=i_\ell$ if $1\leq \ell \leq |p|$ and $p(|p|+1) :=j_\length$. Let $\mathbb{P}_{ij}$ be the set of paths from $i$ to $j$, i.e., such that $p(1) = i$ and $p(|p|+1) = j$.

A path $p$ is {\em non-intersecting} if the map $\ell \mapsto p(\ell)$ is injective. For $i,j \in \{1,\ldots,M\}$,
%i\neq j$, 
let $\mathbb{V}_{ij}$ be the set of non-intersecting paths from $i$ to $j$. %Otherwise,
Observe that $\mathbb{V}_{ii}:= \{\mathbf{0}\}$.

A non-empty path $p$ is a {\em loop} if it starts and ends at the same node, i.e., if $p(1) = p(|p|+1)$. It is a {\em simple loop} if it is a loop but otherwise does not intersect itself, i.e.,
\beqs
p(\ell) = p(m) \implies \big( \ell=m\,\text{ or }\,\{\ell,m\}= \{1,|p|+1\}\big).
\eeqs
We denote by $\mathbb{SL}$ the set of simple loops.

To each edge $e = (i,j)$ of $\cG$, we associated the weight $\oldT_e := \oldT_{ij}$ (the $(i,j)$-th coefficient of the matrix $\oldT$). We also define the weight of the path $p$ as the product of the weights of its edges, i.e., $\oldT_{\mathbf{0}} := 1$ and
\begin{equation*}
%\label{e:path_weight_def}
\oldT_{e_1e_2\ldots e_L} := \oldT_{e_1} \oldT_{e_2} \ldots  \oldT_{e_L}.
\end{equation*}

\begin{definition}[Simple-path matrix]
\label{def:simple_path_mat}
The {\em simple-path matrix} $\transfer = \transfer(\oldT) \in \mathbb{M}(N \times N)$ of a matrix $\oldT \in \mathbb{M}(N\times N)$ is defined by
$$\transfer_{ij} := \sum_{p \in \mathbb{V}_{ij}} \oldT_{p} \,, \quad 1 \leq i,j\leq N.$$
Observe that the diagonal entries of $\transfer$ are $1$ since $\mathbb{V}_{ii} = \{\mathbf{0}\}$.
\end{definition}

\begin{remark}
We show in Theorem \ref{t:onlyTheSimpleOnes} that, provided the simple loops of $\cG$ carry weights bounded by $c < 1$, then the $I - W$ is invertible and $(I- \oldT)^{-1} \leq \transfer$ coefficientwise. 
\end{remark}

\subsection{Statement of the main result}

\begin{theorem}[The main result]\label{t:theRealDeal}
	Let $a_k$ be defined by \eqref{e:def_ak}, $\blueOld{\mathrm{J}} \subset \R_+$ be such that Assumption \ref{a:polyBound} holds, $p$ be a positive integer, and \blueOld{$C_0,\kappa>0$}.   Let $\{\Omega_i\}_{i = 1}^M$ be an open cover of $\Omega$ such that the conditions \eqref{e:domainConditions} hold. For every $i \in \{1,\ldots,M\}$, let $\chibigger_i \in C^\infty(\overline{\Omega})$ be such that
		\begin{equation}
		\label{e:crazyCover0}
\supp(\chibigger_i) \subset \Omega_i \cup \partial \Omega\quad \tand \quad
			\Omega=\bigcup_{i=1}^{\domainnumber} {\rm int}\, \big(\big\{ \chi_i \equiv 1\big\}\big),
		%		\big(\Omega\setminus\supp (1-\chi_i)\big).
		\end{equation}
	where the interior is taken in the subspace topology of $\Omega$. Let $k_0, N > 0$, let $\psis$ satisfy \eqref{e:defpsis} and let $\psi \in C^\infty_c(\R)$ be such that 
$\supp \psis \cap \supp(1-\psi)= \emptyset$.
%	$\psis\prec\psi$. 
%	(i.e., the interiors of the sets $\{\chi_i\equiv 1\}$ cover $\Omega$). 

Then, there exist constants $h_0,\specialC>0$ and, for any $0<c<1$, a constant $C>0$ such that the following holds. For any $k\in  (k_0,\infty)\setminus \blueOld{\mathrm{J}}$, \blueOld{and all finite element spaces $V_{\mathcal{T}}$ over a mesh $\mathcal{T}$ that are well-behaved of order $p$ at frequency $k$ with constants $(C_0,\kappa)$ in the sense of Definition \ref{d:wellbehaved} satisfying $h(\mathcal{T}) \leq h_0$} and
	\begin{equation}
		\label{e:meshConditionsGeneral}
		%\sum_{\substack{q\in\mathbb{SL}\\|q|\geq 1}}\specialC^{|q|}
		\sum_{{L \in \mathbb{SL}}}
		C_\dagger^{|L|} \oldT_{L}\leq c,
	\end{equation}
	where $\oldT$ is defined by \eqref{e:matrixT}, then for all $u \in H^1_k$, there exists a unique solution $u_h \in V$ to the Galerkin problem
	\begin{equation}
	\label{e:Galerkin_ortho}
		a_k(u-u_h,v_h) = 0 \quad \tfa v_h \in V.
	\end{equation}
	Moreover, for any $w_{h,1},\ldots, w_{h,\domainnumber} \in V$, $m \in \{0,\ldots,p\}$, and $i \in \{1\,,\ldots\,,M_\Int\}$,
%	\begin{align}%\nonumber
%		&\| \chibigger_i\Psi(u-u_h)\|_{H_k^{1-m}}\leq 
%		C \sum_{j=1}^\domainnumber 
%\big[\piminus \transfer \matrixB\big]_{i,j}
%		\|u-w_{h,j}\|_{H_k^1(\Omega_j)}+CR%.Ck^{-N}(hk)^{p} \|u-w_{h,j}\|_{H^1_k} \bigg).
%		\label{e:mainResultLow}
%		\\ \nonumber
%		&\|\chibigger_i(1-\Psi)(u-u_h)\|_{H_k^{1-\newell}}\\ \nonumber
%		&\leq 
%		C \sum_{j=1}^\domainnumber \Big[(\Hdiag_{\Int, \Int}k)^{\newell}+(\Hdiag_{\Int, \Int}k)^{p+m}\piminus  \transfer B+  (\Hdiag_{\Int, \Int} k)^{N}\piplus \transfer \matrixB\Big]_{i,j}
%		\|u-w_{h,j}\|_{H_k^1(\Omega_j)}\\
%		&\hspace{11.5cm} + CR,%C k^{-N}(hk)^{\newell} \|u-w_{h,j}\|_{H^1_k}\bigg),
%		\label{e:mainResultHigh}
%		\end{align}
%		and for $i \in \{1,\ldots,M_\Pml\}$, 
%		\begin{align}\nonumber
%				\|\chibigger_{\domainnumber_\Int +i}(u-u_h)\|_{H_k^{1-\newell}}&\leq 
%		C \sum_{j=1}^\domainnumber \big[(\Hdiag_{\Pml, \Pml}k)^{\newell}+(\Hdiag_{\Pml, \Pml}k)^{N}\pi_{\Pml} \transfer B \big]_{i,j}
%		\|u-w_{h,j}\|_{H_k^1(\Omega_j)}\\
%		&\hspace{8cm} +CR% C k^{-N}(hk)^{\newell} \|u-w_{h,j}\|_{H^1_k}\bigg),
%		\label{e:mainResultPML}
%%		\|(1-\Psi)\chi_i(u-u_h)\|_{H_k^{1}} & \leq C \sum_{j=1}^{\domainnumber} \Big[ I+\piplus V(1) \matrixB(1)
%%		+ (\matrixD k)^p\piminus V(p) \matrixB\Big]_{ij}
%%		\|u-w_h\|_{H_k^1(\Omega_j)}\\
%%		&\hspace{2cm}+  C k^{-N}(hk) \|u-w_h\|_{H^1_k}.
%	%	\label{e:mainResultHighH1}
%	\end{align} 
	\setlength{\arraycolsep}{1pt}
	\begin{equation}
	\begin{aligned}
	&\left(
	\begin{aligned} &\big(\big\|\chi_i\Psi (u-u_h)
	\big\|_{H_k^{1-m}}\,\big)_{i=1}^{\domainnumber_{\Int}}\\[.25em]
	 &\big(\big\|\chi_i(1-\Psi) (u-u_h)
	\big\|_{H_k^{1-m}}\,\big)_{i=1}^{\domainnumber_{\Int}}\\[.25em]
	 &\big(\big\|\chi_i (u-u_h)
	\big\|_{H_k^{1-m}}\,\big)_{i=\domainnumber_{\Int}+1}^{\domainnumber}
	\end{aligned}\right)\\
	&\leq \left[\begin{pmatrix}0&0\\
	(\Hdiag_{\Int, \Int}k)^{\newell}&0\\
	0&(\Hdiag_{\Pml, \Pml}k)^{\newell}\end{pmatrix}+\begin{pmatrix}\Id&0&0\\(\Hdiag_{\Int, \Int}k)^{p+m}&(\Hdiag_{\Int, \Int} k)^{N}&0\\
	0&0&(\Hdiag_{\Pml, \Pml}k)^{N}\end{pmatrix}\transfer B\right]\Big(\|u-w_{h,j}\|_{H_k^1(\Omega_j)}\Big)_{j=1}^{\domainnumber}\\
	&\hspace{12cm}+CR.
	\end{aligned}
	\label{e:mainResult}
	\end{equation}
	\setlength{\arraycolsep}{5pt}
	where $\Hdiag$ is defined by \eqref{e:defHdiagMin}, $B$ is defined by \eqref{e:matrixB}, $\transfer$ is the simple-path matrix of $\specialC \oldT$ in the sense of Definition \ref{def:simple_path_mat}, %and $\pi_{\Int,\pm}$, $\pi_{\Pml}$ are defined by \eqref{e:def_pi+-} and \eqref{e:def_pipml}, 
	and 
\beqs
R:= k^{-N}(hk)^{\newell}\sum_{j=1}^\domainnumber \|u-w_{h,j}\|_{H^1_k}.
	\eeqs
	
	In particular, the local Galerkin errors satisfy
		\setlength{\arraycolsep}{1pt}
	\begin{equation}\label{e:onlyH1}
	\begin{aligned}
	&
	\big(\big\|\chi_i (u-u_h)
	\big\|_{H_k^{1-m}}\,\big)_{i=1}^{\domainnumber}
	% &\big(\big\|\chi_i(1-\Psi) (u-u_h)
	%\big\|_{H_k^{1-m}}\,\big)_{i=1}^{\domainnumber_{\Int}}\\[.25em]
	 \\
	&\leq \left[\begin{pmatrix}
	(\Hdiag_{\Int, \Int}k)^{\newell}&0\\
	0&(\Hdiag_{\Pml, \Pml}k)^{\newell}\end{pmatrix}+\begin{pmatrix}\Id&(\Hdiag_{\Int, \Int} k)^{N}&0\\	0&0&(\Hdiag_{\Pml, \Pml}k)^{N}\end{pmatrix}\transfer B\right]\Big(\|u-w_{h,j}\|_{H_k^1(\Omega_j)}\Big)_{j=1}^{\domainnumber}\\
	&\hspace{12cm}+CR.
	\end{aligned}
	\end{equation}
	\setlength{\arraycolsep}{5pt}
%	
%	 for $1\leq i\leq \domainnumber_\Int $
%	\begin{align}
%		\nonumber
%		\|\chibigger_i(u-u_h)\|_{H_k^{1-\newell}} & \leq C \sum_{j=1}^{\domainnumber} \bigg(\Big[ (\Hdiag_{\Int, \Int}k)^\newell
%		+(\Hdiag_{\Int, \Int} k)^{N}\piplus \transfer \matrixB
%		+\piminus \transfer \matrixB\Big]_{ij}
%		\|u-w_{h,j}\|_{H_k^1(\Omega_j)}\\
%		&\hspace{9cm}+  CR.%. k^{-N}(hk)^{\newell} \|u-w_{h,j}\|_{H^1_k}\bigg).
%		\label{e:onlyH1}
%	\end{align}
\end{theorem}

The proof of Theorem \ref{t:theRealDeal} is given in \S \ref{sec:proof_realDeal}.
Figure \ref{f:twitch} shows the weighted graph associated to the matrix $\oldT$ in the setting of \S\ref{sec:intro}; i.e., $\domainnumber_\Int=3$ (with domains $\Omega_\cavity, \Omega_\visible$, and $\Omega_\invisible$) and $\domainnumber_\Pml=1$. 
\begin{figure}[h]
\hspace{-.1\textwidth}
\includegraphics[width=1.2\textwidth]{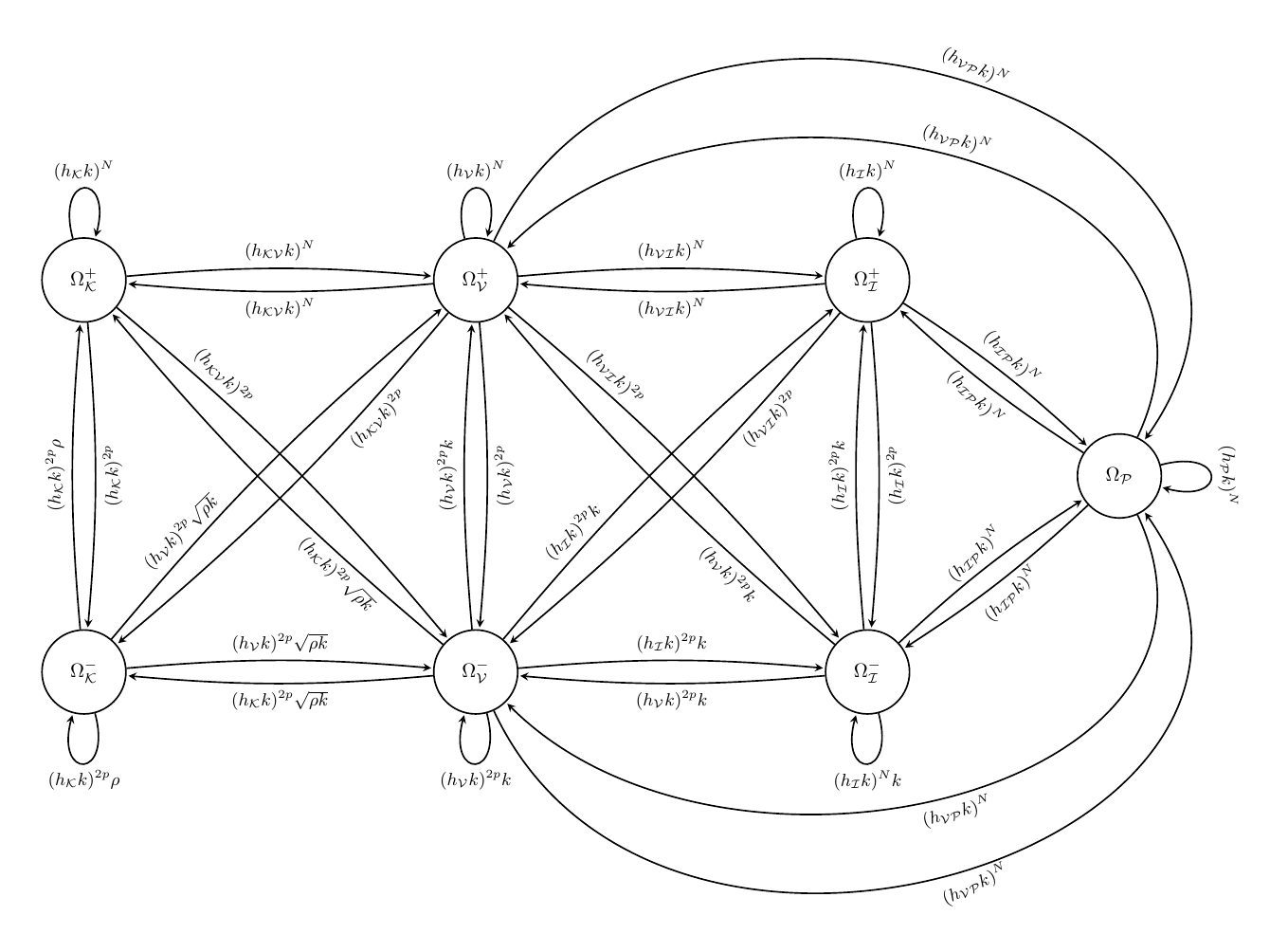}
\caption{The weighted graph associated to the matrix $\oldT$ in the case when $\domainnumber_\Int=3$ (with domains $\Omega_\cavity, \Omega_\visible$, and $\Omega_\invisible$), $\domainnumber_\Pml=1$, and with the $k$-dependence of $\mathcal{C}$ \eqref{e:matrixC} given by the results of \S\ref{sec:boundsCsol}.  Edges with zero weight (between $\Omega^\pm_\cavity$ and $\Omega_\pml$) or $O(k^{-\infty})$ weight (between $\Omega_\cavity^\pm$ and $\Omega_\invisible^\pm$) are not displayed.
Finally $h_{\visible \pml} := \min\{ h_\visible, h_\pml\}$ etc.
 \label{f:twitch}}
\end{figure}
%\begin{landscape}
%\global\pdfpageattr\expandafter{\the\pdfpageattr/Rotate 90}

%\end{landscape}
% %
%\global\pdfpageattr\expandafter{\the\pdfpageattr/Rotate 0}

\section{Local bounds on the Helmholtz solution operator}
\label{sec:boundsCsol}

This section describes two results showing how the Helmholtz solution operator has improved $k$-dependence based on the data and measurement locations. The first result (Theorem \ref{thm:DV}) considers locations relative to the cavity or the ray dynamics, and the second (Theorem \ref{thm:PML}) considers locations relative to the PML.
Both results are proved in Appendix \ref{app:help} with the first result a special case of a more general result phrased in terms of semiclassical pseudodifferential operators.

%If either the data or measurement locations are away from the cavity, then the Helmholtz solution has improved $k$-dependence, as shown by the following theorem.

\begin{theorem}[Improved behaviour away from trapping]\label{thm:DV}
Let $k_0>0$ and let $\blueOld{\mathrm{J}}$ be such that 
Assumption \ref{a:polyBoundIntro} holds. 

(i) For all $\chi \in C^\infty(\overline{\Omega})$ with $\supp \chi \cap \cavity=\emptyset$, there exists $C>0$ such that for all $k\in (k_0,\infty)\setminus\blueOld{\mathrm{J}}$
\begin{gather}\label{eq:DV}
\|\chi R_k\|_{L^2\to L^2}+\|R_k\chi \|_{L^2\to L^2}\leq C\sqrt{k\rho },\qquad \|\chi R_k\chi\|_{L^2\to L^2}\leq Ck.
\end{gather}

(ii) For all $\chi,\psi \in C^\infty(\overline{\Omega})$ with $\supp \chi \subset \invisible$ and $\supp\psi\subset \cavity$, and all $N>0$, there exists $C>0$ such that for all $k\in (k_0,\infty)\setminus\blueOld{\mathrm{J}}$
\begin{gather}\label{eq:negligible}
\|\chi R_k\psi\|_{L^2\to L^2}+\|\psi R_k\chi\|_{L^2\to L^2}\leq Ck^{-N}.
\end{gather}
\end{theorem}

%Theorem \ref{thm:DV} is a special case of Theorem \ref{thm:DV2} below, with the latter phrased in terms of semiclassical pseudodifferential operators. 
In the case of scattering without boundaries, the result analogous to \eqref{eq:DV} was proved in \cite{DaVa:12, DaVa:12a}.

\begin{theorem}[Improved behaviour in the PML]\label{thm:PML}
Let $k_0>0$ and let $\blueOld{\mathrm{J}}$ be such that 
Assumption \ref{a:polyBoundIntro} holds. Then there is $U\subset \Omega$ a neighbourhood of $\Gamma_{\tr}$ such that for all $\chi \in C^\infty(\overline{\Omega})$ with $\supp \chi\subset U$, there exists $C>0$ such that,  for all $k\in (k_0,\infty)\setminus\blueOld{\mathrm{J}}$,
\begin{equation}
\label{e:bounded_in_pml}
\|\chi R_k\|_{L^2\to L^2}+\|R_k\chi \|_{L^2\to L^2}\leq C.
\end{equation}
Moreover, if $\supp \chi \subset U$, and $\psi\in C^\infty(\overline{\Omega})$ with $\supp \chi \cap \supp \psi=\emptyset$, then for any $N$ there exists $C>0$ such that for all $k>k_0$, 
\begin{equation}
\label{e:pseudoloc_in_pml}
\|\chi R_k\psi\|_{L^2\to H_k^N}+\|\psi R_k\chi\|_{L^2\to H_k^N}\leq Ck^{-N}.
\end{equation}
\end{theorem}

Theorem \ref{thm:PML} is based on ellipticity in the PML region.

\section{Abstract pseudolocality results}
\label{sec:pseudolocS}

As described in \S\ref{s:sketch},
the proof of the main result requires pseudolocaity of the operators $S_k$ and $\Pi^\sharp_k$ (see \eqref{e:startDuality}); furthermore, although not stated in \S\ref{s:sketch}, the proof also requires pseudolocality of $(P^\sharp_k)^{-1}$, where $P^\sharp_k:= P_k +S_k$. 
This section proves pseudolocality of $S_k$ and $(P^\sharp_k)^{-1}$, and \S\ref{sec:pseudoLocPi} proves pseudolocality of $\Pi^\sharp_k$.

%In this section, we formulate a set of assumptions about families of self-adjoint operators $(\cP_k)_{k \geq 0}$ indexed by a ``frequency'' parameter $k$. The operators $\cP_k$ correspond to self-adjoint realizations of Helmholtz operators with smooth coefficients in smooth geometries. ``Boundary conditions'' are taken into account abstractly (through the spaces $\cZ_k$ and $\cD_k^2$ below). 

The operator $S_k$ is defined as a function of the self-adjoint operator $\cP_k := \Re  P_k$ via the functional calculus. This section therefore studies general Helmholtz operators (satisfying continuity, a G\aa rding inequality, and elliptic regularity), proves that $\cP_k$ is self-adjoint, and then proves that both functions of $\cP_k$ and $(P^\sharp_k)^{-1}$ are pseudolocal; i.e., 
 %(functions of) $\cP_k$ and its (perturbed) resolvent 
 when sandwiched by disjoint ``spatial" or ``frequency" cutoffs, the result is $O(k^{-\infty})$ and infinitely smoothing.
 
For the ``spatial cutoffs'', we require some control over their repeated commutators with $\cP_k$ in a scale of Hilbert spaces $(\cH_k^n)_{n \geq 0}$ (which will be taken as $H^n_k(\Omega)$). Checking these assumptions in the concrete setting will require the construction of suitable cutoff functions with a special behavior near the boundary

%\item[(ii)] 
For the ``frequency cutoffs'', we require that, in addition, the repeated commutators act in domains $\cD^n$ of powers of the self-adjoint operators $\cP_k$. This essentially asks that repeated commutators preserve an arbitrary number of boundary conditions, which in practice, will be achieved by requiring the frequency cutoffs to be constant near the boundary and $0$ near the PML truncation boundary. 
%\end{itemize}

\subsection{Abstract formulation of Helmholtz operators with smooth coefficients on smooth domains}

\begin{table}
\centering
\begin{tabular}{c|c}
Abstract setting & Model Dirichlet setting \\[0.5em]
\hline
&\\[-0.5em]
$\mathcal{H}_k^n$ & $H^n(\Omega)$ with $k$-weighted norm \\[0.5em]
$a_k(u,v)$ & $\int_{\Omega} k^{-2}\nabla u \cdot \overline{\nabla v} - u\overline{v}$ \\[0.5em]
$\cZ_k$ & $H^1_0(\Omega)$  \\[0.5em]
$\Zspace{2}$ & $H^2(\Omega) \cap H^1_0(\Omega)$ \\[0.5em]
$\Zspace{n}$ & $H^n(\Omega) \cap H^1_0(\Omega)$  \\[0.5em]
$\Dspace{2n}$ & $\{u \in H^{2n}(\Omega) : \gamma u, ...,\gamma \Delta^{n-1} u = 0 \textup{ on } \partial \Omega\}$ \\[0.5em]
\hline
&\\[-0.5em]
Abstract setting &Model Neumann setting\\[0.5em]
\hline
&\\[-0.5em]
$\mathcal{H}_k^n$ & $H^n(\Omega)$  with $k$-weighted norm\\[0.5em]
$a_k(u,v)$ & $\int_{\Omega} k^{-2}\nabla u \cdot \overline{\nabla v} - u\overline{v}$\\[0.5em]
$\cZ_k$ & $H^1(\Omega)$ \\[0.5em]
$\Zspace{2}$  & $\{u \in H^2(\Omega) \,:\, \partial_\nu u = 0 \textup{ on } \partial \Omega\}$\\[0.5em]
$\Zspace{n}$ & $\{u \in H^n(\Omega) \,:\, \partial_\nu u = 0 \textup{ on } \partial \Omega\}$ \\[0.5em]
$\Dspace{2n}$ & $\{u \in H^{2n}(\Omega) : \partial_\nu u, ...,\partial_\nu \Delta^{n-1} u = 0 \textup{ on } \partial \Omega\}$
\end{tabular}
\caption{Model examples for the spaces in Section \ref{sec:pseudolocS}}
\label{tab:model_settings}
\end{table}

In what follows, $(\cH,\|\cdot\|_{\cH})$ is a Hilbert space and for every $k \in \R_+$, $(\Hspace{n},\|\cdot\|_{\Hspace{n}})_{n \in \mathbb{N}}$ is a decreasing sequence of Hilbert spaces with continuous and dense inclusions $\Hspace{n} \subset \Hspace{m}$ for all $m \leq n$, with
$$\|u\|_{\cH_k^m} \leq \|u\|_{\cH_k^{n}} \quad \tfa u \in \cH_k^n$$ 
and such that $\Hspace{0} = \cH$ with equal norms.\footnote{Notice that the abstract setup is in many parts similar to \cite{GS3}, but here it is not assumed that the inclusions $\Hspace{n} \subset \Hspace{m}$ are compact for $n>m$.} For all $n \in \mathbb{N}$, $\Hspace{-n}$ denotes the {\em anti}-dual of $\Hspace{n}$, i.e., the set of continuous complex-valued {\em anti}-linear forms on $\Hspace{n}$. For any $u \in \cH$, one may define an element $L^n_u \in \Hspace{-n}$ by
$$L^n_u(v) := \langle u,v\rangle := (u,v)_{\cH} \qquad \tfa v \in \Hspace{n}.$$  
By density of the embeddings $\Hspace{n} \subset \Hspace{m} \subset \cH$ for $n \geq m$, the mapping $u \mapsto L_u^n$ is injective and $L^n_u$ coincides with $L^m_u$ on $\Hspace{m}$, so we may identify $u$ with $L_u$. Under this identification, the continuous embeddings $\Hspace{n} \subset \Hspace{m}$ for $n\geq m$ extend to all $m,n \in \Z$ and $\langle \cdot,\cdot\rangle$ extends to a continuous sesquilinear (linear on the left, anti-linear on the right) pairing $\cH_k^{-n} \times \cH_k^{n}$ for all $n$.
%Let $X\subset \cap_n\Hspace{n}$ be a topological vector space with dense inclusion in $\cH$. 

Let $(a_k)_{k \in \R_+}$ be a family of sesquilinear forms 
$$a_k : \Hspace{1} \times \Hspace{1}\to \mathbb{C}\quad\tfa k\geq 0,$$

\begin{assumption}[$k$-uniform Continuity] 
	\label{ass:cont}
	For every $k_0 \in \R$, there exists a positive constant $C_0(k_0)>0$  such that, for all $k \geq k_0$ and all $u,v \in \cH^1_k$, 
	\[ |a_k(u,v)| \leq C_0(k_0) \|u\|_{\cH^1_k} \|v\|_{\cH^1_k}.\]
\end{assumption}
\begin{assumption}[G\aa rding inequality] 
	\label{ass:Gar}
	For every $k_0 \in \R$, there exist positive constants $\cGa(k_0)$ and $\CGa(k_0) > 0$ such that, for all $k \geq k_0$, 
	\[\Re( a_k(u,u)) \geq \cGa(k_0)\|u\|_{\Hspace{1}}^2 - \CGa(k_0) \|u\|_{\cH}^2\qquad \tfa u \in \Hspace{1}. \]
\end{assumption}
Let $\Re a_k$ denote the Hermitian part of $a_k$, i.e.
$$(\Re  a_k)(u,v) := \frac{1}{2} \left(a_k(u,v) + \overline{a_k(v,u)}\right).$$
We fix a closed subspace $\cZ_k \subset \cH^1_k$ (possibly $\cH^1_k$ itself) which is dense in $\cH$ with respect to the $\cH$ norm, and make the following assumption:
\begin{assumption}[Domain symmetry]
	\label{ass:domComp}
	The spaces 
	$$\big\{u\in \cZ_k\,:\, \sup_{v \in \cZ_k, \|v\|_{\cH}=1} |a_k(u,v)| < +\infty \big\}\,, \qquad \big\{u\in \cZ_k\,:\, \sup_{v \in \cZ_k, \|v\|_{\cH}=1} |a_k(v,u)| < +\infty \big\}\,,$$
	$$ \tand \quad  \big\{u\in \cZ_k\,:\, \sup_{v \in \cZ_k, \|v\|_{\cH}=1} \big|\big(\Re a_k\big)(u,v)\big| < +\infty \big\}$$ 
	are equal and contained in $\Hspace{2}$. We denote  their common value by $\Zspace{2}$.
\end{assumption}
\begin{remark}[Boundary conditions]
In practice, the space $\Zspace{2}$ will be a subset of $H^2(\Omega)$ with Dirichlet/Neumann conditions on (parts of) $\partial \Omega$. Dirichlet conditions will be enforced ``essentially'' by the choice of $\cZ_k$, and Neumann conditions will appear ``naturally'' in $\Zspace{2}$ as a result of a lack of Dirichlet condition. 
\end{remark}
Due to the density of $\cZ_k$ in $\cH$ and the Riesz representation theorem, this allows to state the following definition
\begin{definition}[The operators $P_k$, $P_k^*$ and $\cP_k$]
For all $u \in \Zspace{2}$, define $P_k u$ and $P_k^* u$ as the unique elements of $\cH$ such that for all $v \in \cZ_k$ 
\beq\label{eq:crapatrunning1}
\langle P_k u,v\rangle = a_k(u,v) \tand  \langle P_k^* u,v\rangle = \overline{a_k(v,u)}.
\eeq
Furthermore, let
$$\cP_k u := \frac{1}{2}\left( P_k u +  P_k^*u\right).$$
\end{definition} 
\begin{proposition}
\label{prop:Pksa}
The space $\cZ_k^2$ is dense in $\cH$ and $\cZ_k$ for their respective norms, i.e.
$$\overline{\Zspace{2}}^{\|\cdot\|_\cH} = \cH \quad \tand \quad \overline{\Zspace{2}}^{\|\cdot\|_{\cZ_k}} = \cZ_k$$
Moreover, $\cP_k : \cZ_k^2 \to \cH$ is an unbounded self-adjoint operator. Its spectrum satisfies
$$\sigma(\cP_k) \subset [-\CGa(k_0),+\infty).$$
\end{proposition}
\begin{proof}
The continuity and G\aa rding inequality, and the fact that $\cZ_k$ is a closed subspace of $\cH^1_k$ that is dense in $\cH$ (for the $\|\cdot\|_{\cH}$ norm) imply that the restriction of $\Re a_k$ to $\cZ_k$ is a lower semi-bounded closed Hermitian form in the sense of \cite[Chap. 10]{schmudgen2012unbounded}, and $\cP_k$ is the operator associated to $\Re a_k$ in the sense of \cite[Definition 10.4]{schmudgen2012unbounded}. The density of $\cZ_k^2$ in $\cH$ and the self-adjointness of $\cP_k$ then follows from \cite[Theorem 10.7]{schmudgen2012unbounded}. The density of $\cZ_k^2$ in $\cZ_k$ is Proposition 10.5(iv) in the same reference. The lower bound on the spectrum is by Proposition 10.4 in the same reference. 
\end{proof}
\begin{proposition}
$\Zspace{2}$ is a Hilbert space under the norm 
\beq\label{e:zoomWorking1}
\|u\|_{\Zspace{2}} := \|\big(\cP_k + (\CGa(k_0) + 1) \Id\big)u\|_{\cH}.
\eeq
\end{proposition}
\begin{proof}
The operator $A := \big(\cP_k + (\CGa(k_0) + 1) \Id\big)u$ is self-adjoint, thus closed (since the adjoint operator is closed by, e.g., \cite[Prop.~1.6]{schmudgen2012unbounded}), so its graph norm makes $\Zspace{2}$ a Hilbert space. Furthermore, its spectrum is contained in $[1,+\infty)$, so that
$$(Au,u)_{\cH} \geq \|u\|^2_{\cH} \implies \|Au\|_{\cH} \geq \|u\|_{\cH}.$$
Thus, 
$$\|Au\|_{\cH} \leq \|Au\|_{\cH} + \|u\|_{\cH} \leq 2\|Au\|_{\cH},$$
concluding the proof. 
\end{proof}

We denote  the dual of $\Zspace{2}$ by $\Zspace{-2}$, and identify $\cH$ and $\cZ_k$ as subspaces of $\Zspace{-2}$ -- this identification is possible by the density of $\Zspace{2}$ in $\cH$. There are then unique linear continuous extensions of the operators $P_k$, $P_k^*$ and $\cP_k$ from $\cH$ to $\Zspace{-2}$ by 
$$\langle P_k u, v\rangle := \langle u,P_k^* v\rangle, \quad \langle P_k^* u, v\rangle := \langle u,P_k v\rangle\,, \quad \langle \cP_k u,v \rangle := \langle u,\cP_k v \rangle,$$
for all $u \in \cH$ and $v \in \Zspace{2}$. With these definitions, the operator $P_k^*$ is indeed the conjugate adjoint of $P_k$, as the notation 
\eqref{eq:crapatrunning1}
suggests. 

\begin{remark}[$P_k$ is not a differential operator]
\label{rem:Pknotdiff}
In the model settings of Table \ref{tab:model_settings}, $P_k$ is not a differential operator. For instance, in the case of the Neumann Laplacian, although
%A typical application of the above construction is with $\mathcal{H}_k^n = H_k^n(\Omega)$ and $\mathcal{H}_k^{-n} = (H_k^n(\Omega))^*$ for $n \geq 0$, where $\Omega \subset \R^d$ is a smooth domain, 
%$$a_k(u,v) :=\int_{\Omega}  k^{-2}\nabla u \cdot \nabla v - uv $$ 
%and $\mathcal{Z}_k = H^1_k(\Omega)$. Then $\mathcal{Z}^2_k$ turns out to be
%$$\mathcal{Z}_k^2 = H^2_n := \{u \in H^2(\Omega) \,|\, \partial_\nu u = 0\}\,,$$
%where $\nu$ is the unit outward normal vector on $\Omega$, endowed with the norm
%$$\|u\|^2_{\cZ_k^2} :=\|(- k^{-2}\Delta + \Id) u\|^2_{L^2}.$$
%On $\cZ_k^2$, 
$P_k$ agrees with the differential operator $-k^{-2}\Delta - \Id$
on $\Zspace{2}$,
its extension 
to $L^2(\Omega)$ differs from it
%this differential operator 
(even when $\Delta$ is interpreted in the sense of distributions). Indeed, for $u \in C^\infty(\overline{\Omega}) \subset L^2(\Omega)$, integration by parts reveals that
$$P_k u = -k^{-2}\Delta u - u + k^{-2} \gamma' \cdot \partial_\nu u$$
where $\gamma : H^\ell(\Omega) \to H^{\ell - 1/2}(\partial \Omega)$, $\ell > \frac12$, is the trace operator and $\gamma'$ is its adjoint. In particular, even if $u \in H^n(\Omega)$ for a large $n$, $P_k u$ is only in $(H^{1/2 + \varepsilon}(\Omega))^*$ for all $\varepsilon > 0$, instead of $H^{n-1}(\Omega)$, unless $\partial_\nu u = 0$. 
\end{remark}

\begin{proposition}
\label{prop:PkZk}
The operators $P_k$ and $P_k^*$ map  $\cZ_k$ to $\cZ_k^*$ continuously, and they satisfy
$$\langle P_k u,v\rangle = a_k(u,v), \quad \langle P_k^* u,v\rangle = \overline{a_k(v,u)} \quad \tfa u,v\in \cZ_k,$$
$$\max(\|P_k u\|_{\cZ_k^*}, \|P_k^* u\|_{\cZ_k^*})\leq C_0(k_0)\|u\|_{\cZ_k} \quad \tfa u \in \cZ_k.$$
\end{proposition}
\begin{proof}
We observe that for $u \in \cZ_k$ and $v \in \Zspace{2}$,
$$\langle P_k u,v\rangle=\langle  u,P_k^*v\rangle = (u,P_k^* v)_\cH = \overline{(P_k^* v,u )_{\cH}}= \overline{\langle P_k^* v,u \rangle} = \overline{\overline{a_k(u,v)}} = a_k(u,v)$$
and thus
$$|\langle P_k u,v\rangle| \leq C_0(k_0) \|u\|_{\cZ_k} \|v\|_{\cZ_k}.$$
The conclusion follows by density of $\Zspace{2}$ in $\cZ_k$ for the $\cZ_k$ norm. The reasoning for $P_k^*$ is similar.
\end{proof}
\begin{definition}[The resolvent norm $\rho(k)$]
Given $k \geq 0$, if $P_k : \cZ_k \to \cZ_k^*$ is invertible, we define 
\beq\label{e:rhoHtoH}
\rho(k) := \sup_{f \in \cH \setminus \{0\}} \frac{\|P_k^{-1}f\|_{\cH}}{\|f\|_{\cH}}.
\eeq
\end{definition}

\begin{proposition}
\label{prop:ResboundZkZk}
Suppose that $P_k : \cZ_k \to \cZ_k^*$ is invertible. Then $P_k^*:\cZ_k \to \cZ_k^*$ is also invertible and for all $k_0 > 0$, there exists $C > 0$ such that for all $k \geq k_0$, 
$$\|P_k^{-1}u\|_{\cZ_k} + \|(P_k^*)^{-1}u\|_{\cZ_k} \leq C\big(1+ \rho(k)\big) \|u\|_{\cZ_k^*}.$$
Moreover, for all $z \in \mathbb{C}\setminus \R$, $(\cP_k - z) : \cZ_k \to \cZ_k^{*}$ is invertible and
$$\|(\cP_k - z)^{-1} u\|_{\cZ_k} \leq C \frac{\langle z \rangle}{|\Im(z)|} \|u\|_{\cZ_k^*}$$
where $\langle z \rangle = (1 + |z|^2)^{1/2}$.
\end{proposition}
\begin{proof}
Let $k_0 > 0$ be fixed and denote by $C$ a generic constant depending only on $k_0$. For $u \in \Zspace{*}$, the G\aa rding inequality gives
\begin{align}\nonumber
\|P_k^{-1}u\|_{\cZ_k}^2 &\leq C\left(\textup{Re}(a_k(P_k^{-1}u,P_k^{-1}u)) + \|P_k^{-1}u\|^2_{\cH}\right) \\\label{e:moreclever}
& \leq C\left(\textup{Re}(\langle u,P_k^{-1}u\rangle) + \|P_k^{-1}u\|^2_{\cH}\right)
\end{align}
Now if, moreover, $u \in \cH$, then
$\|P_k^{-1}u\|_{\cZ_k}^2 \leq C\left(\|u\|_{\cH} \|P_k^{-1}u\|_{\cH} + \|P_k^{-1}u\|^2_{\cH}\right)$
Thus, 
$$\|P_k^{-1}u\|_{\cZ_k} \leq C (\|u\|^2_{\cH} + \|(P_k^{-1}u)\|^2_{\cH}),$$
which implies 
$$ \|P_k^{-1}\|_{\cH \to \cZ_k} \leq C \big(1 + \rho(k)\big).$$ 
By the same argument, using that $\rho(k) = \|P_k^{-1}\|_{\cH \to \cH} = \|(P_k^{*})^{-1}\|_{\cH \to \cH}$ (since $P_k^{-1}: \cH\subset \Zspace{*}\to \Zspace{}\subset \cH$),
$$\|(P_k^*)^{-1}\|_{\cH \to \cZ_k} \leq C\big(1 + \rho(k)\big).$$
Thus by duality, 
$$\|P_k^{-1}\|_{\cZ_k^* \to \cH} \leq C( 1+ \rho(k)).$$
Using this in the right-hand side of \eqref{e:moreclever} as well as the inequality 
\beq\label{e:peterPaul}
2ab\leq \e a^2+\e^{-1}b^2 \quad\tfa a,b,\e>0,
\eeq 
we obtain for all $\varepsilon \in (0,1)$ sufficiently small, 
$$\|P_k^{-1}u\|_{\cZ_k} \leq C \left(\varepsilon \|P_k^{-1}u\|^2_{\cZ_k} + (\varepsilon^{-1}+ \rho(k)) \|u\|^2_{\cZ_k^*}\right)$$
and thus 
$$\|P_k^{-1}\|_{\cZ_k^*\to \cZ_k} \leq C(1+ \rho(k)).$$
We obtain the analogous bound for $(P_k^*)^{-1}$ by duality. The proof of the bound $\|(\cP_k - z)^{-1}\|_{\cZ_k^* \to \cZ_k}$ is similar, first estimating $\|(\cP_k - z)^{-1}\|_{\cH \to \cZ_k}$, using that 
$$\|(\cP_k - z)^{-1}u\|_{\cH} \leq \frac{1}{|\Im(z)|} \|u\|_{\cH},$$
since $\cP_k$ is self-adjoint on $\cH$. 
\end{proof}
\begin{definition}[The spaces $\Zspace{n}$]\label{def:Zspaces}
Let
$$\Zspace{n} := \begin{cases}
\cH & \textup{if } n = 0\\
\cZ_k & \textup{if } n = 1\\
\Zspace{2} \cap \Hspace{n} & \textup{if } n \geq 2\\
\end{cases}$$
(recalling for $n=2$ that $\Zspace{2} \subset \Hspace{2}$ by Assumption \ref{ass:domComp}).
For $n \geq 2$, the norm
$$\|u\|^2_{\cZ_k^n} := \|u\|^2_{\Zspace{2}} + \|u\|^2_{\Hspace{n}},$$
makes $\Zspace{n}$ a Hilbert space. We denote the dual of $\Zspace{n}$ by $\Zspace{-n}$  for all $n \geq 0$. 
\end{definition}
\begin{assumption}[Continuity of $P_k$ and $P_k^*$]
\label{ass:contPk}
For all $n \in \mathbb{N}$, the operators $P_k$ and $P_k^*$ define continuous maps
$$P_k, P_k^* : \Zspace{n+2}\to \cH_k^{n}.$$
For all $k_0 > 0$ and $n \in \mathbb{N}$, there exists $C(k_0) > 0$ such that for all $k \geq k_0$,
$$\|P_k u\|_{\cH_k^n} + \|P_k^* u\|_{\cH_k^n} \leq C(k_0) \|u\|_{\cZ_k^{n+2}}.$$
\end{assumption}

\begin{remark}
The idea behind Assumption \ref{ass:contPk} is that $P_k$ and $P_k^*$ can only act on functions satisfying the chosen boundary condition, and they remove this boundary condition as well as decreasing the regularity index by $2$. Thus, given $a_k(\cdot,\cdot)$, $\Hspace{}, \Hspace{1}, \Hspace{2}$ and $\Zspace{}$, Assumption \ref{ass:contPk} can be considered as constraining the spaces $\Hspace{n}$ for $n\geq 3$.% (e.g., the only condition up to now on the spaces $\Hspace{n}$ for $n\geq 3$ is that the inclusion $\Hspace{n} \subset \Hspace{m}$ is dense for all $m \leq n$).
\end{remark}

Because of this, we can extend $P_k$ and $P_k^*$ uniquely into continuous linear maps from $\Hspace{-n}$ to $\Zspace{-n-2}$ for $n \geq 0$ by setting 
$$\begin{cases}
\langle P_k u,v \rangle := \langle u,P_k^*v\rangle,\\
\langle P_k^* u,v \rangle := \langle u,P_k v\rangle,
\end{cases} \quad \tfa v \in \Zspace{n+2}.$$
To state more conveniently the mapping properties of $P_k$ and $P_k^*$, we define the Hilbert spaces
\beq\label{e:capcupdefs}
\ZcapHspace{n} := \Zspace{n} \cap \Hspace{n}, \qquad \ZcupHspace{n} := \Zspace{n} + \Hspace{n}.
\eeq
By the inclusion $\Zspace{n} \subset \Hspace{n}$ for $n\geq 0$ and duality,  %(and thus $\ZcapHspace{n}= \Zspace{n}$ and $\ZcupHspace{n}= \Hspace{n}$), 
%while the opposite inclusion is true for $n \leq 0$.
%Therefore
\begin{equation}
\ZcapHspace{n} = \begin{cases}
\Zspace{n} & \textup{if } n \geq 0,\\
\Hspace{n} & \textup{if } n \leq 0,
\end{cases} \quad \tand \quad \ZcupHspace{n} = \begin{cases}
\Hspace{n} & \textup{if } n \geq 0,\\
\Zspace{n} & \textup{if } n \leq 0.
\end{cases}
\end{equation}
Thus, since $\Zspace{-n}= (\Zspace{n})^*$, 
\beqs%\label{e:dualNewSpaces}
\ZcupHspace{-n} = (\ZcapHspace{n})^*, \quad\tfa n\in \mathbb{Z}.
\eeqs

One may think of $\cW_k^n$ (respectively $\cY_k^n)$ as the space ``with'' (respectively ``without'') boundary conditions, and the application of $P_k$ ``removes'' the boundary conditions. 
\begin{proposition}[$P_k, P_k^*$ and $\cP_k$ map $\ZcapHspace{n+2}$ to $\ZcupHspace{n}$ continuously]
For all $n \in \mathbb{Z}$ and $k_0 > 0$, there exists $C > 0$ such that for all $k \geq k_0$ and all $u \in \ZcapHspace{n}$, 
$$\|P_k u\|_{\ZcupHspace{n}} + \|P_k^* u\|_{\ZcupHspace{n}} + \|\cP_k u\|_{\ZcupHspace{n}} \leq C \|u\|_{\ZcapHspace{n+2}}$$
\end{proposition}
\begin{proof}
This is Assumption \ref{ass:contPk} for $n \geq 0$ and follows from it by duality for $n \leq -2$. Finally, Proposition \ref{prop:PkZk} gives the result for $n = -1$ 
\end{proof}
\begin{assumption}[Elliptic regularity]
\label{ass:ell}
Let $Q$ equal either $P_k, P_k^*$ or $\cP_k$ and let $n \in \mathbb{N}$. If 
$$u \in \cH \quad \tand \quad Qu \in \Hspace{n},$$
then $u \in \Zspace{n+2}$, and for all $k_0 > 0$ and $n \in \mathbb{N}$, there exists $C_{\rm ell}(k_0,n) > 0$ such that for all $k \geq k_0$ and $u \in \Zspace{2}$, 
$$\|u\|_{\Zspace{n+2}}  \leq C_{\rm ell}(k_0,n) \Big(\norm{u}_{\cH} + \norm{Qu}_{\Hspace{n}}\Big).$$
\end{assumption}

%  and define the linear operators $\tilde{P}_k,\tilde{P}_k^*:\Hspace{1} \to (\Hspace{1})^*$ by
%\[(\tilde{P}_k u)(v) := a_k(u,v), \quad (\tilde{P}_k^* u)(v) := \overline{a_k(v,u)}\] 
%Let  $\tilde{\cP}_k := \Re (\tilde{P}_k)= (\tilde{P}_k + \tilde{P}_k^*)/2$, with associated form.

%and let $P_k,P_k^*$, and $\cP_k:\Zspace{1}\to \Zspace{-1}=(\cZ_k)^*$ be defined respectively by
%$$
%\langle P_ku,v\rangle:=a_k(u,v),\qquad \langle P_k^*u,v\rangle=\overline{a_k(v,u)},\qquad \langle \cP_ku,v\rangle= (\Re a_k)(u,v),\quad u,v\in\cZ_k,
%$$

\begin{proposition}[Norms of $P_k^{-1}$ and $(P_k^*)^{-1}$ from $\ZcupHspace{n}$ to $\ZcapHspace{n+2}$]\label{prop:Rstar}
Suppose that $P_k : \cZ_k \to \cZ_k^*$ is invertible. Then for all $n \in \mathbb{Z}$,  $P_k^{-1} : \ZcupHspace{n} \to \ZcapHspace{n+2}$ and  $(P_k^*)^{-1} : \ZcupHspace{n} \to \ZcapHspace{n+2}$ are continuous and for all $k_0 > 0$, there exists $C > 0$ such that for all $k \geq k_0$, 
\begin{equation}
\label{e:map_PkPkstar}
\|P_k^{-1}u\|_{\ZcapHspace{n+2}} + \|(P_k^*)^{-1} u\|_{\ZcapHspace{n+2}} \leq C\big(1+\rho(k)\big) \|u\|_{\ZcupHspace{n}}
\end{equation}
where $\rho(k)$ is defined by \eqref{e:rhoHtoH}.
%$$\rho(k) = \sup_{f \in \cH \setminus \{0\}}\frac{\|P_k^{-1} f\|_{\cH}}{\|f\|_{\cH}}$$
\end{proposition}
\begin{proof}
The case $n = -1$ is Proposition \ref{prop:PkZk}. Hence it remains to prove \eqref{e:map_PkPkstar} for $n \geq 0$, since the case $n \leq -2$ follows by duality. We proceed by induction. First, for $n = 0$, let $u \in \cH$. Then $P_k^{-1}u \in \cH$ and thus by elliptic regularity, $u \in \Zspace{2}$ with
$$\|u\|_{\Zspace{2}} \leq C_{\rm ell}(k_0,0) (\|P_k^{-1}u\|_{\cH} + \|u\|_{\cH}) \leq C_{\rm ell}(k_0,0)(1+ \rho(k)) \|u\|_{\cH} \leq C \rho(k) \|u\|_{\cH}$$
by definition of $\rho(k)$, where $C$ depends only on $k_0$. Next, let $n \geq 0$ and suppose that there exists $C > 0$ such that
$$\|P_k^{-1}u\|_{\Zspace{n+2}} \leq C \rho(k) \|u\|_{\Hspace{n}}.$$
Let $u \in \Hspace{n+1}$. Then by elliptic regularity and using the continuous embeddings $\Zspace{n+2} \subset \cH$ and $\Hspace{n+1} \subset \Hspace{n}$, 
\begin{align*}
\|P_k^{-1}u\|_{\Zspace{n+3}} &\leq C_{\ell}(k_0,n+1) (\|P_k^{-1}u\|_{\cH} + \|u\|_{\Hspace{n+1}})\\
& \leq C_{\ell}(k_0,n+1) (\|P_k^{-1}u\|_{\Zspace{n+2}} + \|u\|_{\Hspace{n+1}}) \\
& \leq C_{\ell}(k_0,n+1) (C\rho(k)\|u\|_{\Hspace{n}} + \|u\|_{\Hspace{n+1}}) \\
& \leq C'\rho(k) \|u\|_{\Hspace{n+1}}
\end{align*}
where $C'$ depends only on $k_0$ and $n$. 
\end{proof}

\begin{proposition}[Resolvent norm from $\ZcupHspace{n-1}$ to $\ZcapHspace{n+1}$]\label{prop:rescuedfrombin}
Let $k_0 > 0$ and $n \in \mathbb{Z}$. There exists $C(k_0,n) > 0$ such that for all $k \geq k_0$ and all $z \in \mathbb{C} \setminus \R$, 
\begin{align*}%\label{e:ResBound3}
\|(\cP_k - z)^{-1} \|_{\ZcupHspace{n-1}\to \ZcapHspace{n+1}} &\leq C(k_0,n)\frac{\langle z \rangle^{1+\lfloor  |n|/2\rfloor}}{|\Im(z)|}.
\end{align*}
\end{proposition}

\begin{proof}
The result for $n = 0$ is Proposition \ref{prop:ResboundZkZk}. For $n = 1$, we use that 
$$\|(\cP_k - z)^{-1}u\|_{\cH} \leq \frac{\|u\|_{\cH}}{|\Im(z)|}\,.$$
and the elliptic regularity (Assumption \ref{ass:ell}) to write
\begin{align*}
	\|(\cP_k-z)^{-1}u\|_{\Zspace{2}}&\leq C(\| \cP_k(\cP_k-z)^{-1}u\|_{\cH} +\|(\cP_k-z)^{-1}u\|_{\cH})\\
	&\leq C\Big(\|u\|_{\cH}+|z| \cdot \| (\cP_k-z)^{-1}u\|_{\cH} +\|(\cP_k-z)^{-1}u\|_{\cH}\Big)\\
	&\leq C \frac{\langle z \rangle}{|\Im(z)|}\|u\|_{\cH}.
\end{align*}
Next let $n \geq 0$ and suppose that 
$$\|(\cP_k - z)^{-1}u\|_{\Zspace{n+1}} \leq C \frac{\langle z \rangle^{m}}{|\Im(z)|}\|u\|_{\Zspace{n-1}}$$
for all $u \in \Zspace{n-1}$. Then elliptic regularity gives
\begin{align*}
\|(\cP_k - z)^{-1}u\|_{\Zspace{n+3}} &\leq C \|(\cP_k - z)^{-1} u\|_{\cH} + \|\cP_k(\cP_k - z)^{-1}u\|_{\Zspace{n+1}}\\
& \leq C \Big(|\Im(z)|^{-1}\|u\|_{\cH} + \|u\|_{\Zspace{n+1}} + |z|\cdot \|(\cP_k - z)^{-1} u\|_{\Zspace{n+1}}\Big)\\
& \leq C \Big(|\Im(z)|^{-1}\|u\|_{\cH} + \|u\|_{\Zspace{n+1}} +  \frac{\langle z \rangle^{m+1}}{|\Im(z)|} \|u\|_{\Zspace{n-1}}\Big)\\
& \leq C \frac{\langle z \rangle^{m+1}}{|\Im(z)|} \|u\|_{\Zspace{n+1}}.
\end{align*}
Thus by induction, for all $n \geq 0$, 
$$
\|(\cP_k - z)^{-1}\|_{\ZcupHspace{n-1} \to \ZcapHspace{n+1}} \leq \frac{\langle z \rangle^{ 1+\lfloor |n|/2\rfloor}}{|\Im(z)|}.
$$
The result for $n \leq 0$ follows by duality. 
\end{proof}
%\bpf
%
%		First, observe that by Assumption~\ref{ass:ell}, for $u\in \cH$, 
%\begin{align*}
%		\|(\cP_k-z)^{-1}u\|_{\Zspace{2}}&\leq C(\| \cP_k(\cP_k-z)^{-1}u\|_{\cH} +\|(\cP_k-z)^{-1}u\|_{\cH})\\
%		&= C(\|u\|_{\cH}+|z|\| (\cP_k-z)^{-1}u\|_{\cH} +\|(\cP_k-z)^{-1}u\|_{\cH})\\
%		&\leq C(1+\frac{\langle z\rangle}{|\Im z|})\|u\|_{\cH}.
%		\end{align*}
%		Assume by induction that for some $n\geq 1$, 
%		$$
%		\|(\cP_k-z)^{-1}\|_{\Hspace{2n-2}\to \Zspace{2n}}\leq C\langle z\rangle^{n-1}\Big(1+\frac{\langle z\rangle}{|\Im z|}\Big).
%		$$
%		Then by Assumption~\ref{ass:ell}, for $u\in \Hspace{2n-2}$, 
%\begin{align*}
%		\|(\cP_k-z)^{-1}u\|_{\Zspace{2n+2}}&\leq C(\| \cP_k(\cP_k-z)^{-1}u\|_{\Hspace{2n}} +\|(\cP_k-z)^{-1}u\|_{\cH})\\
%		&= C(\|u\|_{\Hspace{2n}}+|z|\| (\cP_k-z)^{-1}u\|_{\Hspace{2n}} +\|(\cP_k-z)^{-1}u\|_{\cH})\\
%		&\leq C(1+ |z|\langle z\rangle^{n-1}(1+\frac{\langle z\rangle}{|\Im z|}) +\frac{1}{|\Im z|})\|u\|_{\Hspace{2n}}\\
%		&\leq C\langle z\rangle^n\big(1+\frac{\langle z\rangle}{|\Im z|}\big)\|u\|_{\Hspace{2n}}.
%		\end{align*}
%		\epf

\subsection{The spaces $\Dspace{s}$}

Since $\lambda \mapsto (\lambda + \CGa(k_0) + 1)^{s/2}, s \geq 0,$ is finite for $\lambda\in \sigma(\cP_k)$, the functional calculus of unbounded self-adjoint operators (see, e.g., \cite[Section 5.3]{schmudgen2012unbounded}) allows us to define the self-adjoint operator 
\beq\label{e:zoomWorking3}
\cX_{k,s} := (\cP_k + (\CGa(k_0) + 1)\Id)^{s/2}
\eeq
with domain $$\Dspace{s} := \mathcal{D}(\cX_{k,s}) \subset \cH$$ 
where the inclusion is dense in the $\cH$ norm.
Since the functional calculus is an algebra homomorphism, $\cX_{k,s} = \cX_{k}^s$, where $\cX_k := \cX_{k,1}$. Since $\cX_k^s$ is self-adjoint, it is, in particular, a closed operator, so the space $\cD_k^s$ is a Hilbert space for the graph norm
\beqs%\label{e:graphNorm}
\|u\|^2_{\cX_k^s} := \|u\|^2_{\cH} + \|\cX_k^s u\|^2_{\cH}.
\eeqs
Moreover, 
%by the functional calculus, 
$\sigma(\cX_k^s) \subset [1,+\infty)$, hence 
$$\|u\|^2_{\cH} \leq (u,\cX_k^s u)_{\cH} \leq \|u\|_{\cH} \|\cX_k^s u\|_{\cH},$$
so the graph norm associated of $\cX_k^s$ is equivalent to the norm
\beq\label{e:zoomWorking2}
\|u\|^2_{\cD^s_k} := \|\cX_k^{s} u\|^2_{\cH}.
\eeq
This way, the operator $\cX_k^t$ induces an isometry from $\cD^s_k$ to $\cD_k^{s-t}$ for all $s \geq t \geq 0$. 
\begin{proposition}
\label{prop:Z2=D2}
$\Zspace{2} = \Dspace{2}$ with equal norms. Furthermore $\cZ_k = \mathcal{D}^{1}_k$ with equivalent norms; more precisely, for all $k_0 >0$, there exist constant $C(k_0) > 0$ such that for all $k \geq 0$ and for all $u \in \mathcal{D}^2$
$$\frac{1}{C(k_0)} \|u\|_{\cZ_k} \leq \norm{u}^2_{\cD_k^1} \leq C(k_0) \|u\|_{\cZ_k}.$$
\end{proposition}
\begin{proof}
The first statement follows from the fact that $\cX_k^2$ and $\cP_k$ differ by a multiple of identity, and by the definition of the norm of $\Zspace{2}$ (compare \eqref{e:zoomWorking1} with the combination of \eqref{e:zoomWorking3} and \eqref{e:zoomWorking2}). On the other hand, $\cX_k^2$ is the operator associated to the lower semi-bounded form $a_k^+ : \cZ_k \times \cZ_k \to \C$ defined by
$$a_k^+(u,v) := \Re a_k(u,v) + (\CGa(k_0) + 1) (u,v)_{\cH}$$
in the sense of \cite[Definition 10.4]{schmudgen2012unbounded}. In particular, by Theorem 10.7 and Proposition 10.5 in the latter reference, 
$$\cZ_k = \mathcal{D}(|\cX_k^2|^{1/2}) = \mathcal{D}(\cX_k) = \Dspace{1}.$$
The equivalence of the norms follows from the continuity of $\Re a_k$ and the G\aa rding inequality. 
\end{proof}
\begin{corollary}
\label{cor:DnsubZn}
For all $n \in \mathbb{N}$, $\Dspace{n} \subset \Zspace{n}$ and the embedding is continuous. 
\end{corollary}
\begin{proof}
The result is immediate if $n = 0$ and is Proposition \ref{prop:Z2=D2} above for $n = 1,2$. Finally, if $u \in \cD_k^{n+2}$, then 
$$\cP_k u = (\cX_k^2 - \CGa(k_0) + 1)u \in \cD_k^{n}$$
so the result follows by induction using elliptic regularity (Assumption \ref{ass:ell}).
\end{proof}

\

We also define $\cD_k^{-s} := (\cD_k^s)^*$. 
Since $\Dspace{s}$ is dense in $\Hspace{}$ for $s\geq 0$, $\cH$ can be identified as a subspace of $\Dspace{-s}$, so that 
$\cD_k^{t} \subset \cD_k^s$ for all real $s \leq t$. We can extend $\cX_k$ uniquely into a linear map from $\cD_k^s$ to $\cD_k^{s-1}$ for all $s \in [0,1]$ by putting
$$\langle\cX_k u,v\rangle := \langle \cX_k^{s}u,\cX_k^{1-s} v\rangle \quad \tfa (u,v) \in \cD_k^s \times \cD_k^{1-s}.$$ 
This way, $\cX_k : \cD_k^{s} \to \cD_k^{s-1}$ for all $s \in [0,+\infty)$ is an isometry and this is extended to $s \leq 0$ by duality. In turn, this allows us to view $\cP_k$ as a map $\cP_k: \cD_k^s \to \cD_k^{s-2}$ for all $s \in \R$ by $\cP_k := \cX_k^2 - (\CGa(k_0)+1)\Id$, with 
$$\|\cP_k u\|_{\cD_k^{s}} \leq (\CGa(k_0) + 1)\|u\|_{\cD_k^{s+2}}.$$

\begin{proposition}[Resolvent estimates in the $(\cD_k^s)$ scale]\label{prop:neverInBin}
Let $k_0 > 0$ and $s \in \mathbb{R}$. There exists $C > 0$ such that for all $k > 0$ and all $z \in \mathbb{C} \setminus \R$, 
$$\|(\cP_k - z)^{-1}\|_{\cD_k^s \to \cD_k^s} \leq  |\Im(z)|^{-1},$$
$$ \|(\cP_k - z)^{-1}\|_{\cD_k^s \to \cD_k^{s+2}} \leq  C\langle z \rangle|\Im(z)|^{-1},$$
where $\langle z \rangle := 1 + |z|$. 
\end{proposition}
\begin{proof}
Using the fact that $\cX_k^s : \cD_k^s \to \cH$ is an isometry, and functional calculus,
$$\|(\cP_k - z)^{-1}\|_{\cD_k^s \to \cD_k^s} \leq \|\cX_k^s (\cP_k - z)^{-1} \cX_k^{-s}\|_{\cH \to \cH} = \|(\cP_k - z)^{-1}\|_{\cH \to \cH} \leq |\Im(z)|^{-1}.$$
Similarly, 
$$\|(\cP_k - z)^{-1}\|_{\cD_k^{s} \to \cD_k^{s+2}} = \|\cX_k^{s+2} \big(\cX_k^2 - (\CGa + z)\big)^{-1} \cX_k^{-s} \|_{\cH \to \cH}  = \|g(\cX_k^2)\|_{\cH \to \cH} \leq \sup_{x \in \R} |g(x)|$$
where $g(x) := \frac{x}{x- (z + \CGa)}$. Since for all $z \in \C \setminus \R$, 
$$\sup_{x \in \R} \left|\frac{x}{x-z}\right| = \frac{|z|}{|\Im(z)|},$$
we conclude that 
$$\sup_{x \in \R}|g(x)| \leq (1 + \CGa) \langle z \rangle |\Im(z)|^{-1},$$
completing the proof. 
\end{proof}

\begin{proposition}[Functions of $\cP_k$]
\label{prop:f(Pk)}
Let $k_0 > 0$ and $s \geq 0$. There exists $C > 0$ such that, for all $k \geq k_0$ and for any function $f : \mathbb{R} \to \mathbb{C}$ satisfying 
$$\|f\|_{\infty,s} := \sup_{x \in \R} (1 + |x|^s)|f(x)| < \infty,$$
the operator $f(\cP_k) : \cH \to \cH$ defined by the functional calculus extends uniquely into a continuous map from $\cD_k^{-s}$ to $\cD_k^{s}$, with
$$\|f(\cP_k)\|_{\cD_k^{-s} \to \cD_k^{s}} \leq C \|f\|_{\infty,s}.$$
In particular (by Corollary \ref{cor:DnsubZn} and the definitions of $\ZcupHspace{n}$ and $\ZcapHspace{n}$ \eqref{e:capcupdefs}) for any $n\in \mathbb{N}$, $f(\cP_k) : \ZcupHspace{-n} \to \ZcapHspace{n}$ is continuous.
\end{proposition}
\begin{proof}
By functional calculus and using that $\cX_k^t = (\cP_k + \CGa(k_0) + 1)^{t/2} : \cD_k^t \to \cH$ is an isometry for all $t \in \R$, 
$$\|f(\cP_k)\|_{\cD_k^{-s} \to \cD_k^{s}} \leq \|(\cP_k + \CGa(k_0) + 1)^{s/2}f(\cP_k)(\cP_k + \CGa(k_0) + 1)^{s/2}\|_{\cH\to \cH} = \|g(\cP_k)\|_{\cH \to \cH}$$
where $g(x) = (\CGa + 1 + x)^{s} f(x)$ satisfies
$$|g(x)| \leq 2^{s}(\CGa + 1)^s (1 + |x|^s) |f(x)|$$
for all $x \in \sigma(\cP_k)$. Hence, $\|g(\cP_k)\|_{\cH \to \cH} \leq C \|f\|_{\infty,s}$ and the claim follows. 
\end{proof}

\subsection{Elliptic perturbation of $P_k$}
\label{sec:Pksharp}
By Proposition \ref{prop:Pksa}, for every $k_0 > 0$, there exists a real-valued, compactly supported function $\psis \in C^\infty_c(\R)$ such that
\begin{equation}
	\label{eq:conditionPsi}
	\psis(x) \geq \frac{-x + \CGa}{2}\quad\tfa x \in \sigma(\cP_k).
\end{equation}
Following \cite[Lemma 2.1]{GS3}, define 
\begin{equation}
\label{e:def_Sk}
S_k := \psis(\cP_k)
\end{equation}
by the functional calculus. Since $\psis$ has compact support, 
$$S_k : \cD_k^{-n} \to \cD_k^{n}$$ is continuous for all $n \in \mathbb{N}$ by Proposition \ref{prop:f(Pk)}. In what follows, the {\em elliptic perturbation} of $P_k$ is defined by 
\begin{equation}
\label{e:defPksharp}
P^{\sharp}_k := P_k + S_k.
\end{equation}
The associated sesquilinear form, denoted by $a_k^\sharp: \cZ_k \times \cZ_k \to \mathbb{C}$, is thus given by
\begin{equation}
	\label{def:aksharp}
	a_k^\sharp(u,v) := a_k(u,v) + (S_k u,v)_{\cH}.
\end{equation}
\newcommand{\Csharp}{C_{\sharp}}
\begin{proposition}[Properties of $P^\sharp_k$]
	\label{prop:ResPksharp}
	For every $k_0 > 0$ and any integer $n$, there exists a positive constant $\Csharp(k_0,n)$ such that, for all $k \geq k_0$, 
	\begin{equation}
		\label{eq:CoerciveAsharp}
		\Re (a_k^\sharp(u,u)) \geq \Csharp(k_0,1) \|u\|^2_{\cZ_k} \quad \tfa u \in \cZ_k,
	\end{equation}
	the operator $P_k^\sharp : \ZcapHspace{n+2} \to \ZcupHspace{n}$ is an isomorphism, and
	\[\|(P_k^\sharp)^{-1}u\|_{\ZcapHspace{n+2}} \leq \Csharp(k_0,n) \|u\|_{\ZcupHspace{n}} \quad \tfa u \in \ZcupHspace{n}. \]
\end{proposition}
\begin{proof} 
%	Let $a_k^\sharp(u,v):=a_k(u,v) + (S_k u,v)_{\cH}$ for $u,v \in \cZ_k$. This defines a continuous sesquilinear form on $\cZ_k \times \cZ_k$. 
	By \eqref{eq:conditionPsi},
	\beqs
	x + \psi^\sharp(x) \geq \frac{1}{2} (x+ \CGa) \quad \tfa x \in \sigma(\cP_k). 
	\eeqs
	Therefore, by the functional calculus, for all $u \in \Zspace{1}$,
	\[\Re(a_k^{\sharp}(u,u)) = \Re a_k(u,u) + (\psi^\sharp(\cP_k)u,u)_{\cH} =  \big(\cP_k + \psi^\sharp(\cP_k) u,u\big)_{\cH} \geq \frac{1}{2}\big((\cP_k + \CGa \Id) u,u\big)_\cH.\]
	%using the fact that $\cP_k + \psi(\cP_k) \geq \frac{1}{2}(\cP_k+ \CGa \Id)$. 
	Hence, by the G\aa rding inequality, 
	\[\Re(a_k^\sharp(u,u)) \geq \frac{\cGa}{2} \|u\|^2_{\cZ_k}
	\quad\tfa u\in \Zspace{2},\]
and the same inequality holds for all $u \in \cZ_k$ by the density of $\cZ^2_k$ in $\cZ_k$ and continuity of $a_k^\sharp$. Thus, $a_k^\sharp$ is coercive, and the Lax-Milgram lemma implies that $P^\sharp_k :\cZ_k \to (\cZ_k)^*$ is boundedly invertible; this is the required result for $n = 1$. With $n \geq 2$, let $u \in \Hspace{n-2}$ and suppose that $v \in \cZ_k$ satisfies $P^\sharp_k v = u$. Then $P_k v = u - S_k v \in \Hspace{n-2}$ (by the smoothing property of $S_k$ from Proposition \ref{prop:f(Pk)}), so that $v \in \Zspace{n}$ by elliptic regularity (Assumption \ref{ass:ell}). Moreover, since $\|v\|_{\cH} \leq \|v\|_{\cZ_k} \leq \|u\|_{\cZ_k^*}$ (again by the Lax-Milgram lemma), 
\begin{align*}
\|v\|_{\Zspace{n}} 
&\leq C\big(\|v\|_{\cH} + \|u - S_k v\|_{\Hspace{n-2}}\big)\leq C\big(\|u\|_{(\cZ_k)^*} + \|u\|_{\Hspace{n-2}}\big) 
\leq C\|u\|_{\Hspace{n-2}},
\end{align*}
which proves the result for $n \geq 2$. The same reasoning applied to $P_k^* + S_k^* = P_k^* + S_k$ followed by a duality argument (recalling that the dual of  $\ZcupHspace{n}$ is $\ZcapHspace{-n}$) gives the result for $n \leq 0$.
\end{proof}

\subsection{Order notation}\label{sec:order_notation}

\newcommand{\etak}{\eta}
Let
$$\ZcapHspace{\infty} := \bigcap_{n \in \Z} \ZcapHspace{n}\,, \quad \ZcapHspace{-\infty} := \bigcup_{n\in \Z} \ZcapHspace{n},$$
and define $\ZcupHspace{\pm\infty}$ and $\cD_k^{\pm\infty}$ similarly. 

\begin{definition}[Order notation]
\label{def:orderNotation}
Let $(\eta_k)_{k>0}$ be a family of real numbers.
Let $m,n \in \Z$ and let $L : \cW_k^\infty \to \cY_k^{\infty}$ be a linear operator. 
Then
$$L = O_{m}(\eta^{n};\ZcapHspace{} \to \ZcupHspace{})$$
if, for all $k_0 > 0$ and for all $j \in \mathbb{Z}$, there exists a real number $C(k_0,j) > 0$ such that for all $k \geq k_0$ and all $u \in \ZcapHspace{j}$, 
\beq\label{e:rocky1}
\|L u\|_{\ZcupHspace{j-m}} \leq C(k_0,j) \eta_k^{n} \|u\|_{\ZcapHspace{j}}.
\eeq
The notations $L = O_{m}(\eta^{n};\ZcupHspace{} \to \ZcupHspace{})$, $L = O_{m}(\eta^{n};\cD_k \to \cD_k)$ are defined similarly. 
\end{definition}

Observe that these order relations can then be combined multiplicatively; e.g., 
$$L_1 = O_{m}(\eta^{n};\ZcapHspace{} \to \ZcupHspace{})\,, \,\, L_2 = O_{m'}(\eta^{n'};\ZcupHspace{} \to \ZcapHspace{}) \implies L_1L_2 = O_{m + m'}(\eta^{n + n'}; \ZcupHspace{} \to \ZcupHspace{}).$$

\subsection{Spatial pseudolocality}
\label{s:spatialPseudo}
The statement and proof of the main result use both frequency cut-offs in the form $f(\cP_k)$ for $f \in \mathcal{S}(\R)$ and spatial cut-offs coming from smooth compactly-supported functions. The following assumption encapsulates the properties of these spatial cut-offs that are required in this section. \blueOld{In what follows, for a linear operator $A$, we use the notation $\ad_A$ for the formal operator defined by
$$\ad_A : B \mapsto AB - BA\quad \tand \quad \ad_A^N:=\underbrace{\ad_A\circ\ad_A\circ\dots\circ \ad_{A}}_{N\textup{ times}}$$
(in all cases where we use this notation, the domains of $A$ and $B$ are such that $\ad_A B$ makes sense).}

%\blueOld{For a linear operator $A$, we write $\ad_A$ for the operator defined by
%$$\ad_A B := AB - BA$$
%whenever the domains of the operators $A$ and $B$ are compatible.}

%that takes any linear operator $B$ to $[A,B]$.}

\newcommand{\Lcut}{\mathcal{L}_{\rm sc}} 

\begin{definition}[Abstract ``spatial cutoffs"]~
	\label{def:spatialCutoffs}
	Let $(\eta_k)_{k > 0}$ be a family of real numbers. We say that a linear operator 
	is a {\em spatial cutoff of order $m$ with parameter $\eta$} if $R= O_m(1; \ZcupHspace{} \to \ZcupHspace{})$,
	\begin{equation*}
%	\label{e:adNRQ_defLcut}
	\begin{gathered}
		\ad_{R}^N Q = O_{m - N + 2}(\eta^{-N}; \cW_k \to \cY_k), \quad \tand \quad \ad_{R^*}^N Q = O_{m - N + 2}(\eta^{-N};\cW_k \to \cY_k),
		\end{gathered}
	\end{equation*}
where $Q$ is any one of the operators $P_k$, $P_k^*$ and $\cP_k$. 
%	\begin{enumerate}[(i)]
%%		\item \label{ass:Spat:sepProp}{\em (Separation property)} If $A,B \in \Lcut$ are such that $AB = O_{0}(\etak^{-\infty})$, then there exists $R \in \Lcut$ such that $A(\Id - R) = O_0(\etak^{-\infty})$ and $RB = O_0(\etak^{-\infty})$. 
%		\item \label{ass:Spat:commProp1} With $Q$ equal to $P_k$, $P_k^*$ or $\cP_k$,
%		for all $R \in \Lcut^t$ and integers $N \geq 0$, $n\in\mathbb{Z}$, there exists $C>0$ such that 
% 		$$\|\ad_R^N Q \|_{\ZcapHspace{n}\to \ZcupHspace{n-2+N(1-t)}}\leq C\eta^{-N}.$$
%%\item For all $R\in \Lcut$, $R:\Dspace{2}\to \Dspace{2}$.
%\item For all $R\in \Lcut^t$, $R^*\in \Lcut^t$.
%	\end{enumerate}
The set of spatial cutoffs of order $m$ and parameter $\eta$ is denoted by $\Lcut^m(\eta)$, and we write $\Lcut(\eta) := \Lcut^0(\eta)$. We omit the $\eta$ from the notation when it will not lead to confusion. 
%We say that $R$ is bounded by $(\mathscr{C}_N)_{N \in \mathbb{N}}$ if $R$ is bounded by $\mathscr{C}_0$ and, for each $N\geq 1$, the operators in \eqref{e:adNRQ_defLcut} are bounded by $\mathscr{C}_N$.  
\end{definition}

\begin{remark}
Recall from Remark \ref{rem:Pknotdiff} that in the model settings of Table \ref{tab:model_settings}, the operator $Q$ above is not a differential operator. Therefore the commutators $\ad^N_R Q$ and $\ad^N_{R^*} Q$ a priori contain boundary terms, hence some care must be taken to check the continuity properties above. In \S\ref{sec:assumptions}, we show that if $R$ is given by the multiplication with a smooth cut-off function $\chi$ {\em with vanishing normal derivative on $\partial \Omega_-$}, then it satisfies the commutator estimates above.
\end{remark}

Let
$
A,B = O_0(1,\ZcapHspace{}\to\ZcapHspace{}) \cap O_0(1,\ZcupHspace{}\to\ZcupHspace{})\,,
$
where the intersection notation is used to denote that the equation holds with either term on the right-hand side, and
and let $R \in \Lcut$. We say that $A$ and $B$ are \emph{separated by $R$} if both
	\[A(\Id -R) = O_{0}(\eta^{-\infty}; \ZcupHspace{}\to \ZcupHspace{})\cap O_{0}(\eta^{-\infty}; \ZcapHspace{}\to \ZcapHspace{})\,,
\]
\[	
	 RB = O_0(\eta^{-\infty};\ZcupHspace{}\to \ZcupHspace{})\cap O_0(\eta^{-\infty};\ZcapHspace{}\to \ZcapHspace{}),\]
	 We say that $A$ and $B$ are {\em separated} if they are separated by $R$ for some $R\in \Lcut$.
	
	The main result on spatial pseudolocality is as follows.
	
	\begin{theorem}[Pseudolocality of abstract Helmholtz operators]
	\label{thm:pseudolocSpace}
	%Let $M > 0$, for all $N \in \mathbb{N}$, let $\mathscr{C}_N := \{C_{N}(k_0,j)\}_{k_0 > 0,j \in \Z}$ and 
	Let $f \in \mc{S}(\R)$, let
	%Then for all $N \in \mathbb{N}$, there exists $\mathscr{C}_N' := \{C'_{N}(k_0,j)\}_{N \in \mathbb{N},k_0 > 0,j \in \Z}$ such that 
	$A,B$ be separated, and let $Q$ be one of the operators $P_k,P_k^*$ or $\cP_k$. Then
	\begin{align}	\label{eq:pseudoloc1}
	A f(\cP_k) B &= O_{-\infty}(\etak^{-\infty};\ZcupHspace{}\to \ZcapHspace{}),\\% \cap O_{-\infty}(\etak^{-\infty};\ZcapHspace{}\to \ZcapHspace{}),\\
	 A Q B &= O_{2}(\etak^{-\infty};\ZcapHspace{}\to \ZcupHspace{})
		\label{eq:pseudoloc2}
	\\
	A(P_k^{\sharp})^{-1} B &= O_{-2}(\etak^{-\infty};\ZcupHspace{}\to \ZcapHspace{}).
		\label{eq:pseudoloc3}
	\end{align} 
\end{theorem}

\begin{remark}[The constants implicit in Theorem \ref{thm:pseudolocSpace}]
	\label{r:paradise}
Both the assumptions and the conclusions of Theorem \ref{thm:pseudolocSpace} 
involve implicit constants coming from the order notation Definition \ref{def:orderNotation} used to denote bounds of the form \eqref{e:rocky1}. 
 It is clear from the proof of Theorem \ref{thm:pseudolocSpace} (but cumbersome to write precisely) that given $f \in \cS(\R)$ and a list of constants $\mathscr{C}$, there is another list of constants $\mathscr{C'}$ such that for all $(\eta_k)_{k >0}$, $A$, $B$ and $R$ satisfying the assumptions of the theorem with the constants $\mathscr{C}$, the conclusions of Theorem \ref{thm:pseudolocSpace} hold with the constants $\mathscr{C'}$. The same is true of all results in the remainder of this section. 
\end{remark}

We first reduce the proof of Theorem~\ref{thm:pseudolocSpace} to the proof of various commutator estimates.
\begin{proposition}\label{prop:excessive1Space}
%	Let $A, B \in \Lcut$ be separated, $C \in \Lcut$, $f,g \in C^\infty(\R)$ be such that one of $f$ and $g$ is compactly supported and $\supp f \cap \supp g = \emptyset$, and Let 
Suppose that 
for all $R \in \Lcut$, for all $f \in \mathcal{S}(\R)$, and for every $N \in \mathbb{N}$, 
\begin{align}
	\ad_{R}^N f(\cP_k) &= O_{-\infty}(\etak^{-N}; \ZcupHspace{} \to \ZcapHspace{}),%\cap O_{-\infty}(\etak^{-N}; \ZcapHspace{} \to \ZcapHspace{}),
	\label{eq:adNR1}\\
	\ad_{R}^N (P^\sharp_k)^{-1}  &= O_{-2}(\etak^{-N}; \ZcupHspace{} \to \ZcapHspace{}).\label{eq:adNR2}
\end{align}
Then the results of Theorem \ref{thm:pseudolocSpace} hold.
\end{proposition}

%Let $A\in \cL^{m_A}$ and $B\in \cL^{m_B}$. 
\bpf
The structure of the argument is the same for all three results, albeit with different spaces. We therefore omit the spaces from the notation for brevity.  

Since $A$ and $B$ are assumed to be separated, there exists $R \in \Lcut$ such that $A (\Id - R) = O_0(\eta^{-\infty})$ and $RB = O_0(\eta^{-\infty})$. Hence, for any operator $X$ such that $X = O_m(1)$, since $A, B\in \Lcut\subset \cL^0$,
\[\begin{split}
	A X B &= A R X B + A(\Id - R)  X B \\
	 	  &= A (\ad_{R} X) B +  A X R B +  A(\Id - R) X B\\
 		  & = A (\ad_{R}X) B + O_{m}(\eta^{-\infty}).
 	  \end{split}
\]
Repeating this argument $N$ times gives
\beq\label{eq:dangling1}
AXB = A (\ad_{R}^N X) B + O_{m}(\eta^{-\infty}),
\eeq
since $\ad_{R}^i X = O_{m}(1)$ for any $i \in \mathbb{N}$ (as can be easily checked by induction). By \eqref{eq:dangling1} applied with $X = f(\cP_k)$ (and $m=-\infty$ by Proposition \ref{prop:f(Pk)}),
% $i=0,\ldots,N$ by 
$X = Q$  (and $m=2$), and $X = (P_k^\sharp)^{-1}$ (and $m=-2$ by Proposition \ref{prop:ResPksharp}), for all $N \in \mathbb{N}$,
\[Af(\cP_k) B = A (\ad_{R}^N f(\cP_k)) B + O_{-\infty}(\eta^{-\infty})\] 
\[AQ B = A (\ad_{R}^N Q) B + O_{2}(\eta^{-\infty})\]
\[A (P^\sharp_k)^{-1} B = A (\ad_{R}^N (P^\sharp_k)^{-1}) B + O_{-2}(\eta^{-\infty}).\]
%where we have used that $Q\in \cL^2$, $(P^\sharp_k)^{-1}\in \cL^{-2}$, and $f(\cP_k)\in \cL^{-\infty}$.
Hence \eqref{eq:pseudoloc2} follows immediately from the middle equality and Definition \ref{def:spatialCutoffs}, while the first and third equality show that \eqref{eq:pseudoloc1} and \eqref{eq:pseudoloc3} follow from \eqref{eq:adNR1} and \eqref{eq:adNR2}, respectively.

%For the results involving frequency cut-offs,
%For the remaining results \eqref{eq:pseudoloc1a}, \eqref{eq:pseudoloc5}, \eqref{eq:pseudoloc5a} (i.e., those involving operators sandwiched by frequency cut-offs), 
% without loss of generality, assume that $\supp f$ is compact (the proof when $\supp g$ is compact is analogous). 
%Since the sets $\supp f$ and $\supp g$ are compact and disjoint, there exists $f_1 \in C^\infty_c(\R)$ such that $\supp f_1\cap \supp g = \emptyset$ and $\supp (1 - f_1) \cap \supp f = 0$. Therefore, 
%\[f(\cP_k) X g(\cP_k) = f(\cP_k)(\ad_{f_1(\cP_k)}^N X) g(\cP_k).\]
%Therefore, with $X=C, Q,$ and $(P_k^\sharp)^{-1}$, 
% \eqref{eq:pseudoloc1a}, \eqref{eq:pseudoloc5}, \eqref{eq:pseudoloc5a} follow from \eqref{eq:adNf1}, \eqref{eq:adNf2}, and \eqref{eq:adNf3}, respectively.
%and so to prove the pseudolocality results concerning frequency cutoffs it suffices to show that
%XXXX
%In the remainder of this section, we show these estimates, see Propositions \ref{prop:proofadNR1}, \ref{prop:proofadNR2},\ref{prop:proofadNf1},\ref{prop:proofadNf2} and \ref{prop:proofadNf3}, respectively. 
\epf

We now prove that the assumptions \eqref{eq:adNR1}-\eqref{eq:adNR2} of Proposition \ref{prop:excessive1Space} hold true. In order to do this, 
%we write $f(\cP_k)$, 
%and we write this via 
the main tool the Helffer-Sjöstrand formula, which allows to express $f(\cP_k)$ in terms of $(\cP_k - z)^{-1}$. This formula is recalled below (for a proof,
% of this formula, 
see, e.g., \cite[Theorem 14.8]{Zw:12}).

\begin{proposition}[Helffer-Sjöstrand formula]
	\label{prop:HS}
	For all $f \in \mathcal{S}(\R)$, there exists a continuous function $w : \mathbb{C} \to \C$ such that if $\cA$ is a self-adjoint operator on a Hilbert space, then
	\begin{equation*}
%		\label{eq:HJformula}
		f(\cA) = \int_{\mathbb{C}} w(z) (\cA - z)^{-1} d m_{\C}(z)
	\end{equation*}
	where $dm_\C(x + iy)= dx\,dy$ and for every $M \in \mathbb{N}$, there exists $\kappa_{M}$ such that 
	\begin{equation}
		\label{eq:condw}
		\abs{w(z)} \leq \kappa_{M} \langle z \rangle^{-2M}\abs{\Im(z)}^{M} \qquad \tfa z \in \C.
	\end{equation}
\end{proposition}

The function $w$ in Proposition \ref{prop:HS} is obtained (up to a constant factor) via a so-called ``quasi-analytic extension" of $f$; see, e.g., \cite[Theorem 3.6]{Zw:12}. 

By the Helffer-Sjöstrand formula, 
$$\ad_R^N f(\cP_k) = \int_{\C} w(z) \ad_R^N (\cP_k - z)^{-1} dm_\C(z)$$
Therefore, to prove \eqref{eq:adNR1}, \eqref{eq:adNR2} we need to bound 
%, the approach used above
%of writing $f(\cP_k)$ using the Helffer--Sj\"ostrand formula requires 
%bounds on  
%commutator terms of the form
$\ad_{R}^N Y^{-1}$
with $Y = P_k^\sharp$ or $Y = (\cP_k - z)^{-1}$, in terms of $\ad_R^N Y$ and $Y^{-1}$. %, with some control on $z$ in the latter case. %This motivates the next Proposition. 

The next result uses the notation that $A = O_m \big(f(\eta,n,z);\ZcupHspace{}\to \ZcapHspace{}\big)$ if $\| A u \|_{\ZcapHspace{n-m}} \leq f(\eta_k,n,z) \| u\|_{\ZcupHspace{n}}$.
%(and similarly for 

\begin{proposition}
	\label{prop:atoms} 
	Let  $\Omega$ a subset of $\C$. 
Suppose that $X=O_m(1; \ZcupHspace{} \to \ZcupHspace{})\cap O_m(1; \ZcapHspace{} \to \ZcapHspace{})$
	%Let $X:\cH\to \cH $ such that $X,X^*:\Dspace{2}\to\Dspace{2}$, 
	and for every $z \in \Omega$, let $Y_z,Y_z^*:\ZcapHspace{n+2}\to \ZcupHspace{n}$ be invertible. Furthermore, suppose that there are $L_n\geq 0$ such that 
	\begin{itemize}
		\item[(a)] for all $z \in \Omega$, 
		\beqs
		Y_z^{-1}=  C_{1}(z)O_{-2}\big(\langle z\rangle^{L_n};\ZcupHspace{}\to \ZcapHspace{}\big) ,\quad (Y^*_z)^{-1}=  C_{1}(z)O_{-2}\big(\langle z\rangle^{L_n}; \ZcupHspace{}\to \ZcapHspace{}\big) 
		\eeqs
%		(where the index $n$ in $L_n$ refers to the indices in the spaces mapped between)
		\item[(b)] for all $z \in \Omega$, 
		\beqs
		\ad_X^N Y_z= C_2(z)O_{2+N(m-1)}\big(\etak^{-N};\ZcapHspace{}\to \ZcupHspace{}\big), \quad \ad_{X^*}^N Y^*_z= C_2(z)O_{2+N(m-1)}\big(\etak^{-N};\ZcapHspace{}\to \ZcupHspace{}\big)
		\eeqs
	\end{itemize} 
	for some functions $C_1,C_2: \Omega \to \R_+$. Then  for all $N\in \mathbb{N}$, $n\in \mathbb{Z}$, there is $M_n$ such that, for all $z \in \Omega$,
	\[\ad_{X}^N Y_z^{-1} = (1+C_1(z))^{N+1} (1+C_2(z))^{N} O_{-2+N(m-1)}\big(\etak^{-N}\langle z\rangle^{M_n};\ZcupHspace{}\to \ZcapHspace{}\big).\]
\end{proposition}
\begin{proof}
	The main idea is that $\ad_X^N Y_z^{-1}$ is equal to a linear combination of terms of the form $ Y_{z}^{-1}(\ad_X^{i_1} Y_z) Y_z^{-1} (\ad_X^{i_2} Y_z)Y_{z}^{-1} \ldots Y_{z}^{-1} (\ad_{X}^{i_M} Y_z) Y_{z}^{-1}$, and the next definitions formalize this more precisely.\footnote{It is in fact possible to give a full closed-form expression for $\ad_X^{N} Y^{-1}$ involving sums of compositions of quantities of the form $(\ad_{X}^i Y)$ and $Y^{-1}$. However, the fomula and its proof, involving sums over all possible ordered partitions of $\{1,\ldots,N\}$, are slightly cumbersome and for the present purposes, this would be more information than actually needed.} 
	
	We will prove the lemma by showing the estimate for $\ad_X^NY_z^{-1}$ acting on elements of $\cH$ and then (using the second parts of assumptions (a) and (b)) argue by duality to act on $\Hspace{-n}$. 
	
	An operator $a_N:\cH \to \cH$ is called an {\em $(N,z)$-atom} if either
	\begin{itemize}
		\item[(i)] $N = 0$ and $a_N = 1$, or
		\item[(ii)] $a_N = (\ad_{X}^N Y_z) Y_z^{-1}$, or
		\item[(iii)] $a_N = a_i a_j$ where $a_i$ is an $(i,z)$-atom and $a_j$ is an $(j,z)$-atom with $i+j = N$ and $1 \leq i,j \leq N-1$.
	\end{itemize}
	An operator $t_N$ is called an {\em $(N,z)$-term} if it is of the form
	\[t_N =\sum_{j = 1}^J \sigma_j Y_z^{-1}a_{N,j}\] 
	where $J \in \mathbb{N}$, $\sigma_j$ are real coefficients and $a_{N,j}$ are $(N,z)$-atoms. For example, 
	$$t_5 = Y_z^{-1} (\ad_{X}^5 Y_z) Y_z^{-1}  - Y_z^{-1} (\ad_{X}^2 Y_z) Y_z^{-1} (\ad_{X}^{3} Y_z) Y_z^{-1}$$ 
	is a $(5,z)$-term. Notice that if $t_i$ and $t_j$ are $(i,z)$- and $(j,z)$-terms, then $t_i Y_z t_j$ is an $(i+j,z)$-term. 
	
	It follows immediately from assumptions (a) and (b), by induction on $N$, that if $t_N(z)$ is a $(N,z)$-term for all $z \in \Omega$, then
	\[t_N(z) =(1+C_1(z))^{N+1} (1+C_2(z))^{N} O_{- 2+N(m-1)}(\etak^{-N} \langle z\rangle^{M_n};\ZcupHspace{}\to \ZcapHspace{}).\]
	Thus it remains to show that for all $z \in \Omega$, $\ad_{X}^NY_z^{-1}:\cH\to \cH$ is an $(N,z)$-term. For this, it suffices to prove that, for all $N \in \mathbb{N}$, 
	\begin{equation}
		\label{eq:adRNterm}
		t_N \textup{ is an $(N,z)$-term } \quad \implies \quad \ad_X t_N \textup{ is an $(N+1,z)$-term}.
	\end{equation}
	By linearity, it is enough to prove \eqref{eq:adRNterm} in the case where $t_N = Y_z^{-1}a_N$ for some $(N,z)$-atom $a_N$. We consider separately the three cases (i), (ii), (iii) above in the definition of an $(N,z)$-atom.
%	\begin{itemize}
%		\item[Case (i):] 

		Case (i):~If $a_N = 1$, then %, since $X:\Dspace{2}\to \Dspace{2}$ and $X:\cH\to \cH$,
		\[\ad_{X}t_N = \ad_X Y_z^{-1} = XY_z^{-1}-Y_z^{-1}X=Y_{z}^{-1}Y_zXY_z^{-1}-Y_{z}^{-1}XY_zY_{z}^{-1}= -Y_z^{-1} (\ad_{X} Y_z) Y_z^{-1}\]
		which is a $(1,z)$-term acting on $u$. This shows the implication \eqref{eq:adRNterm} for $N = 0$, and in the following cases, we fix $N \geq 1$ and proceed by induction assuming that it holds for all $i \leq N-1$. 
%		\item[Case (ii):] 

		Case (ii):~If $a_N = (\ad_X^N Y_z) Y_z^{-1}$, then   
		\[\begin{split}
			\ad_X t_N = \quad & (\ad_X Y_z^{-1}) (\ad_X^N Y_z) Y_z^{-1}+ Y_z^{-1} (\ad_X^{N+1} Y_z) Y_z^{-1}
			+ Y_z^{-1} (\ad_{X}^N Y_z) (\ad_X Y_z^{-1}).
		\end{split}\] 
		The second term on the right-hand side is an $(N+1,z)$-term. The first term on the right-hand side can be rewritten as
		\[-Y_z^{-1} \underbrace{(\ad_ X Y_z)Y_z^{-1}}_{\textup{$(1,z)$-atom}} \underbrace{(\ad_X^N Y_z) Y_z^{-1}}_{\textup{$(N,z)$-atom}}.\]
		This is thus an $(N+1,z)$-term. Similarly, the third term is an $(N+1,z)$-term, and thus $\ad_X t_N$ is an $(N+1,z)$-term.
%		\item[Case (iii):] 	

		Case (iii):~If $a_N = a_i a_j$ then, since $a_j:\cH\to\cH$, 
		$$
		t_N =Y_z^{-1}a_ia_j=Y_{z}^{-1}a_iY_zY_z^{-1}a_j= t_i Y_z t_j
		$$
		 where $t_i := Y_z^{-1}a_i$ and $t_j := Y_z^{-1}a_j$ are $(i,z)$- and $(j,z)$-terms, respectively, with $i+j=n$.  Thus
		\[\ad_X t_N = (\ad_X t_i) Y_z t_j + t_i (\ad_X Y_z) t_j + t_i  Y_z (\ad_X t_j) .\]
		The first term is an $(N+1,z)$-term by the induction hypothesis. Similarly, the last term is an $(N+1,z)$-term. The middle term can be rewritten as
		\[t_i (\ad_X Y_z) t_j = \overbrace{t_i Y_z \underbrace{\underbrace{Y_z^{-1} (\ad_X Y_z) Y_z^{-1}}_{\textup{$(1,z)$-term}} Y_z t_j}_{\textup{$(j+1,z)$-term}}}^{\textup{$(i+(j+1),z)$-term}} \]
		which is an $(N+1,z)$-term. This concludes the proof.
%	\end{itemize}
\end{proof}

\

We can now prove the estimate  \eqref{eq:adNR1}.
\begin{proposition}\label{prop:proofadNR1}
	For any $N \in \mathbb{N}$, $f \in \mathcal{S}(\R)$ and $R \in \Lcut^m$,
\beqs
	\ad_{R}^N f(\cP_k) = O_{-\infty}(\etak^{-N}; \ZcupHspace{} \to \ZcapHspace{}).%\cap O_{-\infty}(\etak^{-N}; \ZcapHspace{} \to \ZcapHspace{}).
\eeqs
\end{proposition}
\begin{proof}
By the Helffer--Sj\"ostrand formula, 
	\[\ad_{R}^N f(\cP_k) = \int_{\C} w(z) \ad_{R}^N  (\cP_k - z)^{-1}\,dm_{\C}(z).\]
	By Proposition \ref{prop:atoms} with $X = R$, $\Omega = \C \setminus \R$, $Y_z = (\cP_k - z)$, the first commutator property of spatial cutoffs, and the resolvent estimate 
of Proposition \ref{prop:rescuedfrombin},	
	\[\ad_R^{N} (\cP_k - z)^{-1} = \Big(1+\frac{\langle z \rangle}{\abs{\Im(z)}}\Big)^N O_{-2+N(m-1)}(\etak^{-N}\langle z\rangle^{M_n}; \ZcupHspace{} \to \ZcapHspace{}).\]
	Therefore, 
	\[\ad_R^N f(\cP_k)  = O_{-2+N(m-1)}\left(\etak^{-N}\int_{\C^N} w(z) \langle z\rangle^{M_n}\Big(1+\frac{\langle z \rangle}{\abs{\Im(z)}}\Big)^N dm_{\C}(z); \ZcupHspace{} \to \ZcapHspace{}\right).\]
	The bound \eqref{eq:condw} on $w$ implies that the integral is finite, and thus, for all $f \in \mathcal{S}(\R)$,
	\begin{equation}
		\label{eq:toupgrade}
		\ad_R^N f(\cP_k)  = O_{-2+N(m-1)}(\etak^{-N}; \ZcupHspace{} \to \ZcapHspace{}).
	\end{equation}
	We now upgrade the regularity index from $-2+N(m-1)$ to $-\infty$ by induction on $N$. For $N = 0$, 
	\begin{equation}
		\label{eq:startInduct}
		\ad_{R}^0 f(\cP_k) = f(\cP_k) = O_{-\infty}(1;\ZcupHspace{} \to \ZcapHspace{})
	\end{equation}
	by Proposition \ref{prop:f(Pk)}. Next fix an integer $N \geq 1$ and suppose that for all $i \leq N-1$ and all $g \in \mathcal{S}(\R)$, 
	\[\ad_{R}^i g(\cP_k) = O_{-\infty}(\etak^{-i};\ZcupHspace{} \to \ZcapHspace{}).\]
By, e.g., \cite[Theorem 3.2]{Vo:84}, 
given a Schwartz function $f$, there exist two Schwartz functions $f_1$ and $f_2$ such that $f = f_1f_2$. Thus $f(\cP_k)= f_1(\cP_k) f_2(\cP_k)$ with $f_1,f_2 \in \cS(\R)$. Furthermore, by the Leibniz identity
	\[\ad_{X}^N (YZ) = \sum_{i = 0}^N \binom{N}{i} (\ad_{X}^i Y) (\ad_{X}^{N-i}Z).\]
	Thus,
\begin{align*}&\ad_{R}^N (f_1(\cP_k) f_2(\cP_k)) \\
&\qquad= f_1(\cP_k) (\ad_{R}^N f_2(\cP_k)) + (\ad_R^{N} f_1(\cP_k)) f_2(\cP_k) + \sum_{i = 1}^{N-1} \binom{N}{i} \ad_R^{i} f_1(\cP_k) \ad_{R}^{N-i} f_2(\cP_k).
\end{align*}
Bounding the first two terms on the right-hand side by  \eqref{eq:startInduct} and \eqref{eq:toupgrade}, and bounding the third term by the induction hypothesis, we obtain that
%Now, by the induction hypothesis, 
%	\[\ad_{R}^N (f_1(\cP_k) f_2(\cP_k)) = f_1(\cP_k) (\ad_{R}^N f_2(\cP_k)) + (\ad_R^{N} f_1(\cP_k)) f_2(\cP_k) + \sum_{i = 1}^{N-1} O_{-\infty}(\etak^{-i}) O_{-\infty}(\etak^{-N+i}),\]
%	and, by \eqref{eq:startInduct} and \eqref{eq:toupgrade}, 
	\[\ad_{R}^N (f_1(\cP_k) f_2(\cP_k)) = O_{-\infty}(1)O_{-2+N(m-1)}(\etak^{-N}) +  \sum_{i = 1}^{N-1} O_{-\infty}(k^{-i}) O_{-\infty}(\etak^{-N+i}),\]
	where all the operators are $\ZcupHspace{} \to \ZcapHspace{}$. Since $\ZcapHspace{n}\subset \ZcupHspace{n}$ for all $n$, 
$\ad_{R}^N (f_1(\cP_k) f_2(\cP_k)) = O_{-\infty}(\etak^{-N};\ZcupHspace{} \to \ZcapHspace{})$, and the proof is complete.
\end{proof}

\

We now record a variant of the previous result.
\begin{proposition}
\label{prop:variant_adNRf(Pk)}
Let $N \in \mathbb{N}$, $f \in \mathcal{S}(\R)$ and let $R = O_m(1;\Dspace{} \to \Dspace{})$ be such that 
$$\ad_R^N \cP_k = O_{m - N + 2}(\eta^{-N};\Dspace{} \to \Dspace{}).$$
Then 
$$\ad_R^N f(\cP_k) = O_{-\infty}(\eta^{-N};\Dspace{}\to\Dspace{})$$
\end{proposition}
The proof is the same as that of Proposition \ref{prop:proofadNR1},  using a variant of Proposition \ref{prop:atoms} in the scale $(\Dspace{s})_{s\in \mathbb{R}}$, and using the mapping properties of $(\cP_k - z), (\cP_k - z)^{-1}$ and $f(\cP_k)$ in this scale (with the latter two mapping properties coming from Propositions \ref{prop:neverInBin} and \ref{prop:f(Pk)}).

\begin{proposition}
	\label{prop:proofadNR2}
	For all $N \in \mathbb{N}$ and $R \in \Lcut^m$, 
	\[
	\ad_R^{N} (P_k^\sharp)^{-1} = O_{-2+N(m-1)}(\etak^{-N};\ZcupHspace{} \to \ZcapHspace{}).
	\]
\end{proposition}
\begin{proof}
	This follows from Proposition \ref{prop:atoms} applied with $X = R$, $\Omega = \{1\}$ and $Y_1 = P_k^\sharp$. Indeed, for assumption (a), the required estimate is given by Proposition \ref{prop:ResPksharp}, while for assumption (b), 
	\[\ad_{R}^N P_k^{\sharp} = \ad_{R}^N  P_k + \ad_R^N \psi(\cP_k) = O_{2+N(m-1)}(\etak^{-N}) + O_{-\infty}(\etak^{-N}) \]
	by the definition of spatial cutoffs 
	(i.e., Definition \ref{def:spatialCutoffs}) and by Proposition \ref{prop:proofadNR1}.
\end{proof}

\subsection{Boundary compatible operators and pseudolocality in frequency}
\label{s:frequencyPseudolocal}

%When $P_k$ is the Helmholtz operator, $\cZ_k^m = H_k^m(\Omega)\cap H_{k,0}^1(\Omega)$ (as in \S\ref{sec:assumptions} below), and $R$ is a compactly-supported smooth function, the properties (i) and (ii) follow by direct calculation and induction.

%\bre
%
%\ere

In some cases, we will need to use pseudolocality in the \emph{frequency space} in addition to the physical space. To this end, we introduce the set of {\em boundary compatible operators}, $\Lf^m$. % denotes a set of operators satisfying the following set of assumptions:

\begin{definition}[Boundary compatible operators]~
	\label{def:freqOp}
	$X=O_m(1; \Dspace{} \to \Dspace{})$ is a \emph{boundary compatible operator} of order $m$, $X\in\Lf^m$, if, 
%	\begin{enumerate}[(i)]
%%		\item \label{ass:Spat:sepProp}{\em (Separation property)} If $A,B \in \Lcut$ are such that $AB = O_{0}(\etak^{-\infty})$, then there exists $R \in \Lcut$ such that $A(\Id - R) = O_0(\etak^{-\infty})$ and $RB = O_0(\etak^{-\infty})$. 
%		 \item \label{ass:Spat:commProp2}  For all $X \in \Lf^t$ and 
%		 
for		 all integers $N \geq 0$,
		$$\ad_{\cP_k}^N X = O_{N+m}(\etak^{-N};\Dspace{} \to \Dspace{})
		\quad\tand\quad
		\ad_{\cP_k}^N X^* = O_{N+m}(\etak^{-N};\Dspace{} \to \Dspace{}).
		$$
%		\item For all $X\in \Lf^t$, .
%		\item For all $X\in \Lf^t$, $X^*\in \Lf^t$.
%	\end{enumerate}
We write $\Lf := \Lf^0$. 
		\end{definition}
		
		\bre
		We highlight that the requirement that an operator is boundary compatible is much more stringent that the requirement that it is an abstract spatial cut-off, essentially because a function $u \in \cH$ must satisfy many ``boundary conditions'' to belong to the spaces $\Dspace{n}$ for large $n$ (see Table \ref{tab:model_settings}).
		In the setting of the Helmholtz PML problem of \S\ref{sec:verifyboundarycompatible} below, we show that if $R$ is given by the multiplication with a smooth cut-off function {\em that is locally constant near the obstacle boundary, and vanishes near in a neighbourhood of the truncation boundary}, then it satisfies the commutator estimates above.
		\ere

\begin{theorem}[Frequency pseudolocality]
	\label{thm:pseudolocFreq}
	Let $X \in \Lf^m$, $f,g \in C^\infty(\R)$ be polynomially bounded, such that one of $f$ and $g$ is 
	%Schwartz,  
	in $\mc{S}(\R)$, and $\supp f \cap \supp g = \emptyset$. Then
\beq
f(\cP_k) Xg(\cP_k) = O_{-\infty}(\etak^{-\infty};\Dspace{} \to \Dspace{})
\label{eq:pseudoloc1a}.
\eeq
\end{theorem}

As in the case of spatial pseudolocality, the proof of Theorem~\ref{thm:pseudolocFreq} can be reduced to certain commutator estimates.
\begin{proposition}\label{prop:excessive1}
%	Let $A, B \in \Lcut$ be separated, $C \in \Lcut$, $f,g \in C^\infty(\R)$ be such that one of $f$ and $g$ is compactly supported and $\supp f \cap \supp g = \emptyset$, and Let 
Suppose that  for all $X\in \Lf^m$, 
	\begin{equation}
	\ad_{f(\cP_k)}^N X = O_{-\infty}(\etak^{-N};\Dspace{} \to \Dspace{}).\label{eq:adNf1}
	\end{equation}
Then the results of Theorem \ref{thm:pseudolocFreq} hold.
\end{proposition}
\bpf
Without loss of generality, assume that $f$ is Schwartz (the proof when $g$ is Schwartz is analogous). 
Since the sets $\supp f$ and $\supp g$ are disjoint, there exists $f_1 \in \mc{S}(\R)$ such that $\supp f_1\cap \supp g = \emptyset$ and $\supp (1 - f_1) \cap \supp f = 0$. Therefore, 
\[f(\cP_k) X g(\cP_k) = f(\cP_k)(\ad_{f_1(\cP_k)}^N X) g(\cP_k).\]
Therefore,
 \eqref{eq:pseudoloc1a}, follows from \eqref{eq:adNf1} and the mapping properties of $f(\cP_k)$ and $g(\cP_k)$ from Proposition \ref{prop:f(Pk)}.
%and so to prove the pseudolocality results concerning frequency cutoffs it suffices to show that
%XXXX
%In the remainder of this section, we show these estimates, see Propositions \ref{prop:proofadNR1}, \ref{prop:proofadNR2},\ref{prop:proofadNf1},\ref{prop:proofadNf2} and \ref{prop:proofadNf3}, respectively. 
\epf

\begin{proposition}
	\label{prop:proofadNf1}
	Given $f \in \mathcal{S}(\R)$, $m\in\mathbb{R}$, and $X \in \Lf^m$,
	\[\ad_{f(\cP_k)}^N X = O_{-\infty}(\etak^{-N};\Dspace{} \to \Dspace{})\]
	(i.e., the bound \eqref{eq:adNf1} holds).
\end{proposition}
\begin{proof}
Using the Helffer-Sjöstrand formula, commutators with $f(\cP_k)$ can be expressed in terms of commutators with $(\cP_k-z)^{-1}$: for all $f_j \in \mathcal{S}(\R)$, $j=1,\ldots,N$, all $X \in \Lf^m$, and for every integer $N \in \mathbb{N}$, 
\begin{equation}
	\label{eq:adfN}
	\begin{split}
		&\ad_{f_N(\cP_k)} \ad_{f_{N-1}(\cP_k)} \ldots \ad_{f_1(\cP_k)} X \\
		&\qquad= \int_{\C^N} w_1(z_1) \ldots w_N(z_N) (\ad_{(\cP_k - z_N)^{-1}} \ldots \ad_{(\cP_k - z_1)^{-1}} X)\,dm_{\C^N}(z_1,\ldots,z_N),
	\end{split}
\end{equation}
where $w_i$ is as in Proposition \ref{prop:HS} with $f = f_i$. 
Using the identities 
$\ad_{(\cP_k-z)^{-1}} X= -(\cP_k-z)^{-1}\ad_{\cP_k}X(\cP_k-z)^{-1}$, 
$\ad_X(YZ)=(\ad_X Y)Z+Y(\ad_XZ)$,  
and the fact that $\ad_{(\cP_k - z)^{-1}} (\cP_k - z')^{-1} = 0$, one obtains the formula
\begin{equation*}
%	\label{eq:adResCommut}
	\ad_{(\cP_k-z_N)^{-1}} \ldots \ad_{(\cP_k - z_1)^{-1}} X = (-1)^N \prod_{i = 1}^{N}(\cP_k - z_i)^{-1} (\ad_{\cP_k}^N X)\prod_{i = 1}^{N}(\cP_k - z_i)^{-1}.
\end{equation*}
Therefore, by the resolvent estimate in 
Proposition \ref{prop:neverInBin} and the commutator assumption for frequency cutoffs (in Definition \ref{def:freqOp}),
\begin{align*}
	\ad_{(\cP_k-z_N)^{-1}} \ldots \ad_{(\cP_k - z_1)^{-1}} X 
= \prod_{j=1}^N\langle z_j\rangle^{2}|\Im( z_j)|^{-2}O_{-3N+m}(\eta^{-N};\Dspace{}\to \Dspace{}).
\end{align*}
Using this in \eqref{eq:adfN} and recalling 
	the decay properties of the $w_j$ \eqref{eq:condw}, we obtain that, for any $f_1,\ldots,f_N \in \cS(\R)$, 
	\[\ad_{f_N(\cP_k)} \ldots \ad_{f_1(\cP_k)} X= O_{-3N+m}(\etak^{-N};\Dspace{} \to \Dspace{}).\]
	In particular, if $f_j=f$, then $\ad^N_{f(\cP_k)} X= O_{-3N+m}(\etak^{-N};\Dspace{} \to \Dspace{})$. 
	
	To upgrade the regularity index from $-3N+m$ to $-\infty$, we use again that, given a Schwartz $f$, there exist two Schwartz functions $f_1$ and $f_2$ such that $f = f_1f_2$. By the functional calculus, $f(\cP_k)= f_1(\cP_k) f_2(\cP_k)$.  Furthermore, $\ad_{WY}Z = (\ad_W Z) Y + W (\ad_YZ)$. Thus,
	\[\begin{split}
		\ad_{f(\cP_k)}^N X  = \ad_{f_1(\cP_k) f_2(\cP_k)} \ad_{f(\cP_k)}^{N-1} X& =\underbrace{(\ad_{f_{1}(\cP_k)} \ad_{f(\cP_k)}^{N-1} X)}_{O_{-3N}(\etak^{-N})}  \underbrace{f_{2}(\cP_k)}_{O_{-\infty}(1)} + \underbrace{f_{1}(\cP_k)}_{O_{-\infty}(1)} \underbrace{(\ad_{f_{2}(\cP_k)} \ad_{f(\cP_k)}^{N-1} X)}_{O_{-3N}(\etak^{-N})}  \\
		 & =O_{-\infty}(\etak^{-N};\Dspace{} \to \Dspace{})
	\end{split}\] 
	which completes the proof. 
\end{proof}

\section{Pseudolocality results applied to the Helmholtz problem}\label{sec:assumptions}

In this section, we specialise the results of the previous section to the Helmholtz PML operator, that is, to the case where $\mathcal{H}_k^n = H_k^n$, $a_k$ is defined by \eqref{e:def_ak} and $\cZ_k$ is defined by \eqref{e:defZk_concrete}. With these definitions fixed, we keep the rest of the notation from Section \ref{sec:pseudolocS} (indeed, once we check that Assumptions \ref{ass:cont}, \ref{ass:Gar}, \ref{ass:domComp}, \ref{ass:contPk} and \ref{ass:ell} hold, all other objects appearing in this section are then defined in terms of $\cH_k^n$, $\cZ_k$ and $a_k$).  Observe that $\cZ_k$ is a subspace of $H^1_k$ with Dirichlet conditions on either $\partial \Omega$ (Dirichlet setting) or just $\Gamma_{\tr}$ (Neumann setting), and for all $j \geq 0$, $\ZcupHspace{j} = H_k^j$ and the inclusion $H_k^{-j} \subset \ZcupHspace{-j}$ is continuous. In particular, if $R = O_{-\infty}(\eta^n;\ZcupHspace{}\to\ZcupHspace{})$, then 
$$\|Ru\|_{H_k^N} \leq C \eta^{n}\|u\|_{H_k^{-N}} \quad \tfa N \in \mathbb{N}.$$

In this particular setting, we give sufficient conditions for smooth cut-off functions $\chi \in C^\infty(\overline{\Omega})$ to fulfil the conditions of Definition \ref{def:spatialCutoffs} or Definition \ref{def:freqOp}. We also identify some boundary-compatible operators in the sense of Definition \ref{def:freqOp} that are used for the proof of Theorem \ref{t:theRealDeal}.

In the remainder of this paper, we adopt the following notation.
\begin{definition}\label{def:cheeky1}
 Given two cutoff functions $\varphi,\widetilde{\varphi},\psi \in C^\infty(\overline{\Omega})$ and a real number $d > 0$, we write
$$\begin{array}{rcl}
\varphi \perp_d \psi &\iff& \dist(\supp \varphi,\supp \psi) > d\\[0.5em]
\varphi \prec_d \widetilde{\varphi} &\iff& \varphi \perp_d (1 - \widetilde{\varphi}).
\end{array}$$
We abbreviate $\perp_0$ and $\prec_0$ by $\perp$ and $\prec$.
\end{definition}

\subsection{Verifying the assumptions}\label{sec:verify}

We start by showing that the assumptions of Section \ref{sec:pseudolocS} hold for the PML problem. 

With the PDE coefficients in $A_\theta, b_\theta,$ and $n_\theta$ defined in \S\ref{sec:PML}, 
that section shows that there exists $\omega\in \mathbb{R}$ such that \blue{$a_k$} satisfies 
Assumption \ref{ass:cont} and \ref{ass:Gar}

% (with $\omega=0$ for the most commonly-used radial PML construction). 
% Since 
% \beqs
% a_k(u-u_h, v_h) = 0 \quad \text{ if and only if }\quad \re^{\ri \omega} a_k(u-u_h, v_h) = 0,
% \eeqs

%without loss of generality we can assume that $\omega=0$. 
%\blue{Note that, by reducing to the case $\omega=0$, we lose the fact that $A_\theta$ and $n_\theta$ are real away from the PML. Instead, there is $c\neq 0$ such that $\Re A_\theta= cA_\theta$, and $\Re (n_\theta)=cn_\theta$.}

%this multiplication does not change the Galerkin solution \eqref{e:Galerkin_ortho}.
%are satisfied by \S\ref{sec:PML} (see Remark \ref{rem:emptySet}). 
 It is standard that Assumption~\ref{ass:domComp} holds with $\Zspace{2}$ defined in \eqref{e:defZ2concrete}. Moreover, for all $u \in\Zspace{2}$,  
\[P_k u = -k^{-2} \div (\blue{\re^{\ri \omega}}A_\theta \nabla u)+k^{-2}\langle \blue{\re^{\ri \omega}} b_\theta(x),\nabla u\rangle- \blue{\re^{\ri \omega}}n_\theta u.\]
Thus, $P_k: \Zspace{n} \to \Hspace{n-2}$ is continuous for $n \geq 2$, thus Assumption~\ref{ass:contPk} holds. The smoothness of $\partial\Omega$, $A_\theta$, $b_\theta$, and $n_\theta$ ensure Assumption \ref{ass:ell} (elliptic regularity) holds

Since $\Dspace{0}=\cH$, $\Dspace{1}= \Zspace{}$, and $\cX_k^{-2}:=(\cP_k + (\CGa(k_0) + 1)\Id)^{-1}$ is an isomorphism from $\Dspace{n}$ to $\Dspace{n+2}$, it follows by induction that 
\beqs
%\Dspaced{2n}= \big\{ u \in H^{2n}(\Omega) : u, \cP_k u, \ldots, \cP_k^{n-1} u \in H^1_0(\Omega)\big\}
\Dspaced{n}= \big\{ u \in H^{n}(\Omega) : u, \cP_k u, \ldots, \cP_k^{\lceil n/2\rceil-1} u \in H^1_0(\Omega)\big\}
\eeqs
and
\begin{align*}
\Dspacen{n}= \big\{ u \in H^{n}(\Omega) \,: \,&u= \cP_k u= \ldots= \cP_k^{\lceil n/2\rceil-1} u=0  \ton \Gamma_{\tr} 
\\
&\tand
\partial_{\nu, A_\theta} u=\partial_{\nu, A_\theta}(\cP_k u) = \ldots =\partial_{\nu, A_\theta}(\cP_k^{\lceil n/2\rceil-1}u) = 0  \ton \partial\Omega_-
\big\}.
\end{align*}
%where we have used that $\| u\|_{\Dspace{2m+1}}^2 = \| \cX_k^{2s} u \|_{\Zspace{}}$ for all $u\in \Dspace{2m+1}$ by Proposition \ref{prop:Z2=D2}.

Now 
\beqs
\N{u}_{\Dspace{n}}= 
\begin{cases}
\| \cX_k^{2m} u\|_{\Hspace{}}, &n=2m,\\
\| \cX_k^{2m} u\|_{\Zspace{}}, & n=2m+1,
\end{cases}
\eeqs
Since $\cX_k^2 =\cP_k + (\CGa(k_0) + 1)\Id$ so that, by induction, for $m\in\mathbb{N}$, $\cX_k^{2m}$ coincides on $\Dspace{2m}$ (and thus on $\Dspace{2m+1}$) with a differential operator of order $2m$. 
Thus
%By \eqref{e:zoomWorking2},
%\beqs
%\| u\|_{\Dspace{n}}^2 = \big( \cX_k^{2n} u, u\big)_{\cH} \quad \tfa u \in \Dspace{n},
%\eeqs
%so that 
\beq\label{e:warmOffice1}
\| u\|_{\Dspace{n}} \leq C \| u\|_{\Hspace{n}} \quad\tfor u \in \Dspace{n}.
\eeq

%Assumption \ref{ass:domComp} domain symmetry
%
%Assumption \ref{ass:contPk} mapping
%Verification of assumptions on $P$ and related cutoffs}

\subsection{Pseudolocality results with smooth cut-off functions}
%Conditions under which elements of $C^\infty(\overline{\Omega})$ are abstract spatial cut-offs (i.e., satisfy Definition \ref{ass:spatialCutoffs})}
\label{sec:spatialCheck}
% Assumptions ~\ref{ass:domain} to ~\ref{ass:ell}}

The main result of this section is the following.

\begin{theorem}
\label{t:pseudoLocGeneral}
Suppose that $\chi_1,\chi_2\in C^\infty(\overline{\Omega})$ with $\supp \chi_1\cap \supp \chi_2=\emptyset$. Then, for all $f\in \mc{S}(\R)$,
\begin{gather*}
\chi_1f(\cP_k)\chi_2=O_{-\infty}(k^{-\infty};\ZcupHspace{}\to \ZcupHspace{}),\qquad
\chi_1 R_{k}^\sharp\chi_2=O_{-2}(k^{-\infty};\ZcupHspace{}\to \ZcupHspace{}).
\end{gather*}
\end{theorem}

This result is proved by using Theorem~\ref{thm:pseudolocSpace} and showing, first, that if $\chi \in C^\infty(\overline{\Omega})$ and $\partial_\nu \chi|_{\partial\Omega_-}=0$ then $\chi \in \Lcut$ (see Lemma \ref{l:spatialCut} below) and, second, that 
given $\chi_1,\chi_2\in C^\infty(\overline{\Omega})$, there exist 
$\widetilde\chi_j\in C^\infty(\overline{\Omega})$ with $\partial_\nu \widetilde\chi_j|_{\partial\Omega_-}=0$, $j=1,2,$ such that
$\chi_j\prec \widetilde\chi_j$, and $\widetilde\chi_1 \perp \widetilde\chi_2$ (see Lemma \ref{l:goodCut}).

\begin{lemma}
\label{l:spatialCut}
Given $\{C_N\}_{N > 0} \subset \R_+$, there exist $\{C'_{N,n}\}_{N,n} \subset \R_+$ such that the following is true. For all $\varepsilon > 0$, if $\chi \in C^\infty(\overline{\Omega})$ satisfies $\partial_\nu \chi|_{\partial\Omega_-}=0$ and
\begin{equation*}
%\label{e:goodCutDerivatives}
\begin{gathered}
\max_{|\alpha|\leq N} \varepsilon^{|\alpha|}|\partial^\alpha \chi|\leq C_N\quad \tfa N \in \mathbb{N},
\end{gathered}
\end{equation*}
then
\begin{equation}
	\label{e:adNRQ_goodCut}
	\begin{gathered}
		\|\ad_{\chi}^N Q\|_{\ZcapHspace{n} \to \ZcupHspace{n+2-N}} \leq C'_{N,n} (\varepsilon k)^{-N},
	\end{gathered}
\end{equation}
where $Q$ is any one of the operators $P_k$, $P_k^*$ or $\cP_k$. In particular $\chi \in \Lcut.$
\end{lemma}
\begin{proof}
We start by considering $\ad_{\chi}P_k$ acting in $\Zspace{1}$. Indeed, suppose that $u,v\in \Zspace{1}$. Then, $\chi u,\bar{\chi} v\in \Zspace{1}$ and hence
\begin{align}\nonumber
\langle \ad_\chi P_k u,v\rangle &=a_k(u,\bar{\chi} v)-a_k(\chi u,v)\\ \nonumber
&= k^{-2}\Big(\langle \blue{e^{\ri \omega}}A_\theta \nabla u,\nabla (\bar{\chi} v)\rangle -\langle \blue{e^{\ri \omega}}A_\theta \nabla (\chi u),\nabla v\rangle\Big)
+ k^{-2}\big\langle \chi \blue{e^{\ri \omega}}b_\theta\cdot \nabla u  - \blue{e^{\ri \omega}}b_\theta \cdot\nabla(\chi u), v \big\rangle
\\ \nonumber
&= k^{-2}\Big(\langle \blue{e^{\ri \omega}}A_\theta \nabla u,v\nabla \bar{\chi} \rangle -\langle u \blue{e^{\ri \omega}}A_\theta \nabla \chi ,\nabla v\rangle\Big)
-k^{-2} \langle (\blue{e^{\ri \omega}}b_\theta \cdot \nabla \chi)u, v \rangle
\\ \label{e:helpTheFuture1}
&= k^{-2}\big\langle (\blue{e^{\ri \omega}}A_\theta \nabla u) \cdot \nabla \chi+\nabla\cdot(u\blue{e^{\ri \omega}}A_\theta \nabla \chi),v \big\rangle -k^{-2} \langle (\blue{e^{\ri \omega}}b_\theta \cdot \nabla \chi)u, v \rangle,
\end{align}
where in the last line we use that  $\partial_{\nu}\chi|_{\partial\Omega_-}=0$ and $u|_{\Gamma_{\tr}}=0$. In particular, for $n\geq 1$, $\|\ad_\chi P_k u\|_{\Hspace{n-1}}\leq C(\varepsilon k)^{-1}\|u\|_{\Zspace{n}}$.
Next, since
\beqs
\langle \ad_\chi^2P_k u,v\rangle =\langle \ad_\chi P_k u,\overline{\chi} v\rangle -\langle \ad_\chi P_k (\chi u),v\rangle,
\eeqs
a short calculation using \eqref{e:helpTheFuture1} implies that, 
for $u,v\in \Zspace{1}$, 
\beq\label{e:helpTheFuture2}
\langle \ad_\chi^2P_k u,v\rangle = -k^{-2}\big\langle 2u (\blue{e^{\ri \omega}}A_\theta \nabla \chi)\cdot \nabla \chi,v\big\rangle.
\eeq
Thus, for $n\geq 0$, $\|\ad_\chi^2P_k u\|_{\Hspace{n}}\leq C(\varepsilon k)^{-2}\|u\|_{\Zspace{n}}$.  
Since there are no derivatives of $u$ on the right-hand side of \eqref{e:helpTheFuture2}, a similar calculation shows that 
$u\in \Zspace{1}$, $(\ad_\chi^NP_k)u=0$ for $N\geq 3$. Thus 
\begin{equation}
\label{e:highn}
\|\ad_\chi^NP_k\|_{\Zspace{n}\to \Hspace{n-2+N}}
=\|\ad_\chi^NP_k\|_{\ZcapHspace{n}\to \ZcupHspace{n-2+N}}
\leq C(\varepsilon k)^{-N} \quad\tfor n\geq 0,
\end{equation}
and, by identical arguments,
\begin{equation}
\label{e:highn2}
\|\ad_\chi^NP_k^*\|_{\Zspace{n}\to \Hspace{n-2+N}}
=\|\ad_\chi^NP_k^*\|_{\ZcapHspace{n}\to \ZcupHspace{n-2+N}}
\leq C(\varepsilon k)^{-N}\quad\tfor n\geq 0.
\end{equation}
%The bounds for $n\in \mathbb{Z}$ then follow from
%\eqref{e:dualNewSpaces} and the fact that $(\ad_A^N B)^*= (-1)^N \ad_{A^*}^N B^*$ (which one can prove by induction).
Now, for $\ell \geq 0$, $u\in \Hspace{-\ell}$ and $v\in \Zspace{\ell+2}$, by 
the fact that $(\ad_A^N B)^*= (-1)^N \ad_{A^*}^N B^*$ (which one can prove by induction),
\begin{align*}
\big|\big\langle (\ad_\chi^N P_k)u,v\big\rangle\big|=\big|\big\langle u,(-1)^N(\ad_{\bar{\chi}}^NP_k^*)v\big\rangle \big|
&\leq \|u\|_{\Hspace{-\ell}}\| (\ad_{\bar{\chi}}^NP_k^*)v\|_{\Hspace{\ell}}
%\\
%&\leq Ck^{-N}\|u\|_{\Hspace{-\ell}}\|v\|_{\Zspace{\ell+2-N}}.
\end{align*}
If $N\geq 3$, then the right-hand side of the last inequality is zero. Otherwise, \eqref{e:highn2} with $n=\ell+2-N \geq 0$ implies that 
\beqs
\big|\big\langle (\ad_\chi^N P_k)u,v\big\rangle\big|\leq C(\varepsilon k)^{-N}\|u\|_{\Hspace{-\ell}}\|v\|_{\Zspace{\ell+2-N}}.
\eeqs
Since $\Zspace{\ell+2}$ is dense in $\Zspace{\ell+2-N}$, the previous inequality and 
%Hence, using that $\overline{\Zspace{\ell+2}}^{\Hspace{\ell+2-N}}\supset\Zspace{\ell+2-N}$, and 
similar arguments for $P_k^*$ imply that
\begin{equation}
\label{e:lowl}
\|\ad_{\chi}^NP^*_k\|_{\Hspace{-\ell}\to \Zspace{-\ell-2+N}}+\|\ad_{\chi}^NP_k\|_{\Hspace{-\ell}\to \Zspace{-\ell-2+N}}\leq C(\varepsilon k)^{-N},\qquad \ell\geq 0.
\end{equation}
The combination of \eqref{e:highn}, \eqref{e:highn2}, and~\eqref{e:lowl} are then \eqref{e:adNRQ_goodCut}.
\end{proof}

\

We first prove the existence of a cut-off in $\Lcut$ between $\Omega_1$ and $\Omega_2$, under the assumption that $\Omega_1$ is sufficiently small -- this assumption allows us to use Fermi normal coordinates defined by $\partial\Omega$ on $\Omega_1$.

\begin{lemma}
\label{l:big}
%There is $\e_0>0$ such that for all $N>0$, there is $C_N>0$ such that for all $0<\e<\e_0$ and $\Omega_1,\Omega_2\subset \Omega$ such that $d(\Omega_1,\Omega_2)>\e$, and there is $y\in \Omega$ such that $\Omega_1\subset B(y,\e_0)$ there is $\chi \in C^\infty(\overline{\Omega};[0,1])$ satisfying 
%and
There exists $\e_0>0$ such that for all $N>0$, there exists $C_N>0$ such that for all $0<\e<\e_0$ the following is true. If $\Omega_1,\Omega_2\subset \Omega$ are such that $d(\Omega_1,\Omega_2)>\e$ and if there exists $y\in \Omega$ such that $\Omega_1\subset B(y,\e_0)$, then there exists $\chi \in C^\infty(\overline{\Omega};[0,1])$ satisfying 
\begin{equation*}
%\label{e:cutEstimates}
\begin{gathered}
\Omega_1\cap \supp (1-\chi)=\emptyset,\qquad  \supp \chi\cap \Omega_2=\emptyset, \\
|\partial^\alpha \chi|\leq C_N\e^{-|\alpha|}\tfor |\alpha|\leq N,\qquad\tand\quad \partial_\nu \chi|_{\partial\Omega}=0.
\end{gathered}
\end{equation*}
\end{lemma}
\begin{proof}
Let $U_{\rm Fermi}$ be a tubular neighbourhood of $\partial\Omega$ in which there exists a Fermi normal coordinate chart; we denote these coordinates below by $(x_1,x')_F$ (where the subscript $F$ emphasises that these are not Euclidean coordinates).
Let $\delta_0>0$ be such that $U_{\rm Fermi}\supset B(\partial \Omega, 20\delta_0):= \{ x\in \Omega: \dist(x, \partial\Omega)< 20 \delta_0\}$
and let $\e_0=9\delta_0$. 
If $\Omega_1\not\subset U_{\rm Fermi}$ then $\dist(\Omega_1,\partial \Omega) \geq 2\delta_0 = (2/9)\epsilon_0$, and the existence of $\chi$ follows immediately (e.g., by mollification of the indicator function of $\Omega_1$).  
We therefore assume that $\Omega_1\subset U_{\rm Fermi}$.
Now, there exists $c_F>0$ (depending only on $\Omega$) such that, for all $r>0$ and for all $(0,x')\in \partial\Omega$, 
\beq\label{e:cF}
\big\{ (0,y')_F : |y'-x'|\leq c_F r\big\} \subset B\big( (0,x')_F ,r \big) \cap\partial \Omega.
\eeq

We now define some mollifiers and cut-off functions. 
Fix $\psi_m\in C_c^\infty(B_{\mathbb{R}^m}(0,1))$ such that  $\int \psi_m=1$, $m=d-1,d$. Fix also $\widetilde{\psi}_1\in C_c^\infty(-2,2)$ with $(-1,1)\cap \supp(1-\widetilde{\psi}_1)=\emptyset$. Then, for $\delta>0$, set $\psi_{m,\delta}(x):=\delta^{-m}\psi_m(\delta^{-1}x)$, and $\widetilde{\psi}_{1,\delta}(x):=\widetilde{\psi}_1(\delta^{-1}x)$. 

%We consider two cases.
%First, let $\delta:=\e/20$, and suppose that 
%$$
%B(\Omega_1,10\delta)\cap \partial\Omega)=\emptyset.
%$$
%Then, let 
%$$
%\chi:=\psi_{d,\delta}*1_{B(\Omega_1,4\delta)}.
%$$
%Then, for $x\in B(\Omega_1,\delta)$,
%$$
%\chi(x)=\int \psi_{d,\delta}(y)1_{B(\Omega_1,4\delta)}(x-y)dy=\int_{|y|<\delta} \psi_{d,\delta}(y)1_{B(\Omega_1,4\delta)}(x-y)dy=1,
%$$
%and for $x\in B(\Omega_2,\delta)$, $d(x,\Omega_1)>19\delta$ and hence $x\notin B(\Omega_1,5\delta)$. Similarly, if $x\in B(\partial\Omega,\delta)$, then $d(x,\Omega_1)>9\delta$. Therefore, for $x\in B(\Omega_2,\delta)\cup B(\partial\Omega,\delta)$, 
%$$
%\chi(x)=\int \psi_{d,\e}(y)1_{B(\Omega_1,4\e)}(x-y)dy=\int_{|y|<\e} \psi_{d,e}(y)1_{B(\Omega_1,4\e)}(x-y)dy=0.
%$$
%Moreover, 
%$$
%|\partial^\alpha \chi(x)|= |(\partial^\alpha\psi_{d,\e})*1_{B(\Omega_1,4\e)}|\leq \|\partial^\alpha\psi_{d,\e}\|_{L^1}\leq C_\alpha \e^{-|\alpha|}.
%$$
%

Let $\delta=\e/10$ and $\widetilde{\Omega}_1:=B(\Omega_1,4\delta)$; note that $d(\widetilde{\Omega}_1,\Omega_2)>6\e /10$. 
Let $\widetilde{\Omega}_1^\partial:=\widetilde{\Omega}_1\cap \partial\Omega$; note that this set may be empty.
%Let $(x_1,x')$ be (Fermi normal) coordinates on $B(\Omega_i,10\delta_0)\subset B(\Omega_i,9\delta)$.

Now let 
\beq\label{e:FermiChi}
\chi(x_1,x'):= \big(1_{\widetilde{\Omega}^\partial_1}*\psi_{d-1,c_F \delta}\big)(x')\widetilde{\psi}_{1,\delta}(x_1)+\big(1_{\widetilde{\Omega}_1}*\psi_{d,\delta}\big)(x)\big(1-\widetilde{\psi}_{1,\delta}(x_1)\big)=:\chi_{\rm near}+\chi_{\rm far},
\eeq
%where $c>0$ is fixed later in the proof, depending only on $\Omega$. 

We now check that $\chi$ has the required properties. First, 
$$
\|\partial^\alpha (u*v)\|_{L^\infty}=\|(\partial^\alpha u)*v\|_{L^\infty}\leq \|\partial^\alpha u\|_{L^1}\|v\|_{L^\infty},
$$
so that
$$
\|\big(\partial^\alpha (u*v)\big)w\|_{L^\infty}
\leq \N{w}_{L^\infty} \|(\partial^\alpha u)*v\|_{L^\infty}\leq 
 \N{w}_{L^\infty}
\|\partial^\alpha u\|_{L^1}\|v\|_{L^\infty}.
%\|\big(\partial^\alpha (u*v)\big)w\|_{L^\infty}
%\leq \N{w}_{L^\infty} \|(\partial^\alpha u)*v\|_{L^\infty}\leq 
% \N{w}_{L^\infty}
%\|\partial^\alpha u\|_{L^1}\|v\|_{L^\infty}.
$$
Combining this with the product rule and 
$$
\|\partial^\alpha \widetilde{\psi}_{1,\delta}\|_{L^\infty}\leq C_{\alpha}\delta^{-\alpha_1} \quad\tand\quad 
\|\partial^\alpha \psi_{m,\delta}\|_{L^1}\leq C_{\alpha}\delta^{-|\alpha'|}, \,\, m= d-1,d,
$$
implies the required derivative estimates on $\chi$. Next, since $\chi_{\rm far}\equiv 0$ near $\partial\Omega$ and $\partial_{x_1}^\alpha \chi_{\rm near}|_{x_1=0}=0$ for any $\alpha$, it follows that $\partial_{\nu}\chi|_{\partial\Omega}=0$. 

We now show that $\Omega_1\cap \supp (1-\chi)=\emptyset$; we do this by showing that 
\beq\label{e:convNightmare1}
\chi(x)=1 \quad \text{ when } \quad x=(x_1,x')\in B(\Omega_1,\delta)\cap \overline{\Omega}.
\eeq
First observe that, for such $x$, $\big(1_{\widetilde{\Omega}_1}*\psi_{d,\delta}\big)(x)=1$.
Then, since 
%\beqs
$\widetilde{\psi}_{1,\delta}(x_1)=0$ when $ x_1 \geq 2\delta$,
%\begin{cases}
%1, &  x_1\leq \delta,\\
%0, & 
%\end{cases}
%\eeqs
\eqref{e:convNightmare1} then follows from \eqref{e:FermiChi} if we can show that 
\beq\label{e:convNightmare2}
\big(1_{\widetilde{\Omega}^\partial_1}*\psi_{d-1,c_F\delta}\big)(x')=1 \quad\text{ when } x_1\leq 2\delta
\eeq
(i.e., on the support of $\widetilde{\psi}_{1,\delta}$). To prove \eqref{e:convNightmare2}, observe that,
%we claim that this follows 
%
%
%Then, there are three cases: 
%\noindent Case (i) $0\leq x_1<\delta$. Here we show that $\chi_{\rm near}=1$ and $\chi_{\rm far}=0$. 
by the triangle inequality, 
$$
d((0,x')_F,\Omega_1)\leq x_1+d(x,\Omega_1)<2\delta+d(x,\Omega_1)<3\delta.
$$
Since $\widetilde{\Omega}_1:=B(\Omega_1,4\delta)$, $B((0,x')_F,\delta)\subset \widetilde{\Omega}_1$. 
Thus $B((0,x')_F,\delta)\cap \partial\Omega\subset \widetilde{\Omega}_1\cap \partial\Omega= \widetilde{\Omega}_1^\partial$.
and  \eqref{e:convNightmare2} follows by \eqref{e:cF}.

%\noindent Case (ii) $\delta\leq x_1 <2\delta$. Here we show that both the convolutions in \eqref{e:FermiChi} equal one.
%Arguing as in Case (i), $B((0,x')_F,\delta)\subset \widetilde{\Omega}_1$, and hence 
%$$
%(1_{\widetilde{\Omega}_i^\partial}*\widetilde{\psi}_{d-1,\delta}(x'))=1.
%$$
%On the other hand, since $d(x,\Omega_1)<\delta$, 
%$
%(1_{\widetilde{\Omega}_1}*\psi_{d,\delta})(x)=1,
%$
%so that $\chi(x)=1$. 
%
%
%\noindent Case (iii) $2\delta\leq x_1 $. 
%
%For the last case, we simply observe that $\widetilde{\psi}_{1,\delta}(x_1)=0$ and since $d(x,\Omega_1)<\delta$, $(1_{\widetilde{\Omega}_1}*\psi_{d,\delta})(x)=1$.
%

Finally, we show that $\Omega_2\cap \supp \chi=\emptyset$, again by showing that 
\beqs%\label{e:convNightmare3}
\chi(x)=0 \quad \text{ when } \quad x=(x_1,x')_F\in B(\Omega_2,\delta)\cap \overline{\Omega}.
\eeqs
%Next, suppose $x=(x_1,x')\in B(\Omega_2,\delta)$. 
Then, $d(x,\widetilde{\Omega}_1)>\e-5\delta=5\delta$; thus $(1_{\widetilde{\Omega}_1}*\psi_{d,\delta})(x)=0$ and $\chi_{\rm far}=0$. 
If $|x_1|> 2\delta$, $\chi_{\rm near}=0$. Otherwise, we claim that
%Since $\chi_{\rm near}$ is supported in $|x_1|\leq 2\delta$, we can assume that 
%$|x_1|<2\delta$. Then, 
%since 
\beq\label{e:cheeky1}
d((0,x')_F,\widetilde{\Omega}_1^\partial)\geq \epsilon-7\delta= 3 \delta;
\eeq
then, by \eqref{e:cF}, 
\beqs
\big\{ (0,y')_F : |y'-x'|\leq c_F \delta\big\}\cap \widetilde{\Omega}_1^\partial\,\subset\, B((0,x')_F, \delta)\cap \widetilde{\Omega}_1^\partial\,=\,\emptyset
\eeqs
so that  $(1_{\widetilde{\Omega}^\partial_1}*\psi_{d-1,c_F\delta})(x)=0$ (and hence $\chi_{\rm near}=0$). The inequality \eqref{e:cheeky1} follows by the triangle inequality: 
 \begin{align*}
\e\leq d(\Omega_1,\Omega_2)&\leq d(\Omega_1, \widetilde{\Omega}_1^\partial) 
+ d((0,x')_F,\widetilde{\Omega}_1^\partial) + d((0,x')_F,x) + d(x,\Omega_2),\\
&\leq 4 \delta + d((0,x')_F,\widetilde{\Omega}_1^\partial)+2 \delta + \delta,
% &\geq \epsilon - 4\delta - 2\delta - \delta = \epsilon-7\delta.
\end{align*} 
% Thus % d((0,x')_F,\widetilde{\Omega}_1^\partial)>\e-(M_1+3)\delta,
%% 
%% \end{align*}
%$(1_{\widetilde{\Omega}^\partial_1}*\psi_{d-1,\delta})(x)=0$ and $\chi_{\rm near}=0$.
concluding the proof.
\end{proof}

\begin{lemma}[Partition of unity satisfying Neumann boundary conditions]
\label{l:goodPartition}
Let $\Omega_i\subset \Omega$, $i=1,\dots, N$ be open with $\Omega\subset \cup_i\Omega_i$. Then there exist $\varphi_i\in C^\infty(\overline{\Omega})$ satisfying
$$
\supp \varphi_i\subset \Omega_i\cup \partial\Omega,\quad \partial_\nu \varphi_i=0 \,\,\ton\partial\Omega,\quad \tand\quad\sum_{i=1}^N \varphi_i\equiv 1.
$$
\end{lemma}
\begin{proof}
Let $U_i\Subset \Omega_i$ be open sets such that $\Omega\subset \cup_i U_i$. Then, for $\e>0$ small enough, and $y\in \overline{U}_i$, $B(y,2\e)\subset \Omega_i$. In addition, since $\overline{U}_i$ is compact, there are $\{y_{ij}\}_{j=1}^{N_i}\subset \overline{U}_i$ such that $ U_i\subset \cup_{j=1}^{N_i} B(y_{ij},\e)$. 

 By Lemma~\ref{l:big} (applied with $\Omega_1= B(y_{ij},\e)$ and $\Omega_2= \Omega\setminus \overline{\Omega_i}$), for $\e>0$ small enough, there are $\widetilde{\varphi}_{ij}\in C^\infty(\overline{\Omega};[0,1])$ such that $\supp(1-\widetilde{\varphi}_{ij})\cap B(y_{ij},\e)=\emptyset$, $\supp \widetilde{\varphi}_{ij}\subset \Omega_i \cup \partial\Omega$,
 and $\partial_\nu \widetilde\varphi_{ij}=0$ on $\partial\Omega$.   
Noting that $\sum_{i=1}^N\sum_{j=1}^{N_i}\widetilde{\varphi}_{ij}\geq 1$, we define
$$
\varphi_i:=\frac{\sum_{j=1}^{N_i}\widetilde{\varphi}_{ij}}{\sum_{i=1}^N\sum_{j=1}^{N_i}\widetilde{\varphi}_{ij}},
$$
%It is then easy to see that $\varphi_i$ 
which has the required properties.
\end{proof}

%We first prove the existence of a cut-off in $\Lcut$ between $\Omega_1$ and $\Omega_2$, under the assumption that $\Omega_1$ is sufficiently small -- this assumption allows us to use Fermi normal coordinates near $\partial\Omega$.

\

We now remove the assumption from Lemma \ref{l:big} that $\Omega_1$ is sufficiently small.

\begin{lemma}
\label{l:goodCut}
There exists $\e_0>0$ such that for all $N>0$, there exists $C_N>0$ such that for all $0<\e<\e_0$ and $\Omega_1,\Omega_2\subset \Omega$ and $d(\Omega_1,\Omega_2)>\e$, there exists $\chi \in C^\infty(\overline{\Omega})$ satisfying 
\begin{equation}
\label{e:cutEstimates2}
\begin{gathered}
\Omega_1\cap \supp (1-\chi)=\emptyset,\qquad  \supp \chi\cap \Omega_2=\emptyset, \\
|\partial^\alpha \chi|\leq C_N\e^{-|\alpha|},\qquad |\alpha|\leq N,\qquad \partial_\nu \chi|_{\partial\Omega}=0, 
\end{gathered}
\end{equation}
\end{lemma}
\begin{proof}
Let $\e_0$ be as in Lemma~\ref{l:big}. Since $\overline{\Omega}$ is compact, there exist $\{x_i\}_{i=1}^M$ such that $\overline{\Omega}\subset \cup_{i=1}^M B(x_i,\e_0)$. Then, by Lemma~\ref{l:goodPartition}, there exist $\{\varphi_i\}_{i=1}^M$ a partition of unity subordinate to $\{B(x_i,\e_0)\}_{i=1}^M$ satisfying $\partial_\nu \varphi_i=0$ on $\partial\Omega$. Then, let $\Omega_{1,i}:= \Omega_1\cap B(x_i,\epsilon_0)$. 
By Lemma~\ref{l:big}, there exists $\chi_i$ such that the conditions in \eqref{e:cutEstimates2} hold with $\Omega_1$ replaced by $\Omega_{1,i}$. Define $\chi:=\sum_{i=1}^M \chi_i\varphi_i$. The derivative estimates in~\eqref{e:cutEstimates2} then follow from the product rule and the fact that the derivatives of $\varphi_i$ are independent of $\epsilon$ (but depend on $\epsilon_0$).
The condition that $\partial_\nu \chi|_{\partial\Omega}=0$ follows since both $\partial_\nu \chi_i=0$ and $\partial_{\nu}\varphi_i=0$.
%  To check the rest of the conditions observe that, since $\partial_\nu \chi_i=0$, 
%$$
%\partial_\nu \chi=\sum_i \chi_i \partial_{\nu}\varphi_i=0.
%$$ 
Next, since $\supp\chi_i \cap \Omega_2=\emptyset$, $\supp \chi\cap \Omega_2 \subset \cup_i\supp (\chi_i\varphi_i)\cap \Omega_2=\emptyset$.
Finally, since $\chi_i\equiv 1$ on $\Omega_1\cap B(x_i,\epsilon_0)$ and $\supp \varphi_i \subset B(x_i,\epsilon_0)$, $(1-\chi_i)\varphi_i=0$ on $\Omega_1$, and thus 
$$
\supp (1-\chi)\cap \Omega_1 =\supp \Big(\sum_{i=1}^N (1-\chi_i)\varphi_i\Big)\cap \Omega_1=
%\cup_i \supp (1-\chi_i)\cap B(x_i,\epsilon_0)\cap \Omega_1=
\emptyset.
$$
\end{proof}

\begin{proof}[Proof of Theorem \ref{t:pseudoLocGeneral}]
Since $\supp \chi_1\cap \supp \chi_2=\emptyset$, there exist $\Omega_i$ neighbourhoods of $\supp \chi_i$ with $d(\Omega_1,\Omega_2)>0$. Therefore, by Lemma~\ref{l:goodCut} %(applied with $\Omega_j=U_j$) 
there are $\widetilde{\chi}_i$ with $\supp \chi_i\cap \supp (1-\widetilde{\chi}_i)=\emptyset$, $\supp \chi_i\cap \supp\widetilde{\chi}_j=\emptyset$, $i\neq j$, and $\partial_\nu \widetilde{\chi}_i=0$. Hence by Lemma~\ref{l:spatialCut}, $\widetilde{\chi}_i\in \Lcut$ and using Theorem~\ref{thm:pseudolocSpace}, we have
$$
\chi_1f(\cP_k)\chi_2=\chi_1\widetilde{\chi}_1f(\cP_k)\widetilde{\chi}_2\chi_2=\chi_1O_{-\infty}(k^{-\infty};\ZcupHspace{}\to \ZcapHspace{})\chi_2=O_{-\infty}(k^{-\infty}; \ZcupHspace{}\to \ZcupHspace{}),
$$
since $\chi_j = O_0(1;\ZcupHspace{}\to \ZcupHspace{})$. 
The proof for $R_{k}^\sharp$ is identical.
\end{proof}

\subsection{Some boundary compatible operators}\label{sec:verifyboundarycompatible}

\begin{lemma}\label{l:doneWithAds}
If $\varphi\in C^\infty(\overline{\Omega})$, $\supp \nabla\varphi \cap \partial\Omega=\emptyset$, and $\supp \varphi\cap \Gamma_{\tr}=\emptyset$, then 
$\varphi \in \Lf$ in the sense of Definition \ref{def:freqOp}.
\end{lemma}

\bpf
By Definition \ref{def:freqOp},
we need to show that 
$\ad_{\cP_k}^N\varphi =O_{N}(k^{-N};\Dspace{}\to\Dspace{})$.
Let $\widetilde{\varphi}\in C^\infty_c(\Omega)$ be such that $\widetilde{\varphi}\equiv 1$ on $\supp \nabla \varphi$.
We first show that $\ad_{\cP_k}^N \varphi = (\ad_{L_k}^N \varphi)\widetilde{\varphi}$ on $C^\infty(\overline{\Omega})$ where $L_k$ is a second-order differential operator. 
By \eqref{e:helpTheFuture1}, for $u,v\in \Zspace{}$,
\beq\label{e:helpTheFuture3}
\langle (\ad_{\cP_k} \varphi) u,v\rangle
% =a_k(u,\bar{\chi} v)-a_k(\chi u,v)
%&= k^{-2}\Big(\langle A_\theta \nabla u,\nabla (\bar{\chi} v)\rangle -\langle A_\theta (\nabla \chi u),\nabla v\rangle\Big)\\ \nonumber
%&= k^{-2}\Big(\langle A_\theta \nabla u,v\nabla \bar{\chi} \rangle -\langle u A_\theta \nabla \chi ,\nabla v\tangle\Big)\\ \label{e:helpTheFuture1}
= -k^{-2}\big\langle (\Re  \blue{e^{\ri \omega}}A_\theta \nabla u) \cdot \nabla \varphi+\nabla\cdot(u (\Re \blue{e^{\ri \omega}}A_\theta )\nabla \varphi),v \big\rangle .
\eeq
By \eqref{e:helpTheFuture3}, 
\beqs
\ad_{\cP_k} \varphi = \widetilde{\varphi}(\ad_{\cP_k} \varphi) \widetilde{\varphi}= \widetilde{\varphi}(\ad_{L_k} \varphi) \widetilde{\varphi}.
\eeqs
%Thus
%\beqs
%\ad_{\cP_k}^N \varphi = \cP_k 
%\eeqs
%
%
Thus
\beqs
\ad^2_{\cP_k} \varphi= \cP_k \widetilde{\varphi}(\ad_{L_k} \varphi)\widetilde{\varphi}-\widetilde{\varphi}(\ad_{L_k} \varphi)\widetilde{\varphi} \cP_k.
\eeqs
Since $\cP_k \widetilde{\varphi}= L_k  \widetilde{\varphi}$ and $ \widetilde{\varphi} \cP_k = \widetilde{\varphi}L_k$, 
\beqs
\ad^2_{\cP_k} \varphi= L_k \widetilde{\varphi}(\ad_{L_k} \varphi)\widetilde{\varphi}-\widetilde{\varphi}(\ad_{L_k} \varphi)\widetilde{\varphi} L_k=\ad^2_{L_k} \varphi=(\ad^2_{L_k} \varphi)\widetilde{\varphi};
\eeqs
the fact that  $\ad_{\cP_k}^N \varphi = (\ad_{L_k}^N \varphi)\widetilde{\varphi}$ can be proved similarly by induction.
%Since $\supp \nabla\varphi \cap \partial\Omega=\emptyset$, if $u \in \Hspace{n+1}$ then $(\ad_{\cP_k} \varphi))u\in \Dspace{n}$ and, b
Therefore, given $u \in \Dspace{n+N}$, $\ad_{\cP_k}^N \varphi u= (\ad_{L_k}^N \varphi)\widetilde{\varphi}u \in \Hspace{n}$ with compact support in $\Omega$, and thus, when $n\in \mathbb{N}$, $(\ad_{\cP_k}^N \varphi) u \in \Dspace{n}$. Thus, by \eqref{e:warmOffice1}, commutator results for differential operators, and Corollary \ref{cor:DnsubZn}, for $n\in \mathbb{N}$,
\beqs
\| (\ad_{\cP_k}^N \varphi)u\|_{\Dspace{n}}\leq \| (\ad_{\cP_k}^N \varphi)u\|_{\Hspace{n}}
=\| (\ad_{L_k}^N \varphi)u\|_{\Hspace{n}}
\leq C k^{-N} \| u\|_{\Hspace{n+N}} \leq C' k^{-N} \|u\|_{\Dspace{n+N}}.
\eeqs
By the spectral theorem, $(\Dspace{s})_s$ is an interpolation scale, and the result for general $n$ follows by duality and interpolation.
%\ref{rem:Pknotdiff}
%By the definition of $\varphi$,  $\varphi:\Dspace{n}\to \Dspace{n}$. Therefore, since $\cP_k: \Dspace{n} \to \Dspace{n-2}$, 
\epf

\begin{lemma}[$\varphi P_k\varphi \in \Lf^2$ for suitable $\varphi$]
\label{l:commuteP}
Suppose that $\varphi\in C^\infty(\overline{\Omega})$, $\supp \nabla\varphi \cap \partial\Omega=\emptyset$, and $\supp \varphi\cap \Gamma_{\tr}=\emptyset$. Then,
\beq\label{eq:HRevil1}
%\|\ad_{\varphi\cP_k\varphi }^NP_k\|_{\Dspace{n}\cup\ZcupHspace{n}\to \Dspace{n-N-2}\cap \ZcapHspace{n-N-2}}+ \|\ad^N_{\varphi\cP_k\varphi}P^*_k\|_{\Dspace{n}\cup\ZcupHspace{n}\to \Dspace{n-N-2}\cap \ZcapHspace{n-N-2}}\leq C_Nk^{-N}
\ad_{\varphi\cP_k\varphi }^NP_k= O_{N+2}(k^{-N}; \Dspace{}\to \Dspace{}), 
\quad 
\ad_{\varphi\cP^*_k\varphi }^NP_k= O_{N+2}(k^{-N}; \Dspace{}\to \Dspace{}), 
%\Dspace{n}\to \Dspace{n-N-2}}+ \|\ad^N_{\varphi\cP_k\varphi}P^*_k\|_{\Dspace{n}\to \Dspace{n-N-2}}\leq C_Nk^{-N}
\eeq
and
\beq\label{eq:HRevil2}
\ad_{\cP_k}^N\varphi P_k\varphi = O_{N+2}(k^{-N}; \Dspace{}\to \Dspace{}), 
\quad 
\ad_{\cP^*_k}^N\varphi P_k\varphi = O_{N+2}(k^{-N}; \Dspace{}\to \Dspace{})\,;
%\|\ad_{\cP_k}^N\varphi P_k\varphi\|_{\Dspace{n}\cup\ZcupHspace{n}\to \Dspace{n-N-2}\cap \ZcapHspace{n-N-2}}+ \|\ad^N_{\cP_k}\varphi P^*_k\varphi\|_{\Dspace{n}\cup\ZcupHspace{n}\to \Dspace{n-N-2}\cap \ZcapHspace{n-N-2}}\leq C_Nk^{-N}.
\eeq
in particular, $\varphi P_k\varphi \in \Lf^2$ in the sense of Definition \ref{def:freqOp}. 
\end{lemma}
%
%By Definition \ref{def:freqOp}, \eqref{eq:HRevil2} implies that  (with $\varphi$ as in the statement of the lemma) $$.

\

\begin{proof}[Proof of Lemma \ref{l:commuteP}]
We prove \eqref{eq:HRevil1}; the proof of \eqref{eq:HRevil2} is very similar. We write 
\beqs
\ad_{\varphi \cP_k\varphi}^N P_k= \psi_0 \ad_{\varphi \cP_k\varphi}^N P_k+(1-\psi_0) \ad_{\varphi \cP_k\varphi}^N P_k, 
\eeqs
where $\psi_0$ is supported in a region close to $\partial\Omega$ where $\varphi$ is constant.
More precisely, let $\psi_i\in C^\infty(\overline{\Omega})$, $i=-1,0,1$ with $\supp (1-\psi_i)\cap \partial\Omega_-=\emptyset$, $\supp\psi_i\cap \supp(c-\varphi)=\emptyset$, $\supp (1-\psi_i)\cap \supp \psi_{i-1}=\emptyset$, and $\psi_2 P_k=\psi_2 P_k^*$
 (note that such functions exist since $\supp \nabla \varphi\cap \partial\Omega_-=\emptyset$).

%We first consider $\psi_0\ad_{\varphi \cP_k\varphi}^N P_k.$ 
By locality of $\cP_k$ and $P_k$ and the fact that $\cP_k=\blue{(\cos \omega)}P_k$ on $\supp \psi_2\supset \supp\psi_0$,
$$
\psi_0 \ad_{\varphi \cP_k\varphi}^N P_k =\psi_0\ad_{c \blue{\cos \omega}P_kc}^N P_k=0.
%\psi_0\ad_{\varphi \cP_k\varphi}^NP_k=\psi_0c^{2N}\ad_{ \cP_k}^N P_k=0.
$$
Now, let $\tilde{\varphi}\in C^\infty(\overline{\Omega})$ with $\supp \tilde{\varphi}\cap \Gamma_{\tr}=\emptyset$ and $\supp \varphi\cap \supp (1-\tilde{\varphi})=\emptyset$. Then
$$
(1-\psi_0)\ad^N_{\varphi\cP_k\varphi}P_k=(1-\psi_0)\ad^N_{\varphi(1-\psi_{-1})\cP_k\varphi(1-\psi_{-1})}\big[(1-\psi_{-1})\tilde{\varphi}P_k(1-\psi_{-1})\tilde{\varphi}\big].
$$
Since $\supp \tilde{\varphi}\cap \supp (1-\psi_{-1})\cap \partial\Omega=\emptyset$, 
integration by parts (with all the boundary terms vanishing) shows that $\varphi(1-\psi_{-1})\cP_k\varphi(1-\psi_{-1})$ and $(1-\psi_{-1})\tilde{\varphi}P_k(1-\psi_{-1})\tilde{\varphi}$ coincide with differential operators on $C^\infty(\overline{\Omega})$. The result 
$$
\|\ad^N_{\varphi\cP_k\varphi}P_k\|_{\Dspace{n}\to \Dspace{n-N-2}}\leq Ck^{-N} 
$$
then follows by direct differentiation (using the product rule) and then density of $C^\infty(\overline{\Omega})$ in $\Dspace{n}$.
The proof of the analogous bound for $P_k^*$ is identical.
\end{proof}

%\begin{lemma}\label{lem:phiPphi}
%$\tilde{\varphi}\cP_k\tilde{\varphi}\in 
%
%\end{lemma}
%
%\bpf
%
%\epf
%

\begin{lemma}[$\varphi (P_k^\sharp)^{-1}\varphi,\varphi (P_k^{\sharp,*})^{-1}\varphi\in \Lf^{-2}$]
\label{lem:huge_whiteboard1}
Suppose that $\varphi\in C^\infty(\overline{\Omega})$ and $\supp \nabla\varphi\cap \partial\Omega=\emptyset$, and $\supp \varphi\cap \Gamma_{\tr}=\emptyset$. Then, 
$$
\ad_{\cP_k}^N\varphi(P_k^\sharp)^{-1}\varphi=O_{N-2}(k^{-N};\Dspace{}\to \Dspace{}),\qquad \ad_{\cP_k}^N\varphi(P_k^{\sharp,*})^{-1}\varphi=O_{N-2}(k^{-N};\Dspace{}\to \Dspace{})
$$
and thus $\varphi (P_k^\sharp)^{-1}\varphi,\varphi (P_k^{\sharp,*})^{-1}\varphi\in \Lf^{-2}$ in the sense of Definition \ref{def:freqOp}. 
\end{lemma}

%\begin{corollary}
%\label{c:sharpInv}
%Suppose that $\varphi\in C^\infty(\overline{\Omega})$ and $\supp \nabla\varphi\cap \partial\Omega=\emptyset$, and $\supp \varphi\cap \Gamma_{\tr}=\emptyset$. Then 
%\end{corollary}
%

\begin{proof}[Proof of Lemma \ref{lem:huge_whiteboard1}]
We prove the statement for $P_k^\sharp$. The proof for $(P_k^{\sharp})^*$ is identical.
Let $\tilde{\varphi}\in C^\infty(\overline{\Omega})$ with $\supp \nabla\tilde{\varphi}\cap \partial\Omega=\emptyset$ and $\supp (1-\tilde{\varphi})\cap \supp \varphi=\emptyset$. Then, by locality of $\cP_k$,
\beq\label{e:mrcroc2}
\ad_{\cP_k}^N\varphi(P_k^\sharp)^{-1}\varphi= \ad_{\tilde{\varphi}\cP_k\tilde{\varphi}}^N\varphi(P_k^\sharp)^{-1}\varphi,
\eeq
By Lemma \ref{l:doneWithAds}, 
$\ad_{\cP_k}^N\varphi =O_{N}(k^{-N};\Dspace{}\to\Dspace{})$,
so that, by repeated use of the identity $\ad_{AB} C = A(\ad_B C) + (\ad_A C)B$,
\beq\label{e:mrcroc1}
 \ad_{\tilde\varphi\cP_k\tilde\varphi}^N\varphi =O_{N}(k^{-N};\Dspace{}\to\Dspace{}).
%\qquad \ad_{\cP_k}\varphi= \tilde{\varphi}( \ad_{\cP_k}\varphi)\tilde{\varphi}
\eeq
Therefore, by the combination of \eqref{e:mrcroc1}, \eqref{e:mrcroc2}, and repeated use of the identity $\ad_{A} BC = (\ad_A B)C + B(\ad_A C)$, 
it is enough to show that 
\beq\label{e:bitLost1}
\ad_{\tilde{\varphi}\cP_k\tilde{\varphi}}^N(P_k^\sharp)^{-1}=O_{N-2}(k^{-N};\Dspace{}\to\Dspace{}).
\eeq
%$$
%\varphi (\ad_{\cP_k}^N(P_k^\sharp)^{-1})\varphi=O_{N-2}(k^{-N}).
%$$
%For this, observe that 
%$$
%\varphi (\ad_{\cP_k}^N(P_k^\sharp)^{-1})\varphi= \varphi (\ad_{\tilde{\varphi}\cP_k\tilde{\varphi}}^N(P_k^\sharp)^{-1})\varphi
%$$
To prove \eqref{e:bitLost1} we use Proposition~\ref{prop:atoms}. 
For this, observe that
$$
\ad_{\tilde{\varphi}\cP_k\tilde{\varphi}}^NP_k^\sharp= \ad_{\tilde{\varphi}\cP_k\tilde{\varphi}}^N(P_k+\psi(\cP_k)).
$$
By Lemma~\ref{l:commuteP},
$$
\ad_{\tilde{\varphi}\cP_k\tilde{\varphi}}^NP_k=O_{2+N}(k^{-N};\Dspace{}\to \Dspace{}) \quad\tand\quad \ad_{\tilde{\varphi}\cP_k\tilde{\varphi}}^NP_k^*=O_{2+N}(k^{-N};\Dspace{}\to \Dspace{}).
$$
By induction, $\ad_A^N(B+C)= \ad_A^N B + \ad_A^N C$, and so
\beq\label{eq:HWB1}
 \ad_{\tilde{\varphi}\cP_k\tilde{\varphi}}^N\cP_k=O_{2+N}(k^{-N};\Dspace{}\to \Dspace{}).
\eeq
Therefore, by Proposition \ref{prop:variant_adNRf(Pk)},
\begin{equation*}
\ad_{\tilde{\varphi}\cP_k\tilde{\varphi}}^N\psi(\cP_k)=O_{2+N}(k^{-N};\Dspace{}\to \Dspace{}).
\end{equation*}
%(the only potential subtlety is the spaces)
We deduce that
$$
\ad_{\tilde{\varphi}\cP_k\tilde{\varphi}}^N P^\sharp_k=O_{2+N}(k^{-N};\Dspace{}\to \Dspace{})
$$
and \eqref{e:bitLost1} -- and hence also the result -- then follows from Proposition~\ref{prop:atoms}.
\end{proof}

\section{Pseudolocality of the elliptic projection}
\label{sec:pseudoLocPi}

In Sections \ref{sec:boundsCsol}, \ref{sec:pseudolocS}, and \ref{sec:assumptions}, we studied pseudolocality 
properties at the continuous level. Another key tool required for the proof of Theorem \ref{t:theRealDeal} is of a discrete nature, namely, we need to establish the spatial pseudolocality of the Galerkin projection $\Pi_k^\sharp$ associated to (the adjoint of) the sesquilinear form $a_k^\sharp$ defined in Definition \ref{e:defPksharp}, see Theorem \ref{t:pseudoLocalPi} below. 

We keep the notation of Section \ref{sec:assumptions}. The operator $\Pi_k^\sharp$ is defined as follows.
\begin{definition}[Elliptic projection]
	\label{def:ellipticProjection}
	Given $k > 0$ and a linear subspace $\blueOld{V} \subset \cZ_k$, the {\em elliptic projection} onto $\blueOld{V}$ is the linear operator $\Pi_k^\sharp : \cZ_k \to \blueOld{V}$ defined by
	\[a_k^\sharp(v,\Pi_k^\sharp u)  = a_k^\sharp(v,u) \qquad \tfa v \in \blueOld{V},\]
	where we recall that $\cZ_k$ is defined by \eqref{e:defZk_concrete}, $a_k^\sharp(u,v) = a_k(u,v) + (S_k u,v)_{\cH}$ and $S_k$ is defined by \eqref{e:def_Sk}. 
\end{definition}
The operator $\Pi_k^\sharp$ is well-defined for all $k > 0$ by the Lax-Milgram theorem, since $a_k^\sharp$ is coercive (by Proposition \ref{prop:ResPksharp}).

\begin{theorem}[Pseudolocality of $\Id - \Pi_k^\sharp$]
	\label{t:pseudoLocalPi}
Let \blueOld{$p\geq 1$, $C_0,\kappa,k_0,\mathfrak{c}>0$}. There exists $h_0>0$ such that for all $N>0$, $\chi,\psi \in C^\infty(\overline{\Omega})$ satisfying $\chi \perp_{\mathfrak{c}} \psi,$ there exists $C>0$ such that for all $k\geq k_0$, \blueOld{all $V_{\mathcal{T}}$ finite element spaces over a mesh $\mathcal{T}$ that are well-behaved of order $p$ at frequency $k$ with constants $(C_0,\kappa)$ in the sense of Definition \ref{d:wellbehaved}, and satisfy $h(\mathcal{T})\leq h_0$, } and $u\in\cZ_k$
	%	there exists $\mathfrak{c} > 0$ such that for every real number $k_0 > 0$, any integer $N > 0$ and any $\chi$, $\psi \in C^\infty(\overline{\Omega})$ satisfying
%	$$\dist(\supp \chi,\supp \psi) \geq \mathfrak{c} k_0^{-1} ,$$  
%
% 	there exists $\mathfrak{C}>0$ such that for all $k \geq k_0$ and all $u \in \cZ_k$
	\[
	\|\chi (\Id - \Pi_k^\sharp) \psi u\|_{H^1_k} \leq C k^{-N} \|(\Id - \Pi_k^\sharp) \psi u\|_{H_k^{-p}},
	\]
where $\Pi_k^\sharp$ is the elliptic projection onto $\blueOld{V}$.
\end{theorem}

\begin{remark}
Through the constants $\mathfrak{c}$ and $h_0$, the assumptions of Theorem \ref{t:pseudoLocalPi} require a sufficient number of ``layers'' of elements separating the supports of $\chi$ and $\psi$. 
\end{remark}

Theorem~\ref{t:pseudoLocalPi} is an immediate consequence of the following two lemmas.

\begin{lemma}
	\label{l:iGainKBigRemainder}
%Let $(V_k)_{k > 0}$ be a well-behaved finite-element of order $p$ in the sense of Definition \ref{d:wellbehaved}, and
Let $\mathfrak{c},\blueOld{C_0,\kappa}>0$ and $p\geq 1$. Then, there exists $h_0>0$ such that the following holds. For any $k_0 > 0$, $N > 0$, and any $\chi_-,\chi_+,\psi \in C^\infty(\overline{\Omega})$ satisfying
	$$\chi_- \prec_\mathfrak{c} \chi_+ \quad\tand\quad\chi_+ \perp \psi,$$
	there exists $C > 0$ such that, for all $k \geq k_0$, \blueOld{all $V_{\mathcal{T}}$ finite element spaces over a mesh $\mathcal{T}$ that are well-behaved of order $p$ at frequency $k$ with constants $(C_0,\kappa)$ in the sense of Definition \ref{d:wellbehaved}, and satisfy $h(\mathcal{T})\leq h_0$, } and $u \in \cZ_k$,
	\[\|\chi_-(\Id -\Pi_k^\sharp) \psi u\|_{H^1_k} \leq Ck^{-N}\Big( \|\chi_+(\Id - \Pi_k^\sharp)\psi u\|_{L^2} + \|(\Id - \Pi_k^\sharp)\psi u\|_{H_k^{-N}}\Big).\] 
\end{lemma}
\ble
\label{l:iAmSoSmooth}
%Let $(V_k)_{k > 0}$ be a well-behaved finite-element of order $p$ in the sense of Definition \ref{d:wellbehaved},
Let $\blueOld{C_0,\kappa,}k_0>0$ and $\mathfrak{c}>0$. Then, there exists $h_0>0$ such that for all $N > 0$ and every $\chi$, $\psi \in C^\infty(\overline{\Omega})$ satisfying $\chi \perp_{\mathfrak{c}} \psi$, there exists $C > 0$ such that, for all $k \geq k_0$,\blueOld{all $V_{\mathcal{T}}$ finite element spaces over a mesh $\mathcal{T}$ that are well-behaved of order $p$ at frequency $k$ with constants $(C_0,\kappa)$ in the sense of Definition \ref{d:wellbehaved}, and satisfy $h(\mathcal{T})\leq h_0$,} and $u \in \cZ_k$,
\begin{equation}
	\label{e:iAmSoSmooth}\|\chi (\Id - \Pi_k^\sharp) \psi u\|_{L^2} \leq C\|(\Id - \Pi_k^\sharp) \psi u\|_{H_k^{-p}}.
\end{equation}
\ele

If $S_k$ is (formally) set to zero, then Lemmas \ref{l:iGainKBigRemainder} and \ref{l:iAmSoSmooth} are analogous to \cite[Lemmas 5.1 and 5.5]{AvGaSp:24}, respectively.

\

\begin{proof}[Proof of Lemma \ref{l:iGainKBigRemainder}]
Let $\mathfrak{c} > 0$ be given and let $h_0 >0$ be a sufficiently small constant depending only on $\mathfrak{c}$. Fix $k_0 > 0$, $N > 0$ and $\chi_-$, $\chi_+$ and $\psi$ as in the statement. 
In what follows, we denote by $C$ a generic constant depending only on the previous quantities. 

We first claim that, without loss of generality, we can assume that $\partial_\nu \chi_-=0$ and thus
\beq\label{e:reallyGoodSandwich1}
\ad_{\chi_-}P_k=O_1(k^{-1};\ZcapHspace{}\to\ZcupHspace{})
\eeq
by Lemma \ref{l:spatialCut} and Definition \ref{def:spatialCutoffs} and 
\beq\label{e:reallyGoodBaguette1}
	\ad_{\chi_-}S_k = O_{-\infty}(k^{-1}; \ZcupHspace{} \to \ZcapHspace{})
\eeq
by Proposition \ref{prop:proofadNR1}.
 Indeed, if $\partial_\nu \chi_-\neq 0$ we apply 
Lemma \ref{l:goodCut} with $\Omega_1= \supp \chi_-$ and $\Omega_2$ equal to $\supp(1-\chi_+)$ enlarged by distance $\mathfrak{c}/2$. We then relabel the resulting cut-off function $\chi_-$ and replace $\mathfrak{c}$ by $\mathfrak{c}/2$.

Let $k \geq k_0$, $u \in \cZ_k$ and suppose that $h \leq h_0$. It is sufficient to prove that
\begin{equation}
\label{e:localQOToShow}
\|\chi_-\Pi^\sharp_k \psi u\|_{H^{1}_k}\leq Ck^{-1/2}\|\chi_+\Pi_k^\sharp \psi u\|_{L^2} + Ck^{-N}\|(\Id - \Pi_k^\sharp )\psi u\|_{H_k^{-N}}
%C\sum_{j = 0}^{\infty} (2^jkd)^{-N}\|(1-\chi_+)\Pi^\sharp_k \psi u\|^2_{\Zspace{-N}(U_j)},
\end{equation}
since 
%the result then follows 
by iterating \eqref{e:localQOToShow} $2N$ times, (changing the cutoffs $\chi_-$ and $\chi_+$), one arrives at
\[\begin{split}
	\|\chi_-\Pi^\sharp_k \psi u\|_{H^1_k}&\leq Ck^{-N}\|\chi_+\Pi_k^\sharp \psi u\|_{L^2} + Ck^{-N}\|(\Id - \Pi_k^\sharp )\psi u\|_{H_k^{-N}}\\
	&= Ck^{-N}\|\chi_+(\Id - \Pi_k^\sharp) \psi u\|_{L^2} + Ck^{-N}\|(\Id - \Pi_k^\sharp )\psi u\|_{H_k^{-N}}
\end{split}\]
using the fact that $\chi_+ \psi = 0$. 

Let $\chi_0\in C^\infty(\overline{\Omega})$ be such that 
$\chi_- \prec_{\mathfrak{c}/4} \chi_0 \prec_{\mathfrak{c/4}} \chi_+$.
By the coercivity of $P_k^\sharp$ (cf. \eqref{eq:CoerciveAsharp}), the definition of $\Pi_k^\sharp$ (Definition \ref{def:ellipticProjection}), locality of $P_k$ and the fact that $\chi_- \psi = \chi_0 \psi = 0$, 
\begin{align*}
%\label{e:localQO3}
%\begin{aligned}
&\|\chi_-\Pi^\sharp_k \psi u\|_{H^1_k}^2 \leq C\abs{\big\langle  P_k^\sharp \chi_-\Pi^\sharp_k \psi u,\chi_- \Pi^\sharp_k \psi u\big\rangle}\\
&\quad = C\abs{\big\langle  P_k^\sharp \chi_-\Pi^\sharp_k \psi u,\chi_- {(\Id - \Pi^\sharp_k)}\psi u\big\rangle}\\
&\quad = C \Big(\abs{\langle P_k^\sharp \chi_-^2\Pi_k^\sharp \psi u, ({\Id - \Pi_k^\sharp}) \psi u\rangle} + \langle [P_k^\sharp,\chi_-] \chi_- \Pi_k^\sharp \psi u,{(\Id - \Pi_k^\sharp)} \psi u \rangle \Big)\\
&\quad = C \Big(\abs{\langle P_k^\sharp (\chi_-^2\Pi_k^\sharp \psi u - w_h),{(\Id - \Pi_k^\sharp)}\psi u\rangle} + \abs{\langle [P_k + S_k,\chi_-] \chi_- \Pi_k^\sharp\psi u ,{(\Id - \Pi_k^\sharp)} \psi u\rangle} \Big)\\
&\quad = C \big(\big|\big\langle P_k (\chi_-^2\Pi_k^\sharp \psi u - w_h),{(\Id - \Pi_k^\sharp)}\psi u\rangle\big|+ \big|\big\langle [P_k,\chi_-] \chi_- \Pi_k^\sharp\psi u ,\chi_0{(\Id - \Pi_k^\sharp)} \psi u\big\rangle\big|\big) + \Rrr,
\end{align*}
for all $w_h \in \blueOld{V_{\mathcal{T}}}$, where 
\begin{equation}
	\label{eq:nonLocalRemainder}
	\Rrr := \big|\big\langle S_k (\chi_-^2 \Pi_k^\sharp \psi u - w_h),(\Id - \Pi_k^\sharp) \psi u\big\rangle\big| + \big|\big \langle [S_k,\chi_-]\chi_- \Pi_k^\sharp \psi u,(\Id - \Pi_k^\sharp) \psi u \big\rangle\big|
\end{equation}
is the ``non-local" part. 
By \eqref{e:reallyGoodSandwich1} combined with the fact that $\chi_0 \psi = 0$,
\begin{align}\nonumber
\|\chi_-\Pi_k^\sharp \psi u\|_{H^1_k}^2&\leq 
C\big|\langle P_k (\chi_-^2\Pi^\sharp_k \psi u-w_h) ,{(\Id-\Pi_k^\sharp)}\psi u\rangle\big|
+Ck^{-1}\|\chi_-\Pi^\sharp_k \psi u\|_{H^1_k}\|\chi_0\Pi^\sharp_k \psi u\|_{L^2} + r
\label{eq:lunch1}
\end{align}
which implies 
\begin{equation}
	\label{eq:aux1}
	\|\chi_-\Pi_k^\sharp \psi u\|_{H^1_k}^2 \leq C\big|\langle P_k^\sharp (\chi_-^2\Pi^\sharp_k \psi u-w_h) ,{(\Id - \Pi^\sharp_k)} \psi u\rangle\big| +C k^{-2} \|\chi_0 \Pi_k^\sharp \psi u\|_{L^2}^2 + r.
\end{equation}%
Let $U_0$ be a neighbourhood of $\supp \chi_-$, and $U_1$ a set contained in $\{\chi_0 \equiv 1\}$. Since $\chi_- \prec_{\mathfrak{c}/4} \chi_0 \prec_{\mathfrak{c/4}} \chi_+$, we can arrange that 
\[d:=\partial_<(U_0,U_1) \geq \mathfrak{c}/8.\]
Hence, taking $h_0 < \frac{\mathfrak{c}}{8\kappa}$, where $\kappa$ is as in Definition \ref{d:sap}, we ensure that $d \geq \kappa h_0$. Thus, we can find a super-approximation $w_h$ to $\chi_-^2\Pi^\sharp_k \psi u$ with $\supp w_h \subset U_1$.
Now, for all $\epsilon<1$, 
\begin{align}\nonumber
&\big|\big\langle P_k(\chi_-^2\Pi^\sharp_k \psi u -w_h),{(\Id - \Pi_k^\sharp)} \psi u\big\rangle\big|
= { \big|\big\langle P_k(\chi_-^2\Pi^\sharp_k \psi u -w_h),\Pi_k^\sharp \psi u\big\rangle\big|} \qquad \textup{(by locality of $P_k$)}\\\nonumber
&\hspace{2cm}\leq \sum_{\blueOld{T}\in\mc{T}_k: \blueOld{T}\cap \supp \chi_0 \neq\emptyset}\| \chi_-^2\Pi^\sharp_k \psi u -w_h\|_{H^{1}_k(\blueOld{T})}\|\Pi_k^\sharp\psi u\|_{H^{1}_k(\blueOld{T})}\quad \textup{(by the definition of $a_k(\cdot,\cdot)$)}\\ \nonumber
&\hspace{2cm}\leq  \sum_{\blueOld{T}\in\mc{T}_k: \blueOld{T}\cap \supp \chi_0 \neq\emptyset}\|\Pi^\sharp_k\psi u\|_{H^{1}_k(\blueOld{T})}\frac{h_\blueOld{T}}{d}\Big(\|\Pi^\sharp_k \psi u\|_{L^2(\blueOld{T})}+\|\chi_-\Pi^\sharp_k \psi u\|_{H^{1}_k(\blueOld{T})}\Big)\\ \nonumber
& \hspace{6cm}\qquad \textup{(by the super-approximation property, Definition \ref{d:sap})}\\ \nonumber
&\hspace{2cm}\leq  C\sum_{\blueOld{T}\in\mc{T}: \blueOld{T}\cap \supp \chi_0 \neq\emptyset}\|\Pi^\sharp_k\psi u\|_{L^2(\blueOld{T})}\frac{1}{kd}\Big(\|\Pi^\sharp_k \psi u\|_{L^2(\blueOld{T})}+\|\chi_-\Pi^\sharp_k \psi u\|_{H^1_k(\blueOld{T})}\Big) 
\\ \nonumber
& \hspace{7cm}\qquad \textup{(by the inverse inequality, Definition  \ref{d:ii})}\\\nonumber
& \hspace{2cm}\leq  C \sum_{\blueOld{T}\in \cT_k:\blueOld{T} \cap \supp(\chi_0) \neq \emptyset} \frac{(1 + \epsilon^{-1})}{(kd)^2}\|\Pi^\sharp_k\psi u\|_{L^2(\blueOld{T})}^2 + \epsilon \|\chi_- \Pi_k^\sharp \psi u\|_{H^1_k(\blueOld{T})}^2\\ 
%&\hspace{2cm}\leq  C\|\chi_+\Pi^\sharp_k\psi u\|_{\mathcal{H}}\frac{1}{kd}\Big(\|\chi_+\Pi^\sharp_k \psi u\|_{\mathcal{H}}+\|\chi_-\Pi^\sharp_k \psi u\|_{\Zspace{1}}\Big)\qquad \textup{(Assumption \ref{ass:bn})}\\
&\hspace{2cm}\leq  \frac{C(1+\e^{-1})}{(kd)^{{2}}}\|\chi_+\Pi^\sharp_k \psi u\|_{L^2(\Omega)}^2+  \e\|\chi_-\Pi^\sharp_k\psi u\|^2_{\cZ_k}.
\label{e:localQO4}
\end{align}
Inputting \eqref{e:localQO4} into \eqref{eq:aux1} and recalling that $d \geq \mathfrak{c} /8$ and $k\geq k_0$, we find that
\begin{equation}
	\label{eq:PplShoutingTrain}
	\|\chi_- \Pi_k^\sharp \psi u\|^2_{\Zspace{}} \leq C k^{-2} \|\chi_+ \Pi_k^\sharp \psi u\|_{L^2(\Omega)}^2 + \Rrr,
\end{equation}
To estimate $\Rrr$, since 
$\textup{dist}(\supp (w_h,\supp (1 - \chi_0))) > 0$, 
pseudolocality of $S_k$ (Theorem \ref{t:pseudoLocGeneral}) implies that 
\begin{align*}
\big|\langle S_k(\chi_-^2\Pi_k^\sharp \psi u-w_h),(1-\chi_0){(\Id - \Pi_k^\sharp)} \psi u\rangle\big| \leq Ck^{-N}
\big\| \chi_-^2\Pi^\sharp_k \psi u -w_h\big\|_{\Zspace{}}
\big\|{(\Id - \Pi_k^\sharp)} \psi u\big\|_{H_k^{-N}}.
\end{align*}
Arguing as in \eqref{e:localQO4}, but now using \ref{e:swlg} where before we used \blueOld{the inverse inequalities of Definition \ref{d:ii}}, and also recalling that $d \geq \mathfrak{c} /8$, we find that% and the fact that $V_k$ satisfies $\max_{K \in \cT_k} h_K \leq Ck^{-1}$,
\begin{equation*}
%	\label{e:saperror}
	\|\chi_-^2 \Pi_k^\sharp \psi u  - w_h\|^2_{\cZ_k} \leq \frac{C}{k^2}\Big(\|\chi_- \Pi_k^\sharp \psi u\|^2_{\cZ_k} + \|\chi_+ \Pi_k^\sharp \psi u\|^2_{L^2(\Omega)}\Big).
\end{equation*}
Therefore, for all $\epsilon < 1$, 
\begin{equation}
	\label{eq:lunch2}
	\begin{split}
		&\big|\langle S_k(\chi_-^2\Pi_k^\sharp \psi u-w_h),(1-\chi_0){(\Id - \Pi_k^\sharp)} \psi u\rangle\big| \\
		&\qquad\qquad = Ck^{-N}\Big(
		\epsilon\|\chi_-\Pi^\sharp_k\psi u\|_{\Zspace{}}^2+
		\e  \|\chi_+\Pi^\sharp_k \psi u\|_{L^2(\Omega)}^2+ \epsilon^{-1} \big\|{(\Id - \Pi_k^\sharp)}\psi u\big\|_{H_k^{-N}}^2
		\Big).
	\end{split}
\end{equation}
Reasoning similarly and using the mapping properties of $S_k$, we find that 
\begin{equation*}
%	\label{eq:lunch3}
	\begin{aligned}
	\big|\langle S_k(\chi_-^2\Pi_k^\sharp \psi u-w_h),\chi_0{(\Id - \Pi_k^\sharp)} \psi u\rangle\big|& \leq Ck^{-1}\Big(
	\epsilon \|\chi_-\Pi^\sharp_k\psi u\|_{\Zspace{}}^2+ \e 
	\|\chi_+\Pi^\sharp_k \psi u\|_{\mathcal{H}}^2+ \epsilon^{-1} \big\|\chi_0\Pi_k^\sharp \psi u\big\|_{H_k^{-N}}^2
	\Big)\\
	&\leq Ck^{-1}\Big(
	\epsilon \|\chi_-\Pi^\sharp_k\psi u\|_{\Zspace{}}^2+ (\e+\e^{-1}) 
	\|\chi_+\Pi^\sharp_k \psi u\|_{L^2(\Omega)}^2
	\Big).
	\end{aligned}
\end{equation*}
Arguing similarly, we obtain
\begin{equation}
\begin{aligned}
	&\big|\big \langle [S_k,\chi_-]\chi_- \Pi_k^\sharp \psi u,(\Id - \Pi_k^\sharp) \psi u \big\rangle\big|\\
	& \leq C \Big(\epsilon k^{-2} \|\chi_- \Pi_k^\sharp \psi u\|_{\Zspace{}}^2 +  k^{-2}(\e+\e^{-1}) \|\chi_+ \Pi_k^\sharp \psi u\|_{L^2(\Omega)}^2 + \epsilon^{-1}k^{-N}\|(\Id - \Pi_k^\sharp)\psi u\|_{H_k^{-N}}^2\Big);
	\end{aligned}
	\label{eq:lunch4}
\end{equation} 
indeed, $[\chi_-,S_k] = \chi_0[\chi_-,S_k] + (1 - \chi_0)[\chi_-,S_k]$
and 
\beqs
\chi_0[\chi_-,S_k] = O_{-\infty}(k^{-1};\ZcupHspace{} \to\ZcapHspace{})\,\,\tand\,\,(1 - \chi_0)[\chi_-,S_k] = - (1-\chi_0) S_k \chi_- = O_{-\infty}(k^{-\infty};\ZcupHspace{}\to \ZcupHspace{})
\eeqs
by, respectively, 
\eqref{e:reallyGoodBaguette1} and  Theorem \ref{t:pseudoLocGeneral}.

Combining~\eqref{eq:lunch2}-\eqref{eq:lunch4} thus leads to 
\[\Rrr\leq C\epsilon \|\chi_- \Pi_k^\sharp \psi u\|_{\cZ_k}^2 + C(1+\e^{-1})k^{-1}\|\chi_+ \Pi_k^\sharp \psi u\|^2_{L^2(\Omega)} + C\epsilon^{-1}k^{-N} \|(\Id - \Pi_k^\sharp)\psi u\|_{H_k^{-N}}^2.\]
for all $\epsilon < 1$. Inserting this estimate in \eqref{eq:PplShoutingTrain} and taking $\epsilon$ sufficiently small, we obtain~\eqref{e:localQOToShow} and hence the result. 
\end{proof}

\

In the proof of Lemma \ref{l:iAmSoSmooth}, we need a variant of Lemma \ref{l:iGainKBigRemainder} where, roughly speaking, the contributions at distance $k^{-1}2^n$ from $\supp \chi$ are multiplied by a weight decaying exponentially in $n$. The main tool is the following lemma:

\begin{lemma}[Dyadic decomposition for $S_k$]
\label{l:dyadicSk}
Let $\phi_0 \in C^\infty_c(\R^d)$ with $\phi_0(x) = 1$ for $|x| \leq \frac12$, $\phi_0(x) = 0$ for $|x|\geq 1$, and let $\phi_n(x) := \phi_0(x/2^{n}) - \phi_0(x/2^{n-1})$ for $n \geq 1$. Let $x_0 \in \overline{\Omega}$, and $R > 0$, and let $\varphi_{n,k} \in C^\infty(\overline{\Omega})$ be defined by
\beqs%\label{e:varphink}
\varphi_{n,k}(x) := \phi_n((x - x_0)/(Rk^{-1})).
\eeqs
Then, for any $k_0 > 0$ and $N \in\mathbb{N}$, there exists $C(N,k_0,R) > 0$ such that for all $k \geq k_0$ and $n \in \mathbb{N}$, 
$$\|\varphi_{n,k}S_k\varphi_{0,k}\|_{L^2 \to H_k^N} \leq C(N,k_0,R) 2^{-nN}.$$
\end{lemma}
\begin{proof}
We start by observing that $\supp \phi_0\subset B(0,1)\blueOld{:=\{ x : |x|< 1\}}$, while for $n \geq 1$,  
$$\supp (\phi_n) \subset B(0,2^{n}) \setminus B(0,2^{n-2}),$$
and in particular, $\phi_n \perp \phi_0$ for $n \geq 2$. Let $\chi_n$ be given by Lemma \ref{l:goodCut} applied with $\Omega_1 := \supp \varphi_{k,n}$ and $\Omega_2 := \supp \varphi_{k,0}$ and $\varepsilon = O(k^{-1}2^n)$. Then $\varphi_{n,k}$ and $\varphi_{0,k}$ are separated by $\chi_n$, and the result of the Lemma follows by the combination of Lemma \ref{l:spatialCut}, Theorem \ref{thm:pseudolocSpace} and Remark \ref{r:paradise}. 
\end{proof}

\begin{lemma}
	\label{l:iStayPutOnBalls}
%	Let $(V_k)_{k > 0}$ be a well-behaved finite-element of order $\blueOld{p}$, 
    Let \blueOld{$p\geq 1$, $C_0,\kappa$,} $k_0 > 0$, $C_\dagger > 0$, let $\psi \in C^\infty(\overline{\Omega})$ and let $N > 0$.
Then there exist $C,\mu,\blueOld{h_0} >0$ such that the following is true. If  $k \geq k_0$,  $x_0 \in \overline{\Omega}$, $R > 0$ and $\chi_-,\chi_+ \in C^\infty(\overline{\Omega})$ satisfy
	\begin{enumerate}
		\item $\supp \chi_- \subset \supp \chi_+ \subset B(x_0, Rk^{-1}/4)$\label{l:stayPutBallCond}
		\item $d:=\dist(\supp \chi_-,\supp(1 - \chi_+)) > \mu  k^{-1}$\label{l:stayPutDistCond}
		\item $\max_{|\alpha| = n} \|\partial^\alpha \chi_-\|_{L^\infty} \leq C_{\dagger} k^{n}$, \quad $n = 0,\ldots,p$, \label{l:stayPutCutoffCond}
		\item $\supp \chi_+ \cap \supp \psi = \emptyset$, \label{l:stayPutInterCond},
	\end{enumerate} 
	then for all \blueOld{all $V_{\mathcal{T}}$ finite element spaces over a mesh $\mathcal{T}$ that are well-behaved of order $p$ at frequency $k$ with constants $(C_0,\kappa)$ in the sense of Definition \ref{d:wellbehaved} with $h(\mathcal{T})\leq h_0$} and all $u \in \cZ_k$, 
	\[\|\chi_-\Pi_k^\sharp \psi u\|_{\cZ_k} \leq C \|\chi_+\Pi_k^\sharp \psi u\|_{L^2(\Omega)} + C\sum_{n = 0}^{\infty} 2^{-Nn} \|\varphi_{n,k}( \Id - \Pi_k^\sharp) \psi u\|_{H_k^{-N}}\]
	where 
	%$\varphi_{-1}(x) = \phi_0((x-x_0)/(Rk^{-1}))$ and for $n \geq 0$, $\varphi_{n,k}(x) := \varphi((x-x_0)/(2^{n} Rk^{-1}))$. 
	$\varphi_{n,k}$ is as in Lemma \ref{l:dyadicSk}.
\end{lemma} 

\begin{proof}%[Proof of Lemma \ref{l:iStayPutOnBalls}]
	Let $k_0 > 0$ be given and let $\mu  := C_0 \kappa$
    %where $C_0$ 
%\blueOld{is the constant in \eqref{e:swlg}}
 %   and $\kappa$ is as in Definition \ref{d:sap}. 
 Let $C_{\dagger} > 0$, $\psi \in C^\infty(\overline{\Omega})$, $N >0$ and denote by $C > 0$ any generic constant depending only on the previous quantities. Let $k \geq k_0$, suppose that $h \leq h_0$ and let $x_0$, $R$ and $\chi_-$, $\chi_+$ as in the statement. The choice of $\mu$ 
 \blueOld{and the inequality \eqref{e:swlg}}
 imply that
	\[d \geq \kappa h.\]
	Therefore, one can proceed as in the proof of Lemma \ref{l:iGainKBigRemainder} using the super-approximation property (Definition \ref{d:sap}), but taking into account that, now, $d$ scales as $k^{-1}$ instead of $1$, so that the analogue of \eqref{e:reallyGoodSandwich1} is 
\beqs
\|\ad_{\chi_-}P_k\|_{\ZcapHspace{n}\to\ZcupHspace{n-1}}\leq C.
\eeqs
This leads to
	\begin{equation}
		\label{eq:newComputer}
		\|\chi_- \Pi_k^\sharp \psi u\|_{\cZ_k}^2 \leq C \|\chi_+ \Pi_k^\sharp \psi u\|_{L^2(\Omega)}^2 + \Rrr,
	\end{equation}
	% where $\Rrr$ is defined as in
	where, as in
 	\eqref{eq:nonLocalRemainder}, 
	
	\newcommand{\Rrrrrrrr}{r}
	\begin{equation}
	\label{e:defRrrrrrrr}
\Rrrrrrrr := \big|\big\langle S_k (\chi_-^2 \Pi_k^\sharp \psi u - w_h),(\Id - \Pi_k^\sharp) \psi u\big\rangle\big| + \big|\big \langle [S_k,\chi_-]\chi_- \Pi_k^\sharp \psi u,(\Id - \Pi_k^\sharp) \psi u \big\rangle\big|.
	\end{equation}

	We now use the property that for all $x \in \mathbb{R}^d$, 
	\[\sum_{n \in \mathbb{N}} \varphi_{n,k}^2(x) \geq 1/2 \]
	(see e.g. \cite[Lemma 1.1.1]{AlGe:07}) to write, for any $f,g \in \cZ_k$, 
	\[\abs{\langle S_k f,g\rangle} \leq  2\sum_{n = 0}^{\infty} \abs{\langle \varphi_{n,k} S_k f, \varphi_{n,k} g\rangle}.\]	
 	Taking $f = \chi_-^2 \Pi_k^\sharp \psi u - w_h$ and $g = (\Id - \Pi_k^\sharp)\psi u$, and using the fact that $f = \varphi_{0,k}f$, the first term of \eqref{e:defRrrrrrrr} is estimated by
	\[\begin{split}
		&\abs{\big\langle S_k ( \chi_-^2 \Pi_k^\sharp \psi u - w_h),(\Id - \Pi_k^\sharp)\psi u\big\rangle}\\ 
		& \qquad \leq 2 \sum_{n = 0}^{\infty} \big|\big\langle \varphi_{n,k} S_k \varphi_{0,k}(\chi_-^2 \Pi_k^\sharp \psi u - w_h),\varphi_{n,k}(\Id - \Pi_k^\sharp)\psi u\big\rangle\big|\\
		&\qquad \leq C\|\chi_-^2 \Pi_k^\sharp \psi u - w_h\|_{\cZ_k} \sum_{n = 0}^{\infty} \|\varphi_{n,k} S_k \varphi_{0,k}\|_{\cZ_k \to H_k^{N}} \|\varphi_{n,k} (\Id - \Pi_k^\sharp)\psi u\|_{H_k^{-N}} \\
		&\qquad\leq C(\|\chi_- \Pi_k^\sharp u\|_{\cZ_k} + \|\chi_+ \Pi_k^\sharp u\|_{L^2(\Omega)}) \sum_{n = 0}^{\infty} \|\varphi_{n,k} S_k \varphi_{0,k}\|_{\cZ_k \to H_k^{N}} \|\varphi_{n,k} (\Id - \Pi_k^\sharp)\psi u\|_{H_k^{-N}},
	\end{split}\]
	where we have used that
	\[\|\chi_-^2 \Pi_k^\sharp \psi u - w_h\|_{\cZ_k} \leq C \frac{1}{(kd)^2} \|\chi_- \Pi_k^\sharp u\|_{\cZ_k} + \|\chi_+ \Pi_k^\sharp u\|_{L^2(\Omega)}\]
	(obtained by reasoning as in \eqref{e:localQO4}), and taken into account that $(kd)^{-1} \leq M^{-1} \leq C$. By Lemma \ref{l:dyadicSk}, 
	\[\|\varphi_{n,k} S_k \varphi_{0,k}\|_{\cZ_k \to H_k^{N}} \leq C 2^{-Nn} \quad \tfa n \in \mathbb{N}.\]
	Hence, for all $\epsilon < 1$, 
	\begin{align}\nonumber
		&\abs{\big\langle S_k ( \chi_-^2 \Pi_k^\sharp \psi u - w_h),(\Id - \Pi_k^\sharp)\psi u\big\rangle} \\
		%\nonumber
		&\qquad \leq C\epsilon \|\chi_- \Pi_k^\sharp \psi u\|_{\cZ_k}^2 + C\epsilon \|\chi_+ \Pi_k^\sharp \psi u\|_{L^2(\Omega)}^2 + C\epsilon^{-1} \left(\sum_{n = 0}^{\infty} 2^{-Nn} \|\varphi_{n,k} (\Id - \Pi_k^\sharp)\psi u\|_{H_k^{-N}}\right)^2 
		%\nonumber
		\label{eq:Sestim1}
		%&\qquad \leq  C\epsilon \|\chi_- \Pi_k^\sharp \psi u\|_{\cZ_k}^2 + C\epsilon \|\chi_+ \Pi_k^\sharp \psi u\|_{L^2(\Omega)}^2\\ 
		%\nonumber
		%&\hspace{5cm} + C\epsilon^{-1} \left(\sum_{n = 0}^{\infty} 2^{-2nN} \right) \left(\sum_{n = 0}^{\infty}2^{-2nN}\|\varphi_{n,k} (\Id - \Pi_k^\sharp)\psi u\|_{H_k^{-N}}^2\right)\\
		%&\qquad \leq  C\epsilon \|\chi_- \Pi_k^\sharp \psi u\|_{\cZ_k}^2 + C\epsilon \|\chi_+ \Pi_k^\sharp \psi u\|_{\cH}^2+ C\epsilon^{-1}\sum_{n = 0}^{\infty}2^{-2nN}\|\varphi_{n,k} (\Id - \Pi_k^\sharp)\psi u\|_{H_k^{-N}}^2
	\end{align}
	Similarly, using that $\varphi_{n,k} \chi_-  =  0$ for $n \geq 1$, we deduce that for all $\varepsilon < 1$, 
	\begin{align}\nonumber
		&\abs{\big\langle [\chi_-,S_k] \chi_- \Pi_k^\sharp \psi u, (\Id - \Pi_k^\sharp) \psi u\big\rangle}\\\nonumber &\qquad\leq 2 \abs{\langle \varphi_{0,k} [S_k,\chi_-] \chi_- \Pi_k^\sharp \psi u, \varphi_{0,k} (\Id - \Pi_k^\sharp)\psi u\rangle} + \sum_{n = 1}^{\infty} \abs{\big\langle \varphi_{n,k} S_k\chi_-^2\Pi_k^\sharp \psi u, \varphi_{n,k} (\Id - \Pi_k^\sharp)\psi u\big\rangle} \\\nonumber
		&\qquad \leq C \big\|\chi_- \Pi_k^\sharp \psi u\big\|_{\cZ_k} \big\|\varphi_{0,k} [S_k,\chi_-]\big\|_{\cZ_k \to H_k^{N}} \big\|\varphi_{0,k}(\Id - \Pi_k^\sharp)\psi u\big\|_{H_k^{-N}}\\\nonumber
		&\qquad\qquad \qquad + C \|\chi_- \Pi_k^\sharp \psi u\|_{\cZ_k} \sum_{n = 1}^{\infty} \|\varphi_{n,k} S_k \varphi_{0,k}\|_{\cZ_k \to H_k^{N}} \|\varphi_{n,k} (\Id - \Pi_k^\sharp) \psi u\|_{H_k^{-N}}\\
		& \qquad \leq C \varepsilon \|\chi_- \Pi_k^\sharp \psi u\|_{\cZ_k}^2 + C \varepsilon^{-1} \left( \sum_{n = 0}^{\infty} 2^{-Nn} \|\varphi_{n,k} (\Id - \Pi_k^\sharp)\psi u\|_{H_k^{-N}}\right)^2.
		\label{eq:Sestim2}
	\end{align}
Adding \eqref{eq:Sestim1} and \eqref{eq:Sestim2}, inserting the result in \eqref{e:defRrrrrrrr} and then in \eqref{eq:newComputer}, %using that 
%$(\sum_{n} |a_n|)^2 \geq \sum_{n} |a_n|^2$ 
and letting $\varepsilon$ be small enough, we obtain the result.
\end{proof}

\

\begin{proof}[Proof of Lemma \ref{l:iAmSoSmooth}]
We claim that, given $\mathfrak{c} > 0$, there exists $h_0 > 0$ such that for all $k_0 > 0$ and $\chi_-, \chi_+ ,\psi \in C^\infty(\overline{\Omega})$ satisfying 
\begin{equation}
	\label{e:sepChi-Chi+}
\chi_-\prec_{\mathfrak{c}} \chi_+ \quad \tand \quad \chi_- \perp \psi
\end{equation}
there exists 
% \mathfrak{C}
$C > 0$ such that for all $k \geq k_0$, $h\leq h_0$, $u \in \cZ_k$ and $0\leq j\leq p-1$,
\begin{equation}
	\label{e:iAmSoSmooth2}
	\|\chi_{-} \Pi^\sharp_k \psi u\|_{H_k^{-j}}\leq C\Big(\|\chi_+\Pi_k^\sharp \psi u\|_{H_k^{-(j+1)}} 
	+\big\|(\Id - \Pi_k^\sharp)\psi u\big\|_{H_k^{-p}}
	\Big).
\end{equation}
If this is true, the lemma follows easily. Indeed, given $\mathfrak{c} > 0$, let $\mathfrak{c'} = \frac{\mathfrak{c}}{2p}$, and given $\chi, \psi$ as in the statement, let $\chi_1,\ldots, \chi_p$ be a sequence of nested cutoffs, i.e. such that
\begin{equation*}
\chi_i \prec_{\mathfrak{c'}}\chi_{i+1}
	\qquad \tfa i \in \{1,\ldots,p-1\}\quad\tand\quad \chi_p \perp \psi.
\end{equation*}
One then applies \eqref{e:iAmSoSmooth2} $p$ times with $\chi_- = \chi_i$ and $\chi_+= \chi_{i+1}$, using at the end that $\chi_p \Pi_k^\sharp \psi u = \chi_p (\Id - \Pi_k^\sharp) \psi u$ to obtain \eqref{e:iAmSoSmooth}.

It therefore remains to prove  \eqref{e:iAmSoSmooth2}. Let $\mathfrak{c} > 0$ be given and let $h_0 > 0$ be a sufficiently small constant. Let $k_0 > 0$, $\chi_-,\chi_+,\psi \in C^\infty(\overline{\Omega})$ such that \eqref{e:sepChi-Chi+} holds and let $C > 0$ denote a generic constant depending only on the previous quantities. Let $k \geq k_0$, suppose that $h\leq h_0$, let $u \in \cZ_k$ and let $0 \leq j \leq p-1$.

By arguing exactly as at the start of the proof of Lemma \ref{l:iGainKBigRemainder}, without loss of generality, we can assume that $\partial_\nu \chi_-=0$ 
and thus the commutator estimates \eqref{e:reallyGoodSandwich1} and \eqref{e:reallyGoodBaguette1} hold
(the first one by Lemma \ref{l:spatialCut} and Definition \ref{def:spatialCutoffs}, and the second one
by Proposition \ref{prop:proofadNR1}).

Let $\chi_0,\chi_1 \in C^\infty(\overline{\Omega})$ be such that 
\[
 \chi_- \prec_{(\mathfrak{c}/3)} \chi_0 \prec_{(\mathfrak{c}/3)} \chi_1 \prec_{(\mathfrak{c}/3)} \chi_+.
\]
To prove \eqref{e:iAmSoSmooth2}, it is sufficient to show that, for all $v \in H_k^{j}$, 
\begin{equation}
\label{e:regularShow}
\big|\big\langle v,\chi_- \Pi^\sharp_k \psi u\big\rangle\big|\leq C\Big(\|\chi_+\Pi^\sharp_k \psi u\|_{H_k^{-(j+1)}}+\big\|(\Id-\Pi_k^\sharp)\psi u\big\|_{H_k^{-p}}
\Big)\|v\|_{H_k^{j}}.
\end{equation}
By the relation $P_k^\sharp\RPd = \Id$, the definition of $\Pi_k^\sharp$ (Definition \ref{def:ellipticProjection}), and the fact that $\chi_-\psi = 0$, for all $w_h \in \blueOld{V_{\mathcal{T}}}$,
\begin{align}\nonumber
	\langle v,\chi_{-} \Pi_k^\sharp \psi u\rangle &=	\big\langle  P_k^\sharp R_{k}^\sharp v, \chi_{-}{(\Id - \Pi_k^\sharp) }\psi u\big\rangle  \\ \nonumber
	&=	\big\langle  \chi_{-} P_k^\sharp R_{k}^\sharp v, {(\Id - \Pi_k^\sharp)}\psi u\big\rangle  \\ \nonumber
		& = \big\langle  P_k^\sharp \chi_-R_{k}^\sharp v,{(\Id - \Pi_k^\sharp)}\psi u\big\rangle + \big\langle  [\chi_-,P_k^\sharp] R_{k}^\sharp v,{(\Id - \Pi_k^\sharp)}\psi u\big\rangle\\ 
				& = \big\langle  P_k^\sharp\big( \chi_-R_{k}^\sharp v -w_h\big), {(\Id - \Pi_k^\sharp)} \psi u\big\rangle +
				 \big\langle  [\chi_-,P_k^\sharp] R_{k}^\sharp v,{(\Id - \Pi_k^\sharp)}\psi u\big\rangle.
%	&= a_k^\sharp(\Pi_k^\sharp \psi u, \chi_- (R_k^\sharp)^* v - w_h) + a_k(\psi u,w_h) + \big\langle [P_k^\sharp,\chi_-] \Pi_k^\sharp \psi u,(R_k^\sharp)^* v\big\rangle.
\label{e:EuanYawn1}
\end{align}
For the second term on the right-hand side of \eqref{e:EuanYawn1}, since $\chi_- = \chi_- \chi_+$, by locality of $P_k$, and by the mapping properties of $R_k^\sharp$ (Proposition \ref{prop:ResPksharp}), $[\chi_-,P_k]$ 
(from \eqref{e:reallyGoodSandwich1}), and  $[\chi_-,S_k]$ (from \eqref{e:reallyGoodBaguette1}),
%Assumption \ref{ass:spatialCutoffs}\eqref{ass:Spat:commProp1} and Proposition \ref{prop:proofadNR1} with $N = 1$),  
\begin{align}\nonumber
	\abs{\big\langle  [\chi_-,P_k^\sharp] R_{k}^\sharp v,{(\Id - \Pi_k^\sharp)} \psi u\big\rangle} &\leq \abs{\big\langle  \chi_+ [\chi_-,{P_k}] R_{k}^\sharp v,{(\Id - \Pi_k^\sharp)} \psi u\big\rangle}+ \abs{ \big\langle  [\chi_-,{S_k}] R_{k}^\sharp v,{(\Id - \Pi_k^\sharp)}  \psi u\big\rangle} \\\nonumber
	&\leq Ck^{-1} \|v\|_{H_k^{j}} \Big(\|\chi_+(\Id - \Pi_k^\sharp )\psi u\|_{H_k^{-j-1}} + \big\|(\Id - \Pi_k^\sharp)\psi u\big\|_{H_k^{-p}} \Big).\nonumber
\end{align}
Furthermore, by the mapping properties of $S_k$ (Proposition \ref{prop:f(Pk)}),
\begin{align*}
\abs{\big\langle S_k (\chi_- R_k^\sharp v - w_h),(\Id - \Pi_k^\sharp)\psi u\big\rangle} 
%&\leq C \big\|\chi_- R_k^\sharp v - w_h\big\|_{\Zspace{1}} \big\|(\Id - \Pi_k^\sharp)\psi u\big\|_{\Zspace{-p}}\\
&\leq C \big\|\chi_- R_k^\sharp v - w_h\big\|_{\Zspace{1}} \big\|(\Id - \Pi_k^\sharp)\psi u\big\|_{H_k^{-p}}.
\end{align*}
Using the approximation property of $\blueOld{V_{\mathcal{T}}}$ (Assumption \ref{d:app}) with $m = 1$, and using that $j+2 \leq p+1$, we can choose $w_h$ supported in $\supp \chi_0$ such that
\beq\label{e:bankHoliday2}
\sum_{\blueOld{T} \in \mathcal{T}_k}(h_\blueOld{T} k)^{-2(j+1)} \|\chi_- R_k^\sharp v- w_h\|_{H^1_k(\blueOld{T})}^{2} \leq \Cfem \|\chi_- R_k^\sharp v\|^2_{H_k^{j+2}} \leq C \|v\|_{H_k^{j}}^2.
\eeq
In particular, 
\[\begin{split}
	\big\|\chi_- R_k^\sharp v - w_h\big\|_{\Zspace{1}}^2 = \sum_{\blueOld{T} \in \cT_k} \|\chi_- R_k^\sharp v - w_h\|_{H^1_k(\blueOld{T})} 
	& \leq (hk)^{2j}  \sum_{\blueOld{T} \in \cT_k} h_\blueOld{T}^{-2j} \|\chi_- R_k^\sharp v - w_h\|_{H^1_k(\blueOld{T})} \\
	& \leq C(hk)^{2j} \N{v}_{H_k^{j}}^2,
\end{split}\]
where we recall that $h: = \max_{\blueOld{T} \in \cT_k} h_\blueOld{T}$. Hence, with this choice of $w_h$, since $\blueOld{V_{\mathcal{T}}}$ satisfies $hk \leq C$, 
\[\abs{\big\langle S_k (\chi_- R_k^\sharp v - w_h),(\Id - \Pi_k^\sharp)\psi u\big\rangle} \leq C \|v\|_{H_k^{j}} \big\|(\Id - \Pi_k^\sharp)\psi u\big\|_{H_k^{-p}}.\]
%In particular, as $k \to \infty$, this is controlled by the right-hand side of \eqref{e:regularShow}.  
Therefore, to prove \eqref{e:regularShow}, it remains to prove that for this choice of $w_h\in \blueOld{V_{\mathcal{T}}}$,
\begin{equation}
	\label{e:EuanYawn3}
	\big|\big\langle  P_k\big( \chi_-R_{k}^\sharp v -w_h\big),{(\Id - \Pi_k^\sharp )}\psi u\big\rangle\big|\leq C \Big(\|\chi_+ \Pi_k^\sharp \psi u\|_{H_k^{-(j+1)}} + \big\|(\Id - \Pi_k^\sharp)\psi u\big\|_{H_k^{-p}}\Big)\|v\|_{H_k^{j}}.
\end{equation}
By (in this order) the locality of $P_k$ and  the fact that $\chi_0\equiv 1$ on the support of $\chi_- \RPd v - w_h$, 
continuity of $P_k$, and \eqref{e:bankHoliday2},
\begin{align}\nonumber
&|\langle P_k(\chi_- R_{k}^\sharp v-w_h),(\Id - \Pi_k^\sharp) \psi u\rangle|\\ \nonumber
&\qquad = |\langle P_k(\chi_- R_{k}^\sharp v-w_h),\chi_0(\Id - \Pi_k^\sharp) \psi u\rangle|\\\nonumber
&\qquad\leq \sum_{\blueOld{T}\in \mc{T}_k}\|\chi_0 (\Id - \Pi_k^\sharp) \psi u\|_{H^1_k(\blueOld{T})}\|\chi_- R_{P^\sharp_k}v-w_h\|_{H^1_k(\blueOld{T})}
\\\nonumber
&\qquad\leq \Big(\sum_{\blueOld{T}\in \mc{T}_k}(h_\blueOld{T}k)^{2(j+1)}\|\chi_0 \Pi^\sharp_k \psi u\|^2_{H^1_k(\blueOld{T})}\Big)^{\frac{1}{2}}\Big(\sum_{\blueOld{T}\in \mc{T}_k}(h_\blueOld{T} k)^{-2(j+1)}\|\chi_- R_{P^\sharp_k}v-w_h\|^2_{H^1_k(\blueOld{T})}\Big)^{\frac{1}{2}}\\
%&\qquad\leq  C\Big(\sum_{\blueOld{T}\in \mc{T}}(h_\blueOld{T}k)^{2(j+1)}\|\chi_0\Pi^\sharp_k \psi u\|^2_{H^1_k(\blueOld{T})}\Big)^{\frac{1}{2}}\|\chi_- R_{P^\sharp_k}v\|_{\Zspace{j+2}}\\
%&\leq  C\Big(\sum_{\blueOld{T}\in \mc{T}}(h_\blueOld{T}k)^{2(j+1)}\|\Pi^\sharp_k u\|^2_{\Zspace{1}}\Big)^{\frac{1}{2}}\|R_{P^\sharp_k}v\|^2_{\Zspace{j+2}}\\
&\qquad\leq C \Big(\sum_{\blueOld{T}\in \mc{T}_k}(h_\blueOld{T}k)^{2(j+1)}\|\chi_0\Pi^\sharp_k\psi  u\|^2_{H^1_k(\blueOld{T})}\Big)^{\frac{1}{2}}\|v\|_{H_k^{j}}.
\label{e:dualityGivesRegularity}
\end{align}
To apply the arguments from \cite[Lemma 5.5]{AvGaSp:24}, and especially the wavelength-scale quasi-uniformity \blueOld{\eqref{e:qu}}, we now need to group the elements $\blueOld{T} \in \cT_k$ into sets lying within balls of radius $\approx k^{-1}$. 

To this end, we choose a sufficiently large constant $\mu > 0$ depending only on $h_0$ and $k_0$ and let $\{x_\ell\}_{\ell=1}^{\cL}\subset\Omega$ be a ``maximal $\mu k^{-1}$ separated set", constructed inductively by choosing an initial point $x_1 \in \Omega$, and if $x_1,\ldots,x_\ell$ are constructed, choosing $x_{\ell + 1} \in \Omega \setminus \cup_{m = 1}^\ell B(x_\ell,\mu k^{-1})$ if this set is not empty, or finishing the construction with $\ell = \cL$ otherwise. By construction,
\begin{equation*}
%\label{e:covers}
\Omega\subset \bigcup_{\ell=1}^{\cL} B(x_\ell, \mu k^{-1}),
\end{equation*}
and one can check that for all $M>0$, there exists $\mathfrak{D}_M>0$ depending solely on $M$ and the space dimension $d$, and there exists a partition of $\{1,\ldots,\mathcal{L}\}$ into $\mathfrak{D}_M$ sets $\mathcal{J}_1^M, \mathcal{J}_2^M,\ldots,\mathcal{J}_{\mathfrak{D}_M}^M$, such that
\begin{equation*}
%\label{e:notTooManyOverlaps}
(\ell_1,\ell_2\in \mc{J}_m^M \,\,\tand \,\,\ell_1\neq \ell_2 )\,\Rightarrow\, B(x_{\ell_1},M\mu k^{-1})\cap B(x_{\ell_2},M\mu  k^{-1})=\emptyset,
\end{equation*}
i.e., the maximal number of overlaps between balls of radius $M\mu k^{-1}$ with centers in $\{x_{\ell}\}_{\ell = 1}^{\cL}$ is $\mathfrak{D}_M$.  
Define
$$
h_\ell:=\max\big\{ h_\blueOld{T}\,:\, \blueOld{T}\cap B(x_\ell,k^{-1}) \neq \emptyset\big\}\leq C_M\inf\big\{ h_\blueOld{T}\,:\, \blueOld{T}\cap B(x_\ell, M\mu k^{-1})\neq \emptyset\big\},
$$
where the second inequality follows from \blueOld{\eqref{e:qu}}.
%Assumption~\ref{ass:qu}. 
For all $1 \leq \ell \leq \cL$ and for $m \geq 1$, let  $\chi_{\ell,m} \in C^\infty(\overline{\Omega})$ be such that 
\begin{equation}
	\label{e:condChiEllM}
	\supp \chi_{\ell,m} \subset B(x_\ell,(m+1)\mu k^{-1}) \cap \overline{\Omega}, \quad \supp (1 - \chi_{\ell,m}) \cap B(x_{\ell},m\mu k^{-1}) \cap \overline{\Omega} = \emptyset.
\end{equation}
Using a construction via scaling, one can arrange that there exists a universal constant $C_\dagger$ such that
\begin{equation}
	\|\partial^\alpha \chi_{\ell,m}\|_{\infty} \leq C_\dagger  (\mu  k^{-1})^{-|\alpha|} \quad \tfa |\alpha|\leq p.
	\label{e:controlChiLM}
\end{equation} 
By choosing $\mu $ large enough, one ensures that $\mu  k^{-1} \geq 2h$, which implies that $\blueOld{T} \cap B(x_\ell,\mu  k^{-1}) \neq \emptyset \implies \blueOld{T} \subset B(x_\ell,2\mu  k^{-1})$. Therefore, 
\begin{align}\nonumber
\sum_{\blueOld{T}\in \mc{T}_k}(h_\blueOld{T}k)^{2(j+1)}\|\chi_0\Pi^\sharp_k \psi u\|^2_{H^{1}_k(\blueOld{T})}&\leq C\sum_{\ell=1}^{\cL} (h_{\ell}k)^{2(j+1)}\sum_{\blueOld{T}\cap B(x_\ell,\mu k^{-1})\neq \emptyset}\|\chi_0\Pi^\sharp_k \psi u\|^2_{H^1_k(\blueOld{T})}\\ \nonumber
&\leq C\sum_{\ell=1}^{\cL} (h_{\ell}k)^{2(j+1)}\sum_{\blueOld{T}\subset B(x_\ell,2\mu k^{-1})}\|\chi_0\Pi^\sharp_k \psi u\|^2_{H^1_k(\blueOld{T})}\\
&\leq C\sum_{\ell = 1}^{\cL} (h_{\ell}k)^{2(j+1)}\|\chi_0\chi_{\ell,2}\Pi^\sharp_k \psi u\|^2_{\cZ_k}.
\label{e:regrouping}
\end{align}
Next we apply Lemma~\ref{l:iStayPutOnBalls} to estimate the norms $\|\chi_0\chi_{\ell,2} \Pi_k^\sharp \psi u\|_{\cZ_k}$. Choosing $R = 20\mu $, one gets that for every $\ell = 1,\ldots,\cL$, 
\[ \supp (\chi_0 \chi_{\ell,2}) \subset \supp (\chi_1 \chi_{\ell,4}) \subset B(x_\ell,Rk^{-1}/4),\]
so that assumption \eqref{l:stayPutBallCond} is satisfied.
Moreover, by definition of $\chi_{\ell,m}$ (see \eqref{e:condChiEllM}), 
\[d:=\dist (\supp(\chi_0 \chi_{\ell,2}), \supp (1- \chi_1 \chi_{\ell_4})) \geq \dist(B(x_\ell,3\mu k^{-1}),B(x_\ell,4 \mu k^{-1})^c) \geq \mu k^{-1},\]
showing that assumption \eqref{l:stayPutDistCond} is also satisfied by taking $\mu$ sufficiently large.
Assumption \eqref{l:stayPutCutoffCond} follows from this and the control on the derivatives of $\chi_{\ell,m}$ in \eqref{e:controlChiLM}. Finally, we have $\supp (\chi_1 \chi_{\ell,4})\subset \supp \chi_{1} \subset (\supp \psi)^c$ so that assumption \eqref{l:stayPutInterCond} holds. Therefore,  by Lemma~\ref{l:iStayPutOnBalls}, for $N\geq p$,
\begin{equation}
\label{e:crossfire1}
\| \chi_0 \chi_{\ell,2}\Pi^\sharp_k \psi u\|^2_{\cZ_k}\leq C\|\chi_{1} \chi_{\ell,4}\Pi^\sharp_k \psi  u\|^2_{L^2(\Omega)} + C\sum_{n= 0}^{\infty}2^{-Nn}\|\varphi_{n,\ell}(\Id - \Pi_k^\sharp) \psi u\|_{\Hspace{-p}}^2
\end{equation}
where 
\begin{equation}
\label{e:varphi}
\varphi_{n,\ell} = \phi_n\big((x - x_\ell)/( Rk^{-1})\big) \tfor n \geq 0,
\end{equation}
with $\phi_n$ as in Lemma \ref{l:dyadicSk}.
(note that here it is crucial that $\mathfrak{C}$ in Lemma \ref{l:iStayPutOnBalls} does not depend on $\chi_-$ and $\chi_+$). 

By \eqref{e:EuanYawn3}, \eqref{e:dualityGivesRegularity}, \eqref{e:regrouping}, and \eqref{e:crossfire1}, to prove \eqref{e:regularShow}, it is sufficient to prove that 
\beq\label{e:crossfire2}
	\sum_{\ell = 1}^\cL (h_\ell k)^{2(j+1)}\|\chi_{1}\chi_{\ell,4} \Pi^\sharp_k \psi  u\|^2_{L^2(\Omega)} \leq C \|\chi_+\Pi_k^\sharp \psi u\|_{H_k^{-(j+1)}}^2.
\eeq
and 
\beq
\sum_{\ell= 1}^{\cL} (h_\ell k)^{2(j+1)}\sum_{n = 0}^{\infty} 2^{-Nn} \|\varphi_{n,\ell} (\Id - \Pi_k^\sharp)\psi u\|_{H_k^{-p}}^2 \leq C\|(\Id - \Pi_k^\sharp) \psi u\|_{H_k^{-p}}^2.
\label{e:crossfire3}
\eeq

By the wavelength-scale quasi-uniformity \blueOld{\eqref{e:qu}} and  the inverse inequality (Definition \ref{d:ii}),
\[\begin{split}
	\sum_{\ell = 1}^\cL (h_\ell k)^{2j+2}\|\chi_{1}\chi_{\ell,4} \Pi^\sharp_k \psi  u\|^2_{L^2(\Omega)} &\leq C \sum_{\ell = 1}^{\cL} \sum_{\blueOld{T} \cap \supp \chi_{1}\chi_{\ell,4} \neq \emptyset} (h_\blueOld{T} k)^{2j+2} \|\Pi_k^\sharp \psi u\|^2_{L^2(\blueOld{T})} \\
	& \leq C \sum_{\ell = 1}^{\cL} \sum_{\blueOld{T} \cap \supp \chi_{1}\chi_{\ell,4} \neq \emptyset}  \|\Pi_k^\sharp \psi u\|^2_{H^{-(j+1)}(\blueOld{T})}.
	\end{split}\]
Applying \cite[Lemma 5.2]{AvGaSp:24} (a simple bound on sums of negative Sobolev norms on elements by a global dual norm) and Lemma \ref{l:sumNegNorms} below, there exists $M > 0$ large enough such that 
\[
\sum_{l = 1}^{\cL}\sum_{\blueOld{T} \cap \supp \chi_{1}\chi_{\ell,4} \neq \emptyset} \|\Pi_k^\sharp \psi u\|^2_{H^{-(j+1)}(\blueOld{T})} \leq C \sum_{l = 1}^{\cL} \big\|\chi_+ \chi_{\ell,M} \Pi_k^\sharp \psi u\big\|^2_{H_k^{-(j+1)}} \leq C \mathfrak{D}_M \|\chi_+\Pi_k^\sharp \psi u\|_{H_k^{-(j+1)}}^2,
\]
and the combination of these last two displayed equations is \eqref{e:crossfire2}.

On the other hand, 
by the definition of $\varphi_{n,\ell}$ \eqref{e:varphi}, there exists $C_{p,R}>0$ such that 
for every $x \in \R^d$, 
\beq\label{e:crossfire4}
	\sum_{\ell = 1}^{\cL}\sum_{n = 0}^{\infty} 2^{-Nn}
\Big(\max_{1\leq \alpha\leq p} k^{-|\alpha|}\abs{\varphi_{n,\ell}(x)}\Big)^2	
	 \leq C_{p,R}\sum_{n = 0}^{\infty} 2^{-Nn} \kappa_n(x)
\eeq
where
\[\kappa_n(x) = \textup{Card}\Big(\big\{ 1 \leq \ell \leq \cL \,:\, \dist(x,x_\ell) \leq 2^{n+1}Rk^{-1} \big\}\Big).\]
To estimate $\kappa_n(x)$, we write
\[ \kappa_n(x) = \sum_{m = 1}^{\mathfrak{D}_1} \textup{Card}(K_m(x)),\quad \textup{where}\quad K_m(x) := \big\{ \ell \in \mathcal{J}_m^1 \,:\, \dist(x,x_\ell) \leq 2^{n+1} Rk^{-1}\big\}\] 
and since the $\mu k^{-1}$ balls centered at $x_\ell$ for $\ell \in \mathcal{J}_m^1$ are pairwise disjoint, 
\[\sum_{\ell \in K_m(x)} \mu(B(x_\ell,k\mu^{-1})) \leq \mu (B(x,2^{n+1}Rk^{-1})) \quad \iff \quad \textup{Card}(K_m(x)) \leq 2^{d(n+1)}\frac{R^d}{\mu^d}.\]
This proves that $\kappa_n(x) \leq C 2^{nd}$. 
Inserting this bound into \eqref{e:crossfire4} and choosing $N$ large enough,  we find that 
\beqs%\label{e:crossfire4a}
	\sum_{\ell = 1}^{\cL}\sum_{n = 0}^{\infty} 2^{-Nn}
\Big(\max_{1\leq \alpha\leq p} k^{-|\alpha|}\abs{\varphi_{n,\ell}(x)}\Big)^2	
	 \leq C C_{p,R}.
\eeqs
We now use Lemma \ref{l:sumNegNorms} with $\{\chi_n\}_{n} = \{ 2^{-nN/2} \varphi_{n,\ell}, : \ell=1,\ldots, \cL , n=0,\ldots, \infty\}$ to obtain that
%
%Thus there exists $C_{\rm over}'>0$ such that, for all $N$ sufficiently large,
%\[
%\sum_{\ell = 1}^{\cL}\sum_{n = -1}^{\infty} 2^{-Nn} 1_{x\in \supp \varphi_{n,\ell}} < C_{\rm over}'.
%\]
%Therefore, by the definition of $\varphi_{n,\ell}$ \eqref{e:varphi}, given $N>0$, there exists $C_{\rm over}$ such that for all $x \in \R^d$ and all $k \geq k_0$, 
%\[\sum_{\ell = 1}^{\cL}\sum_{n = -1}^{\infty} 2^{-Nn} \Big(\max_{1\leq \alpha\leq N} k^{-|\alpha|}\abs{\varphi_{n,\ell}(x)}\Big)^2  < C_{\rm over}.\]
%Thus, by Lemma \ref{l:sumNegNorms}, 
\[\sum_{\ell= 1}^{\cL} \sum_{n = 0}^{\infty} 2^{-Nn} \|\varphi_{n,\ell} (\Id - \Pi_k^\sharp)\psi u\|_{H_k^{-p}}^2 \leq CC_{p,R} \|(\Id - \Pi_k^\sharp) \psi u\|_{H_k^{-p}}^2,\]
which is \eqref{e:crossfire3}, and the proof is complete.
\end{proof}

\

Finally, we prove the following technical lemma used in the proof of Lemma \ref{l:iAmSoSmooth}.

% Given a sequence $(\chi_n)_n$ of functions in $C^\infty(\R^d)$, suppose there exists $C_{\rm over}> 0$ such that for all $x \in \R^d$,
%		\[\sum_{n = 0}^{\infty} |\chi_n(x)|^2 \leq C_{\rm over}.\]
%		Then for any sequence $(u_n)_n$ of elements of $\Hspace{j}$, 
%		\[\Big\|\sum_{n = 0}^{\infty} \chi_n u_n\Big\|^2_{\Hspace{j}} \leq C C_{\rm over} \sum_{n = 0}^{\infty} \|\chi_n u_n\|^2_{\Hspace{j}}.\] 

\begin{lemma}
	\label{l:sumNegNorms}
Given $N>0$ there exists $C_N>0$ such that the following is true. Suppose that $(\chi_n)_n$ is such that
	there exists $C_{\rm over}> 0$ such that for all $x \in \R^d$,
		\[\sum_{n = 0}^{\infty} \Big(\max_{|\alpha|\leq N}k^{-|\alpha|}|\partial^\alpha \chi_n(x)|\Big)^2 \leq C_{\rm over}.\]
%		 for all $u_n \in \Hspace{N}$, 
%	\[\Big\|\sum_{n} \chi_n u_n\Big\|_{\Hspace{N}}^2 \leq C_{\rm over} \sum_{n} \|u_n\|_{\Hspace{N}}^2.\]
Then, for all $v \in H_k^{-N}$, 
	\begin{equation}
		\label{eq:sumNegNorms}
		\sum_{n} \|\chi_n v\|^2_{H_k^{-N}} \leq C_N C_{\rm over}\N{v}^2_{H_k^{-N}}.
	\end{equation}
\end{lemma}
\begin{proof}
	For every $n$, let $\theta_n := \|\chi_n v\|_{H_k^{-N}}$, let $\varphi_n \in H_k^{N}$ with unit norm, and let 
	\[\varphi := \sum_{n} \chi_n \theta_n \varphi_n.\]
	Then by assumption, 
	\[\N{\varphi}_{\Hspace{N}}^2 \leq C_N C_{\rm over} \sum_{n} \theta_n^2.\]
	Therefore, 
	\[\N{v}_{H_k^{-N}} \geq \frac{|(v,\varphi)|^2}{\|\varphi\|^2_{H_k^{N}}} \geq \frac{\big|\sum_{n} (v,\theta_n\chi_n \varphi_n)\big|^2}{C_N C_{\rm over}\sum_{n}\theta_n^2} = \frac{\big|\sum_{n} \theta_n(\chi_n v, \varphi_n)\big|^2}{C_N C_{\rm over}\sum_{n}\theta_n^2} .
	\]
	By taking the supremum over each $\varphi_n$,
	\[\N{v}_{H_k^{-N}}  \geq \frac{1}{C_N C_{\rm over}} \frac{\Big|\sum_{n} \theta_n^2\Big|^2}{\sum_{n} \theta_n^2} =  \frac{1}{C_N C_{\rm over}}\sum_{n} \theta_n^2,\]
and the result \eqref{eq:sumNegNorms} follows.
\end{proof}

\section{Proof of the main result (Theorem \ref{t:theRealDeal})}
\label{sec:proof_realDeal}

We fix $a_k : H^1_k \times H^1_k \to \R$, $\blueOld{\mathrm{J}} \subset \R_+$, $p$, $\blueOld{V_{\mathcal{T}}}$ as in the statement of Theorem \ref{t:theRealDeal}, and keep the definitions and notations from Sections \ref{sec:state_main_result}, \ref{sec:pseudolocS}, and \ref{sec:pseudoLocPi}. We denote by $h = h(k) := \max_{\blueOld{T} \in \cT_k} h_\blueOld{T}$. 

\subsection{Outline of the proof}

%Let $\{\varphi_j\}_{j = 1}^\domainnumber$ be a partition of unity on $\overline{\Omega}$ subordinate to the cover $\{\Omega_j\}_{j = 1}^\domainnumber$, and let $\{\chi_j\}_{j = 1}^\domainnumber$ be a set of compactly supported functions on $\R^d$ such that $\supp \chi_j\subset \overline{\Omega_j}$
%and $\varphi_j\prec \chi_j$.
%, $\supp (1 - \chi_j) \cap \supp \varphi_j = \emptyset$. 
Let $X^-(\ell), X^+(\ell) \in \R^{M_\Int}, X^\Pml \in \R^{M_\Pml}$ be the column vectors defined by
\[X^{-}_i(\ell) := \| \chi_i\Psi (u-u_h)\|_{H^\ell_k}, \quad X_i^{+}(\ell) := \|\chi_i(1-\Psi)  (u-u_h)\|_{H_k^\ell}, \quad 1 \leq i \leq M_\Int,\]
\[\tand\quad X^\Pml_i(\ell) := \|\chi_{i + M_\Int}(u-u_h)\|_{H_k^\ell} \quad 1 \leq i \leq M_\Pml
\] 
%(the case $\ell = 1$ corresponds to estimates in the energy norm, while $\ell \leq 0$ gives negative norm estimates), 
and let $Z \in \R^M$ be the column vector of local best approximation errors, i.e., 
$$Z_i = \|u - w_h\|_{H^1_k(\Omega_i)}$$
where $w_h$ is an arbitrary, fixed element of $\blueOld{V_{\mathcal{T}}}$. The heart of the proof of Theorem \ref{t:theRealDeal} consists of forming a matrix system of inequalities for the vector $X(\ell) = \begin{pmatrix}
	X^-(\ell) \\ X^{+}(\ell) \\ X^\Pml(\ell)
\end{pmatrix}$. We start by obtaining this system in the lowest possible norm, which is dictated by the polynomial order $p$, i.e., with $\ell = -p+1$. In Lemma \ref{l:lastDay}, we show that 
\begin{equation}
\label{e:systemOutline}
X(-p+1) \leq \specialC \oldT X(-p+1) + C B Z + R
\end{equation}
where $R$ is a superalgebraically small remainder term, where $\oldT$ and $B$ are the matrices defined in \eqref{e:matrixT} and \eqref{e:matrixB}, and where $C_\dagger, C$ are positive constants. 
Therefore, if $(I- \specialC \oldT)^{-1}$ exists, then 
\beqs
%\label{e:system_lowest_outline}
X(-p+1) \leq \mathfrak{C}(I - \specialC \oldT)^{-1} BZ + R\,.
\eeqs
Each line of the inequality \eqref{e:systemOutline} is obtained by applying a localised version of the elliptic-projection argument, Lemma \ref{l:localDualityArg}, and exploiting both the local behaviour of the mesh size and the microlocal behaviour of the solution operator of the continuous problem from \S\ref{sec:boundsCsol} (and in particular, its improved behaviour on high-frequencies or in the PML region, leading to Lemmas \ref{l:highFreqUpgrade} and \ref{l:PMLUpgrade}).

We then use Theorem \ref{t:onlyTheSimpleOnes} to bound $(I- \specialC \oldT)^{-1}$ in terms of the simple-path matrix $\transfer$ of $\specialC \oldT$ (Definition \ref{def:simple_path_mat}), giving 
\begin{equation}
\label{e:system_lowest_outline2}
X(-p+1) \leq \mathfrak{C}\transfer BZ + R\,.
\end{equation}

Next, we upgrade \eqref{e:system_lowest_outline2} to higher norms, i.e., we estimate $X(\ell)$ for $1-p\leq \ell\leq 1$. For this, we notice that, on the one hand, since $\Psi$ is smoothing and pseudo-local,
\begin{equation}
\label{e:X-_upgrade_outline}
X^{-}(\ell) \leq C X^{-}(-p+1) + R
\end{equation}
for all $1-p\leq \ell\leq 1$. 
(In fact,  \eqref{e:X-_upgrade_outline} should actually have an $\widetilde{X}^-$ on the right-hand side involving $\widetilde{\chi}_i$ such that $\chi_i \prec \widetilde{\chi}_i$, but we have neglected this in this outline for brevity.)

On the other hand, by the ``improved'' local duality arguments of Lemmas \ref{l:highFreqUpgrade} and \ref{l:PMLUpgrade}, 
\begin{equation}
\label{e:X+_upgrade_outline}
X^{+}(\ell) \lesssim (\mathcal{H}k)^{\ell+1} Z +  (\mathcal{H}k)^{p+\ell + 1} X^{-}(-p+1) + (\Hmin(N) k^N) X^+(-p+1)+ R
\end{equation}
\begin{equation}
\label{e:Pml_upgrade_outline}
X^{\Pml}(\ell) \lesssim (\mathcal{H}k)^{\ell+1} Z + (\Hmin(N) k^N) \big(X^+(-p+1) +  X^-(-p+1)\big) + R
\end{equation}
for all $1-p\leq \ell\leq 0$. Combined with \eqref{e:system_lowest_outline2}, this gives the 
bounds in the second and third block rows of \eqref{e:mainResult}, 
up to the $L^2$ norm. Finally, to obtain \eqref{e:X+_upgrade_outline} and \eqref{e:Pml_upgrade_outline} in the $H^1_k$ norm, i.e., for $\ell=1$, we use Lemmas \ref{l:highReg} and \ref{l:PMLHighest} which give
\begin{equation}
\label{e:highReg_outline}
X^+(1) \lesssim Z + X^+(0) + (\Hdiag_{\Int,\Int}k)^pX^{-}(-p+1) + R,
\end{equation}
\begin{equation}
\label{e:PMLhighest_outlie}
X^\Pml(1) \lesssim Z + X^\Pml(0) + R.
\end{equation}
The estimates in the $H^1_k$ norm are then obtained by inserting \eqref{e:system_lowest_outline2}, \eqref{e:X+_upgrade_outline}, \eqref{e:Pml_upgrade_outline} into \eqref{e:highReg_outline} and \eqref{e:PMLhighest_outlie}.

\subsection{Localised duality argument}

%In the remainder of \S \ref{sec:proof_realDeal}, we fix a well-behaved finite element $(V_k)_{k>0}$ or order $p$.  

%From now on (i.e., in the rest of \S\ref{sec:mainResult}), we assume that all the assumptions in \S\ref{sec:assumptions} hold, with any constants in these results depending on the constants in these assumptions, namely the constants in 
%Assumptions \ref{ass:cont}--\ref{ass:ell} and Assumptions \ref{ass:ln}--\ref{ass:bn}. 

%In the proof, we use $C$ as a constant, whose value may change from line to line, that depends on the predefined quantities in the order of quantifiers.

The next result %is a central argument of the proof of Theorem \ref{t:theRealDeal}. It 
relates the Galerkin error in some region $A$ of phase-space to (i) the set of local best-approximation errors in subdomains covering $\Omega$ and (ii) the set of local Galerkin errors in these subdomains, modulo a small global term. The subdomain contributions are weighted by ``transfer coefficients'' $\eta_{j \to A}$ that describe the corresponding local behavior of the Helmholtz solution operator. This result is applied several times later in the proof of Theorem \ref{t:theRealDeal}, for special choices of the partition of unity $\{\phi_j\}$ and operators $A$.

\begin{lemma}[Localised duality argument]
	\label{l:localDualityArg}
Given $\mathfrak{c}, \,k_0>0$, there exists $h_0$ such that the following holds. Let $N > 0$ and let
$\{\phi_j\}_{j = 1}^{J}$ and $\{\tilde{\phi}_j\}_{j = 1}^J$ be such that
	\beq\label{e:supportCondition1}
		\phi_j,\tilde{\phi}_j \in C^\infty_c(\R^d,[0,1]), \quad  \quad 
\phi_j \prec_{\mathfrak{c}} \widetilde{\phi}_j,
		\qquad j = 1,\ldots, J, 
\eeq	
and $\sum_{j = 1}^J \phi_j = 1$ on $\Omega$. For any $k > 0$, 
define
	\[h_j := \max \Big\{h_\blueOld{T} \, :\, \blueOld{T} \in \cT_k \,\,\textup{ s.t. }\, \blueOld{T} \cap \supp \tilde{\phi}_j \neq \emptyset\Big\}.\]  
Then there exists $C > 0$ such that for each $\ell\in \{0,\ldots, p-1\}$, for all $k \geq k_0$, $k \notin \blueOld{\mathrm{J}}$, with $h\leq h_0$, for all $A : H_k^{-\ell} \to L^2(\Omega)$, for all $u-u_h$ satisfying \eqref{e:Galerkin_ortho}, and for all $w_{h,j} \in \blueOld{V_{\mathcal{T}}}, j=1,\ldots,J$,
	\[\begin{split}
		\|A(u-u_h)\|_{L^2} &\leq C\sum_{j = 1}^J \eta_{j \to A}(h_jk)^p  \Bigg((h_jk)^{-p}\|\tilde{\phi}_j(u - w_{h,j})\|_{H^1_k} +\|\tilde{\phi}_j (u - u_h)\|_{H_k^{-N}} \Bigg)\\
		&\qquad+ Ck^{-N}(hk)^{p+\ell + 1} \|A\|_{H_k^{-\ell} \to L^2}\left( \sum_{j=1}^J(hk)^{-p}\|u-w_{h,j}\|_{H^1_k} +\|u - u_h\|_{H_k^{-N}}\right)
	\end{split}\]
	%for all $w_h \in V_k$, 
	where for all $j \in \{1,\ldots,J\}$, 
\beq\label{eq:alphaj}
\eta_{j \to A} := (h_j k)^p \|\tilde{\phi}_j \RPs A^*\|_{L^2 \to L^2} + (h_j k)^{\ell+1} \|\tilde{\phi}_j (\RPd)^* A^*\|_{L^2 \to H_k^{\ell+2}}.
\eeq
\end{lemma}

To prove Lemma \ref{l:localDualityArg} we use the following two lemmas; the first is a localised version of the classic Aubin-Nitsche duality argument applied to the operator $P_k^\sharp$ defined in \eqref{def:aksharp}, and the second is a localised version of the bound on the adjoint-approximability constant from \cite[Theorem 1.7]{GS3} (with similar bounds appearing in \cite{MeSa:10, MeSa:11, ChNi:20, LSW3, LSW4, GLSW1, BeChMe:25}). Recall the definition of $\Pi_k^\sharp$ from Definition \ref{def:ellipticProjection}, and let 
$$\Pi_k : H^1_k \to \blueOld{V_{\mathcal{T}}}$$
be the $H^1_k$-orthogonal projection onto $\blueOld{V_{\mathcal{T}}}$. 

\begin{lemma}[Localised Aubin-Nitsche argument for $P_k^\sharp$]
	\label{l:localSharp}
	For any $\mathfrak{c} > 0$, there exists $h_0 > 0$ such that the following holds. Let $k_0 > 0$, $N > 0$, $\chi \in C^\infty(\overline{\Omega})$ and $U \subset \overline{\Omega}$ be such that 
	\[
	\supp \chi \subset U, \quad \partial_<(\supp \chi,U) \geq \mathfrak{c},
	\]
	where the notation $\partial_<$ is defined by \eqref{e:def_partial}.
	For any $k > 0$, let
	$$h_U := \max\big\{h_\blueOld{T} \,:\, \blueOld{T} \in \cT_k,\,\, \blueOld{T} \cap U \neq \emptyset\,\big\}.$$
	Then, there exists 
	$C>0$ such that for all 
	$\ell =0\ldots,p-1$, $k \geq k_0$, $h\leq h_0$, and $u \in H^1_k$, 
	\[
	\|\chi  (\Id - \Pi_k^\sharp) u\|_{H_k^{-\ell}} \leq C \left((h_U k)^{\ell+1} + k^{-N} (hk)^{\ell + 1}\right) \|(\Id - \Pi_k)u\|_{H^1_k}.
	\]
\end{lemma}

\begin{proof}
Fix $\mathfrak{c} > 0$, and let $h_0$ be such that $\mathfrak{c} \geq 2\kappa h_0$ 
where $\kappa$ is as in Assumption \ref{d:app}. Let 
$N$, $\chi$ and $U$ as in the statement, and
let $C$ denote a generic constant depending only on the previous quantities. Let $\widetilde{\chi} \in C^\infty(\overline{\Omega})$ be such that $\chi \prec_{\mathfrak{c}/2} \widetilde{\chi}$, $\supp \widetilde{\chi} \subset U$, and 
$$\partial_<(\supp \widetilde{\chi},U) \geq \mathfrak{c}/2.$$
Let $\ell \in \{0,\dots,p-1\}$, and let $v \in H_k^{\ell}$ be such that $\|v\|_{H_k^{\ell}} = 1$.

By the Definition of $\Pi_k^\sharp$ (Definition \ref{def:ellipticProjection}), for all $w_{h,1},w_{h,2} \in \blueOld{V_{\mathcal{T}}}$, letting $w_h := w_{h,1} + w_{h,2}$, 
\begin{align}\nonumber
\big|\langle v, \chi (\Id - \Pi_k^\sharp)u\rangle\big| &= \big|\langle \chi v, (\Id - \Pi_k^\sharp) u\rangle\big|\\\nonumber
 &= \big|\langle P_k^\sharp (R_k^\sharp \chi v - w_h), (\Id - \Pi_k^\sharp) u\rangle\big|\\\nonumber
& \leq C\|R_k^\sharp \chi v - w_h\|_{H^1_k} \|(\Id - \Pi_k^\sharp) u\|_{H^1_k}\\
& \leq C \left(\|\widetilde{\chi} R_k^\sharp \chi v - w_{h,1}\|_{H^1_k} + \|(1 - \widetilde{\chi})R_k^\sharp \chi v - w_{h,2}\|_{H^1_k}\right) \inf_{w_h \in \blueOld{V_{\mathcal{T}}}} \|u - w_h\|_{H^1_k},\label{e:TGV}
\end{align}
where we used Céa's lemma for the coercive operator $P_k^\sharp$ in the last step. By the approximation property of $\blueOld{V_{\mathcal{T}}}$ (Assumption \ref{d:app}), $w_{h,1},w_{h,2} \in \blueOld{V_{\mathcal{T}}}$ can be chosen such that 
$$\sum_{\blueOld{T} \in \cT_k} (h_\blueOld{T} k)^{2 - 2(\ell + 2)} \|\widetilde{\chi} R_k^\sharp \chi v - w_{h,1}\|^2_{H^1_k(\blueOld{T})} \leq C \|\widetilde{\chi} R_k^\sharp \chi v\|_{H_k^{\ell+2}}^2\,,$$
$$\sum_{\blueOld{T} \in \cT_k} (h_\blueOld{T} k)^{2 - 2(\ell + 2)} \|(1-\widetilde{\chi}) R_k^\sharp \chi v - w_{h,2}\|^2_{H^1_k(\blueOld{T})} \leq C \|(1-\widetilde{\chi})R_k^\sharp \chi v\|_{H_k^{\ell+2}}^2\,,$$
with in addition $\supp w_{h,1} \subset U$. In this case, by the definition of $h_U$ and $h$,  
\begin{equation}
\label{e:loc_ba}
\|\widetilde{\chi} R_k^\sharp \chi v - w_{h,1}\|_{H_k^{1}} \leq C (h_U k)^{\ell + 1} \|\widetilde{\chi} R_k^\sharp \chi v\|_{H_k^{\ell+2}}, \quad \tand
\end{equation}
\begin{equation}
\label{e:glob_ba}
\|(1-\widetilde{\chi}) R_k^\sharp \chi v - w_{h,1}\|_{H_k^{1}} \leq C (hk)^{\ell+ 1}\|(1-\widetilde{\chi}) R_k^\sharp \chi v\|_{H_k^{\ell+2}}.
\end{equation}
Using \eqref{e:loc_ba} and \eqref{e:glob_ba} in \eqref{e:TGV} and the estimates 
$$\|\widetilde{\chi} R_k^\sharp \chi v\|_{H_k^{\ell+2}} \leq C\|R_k^\sharp v\|_{H_k^{\ell + 2}} \leq C \|v\|_{H_k^{\ell}}\,,$$
(by the mapping properties of $R_k^\sharp$, Proposition \ref{prop:ResPksharp}) and 
$$\|(1-\widetilde{\chi}) R_k^\sharp \chi v\|_{H_k^{\ell+2}} \leq C\|R_k^\sharp v\|_{H_k^{\ell + 2}} \leq C k^{-N}\|v\|_{H_k^{\ell}}\,,$$
(by pseudo-locality of $R_k^\sharp$, Theorem \ref{t:pseudoLocGeneral}), we obtain
$$\big|\langle v, \chi (\Id - \Pi_k^\sharp) u\rangle\big| \leq C \big( (h_Uk)^{\ell+1}+ k^{-N}(hk)^{\ell+1}\big) \inf_{w_h \in \blueOld{V_{\mathcal{T}}}}\|u - w_h\|_{H^1_k}$$
and the conclusion follows by taking the supremum over $v$.
\end{proof}

\

Lemma \ref{l:localSharp} has the following special case when $\chi \equiv 1$ on $\overline{\Omega}$:
\begin{corollary}
	\label{l:AubinNitsche}
Given $k_0>0$, there exists $C>0$ such that for all $\ell\in \{0,\ldots, p-1\}$, and for all $u \in H_k^{\ell} \cap \cZ_k$,
	\[\|(\Id - \Pi_k^\sharp) u\|_{H_k^{-\ell}} \leq C(hk)^{\ell+1}\|(\Id - \Pi_k)u\|_{H^1_k}.\]
\end{corollary}

\begin{definition}[Localised adjoint-approximability constant]
\label{def:eta_microlocal}
For $A: H_k^{-\ell} \to L^2$ and $\phi \in C^\infty(\overline{\Omega})$, define 
the {\em localised adjoint-approximability constant} associated to $\phi$ and $A$ as
$$\eta(\phi \to A) := \|(\Id-\Pi_k)\phi R_k^* A^*\|_{L^2 \to H^1_k}.$$
%$$\eta(\phi\to A) := \sup_{f \in L^2} \inf_{w_h \in V_k} \frac{\|(\phi R_k^* A^*) f - w_h\|_{H^1_k}}{\|f\|_{L^2}}$$
\end{definition}

\begin{lemma}[Bound on $\eta(\phi \to A)$]
\label{l:bound_eta_microlocal}
For all $k_0 > 0$ and $\mathfrak{c}> 0$, there exists $h_0 > 0$ such that, for all $N > 0$, $\phi,\widetilde{\phi} \in C^\infty(\overline{\Omega})$ with $\phi_j \prec_{\mathfrak{c}} \tilde{\phi}$, there exists $C > 0$ such that for all $h \leq h_0$, for all $k \geq k_0$, and for all $A : H_k^{-\ell}\to L^2$,
$$\eta(\phi \to A) \leq C \Big((h_{\tilde{\phi}} k)^{p}\|\tilde{\phi}R_k^* A^*\|_{L^2 \to L^2} + (h_{\tilde{\phi}} k)^{\ell + 1} \|\phi (R_k^\sharp)^* A^*\|_{L^2 \to H_k^{\ell + 2}} + (hk)^p k^{-N}\Big),$$
where $h_{\tilde{\phi}} :=\max \Big\{h_\blueOld{T} \, :\, \blueOld{T} \in \cT_k \,\,\textup{ s.t. }\, \blueOld{T} \cap \supp \tilde{\phi} \neq \emptyset\Big\}$.
\end{lemma}

\begin{proof}
Let $k_0, \mathfrak{c} > 0$ and let $h_0 > 0$ be such that $\mathfrak{c} \geq \kappa h_0$ where $\kappa$ is as in Assumption \ref{d:app}. Let $N > 0$, $\phi, \widetilde{\phi}$ be as in the statement. Let $\check{\phi} \in C^\infty(\overline{\Omega})$ be such that $\phi \prec \check{\phi} \prec \widetilde{\phi}$. Let $C$ denote a generic positive constant depending only on the previous quantities.
%Let $\Pi_k$ be the $H^1_k$-orthogonal projection onto $V_k$. Then
%$$\eta(\phi \to A) = \|(\Id-\Pi_k)\phi R_k^* A^*\|_{L^2 \to H^1_k}.$$
Since $(P_k^\sharp)^* =  P_k^* + S_k$, applying $(R_k^\sharp)^*$ to the left, and then $R_k^*$ to the right, we obtain that
\begin{equation}
	\label{e:EuanFavForm2}
	R_k^* =   (R_k^\sharp)^* + (R_k^\sharp)^* S_k R_k^*.
\end{equation}
Thus
\begin{align}\nonumber
\eta(\phi \to A) &\leq \|(\Id - \Pi_k) \phi (R_k^{\sharp})^* A^*\|_{L^2 \to H^1_k} + \|(\Id - \Pi_k)\phi (R_k^\sharp)^*S_k R_k^* A^*\|_{L^2 \to H^1_k}\\
& \leq C (h_{\tilde{\phi}} k)^{\ell +1}\|\phi (R_k^\sharp)^* A^*\|_{L^2 \to H_k^{\ell+2}} + C(h_{\tilde{\phi}} k)^p\|\phi (R_k^\sharp)^* S_k R_k^* A^*\|_{L^2 \to H_k^{p+1}}
\label{e:split_eta}
\end{align}
by the approximation property of $\blueOld{V_{\mathcal{T}}}$ (Assumption \ref{d:app}), which can be applied since $\phi R_k^\sharp$ maps $L^2$ into $\cZ_k$ (i.e., $\phi R_k^\sharp$ satisfies a 
zero Dirichlet boundary condition on $\Gamma_{\rm tr}$ and, if necessary, also on $\partial\Omega_-$). Finally, one can use pseudolocality of $(R_k^\sharp)^*$ and $S_k$ (Theorem \ref{t:pseudoLocGeneral}) to ``move $\phi$ to the right of $S_k$'' in the second term, as follows
\begin{align}\nonumber
\phi (R_k^\sharp)^* S_k &= \phi (R_k^\sharp )^* S_k\widetilde{\phi} + \phi (R_k^\sharp)^* S_k(1 - \widetilde{\phi})\\\nonumber
& = \phi (R_k^\sharp)^* S_k\widetilde{\phi} + \phi (R_k^\sharp)^* [\check{\phi} S_k(1 - \widetilde{\phi})]
+  \phi [(R_k^\sharp)^* (1 -\check{\phi})] S_k(1 - \widetilde{\phi}) \\ \nonumber
& = \phi (R_k^\sharp )^* S_k\widetilde{\phi} + \phi (R_k^\sharp)^* O_{-\infty}(k^{-\infty};\cY_k \to \cY_k) + O_{-\infty}(k^{-\infty};\cY_k \to \cY_k) S_k(1 - \widetilde{\phi})\\
& = \phi (R_k^\sharp )^* S_k\widetilde{\phi} + O_{-\infty}(k^{-\infty};\cY_k \to \cY_k),
\label{e:moved_to_the_right}
\end{align}
using the mapping properties of $S_k$ (Proposition \ref{prop:f(Pk)}) and of $R_k^\sharp$ (Proposition \ref{prop:ResPksharp}). Inserting \eqref{e:moved_to_the_right} into \eqref{e:split_eta} and using the continuity of $R_k^\sharp$ from $H_k^{p-1}$ to $H_k^{p+1}$ (Proposition \ref{prop:ResPksharp}) and of $S_k$ from $L^2 \to H_k^{p-1}$ (Proposition \ref{prop:f(Pk)}), the result follows.
\end{proof}

\

\begin{proof}[Proof of Lemma \ref{l:localDualityArg}]
Let $\mathfrak{c},k_0 > 0$, and let $h_0 > 0$ be small enough to apply Theorem \ref{t:pseudoLocalPi}, Lemma \ref{l:localSharp} and Lemma \ref{l:bound_eta_microlocal}. Fix $\{\phi_j\}_{j = 1}^J$, $\{\tilde{\phi}_j\}_{j = 1}^J$ as in the statement, and let $N > 0$. Let $C$ denote a generic constant (whose value may change from line to line) depending only on the previous quantities. Let $k \geq k_0$ with $k \notin \blueOld{\mathrm{J}}$. By Assumption \ref{a:polyBound}, there exists $N' > 0$ such that 
\begin{equation}
\label{e:defN'}
k^{-N'}\rho(k) \leq  Ck^{-N}.
\end{equation}
	Let $v \in L^2$ with $\|v\|_{L^2} = 1$. Arguing as in \eqref{e:startDuality}, we obtain that, 	for all $w_{h,j} \in \blueOld{V_{\mathcal{T}}}$, $j=1,\ldots,J$,
		\begin{equation}
		\label{eq:dismantled}
		\begin{split}
			%\|A(u-u_h)\|_{L^2}^2 
			\big\langle A(u-u_h),v\big\rangle& = \sum_{j = 1}^J \big\langle u-w_{h,j} , (P_k^\sharp)^*  (\Id - \Pi_k^\sharp) \phi_j \RPs A^* v \big\rangle - \sum_{j = 1}^J \big\langle u-u_h ,   S_k (\Id - \Pi_k^\sharp) \phi_j \RPs A^* v \big\rangle.\end{split}
	\end{equation}
 % recalling that $S_k = P_k^\sharp - P_k = \psis(\cP_k)$.
%	where $\widetilde{S}_k = \widetilde{\psi}(\cP_k)$, with $\widetilde{\psi} \in C^\infty_c(\R)$ such that $\supp(1 - \widetilde{\psi}) \cap \supp \psi = \emptyset$ (thus ensuring that $S_k = \widetilde{S}_k S_k$). 
The plan is to use the pseudo-locality properties of $(P_k^\sharp)$, $S_k$ and $(\Id - \Pi_k^\sharp)$ shown in Sections \ref{sec:pseudolocS}-\ref{sec:pseudoLocPi}, 
to show that, up to small remainders, 
$$P_k^\sharp (\Id - \Pi_k^\sharp)\phi_j \approx \widetilde{\phi}_j  P_k^\sharp \check\phi_j (\Id - \Pi_k^\sharp)\phi_j \quad \tand\quad S_k(\Id - \Pi_k^\sharp)\phi_j \approx \widetilde{\phi}_j S_k \check\phi_j (\Id - \Pi_k^\sharp)\phi_j\,,$$
where $\check\phi_{j} \in C^\infty_c(\R^2,[0,1])$ is such that
\[
\phi_j \prec_{\mathfrak{c}/4} \check\phi_j \prec_{\mathfrak{c}/4} \widetilde{\phi}_j \quad \tfa j \in \{1,\ldots,J\}.
\]
To achieve this, we rewrite the difference as
\begin{equation}
	\label{eq:idDoublePseudoLoc}
	X(\Id - \Pi_k^\sharp)\phi_j  -\widetilde{\phi}_j X \check\phi_j (\Id - \Pi_k^\sharp)\phi_jv  = X (1-\check\phi_j)(\Id - \Pi_k^\sharp) \phi_j + (1 - \widetilde{\phi}_j) X \check\phi_j (\Id - \Pi_k^\sharp) \phi_j,
\end{equation}
where $X$ is either $(P_k^\sharp)^*$
or $S_k$.
First, when $X = S_k$, \eqref{eq:idDoublePseudoLoc} gives,
for all $w \in H_k^{\ell+2} \cap \cZ_k$, 
	\begin{align}\nonumber
			&\|S_k(\Id - \Pi_k^\sharp) \phi_jw - \widetilde{\phi}_j S_k \check\phi_j (\Id - \Pi_k^\sharp)\phi_j w\|_{H_k^N} \\ \nonumber
			& \qquad \leq \|S_k\|_{H_k^{1}\to H_k^N} \|(1-\check\phi_j)(\Id - \Pi_k^\sharp) \phi_j w\|_{H_k^{1}} +  \|(1 -\widetilde{\phi}_j) S_k \check\phi_j \|_{H_k^{-p} \to H_k^N} \|(\Id - \Pi_k^\sharp)\phi_jw\|_{H_k^{-p}} \\ \nonumber
			& \qquad\leq Ck^{-N'}\|(\Id - \Pi_k^\sharp)\phi_jw\|_{H_k^{-p}} \qquad \text{ (by Theorem \ref{t:pseudoLocalPi} and \eqref{eq:pseudoloc1} of Theorem \ref{thm:pseudolocSpace})} \\ \nonumber
			& \qquad\leq C k^{-N'} \|(\Id - \Pi_k^\sharp)\phi_jw\|_{H_k^{-p+1}}\\
			& \qquad \leq Ck^{-N'} (hk)^{p+\ell+1} \| w\|_{H_k^{\ell+2}} \qquad\text{ (by Corollary \ref{l:AubinNitsche})},
					\label{eq:sicily1}
	\end{align}
	where the condition that $w \in \cZ_k$ is need to apply Corollary \ref{l:AubinNitsche}.
In particular, taking $w = R_k^* A^* v$, \eqref{eq:sicily1} gives
	\begin{align}\nonumber
		&\Big\|\Big(S_k(\Id - \Pi_k^\sharp) \phi_j - \widetilde{\phi}_j S_k \check\phi_j (\Id - \Pi_k^\sharp)\phi_j \Big) R_k^* A^*v\Big\|_{H_k^N}
		\\ \nonumber
		&\hspace{5cm} \leq Ck^{-N'}(hk)^{p+\ell+1}\big(1+\rho(k)\big)\|A^*\|_{L^2 \to H_k^\ell}\|v\|_{L^2}\\
		&\hspace{5cm}\leq Ck^{-N}(hk)^{p+\ell+1}\|A^*\|_{L^2 \to H_k^\ell}\label{eq:aloeVera1}
	\end{align}
using the mapping properties of $R_k^*$ from Proposition \ref{prop:Rstar} and the definition of $N'$ in \eqref{e:defN'}.

Similarly, when $X = P_k^\sharp$, \eqref{eq:idDoublePseudoLoc} gives
\begin{align}\nonumber
			&\|(P_k^\sharp)^*(\Id - \Pi_k^\sharp) \phi_jw - \widetilde{\phi}_j (P_k^\sharp)^* \check\phi_j (\Id - \Pi_k^\sharp)\phi_j w\|_{H_k^{-1}} \\ \nonumber
			&\qquad \leq \|P_k^{\sharp}\|_{H^1_k\to H^{-1}_k} \|(1 - \check\phi_j) (\Id - \Pi_k^\sharp) \phi_j w\|_{H^1_k} + \|(1 - \widetilde{\phi}_j) (P_k^\sharp)^* \check\phi_j\|_{H_k^1 \to H_k^{-1}} \|(\Id - \Pi_k^\sharp) \phi_j w\|_{H^1_k} \\ \nonumber
			& \qquad \leq C k^{-N'} \|(\Id - \Pi_k^\sharp) \phi_jw\|_{H_k^{-p}} + Ck^{-N} (hk)^{\ell + 1} \|w\|_{H_k^{\ell+2}}\\ \nonumber
			&\hspace{4cm}
			\text{ (by Theorem \ref{t:pseudoLocalPi},  \eqref{eq:pseudoloc1} of Theorem \ref{thm:pseudolocFreq}, and Assumption \ref{d:app})} 
			 \\
			& \qquad \leq Ck^{-N'} \big((hk)^{p+\ell+1}+(hk)^{\ell + 1}\big)\|w\|_{H_k^{\ell+2}},\qquad\text{ (by Corollary \ref{l:AubinNitsche}).}	\label{eq:sicily2}
	\end{align}
Choosing again $w = R_k^* A^* v$ in \eqref{eq:sicily2},
 	\begin{align}\nonumber
& 		\Big\|\Big((P_k^\sharp)^*(\Id - \Pi_k^\sharp) \phi_j - \widetilde{\phi}_j (P_k^\sharp)^* \check\phi_j (\Id - \Pi_k^\sharp)\phi_j \Big)
		 R_k^* A^* v
		\Big\|_{H_k^{-1}} \\
		&
		\hspace{5cm}\leq Ck^{-N}(hk)^{\ell + 1} \|A^*\|_{L^2 \to H_k^\ell} \|v\|_{L^2}.
		 		\label{eq:aloeVera2}
 	\end{align}
Therefore, by the combination of \eqref{eq:dismantled}, \eqref{eq:aloeVera1} and \eqref{eq:aloeVera2},  
\begin{align}
\nonumber
&%\|A(u-u_h)\|^2_{L^2} 
\Big|\big \langle A(u-u_h),v\big\rangle\Big|\\
\nonumber
&\leq C\Bigg(\sum_{j = 1}^{J} \abs{\big\langle u - w_{h,j} , \widetilde{\phi}_j (P_k^\sharp)^* \check\phi_j (\Id - \Pi_k^\sharp) \phi_j R_k^* A^* v\big\rangle} +  \sum_{j = 1}^{J} \abs{\big\langle u - u_h , \widetilde{\phi}_j S_k \check\phi_j (\Id - \Pi_k^\sharp) \phi_j R_k^* A^* v \big\rangle} +  R\Bigg),\\
&=C\Bigg(\sum_{j = 1}^{J} \abs{\big\langle P_k^\sharp \widetilde{\phi}_j(u - w_{h,j}) ,  \check\phi_j (\Id - \Pi_k^\sharp) \phi_j R_k^* A^* v\big\rangle} +  \sum_{j = 1}^{J} \abs{\big\langle S_k \widetilde{\phi}_j(u - u_h) , \check\phi_j (\Id - \Pi_k^\sharp)\phi_j R_k^* A^* v \big\rangle} +  R\Bigg),
\label{ready_to_go}
\end{align}
 	%where $w_j := \phi_j R_k^{*} A^* A(u-u_h)$,
 where
\begin{equation*}
	\begin{gathered}
%		\label{e:defR}
		R := k^{-N}(hk)^{\ell + 1}\|A^*\|_{L^2\to H_k^\ell} \Big(\sum_{j=1}^J\|u - w_{h,j}\|_{H^1_k} + (hk)^{p} \|u-u_h\|_{H_k^{-N}}\Big)
	\end{gathered}
\end{equation*}
(where we have used that $\|v\|_{L^2} = 1$).	
	Since $a_k^\sharp$ is coercive, Céa's lemma implies that %\eqref{e:PiSharpCea} holds. 
	$$\|(\Id - \Pi_k^\sharp)v\|_{H^1_k} \leq C \|(\Id - \Pi_k)v\|_{H^1_k}.$$
	Therefore, for each $j \in \{1,\ldots,J\}$, 
	\begin{equation}
	\label{eq:I-Piwj}
	\|(\Id - \Pi_k^\sharp)\phi_j R_k^* A^* v\|_{H^1_k} \leq C\eta(\phi_j\to A) \|v\|_{L^2} = C \eta(\phi_j \to A),
	\end{equation}
	where $\eta(\phi_j\to A)$ is the localised adjoint-approximability constant defined in Definition \ref{def:eta_microlocal}. Similarly, by Lemma \ref{l:localSharp} with $\ell = p-1$,  
	\begin{equation}
		\label{eq:I-Piwj_localSharped}
		\|\check{\phi}_j(\Id -\Pi_k^\sharp) \phi_j R_k^* A^* v\|_{H_k^{-p+1}} \leq C \big((h_jk)^p + k^{-N'}(hk)^p\big)\eta(\phi_j\to A) ,
	\end{equation}
By \eqref{eq:I-Piwj}, \eqref{eq:I-Piwj_localSharped} and the mapping properties of $S_k$ (Proposition \ref{prop:f(Pk)}) in \eqref{ready_to_go}, 
	\begin{align*}
		\Big|\big \langle A(u-u_h),v\big\rangle\Big| &\leq C \sum_{j = 1}^J \eta(\phi_j \to A)\Big(\|\widetilde{\phi}_j(u-w_{h,j})\|_{H^1_k}
		+(h_jk)^{p}\|\widetilde{\phi}_j(u-u_h)\|_{H^{-N}_k}\Big)+ C(R + R')
	\end{align*}
	where $R' = k^{-N}(hk)^{p} (\sum_{j=1}^J\|u-w_{h,j}\|_{H^1_k} + (hk)^{p}\|u-u_h\|_{H_k^{-N}})$, using the fact that $R_k^*$, and thus $\eta(\phi_j \to A)$, are polynomially bounded on $\R_+\setminus \blueOld{\mathrm{J}}$, thanks to Assumption \ref{a:polyBound} and Proposition \ref{prop:Rstar} (while $R_k^\sharp$ is bounded by Proposition \ref{prop:ResPksharp}). The result then follows by using Lemma \ref{l:bound_eta_microlocal} to estimate the constants $\eta(\phi_j \to A)$, and taking the supremum over $v$.
\end{proof}

\subsection{Improvements at high frequency and in the PML region}

We now use the improved behavior of the resolvent on (i) high-frequency functions and (ii) functions localised in the PML region to improve Lemma \ref{l:localDualityArg}.

\ble[Improvement of $R_k^*$ on high-frequencies]\label{l:bigScreen}
%(i) 
Let $\psi\in C_c^\infty(\mathbb{R})$ satisfy $\psi^\sharp \prec \psi$ and let $\Psi := \psi(\cP_k)$ and let $\varphi\in C^\infty(\overline{\Omega})$ be such that 
$\supp \varphi \cap \Gamma_{\tr} = \emptyset$. Then 
$$
R_{k}^* (1 - \Psi) \varphi = (R_{k}^\sharp)^* (1 - \Psi)\varphi +O_{-\infty}(k^{-\infty};\ZcupHspace{}\to\ZcupHspace{}).
$$
\ele

\bpf
We use again the resolvent identity \eqref{e:EuanFavForm2} 
to write
\beqs
R_k^* (1 - \Psi)  =  (R_{k}^\sharp)^* (1 - \Psi) +R_k^* S_k (R_{k}^\sharp)^*  (1 - \Psi).
\eeqs
The idea is to now use pseudolocality of $(R_{k}^\sharp)^*$ to move $(1 - \Psi)$ next to $S_k$, with this product then zero since $\psi^\sharp(1-\psi)=0$. 
The issue is that we have only shown that $(R_k^\sharp)^*$ is pseudolocal with respect to frequency cut-offs when sandwiched by appropriate spatial cut offs -- see 
Lemma \ref{lem:huge_whiteboard1} 
 and Theorem \ref{thm:pseudolocFreq}.
		
To this end, let $\varphi_{\Pml,1},\varphi_{\Pml,2} \in C^\infty(\overline{\Omega})$ be such that
$\varphi \prec \varphi_{\Pml,1} \prec \varphi_{\Pml,2}$, and 
\beq\label{e:cutoffPML}
\supp (\varphi_{\Pml,2})\cap \Gamma_{\tr}=\emptyset\quad \tand \quad \supp(1- \varphi_{\Pml,1}) \cap \partial \Omega_- =\emptyset.
\eeq
By Lemma~\ref{t:pseudoLocGeneral} applied to both $(1-\Psi)$ and $(R_{k}^\sharp)^*$,
\begin{align*}
	R_k^* S_k (R_{k}^\sharp)^*  (1 - \Psi)\varphi
	&=R_k^* S_k (R_{k}^\sharp)^* \varphi_{\Pml,1} (1 - \Psi)\varphi +O_{-\infty}(k^{-\infty};\ZcupHspace{}\to\ZcupHspace{})\\
	&=R_k^* S_k \varphi_{\Pml,2}(R_{k}^\sharp)^*  \varphi_{\Pml,1} (1 - \Psi)\varphi +O_{-\infty}(k^{-\infty};\ZcupHspace{}\to\ZcupHspace{}),\\
	&=R_k^* S_k \varphi_{\Pml,2}(R_k^\sharp)^*\varphi_{\Pml,2}  \varphi_{\Pml,1} (1 - \Psi)\varphi +O_{-\infty}(k^{-\infty};\ZcupHspace{}\to\ZcupHspace{}),\\
	&=R_k^* S_k \big(\varphi_{\Pml,2}(R_k^\sharp)^*\varphi_{\Pml,2} \big) (1 - \Psi)\varphi +O_{-\infty}(k^{-\infty};\ZcupHspace{}\to\ZcupHspace{}).
\end{align*}
By Lemma \ref{lem:huge_whiteboard1}
(with $\varphi=\varphi_{\Pml,2}$) 
$\varphi_{\Pml,2}(R_k^\sharp)^*  \varphi_{\Pml,2}\in \cL^{\rm f}_{-2}$. Thus
by Theorem~\ref{thm:pseudolocFreq}, with $\psi^\sharp \prec \widetilde\psi\prec \psi$,  
		\begin{align*}
		R_k^* S (R_k^\sharp)^*  (1 - \Psi)\varphi
				&=R_k^* S \varphi_{\Pml,2}(R_k^\sharp)^*  \varphi_{\Pml,2}(1 - \Psi)\varphi +O_{-\infty}(k^{-\infty};\ZcupHspace{}\to\ZcupHspace{})\\
				&=R_k^* S (1 - \widetilde\Psi)\varphi_{\Pml,2}(R_k^\sharp)^*  \varphi_{\Pml,2} (1 - \Psi) \varphi +O_{-\infty}(k^{-\infty};\ZcupHspace{}\to\ZcupHspace{})\\
		&=		O_{-\infty}(k^{-\infty};\ZcupHspace{}\to\ZcupHspace{}),
		\end{align*}
		where we have used that 
		\beqs%\label{e:weirdFormat1}
		O_{-\infty}(k^{-\infty};\Dspace{}\to\Dspace{})=O_{-\infty}(k^{-\infty};\ZcupHspace{}\to\ZcupHspace{}),
		\eeqs
		 since, for any $n\in \mathbb{Z}$, $\Dspace{|n|}\subset \ZcupHspace{n}\subset \Dspace{-|n|}$ with continuous inclusions (by Corollary \ref{cor:DnsubZn}).
\epf

\begin{lemma}[High-frequency upgrade]
\label{l:highFreqUpgrade}
For any $\mathfrak{c}, \,k_0$, there exists  $h_0$ such that the following is true.
Let $N > 0$, let $\psi\in C_c^\infty(\mathbb{R})$ satisfy $\psi^\sharp \prec \psi$ and let $\Psi := \psi(\cP_k)$. Let 
$\phi,\widetilde{\phi} \in C^\infty(\overline{\Omega})$ be such that $\phi \prec_\mathfrak{c} \tilde{\phi}$ and $\supp \widetilde{\phi} \cap \Gamma_{\rm tr} = \emptyset$. 
Then there exists $C > 0$ such that 
for all $\ell\in \{0,\ldots, p-1\}$, for 
%$1\leq m\leq p+1$, 
all $k \geq k_0$, $k \notin \blueOld{\mathrm{J}}$, $h\leq h_0$ and $w_h \in \blueOld{V_{\mathcal{T}}}$,
\begin{align}\nonumber
&\|\phi(1-\Psi)(u-u_h)\|_{H_k^{-\ell}}\\ \nonumber
&\leq C(h_{\tilde\phi} k)^{p+\ell+1}\Big((h_{\tilde\phi} k)^{-p}\|\widetilde{\phi}(u-w_h)\|_{H_k^1}+\|\widetilde{\phi}\Psi(u-u_h)\|_{H_k^{-N}}+(h_{\tilde\phi}k)^N\|\widetilde{\phi}(1-\Psi)(u-u_h)\|_{H_k^{-N}}\Big)\\ \label{e:highFreqUpgrade}
&\qquad +Ck^{-N}(hk)^{p+\ell+1}\Big((hk)^{-p}\|u-w_h\|_{H_k^1}+\|u-u_h\|_{H_k^{-N}}\Big),
\end{align}
where $
	h_{\tilde\phi} := \max \big\{\diam(\blueOld{T}) \, :\, \blueOld{T} \in \cT_k \,\,\textup{ s.t. }\, \blueOld{T} \cap \supp \widetilde{\phi} \neq \emptyset\big\}.$
%	Then, for each $\ell\in \{0,\ldots, p-1\}$, for all $k \geq k_0$, $h\leq h_0$, all operators $A: H_k^{-\ell} \to L^2$ and 
\end{lemma}

We highlight that the advantage of \eqref{e:highFreqUpgrade} over the bound in Lemma \ref{l:localDualityArg} is the arbitrary power $N$ in the term $(h_{\tilde\phi}k)^N\|\widetilde{\phi}(1-\Psi)(u-u_h)\|_{H_k^{-N}}$.
%The bound \eqref{e:highFreqUpgrade} with $(h_{\tilde\varphi}k)^N$ replaced by $(h_{\tilde\varphi}k)^p$ follows immediately from Lemma \ref{l:localDualityArg}; the 

\

\begin{proof}[Proof of Lemma \ref{l:highFreqUpgrade}]
Let $\mathfrak{c} > 0$, $k_0 > 0$. Let $c \in (0,\mathfrak{c})$ be arbitrary and let $h_0$ be small enough to apply Lemma \ref{l:localDualityArg} with $\mathfrak{c} = c$. Let $N$, $\psi$, $\phi$ and $\widetilde{\phi}$ as in the statement, and let $\chi, \widetilde{\chi} \in C^\infty(\overline{\Omega})$ be any cutoff functions chosen such that $\chi \prec_c \widetilde{\chi}$ and $\supp \widetilde{\chi} \cap \Gamma_{\rm tr} = \emptyset$. Denote by $C$ any positive constant whose value depends only on the previous quantities. Then, given $k \geq k_0$, $k \notin \blueOld{\mathrm{J}}$, $h \leq h_0$ and $w_h \in \blueOld{V_{\mathcal{T}}}$,
it is enough to show that %$\AtForty$ $\NoIroning$
%\beqs
%$\partial_<( \supp {\varphi},\supp (1 - \widetilde{\varphi}) \geq \mathfrak{c}$
\begin{align}\nonumber
&\|\chi(1-\Psi)(u-u_h)\|_{H_k^{-\ell}}\\\nonumber
&\leq C(h_{\tilde\chi} k)^{\ell+1}\Big(\|\widetilde{\chi}(u-w_h)\|_{H_k^1}+(h_{\tilde\chi}k)^p\|\widetilde\chi\Psi(u-u_h)\|_{H_k^{-N}}+(h_{\tilde\chi}k)^p\|\widetilde\chi(1-\Psi)(u-u_h)\|_{H_k^{-N}}\Big)\\
&\qquad +Ck^{-N}(hk)^{\ell+1}\Big(\|u-w_h\|_{H_k^1}+(hk)^p\|u-u_h\|_{H_k^{-N}}\Big).\label{eq:sufficientInduct1}
\end{align}
where $h_{\tilde{\chi}}$ is defined analogously to $h_{\tilde{\phi}}$. Indeed, one can then apply \eqref{eq:sufficientInduct1} iteratively with a sequence of cut-offs appropriately nested between $\phi$ and $\widetilde\phi$.

Let $\check\chi,\hat\chi$ be such that 
$\chi\prec_{c/4}\check\chi\prec_{c/4} \hat\chi\prec_{c/4}\widetilde\chi$. 
We apply Lemma~\ref{l:localDualityArg} with %$\ell=\newell-1$, 
$A=E_\ell\chi(1-\Psi)$, where $E_\ell:H_k^{-\ell}\to L^2$ an isomorphism, with 
$\{\phi_j\}_{j = 1}^{2} := \{ \hat{\chi},1-\hat{\chi}\}$ 
%$\phi_1=\hat{\chi}$, $\phi_2=(1-\hat\chi)$ 
(i.e., only two functions in the partition of unity)
and $\{\tilde\phi_j\}_{j = 1}^2=\{ \widetilde\chi,1- \check\chi\}$. 
%$\tilde{\phi}_1=\widetilde\chi$, and $\tilde{\phi}_2=(1-\chi))$. 
Then, by Lemma \ref{l:bigScreen} (since $\supp\chi\cap \Gamma_\tr=\emptyset$),
$$
\|\widetilde{\chi}R_k^*A^*\|_{L^2\to L^2} \leq \|\widetilde{\chi}(R^\sharp_k)^*A^*\|_{L^2\to L^2}+C k^{-N}
$$
and 
$$
\|(1-\check\chi)R_k^*A^*\|_{L^2\to L^2}\leq\|(1-\check\chi)(R^\sharp_k)^*A^*\|_{L^2\to L^2}+Ck^{-N}.
$$
Moreover, by Theorem~\ref{t:pseudoLocGeneral}, 
\beqs
\|(1-\check\chi)(R^\sharp_k)^*A^*\|_{L^2\to H_{k}^{\ell+2}} \leq Ck^{-N}
\eeqs
and by Proposition~\ref{prop:ResPksharp}
\beqs
\|\widetilde\chi(R^\sharp_k)^*A^*\|_{L^2\to H_{k}^{\ell+2}}\leq C.
\eeqs
Therefore, with $\eta_{j \to A}$ defined by \eqref{eq:alphaj},
\beqs
\eta_{1\to A} \leq C \Big( 
(h_{\tilde\chi}k)^p\|\widetilde\chi R_k^*A^*\|_{L^2\to L^2}+(h_{\tilde\chi}k)^{\ell+1}\|\widetilde\chi(R^\sharp_k)^*A^*\|_{L^2\to H_{k}^{\ell+2}}\Big)\leq C(h_{\tilde\chi}k)^{\ell+1},
\eeqs
and $\eta_{2\to A}\leq C k^{-N}(hk)^{\ell+1}$.
Lemma \ref{l:localDualityArg} thus gives 
\begin{align*}
\|\chi(1-\Psi)(u-u_h)\|_{H_k^{-\ell}}&\leq (h_{\tilde{\chi}} k)^{\ell+1}C\Big(\|\widetilde{\chi}(u-w_h)\|_{H_k^1}+(h_{\tilde{\chi}}k)^p\|\widetilde{\chi}(u-u_h)\|_{H_k^{-N}}\Big)\\
&\qquad\qquad+ Ck^{-N}(hk)^{\ell+1}\Big(\|u-w_h\|_{H_k^1}+(hk)^{p}\|u-u_h\|_{H_k^{-N}}\Big),
\end{align*}
and \eqref{eq:sufficientInduct1} follows using $\|\widetilde{\chi}(u-u_h)\|_{H_k^{-N}} \leq  \| \widetilde{\chi}\Psi(u-u_h)\|_{H_k^{-N}} + \|\widetilde{\chi}(1 - \Psi)(u-u_h)\|_{H_k^{-N}}$.
\end{proof}

Recall from \S\ref{sec:state_main_result} that $U_{\Pml}$ is a neighbourhood of $\Gamma_{\tr}$ such that Theorem~\ref{thm:PML} holds on $U_{\Pml}$.

\begin{lemma}[PML upgrade]
\label{l:PMLUpgrade}
For any $\mathfrak{c}, \,k_0$, there exists $h_0 > 0$ such that the following is true. Let $N > 0$, let $\psi \in C^\infty_c(\R)$, let $\Psi := \psi(\cP_k)$, and let $\phi,\widetilde{\phi}\in C^\infty(\overline{\Omega})$ be such that 
$\phi \prec_{\mathfrak{c}} \widetilde{\phi}$ and $\supp \widetilde\phi \subset U_{\Pml}$. Then there exists $C > 0$ such that, for all $\ell\in \{0,\ldots, p-1\}$, 
$k \geq k_0$, $k \notin \blueOld{\mathrm{J}}$, $h\leq h_0$ and $w_h \in \blueOld{V_{\mathcal{T}}}$,
\begin{align*}
&\|\phi\Psi (u-u_h)\|_{H_k^{-\ell}} + \|\phi(1 - \Psi) (u-u_h)\|_{H_k^{-\ell}}\\
&\hspace{2cm} \leq C(h_{\tilde\phi} k)^{p+\ell+1}\Big((h_{\tilde{\phi}}k)^{-p}\|\widetilde{\phi}(u-w_h)\|_{H_k^1}+(h_{\tilde\phi}k)^N\|\widetilde{\phi}(u-u_h)\|_{H_k^{-N}}\Big)\\
&\hspace{5cm}+Ck^{-N}(hk)^{p+\ell+1}\Big((hk)^{-p}\|u-w_h\|_{H_k^1}+\|u-u_h\|_{H_k^{-N}}\Big),
\end{align*}
where $h_{\tilde\phi} := \max \Big\{\diam(\blueOld{T}) \, :\, \blueOld{T} \in \cT_k \,\,\textup{ s.t. }\, \blueOld{T} \cap \supp \widetilde{\phi} \neq \emptyset\Big\}$.
\end{lemma}

\begin{proof}
We proceed as in the proof of Lemma \ref{l:highFreqUpgrade}, with $A=E_\ell \chi \Psi$ or $A = E_\ell \chi (1 - \Psi)$, but this time, the 
terms $\eta_{j \to A}$ in Lemma \ref{l:localDualityArg} are bounded by first using pseudolocality of $\Psi$ (or $1 - \Psi$) to write
$$R_k^*\Psi \chi  = R_k^* \underline{\chi} \Psi \chi + O_{-\infty}(k^{-\infty};\cY_k \to \cY_k),$$
where $\chi \prec \underline{\chi} \prec \check{\chi}$. The conclusion is then obtained from
$$
\|(1-\check\chi) R_k^* \underline{\chi}\|_{L^2\to L^2} \leq Ck^{-N} \quad\tand\quad \|\widetilde\chi R_k^* \underline{\chi}\|_{L^2\to L^2}\leq C,
$$
with these bounds following from Theorem~\ref{thm:PML}, since $1 - \check{\chi} \perp \underline{\chi}$ and $\supp \widetilde{\chi}\subset U_{\Pml}$. 
\end{proof}

\bre\label{r:iteration}
The iteration in the proofs  of Lemmas \ref{l:highFreqUpgrade} and  \ref{l:PMLUpgrade} is possible because the $\eta_{j \to A}$ are small, precisely because of the ``good" behaviour of the solution operator on high frequencies/in the PML, respectively. 
\ere

\subsection{Estimates in lowest regularity}

In the remainder of this section, we fix a cover $\{\Omega_j\}_{1 \leq j \leq M}$ satisfying \eqref{e:domainConditions}.

\begin{lemma}[The system of inequalities involving $X$]\label{l:lastDay}
Let $\{\chi_i\}_{i = 1}^M$ 
be such that \eqref{e:crazyCover0} holds, 
let $\psi \in C^\infty_c(\R)$ with $\psis \prec \psi$, let $\Psi = \psi(\cP_k)$, and let $k_0, N>0$. Then, there exists $h_0,\,C_\dagger,C>0$ such that the following holds for all 
$k \geq k_0$, $k \notin \blueOld{\mathrm{J}}$, $h\leq h_0$, $u - u_h$ satisfying \eqref{e:Galerkin_ortho} and $w_{h,i} \in \blueOld{V_{\mathcal{T}}}$, $i\in\{1,\ldots,\domainnumber\}$. 
	 Letting $X^-, X^+$, $X^{\Pml}$ be the column vectors of local Galerkin errors defined by 
	\begin{equation}
	\begin{gathered}
		\label{e:defX}
		X^-_i= \norm{ \chibigger_{i} \Psi(u - u_h)}_{H^{-p+1}_k}, \quad {X}_i^+ = \norm{\chibigger_{i} (1 - \Psi)(u - u_h)}_{H_k^{-p+1}},\quad i=1,\dots,\domainnumber_\Int ,\\
		 X^{\Pml}_i:=\|\chibigger_{\domainnumber_\Int +i}(u-u_h)\|_{H_k^{-p+1}},\quad i=1,\dots,\domainnumber_\Pml\qquad 
%			 X := \begin{pmatrix}
%		(X^-)^t & (X^+)^t&Y^t
%		\end{pmatrix}^t
		 X := \begin{pmatrix}
		X^- \\X^+\\X^{\Pml}
		\end{pmatrix}
		\end{gathered}
\end{equation}
	(with $\pm$ standing for high and low frequency), $\matrixZ$ the column vector of local best approximation errors defined by
	\begin{equation*}
%		\label{e:defZ}
		{Z}_i = \norm{u - w_{h,i}}_{H^1_k(\Omega_i)},
	\end{equation*}
and $B$, $\oldT$ the matrices defined by \eqref{e:matrixB} and \eqref{e:matrixT}, the
following system of inequalities 
\beq\label{e:system}
\big( I - \specialC \oldT\big)X \leq 
C \big(B\matrixZ + R \,\mathbf{1}\big)
\eeq
holds in the component-wise sense, with 
$\mathbf{1}:=\begin{pmatrix}
1&\cdots&1
\end{pmatrix}^T
$
and
$R:=%k^{-N}(hk)^p
R_1+R_2$, where
\begin{align*}
R_1:=
k^{-N}(hk)^p
\sum_{i=1}^\domainnumber\|u-w_{h,i}\|_{H^1_k},
\qquad R_2:= 
k^{-N}(hk)^{2p} \|u-u_h\|_{H^{-N}_k}.
%\label{e:finalRemainder}
\end{align*}
\end{lemma}

Let $\pi_{\Int,\pm} \in \mathbb{M}((2M_\Int + M_\Pml) \times M)$ and $\pi_{\Pml} \in \mathbb{M}(M_{\Pml} \times (2M_\Int + M_\Pml))$ be defined by
\begin{equation}
\label{e:def_pi+-}
 \piminus  := \begin{pmatrix}
I_{\domainnumber_\Int } & 0_{\domainnumber_\Int } & 0_{\domainnumber_\Int \times\domainnumber_\Pml}
\end{pmatrix},\quad \piplus := 
\begin{pmatrix}
0_{\domainnumber_\Int } & I_{\domainnumber_\Int }&0_{\domainnumber_\Int \times\domainnumber_\Pml}
\end{pmatrix},
\end{equation}
\begin{equation}
\label{e:def_pipml}
\pipml=
\begin{pmatrix}
0_{\domainnumber_\Pml\times\domainnumber_\Int } & 0_{\domainnumber_\Pml\times\domainnumber_\Int }&I_{\domainnumber_\Pml}
\end{pmatrix}.
\end{equation}

Lemma \ref{l:lastDay} has the following 
%immediate 
corollary.

\begin{corollary}
	\label{c:Eurostar}
Let $\{\chi_i\}_{i = 1}^M$, $\psi$ be as in the statement of Lemma \ref{l:lastDay} and let $k_0, N > 0$. 
Then there exist $\specialC, h_0 > 0$ such that for every $M > 0$ and $C_M > 0$ there exists $C >0$ such that the following holds. For all $k \geq k_0$, $k \notin \blueOld{\mathrm{J}}$, $h \leq h_0$, if 
\beqs
%\label{e:TheCondition} 
\sum_{n=0}^\infty (\specialC\oldT)^n \leq C_M k^{M},
\eeqs
and if $u - u_h$ satisfies \eqref{e:Galerkin_ortho}, then
\[X \leq C
(I-\specialC \oldT)^{-1}
 B Z+CR_1\mathbf{1}
% \begin{pmatrix}
%1&\cdots&1
%\end{pmatrix}^T
 \]
for all $w_h \in \blueOld{V_{\mathcal{T}}}$, with $X$, $Z$,  $R_1$, and $\mathbf{1}$ defined as in Lemma \ref{l:lastDay}.

That is, for $1 \leq i\leq M_{\Int}$,
\begin{align*}
&\|\chi_i\Psi(u-u_h)\|_{H^{-p+1}_k}\leq 
C \sum_{j=1}^\domainnumber \left[\piminus 
(I-\specialC \oldT)^{-1}
 \matrixB\right]_{i,j}
\|u-w_{h,j}\|_{H_k^1(\Omega_j)}
+CR_1,
%\label{e:corLow}
\\ %\nonumber
&
\|\chi_i(1-\Psi)(u-u_h)\|_{H_k^{-p+1}}\leq 
C\sum_{j=1}^\domainnumber \left[\piplus
(I-\specialC \oldT)^{-1}
 \matrixB\right]_{i,j}
\|u-w_{h,j}\|_{H_k^1(\Omega_j)}
+CR_1, 
%\label{e:corHigh}
\end{align*}
and for $1\leq i \leq \domainnumber_\Pml$,
\begin{align*}
&\|\chi_{\domainnumber_\Int+i}(u-u_h)\|_{H_k^{-p+1}}\leq 
C\sum_{j=1}^\domainnumber \left[\pi_{\Pml}
(I-\specialC \oldT)^{-1}
 \matrixB\right]_{i,j}
\|u-w_h\|_{H_k^1(\Omega_j)}
%&\hspace{9.5cm}
+C R_1,
%\label{e:corPML}
\end{align*}
where $\pi_{\Int,\pm}$ and $\pi_{\Pml}$ are defined by \eqref{e:def_pi+-} and \eqref{e:def_pipml}.
\end{corollary}

Corollary~\ref{c:Eurostar} follows from Lemma~\ref{l:lastDay} using the following lemma.

\ble\label{l:inaugural1}
For all $k_0 > 0$, there exist constants $C,h_0,N' > 0$ such that for all $k \geq k_0$, $k \notin \blueOld{\mathrm{J}}$, $h \leq h_0$, $u-u_h$ satisfying \eqref{e:Galerkin_ortho} and $w_h \in \blueOld{V_{\mathcal{T}}}$, 
$$
\|u-u_h\|_{H_k^{-p+1}}\leq Ck^{N'}(hk)^p\|u-w_h\|_{H_k^1}.
$$
\ele
\begin{proof}
We apply Lemma \ref{l:lastDay} with any cover $\{\chi_i\}_{i = 1}^M$ satisfying \eqref{e:crazyCover0} and with $N = 2p$. 
By the definition of $X$ \eqref{e:defX} 
and the fact that for $k \notin \blueOld{\mathrm{J}}$, all the elements of $B$ \eqref{e:matrixB} are bounded by $Ck^{N'}(hk)^p$ for some $N' > 0$ (by Assumption \ref{a:polyBound})
\begin{align*}
\|u-u_h\|_{H_k^{-p+1}}\leq \sum_{i=1}^{2\domainnumber_\Int +\domainnumber_\Pml }X_i&\leq Ck^{N'}(hk)^p\|u-w_h\|_{H_k^1}+Ck^{-2p}(hk)^{2p}\|u-u_h\|_{H_k^{-p+1}}\\
& = Ck^{N'}(hk)^p\|u-w_h\|_{H_k^1}+Ch^p\|u-u_h\|_{H_k^{-p+1}};
\end{align*}
the result then follows by choosing $h_0$ small enough. 
\end{proof}

\paragraph{Outline of the proof of Lemma \ref{l:lastDay}}

The main idea of this proof -- and, indeed, the heart of the paper -- is that one can use the localised duality argument (Lemma \ref{l:localDualityArg}) to obtain a system of inequalities (as in \eqref{e:system}) relating local Galerkin errors and local best approximation errors. By choosing $A = \chi_i \Psi$ or $\chi_i(1 - \Psi)$ in Lemma \ref{l:localDualityArg}, this allows to obtain bounds for $X^-$ and $X^+$.
However, it turns out that this idea is not quite sufficient to 
fully exploit the fact that the solution operator on either the PML or high frequencies is pseudolocal (via Theorem \ref{thm:PML} and Lemma \ref{l:bigScreen} below). 
Our method is to split the domains 
$\{\Omega_i\}$ more finely, use Lemma \ref{l:localDualityArg} on this finer cover and then gather back the errors on the original domains. 
The improvements over the straightforward application of Lemma \ref{l:localDualityArg} are that, thanks to pseudolocality, we obtain instances of $h_{ij}$ instead of $h_j$, and we exploit the situations where the resolvent on $\Omega_i\cap \Omega_j$ behaves better than on $\Omega_j$.

\

\begin{proof}[Proof of Lemma \ref{l:lastDay}]
Throughout this proof, let $\chi_i$, $\psi$, $\Psi$, $k_0$ and $N$ be as in the statement. Without loss in generality, we can assume $N \geq p-1$. Denote by $C$ any positive constant, and by $h_0 > 0$ a small enough constant (to be specified in the proof), whose values only depends on the previous quantities. Now let $k \geq k_0$, $k \notin \blueOld{\mathrm{J}}$, such that $h \leq h_0$, $u-u_h$ satisfying \eqref{e:Galerkin_ortho}, and $w_{h,j} \in \blueOld{V_{\mathcal{T}}}$, $j\in \{1,\ldots,\domainnumber\}$.

\paragraph{Definition of an expanded set of domains.}

We first define an expanded set of domains $\{\widehat{\Omega}_i\}_{1 \leq i \leq \widehat{\domainnumber}}$. These appear only in the proof, and allow us to exploit the fact that the intersection of an ``interior domain" (i.e., an $\Omega_i$ for $1\leq i\leq \domainnumber_\Int$) and a ``PML domain" (i.e., an $\Omega_i$ for $\domainnumber_{\Int+1}\leq i\leq \domainnumber$) occurs only in the PML.

Let 
$$\Omega_{\Pml} := \bigcup_{j = \domainnumber_{\Int + 1}}^{\domainnumber}\Omega_j$$
be the union of all domains lying in the PML region. Recall that $\Omega_{\Pml} \Subset U_{\Pml}$ by assumption \eqref{e:domainConditions}. Thus, for each $i \in \{1,\ldots,\domainnumber_{\Int}\}$, we may find two open sets $V_{i},W_{i}$ such that 
$$ \Omega_i \cap \Omega_{\Pml} \Subset V_i \Subset W_i \Subset U_{\Pml}.$$
Let 
$$\Omega_i^\circ := \Omega_i \setminus \overline{W_i}\,, \quad \tand  \quad \Omega_i^\x := \Omega_i \cap V_i$$
(where $\Omega_i^\x$ may be empty)
and observe that 
$$\Omega_i = \Omega_i^\circ \cup \Omega_i^\x\,, \quad \Omega_i^\circ \cap \overline{\Omega_\Pml} = \emptyset\,, \quad \tand \quad \Omega_i^\x \Subset U_{\Pml}$$
(the notation $\x$ is chosen because these domains ``cross'' the PML). 
Let $\varphi_i^\circ, \tilde\varphi_i^\circ,\varphi_i^\x\in C^\infty(\overline{\Omega})$ be such that 
\beq\label{e:EuanConfused2}
\varphi_i^\circ \prec \tilde{\varphi}_i^\circ,%\qquad%\prec \tilde{\tilde{\varphi}}^\circ_i,\qquad
%\varphi_i^\x \prec \tilde{\varphi}_i^\x,
\eeq
\beq\label{e:EuanConfused3}
\chi_i \prec \varphi_i^\circ + \varphi_i^\x\,, \quad \supp(\tilde{\varphi}_i^\circ) \subset \Omega_i^\circ \cup \partial \Omega. %\quad \tand \quad \supp(\tilde{\varphi}_i^\x) \subset \Omega_i^\x.
\eeq
Let 
$$
\chi_i^\circ := \chi_i \tilde{\varphi}_i^\circ
\quad\tand\quad
\chi_i^\x := \chi_i \varphi_i^\x,
%\quad
%\tilde{\chi}_i^\circ := \tilde{\chi}_i \tilde{\tilde{\varphi}}_i^\circ, \quad\tand\quad
%\tilde{\chi}_i^\x := \tilde{\chi}_i \tilde{\varphi}_i^\x, 
$$
so that, in particular, 
%\beq\label{e:EuanGettingThere1}
%\chi_i^\circ\prec \tilde\chi_i^\circ,\quad
%\chi_i^\x\prec \tilde\chi_i^\x,
%\eeq
%and 
\beq\label{e:EuanConfused1}
\big\{ \chi_i \equiv 1\big\} \subset \big\{ \chi_i^\x \equiv 1\big\} \cup \big\{ \chi_i^\circ \equiv 1\big\}, \quad i=1,\ldots,\domainnumber_\Int.
\eeq
To see \eqref{e:EuanConfused1}, observe that if $\chi_i(x)=1$ and $\varphi_i^\x(x)\neq 1$, then $\varphi_i^\circ(x)\neq 0$ (since $\varphi_i^\circ =1 - \varphi_i^\x$ on $\supp \tilde\chi_i \supset \supp \chi$ by \eqref{e:EuanConfused3}), and thus $\tilde\varphi_i^\circ (x)=1$ by \eqref{e:EuanConfused2}.

We now renumber
$$
\chi_1^\circ,\ldots,\chi_{\domainnumber_\Int}^\circ, \chi_1^\x,\ldots,\chi_{\domainnumber_\Int}^\x, \chi_{\domainnumber_{\Int+1}},\ldots,\chi_M \quad\tas \{\varphi_{i,1}\}_{1 \leq i \leq \widehat{\domainnumber}},
$$
%as $\{\varphi_{i,1}\}_{1 \leq i \leq \widehat{\domainnumber}}$ 
with $\widehat{\domainnumber} = 2\domainnumber_{\Int} + \domainnumber_{\Pml}$, 
%$$
%\tilde\chi_1^\circ,\ldots,\tilde\chi_{\domainnumber_\Int}^\circ, \tilde\chi_1^\x,\ldots,\chi_{\domainnumber_\Int}^\x, \tilde\chi_{\domainnumber_{\Int+1}},\ldots,\tilde\chi_M \quad\tas \{\varphi_{i,2}\}_{1 \leq i \leq \widehat{\domainnumber}},
%$$
and 
$$
\Omega_1^\circ,\ldots, \Omega^\circ_{\domainnumber_\Int}, \Omega_1^\x,\ldots, \Omega^\x_{\domainnumber_\Int}, 
\Omega_{\domainnumber_{\Int+1}},\ldots,\Omega_M \quad\tas \{\widehat{\Omega}_i\}_{1 \leq i \leq \widehat{\domainnumber}}.
$$
The key properties of these domains and cutoffs that we use in the rest of the proof are that 
the condition \eqref{e:crazyCover0} still holds, i.e.
\begin{equation}
\label{e:crazyCoverTilde}
\Omega \subset \bigcup_{i = 1}^{\widehat{\domainnumber}} \textup{int}\big(\{\varphi_{i,1} \equiv 1\}\big)
\end{equation}
(by \eqref{e:EuanConfused1})
%\beqs
%\varphi_{i,1}\prec \varphi_{i,2} \quad \tfa i \in \{1,\ldots,\widehat{M}\}
%\eeqs
%(by \eqref{e:EuanGettingThere1}), 
and, for $i=1,\ldots,\domainnumber_{\Int}$,
\beq\label{e:newCutOff}
\max\big\{\varphi_{i,1}, \varphi_{i+\domainnumber_\Int,1}\big\} \leq \chi_i\leq 
%\tilde\chi_i \leq
 \varphi_{i,1} + \varphi_{i + \domainnumber_{\Int},1} \quad \tand \quad \Omega_i = \widehat{\Omega}_i \cup \widehat{\Omega}_{i + \domainnumber_\Int} 
\eeq
(by \eqref{e:EuanConfused3}).
Let $\widehat h_j$ be upper bounds for the local meshwidth on $\widehat{\Omega}_j$ and 
define $\widehat h_{ij}$ analogously to  \eqref{e:hij}.

\paragraph{Definition of suitable cut-off functions.}

Let $\{\varphi_{i,0}\}_{i = 1}^{\widehat{\domainnumber}}$ be a partition of unity subordinate to the cover \eqref{e:crazyCoverTilde} of $\Omega$, and thus such that $\varphi_{i,0}\prec \varphi_{i,1}$.

Given $\{\varphi_{i,0}\}_{i=1}^{\widehat{\domainnumber}}$
% $\{\varphi_{i,1}\}_{i=1}^{\widehat{\domainnumber}}$, 
and $\{\varphi_{i,1}\}_{i=1}^{\widehat{\domainnumber}}$,
there exists $\mathfrak{c}>0$ and  %n ``increasing"
sequences $\{\varphi_{i,\nu}\}$, $\nu=2,3,4$, of elements of $C^\infty_c(\R^d)$ supported in $\Omega_i$ such that 
$\varphi_{i,\nu} \prec_{\mathfrak{c}} \varphi_{i,\nu+1}$, for $i=1,\ldots,\widehat{\domainnumber}$ and $\nu =0,\ldots,3$ 
(the conditions involving $\prec_{\mathfrak{c}}$ are used below to apply Lemma \ref{l:localDualityArg}). Let
\beq\label{e:MartinLovesNotation1}
\varphi_{i,\nu}^\circ:= \varphi_{i,\nu}\,, \quad \varphi_{i,\nu}^\x:= \varphi_{i+M_\Int,\nu}\,, \quad \varphi_{i,\nu}^\Pml := \varphi_{i+2M_\Int}.
\eeq

%Let $N \geq p-1$ be large enough, and let $C$ denote a generic constant (whose value may change from line to line) 
%only depending on the quantities defined so far.

\paragraph{Bound on $X^-$.}

We first show that 
\begin{align}
\nonumber
X^- &\leq C \begin{pmatrix}
\mathcal{C} (\Hdiag_{\Int,\Int} k)^{2p} & \mathcal{C}  (\Hdiag_{\Int,\Int} k)^{2p} & \Hmin_{\Int,\Pml}(N) k^N 
\end{pmatrix} X \\
& \hspace{5cm} + C\begin{pmatrix}
\mathcal{C} (\Hdiag_{\Int,\Int} k)^{p} & 0%\Hmin_{\Int,\Pml}(p) k^{p} 
\end{pmatrix} Z + CR,
\label{e:mainBoundX-}
\end{align}
which gives the first block row  of \eqref{e:system} (where here, and in the rest of the proof, we use the convention that a vector plus a scalar is the vector obtained by adding the scalar to every entry). 
To do this, we estimate $X^-$ (defined by \eqref{e:defX})
$$X^- \leq X^{\circ,-} + X^{\x,-}$$
where $X^{-,\circ} := \|\chi_i^\circ \Psi(u-u_h)\|_{H_k^{-p+1}}$ and $X^{-,\x} := \|\chi_i^\x \Psi(u-u_h)\|_{H_k^{-p+1}}$.\footnote{Without this splitting, one only gets $\Hmin_{\Int,\Pml}(2p) k^{2p}$ instead $\Hmin_{\Int,\Pml}(N) k^{N}$ in the third block of the first matrix in the right-hand side of \eqref{e:mainBoundX-}.} We estimate $X^{-,\circ}$ and $X^{-,\x}$ separately. 

{\em Bound on $X^{\circ,-}$.} The main work is to bound $X^{\circ,-}$.  
%\begin{equation*}
%\begin{array}{rcl}
%X^{\circ,-} &\leq& 
%C
%\begin{pmatrix}
%\mathcal{C} (\Hdiag_{\Int,\Int} k)^{2p} & \mathcal{C} (\Hdiag_{\Int,\Int} k)^{2p} & 0
%\end{pmatrix} X 
%+ C
%\begin{pmatrix}
%\mathcal{C} (\Hdiag_{\Int,\Int} k)^{p} & 0
%\end{pmatrix} Z+ R\\
%X^{\x,-} &\leq &
%C
%\begin{pmatrix}
%\mathcal{C} \Hmin_{\Int,\Int}(N) k^{N} &\Hmin_{\Int,\Int}(N) k^{N} & \Hmin_{\Int,\Pml}(N) k^{N}
%\end{pmatrix} X 
%+ C
%\begin{pmatrix}
%\mathcal{C} (\Hdiag_{\Int,\Int} k)^{p} & 0
%\end{pmatrix} Z+ R
%\end{array}
%\end{equation*}
To this end, we fix $i \in \{1,\ldots,M_\Int\}$ and apply Lemma \ref{l:localDualityArg} with $A = A_i := \chi_i^\circ \Psi$ (observe that the smoothing property of $\Psi$, Proposition \ref{prop:f(Pk)}, implies that $A_i: H_k^{-p+1} \to L^2$) and the functions $\{\phi_j\}_{1 \leq j \leq 2\widehat{\domainnumber}}$, $\{\widetilde{\phi}_j\}_{1 \leq j \leq 2\widehat{\domainnumber}}$ defined by
\[
\phi_j := 
\begin{cases}
 \varphi_{i,3}^\circ\varphi_{j,0}, & j=1,\ldots,\widehat{\domainnumber},\\
 (1 -\varphi_{i,3}^\circ)\varphi_{j-\widehat{\domainnumber},0}, & j=\widehat{\domainnumber}+1,\ldots, 2\widehat{\domainnumber},
\end{cases}
\]
and
\[
\widetilde{\phi}_j := 
\begin{cases}
 \varphi_{i,4}^\circ\varphi_{j,1}, & j=1,\ldots,\widehat{\domainnumber},\\
 (1 -\varphi_{i,2}^\circ)\varphi_{j-\widehat{\domainnumber},1}, & j=\widehat{\domainnumber}+1,\ldots, 2\widehat{\domainnumber}.
\end{cases}
\]
With these definitions, $\big\{\phi_j\big\}_{1\leq j \leq 2\widehat{\domainnumber}}$ is indeed a partition of unity on $\Omega$ and $\big\{\widetilde{\phi}_j\big\}_{1\leq j \leq 2\widehat{\domainnumber}}$ satisfies the condition 
\eqref{e:supportCondition1} by the definition of $\varphi_{j,\nu}$. Therefore, choosing $h_0$ small enough, Lemma \ref{l:localDualityArg} ensures that
\begin{equation}
\label{eq:alphaji-alphaprimeji}
X_i^{\circ,-} \leq C\sum_{j = 1}^{\widehat{\domainnumber}}  \big[(\widehat{h}_{ij}k)^{2p} \alpha_{j \to i} + (\widehat{h}_{j}k)^{2p} \alpha'_{j \to i}\big] \widehat{X}_j + \big[(\widehat{h}_{ij}k)^{p} \alpha_{j \to i} + (\widehat{h}_{j}k)^{p} \alpha'_{j \to i}\big] \widehat{Z}_j+CR,
\end{equation}
where $R = k^{-N}\left((hk)^{2p}  \|u-u_h\|_{H^{-N}_k} + (hk)^p \sum_{j=1}^\domainnumber\|u - w_{h,j}\|_{H^1_k}\right)$, 
\begin{equation}
\label{e:defXjZjtilde}
\widehat{X}_j := \|\varphi_{j,1} (u-u_h)\|_{H^{-N}_k}\,, \quad \widehat{Z}_j := \|\varphi_{j,1} (u-w_{h,j})\|_{H^1_k},\quad j=1,\ldots,\widehat{\domainnumber},
\end{equation}
and 
$$\alpha_{j \to i} := \|\varphi_{i,4}^\circ \varphi_{j,1} R_k^* \Psi \chi_i^{\circ}\|_{L^2 \to L^2} + \|\varphi_{i,4}^\circ \varphi_{j,1} (R_k^\sharp)^* \Psi \chi_i^{\circ}\|_{L^2 \to H^{p+1}_k},$$
$$\alpha'_{j \to i} := \|(1-\varphi_{i,2}^\circ) \varphi_{j,1} R_k^* \Psi \chi_i^{\circ}\|_{L^2 \to L^2} + \|(1-\varphi_{i,2}^\circ) \varphi_{j,1} (R_k^\sharp)^* \Psi \chi_i^{\circ}\|_{L^2 \to H^{p+1}_k}.$$
Since $N \geq p-1$, by \eqref{e:defXjZjtilde}, \eqref{e:newCutOff}, and \eqref{e:defX},
\begin{align}\label{e:pinkBunny1}
\widehat{X}_j &\leq \|\varphi^\circ_{j,1} \Psi (u-u_h)\|_{H^{-p+1}_k} + \|\varphi^\circ_{j,1} (1 - \Psi) (u-u_h)\|_{H^{-p+1}_k} \leq X_j^- + X_j^+\,, \qquad 1 \leq j \leq M_{\Int},\\
\widehat{X}_{j + M_\Int} &\leq \|\varphi^\x_{j,1} \Psi (u-u_h)\|_{H^{-p+1}_k} + \|\varphi^\x_{j,1} (1 - \Psi) (u-u_h)\|_{H^{-p+1}_k}\leq X_j^- + X_j^+,\qquad 1 \leq j \leq M_{\Int},\\
\widehat{X}_{j + 2M_\Int} &\leq \|\varphi^\Pml_{j,1} (u-u_h)\|_{H^{-p+1}_k} = X_j^\Pml, \qquad 1 \leq j \leq M_{\Pml},
\end{align}
and similarly, 
\beq\label{e:pinkBunny2}
\widehat{Z}_j \leq \begin{cases}
Z_j & 1 \leq j \leq M_{\Int},\\
Z_{j - M_\Int} & M_{\Int}+1 \leq j \leq 2M_{\Int},\\
Z_{j - M_\Int} & 2M_{\Int} + 1 \leq j \leq \widehat{\domainnumber}.%M_{\Pml}.
\end{cases}
\eeq
Let
$$\omega_{j \to i} := (\widehat{h}_{ij}k)^{2p} \alpha_{j \to i} + (\widehat{h}_{j}k)^{2p} \alpha'_{j \to i}\,, \quad \beta_{j \to i} := (\widehat{h}_{ij}k)^{p} \alpha_{j \to i} + (\widehat{h}_{j}k)^{p} \alpha'_{j \to i}$$
and
$$\omega_{j \to i}^\circ := \omega_{j \to i}\,, \quad \omega_{j \to i}^\x := \omega_{j+M_\Int\to i},\quad j=1,\ldots, \domainnumber_\Int,$$ 
$$\omega_{j \to i}^{\Pml} := \omega_{j + 2M_\Int \to i},\quad j=1,\ldots, \domainnumber_\Pml,$$
and define $\beta^{\circ}_{j \to i},\beta^{\x}_{j \to i}$ and $\beta^{\Pml}_{j \to i}$ analogously. Then \eqref{eq:alphaji-alphaprimeji} can be written as 
\begin{equation}
\label{e:X0-_1}
X^{\circ,-} \leq C \begin{pmatrix}
\mathscr{A} & \mathscr{B}   & \mathscr{C} 
\end{pmatrix} X + 
\begin{pmatrix}
\mathscr{D}  & \mathscr{E} 
\end{pmatrix}Z + R
\end{equation}
where, for $1 \leq i \leq M_\Int$, 
\beq\label{e:heroic1}
\mathscr{A}_{ij} = \mathscr{B}_{ij} = \omega^\circ_{j\to i} + \omega^\x_{j\to i}\,, \quad 1 \leq j \leq M_\Int, \qquad \mathscr{C}_{ij} = \omega_{j \to i}^{\Pml}\,, \quad 1 \leq j\leq M_\Pml,
\eeq
and
\beq\label{e:heroic2}
\mathscr{D}_{ij} = \beta^\circ_{j\to i} + \beta^\x_{j\to i}\,, \quad 1 \leq j \leq M_\Int, \qquad \mathscr{E}_{ij} = \beta_{j \to i}^{\Pml}\,, \quad 1 \leq j\leq M_\Pml.
\eeq
To proceed, we now bound the following four terms appearing in the definitions of $\alpha_{j\to i}$ and $\alpha'_{j \to i}$:
\begin{equation*}
\|\varphi_{i,4} \varphi_{j,1}R_k^*\Psi \varphi_{i,1}^\circ\|_{L^2\to L^2}, 
\quad \| \varphi_{i,4} \varphi_{j,1} (R_{k}^\sharp)^* \Psi \varphi_{i,1}^\circ\|_{ L^2
			\to H_k^{p+1}},\quad 
\end{equation*}
\begin{equation*}
\|(1 -\varphi_{i,2})\varphi_{j,1}R_k^*\Psi \varphi_{i,1}^\circ\|_{L^2\to L^2},
\quad \tand\quad 
\| (1 -\varphi_{i,2})\varphi_{j,1} (R_{k}^\sharp)^* \Psi \varphi_{i,1}^\circ\|_{ L^2
			\to H_k^{p+1}},
\end{equation*}
where we have used that $\chi_i^\circ=\varphi_{i,1}=\varphi_{i,1}^\circ$ for $i=1,\ldots,\domainnumber_\Int$ by \eqref{e:MartinLovesNotation1}.

First, by pseudolocality of $\Psi$ (Lemma~\ref{t:pseudoLocGeneral}), polynomial boundedness of $R_k^*$ (Assumption \ref{a:polyBound}) and boundedness of $R_k^*$ in $U_{\Pml}$ (estimate \eqref{e:bounded_in_pml} in Theorem \ref{thm:PML}), %and boundedness of $\Psi:H_k^{-s}\to L^2$, %the estimate
	\begin{align}
%\nonumber
		\|\varphi_{i,4} \varphi_{j,1}R_k^*\Psi\varphi_{i,1}^\circ \|_{L^2\to L^2}
%		&\leq C \|1_{\Omega_j}R_k^* 1_{\Omega_i}\|_{L^2\to L^2}+Ck^{-N},\\ %O(k^{-\infty}).
\leq C1_{\{\widehat{\Omega}_j\cap \Omega_i^\circ\neq \emptyset\}} \begin{cases}
	  k^{-N} + \|1_{\Omega_j^\circ}R_k^* 1_{\Omega_i^\circ}\|_{L^2\to L^2}, &1\leq j\leq \domainnumber_\Int \\
	 1, &\domainnumber_\Int +1\leq j\leq 2M_{\Int}\\
	 0, & 2M_\Int + 1 \leq j \leq \widehat{\domainnumber},
		\end{cases}
		\label{e:box1_bis}
\end{align}
since by definition, for $j \in \{M_\Int + 1,\ldots,2M_\Int\}$, $\widehat{\Omega}_j \subset U_{\Pml}$, and for $j \in \{2M_\Int + 1,\ldots,\widehat{\domainnumber}\}$, $\widehat{\Omega}_j \subset \Omega_{\Pml}$, while $\Omega_j^\circ \cap \Omega_\Pml = \emptyset$.  

Second, by the mapping properties of $R_k^\sharp$  (Proposition \ref{prop:ResPksharp}), %and pseudolocality of $\Psi$ (Lemma~\ref{t:pseudoLocGeneral}),  
boundedness of $\Psi:L^2\to H_k^{p-1}$ (Proposition \ref{prop:f(Pk)}), and similar arguments,
\begin{equation*}
%\label{e:box2_bis}
		\|\varphi_{i,4} \varphi_{j,1}(R_k^\sharp)^*\Psi\varphi_{i,1}^\circ \|_{L^2\to L^2}
\leq C1_{\{\widehat{\Omega}_j\cap \Omega_i^\circ\neq \emptyset\}} \begin{cases}
	  1, &1\leq j\leq \domainnumber_\Int \\
	 1, &\domainnumber_\Int +1\leq j\leq 2M_{\Int}\\
	 0, & 2M_\Int + 1 \leq j \leq \widehat{\domainnumber}.
		\end{cases}
\end{equation*}

Third, by pseudolocality of $\Psi$ again, and of $R_k$ in $U_{\Pml}$ (estimate \eqref{e:pseudoloc_in_pml} of Theorem \ref{thm:PML}), and since $\varphi_{i,2}^\circ \prec \varphi_{i,3}$, 
\begin{align*}
\|(1-\varphi_{i,2}^\circ)\varphi_{j,1}R_k^*\Psi\varphi^\circ_{i,1}\|_{L^2\to L^2}
\leq Ck^{-N}+ C\begin{cases}
	 \|1_{\Omega_j^\circ} R_k^* 1_{\Omega_i^\circ}\|_{L^2\to L^2}, &1\leq j\leq \domainnumber_\Int \\
	 0,&\domainnumber_\Int +1\leq j\leq \widehat{\domainnumber}.
		\end{cases}
		\label{e:box3_bis}
\end{align*}

Finally, by pseudolocality of $(R_k^\sharp)^*$ 
and $\Psi$ (Lemma~\ref{t:pseudoLocGeneral}), since $\varphi_{i,2}^\circ \prec_c\varphi_{i,3}^\circ$, 
\begin{equation}
	\label{e:box4_bis}
	\begin{aligned}
	\| (1 - \varphi_{i,2}^\circ) \varphi_{j,1}(R_k^\sharp)^*\Psi \varphi_{i,1}^\circ\|_{ L^2 \to H_k^{p+1}} &\leq 
	%\| (1 - \varphi_{i,1})\varphi_{i,2} \varphi_{j,1}R_{P^\sharp_k}^* \varphi_{i,2}\|_{ L^2 \to L^2}+
	Ck^{-N}.\\
%		&\leq C\| 1_{\Omega_i \cap \Omega_j}R_{P^\sharp_k}^* 1_{\Omega_i}\|_{ L^2
%			\to L^2}+Ck^{-N}.
		\end{aligned}
\end{equation}
From the estimates \eqref{e:box1_bis}-\eqref{e:box4_bis}, we deduce that 
$$\omega_{j\to i}^\circ \leq C\Big((\widehat{h}_{ij}k)^{2p} 1_{\{\widehat{\Omega}_j\cap \Omega_i^\circ\neq \emptyset\}} + (\widehat{h}_j k)^{2p}\Big)\Big(k^{-N} +\|1_{\Omega^\circ_j} R_k^* 1_{\Omega_i^\circ}\|_{L^2\to L^2}\Big)  \leq C(h_jk)^{2p} \mathcal{C}_{ij} + C(hk)^{2p}k^{-N},$$
using the inclusions $\Omega_j^\circ \subset \Omega_j$, the fact that if $A' \subset A$, $B' \subset B$, then 
$$\|1_{A'}R_k^* 1_{B'}\| =\|1_{A'}1_AR_k^*1_B 1_{B'}\| \leq \|1_{A'}\| \|1_A R_k^*1_B\|\|1_{B'}\|  \leq  \|1_A R_k^*1_B\|,$$
and the fact that $1 \leq \mathcal{C}_{ij}$ when $\Omega_i^\circ \cap \Omega_j^\circ \neq \emptyset$. Similarly, 
$$\omega_{j \to i}^\x \leq 1_{\{\Omega_j^\x \cap \Omega_i^\circ \neq \emptyset\}} (\widehat{h}_{ij}k)^{2p} + C(\widehat{h}_j k)^{2p}k^{-N}\leq C (\Hmin(2p))_{ij}k^{2p} + C(hk)^{2p}k^{-N}.$$
Therefore, by \eqref{e:heroic1}, the following inequalities hold componentwise
$$\mathscr{A} \leq C \mathcal{C} (\Hdiag k)^{2p} + C(hk)^{2p}k^{-N}\,, \quad \mathscr{B} \leq C \mathcal{C} (\Hdiag k)^{2p} + C(hk)^{2p}k^{-N}.$$
One can check in a similar way that, by \eqref{e:heroic2},
$$\mathscr{D} \leq C \mathcal{C} (\Hdiag k)^{p} + C(hk)^pk^{-N}.$$
Finally, the estimates \eqref{e:box1_bis}-\eqref{e:box4_bis} also imply that, for $j \in \{1,\ldots,M_{\Pml}\}$, 
$$\omega_{j \to i}^{\Pml} \leq  C(hk)^{2p} k^{-N}\,, \qquad \beta_{j \to i}^{\Pml} \leq C(hk)^{p} k^{-N},$$
and thus, componentwise, by \eqref{e:heroic1} and \eqref{e:heroic2},
$$\mathscr{C} \leq C(hk)^{2p} k^{-N}\,, \quad \mathscr{E} \leq C(hk)^p k^{-N}.$$
Taking into account the definition of $R$, \eqref{e:X0-_1} thus yields
\begin{align}
X^{\circ,-} &\leq C \begin{pmatrix}
\mathcal{C} (\Hdiag_{\Int,\Int} k)^{2p} & \mathcal{C} (\Hdiag_{\Int,\Int} k)^{2p} & 0
\end{pmatrix} X+ C\begin{pmatrix}
\mathcal{C} (\Hdiag_{\Int,\Int} k)^{p} &0 
\end{pmatrix} Z + CR
%&\leq C \begin{pmatrix}\mathcal{C} (\Hdiag_{\Int,\Int} k)^{2p} & \mathcal{C} (\Hdiag_{\Int,\Int} k)^{2p} & 0\end{pmatrix} \Bigg(X+ \begin{pmatrix} (\Hdiag_{\Int,\Int} k)^{-p} &0\\0&0\\0&0 \end{pmatrix} Z\Bigg) + CR.
\label{e:mainBoundXo-}
\end{align}
{\em Bound on $X^{\x,-}$.} To bound $X_i^{\x,-}$ for $i \in \{1,\ldots,M_{\Int}\}$, we write
\begin{align*}
X_i^{\x,-} &= \|\varphi_{i,1}^\times \Psi (u-u_h)\|_{H_k^{-p+1}} \leq \sum_{j = 1}^{\widehat{\domainnumber}}\|\varphi_{i,1}^\times \varphi_{j,0}\Psi (u-u_h)\|_{H_k^{-p+1}}
\end{align*}
since $\{\varphi_{j,0}\}_{1 \leq j \leq \widehat{\domainnumber}}$ is a partition of unity. Applying Lemma \ref{l:PMLUpgrade} with $\ell = p-1$ to each term with $\varphi = \varphi_{i,1}^\x \varphi_{j,0}$, $\tilde{\varphi} = \varphi_{i,2}^\x \varphi_{j,1}$, we deduce that
\beq\label{e:purpleCheeky}
X_i^{\x,-} \leq C \sum_{j = 1}^{\widehat{\domainnumber}} 1_{\{\Omega_{i}^\x \cap \widehat{\Omega}_j \neq \emptyset\}} \Big((\widehat{h}_{ij}k)^N \widehat{X}_j + (\widehat{h}_{ij}k)^p \min(\widehat{Z}_j ,Z_i)\Big) + CR,
\eeq
with $\widehat{Z}_j$ and $\widehat{X}_j$ given by \eqref{e:defXjZjtilde} and $Z_i$ given by \eqref{e:defX}. Estimating $\{\widehat{X}_j\}_{1 \leq j \leq \widehat{\domainnumber}}$ and $\{\widehat{Z}_j\}_{1 \leq j \leq \widehat{\domainnumber}}$ in terms of $\{Z_j\}_{1 \leq j \leq M}$ and $\{X_j^\pm\}_{1 \leq j \leq M}$ as in \eqref{e:pinkBunny1}-\eqref{e:pinkBunny2}, we obtain
\begin{align}
\nonumber
X^{\x,-} &\leq 
C
\begin{pmatrix}
\Hmin_{\Int,\Int}(N)k^N & \Hmin_{\Int,\Int}(N)k^N  & \Hmin_{\Int,\Pml}(N)k^N 
\end{pmatrix}
X \\
&\hspace{5cm}+ 
C
\begin{pmatrix}
(\Hdiag_{\Int,\Int}k)^p & 0
\end{pmatrix}
Z
+CR,
\label{e:mainBoundXx-}
\end{align}
where, in the minimum in \eqref{e:purpleCheeky}, we always choose the $Z_i$. 
Summing the estimates \eqref{e:mainBoundXo-} and \eqref{e:mainBoundXx-} gives the claimed estimate \eqref{e:mainBoundX-} for $X^-$.

\paragraph{Bound on $X^+$.} 
We now show that 
\begin{align}
\nonumber
X^+ &\leq C \begin{pmatrix}
(\Hmin_{\Int,\Int}(2p) k^{2p} & \Hmin_{\Int,\Int}(N) k^{N} & \Hmin_{\Int,\Pml}(N) k^N 
\end{pmatrix} X \\
& \hspace{5cm} + C\begin{pmatrix}
(\Hdiag_{\Int,\Int}k)^p& 0%\Hmin_{\Int,\Pml}(p) k^{p} 
\end{pmatrix} Z + CR
\label{e:mainBoundX+}
\end{align}
which gives the second block row of \eqref{e:system}.
As before, we write 
$$X^+ \leq X^{\circ,+} + X^{\x,+}$$
where $X^{\circ,+} := \|\chi_i^\circ(1 - \Psi)(u-u_h)\|_{H_k^{-p+1}}$ and $X^{\x,+} := \|\chi_i^\x(1 - \Psi)(u-u_h)\|_{H_k^{-p+1}}$. 
Using exactly the same method as above for the bound on $X^{\x,-}$, we obtain
\begin{align}
\nonumber
X^{\x,+} &\leq 
C
\begin{pmatrix}
\Hmin_{\Int,\Int}(N)k^N & \Hmin_{\Int,\Int}(N)k^N  & \Hmin_{\Int,\Pml}(N)k^N 
\end{pmatrix}
X \\
&\hspace{5cm}
+C
\begin{pmatrix}
(\Hdiag_{\Int,\Int}k)^p & 0%\Hmin_{\Int,\Int}(p)k^p
\end{pmatrix}
Z
+CR.
\label{e:mainBoundXx+}
\end{align}
Using the same arguments, 
but applying Lemma \ref{l:highFreqUpgrade} instead of Lemma \ref{l:PMLUpgrade}, we obtain 
\begin{align}
\nonumber
X^{\circ,+} &\leq 
C
\begin{pmatrix}
\Hmin_{\Int,\Int}(2p)k^{2p} & \Hmin_{\Int,\Int}(N)k^N  & 0
\end{pmatrix}
X \\
&\hspace{5cm}+ 
C
\begin{pmatrix}
(\Hdiag_{\Int,\Int}k)^p & 0
\end{pmatrix}
Z
+CR.
\label{e:mainBoundXo+}
\end{align}
The bound \eqref{e:mainBoundX+} then follows by adding \eqref{e:mainBoundXx+} and \eqref{e:mainBoundXo+}.

\paragraph{Bound on $X^\Pml$.} Following the same method as for $X^{\x,-}$, we obtain
\begin{align}
\nonumber
X^{\Pml} &\leq 
C
\begin{pmatrix}
\Hmin_{\Pml,\Int}(N)k^{N} & \Hmin_{\Pml,\Int}(N)k^N  & \Hmin_{\Pml,\Pml}(N)k^N 
\end{pmatrix}
X \\
&\hspace{5cm}+ 
C
\begin{pmatrix}
0%\Hmin_{\Int,\Int}(p)k^p 
& (\cH_{\Pml,\Pml}k)^p
\end{pmatrix}
Z
+CR.
\label{e:mainBoundXpml}
\end{align}

Gathering the estimates \eqref{e:mainBoundX-}, \eqref{e:mainBoundX+} and \eqref{e:mainBoundXpml} and taking into account the definitions of $B$ and $W$ in \eqref{e:matrixB} and \eqref{e:matrixT}, we obtain \eqref{e:system}, which concludes the proof of the lemma.
\end{proof}

\subsection{A bound on $(I-\specialC \oldT)^{-1}$ via graph arguments}\label{s:graph}

We now state the result that allows to bound the matrix $(I - \specialC \oldT)^{-1}$ coefficientwise by the simple-path matrix (see Definition \ref{def:simple_path_mat}) of $\specialC \oldT$ in Corollary \ref{c:Eurostar}. The proof is deferred to Appendix \ref{sec:loops}. Recall the graph notation from \S \ref{sec:simple_path_mat}.

\begin{theorem}[A bound on $(I-\oldT)^{-1}$ by the simple-path matrix]
	\label{t:onlyTheSimpleOnes}
Let $M \in \mathbb{N}$, let $\oldT \in \mathbb{M}(M\times M)$ be a matrix with non-negative coefficients. % and let $c := \sum_{L \in \mathbb{SL}} W_L$. Then 
\beq\label{e:loop_condition}
\text{If } \quad c  := \sum_{L \in \mathbb{SL}} W_L< 1, \quad\text{ then } \quad
\sum_{n = 0}^{\infty}W^n < \infty,
\eeq
and
\beq\label{e:transfer}
\transfer\leq \sum_{n = 0}^{\infty} W^n \leq \frac{1}{1-c} \transfer
\eeq
in the componentwise sense, where $\transfer$ is the simple-path matrix of $W$.
\end{theorem}	

\subsection{Estimates in higher norms and completion of the proof of Theorem \ref{t:theRealDeal}}\label{s:proofRegularity}

In this section, we complete the proof of Theorem \ref{t:theRealDeal}. In view of Corollary \ref{c:Eurostar} and Theorem \ref{t:onlyTheSimpleOnes}, the main task is to obtain higher norm estimates for the Galerkin error.

We now fix $\{\chi\}_{i = 1}^M$, $k_0$, $N$ and $\psi$ as in the statement of Theorem \ref{t:theRealDeal}. For $i = 1,\ldots,M$, let $\chi_{i,\nu} \in C^\infty(\overline{\Omega})$, $\nu = 0,1,2,3$, be such that 
$$\chi_{i,0} \prec \chi_{i,1} \prec \chi_{i,2} \prec \chi_{i,3}$$
with $\chi_{i,0} = \chi_i$ and $\supp(\chi_{i,\nu}) \subset \Omega_i \cap \partial\Omega$. 
Let 
 $\specialC > 0$ be as in Corollary \ref{c:Eurostar} applied with $\{\chi_{i,2}\}_{i = 1}^M$. 
Let $h_0$ be sufficiently small, depending on $k_0$ and the cutoff functions (this restriction that $h_0$ is sufficiently small comes from the applications below of 
Corollary \ref{c:Eurostar} and Lemmas \ref{l:highFreqUpgrade}, \ref{l:PMLUpgrade},  \ref{l:inaugural1}, \ref{l:highReg}, and \ref{l:PMLHighest}).
 Let $c \in (0,1)$ and suppose that the simple loop condition \eqref{e:meshConditionsGeneral} holds. 
Let $C$ denote any positive constant whose value only depends on the previous quantities. Let $k \geq k_0$, $k \notin \blueOld{\mathrm{J}}$, $u \in H^1_k$ and $w_h \in \blueOld{V_{\mathcal{T}}}$. To show that there exists a unique $u_h \in \blueOld{V_{\mathcal{T}}}$ such that \eqref{e:Galerkin_ortho} holds, by linearity and the fact that $\blueOld{V_{\mathcal{T}}}$ is finite-dimensional, it suffices to show that when $u = 0$, the unique solution to \eqref{e:Galerkin_ortho} is $u_h = 0$. But since the latter is a consequence of \eqref{e:onlyH1}, without loss of generality, we may assume that there exists $u_h \in \blueOld{V_{\mathcal{T}}}$ satisfying \eqref{e:Galerkin_ortho} and it remains to prove the bound \eqref{e:mainResult}.

By Theorem \ref{t:onlyTheSimpleOnes}, 
$$\sum_{n =0}^{\infty} (\specialC \oldT)^n \leq \frac{1}{1-c} \transfer,$$
where $\transfer$ is the simple-path matrix of $\specialC W$. Since $\rho(k)$ is polynomially bounded on $\R_+\setminus \blueOld{\mathrm{J}}$, $\transfer\leq C k^{M}$ coefficient-wise (since by definition, the coefficients of $\transfer$ are finite linear combination of finite products of coefficients of $\specialC \oldT$). Thus, we can apply Corollary \ref{c:Eurostar} to deduce that 
for $i=1,\dots, \domainnumber_\Int $,
\begin{align}
\label{e:low_freq_check}
\|\chi_{i,2}\Psi(u-u_h)\|_{H^{-p+1}_k}&\leq 
C \sum_{j=1}^\domainnumber \left[\piminus 
\transfer
 \matrixB\right]_{i,j}
\|u-w_{h,j}\|_{H_k^1(\Omega_j)}
+CR,
\\
\label{e:high_freq_check}
\|\chi_{i,2}(1-\Psi)(u-u_h)\|_{H_k^{-p+1}}&\leq 
C\sum_{j=1}^\domainnumber \left[\piplus
\transfer
%\sum_{n=0}^\infty \big(\specialC\matrixT\big)^n
 \matrixB\right]_{i,j}
\|u-w_{h,j}\|_{H_k^1(\Omega_j)} + CR
\end{align}
and for $1\leq i \leq \domainnumber_\Pml$,
\begin{align}
\label{e:pml_check}
&\|\chi_{\domainnumber_\Int+i,2}(u-u_h)\|_{H_k^{-p+1}}\leq 
C\sum_{j=1}^\domainnumber \left[\pi_{\Pml}
\transfer
%\sum_{n=0}^\infty \big(\specialC\matrixT\big)^n
 \matrixB\right]_{i,j}
\|u-w_{h,j}\|_{H_k^1(\Omega_j)}+ CR,
\end{align}
with $R = k^{-N}(hk)^p\sum_{j=1}^\domainnumber\|u - w_{h,j}\|_{H^1_k}$. 

\subsubsection*{Low-frequency bound in arbitrary norms}%bounds for $1\leq m\leq p$.}

The first block row of  \eqref{e:mainResult}, i.e., the low-frequency bound, follows from \eqref{e:low_freq_check} by applying Lemma \ref{l:lfshift} below for each $i = 1,\ldots,M_\Int$, with $\phi=\chi_{i,0} = \chi_i$, $\tilde{\phi} := \chi_{i,2}$, $v = u-u_h$ and $N = N'$ large enough, and then using Lemma \ref{l:inaugural1} 
(with $w_h = (\sum_{j=1}^J w_{h,j})/J$) 
to estimate $k^{-N'}\|u-u_h\|_{H_k^{-N'}}$ by $R$.

\begin{lemma}[Low-frequency shift]
\label{l:lfshift}
Let $\phi, \tilde{\phi}, \in C^\infty(\overline{\Omega})$ be such that $\phi \prec \tilde{\phi}$ and $\supp \tilde{\phi} \cap \Gamma_{\tr} = \emptyset$. Let $\psi \in  C^\infty_c(\R)$ and let $\Psi := \psi(\cP_k)$. Then, for all $k_0 > 0$ and $N \in \mathbb{N}$, there exists $C > 0$ such that
$$\|\phi \Psi v\|_{H^N_k} \leq C \|\tilde{\phi} \Psi v\|_{H^{-N}_k} + Ck^{-N} \|v\|_{H_k^{-N}} \quad \tfa k \geq k_0 \tand  v \in H^{-N}_k.$$ 
\end{lemma}
\begin{proof}
As in the proof of Lemma \ref{l:bigScreen}, the assumptions let us define $\varphi_{\Pml} \in C^\infty(\overline{\Omega})$ such that (i) $\varphi_\Pml \equiv 1$ near $\partial \Omega_-$, (ii) $\varphi_\Pml \equiv 0$ near $\Gamma_{\tr}$ and (iii) $\tilde{\phi} \prec \varphi_{\Pml}$. By Lemma \ref{l:doneWithAds}, $\varphi_\Pml$ is a boundary compatible operator in the sense of Definition \ref{def:freqOp}, and thus by Theorem \ref{thm:pseudolocFreq},
$$\phi\Psi = \phi \varphi_{\Pml} \Psi = \phi \widetilde{\Psi} \varphi_{\Pml} \Psi + \phi O_{-\infty}(k^{-\infty};\cD_k \to \cD_k) = \phi \widetilde{\Psi} \varphi_{\Pml} \Psi + O_{-\infty}(k^{-\infty};\cY_k \to \cY_k)$$
where $\widetilde{\psi} \in \mathcal{S}(\R)$ is such that $\psi \prec \widetilde{\psi}$ and $\widetilde{\Psi}:=\widetilde{\psi}(\cP_k)$, and 
where the last step uses the fact that for $n \geq 0$, $\ZcupHspace{-n} \subset \Dspace{-n}$, $\Dspace{n} \subset \ZcupHspace{n}$ are continuous inclusions and $\phi \in O_0(1,\cY_k \to \cY_k)$. Hence, since $\phi \widetilde{\Psi} = \phi \widetilde{\Psi} \tilde{\phi} + O_{-\infty}(k^{-\infty};\ZcupHspace{} \to \ZcupHspace{})$ by Theorem \ref{t:pseudoLocGeneral},
$$\phi \Psi  = \phi \widetilde{\Psi} \tilde{\phi} \Psi + O_{-\infty}(k^{-\infty};\ZcupHspace{} \to \ZcupHspace{}),$$
using that $\varphi_{\Pml} \tilde{\phi} = \tilde{\phi}$. Thus, 
$$\|\phi \Psi u\|_{H_k^N} \leq \|\phi\|_{H_k^N  \to H_k^{N}} \|\widetilde{\Psi}\|_{H_k^{-N} \to H_k^{N}} \|\tilde{\phi} \Psi u\|_{H_k^{-N}} + Ck^{-N} \|u\|_{H_k^{-N}}$$
and the conclusion follows using the mapping properties of $\widetilde{\Psi}$ from Proposition \ref{prop:f(Pk)}.
\end{proof}

\subsubsection*{High-frequency and PML bounds up to the $L^2$ norm.}
The second and third block rows of \eqref{e:mainResult} 
when $m \in \{1,\ldots,p\}$ (i.e., up to the $L^2$ norm) are obtained by using Lemma \ref{l:highFreqUpgrade} and  \ref{l:PMLUpgrade}, and then using Lemma \ref{l:inaugural1} 
(again with $w_h = (\sum_{j=1}^J w_{h,j})/J$) to estimate $\|u-u_h\|_{H_k^{-N}}$ in the remainder term.

Indeed, to prove the second block row of \eqref{e:mainResult} for $m \in \{1,\ldots,p\}$, observe that from Lemma~\ref{l:highFreqUpgrade} (with $\ell=m-1$) combined with Lemma \ref{l:inaugural1} (with $w_h=w_{h,i}$), with $1\leq i\leq \domainnumber_\Int$,
\begin{align}\nonumber
\|\chi_{i,1}(1-\Psi)(u-u_h)\|_{H_k^{1-m}}&\leq C(h_i k)^m\Big(\|u-w_{h,i}\|_{H_k^1(\Omega_i)}+(h_ik)^p\|\chi_{i,2}\Psi(u-u_h)\|_{H_k^{-p+1}}\\
&\hspace{1cm}+(h_ik)^N\|\chi_{i,2}(1-\Psi)(u-u_h)\|_{H_k^{-p+1}}\Big) +Ck^{-N}(hk)^m\|u-w_{h,i}\|_{H_k^1};
\label{e:quickSand1}
\end{align}
the second block row of \eqref{e:mainResult} then follows from \eqref{e:quickSand1} and \eqref{e:high_freq_check}, using that $\chi_{i,0}\prec \chi_{i,1}$ (this extra ``layer" is used in the proof for $m=0$ below). 

The third block row of \eqref{e:mainResult}, i.e., the bound on the PML error, is proved in a similar way to the high-frequency bound, using Lemma~\ref{l:PMLUpgrade}  instead of Lemma~\ref{l:highFreqUpgrade}. Indeed, Lemma~\ref{l:PMLUpgrade} (with $\ell=m-1$ and $N$ sufficiently large) combined with 
Lemma \ref{l:inaugural1} (with $w_h=w_{h,i}$) implies that, with $\domainnumber_\Int +1\leq i\leq \domainnumber$, %$\tilde{i}:=\domainnumber_\Int +i$,
\begin{align}\nonumber
\|\chi_{i,1}(u-u_h)\|_{H_k^{1-m}}&\leq C(h_{i} k)^m\Big(\|u-w_{h,i}\|_{H_k^1(\Omega_{i})}+(h_ik)^N\|\chi_{i,2}(u-u_h)\|_{H_k^{-p+1}}\Big) \\
&\hspace{5cm}+Ck^{-N}(hk)^m\|u-w_{h,i}\|_{H_k^1}.
\label{e:quickSand2}
\end{align}
The third block row of \eqref{e:mainResult} then follows from \eqref{e:quickSand2} and \eqref{e:pml_check}, using again that $\chi_{i,0}\prec \chi_{i,1}$.

\subsubsection*{High-frequency and PML bound in the energy norm.}
%\paragraph{Proof of the bounds for $m=0$.}

The second block row of \eqref{e:mainResult}
 for $m = 0$ (i.e., in the $H^1_k$ norm) follows from \eqref{e:high_freq_check}, \eqref{e:quickSand1} with $m = 1$, and the following lemma applied with $\phi = \chi_{i,0}$, $\widetilde{\phi} := \chi_{i,1}$, $N := N'$ large enough and using Lemma \ref{l:inaugural1} 
(with $w_h=w_{h,i}$)  
to estimate $k^{-N'}\|u-u_h\|_{H_k^{-N'}}$ by $R$.

\begin{lemma}
\label{l:highReg}
For any $k_0 > 0$ and $\mathfrak{c} > 0$, there exists $h_0 > 0$ such that the following holds. Let $\phi,\widetilde{\phi} \in C^\infty(\overline{\Omega})$ be such that
$$\phi \prec_{\mathfrak{c}} \widetilde{\phi} \quad \tand \quad \supp(\widetilde{\phi}) \cap \Gamma_{\tr} = \emptyset.$$
Furthermore, let $\psi,\psi_0 \in C^\infty_c(\R)$ be such that $\psi_0 \prec \psi$, let $\Psi := \psi(\cP_k)$, $\Psi_0:= \psi_0(\cP_k)$ and
$$A:= \phi(1 - \Psi) \quad \tand \quad \widetilde{A} := \widetilde{\phi}(1 - \Psi_0).$$
Then, for all $N > 0$, there exists $C > 0$ such that for all $k \geq k_0$, $h \leq h_0$, $u-u_h$ satisfying \eqref{e:Galerkin_ortho} and for all $w_h \in \blueOld{V_{\mathcal{T}}}$,  
\begin{align*}
\|A(u-u_h)\|_{H^1_k} &\leq C \left( \|\widetilde{\phi} (u-w_h)\|_{H^1_k} + \|\widetilde{A}(u-u_h)\|_{L^2} + (h_{\tilde{\phi}}k)^p \|\widetilde{\phi}(u-u_h)\|_{H_k^{-N}}\right)\\
& \hspace{5.5cm} + Ck^{-N}(\|u -w_h\|_{H_k^1} + \|u - u_h\|_{H_k^{-N}}).
\end{align*}
where $h_{\tilde{\phi}} :=\max \big\{h_\blueOld{T} \, :\, \blueOld{T} \in \cT_k \,\,\textup{ s.t. }\, \blueOld{T} \cap \supp (\widetilde{\phi}) \neq \emptyset\big\}$.
\end{lemma}

The heart of the proof of Lemma \ref{l:highReg} is that, by the G\aa rding inequality, Galerkin orthogonality, and the definition of $\Pi^\sharp_k$, 
\begin{equation}
\label{e:luigi2_prepare}
 \begin{aligned}
&\|A(u-u_h)\|_{H_k^1}^2 \leq \Re \big\langle  P_k A(u-u_h), A(u-u_h)\big\rangle + C \| A(u-u_h)\|^2_{L^2}\\
&\leq \big| \big\langle P_k(u-u_h), (I-\Pi^\sharp_k)A^*A(u-u_h)\big\rangle\big| +\big|\big\langle [P_k,A](u-u_h),A(u-u_h)\big\rangle\big|+C\|A(u-u_h)\|^2_{L^2}.
\end{aligned}
\end{equation}
The first term on the right-hand side of \eqref{e:luigi2_prepare} is dealt with in a similar way to the proof of Lemma \ref{l:localDualityArg} (compare the first term on the right-hand side of \eqref{e:luigi2_prepare} to the right-hand side of \eqref{eq:dismantled}). The following lemma deals with the second term on the right-hand side of \eqref{e:luigi2_prepare} (i.e., the commutator term).

\begin{lemma}\label{l:commutator}
Let $\phi, \widetilde{\phi}\in C^\infty(\overline{\Omega})$ be such that $\partial_\nu \phi |_{\partial\Omega_-}=0$, $\phi \prec \widetilde{\phi}$ and $\supp (\widetilde{\phi})\cap \Gamma_{\tr}=\emptyset$, let 
$\psi,\psi_{0}\in C^\infty_c(\mathbb{R})$ be such that $\psi_0 \prec \psi$, let $\Psi:= \psi(\cP_k)$ and $\Psi_0:= \psi_0(\cP_k)$ and let $$A:= \phi (1-\Psi) \quad \tand \quad \widetilde{A}:= \widetilde{\phi} (1-\Psi_0)$$ 
Then
\beqs%\label{e:commutator0}
A = A\widetilde{A}+O_{-\infty}(k^{-\infty};\ZcupHspace{}\to \ZcupHspace{})
\eeqs
\beq\label{e:commutator1}
[P_k,A] = [P_k,A] \widetilde{A}+O_{-\infty}(k^{-\infty};\ZcupHspace{}\to \ZcupHspace{})
\eeq
and for all $k_0 > 0$, there exists $C > 0$ such that for all $k \geq k_0$,
\beq\label{e:commutator2}
\|[P_k,A]\|_{L^2\to (\cZ_k)^*}\leq  Ck^{-1}.
\eeq
\end{lemma}

\bpf 
Similar to in the proof of Lemma \ref{l:bigScreen}, let $\varphi_{\Pml} \in C^\infty(\overline{\Omega})$, 
be such that $\widetilde{\phi}\prec\varphi_{\Pml}$ and that
$$\supp (\varphi_{\Pml})\cap \Gamma_{\tr}=\emptyset\quad \tand \quad \supp(1- \varphi_{\Pml}) \cap \partial \Omega_- =\emptyset$$
(compare to \eqref{e:cutoffPML}). We claim that
\begin{align}
[P_k,A] &= [\varphi_{\Pml} P_k \varphi_{\Pml}, A]+ O_{-\infty}(k^{-\infty};\ZcupHspace{} \to \ZcupHspace{}) \label{e:PkAtosandwichA}\\
A &= A\widetilde{A} + O_{-\infty}(k^{-\infty};\ZcupHspace{} \to \ZcupHspace{}),
\label{e:AtoAAtilde}\\
A(\varphi_{\Pml} P_k \varphi_{\Pml}) &= A(\varphi_{\Pml} P_k \varphi_{\Pml})\widetilde{A} + O_{-\infty}(k^{-\infty};\ZcupHspace{} \to \ZcupHspace{}).\label{e:AsandwitchtoAsandwichAtilde}
\end{align}
Once these three properties are shown, we obtain \eqref{e:commutator1} and \eqref{e:commutator2} as follows. First,
\begin{align*}
[P_k,A] &= [\varphi_{\Pml} P_k \varphi_{\Pml},A] + O_{-\infty}(k^{-\infty};\ZcupHspace{} \to \ZcupHspace{})\qquad \text{(by \eqref{e:PkAtosandwichA})}\\
& = [\varphi_{\Pml} P_k \varphi_{\Pml},A]\widetilde{A} + O_{-\infty}(k^{-\infty};\ZcupHspace{} \to \ZcupHspace{}) \qquad 
\text{(by \eqref{e:AtoAAtilde} and \eqref{e:AsandwitchtoAsandwichAtilde})}\\
& = [P_k,A]\widetilde{A} + O_{-\infty}(k^{-\infty};\ZcupHspace{} \to \ZcupHspace{}) \qquad 
\text{(by \eqref{e:PkAtosandwichA})}
\end{align*}
which is \eqref{e:commutator1}. Second,
\begin{align}\nonumber
[P_k,A] &= [\varphi_{\Pml} P_k \varphi_{\Pml},A] + O_{-\infty}(k^{-\infty};\ZcupHspace{} \to \ZcupHspace{}) \qquad \text{(by \eqref{e:PkAtosandwichA}}\\\nonumber
&= [\varphi_{\Pml} P_k \varphi_{\Pml},\phi](1 - \Psi) + \phi [\varphi_{\Pml} P_k \varphi_{\Pml},(1 - \Psi)] + O_{-\infty}(k^{-\infty};\ZcupHspace{} \to \ZcupHspace{}) \\\nonumber
& \hspace{8cm} \textup{(by definition of $A$)} \\
& = \varphi_{\Pml}[P_k,\phi]\varphi_{\Pml} (1 - \Psi) + \phi [\varphi_{\Pml} P_k \varphi_{\Pml},(1 - \Psi)] \qquad \textup{(since $[\varphi_{\Pml},\phi] = 0$)}\label{e:twoComm}
\end{align}
By Lemma \ref{l:spatialCut}, $\phi \in \Lcut$, and by Lemma \ref{l:commuteP}, $\varphi_{\Pml} P_k \varphi_{\Pml} \in \Lf$. Thus, by the Definition of these spaces (Definitions \ref{def:spatialCutoffs} and \ref{def:freqOp}), \eqref{e:twoComm} gives \eqref{e:commutator2}.

We now prove \eqref{e:PkAtosandwichA}-\eqref{e:AsandwitchtoAsandwichAtilde}. First, by locality of $P_k$, 
\begin{equation}
\label{e:PkAsandwichA}
P_k A = P_k \phi (1 - \Psi) = \varphi_{\Pml} P_k \varphi_{\Pml} \phi (1 - \Psi) = \varphi_{\Pml} P_k \varphi_{\Pml} A.
\end{equation}
Moreover, by Theorem \ref{t:pseudoLocGeneral}, the locality of $P_k$, the fact that $\widetilde\phi= \widetilde\phi \varphi_{\Pml}$, and then Theorem \ref{t:pseudoLocGeneral} again,
\begin{align}\nonumber
A P_k = \phi (1-\Psi) P_k &= \phi(1-\Psi)\widetilde\phi P_k + O_{-\infty}(k^{-\infty};\ZcupHspace{}\to \ZcupHspace{})\\\nonumber
&=\phi(1-\Psi)\widetilde{\phi}P_k \varphi_{\Pml}
+O_{-\infty}(k^{-\infty};\ZcupHspace{}\to \ZcupHspace{})\\\nonumber
&=\phi(1-\Psi)\widetilde{\phi}\varphi_{\Pml}P_k \varphi_{\Pml}
+O_{-\infty}(k^{-\infty};\ZcupHspace{}\to \ZcupHspace{})\\\nonumber
&=\phi(1-\Psi)\varphi_{\Pml}P_k \varphi_{\Pml}
+O_{-\infty}(k^{-\infty};\ZcupHspace{}\to \ZcupHspace{}),\\
& = A \varphi_{\Pml} P_k \varphi_{\Pml}+O_{-\infty}(k^{-\infty};\ZcupHspace{}\to \ZcupHspace{})
\label{e:APkAsandwitch}.
\end{align}
Combining \eqref{e:PkAsandwichA} and \eqref{e:APkAsandwitch} gives \eqref{e:PkAtosandwichA}. Second, by the fact that  $(1-\psi)= (1-\psi)(1-\psi_0)$ and by Theorem \ref{t:pseudoLocGeneral},
\begin{align*}
A = \phi(1- \Psi) = \phi(1-\Psi) (1-{\Psi}_0) &= \phi(1-\Psi) \widetilde\phi(1-{\Psi}_0) +O_{-\infty}(k^{-\infty};\ZcupHspace{}\to \ZcupHspace{}),\\
&=A\widetilde{A} + O_{-\infty}(k^{-\infty};\ZcupHspace{}\to \ZcupHspace{}).
\end{align*}
which is \eqref{e:AtoAAtilde}. Finally, using again that $\varphi_{\Pml} P_k \varphi_{\Pml} \in \Lf$, we obtain
$$(1 - {\Psi})(\varphi_{\Pml} P_k \varphi_{\Pml}) = (1 - {\Psi})(\varphi_{\Pml} P_k \varphi_{\Pml}) (1 -{\Psi}_0) + O_{-\infty}(k^{-\infty},\ZcupHspace{} \to \ZcupHspace{})$$
by Theorem \ref{thm:pseudolocFreq}. Left-multiplying by $\phi$ thus gives
\begin{align*}
A(\varphi_{\Pml} P_k \varphi_{\Pml}) &= \phi(1 - {\Psi})(\varphi_{\Pml} P_k \varphi_{\Pml}) (1 -\Psi_0) + O_{-\infty}(k^{-\infty},\ZcupHspace{} \to \ZcupHspace{}) \\
& = \phi(1 - {\Psi})(\varphi_{\Pml} P_k \varphi_{\Pml}) \widetilde{\phi}(1 -\Psi_0) \\
&\qquad + 
\phi (1 - \Psi) (\varphi_{\Pml} P_k \varphi_{\Pml})(1 - \widetilde{\phi})(1 -\Psi_0)
+ O_{-\infty}(k^{-\infty},\ZcupHspace{} \to \ZcupHspace{}) \\
& = A(\varphi_{\Pml} P_k \varphi_{\Pml})\widetilde{A} \\
& \qquad + \phi (1 - \Psi) (\varphi_{\Pml} P_k \varphi_{\Pml})(1 - \widetilde{\phi})(1 -\Psi_0)
+ O_{-\infty}(k^{-\infty},\ZcupHspace{} \to \ZcupHspace{}),
\end{align*}
and \eqref{e:AsandwitchtoAsandwichAtilde} then follows from locality of $\varphi_{\Pml} P_k \varphi_{\Pml}$ and pseudolocality of $\Psi$ (Theorem \ref{t:pseudoLocGeneral}), since
\begin{align*}
&\phi (1 - \Psi) (\varphi_{\Pml} P_k \varphi_{\Pml})(1 - \widetilde{\phi}) \\
&\qquad = [\phi (1 - \Psi) (1 -\check{\phi})] (\varphi_{\Pml} P_k \varphi_{\Pml})(1 - \widetilde{\phi})+\phi (1 - \Psi) [\check{\phi} (\varphi_{\Pml} P_k \varphi_{\Pml})(1 - \widetilde{\phi})] \\
&\qquad = O_{-\infty}(k^{-\infty};\ZcupHspace{}\to\ZcupHspace{}) (\varphi_{\Pml} P_k \varphi_{\Pml})(1 - \widetilde{\phi}) + 0\\
&\qquad = O_{-\infty}(k^{-\infty};\ZcupHspace{} \to \ZcupHspace{})
\end{align*}
where $\check{\phi} \in C^\infty(\overline{\Omega})$ is such that $\phi \prec \check{\phi} \prec \widetilde{\phi}$. 
\epf

\

\begin{proof}[Proof of Lemma \ref{l:highReg}]
Let $k_0 > 0$ and $\mathfrak{c} > 0$, and let $h_0$ be small enough to apply Theorem \ref{t:pseudoLocalPi} and Lemma \ref{l:localSharp}. Let $\phi$, $\widetilde{\phi}$, $\psi$, $\psi_0$ and $N$ be as in the statement, and denote by $C$ any positive constant depending only on the previous quantities. Let $k \geq k_0$, suppose that $h \leq h_0$ let $u-u_h$ be such that \eqref{e:Galerkin_ortho} holds.
Let $\phi_{1},\phi_2,\phi_3\in C^\infty(\overline{\Omega})$ be such that 
\beqs
\phi \prec_{\mathfrak{c}/4} \phi_{1} \prec_{\mathfrak{c}/4} \phi_2 \prec_{\mathfrak{c}/4} \phi_3 \prec_{\mathfrak{c}/4} \widetilde{\phi}
\eeqs
with, additionally, $\partial_{\nu} (\phi_1)|_{\partial\Omega_-}=0$; such a $\phi_1$ exists by Lemma \ref{l:goodCut}. 
Since
$$
\|\phi(1-\Psi)u\|_{H_k^{1}}\leq \|\phi_1(1-\Psi)u\|_{H_k^{1}},
$$
it is enough to estimate the latter.
Let $A= \check{\phi} (1 - \Psi)$. 
By the G\aa rding inequality and Galerkin orthogonality \eqref{e:Galerkin_ortho},
\begin{equation}
\label{e:luigi2}
 \begin{aligned}
\|A(u-u_h)\|_{H_k^1}^2 &\leq \Re \big\langle  P_k A(u-u_h), A(u-u_h)\big\rangle + C \| A(u-u_h)\|^2_{L^2}\\
&\leq \big| \big\langle P_k(u-u_h),(\Id - \Pi_k^\sharp) A^*A(u-u_h)\big\rangle\big| \\
& \hspace{3cm} +\big|\big\langle  [P_k,A](u-u_h),A(u-u_h)\big\rangle\big|+C\|A(u-u_h)\|^2_{L^2}.
\end{aligned}
\end{equation}
With $\widetilde{A}$ as in the statement, Lemma \ref{l:commutator} gives 
\begin{align}\nonumber
&\big|\big\langle  [P_k,A](u-u_h),A(u-u_h)\big\rangle\big|
\\\nonumber
&\qquad  \leq \big|\big\langle  [P_k,A]\widetilde{A}(u-u_h),A(u-u_h)\big\rangle\big| + Ck^{-N} \|u-u_h\|_{H_k^{-N}} \|A(u-u_h)\|_{H^1_k}\\
& \qquad \leq C\left(k^{-1}\|\widetilde{A}(u-u_h)\|_{L^2} + Ck^{-N} \|u-u_h\|_{H_k^{-N}}\right) \|A(u-u_h)\|_{H^1_k},
\label{e:bound_commutator}
\end{align}
and
\begin{align}
\nonumber
\|A(u-u_h)\|_{L^2}^2 & \leq \|A(u-u_h)\|_{L^2} \|A(u-u_h)\|_{H^1_k}\\ 
&\leq C \left(\|\widetilde{A}(u-u_h)\|_{L^2} + Ck^{-N} \|u-u_h\|_{H_k^{-N}}\right) \|A(u-u_h)\|_{H^1_k}.
\label{e:bound_L2term}
\end{align}
Combining \eqref{e:luigi2}, \eqref{e:bound_commutator} and \eqref{e:bound_L2term}, we deduce that 
$$\|A(u-u_h)\|^{2}_{H^1_k} \leq  \big| \big\langle P_k(u-u_h),(\Id - \Pi_k^\sharp) A^*A(u-u_h)\big\rangle\big| +C \|\widetilde{A}(u-u_h)\|^2_{L^2} + Ck^{-2N} \|u-u_h\|^2_{H_k^{-N}},$$
and to conclude the proof, it remains to show that 
\begin{align}\nonumber
&\big| \big\langle P_k(u-u_h),(\Id - \Pi_k^\sharp) A^*A(u-u_h)\big\rangle\big|\\
& \hspace{3cm} \leq (\|\widetilde{\phi}(u-w_h)\|_{H^1_k} + (h_{\tilde{\phi}}k)^p \|\widetilde{\phi}(u - u_h)\|_{H_k^{-N}} + C R)\|A(u-u_h)\|_{H^1_k}
\label{e:main_term_Garding}
\end{align}
where $R := k^{-N} (\|u-u_h\|_{H_k^{-N}} + \|u-w_h\|_{H^1_k}).$

To establish \eqref{e:main_term_Garding}, we use the identity
\begin{align*}
&\big\langle P_k(u-u_h),(\Id - \Pi_k^\sharp) A^*A(u-u_h)\big\rangle \\
&\qquad =   \big\langle u-w_h,(P_k^\sharp)^*(\Id - \Pi_k^\sharp) A^*A(u-u_h)\big\rangle - \langle  u-u_h, S_k(\Id - \Pi_k^\sharp) A^*A(u-u_h)\big\rangle
\end{align*}
(shown in the same manner as \eqref{eq:dismantled} in the proof of the localised duality argument, Lemma \ref{l:localDualityArg}). Next, by pseudo-locality of $1 - \Psi$, $A^* = \phi_2 A^* + O_{-\infty}(k^{-\infty};\ZcupHspace{}\to\ZcupHspace{})$, so that
\begin{align*}
&\abs{\big\langle P_k(u-u_h),(\Id - \Pi_k^\sharp) A^*A(u-u_h)\big\rangle} \\
&\qquad \leq \abs{ \big\langle u-w_h,(P_k^\sharp)^*(\Id - \Pi_k^\sharp) \phi_2 A^*A(u-u_h)\big\rangle} +\abs{ \langle  (u-u_h), S_k(\Id - \Pi_k^\sharp) \phi_2A^*A(u-u_h)\big\rangle} \\
& \qquad \qquad + CR\|A(u-u_h)\|_{H^1_k}.
\end{align*}
Now, adapting the arguments in the proof of Lemma \ref{l:localDualityArg} (from \eqref{eq:dismantled} to \eqref{ready_to_go}), we obtain
\begin{align}
\nonumber
&\abs{\big\langle P_k(u-u_h),(\Id - \Pi_k^\sharp) A^*A(u-u_h)\big\rangle}\\
& \quad \leq \abs{\big\langle P_k^\sharp \widetilde{\phi}(u-w_h),\phi_3(\Id - \Pi_k^\sharp) w\big\rangle} + \abs{\big\langle S_k \widetilde{\phi}(u-u_h),\widetilde{S}_k \phi_3(\Id - \Pi_k^\sharp) w\big\rangle} + CR\|A(u-u_h)\|_{H^1_k},
\label{e:almost_there}
\end{align}
where $w = \phi_2 A^* A (u-u_h)$ and $\widetilde{S}_k := \widetilde{\psi}(\cP_k)$ where $\widetilde{\psi} \in C^\infty_c(\R)$ is such that $\psi \prec \widetilde{\psi}$. Namely, we follow exactly the same steps as in \eqref{eq:sicily1} and \eqref{eq:sicily2} but with $\ell = -1$, and in \eqref{eq:aloeVera1} and \eqref{eq:aloeVera2}, choosing $v = A^* A (u-u_h)$, we use the estimate $\|v\|_{H^1_k} \leq \|A^*\|_{H^1_k \to H^1_k} \|A(u-u_h)\|_{H^1}$, and the fact that $\|A^*\|_{H^1_k\to H^1_k} \leq C$. Finally, by the quasi-optimality of $\Pi_k^\sharp$ and the previous bound on $\|A^*\|_{H^1_k \to H^1_k}$, 
$$\|(\Id - \Pi_k^\sharp) w\|_{H^1_k} \leq C \|A(u-u_h)\|_{H^1_k},$$
and in turn, by Lemma \ref{l:localSharp}, 
$$\|\widetilde{S}_k \phi_3 (\Id - \Pi_k^\sharp) w\|_{L^2} \leq C\Big((h_{\tilde{\phi}}k)^p + k^{-N}(hk)^p\Big)\|A(u-u_h)\|_{H^1_k}.$$
Using these bounds in \eqref{e:almost_there}, 
\begin{align*}
&\abs{\big\langle P_k(u-u_h),(\Id - \Pi_k^\sharp) A^*A(u-u_h)\big\rangle}\\
&\qquad \leq C\left(\|\widetilde{\phi}(u-w_h)\|_{H^1_k} + \big((\widetilde{h}_{\phi}k)^p + k^{-N}(hk)^p\big)\|S_k \widetilde{\phi} (u-u_h)\|_{L^2} + CR\right)\|A(u-u_h)\|_{H^1_k}
\end{align*}
and \eqref{e:main_term_Garding} follows by using the mapping properties of $S_k$. 
\end{proof}

\

The proof of the third block row of \eqref{e:mainResult} in the $H^1_k$ norm (i.e., $m = 0$), is similar to that of the second block row, using Lemma \ref{l:PMLHighest} below instead of Lemma \ref{l:highReg}.

\begin{lemma}\label{l:PMLHighest}
For any $k_0 > 0$ and $\mathfrak{c} > 0$, there exists $h_0 > 0$ such that the following holds. Let $\phi,\widetilde{\phi} \in C^\infty(\overline{\Omega})$ be such that $\phi \prec_{\mathfrak{c}} \widetilde{\phi}$.
Then, for all $N > 0$, there exists $C > 0$ such that for all $k \geq k_0$, $h \leq h_0$, $u-u_h$ satisfying \eqref{e:Galerkin_ortho} and for all $w_h \in \blueOld{V_{\mathcal{T}}}$,  
\begin{align*}
\|\phi(u-u_h)\|_{H^1_k} &\leq C \left( \|\widetilde{\phi} (u-w_h)\|_{H^1_k} + \|\widetilde{\phi}(u-u_h)\|_{L^2}\right)\\
& \hspace{3.5cm} + Ck^{-N}\big(\|u -w_h\|_{H_k^1} + \|u - u_h\|_{H_k^{-N}}\big).
\end{align*}
where $h_{\tilde{\phi}} :=\max \Big\{h_\blueOld{T} \, :\, \blueOld{T} \in \cT_k \,\,\textup{ s.t. }\, \blueOld{T} \cap \supp (\widetilde{\phi}) \neq \emptyset\Big\}$.
\end{lemma}
\begin{proof}
Let $\phi_{1},\phi_2 \in C^\infty(\overline{\Omega})$ be such that 
$$\phi \prec \phi_1 \prec \phi_2 \prec \widetilde{\phi},$$
with in addition, $\partial_\nu (\phi_1)|_{\partial\Omega_-}=0$. Then, observe that 
$$
\|\phi (u-u_h)\|_{H_k^1}\leq C\|\phi_1(u-u_h)\|_{H_k^1}. 
$$
The proof is now identical to that of Lemma~\ref{l:highReg}, with the following replacement for Lemma \ref{l:commutator}: (i) $[P_k,\phi_1]=[P_k,\phi_1]\phi_2$, which follows from locality of $P_k$ and (ii), $\|[P_k,\phi_1]\|_{L^2 \to \cZ_k^*} \leq Ck^{-1}$ is continuous, which follows from Lemma~\ref{l:spatialCut} and Definition \ref{def:spatialCutoffs}.
\end{proof}

%The inequality \eqref{e:mainResultPML} 
%(with $\transfer$ replaced by 
%$(I-\specialC \oldT)^{-1}$) 
%for $m=0$ now follows from Lemma \ref{l:PMLHighest} and the bound \eqref{e:mainResultPML} with $m=1$.

\section{Proof of Theorem \ref{t:simple}}
\label{sec:proof_simple}

Under the assumptions on $\cT$ in Theorem \ref{t:simple}, the family $V_{\cT}^p)$ is a well-behaved finite-element \blueOld{space} of order $p$ at frequency $k$ in the sense of Definition \ref{d:wellbehaved} \blueOld{by Theorem \ref{t:itIsWellBehaved}}; we can therefore apply Theorem \ref{t:theRealDeal} (in particular, \eqref{e:onlyH1}).
By \eqref{e:transfer}, $\transfer \leq \sum_{n = 0}^{\infty} (\specialC W)^n$.
%In the proof of Theorem \ref{t:theRealDeal}, the simple-path matrix $\transfer$ was used to bound $(I - \specialC W)^{-1}$. Therefore, the bound 
%\eqref{e:onlyH1} holds with $\transfer$ replaced by $(I - \specialC W)^{-1}$ (see Remark \ref{rem:reallyLastDay}). 
To prove Theorem \ref{t:simple}, 
it is therefore sufficient to show that, provided the mesh conditions \eqref{e:meshConditions} holds, the loop condition \eqref{e:meshConditionsGeneral} holds and
\begin{equation}
\label{e:musetti}
\begin{pmatrix}
\Id & (\Hdiag_{\Int,\Int} k)^N & 0 \\
0 & 0 & (h_\pml k)^N 
\end{pmatrix} 
\sum_{n = 0}^{\infty} (\specialC W)^n
B \leq C (\mathscr{T}\mathscr{B} + \mathscr{R})
\end{equation}
where %$\transfer$ is the simple-path matrix of $\specialC W$, and 
$B, W$ are defined by \eqref{e:matrixB}, \eqref{e:matrixT}, while $\Hdiag_{\Int,\Int}$, $\mathscr{T}$ and $\mathscr{B}$ and $\mathscr{R}$ are defined by \eqref{e:matrices}. In fact, by Theorem \ref{t:onlyTheSimpleOnes}, the loop condition \eqref{e:meshConditionsGeneral} holds if and only if the sum $\sum_{n = 0}^{\infty} W^n$ converges, and from the way the simple-path matrix $\transfer$ was used in the proof of Theorem \ref{t:theRealDeal} to bound $(I - \specialC W)^{-1}$, it suffices to show \eqref{e:musetti} with $\transfer$ replaced by $\sum_{n = 0}^{\infty}W^n$. In addition, since, under the mesh conditions \eqref{e:meshConditions}, $(\Hdiag_{\Int,\Int} k)^{2pN'} \leq k^{-N'}\Id$ and $(h_\Pml k) \leq c$, it suffices to show that
\begin{equation}
\label{e:matrix_estimate}
\begin{pmatrix}
\Id & 0 & 0 \\
0 & 0 & \Id 
\end{pmatrix} \sum_{n = 0}^{\infty} (\specialC W)^n B \leq C\mathscr{T}\mathscr{B}.
\end{equation}

We obtain \eqref{e:matrix_estimate} by ``forgetting'' about the improvements on the high-frequency components of the Galerkin error. That is, we consider the directed graph $\mathcal{G}$ in Figure \ref{f:graph2} -- which describes the error propagation without any frequency splitting (where we have used the bounds on the solution operator from \S\ref{sec:boundsCsol}) -- and let 
$$\mathscr{W} := \begin{pmatrix}
\mathcal{C}_{\Int,\Int} (\Hdiag_{\Int,\Int} k)^p & h_{\min}(N) k^{N}\\
h_{\min}(N)^T k^{N} & (h_{\pml} k)^N
\end{pmatrix}$$
be the associated weighted adjacency matrix. Here, 
\beqs
h_{\min}(N) = \begin{pmatrix}
0 & h_{\visible,\pml}^N & h_{\invisible,\pml}^N & h_{\pml}^N
\end{pmatrix}^T.
\eeqs
 The point is that $\mathscr{T}$ is, up to a constant, the simple-path matrix of $\mathscr{W}$. More precisely, the following result holds.

\begin{figure}[htbp]
\begin{center}
\begin{tikzpicture}[->,>=stealth,shorten >=1pt,auto,node distance=7cm,semithick]

\begin{scope}[xscale=2,yscale=1.5]
  % Define nodes in a line
  \node[draw, circle] (A) at (0, 0) {${\Omega_\cavity}$};
  \node[draw, circle] (B) at (2, 0) {${\Omega_\visible}$};
  \node[draw, circle] (C) at (4, 0) {${\Omega_\invisible}$};
  \node[draw, circle] (D) at (6, 0) {${\Omega_{\pml}}$};

  % Define edges and add labels
  \draw[<-] (A) to[bend left] node[midway, above] {$({h_\cavity}k )^{2p}\sqrt{k {\rho}}$} (B);
  \draw[<-] (B) to[bend left] node[midway, below] {$({h_\visible} k )^{2p}\sqrt{k {\rho}}$} (A);
  \draw[<-] (B) to[bend left] node[midway, above] {$({h_\visible} k )^{2p}k $} (C);
  \draw[<-] (C) to[bend left] node[midway, below] {$({h_\invisible}k )^{2p}k $} (B);
  \draw[<-] (D) to[in=0,out=-135] (4,-1.3)node[ below] {$(h_{\visible,\pml}k )^{N}$}to[in = -45,out=180] (B);
  \draw[<-] (B) to[in=180, out=45] (4,1.3)node[ above] {$(h_{\visible,\pml}k )^{N}$}to[in=135,out=0] (D);
    \draw[<-] (C) to[bend left] node[midway, above] {\scriptsize{$(h_{\invisible,\pml}k )^{N}$}} (D);
  \draw[<-] (D) to[bend left] node[midway, below] {\scriptsize{$(h_{\invisible,\pml}k )^{N}$}} (C);
%
  % Add self-loops with labels
  \draw[<-] (A) to[loop above] node[midway, above] {$({h_\cavity}k )^{2p}{\rho}$} (A);
  \draw[<-] (B) to[loop above] node[midway, above] {$({h_\visible} k )^{2p}k $} (B);
  \draw[<-] (C) to[loop above] node[midway, above] {$({h_{I}}k )^{2p}k $} (C);
  \draw[<-] (D) to[loop above] node[midway, above] {$({h_\pml} k )^{N}$} (D);
  \end{scope}
\end{tikzpicture}
\end{center}
\caption{The graph showing propagation of errors for the decomposition into $\Omega_\cavity$, $\Omega_\visible$, $\Omega_\invisible$, \blueOld{and $\Omega_\pml$.} Recall that $h_{\visible,\pml} = \min (h_\visible,h_\pml)$ and $h_{\invisible,\pml} = \min (h_\invisible,h_\pml)$.}
\label{f:graph2} 
\end{figure}

\begin{lemma}
\label{l:what_is_t}
For all $\specialC > 0$, there exists $c,C > 0$ such that if \eqref{e:meshConditions} holds with $c$, then 
$$\sum_{n = 0}^{\infty} \specialC^n \mathscr{W}^n \leq C\mathscr{T}.$$
\end{lemma}
\begin{proof}
Observe that under \eqref{e:meshConditions}, only one edge in this graph can possibly have a weight $>c$, namely, 
	\[\mathscr{W}_{1,2} = \sqrt{k \rho(k)} (h_\visible k)^{2p}.\] 
	Moreover, any simple loop in this graph containing the edge $\mathscr{W}_{1,2}$ must also contain the edge
	\[\mathscr{W}_{2,1}  = (h_\cavity k)^{2p} \sqrt{k \rho(k)},\]
	and the product of these two weights is 
	\[\mathscr{W}_{1,2} \mathscr{W}_{2,1} = \underbrace{(h_\cavity k)^{2p}  \rho(k)}_{<c} \underbrace{(h_\visible k)^{2p}  k}_{<c} \leq c^2.\]
	Therefore, provided that $c$ is sufficiently small, the sum of the weights of all simple loops in $\mathcal{G}$ can be made $<1$. 
	%\item[(ii)] 
	The conclusion follows by remarking that, under the mesh conditions \eqref{e:meshConditions}, 
	\[\sum_{p \in \mathbb{V}_{ij}} \specialC^{|p|}\mathscr{W}_p \leq C\transferIntro_{ij},\]
	as can be checked by direct calculation.
\end{proof}
\

By Lemma \ref{l:what_is_t}, it suffices to show that for all $\specialC > 0$, there exists $\specialC' > 0$ and $C > 0$ such that 
\begin{equation}
\label{e:matrix_estimate2}
\begin{pmatrix}
\Id & 0 & 0 \\
0 & 0 & \Id 
\end{pmatrix}  (\specialC W)^n B \leq C \specialC'^n\mathscr{W}^n\mathscr{B}.
\end{equation}
To prove this, we first observe that, with $\ell = p$ or $2p$, 
$$\Hmin_{\Int,\Int}(\ell) k^{\ell} \leq \mathcal{C}_{\Int,\Int} (\Hdiag_{\Int,\Int} k)^{\ell}.$$
Indeed, $(\mathcal{C}_{\Int,\Int})_{ij} \geq 1$ for all $i,j$ such that $(\mathcal{H}^{\min}_\Int(\ell))_{ij} \neq 0$ (since when the domains overlap, the norm of the solution operator is $\geq 1$). Since $\Hdiag$ is diagonal, it follows that for such pairs $i,j$ and all $\ell \geq 0$, 
\[(\cH_{\Int,\Int}^{\min}(\ell))_{ij} = \min(h_i,h_j)^\ell \leq (\mathcal{C}_{\Int,\Int})_{ij} h_j^{\ell} = (\mathcal{C}_{\Int,\Int} (\Hdiag)^{\ell})_{ij}.\]

Therefore, the matrix $\oldT$ associated to the full graph (Figure \ref{f:twitch}) and the matrix $B$ can be estimated by blocks as
$$\oldT \leq C\begin{pmatrix}
\mathcal{C}_{\Int,\Int} (\Hdiag_{\Int,\Int} k)^{2p}& \mathcal{C}_{\Int,\Int}(\Hdiag_{\Int,\Int} k)^{2p} &h_{\min}(N) k^{N}\\
\mathcal{C}_{\Int,\Int} (\Hdiag_{\Int,\Int} k)^{2p}& \mathcal{C}_{\Int,\Int}(\Hdiag_{\Int,\Int} k)^{2p} &h_{\min}(N) k^{N}\\
h_{\min}^T(N) k^{N} & h_{\min}^T(N) k^{N}& (h_\pml k)^N\\
\end{pmatrix} = C\begin{pmatrix}
\textup{diag}(A(2p)) J & Kb(N)\\[0.2em]
b(N)^T K^T & (h_\pml k)^N
\end{pmatrix}$$ 
$$B \leq C\begin{pmatrix}
\mathcal{C}_{\Int,\Int} (\Hdiag_{\Int,\Int} k)^p& 0\\
\mathcal{C}_{\Int,\Int} (\Hdiag_{\Int,\Int} k)^{p}& 0\\
0
& (h_\pml k)^p\\
\end{pmatrix} = C\begin{pmatrix}
\textup{diag}({A(p)})K & 0%Kb(p)
\\
0%b(p)
 & (h_\pml k)^p
\end{pmatrix}$$ 
where
$$A(\ell) = \mathcal{C}_{\Int,\Int} (\Hdiag_{\Int,\Int} k)^\ell\,, \quad \textup{diag}(A) := \begin{pmatrix}
A & 0 \\
0 & A
\end{pmatrix}\,, \quad b(\ell) = h_{\min}(\ell) k^{\ell}\,, \quad J = \begin{pmatrix}
I_3 & I_3\\
I_3 & I_3
\end{pmatrix} \quad K = \begin{pmatrix}
I_3 \\
I_3
\end{pmatrix}$$ 
with $I_3$ denoting the $3\times3$ identity matrix. With these definitions, observe that
$$\mathscr{W} = \begin{pmatrix}
A(2p) &b(N)\\
b(N)^T & (h_\pml k)^N
\end{pmatrix}\,, \quad \mathscr{B} = \begin{pmatrix}
A(p) & 0%b(p)
\\
0%b(p)^T 
& (h_\pml k)^p
\end{pmatrix};$$
thus the estimate \eqref{e:matrix_estimate2} immediately follows from the next lemma.
\begin{lemma}
Let $A,B$ be $M \times M$ matrices with positive coefficients, $b,b' \in \blueOld{(0,\infty)}^M$, $c \in \blueOld{(0,\infty)}$. Then, for all $n$,
$$\begin{pmatrix}
\Id_M & 0 & 0\\
0 & 0 & \Id_M
\end{pmatrix}\begin{pmatrix}
\end{pmatrix}\begin{pmatrix}
\textup{diag}(A) J & Kb\\
b^TK^T & c
\end{pmatrix}^n \begin{pmatrix}
\textup{diag}(A')K & 0\\
0 & c'
\end{pmatrix}\leq 2^{n+1}\begin{pmatrix}
A & b\\
b^T & c
\end{pmatrix}^n \begin{pmatrix}
A' & 0\\
0 & c'
\end{pmatrix}.$$
\end{lemma}
\begin{proof}
Let
$$\begin{pmatrix}
A_n & b_n\\
b_n^T & c_n
\end{pmatrix} := \begin{pmatrix}
A & b \\
b^T & c
\end{pmatrix}^n.$$
Using that
$J^2 = 2J$, $JK = 2K$ and $K^T K = 2$, one can check by an easy induction that
$$
\begin{pmatrix}
\textup{diag}(A) J & Kb\\
b^TK^T & c
\end{pmatrix}^n \leq 2^n\begin{pmatrix}
\textup{diag}(A_n) J & Kb_n\\
b_n^TK^T & c_n.\end{pmatrix}$$
The result then follows using that
\begin{align*}
\begin{pmatrix}
\Id_M & 0 & 0\\
0 & 0 & \Id_M
\end{pmatrix}\begin{pmatrix}
\textup{diag}(A_n) J & Kb_n\\
b_n^TK^T & c_n
\end{pmatrix}\begin{pmatrix}
\textup{diag}(A')K & Kb'\\
b'^T & c'
\end{pmatrix} &= \begin{pmatrix}
A_n & A_n & b\\
b & b & c
\end{pmatrix}\begin{pmatrix}
A' & 0\\
A' & 0\\
0 & c'
\end{pmatrix}\\
& \leq 2 \begin{pmatrix}
A_n & b_n\\
b_n^T & c_n
\end{pmatrix}\begin{pmatrix}
A' & 0\\
0 & c'
\end{pmatrix}. 
\end{align*}
\end{proof}

\appendix

\section{Definition of radial perfectly matched layers}
\label{sec:PML}

Let $R_{\rm scat}$ be such that
\beq\label{e:Rscat}
\supp(A-I) \cup \supp(n-1) \cup \Omega\Subset \blueOld{B(0,R_{\rm scat})
:=\big\{ x: |x|< R_{\rm scat}\big\}
}.
\eeq
Let $\RPMLo>R_{\rm scat}$ be such that $\blueOld{B(0,\RPMLo) }\Subset \Omega_{\tr}$.

As in \S\ref{sec:1.1} and \S\ref{sec:3.1}, let $\Omega:=\Omega_{+}\cap \Omega_{\tr}$ and $\Gamma_{\tr}:=\partial\Omega_{\tr}$.
For $0\leq \theta<\pi/2$, let the PML scaling function $f_\theta\in C^{\infty}([0,\infty);\mathbb{R})$ be defined by $f_\theta(r):=f(r)\tan\theta$ for some $f$ satisfying
\begin{equation}
\label{e:fProp}
\begin{gathered}
\big\{f(r)=0\big\}=\big\{f'(r)=0\big\}=\big\{r\leq \RPMLo\big\},\quad f'(r)\geq 0;
\end{gathered}
\end{equation}
i.e., the scaling ``turns on'' at $r=\RPMLo$. 
Given $f_\theta(r)$,  let 
\beqs
\alpha(r) := 1 + \ri f_\theta'(r) \quad \tand\quad \beta(r) := 1 + \ri f_\theta(r)/r.
\eeqs
We now define two possible PML problems \eqref{e:PML1}; both are formed by first replacing $\Delta$ in \eqref{e:edp} by 
\begin{align*}
\Delta_\theta&= \Big(\frac{1}{1+\ri f_\theta'(r)}\pdiff{}{r}\Big)^2+\frac{d-1}{(r+\ri f_\theta(r))(1+\ri f'_\theta(r))}\pdiff{}{r} +\frac{1}{(r+\ri f_\theta(r))^2}\Delta_\omega,\\
&= \frac{1}{(1+\ri f_\theta'(r))(r+\ri f_\theta(r))^{d-1}}\frac{\partial}{\partial r}
\left( \frac{ (r+\ri f_\theta(r))^{d-1}}{1+\ri f_\theta'(r)}\frac{\partial}{\partial r}
\right)+\frac{1}{(r+\ri f_\theta(r))^2}\Delta_\omega
\end{align*}
(with $\Delta_\omega$ the surface Laplacian on $S^{d-1}$) and then either multiplying by $ \alpha \beta^{d-1}$ or not -- the coefficients $A_\theta,b_\theta$, and $n_\theta$ for both options are defined below.

\paragraph{Comparison of the two different formulations.}

The multiplication by $\alpha \beta^{d-1}$ has the advantage that the resulting operator is in divergence form; however, for $P_k$ to satisfy~\eqref{e:Garding1}, one requires additional assumptions. In particular,~\cite[Lemma 2.3]{GLSW1} shows that~\eqref{e:Garding1} holds for any $f_\theta(r)$ satisfying the above assumptions in $d=2$ and holds in $d=3$ when $f_\theta(r)/r$ is, in addition, non-decreasing and~\cite[Remark 2.1]{GLSW1} shows that such an additional assumption is needed.

If one instead integrates by parts the complex-scaled PDE directly (i.e., avoids the above multiplication), 
then the resulting sesquilinear form satisfies the G\aa rding inequality after multiplication by $\re^{\ri \omega}$, for some suitable constant $\omega$~\cite[Lemma A.6]{GLS1}. %The results of this paper hold equally for both these sesquilinear forms.
%\bre\label{rem:emptySet}
%When the complex-scaled PDE is integrated by parts directly (without multiplication by $\alpha \beta^{d-1}$), Theorem \ref{t:theRealDeal} is applicable to the sesquilinear form $\re^{\ri \omega} a_k(\cdot,\cdot)$. Since 
%\beqs
%a_k(u-u_h, v_h) = 0 \quad \text{ if and only if }\quad \re^{\ri \omega} a_k(u-u_h, v_h) = 0,
%\eeqs
%this multiplication does not change the Galerkin solution \eqref{e:Galerkin_ortho}.
%\ere
%

We highlight that, in other papers on PMLs, the scaled variable, which in our case is $r+\ri f_\theta(r)$, is often written as $r(1+ \ri \widetilde{\sigma}(r))$ with $\widetilde{\sigma}(r)= \sigma_0$ for $r$ sufficiently large; see, e.g., \cite[\S4]{HoScZs:03}, \cite[\S2]{BrPa:07}. Therefore, to convert from our notation, set $\widetilde{\sigma}(r)= f_\theta(r)/r$ and $\sigma_0= \tan\theta$.
In this alternative notation, the assumption that $f_\theta(r)/r$ is nondecreasing is therefore that $\widetilde{\sigma}$ is nondecreasing -- see \cite[\S2]{BrPa:07}.

\paragraph{The sesquilinear form after multiplication by $\alpha \beta^{d-1}$.}
Define $b_\theta:=0$,
\beq\label{e:firstPML}
A_\theta := 
\begin{cases}
A
\hspace{-1ex}
& \tin \Omega,\\
HDH^T 
\hspace{-1ex}
&\tin (\blueOld{B(0,\RPMLo)})^c
\end{cases}
\quad\tand\quad
n_\theta := 
\begin{cases}
 n &\tin \Omega \cap \blueOld{B(0,\RPMLo)},\\
\alpha(r) \beta(r)^{d-1} 
\hspace{-1ex}
&\tin (\blueOld{B(0,\RPMLo)})^c,
\end{cases}
\eeq
where, in polar coordinates $(r,\varphi)$,
\beqs
D =
\left(
\begin{array}{cc}
\beta(r)\alpha(r)^{-1} &0 \\
0 & \alpha(r) \beta(r)^{-1}
\end{array}
\right) 
\quad\tand\quad
H =
\left(
\begin{array}{cc}
\cos \varphi & - \sin\varphi \\
\sin \varphi & \cos\varphi
\end{array}
\right) 
\tfor d=2,
\eeqs
and, in spherical polar coordinates $(r,\varphi, \phi)$,
\beqs
D =
\left(
\begin{array}{ccc}
\beta(r)^2\alpha(r)^{-1} &0 &0\\
0 & \alpha(r) &0 \\
0 & 0 &\alpha(r)
\end{array}
\right) 
\tand
H =
\left(
\begin{array}{ccc}
\sin \varphi \cos\phi & \cos \varphi \cos\phi & - \sin \phi \\
\sin \varphi \sin\phi & \cos \varphi \sin\phi & \cos \phi \\
\cos \varphi & - \sin \varphi & 0 
\end{array}
\right) 
\eeqs
for $d=3$.
(observe that then $A=I$ and $n=1$ when $r=\RPMLo$ and thus $A_\theta$ and $n_\theta$ are continuous at $r=\RPMLo$).

\ble[{\cite[Lemma 2.3]{GLSW1}}]
Let $f_\theta$ satisfy \eqref{e:fProp} 
and the additional assumption when $d=3$ that $f_\theta(r)/r$ is nondecreasing,
given $\epsilon>0$
there exists $c>0$ such that, 
for all $\epsilon \leq \theta\leq \pi/2-\epsilon$, $A_\theta$ defined by \eqref{e:firstPML} satisfies
\beqs
\Re \big( A_\theta(x) \xi, \xi\big)_2 \geq c \|\xi\|_2^2 \quad\tfa \xi \in \mathbb{C}^d \tand x \in \Omega;
\eeqs
thus the G\aa rding inequality \eqref{e:Garding1} holds.
\ele

\paragraph{The sesquilinear form without multiplication by $\alpha \beta^{d-1}$.}

Define
\beq\label{e:secondPML}
A_\theta := 
\begin{cases}
A
\hspace{-1ex}
& \tin \Omega,\\
HDH^T 
\hspace{-1ex}
&\tin (\blueOld{B(0,\RPMLo)})^c
\end{cases}
\tand
n_\theta  := 
\begin{cases}
 n &\tin \Omega \cap \blueOld{B(0,\RPMLo)},\\
1
\hspace{-1ex}
&\tin (\blueOld{B(0,\RPMLo)})^c,
\end{cases}
\eeq
where, in polar coordinates $(r,\varphi)$,
\beqs
D =
\left(
\begin{array}{cc}
\alpha(r)^{-2} &0 \\
0 & \beta(r)^{-2}
\end{array}
\right) 
\quad\tand\quad
H =
\left(
\begin{array}{cc}
\cos \varphi & - \sin\varphi \\
\sin \varphi & \cos\varphi
\end{array}
\right) 
\tfor d=2,
\eeqs
and, in spherical polar coordinates $(r,\varphi, \phi)$,
\beqs
D =
\left(
\begin{array}{ccc}
\alpha(r)^{-2} &0 &0\\
0 & \beta(r)^{-2} &0 \\
0 & 0 &\beta(r)^{-2}
\end{array}
\right) 
\tand
H =
\left(
\begin{array}{ccc}
\sin \varphi \cos\phi & \cos \varphi \cos\phi & - \sin \phi \\
\sin \varphi \sin\phi & \cos \varphi \sin\phi & \cos \phi \\
\cos \varphi & - \sin \varphi & 0 
\end{array}
\right) 
\eeqs
for $d=3$ 
(observe that then $A_\theta=I$ and $n_\theta=1$ when $r=\RPMLo$ and thus $A$ and $n$ are continuous at $r=\RPMLo$).
In addition, for $d=2$, 
$$
b_\theta(r)=\begin{cases}0&\tin \Omega\cap \blueOld{B(0,\RPMLo)}\\
H \begin{pmatrix} 
\alpha^{-2}\big( \log (\alpha\beta)\big)'
\\0\end{pmatrix}&\tin (\blueOld{B(0,\RPMLo)})^c,\end{cases}
$$
and for $d=3$
$$
b_\theta(r)=\begin{cases}0&\tin \Omega\cap \blueOld{B(0,\RPMLo)}\\\
H\begin{pmatrix}
\alpha^{-2}\big( \log (\alpha\beta^2)\big)'\\0\\0\end{pmatrix}&\tin (\blueOld{B(0,\RPMLo)})^c.\end{cases}
$$

\ble[{\cite[Lemma A.6]{GLS2}}]
Let $f_\theta$ satisfy \eqref{e:fProp}. 
Given $\epsilon>0$
there exists $\omega\in \Rea, c>0$ such that, 
for all $\epsilon \leq \theta\leq \pi/2-\epsilon$, $A_\theta$ defined by \eqref{e:secondPML} satisfies
\beqs
\Re \big( \re^{\ri \omega} A_\theta(x) \xi, \xi\big)_2 \geq c\|\xi\|_2^2 \quad\tfa \xi \in \mathbb{C}^d \tand x \in \Omega;
\eeqs
thus the G\aa rding inequality \eqref{e:Garding1} holds for the sesquilinear form $\re^{\ri \omega}a_k(\cdot,\cdot)$.
\ele

\section{Loop decompositions in directed graphs (Theorem \ref{t:onlyTheSimpleOnes})}

\label{sec:loops}

Fix a matrix $\oldT \in \mathbb{M}(M\times M)$, let $\mathcal{G}$ be the graph associated to $\oldT$ as defined in \S \ref{sec:simple_path_mat} and $\transfer$ the simple-path matrix of $\oldT$. Denote by $\mathbb{P}_{ij}$ the set of paths from $i$ to $j$. Recalling the classical identity 
$$(\oldT^n)_{ij} = \sum_{p \in \mathbb{P}_{ij} \textup{ s.t. } |p|=n}\oldT_p$$
and summing over $n$, one obtains that
$$\left[\sum_{n \in \mathbb{N}} W^n\right]_{ij} = \sum_{p \in \mathbb{P}_{ij}} \oldT_p,$$
provided that the right-hand side converges.
The first inequality in \eqref{e:transfer} then follows immediately. 

To prove the implication in \eqref{e:loop_condition} and the second inequality in \eqref{e:transfer}, it is sufficient to show that 
\begin{equation}
\label{e:transfer_bound}
\sum_{p \in \mathbb{P}_{ij}} \oldT_p \leq\frac{1}{1-c} \transfer_{ij} 
\end{equation}
	
\subsection{Outline}
	
We show 
% Lemma \ref{t:onlyTheSimpleOnes}
\eqref{e:transfer_bound} by constructing an injective map
\[\mathcal{D}ec: \mathbb{\mathbb{P}}_{ij} \to
%\begin{cases}
 \mathbb{V}_{ij} \times \mathbb{SL}^{(\mathbb{N})} %, & i\neq j,\\
% \bigcup_{Q = 0}^{+ \infty} (\mathbb{SL})^{Q}, & i=j.
%\end{cases} 
  \]
 where for any set $A$, $A^{(\mathbb{N})}$ denotes the set of finite ordered sequences of elements of $A$ (possibly of size $0$). The map $\mathcal{D}ec$ is defined in Definition \ref{def:loopDec} below, and its properties are stated in Lemma \ref{l:decIsInj}. It corresponds to a decomposition of every path $p \in \mathbb{P}_{ij}$ into a non-intersecting segment 
$v \in \mathbb{V}_{ij}$ and a tuple of simple loops $(L_1,\ldots,L_Q) \in \mathbb{SL}^{(\mathbb{N})}$. The idea is that one obtains the decomposition by recursively removing loops from $p$ until the remainder is non-intersecting.
If one defines 
\begin{equation}
\label{e:extT}
\oldT_{(v,(L_1,\ldots,L_Q))} := \oldT_{v} \oldT_{L_1}\ldots \oldT_{L_Q},	
\end{equation}
then it will be seen that $\oldT_p = \oldT_{\mathcal{D}ec(p)}$ for all $p \in \mathbb{P}_{ij}$. The proof of Lemma \ref{t:onlyTheSimpleOnes} is then obtained as follows:
\[\sum_{p \in \mathbb{P}_{ij}} \oldT_{p} = \sum_{p \in \mathbb{P}_{ij}} \oldT_{\mathcal{D}ec(p)} = \sum_{q \in \mathcal{D}ec(\mathbb{P}_{ij})} \oldT_q \leq \sum_{q \in \mathbb{V}_{ij} \times \mathbb{SL}^{(\mathbb{N})}} \oldT_q\]
since $\mathcal{D}ec$ is injective. The last term can be rewritten as 
\begin{align*}
	 \sum_{q \in \mathbb{V}_{ij} \times \mathbb{SL}^{(\mathbb{N})}} \oldT_q& = \sum_{v \in \mathbb{V}_{ij}} \sum_{Q = 0}^{\infty} \sum_{L_1,\ldots,L_Q \in \mathbb{SL}}\oldT_{v} \oldT_{L_1}\ldots \oldT_{L_Q} = \sum_{v \in \mathbb{V}_{ij}} \oldT_v \sum_{Q = 0}^\infty\left(\sum_{L \in \mathbb{SL}} \oldT_L\right)^Q.
\end{align*}
Therefore if $\left(\sum_{L \in \mathbb{SL}} \oldT_L \right)\leq c < 1$, then
\[\sum_{p \in \mathbb{P}_{ij}} \oldT_{p} \leq \frac{1}{1-c} \sum_{v \in \mathbb{V}_{ij}} \oldT_v = \transfer_{ij}.\]

\subsection{Construction of the map $\operatorname{\cD ec}$}

For $1 \leq \ell \leq m \leq \abs{p}+1$, the {\em splice} of $p$ between $\ell$ and $m$, denoted by $p{[\ell,m)}$, is the path obtained from $p$ by only keeping the edges from $\ell$ to $m-1$, that is 
\[p{[\ell,m)} := e_{\ell} e_{\ell+1}\ldots e_{m-1},\]
with the convention that $p{[\ell,\ell)} = \mathbf{0}$. Given two paths $p = e_1 e_2\ldots e_{\abs{p}}$ and $q = f_{1}f_{2}\ldots f_{\abs{q}}$ such that $p(\abs{p}+1) = q(1)$, the {\em concatenation} of $p$ and $q$ is defined by
\[p\cdot q = e_1e_2\ldots e_{\abs{p}} f_{1} f_2\ldots f_{\abs{q}},\]
with the convention that for all paths $p$,
$p \cdot \mathbf{0} = \mathbf{0} \cdot p = p.$
In particular, for all $p,q \in \mathbb{P}$, $$|p\cdot q| = |p| + |q|.$$ 
Furthermore, when $m>\ell$, $p{[\ell,m)}$ is a path from $p(\ell)$ to $p(m)$, and for all $1 \leq \ell \leq |p|+1$, 
$$p = p[1,\ell) \cdot p[\ell,|p|+1).$$
If $p(\ell) = i_0$ and $L_{i_0}$ is either $\mathbf{0}$ or a loop through $i_0$, one can then define the {\em insertion of $L_{i_0}$ in $p$ at index $\ell$} by
\begin{equation*}
	p\overset{{\ell}}{\leftharpoonup}L_{i_0} := p{[1,\ell)}\cdot L_{i_0}\cdot p{[\ell,\abs{p}+1)}.
\end{equation*}

To extract the first loop of a self-intersecting path, one can ``follow" the path until some vertex $i_\times$ occurs for the second time. One then backtracks to the first occurence of that vertex, and the splice in between those two occurences defines a simple loop that can be extracted from $p$. More precisely, let
\[\ell_\times(p) := \inf \Big\{\ell \in \{1,\ldots,|p|+1\} \,: p(\ell) \in \{p(1),\ldots,p(\ell-1)\}\Big\},\]
the {\em index of first crossing}. Note that $\ell_\times(p) = \infty$ if, and only if, $p$ is non-intersecting. If $\ell_\times(p) < \infty$,  define $i_\times(p):= p(\ell_\times(p))$ the {\em first crossing point} of $p$, and  
\[\ell_0(p) := 
	\inf\Big\{\ell \in \{0,\ldots,\abs{p}+1\} \,: p(\ell) = i_\times(p)\Big\}\]
the first index at which $p$ visits $i_\times(p)$. These definitions are illustrated in Figure \ref{f:defIcross}. Define the maps $L: \mathbb{P}\to \mathbb{P}$ and $E:\mathbb{P} \to \mathbb{P}$ by\\
\begin{minipage}{0.65\textwidth}
	\begin{itemize}
		\item $L(p) := \begin{cases}
			p[\ell_0(p),\ell_\times(p)) & \textup{if } \ell_\times(p) < \infty,\\[1em]
			\mathbf{0} & \textup{if } \ell_\times(p) = \infty.
		\end{cases} 
		$\\[1em]
		the {\em first loop} in $p$, and\\[1em]
		\item $	E(p) := \begin{cases}
			p[1,\ell_0(p))\cdot p[\ell_\times(p),\abs{p}+1) & \textup{if } \ell_\times(p) < \infty,\\[1em]
			p & \textup{if } \ell_\times(p) = \infty,
		\end{cases}$\\[1em]
		the remainder after extracting the loop $L(p)$.
	\end{itemize}
\end{minipage}
\hfill
\begin{minipage}{0.3\textwidth}
	\begin{figure}[H]
		\centering
		\includegraphics[width=\textwidth]{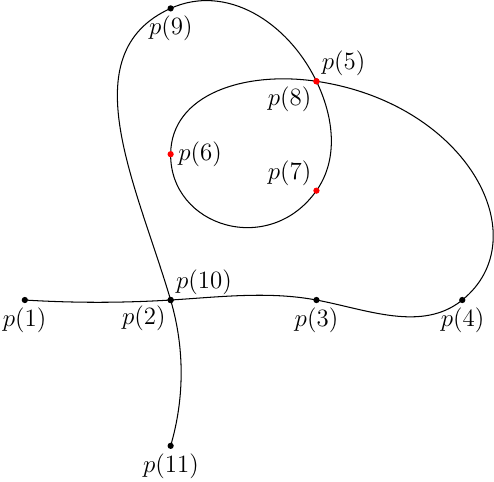}
		\caption{Example of a self-intersecting path. Here, $\ell_0(p) = 5$ and $\ell_\times(p) = 8$. The vertices of $L(p)$ are highlighted in red.}
		\label{f:defIcross}
	\end{figure}
\end{minipage}\\[1em]
The properties of $L$ and $E$ are summarized in the following lemma.  
The proof is immediate from the definitions.
\begin{lemma}
	\label{l:recreate2}
	For all paths $p$, 
	\[p = E(p)\overset{{\ell_0(p)}}{\leftharpoonup}L(p),\]
	\[\abs{p} = \abs{E(p)} + \abs{L(p)} \quad \tand \quad \oldT_p = \oldT_{L(p)} \oldT_{E(p)}.\]
	%If $\abs{p} > 0$, then $L(p) \in \mathbb{SL}$, $\abs{L(p)} \geq 1$ and $\abs{E(p)} \leq \abs{p}-1$.
	The path $p$ is non-intersecting if, and only if, $L(p) = \mathbf{0}$, in which case, $p = E(p)$. Otherwise, $L(p) \in \mathbb{SL}$, $|L(p)|\geq 1$ and
	\[\ell_0(p) = \inf\Big\{\ell \in \{0,\ldots,\abs{E(p)}+1\}\,: \big(E(p)\big)(\ell) =\big(L(p)\big)(1)\Big\}.\]
	If $p \in \mathbb{P}_{ij}$, then either 
	\begin{enumerate}
		\item $E(p) \in \mathbb{P}_{ij}$, or
		\item $E(p) = \mathbf{0}$, and this can only occur if $i = j$. 
	\end{enumerate}
\end{lemma}
Define $E^n(p) := E(E^{n-1}(p))$, with $E^0(p) := p$. 
\begin{corollary}
	\label{c:numberLoops}
	Let $p \in \mathbb{P}$. Then there exists a unique number $n_0 \in \mathbb{N}$, the {\em number of loops in $p$}, such that the following properties hold:
	\begin{itemize}
		\item  either $n_0 = 0$ or $E^{n_0 - 1}(p) \neq E^{n_0}(p)$,
		\item for all $n \geq n_0$, $E^n(p) = E^{n_0}(p)$.
	\end{itemize}  
\end{corollary}
\begin{proof}
	If $E^{n+1}(p) \neq E^n(p)$, then by Lemma \ref{l:recreate2}, $|E^{n+1}(p)| \leq |E^n(p)| - 1$. Since infinite sequences of natural numbers cannot be strictly decreasing, the sequence $(E^n(p))_n$ must eventually stagnate. 
\end{proof}
\begin{corollary}
	\label{c:recreate2}
	The map 
	\[\begin{array}{llll}
		\textup{Dec}:&\mathbb{P}\times \mathbb{P}^{(\mathbb{N})} &\to& \mathbb{P}\times \mathbb{P}^{(\mathbb{N})}\\
		&p,(L_1,\ldots,L_Q) & \mapsto & E(p), (L_1,\ldots,L_Q,L(p))
	\end{array}\]
	is injective. If $X \in \mathbb{P} \times \mathbb{P}^{(\mathbb{N})}$ and if $\oldT_X$ is defined as in \eqref{e:extT}, then 
	\begin{equation}
		\label{e:conserv}
		\oldT_{\operatorname{Dec}(X)} = \oldT_{X}.
	\end{equation}
\end{corollary}
\begin{proof}
	Suppose that
	\[(E(p),(L_1,\ldots,L_Q,L(p))) = (E(p'),(L_1',\ldots,L_Q',L(p')))\]
	and let $E = E(p) = E(p')$ and $L = L(p) = L(p')$. To conclude, it suffices to show that $p = p'$ (since it is obvious that $L_i = L'_i$ for $1 \leq i \leq Q$). There are two cases: either $|L| = 0$ or $|L|\geq 1$. By Lemma \ref{l:recreate2}, in the first case, $E = p = p'$. In the second case, since $E(p)=E(p')$ and $L(p)=L(p')$, 
	\[\ell_0(p) = \ell_0(p') = \inf\Big\{\ell \in \{0,\ldots,\abs{E}+1\}\,: E(\ell) =L(1)\Big\} =:\ell_0 \]
	and $p = p' = E\overset{{\ell_0}}{\leftharpoonup}L$. Thus in both cases, $p = p'$. 
	The proof of \eqref{e:conserv} is immediate.
\end{proof}\
\begin{definition}[Loop decomposition of a path]
	\label{def:loopDec}
	Given $p \in \mathbb{P}$, the {\em loop-decomposition} of $p$, denoted by $\operatorname{\mathcal{D}ec}(p) \in \mathbb{P} \times \mathbb{P}^{(\mathbb{N})}$, is defined by
	\[\operatorname{\mathcal{D}ec}(p) := \operatorname{Dec}^{n_0}(p,\varnothing)\]
	where $\varnothing$ is the empty sequence of paths, and $n_0$ is the number of loops in $p$. 
\end{definition}
\begin{lemma}
	\label{l:decIsInj}
	For all $p \in \mathbb{P}_{ij}$, $\operatorname{\cD ec}(p) \in \mathbb{V}_{ij} \times \mathbb{SL}^{(\mathbb{N})}$, and 
	\begin{equation}
		\label{e:conserv2}
		\oldT_{\operatorname{\cD ec}(p)} = \oldT_p.
	\end{equation} 
	Furthermore, the map 
	$\operatorname{\cD ec} : \mathbb{P} \to \mathbb{V} \times \mathbb{SL}^{(\mathbb{N})}$
	is injective. 
\end{lemma}
\begin{proof}
	Write $\operatorname{Dec}^{n_0}(p,\varnothing) = (v,(L_1,\ldots,L_{n_0}))$, and observe that $v = E^{n_0}(p)$. If $v$ were self-intersecting, then it would follow that $E^{n_0+1}(p) \neq E^{n_0}(p)$, contradicting Corollary \ref{c:numberLoops}. Thus, $v \in \mathbb{V}$. If $p \in \mathbb{P}_{ij}$ then either $i \neq j$, in which case it follows by \ref{l:recreate2} that $v \in \mathbb{V}_{ij}$, or $i = j$ in which case $v = \mathbf{0}$ (otherwise we would have an non-intersecting path in $\mathbb{P}_{ii}$, which is impossible). On the other hand, one can check easily by induction that 
	\[L_1 = L(p),\quad L_2 = L(E(p)), \quad \ldots\quad L_{n_0} = L(E^{n_0-1}(p)),\]
	and thus, by Lemma \ref{l:recreate2}, for $1 \leq i \leq n_0$, $L_i$ is either $\mathbf{0}$ or a simple loop. But $L_i$ cannot be $\mathbf{0}$, since this would imply that $E^{i-1}(p) = E^i(p)$, contradicting again Corollary \ref{c:numberLoops}. Thus $L_1,\ldots,L_{n_0} \in \mathbb{SL}$. 
	
The relation \eqref{e:conserv2} follows immediately from \eqref{e:conserv}. 
	
	Finally, suppose that $\operatorname{\mathcal{D}ec}(p) = \operatorname{\cD ec}(p')$. Then $p$ and $p'$ have the same number of loops $n_0$ (otherwise, the list of loops in their loop-decomposition could not be the same) and thus 
	$\operatorname{\cD ec}(p) = \operatorname{Dec}^{n_0}(p,\varnothing)$, $\operatorname{\cD ec}(p') = \operatorname{Dec}^{n_0}(p',\varnothing)$ and therefore
	\[\operatorname{Dec}^{n_0}(p,\varnothing) = \operatorname{Dec}^{n_0}(p',\varnothing)\]
	But since $\operatorname{Dec}$ is injective (by Corollary \ref{c:recreate2}), $\operatorname{Dec}^{n_0}$ is injective, and thus it must be that $p = p'$. 
	\end{proof}

\section{Proofs of the local bounds on the solution operator (Theorems \ref{thm:DV} and \ref{thm:PML})}\label{app:help}

In this section we prove Theorems \ref{thm:DV} and \ref{thm:PML}. In fact, we prove a stronger 
analogue of Theorem \ref{thm:DV} phrased using semiclassical pseudodifferential operators -- i.e., pseudodifferential operators in a calculus where each derivative is weighted by $k^{-1}$. Furthermore, because we work on a bounded domain, we need a special class of pseudodifferential operators adapted to the boundary. 

\subsection{Pseudodifferential operators and $b$-pseudodifferential operators}

\subsubsection{Semiclassical pseudodifferential operators.}
\label{s:semiclassicalPseudos}

Semiclassical pseudodifferential operators are generalisations of Fourier multipliers acting as
$$
\Op(a)u(x):=\frac{1}{(2\pi h)^d}\int e^{\frac{i}{h}\langle x-y,\xi\rangle}a(x,\xi)u(y)dyd\xi,
$$
where, for some $m\in \mathbb{R}$, $a$ satisfies
$$
|\partial_x^\alpha \partial_\xi^\beta a(x,\xi)|\leq C_{\alpha \beta}\langle \xi\rangle^{m-|\alpha|}.
$$ 
In this case we write $a\in S^m(T^*\mathbb{R}^d)$ and $\Op(a)\in \Psi^m(\mathbb{R}^d)$. When $m=0$ we write $S(T^*\mathbb{R}^d)$ and $\Psi(\mathbb{R}^d)$ respectively. This class of pseudodifferential operators is the natural class of operators generalising quantization of $b(x)(hD)^\alpha$ for some $b\in C^\infty(\mathbb{R}^d)$ and $\alpha\in\mathbb{N}^d$. For more details and information about the calculus of such operators see e.g.~\cite[Appendix E]{DyZw:19} and~\cite{Zw:12}.

The class of $b$-pseudodifferential operators that we work with is, instead, the natural class of operators quantizing differential operators that are tangential to the boundary of $\Omega_+$. Away from $\partial\Omega_+$ they are pseudodifferential operators in the sense above, but near $\partial\Omega_+$ they have a different form. In particular, in coordinates $(x_1,x')$ with $\partial\Omega_+=\{x_1=0\}$, their symbols are functions on the $b$-cotangent bundle, ${}^bT^*\Omega_+$, whose sections are of the form
$$
\sigma \frac{d x_1}{x_1}+\xi' dx'.
$$
Notice that ${}^bT^*\Omega_+$ is the dual to sections of $T^*\Omega_+$ that are tangent to $\partial\Omega_+$. We also write ${}^b\overline{T^*\Omega_+}$ for the fiber radially compactified $b$-contangent bundle; i.e., ${}^bT^*\Omega_+$ with the sphere at infinity in $(\sigma, \xi')$ attached.

%There is a natural inclusion $\pi^b:T^*\Omega_+\to {}^bT^*\Omega_+$ which, in coordinates, is given by $(x,\xi)\mapsto (x,x_1\xi_1,\xi')$ and whose image is denoted by ${}^b\dot{T}^*\Omega_+$. 
In coordinates, $b$-pseudodifferential operators are of the form
$$
\Op_{b}(a)(u)(x)=\frac{1}{(2\pi h)^d}\int e^{\frac{i}{h}((x_1-y_1)\xi_1+(x'-y'),\xi')}\phi(x_1/y_1)a(x_1,x',x_1\xi_1,\xi')u(y)dyd\xi,
$$
where $\phi\in C_c^\infty(1/2,2)$ with $\phi\equiv 1$ near $1$  and for some $m$
$$
|D_{x}^\alpha D_\sigma^j D_{\xi'}^\beta a(x_1,x',\sigma,\xi')|\leq C_{j\alpha\beta}\langle (\sigma,\xi')\rangle^{m-j-|\beta|}.
$$
In this case, we write $\Op_{b}(a)\in \Psi_b^m(\Omega_+)$ and $a\in S^m({}^bT^*\Omega_+)$. When $m=0$ we write $S({}^bT^*\Omega_+)$ and $\Psi_b(\Omega_+)$ respectively. We also write $\Psi_b^{-\infty}=\cap_m \Psi_b^m$. 

 The class comes equipped with principal symbol map ${}^b\sigma: {}^b\Psi^m(\Omega_+)\to S^m({}^bT^*\Omega_+)/ h {}^bS^{m-1}(T^*\Omega_+)$ such that if $A\in\Psi_b(\Omega_+)$ and $\sigma(A)=0$ then $A\in h\Psi_b^{m-1}(\Omega_+)$. We now introduce two important sets for $b$-pseudodifferential operators. For $A\in \Psi_b^m(\Omega_+)$ and $q\in {}^b\overline{T^*\Omega_+}$, we say $q\in {}^b\Ell(A)$ if there is a neighbourhood, $U$ of $q$ such that
$$
|\sigma(A)(q')|\langle (\sigma,\xi')\rangle^{-m}>c>0,\qquad q'\in U\cap {}^bT^*\Omega_+.
$$ 
Next, we say $q\notin {}^b\WF(A)$ if there is $E\in {}^b\Psi(\Omega_+)$ with $q\in {}^b\Ell(E)$ such that 
$$
EA\in h^\infty \Psi_b^{-\infty}.
$$ 
For a more complete treatment of these operators, we refer the reader to~\cite[Appendix A]{HiVa:18} and the references therein.

\subsubsection{The generalised bicharacteristic flow on ${}^bT^*\Omega_+$.}

\blueOld{
To define carefully the notion of a generalised broken bicharacteristic we follow~\cite[Section 24.3]{Ho:85} and introduce some additional concepts. The semiclassical principal symbol of the operator $
P_\theta = -k^{-2}n_\theta^{-1}(x)\div(A_\theta(x) \nabla)-1$
is given by 
$$
p_\theta(x,\xi):=n_\theta^{-1}(x)A_\theta^{ij}(x)\xi_i\xi_j-1. 
$$

We can choose coordinates near a point $x_0\in \partial\Omega_+$ such that there is a neighborhood $U$ of $x_0$ and $V$ of $(0,0)$ satisfying 
$$
p_\theta=\xi_1^2-r(x,\xi'),\qquad \partial\Omega\cap U=\{x_1=0\}\cap V,\quad \Omega\cap U=\{ x_1>0\}\cap V,
$$
where $r$ is real-valued.
When the define the \emph{hyperbolic points}
$$
\mathcal{H}:=\{ (x,\xi)\,:\, x_1=0,\, \Re p_\theta=0,\, r(x,\xi')> 0\}.
$$
Notice that 
$$
\mathcal{H}\cap \{\Re  p_\theta=0\}=\cup_{\pm}\mathcal{H}_{\pm},\qquad \mathcal{H}_{\pm}:=\Big\{ (x,\xi)\,:\, x_1=0,\,\xi_1=\pm \sqrt{r(x,\xi')}\Big\}.
$$
The hyperbolic points (over $\{\Re p_\theta=0\}$) are those points at which the bicharacteristic flow for $\Re p_\theta$ is transverse to the boundary and hence at which generalised bicharacteristics will obey the usual Snell--Descartes law of reflection.

Next, we define the \emph{glancing points}
$$
\mathcal{G}:=\Big\{ (x,\xi)\,:\, \Re p_\theta(x,\xi) = x_1= r(x,\xi')= \xi_1 = 0\Big\},
$$
 the \emph{diffractive points}
$$
\mathcal{G}_d:=\{ (x,\xi)\in \mathcal{G}\,:\, H_{\Re p_\theta}^2x_1>0\},
$$
and the \emph{gliding vector field}:
$$
H_{\Re p_\theta}^{\mathcal{G}}:=H_{\Re p_\theta}+(H_{\Re p_{\theta}}^2x_1/H_{x_1}^2{\Re p_{\theta}})H_{x_1}. 
$$
(Notice that $H_{\Re p_\theta}^{\mathcal{G}}$ is tangent to $\mathcal{G}$.)

Roughly speaking, the glancing points are those points at which the Hamiltonian vector field for $\Re p_\theta$ is tangent to the boundary and are divided into two pieces, $\mathcal{G}_d$, where the Hamiltonian trajectory touches the boundary at a single point and has exactly second order contact, and $\mathcal{G} \setminus \mathcal{G}_d$, where the trajectories will follow the gliding vector field. In $\mathcal{G}\setminus \mathcal{G}_d$, the trajectory may have order of contact larger than 2 and immediately leave the boundary or it may ``stick'' to the boundary until the curvature changes sign.

\begin{definition}\label{d:PGBB}
    A \emph{pre-generalised broken bicharacteristic} is a differentiable map
    $$\gamma:I\setminus B\to (T^*\Omega_+\setminus \partial (T^*\Omega_+))\cup \mathcal{G},$$ 
    where $I\subset \mathbb{R}$ is an interval with $B\subset I$, such that $\Re p_{\theta}\circ \gamma=0$ and
    \begin{enumerate}
        \item $\gamma'(t)=H_{\Re p_\theta}(\gamma(t))$ if $\gamma(t)\in T^*\Omega_+\setminus \partial (T^*\Omega_+)$ or $\gamma(t)\in \mathcal{G}_d$
        \item $\gamma'(t)=H_{\Re p_{\theta}}^{\mathcal{G}}(\gamma(t))$ if $\gamma(t)\in \mathcal{G}\setminus \mathcal{G}_d$
        \item Every $t\in B$ is isolated, for $t\in B$ $\gamma(s)\in T^*\Omega_+\setminus \partial (T^*\Omega_+)$ if $s\neq t$ and $|s-t|$ is small enough, the limits $\gamma(t\pm 0)$ exist and, in the coordinates above, there are $x',\xi'$ such that 
        $$
        \gamma(t\pm 0)= (0,x', \xi_1^{\pm},\xi'),\qquad \xi_1^+=-\xi_1^-, \, (\xi_1^{+})^2=r(x,\xi');
        $$
        i.e. they are different points in the same hyperbolic fiber of $ \partial\Omega_+$. 
    \end{enumerate}
    A \emph{generalised broken bicharacteristic} is $\tilde{\gamma}:I\to \pi^{b}(\{\Re p_{\theta}=0\})$ a continuous map such that there are $B$, $\gamma$ with $\gamma:I\setminus B\to \{\Re p_{\theta}=0\}$ a pre-generalised broken bicharacteristic and, for $t\in I\setminus B$, $\tilde{\gamma}(t)=\pi^b(\gamma(t))$. (Here, $\pi^b$ is as in \S\ref{s:semiclassicalPseudos}.)
\end{definition}

We highlight that, given a point in $\rho\in \pi^{b}(\{\Re p_{\theta}=0\})$, there may be more than one generalised broken bicharacteristic through $\rho$. This non-uniqueness occurs at points of infinite order tangency with Hamiltonian for $p$ \cite[Corollary 24.3.10]{Ho:07}. (For a construction of a point with multiple generalised broken bicharacteristics, see~\cite[Example 24.3.11]{Ho:07}. This construction relies on the existence of a trajectory with infinitely many transversal intersections with the boundary in finite time whose endpoint is tangent to the boundary.) It is also easy to see that uniqueness holds when at every such point of infinite order tangency, $\partial\Omega_+$ is geodesically concave.

\begin{definition}
    Let $\mathcal{U}\subset \pi^b(\{\Re p_\theta =0\})$ be the set of points through which there is a unique generalised broken bicharacteristic. We define the generalised broken bicharacteristic flow, $\varphi_t:\mathcal{U}\to\mathcal{U}$ by 
    $$
    \varphi_t(\rho)=\gamma(t),\qquad \gamma(0)=\rho,
    $$
    where $\gamma$ is a generalised broken bicharacteristic. 
\end{definition}
One can show that $\varphi_t$ is continuous~\cite[24.3.12]{Ho:07}.

\bre 
We assume for simplicity below that the generalised broken bicharacteristics are everywhere unique (i.e. $\mathcal{U}=\pi^b(\{\Re p_\theta=0\})$), but highlight that the definitions of $K$ and $\Gamma_\pm$, and the statements of Theorems \ref{thm:negligible2} and \ref{thm:negligible3}
below can all be rewritten without the need for uniqueness of the flow (with the results of the theorems
still true). We also point out that, as discussed above, $\mathcal{U}=\pi^{b}(\{\Re p_\theta=0\})$ whenever every point of infinite order tangency is geodesically concave with respect to the metric $g^{-1}=n^{-1}A$.
\ere
}

%\manote{I will make a figure.}

% Let $p_\theta\in S^2(T^*\Omega)$ denote the semiclassical principal symbol of $P_k$ and observe that on $B(0,R_{\operatorname{scat}})$, $p_\theta = \sum_{ij}g^{ij}(x)\xi_i\xi_j-1$, where $g^{-1}(x)=A(x)/n(x)$. 
% We then let $\varphi_t:\pi^b\{\Re p_\theta=0\}\to \pi^b\{\Re p_\theta=0\}$ be the generalised bicharacteristic flow for $\Re p_\theta$ in the sense of~\cite[Definition 1.1]{Va:08}. 

We are now in a position to define the \emph{forward and backward trapped sets} $\Gamma_-$ and $\Gamma_+$, respectively,
$$
\Gamma_{\pm}:=\Big\{ q\in \pi^b(\{ \Re p_\theta=0\})\,: \sup \{t>0\,:\, \varphi_{\mp t}(q)\in {}^bT^*\Omega\}=\infty\Big\},
$$
as well as the trapped set,
$$
K:=\Gamma_+\cap \Gamma_-.
$$
One can show that $\Gamma_{\pm}$ and hence $K$ are closed (see e.g.~\cite[Proposition 6.3]{DyZw:19}).

\subsection{Improved resolvent estimates}

We can now state our improved estimates on the solution operator. Below, we use the notation ${}^b\Psi(\Omega)$ for elements of ${}^b\Psi(\Omega_+)$ whose kernels are supported away from $\Gamma_{\tr}$.
\begin{theorem}\label{thm:DV2}
Let $k_0>0$, and  let $\blueOld{\mathrm{J}}$ be such that 
Assumption \ref{a:polyBoundIntro} holds. 
 Then for all $A\in {}^b\Psi(\Omega)$ with ${}^b\WF(A)\cap K=\emptyset$, there exists $C>0$ such that, for all $k\in (k_0,\infty)\setminus\blueOld{\mathrm{J}}$,
\begin{gather*}
\|AR_k\|_{L^2\to L^2}+\|R_kA\|_{L^2\to L^2}\leq C\sqrt{\|R_k\|k},\qquad \|AR_kA\|_{L^2\to L^2}\leq Ck.
\end{gather*}
\end{theorem}

%\begin{theorem}\label{thm:PML2}
%Let $k_0>0$, and  let $\blueOld{\mathrm{J}}$ be such that 
%Assumption \ref{a:polyBoundIntro} holds. 
%Then there is $U\subset \Omega$ a neighbourhood of $\Gamma_{\tr}$ such that for all $A\in {}^b\Psi(\Omega)$ with ${}^b\WF(A)\subset U$, there exists $C>0$ such that  for all $k>k_0$
%\begin{equation}
%\label{e:bounded_in_pml}
%\|AR_k\|_{L^2\to L^2}+\|R_kA\|_{L^2\to L^2}\leq C.
%\end{equation}
%Moreover, if ${}^b\WF(A)\subset U$, and ${}^b\WF(A)\cap {}^b\WF(B)=\emptyset$, then for any $N$ there exists $C>0$ such that for all $k>k_0$, 
%\begin{equation}
%\label{e:pseudoloc_in_pml}
%\|AR_kB\|_{L^2\to H_k^N}+\|BR_kA\|_{L^2\to H_k^N}\leq Ck^{-N}.
%\end{equation}
%\end{theorem}

\begin{theorem}\label{thm:negligible2}
Let $k_0>0$ and let $\blueOld{\mathrm{J}}$ be such that 
Assumption \ref{a:polyBoundIntro} holds. 
 Then for all $A,B \in {}^b\Psi(\Omega)$ with 
 \begin{equation*}
% \label{e:backFlow}
 \overline{{}^b\WF(A)\cup \bigcup_{t\geq 0} \varphi_{-t}({}^b\WF( A)\cap \pi^b(\{\Re p_\theta=0\}))}\cap {}^b\WF(B)=\emptyset,\qquad {}^b\WF(A)\cap \Gamma_+=\emptyset
 \end{equation*}
 and all $N>0$ there exists $C>0$ such that, for all $k\in (k_0,\infty)\setminus\blueOld{\mathrm{J}}$,
\begin{gather*}
\|A R_kB\|_{L^2\to L^2}\leq Ck^{-N}.
\end{gather*}
%If, in addition, $A,B\in C^\infty(\overline{\Omega})$, then
%\begin{gather*}
%\|A R_kB\|_{L^2\to H_k^N}\leq Ck^{-N}.
%\end{gather*}
\end{theorem}

\begin{theorem}\label{thm:negligible3}
Let $k_0>0$ and let $\blueOld{\mathrm{J}}$ be such that 
Assumption \ref{a:polyBoundIntro} holds. 
 Then for all $A,B \in {}^b\Psi(\Omega)$ with 
 \begin{equation*}
% \label{e:forwardFlow}
 \overline{{}^b\WF(A)\cup\bigcup_{t\geq 0} \varphi_{t}({}^b\WF( A)\cap\pi^b(\{\Re p_\theta=0\}))}\cap {}^b\WF(B)=\emptyset,\qquad {}^b\WF(A)\cap \Gamma_-=\emptyset
 \end{equation*}
 and all $N>0$ there exists $C>0$ such that, for all $k\in (k_0,\infty)\setminus\blueOld{\mathrm{J}}$,
\begin{gather*}
\|A R_k^*B\|_{L^2\to L^2}\leq Ck^{-N}.
\end{gather*}
%If, in addition, $A,B\in C^\infty(\overline{\Omega})$, then
%\begin{gather*}
%\|A R_k^*B\|_{L^2\to H_k^N}\leq Ck^{-N}.
%\end{gather*}
\end{theorem}

\bpf[Proof of Theorem \ref{thm:DV} using Theorems \ref{thm:DV2}-\ref{thm:negligible2}] Part (i) of Theorem \ref{thm:DV} follows immediately from
Theorem \ref{thm:DV2}.
Part (ii) of Theorem \ref{thm:DV} follows from Theorems \ref{thm:negligible2} and \ref{thm:negligible2} by choosing $B=\psi$ (i.e., the cutoff
in $\cavity$) and then noting that the choice $A=\chi$ (i.e., the cutoff in $\invisible$) satisfies the assumptions in
Theorems \ref{thm:negligible2} and \ref{thm:negligible2}.
\epf

\subsection{Estimates away from the scatterer}

We start by proving an estimate `deep' in the PML region; i.e. near the truncation boundary. 
\begin{lemma}
\label{l:nearPMLBoundary}
There exists $U\subset \Omega$ such that $\overline{U}$ is a neighbourhood of $\Gamma_{\tr}$ and for all $k_0>0$, $\psi,\widetilde{\psi}\in C^\infty(\Omega)$ with $\supp \psi,\supp\widetilde\psi\subset U$, $\psi\prec \widetilde\psi$, and $\supp (1-\psi)\cap \partial\Omega=\emptyset$, there exists $C>0$ such that for $k>k_0$,
\begin{equation}
\label{e:nearPMLByAway}
\|\psi u\|_{H_k^1}\leq C\Big(\| \widetilde\psi P_k u\|_{L^2}+Ck^{-N}\| u\|_{H_k^{-N}}\Big).
\end{equation}
\end{lemma}
\begin{proof}
By~\cite[Lemma 4.4]{GLS2}, there is $U$ with $\overline{U}$ a neighbourhood of $\partial \Omega_{\tr}$ such that for $v\in H_k^1(\Omega)$ with $v|_{\Gamma_{\tr}}=0$ and $\supp v\subset U$,
\begin{equation}
\label{e:PMLImport}
\|v\|_{H_k^1}\leq C\|P_k v\|_{L^2}.
\end{equation}
%The by boundary elliptic regularity, for any $m\geq 2$, 
%\begin{equation}
%%\label{e:PMLImport}
%\|v\|_{H_k^{m}}\leq C\|Pv\|_{H_k^{m-2}}.
%\end{equation}
Let $\psi_j\in C^\infty_c(U)$ be such that $\psi_1\prec\psi_2\prec\widetilde\psi$ and $\supp(1-\psi_1)\cap \supp \partial \psi=\emptyset$.
Applying \eqref{e:PMLImport} with $v= \psi u$, we obtain 
\beq\label{e:actualLastDay1}
\|\psi u\|_{H_k^1}\leq C\|P_k\psi u\|_{L^2}\leq C\big(\|\psi P_ku\|_{L^2}+\big\| [P_k, \psi] u\big\|_{L^2}\big)
\leq C\big(\|\psi P_ku\|_{L^2}+ Ck^{-1}\| \psi_1 u\|_{H_k^1}\big)
\eeq
Now, shrinking $U$ if necessary so that 
$$
\{x\in U\,:\, \text{ there exists }\xi \text{ such that }p_\theta(x,\xi) = 0 \} =\emptyset,
$$
the elliptic parametrix construction~\cite[Proposition E.32]{DyZw:19} implies that
\beq\label{e:actualLastDay2}
\|\psi_1 u\|_{H^1_k}\leq C\|\psi_2P_ku\|_{L^2}+Ck^{-N}\|u\|_{H_k^{-N}},
\eeq
and the result follows by combining \eqref{e:actualLastDay1} and \eqref{e:actualLastDay2}.
\end{proof}

\

We now prove Theorem \ref{thm:PML}.

\

\bpf[Proof of Theorem \ref{thm:PML}]
The bound $\|\chi R_k\|_{L^2\to L^2}\leq C$ follows immediately from \eqref{e:nearPMLByAway}, with the bound $\|R_k\chi \|_{L^2\to L^2}\leq C$ then following by applying the previous bound with $P_k$ replaced by $P_k^*$. 

The bound $\|\chi R_k\psi\|_{L^2\to H^1_k}\leq Ck^{-N}$ also follows immediately from \eqref{e:nearPMLByAway}, with then the bound $\|\chi R_k\psi\|_{L^2\to H^N_k}\leq Ck^{-N}$ following by elliptic regularity up to the boundary. 

Finally, the bound $\|\psi R_k\chi\|_{L^2\to L^2}\leq Ck^{-N}$ follows by applying the  bound 
$\|\chi R_k\psi\|_{L^2\to L^2}\leq Ck^{-N}$
with $P_k$ replaced by $P_k^*$, and then the bound $\|\psi R_k\chi\|_{L^2\to H^N_k}\leq Ck^{-N}$ follows by elliptic regularity up to the boundary. 
\epf

\

Next, we prove an estimate near incoming points away from the truncation boundary.

\begin{lemma}
\label{l:outgoing}
Let $m\geq 2$, $\chi\in C_c^\infty(\Omega\setminus \overline{B(0,R_{\operatorname{scat}}}))$ (where $R_{\rm scat}$ is defined by \eqref{e:Rscat}). Then there is $\epsilon>0$ such that for $A,B\in \Psi^0$ with 
$$
\WF(A)\cap \Big\{ \big\langle \tfrac{x}{|x|},\xi\big\rangle\geq \e,\,p_\theta(x,\xi)=0\Big\}=\emptyset,
$$
$$
\WF(A)\cup \bigcap_{t\geq 0}\big\{(x-t\xi,\xi)\,:\, (x,\xi)\in \WF(A)\cap \{p_\theta=0\}\big\}\cap\{p_\theta=0\}\subset\operatorname{Ell}(B)
$$
and $N>0$, given $k_0>0$ there exists $C>0$ such that for all $k\geq k_0$
$$
\|A\chi u\|_{H_k^m}\leq Ck\|BP_ku\|_{H_k^{m-2}}+Ck^{-N}\|u\|_{H_k^{-N}}.
$$
\end{lemma}
\begin{proof}Since 
$$
\WF(A)\cap \Big\{ \langle \tfrac{x}{|x|},\xi\rangle\geq \e,\,p_\theta(x,\xi)=0\Big\},
$$ 
there is a neighbourhood, $V$ of $\{p_\theta=0\}$ such that 
$$
\WF(A\chi )\cap V\subset\{ \langle \tfrac{x}{|x|},\xi\rangle< 2\e\}.
$$
In particular, for $(x,\xi)\in \WF(A\chi)\cap V$, and $t\geq 0$,
$$
|x-t\xi|^2=|x|^2-2t|x|\langle \tfrac{x}{|x|},\xi\rangle +t|\xi|^2\geq |x|^2-4t|x|\e+t^2\geq |x|^2(1-2\e)+t^2(1-2\e).
$$
Therefore, there is $T>0$ such that for all $(x,\xi)\in\WF(A\chi)\cap V$, there is $0\leq t\leq T$ such that $(x-t\xi,\xi)\notin \{p_\theta=0\}$. Using a microlocal partition of unity on $\WF(A\chi)\cap V$, $\{X_j\}_{j=1}^N$, there are $0<T_j\leq T$ and $E_j\in \Psi^{\comp}$ with $\WF(E_j)\subset \{p_\theta\neq 0\}\cap \Omega_{\tr}\cap \operatorname{Ell}(B)$ such that
$$
\{ (x-t\xi,\xi)\,:\, (x,\xi)\in \WF(AX_j\chi)\cap V,\, 0\leq t\leq T_j\}\subset \operatorname{Ell}(B)
$$
and 
$$
 \{ (x-T_j\xi,\xi)\,:\, (x,\xi)\in \WF(AX_j\chi)\cap V\}\subset \operatorname{Ell}(E_j).
$$

%Letting $U$, $\psi$, and $\psi_1$ as in Lemma~\ref{l:nearPMLBoundary} such that $\supp \psi_1\subset \Omega\setminus \overline{B(0,R_{PML_-})}$, there is $T>0$ such that for all $(x,\xi)\in \WF(A\chi)\cap V$, there is $0\leq t\leq T$ such that $x-t\xi\in \{\psi\equiv 1\}$. 

Now, let $X\in \Psi^{\comp}$ with $\WF(X)\subset V$ and $\WF(I-X)\cap \{p_\theta=0\}=\emptyset$. Then, by the elliptic parametrix construction~\cite[Proposition E.32]{DyZw:19}
\begin{equation}
\label{e:high}
\|(I-X)A\chi u\|_{H_k^m}\leq C\|BP_ku\|_{H_k^{m-2}}+Ck^{-N}\|u\|_{H_k^{-N}}.
\end{equation}
On the other hand, using that $X\in \Psi^{\comp}$ and then that $\Im p_\theta \leq 0$ near $p_\theta=0$, by~\cite[Theorem E.47]{DyZw:19} we have
\begin{equation}
\label{e:low}
\|XAX_j\chi u\|_{H_k^m}\leq C\|XAX_j\chi u\|_{H_k^{-N}}+Ck^{-N}\|u\|_{H_k^{-N}}\leq Ck\|BP_ku\|_{L^2}+\|E_j u\|_{L^2}+Ck^{-N}\|u\|_{H_k^{-N}}.
\end{equation}
Finally, since $\WF(E_j)\subset \{p_\theta\neq 0\}\cap \operatorname{Ell}(B)$, we have
\begin{equation}
\label{e:low2}
\|E_j u\|_{L^2}\leq C\|BP_ku\|_{L^2}+Ck^{-N}\|u\|_{H_k^{-N}}.
\end{equation}
Combining~\eqref{e:high},~\eqref{e:low}, and~\eqref{e:low2}, and summing in $j$,
$$
\|A\chi u\|_{H_k^m}\leq Ck\|BP_ku\|_{H_k^{m-2}}+Ck^{-N}\|u\|_{H_k^{-N}}.
$$

\end{proof}

We now prove the key propagation lemma that allows us to improve resolvent estimates away from trapping. In particular, we estimate $u$ in an annulus away from $\Omega_-$ but inside $B(0,R_{\operatorname{PML}_-})$ 
\begin{lemma}
\label{l:DV}
Let $R_{\operatorname{PML}_+}>R_{\operatorname{PML}_-}$ with $B(0,R_{\operatorname{PML}_+})\Subset \Omega_{\tr}$, $a\in C_c^\infty((R_{\operatorname{scat}},R_{\operatorname{PML}_-}))$ and $b\in C_c^\infty(R_{\operatorname{scat}},R_{\operatorname{PML}_+})$ with 
$$
b\equiv 1 \text{ on } \{ x\in (R_{\operatorname{scat}}, R_{\operatorname{PML}_-})\,:\, x\geq \inf \supp a\}.
$$ 
and define $A=a(|x|)$, $B=B(|x|)$. Then, for $X\in {}^b\Psi^0$ with ${}^b\WF(I-X)\cap {}^b\WF(P_ku)=\emptyset$, given $k_0>0$ there exists $C>0$ such that for all $k\geq k_0$
$$
\|A u\|_{L^2}^2\leq Ck\|P_ku\|_{L^2}\|X u\|_{L^2}+Ck^2\|BP_ku\|_{L^2}^2+C_Nk^{-N}\|u\|^2_{L^2}.
$$
\end{lemma}
\begin{proof}
Let  $a,b_1\in C_c^\infty((R_{\operatorname{scat}},R_{\operatorname{PML}_-}))$ with 
$a\prec b_1$,
$\supp b_1\cap \{b<\tfrac{1}{2}\} =\emptyset$. Let $b_2\in C_c^\infty(R_{\operatorname{scat}},R_{\operatorname{PML}_+})$ with $\supp b_2\cap \{b<\frac{1}{2}\}=\emptyset$ and 
$$
b_2\equiv 1 \text{ on } \big\{ x\in (R_{\operatorname{scat}}, R_{\operatorname{PML}_-})\,:\, x\geq \inf \supp b_1\big\}.
$$ 
Let $A=a(|x|)$ and $B_j=b_j(|x|)$, $j=1,2$.  We claim that 
\begin{equation}
\label{e:step1DV}
c\|Au\|_{H_k^s}^2\leq Ck\|P_ku\|_{L^2}\| Xu\|_{L^2}+k^2\|B_2P_ku\|_{L^2}^2+Ck^{-1}\|B_1u\|_{L^2}^2+Ck^{-N}\|u\|_{L^2}^2.
\end{equation}
To establish \eqref{e:step1DV}, first let $g\in C_c^\infty(\mathbb{R})$ with $\supp g\subset  [0, R_{\operatorname{PML}_-})$, $g\geq 0$, $g'\leq 0$, $\supp g'\subset (R_{\operatorname{scat}},R_{\operatorname{PML}_-})$ 
$g'\leq -1$ on $\supp a$, and $\supp(1-b_1)\cap \supp g'=\emptyset$. Next, let $E\in \Psi^{0}$ with $0\leq \sigma(E)\leq 1$, and
\begin{gather*}
\WF(E)\subset \Big\{(x,\xi)\,: \, \langle \tfrac{x}{|x|},\xi\rangle<2 \e\langle\xi\rangle\Big\},\qquad \WF(I-E)\cap \Big\{ (x,\xi)\,:\,\langle \tfrac{x}{|x|},\xi\rangle<\e\langle\xi\rangle\Big\}=\emptyset.
\end{gather*}
Finally, let $b_0\in C_c^\infty ((R_{\operatorname{scat}},R_{\operatorname{PML}_-})$ with 
$g' \prec b_0\prec b_1$.

Put $G= g(|x|)$, $B_0=b_0(|x|)$ and consider 
$$
k\Im \big\langle P_ku,G^2 u\big\rangle = \frac{k}{2i}\big\langle [P_k,G^2]u,u\big\rangle=\frac{k}{2i}\big\langle B_0[P_k,G^2]B_0B_1u,B_1u\big\rangle. $$
Now, define $Z:=\frac{k}{2i}B_0[P_k,G^2]B_0\in \Psi^1$ and observe that
\begin{align*}
\sigma(Z)=b_0^2g(|x|)\langle \xi,\partial_xg(|x|)\rangle &= b_0^2g(|x|)\langle \xi,\tfrac{x}{|x|}\rangle g'(|x|)\\
&=b_0^2\Big(g(|x|)\langle \xi,\tfrac{x}{|x|}\rangle g'(|x|)(1-\sigma(E^2))+g(|x|)\langle \xi,\tfrac{x}{|x|}\rangle g'(|x|)\sigma(E^2)\Big)\\
&\leq b_0^2\Big(-c\e a^2\langle\xi\rangle(1-\sigma(E^2))+g(|x|)\langle \xi,\tfrac{x}{|x|}\rangle g'(|x|)\sigma(E^2)\Big)\\
&\leq b_0^2\Big(-c\e a^2\langle\xi\rangle+C\sigma(E^2)|\xi|\Big)\\
&\leq b_0^2\Big(-c\e a^2\langle \xi\rangle+C\sigma(E^2)\langle \xi\rangle + Cp_\theta^2\Big)
\end{align*}
Therefore, by the microlocal G\aa rding inequality~\cite[Proposition E.34]{DyZw:19}, 
$$
\frac{k}{2i}\langle [P_k,G^2]B_1u,B_1u\rangle \leq -c_\e\|B_0AB_1u\|_{H^1_k}^2+C\|B_0EB_1u\|_{H^1_k}^2 +C\|B_0P_kB_1u\|_{L^2}^2 +Ck^{-1}\|B_1u\|_{L^2}^2
$$
Thus, since $a\prec b_0\prec b_1$ and ${}^b\WF(I-X)\cap {}^b\WF(P_ku)=\emptyset$,
$$
c_\e\|A u\|_{L^2}^2\leq Ck\|P_ku\|_{L^2}\| G  X u\|_{L^2}+C\|EB_1u\|_{H^1_k}^2+C\|B_1P_ku\|_{L^2}^2+Ck^{-1}\|B_1u\|_{L^2}^2 +Ck^{-N}\|u\|_{H_k^{-N}}^2
$$
Then, by Lemma~\ref{l:outgoing}, 
$$
c_\e\|Au\|_{L^2}^2\leq Ck\|P_ku\|_{L^2}\| X u\|_{L^2}+ Ck^2\| B_2P_ku\|_{L^2}^2+Ck^{-1}\|B_1u\|_{L^2}^2+Ck^{-N}\|u\|_{H_k^{-N}}^2
$$
as claimed in \eqref{e:step1DV}.

Now, suppose by induction that for  $a,b_1\in C_c^\infty((R_{\operatorname{scat}},R_{\operatorname{PML}_-}))$ with $a\prec b_1$ and $b_2\in C_c^\infty(R_{\operatorname{scat}},R_{\operatorname{PML}_+})$ with $\supp b_2\cap \{b<\frac{1}{2}\}=\emptyset$ and 
$$
b_2\equiv 1 \text{ on } \Big\{ x\in (R_{\operatorname{scat}}, R_{\operatorname{PML}_-})\,:\, x\geq \inf \supp b_1\Big\},
$$ we have, with $A=a(|x|)$ and $B_1=b_1(|x|)$, $B_2=b_2(|x|)$, 
\begin{equation}
\label{e:step2DV}
c\|Au\|_{L^2}^2\leq Ck\|P_ku\|_{L^2}\| X u\|_{L^2}+Ck^2\|B_2 P_ku\|_{L^2}^2+Ck^{-L}\|B_1u\|_{L^2}^2+Ck^{-N}\|u\|_{H_k^{-N}}^2.
\end{equation}

Now, fix $a,b_1\in C_c^\infty((R_{\operatorname{scat}},R_{\operatorname{PML}_-}))$ with $a\prec b_1$, and $b_2\in C_c^\infty(R_{\operatorname{scat}},R_{\operatorname{PML}_+})$ with $\supp b_2\cap \{b<\frac{1}{2}\}=\emptyset$ and 
$$
b_2\equiv 1 \text{ on } \big\{ x\in (R_{\operatorname{scat}}, R_{\operatorname{PML}_-})\,:\, x\geq \inf \supp b_1\big\}.
$$

Then, by~\eqref{e:step2DV}, letting $\widetilde{b}_1\in C_c^\infty((R_{\operatorname{scat}},R_{\operatorname{PML}_-})) $ with 
$a\prec \widetilde{b}_1 \prec b_1$
%$\supp (1-\widetilde{b}_1)\cap \supp a=\emptyset$, $\supp (1-b_1)\cap \supp \widetilde{b}_1=\emptyset$, 
and $\widetilde{b}_2\in C_c^\infty(R_{\operatorname{scat}},R_{\operatorname{PML}_+})$ with $\supp \widetilde{b}_2\cap \{b_2<\frac{1}{2}\}=\emptyset$ and 
$$
\widetilde{b}_2\equiv 1 \text{ on } \{ x\in (R_{\operatorname{scat}}, R_{\operatorname{PML}_-})\,:\, x\geq \inf \supp \widetilde{b}_1\},
$$
by~\eqref{e:step2DV} with $A=a(|x|)$ and $\widetilde{B}_1=b_1(|x|)$, $\widetilde{B}_2=b_2(|x|)$, 
\begin{equation*}
c\|Au\|_{L^2}^2\leq Ck\|P_ku\|_{L^2}\| X u\|_{L^2}+Ck^2\|\widetilde{B}_2P_ku\|_{L^2}^2+Ck^{-L}\|\widetilde{B}_1u\|_{L^2}^2+Ck^{-N}\|u\|_{H_k^{-N}}^2.
\end{equation*}
Now, by~\eqref{e:step1DV} with $A$ replaced by $\widetilde{B}_1$, 
\begin{equation*}
c\|\widetilde{B}_1u\|_{L^2}^2\leq Ck\|P_ku\|_{L^2}\| X u\|_{L^2}+Ck^{-1}\|B_1u\|_{L^2}^2+Ck^{-N}\|u\|_{H_k^{-N}}^2.
\end{equation*}
Hence,
\begin{equation*}
\begin{aligned}c\|Au\|_{L^2}^2&\leq Ck\|P_ku\|_{L^2}\| X u\|_{L^2}+Ck^2\|\widetilde{B}_2P_ku\|_{L^2}^2+Ck^{-L-1}\|B_1u\|_{L^2}^2+Ck^{-N}\|u\|_{H_k^{-N}}^2\\
&\leq  Ck\|P_ku\|_{L^2}\| X u\|_{L^2}+Ck^2\|B_2P_ku\|_{L^2}^2+Ck^{-L-1}\|B_1u\|_{L^2}^2+Ck^{-N}\|u\|_{H_k^{-N}}^2;
\end{aligned}
\end{equation*}
we have therefore obtained \eqref{e:step2DV} with $L$ replaced by $L+1$, and the result then follows by induction.
\end{proof}

\subsection{Estimates near the scatterer and away from trapping}

Before proceeding, we record the following consequences of~\cite[Proposition 4.6, Theorems 8.1 and 8.5]{Va:08} (see also~\cite[Appendix B]{GaRaSp:25} for a proof in the semiclassical setting).
\begin{theorem}\label{l:ellipticVasy}
Let $A,E\in {}^b\Psi(\Omega)$ with ${}^b\WF(A)\cup{}^b\WF(E)\subset \Omega_{\tr}$ with ${}^b\WF(A)\subset {}^b\Ell(E)$ and ${}^b\WF(A)\cap {}^b\pi(\{p_\theta=0\})=\emptyset$. Then, for all $k_0>0$ there exists $C>0$ such that 
$$
\|Au\|_{L^2}\leq C\|EP_ku\|_{L^2}+C_Nk^{-N}\|u\|_{L^2}
$$
\end{theorem}
\begin{proof}
The estimate follows from~\cite[Proposition 4.6]{Va:08} when ${}^b\WF(A)\subset B(0,R_{\operatorname{scat}})$ and from the standard elliptic parametrix construction~\cite[Proposition E.32]{DyZw:19}, when ${}^b\WF(A)\cap T^*\partial \Omega_-=\emptyset$. 
\end{proof}

\begin{theorem}\label{l:propagateVasy}
Let $A,B,E\in {}^b\Psi(\Omega)$ such that 
\begin{gather*}
{}^b\WF(A)\cup\bigcup_{t=0}^T \varphi_{-t}({}^b\WF(A)\cap\pi^b(\{ \Re p_\theta=0\})\subset {}^b\Ell(E),\\
\text{ and for all }  q\in {}^b\WF(A)\cap\pi^b(\{ \Re p_\theta=0\}),\qquad \bigcup_{t=0}^T\varphi_{-t}(q)\cap {}^b\Ell(B)\neq \emptyset.
\end{gather*}
 Then, for all $k_0>0$ there exists $C>0$ such that 
$$
\|Au\|_{L^2}\leq Ck\|EP_ku\|_{L^2}+\|Bu\|_{L^2}+C_Nk^{-N}\|u\|_{L^2}
$$
\end{theorem}
\begin{proof}
The estimates follow from the combination of the propagation results in~\cite[Theorem 8.1]{Va:08} (for Dirichlet boundary conditions on $\partial\Omega_-$) and \cite[Theorem 8.5]{Va:08} (for Neumann boundary conditions on $\partial\Omega_-$) applied near the $\partial\Omega$ and~\cite[Theorem E.47]{DyZw:19} applied away from $\partial\Omega_-$. 
\end{proof}

\

Our next lemma shows that, to measure $u$ away from trapping, we need only have an estimate for $u$ in an annulus.
\begin{lemma}
\label{l:propagate}
Suppose that $A\in {}^b\Psi(\Omega)$ and ${}^b\WF(A)\cap K=\emptyset$. Then, for any $R_{\operatorname{scat}}<R_1<R_{\operatorname{PML}_-}$ and $B\in C_c^\infty(\Omega)$ with $\supp (1-B)\cap \{|x|=R_1\}=\emptyset$, given $k_0>0$ there exists $C>0$ such that for all $k\geq k_0$
\beq\label{e:savingSpace}
\|Au\|_{L^2}\leq Ck\|P_ku\|_{L^2}+\|Bu\|_{L^2}+C_Nk^{-N}\|u\|_{L^2}^2.
\eeq
\end{lemma}
\begin{proof}
First, by Lemma~\ref{l:nearPMLBoundary} %and~\ref{l:outgoing} and 
we may assume that 
$$
{}^b\WF(A)\subset \Omega_{\tr}.%\qquad {}^b\WF(A)\cap \{|x|>R_{\operatorname{scat}}\}\subset \{ \langle \tfrac{x}{|x|},\xi\rangle \geq \e\}.
$$
Next observe that if ${}^b\WF(A)\cap {}^b\pi(\{p_\theta=0\})=\emptyset$, then by the ellipticity results in Theorem~\ref{l:ellipticVasy}
$$
\|Au\|_{L^2}\leq C\|P_ku\|_{L^2}+C_Nk^{-N}\|u\|_{L^2}.
$$
Therefore, we may assume that ${}^b\WF(A)$ is contained in a small neighbourhood of ${}^b\pi(\{p_\theta=0\})$. 

If ${}^b\WF(A)\subset \{B\equiv 1\}$ then, by the elliptic parametrix \cite[Proposition E.32]{DyZw:19},
$$
\|Au\|_{L^2}\leq C\|Bu\|_{L^2}+Ck^{-N}\|u\|_{L^2}.
$$ 
Therefore, using a partition of unity, we need only consider two cases: ${}^b\WF(A)\subset \{|x|>R_1\}$ and ${}^b\WF(A)\subset \{|x|<R_1\}$. 

First, suppose that ${}^b\WF(A)\subset \{|x|>R_1\}$. Let $U$ be as in Lemma~\ref{l:nearPMLBoundary}. Then, there exists $T>0$ such that for all $(x,\xi)\in {}^b\WF(A)$, there is $t\in[0,T]$ such that 
$$
\varphi_{-t}(x,\xi)\in \{B\equiv 1\}\cup \{(x,\xi)\,:x\in U\}
$$
(i.e., flowing backwards, one either hits $B\equiv 1$ or the PML).
In particular, by the propagation results~\cite[Theorem E.47]{DyZw:19}
 there is $\psi \in C^\infty(U)$ with $\psi\equiv 1$ near $\Gamma_{\tr}$ such that 
$$
\|Au\|_{L^2}\leq Ck\|P_ku\|_{L^2}+\|Bu\|_{L^2}+\|\psi u\|_{L^2}+Ck^{-N}\|u\|_{L^2}.
$$
By Lemma~\ref{l:nearPMLBoundary}, we then obtain
$$
\|Au\|_{L^2}\leq Ck\|P_ku\|_{L^2}+\|Bu\|_{L^2}+Ck^{-N}\|u\|_{L^2}
$$
as required.

Next, suppose ${}^b\WF(A)\subset\{|x|<R_1\}$. Then, since ${}^b\WF(A)\cap K=\emptyset$, applying a partition of unity again, we may assume there exists $T>0$ such that for all $(x,\xi)\in {}^b\WF(A)$, either there is $t\in[0,T]$ such that 
$$
\varphi_t(x,\xi)\in \Big\{(x,\xi)\,:\, B(x)>\tfrac{1}{2}\Big\},\qquad \bigcup_{s\in[0,t]}\varphi_s(x,\xi)\in B(0,R_{\operatorname{PML}_-}),
$$
(informally, one flows forwards from $A$, staying away from the PML region, and reaches where $B>1/2$ at time $t$)
or there is $t\in[0,T]$ such that 
$$
\varphi_{-t}(x,\xi)\in \Big\{(x,\xi)\,:\, B(x)>\tfrac{1}{2}\Big\},\qquad \bigcup_{s\in[-t,0]}\varphi_s(x,\xi)\in B(0,R_{\operatorname{PML}_-}).
$$
(informally, one flows backwards from $A$, staying away from the PML region, and reaches where $B>1/2$ at time $t$).
The result \eqref{e:savingSpace} then follows by the propagation results of Theorem \ref{l:propagateVasy}.
\end{proof}

\

Finally, we combine the above lemmas to show that we may estimate $u$ away from trapping by $u$ near the wavefront set of $P_ku$. In particular, this will improve the resolvent estimate when the measurement is away from trapping.
\begin{lemma}
\label{l:halfAway}
Suppose that $A\in {}^b\Psi(\Omega)$ and ${}^b\WF(A)\cap K=\emptyset$. Then, for any $X\in {}^b\Psi^0$ with ${}^b\WF(I-X)\cap {}^b\WF(P_ku)=\emptyset$, given $k_0>0$ there exists $C>0$ such that for all $k\geq k_0$
$$
\|Au\|^2_{L^2}\leq Ck\|P_ku\|_{L^2}\|Xu\|_{L^2}+Ck^2\|P_ku\|_{L^2}^2+C_Nk^{-N}\|u\|_{L^2}^2.
$$
\end{lemma}
\begin{proof}
Let $R_{\operatorname{scat}}<R_1<R_{\operatorname{PML}_-}$ and $b\in C_c^\infty(R_{\operatorname{scat}},R_{\operatorname{PML}_-})$ with $\supp (1-b)\cap \{|x|=R_1\}=\emptyset. $ Then, by Lemma~\ref{l:propagate}
$$
\|Au\|^2_{L^2}\leq Ck^2\|P_ku\|_{L^2}^2+C\|Bu\|_{L^2}^2+C_Nk^{-N}\|u\|_{L^2}^2,
$$
and, by Lemma~\ref{l:DV},
$$
\|Bu\|_{L^2}^2\leq Ck\|P_ku\|_{L^2}\|X u\|_{L^2}+Ck^2\|P_ku\|_{L^2}^2+C_Nk^{-N}\|u\|_{L^2}^2,
$$
which completes the proof.
\end{proof}

\

When the both the data and measurement are away from trapping, we can use the previous lemma to improve our estimates further-- all the way to a non-trapping type bound.
\begin{lemma}
\label{l:fullAway}
Suppose that ${}^b\WF(P_ku)\cap K=\emptyset$, then for any $A\in {}^b\Psi(\Omega)$ with ${}^b\WF(A)\cap K=\emptyset$ given $k_0>0$ there exists $C>0$ such that for all $k\geq k_0$
$$
\|Au\|_{L^2}\leq Ck\|P_ku\|_{L^2}+Ck^{-N}\|u\|_{L^2}.
$$
\end{lemma}
\begin{proof}
Let $\tilde{A},X\in {}^b\Psi(\Omega)$ with ${}^b\WF(I-X)\cap {}^b\WF(P_ku)=\emptyset$, ${}^b\WF(\tilde{A})\cap K=\emptyset$, and ${}^b\WF(I-\tilde{A})\cap( {}^b\WF(A)\cup{}^b\WF(X))=\emptyset$. 
By Lemma~\ref{l:halfAway},
$$
\|\tilde{A}u\|_{L^2}^2\leq Ck\|P_ku\|_{L^2}\|Xu\|_{L^2}+Ck^2\|P_ku\|_{L^2}^2+Ck^{-N}\|u\|_{L^2}.
$$
Then, by the elliptic parametrix construction in the b-calculus~\cite[Equation 3.11]{GaWu:23} (see also \cite[Appendix A]{HiVa:18}),
$$
\|Xu\|_{L^2}\leq 
C\|\tilde{A}u\|_{L^2}+Ck^{-N}\|u\|_{L^2}
$$
Combining the last two inequalities and using the inequality \eqref{e:peterPaul}, we obtain that
$$
\|\tilde{A}u\|_{L^2}^2\leq Ck^2\|P_ku\|_{L^2}^2+Ck^{-N}\|u\|_{L^2}^2.
$$
Finally, since $\tilde{A}$ is elliptic on $\WF(A)$,
$$
\|Au\|_{L^2}\leq  C\|\tilde{A}u\|_{L^2}+Ck^{-N}\|u\|_{L^2},
$$
which completes the proof
\end{proof}

\subsection{Proof of Theorem \ref{thm:DV2}}%Bounds on the solution operator}

%The following two lemmas prove Theorem \ref{thm:DV2}.

%\begin{lemma}[Either data or measurement is away from trapping]
%Let $k_0>0$, and  let $\blueOld{\mathrm{J}}$ be such that 
%Assumption \ref{a:polyBoundIntro} holds. 
%Then for all $A\in {}^b\Psi(\Omega)$ with ${}^b\WF(A)\cap K=\emptyset$, there exists $C>0$ such that for all $k\in (k_0,\infty)\setminus \blueOld{\mathrm{J}}$
%$$
%\|AR_k\|_{L^2\to L^2}+\|R_kA\|_{L^2\to L^2}\leq C\sqrt{\|R_k\|k}.
%$$
%\end{lemma}
%\begin{proof}

To prove the first bound in Theorem \ref{thm:DV2}, 
let $u=R_kf$. Then, by Lemma~\ref{l:halfAway} with $X=I$, 
\begin{align*}
\|Au\|^2_{L^2}&\leq Ck\|P_ku\|_{L^2}\|u\|_{L^2}+Ck^2\|P_ku\|_{L^2}^2+C_Nk^{-N}\|u\|_{L^2}^2\\
&\leq Ck\|R_k\|_{L^2\to L^2}\|f\|^2_{L^2}+Ck^2\|f\|_{L^2}^2+C_Nk^{-N}\|R_k\|_{L^2\to L^2}^2\|f\|^2\\
&\leq C(k\|R_k\|_{L^2\to L^2}+k^2)\|f\|_{L^2}^2.
\end{align*}
In particular, 
$$
\|AR_k\|_{L^2\to L^2}\leq C(\sqrt{k\|R_k\|_{L^2\to L^2}}+k)\leq C\sqrt{k\|R_k\|_{L^2\to L^2}},
$$
where the last inequality follows since $\|R_k\|_{L^2\to L^2}\geq ck$. 

Reversing the direction of the flow in all of the above lemmas (or, equivalently, applying the above results to $-P_k^*$), the proof of Lemma~\ref{l:halfAway} also yields
$$
\|A^*u\|_{L^2}^2\leq Ck\|P_k^*u\|_{L^2}\|u\|_{L^2}+Ck^2\|P_k^*u\|_{L^2}^2+C_Nk^{-N}\|u\|_{L^2}^2.
$$
Therefore, putting $u=R_{k}^*f$, and arguing in the same way as above, we obtain
$$
\|R_kA\|_{L^2\to L^2}=\|A^*R_{k}^*\|_{L^2\to L^2}\leq C\sqrt{k\|R_k\|_{L^2\to L^2}}.
$$
%\end{proof}

%\begin{lemma}[Both data and measurement are away from trapping]
%Let $k_0>0$, and  let $\blueOld{\mathrm{J}}$ be such that 
%Assumption \ref{a:polyBoundIntro} holds. 
%Then for all $A,B\in {}^b\Psi(\Omega)$ with ${}^b\WF(A)\cap K={}^b\WF(B)\cap K=\emptyset$, there exists $C>0$ such that for all $k\in (k_0,\infty)\setminus \blueOld{\mathrm{J}}$,
%$$
%\|AR_kB\|_{L^2\to L^2}\leq Ck.
%$$
%\end{lemma}
%\begin{proof}
To prove the second bound in Theorem \ref{thm:DV2}, let $u=R_kAf$. Then, by Lemma~\ref{l:fullAway}, since $P_ku=Af$, and ${}^b\WF(A)\cap K=\emptyset$,
$$
\|Au\|_{L^2}\leq Ck\|Af\|_{L^2}+Ck^{-N}\|u\|_{L^2}\leq Ck\|f\|_{L^2}+Ck^{-N}\|R_k\|_{L^2\to L^2}\|f\|_{L^2}.
$$
%\end{proof}

\subsection{Proof of Theorems~\ref{thm:negligible2} and~\ref{thm:negligible3}}

%If 
%$$
%{}^b\WF(A)\subset \Omega\setminus \overline{B(0,R_{\operatorname{scat}})}
%$$
%and
%$$
%{}^b\WF(A)\cap \Big\{ \big\langle \tfrac{x}{|x|},\xi\big\rangle\geq \e,\,p_\theta(x,\xi)=0\Big\}=\emptyset,
%$$ 
%the Lemma follows from Lemma~\ref{l:outgoing}. Therefore,  we may assume that 
%\begin{equation}
%\label{e:outgoingABC}
%{}^b\WF(A)\subset  \Big\{ (x,\xi)\,:\, |x|>R_{\operatorname{scat}},\,\big\langle \tfrac{x}{|x|},\xi\big\rangle\geq \e/2,\,|p_\theta(x,\xi)|<\e\Big\}\cup \overline{B(0,R_{\operatorname{scat}})}.
%\end{equation}

\begin{proof}[Proof of Theorem~\ref{thm:negligible2}]
Since ${}^b\WF(A)\cap \Gamma_+=\emptyset$ there exists $T>0$ and $B_1\in \Psi(\Omega)$ such that  $\WF(B_1)\subset ((T^*\Omega\setminus \overline{T^*B(0,R_{\operatorname{scat}})})\cap \{p_\theta\neq 0\})\setminus {}^b\WF(B)$ 
and for all $q\in {}^b\WF(A)\cap \pi^b(\{ \Re p_\theta=0\})$, 
\begin{equation*}
%\label{e:itHasEscaped}
\bigcup_{t=0}^T\varphi_{-t}(q)\cap \Ell(B_1)\neq \emptyset
\end{equation*}
(informally, $B_1$ is supported in the PML region away from $B$, and flowing backwards from $A$ one hits $B_1$).
In addition, since 
$$
{}^b\WF(A)\cup\bigcup_{t=0}^T\varphi_{-t}\big({}^b\WF(A)\cap\pi^b(\{ \Re p_\theta=0\})\big)\cap {}^b\WF(B)=\emptyset,
$$
there exists $E\in {}^b\Psi(\Omega)$ such that 
$$
{}^b\WF(A)\cup\bigcup_{t=0}^T\varphi_{-t}\big({}^b\WF(A)\cap \pi^b(\{ \Re p_\theta=0\})\big)\subset {}^b\Ell(E),\qquad {}^b\WF(E)\cap {}^b\WF(B)=\emptyset.
$$
Therefore, applying Theorem~\ref{l:propagateVasy} with $u=R_kBf$, and then using both ${}^b\WF(E)\cap {}^b\WF(B)=\emptyset$ and Assumption \ref{a:polyBoundIntro}, we obtain
\begin{equation}
\label{e:propBack}
\|Au\|_{L^2}\leq Ck\|EBf\|_{L^2}+C\|B_1u\|_{L^2}+C_Nk^{-N}\|u\|_{L^2}\leq \|B_1u\|_{L^2}+C_Nk^{-N}\|f\|_{L^2}.
\end{equation}
The elliptic parametrix construction~\cite[Proposition E.32]{DyZw:19} then implies that
\begin{equation}
\label{e:ellipBack}
\|B_1u\|_{L^2}\leq C\|B_1Bf\|_{L^2}+C_Nk^{-N}\|u\|_{L^2}\leq C_Nk^{-N}\|f\|_{L^2}.
\end{equation}
Combining~\eqref{e:propBack} and~\eqref{e:ellipBack}, we obtain that
$$
\|Au\|_{L^2}\leq C_Nk^{-N}\|f\|_{L^2}
$$
and the result $\|AR_kB\|_{L^2\to L^2}\leq C_Nk^{-N}$ follows.
%Now, when $A,B\in C^\infty(\overline{\Omega})$, then $\supp A\cap \supp B=\emptyset$. Let $\tilde{A}\in C^\infty(\overline{\Omega})$ such that $A\prec \tilde{A}$ and $\tilde{A}$ satisfies back flow. Then, elliptic regularity up to the boundary implies
%$$
%\|Au\|_{H_k^N}\leq C_N\|\tilde{A}u\|_{L^2}\leq C_Nk^{-N}\|f\|_{L^2},
%$$
%and hence $\|AR_k B\|_{L^2\to H_k^N}\leq C_Nk^{-N}$ as claimed.
\end{proof}

\

The proof of Theorem~\ref{thm:negligible3} is nearly identical with $P_k$ replaced by $-P_k^*$.

\section{Spaces of mapped piecewise polynomials}

In this section, we define carefully our notion of mapped piecewise polynomial spaces that allow for curved elements following~\cite[Appendix A]{GaSp:25b}. 
%We don't use more standard reference because....
We highlight that the presentation below allows general $C^{p+1}$ element maps, as opposed to, e.g., \cite[Chapter 13]{ErGu:21} which considers polynomial element maps, or \cite{Be:89}, which considers element maps that are small perturbations of affine maps (\cite[Definition 2.1]{Be:89}). This allows us to resolve general $C^{p+1}$ boundaries with our elements.

\subsection{Definition of spaces of mapped polynomials}

\begin{definition}[Conforming mesh]
\label{d:conformingMesh}
A finite collection of closed sets $\mathcal{T}$ is a \emph{conforming mesh} of $\Omega\subset \mathbb{R}^d$ if
\begin{itemize}
\item[(i)] for each $T \in \mathcal{T}$, $T$ is closed, $\mathring{T}$ is nonempty and connected,
\item[(ii)] $\overline{\Omega}= \cup_{T\in \mathcal{T}}T$
\item[(iii)] if $T_1, T_2\in \mathcal{T}$ and $\mathring{T_1}\cap \mathring{T_2} \neq \emptyset$ then $T_1=T_2$.
\item[(iv)] for each $T \in \mathcal{T}$, $T$ has Lipschitz boundary and 
$$
\partial T=\bigcup_{j=1}^d \bigcup_{i=1}^{N_{j,d}}F_{j,i}(T)
$$
with $F_{j,i}\blueOld{(T)}$ a smooth, open, embedded submanifold of dimension $d-j$ such that $\partial \big(F_{j,i}\blueOld{(T)}\big)=\partial\overline{F_{j,i}\blueOld{(T)}}$,
\item[(v)] if $T_1,T_2\in\mathcal{T}$ and $F_{j_1,i_1}(T_1)\cap F_{j_2,i_2}(T_2)\neq \emptyset$, then $F_{j_1,i_1}(T_1)=F_{j_2,i_2}(T_2)$. 
\end{itemize}
\end{definition}

% \begin{definition}[Shape-regular sequence of triangulations]
% $\{\mathcal{T}_h\}_{h>0}$ is \emph{shape regular}
% (also called \emph{regular} or \emph{non-degenerate}) if there exists $C>0$ such that 
% for all $h>0$, for all $T_h\in \mathcal{T}_h$, $\gamma(T_h)\leq C$.
% \end{definition}

\begin{definition}[$m$-simplex]
Let $m>0$. A set $\blueOld{T}\Subset \mathbb{R}^d$ is an \emph{$m$-simplex} if there are $\{x_j\}_{j=0}^m\subset \mathbb{R}^d$ such that
$$
\blueOld{T}=\text{convex hull}\big(\{ x_j\}_{j=0}^m\}\big)
$$
$\blueOld{T}$ is an \emph{open $m$-simplex} if $\blueOld{T}=(\widetilde{\blueOld{T}})^\circ$ where $\widetilde{\blueOld{T}}$ is an $m$ simplex  and $\blueOld{T}\neq \emptyset$. 
\end{definition}

\begin{definition}[Reference element and element maps]
\label{d:referenceElement}
$\widehat{T}\Subset \mathbb{R}^d$ is a reference element for a mesh $\mathcal{T}$ if $\operatorname{diam}(\widehat{T})=1$ and there exist a family of bi-Lipschitz maps 
$\{\mapF_{T}\}_{T\in \mathcal{T}}$ such that $T = \mapF_T( \widehat{T})$ for all $T\in \mathcal{T}$ \blueOld{and for all $(j,i)$ there is $i'$ such that  $F_{j,i}(T)=\mapF_T(F_{j,i'}(\widehat{T}))$. 
}The mesh $\mathcal{T}$ is \emph{simplicial} if $\widehat{T}$ is a $d$ simplex.
\end{definition}

Given a mesh $\mathcal{T}$ with element maps $\{\mapF_T\}_{T\in\mathcal{T}}$ and reference element $\widehat{T}$, the corresponding spaces of mapped polynomials are defined by
\beq\label{e:polySpace}
V^p(\mathcal{T},\{\mapF_T\}):= \big\{ 
v\in H^1(\Omega) \, :\, \text{ for all }  T \in \cT : v|_T\circ \mapF_T \in \mathbb{P}^p(\widehat{T})
\big\},
\eeq
where $\mathbb{P}^p(\widehat{T})$ denotes the space of polynomials of degree $\leq p$ on the reference element. We abbreviate $V^p(\mathcal{T},\{\mapF_T\})$ to $V^p_{\mathcal{T}}$.

We use the following measure of how star-shaped an open subset of $\mathbb{R}^d$ is. 
\begin{definition}[Shape-regularity constant]
Let $T\subset \mathbb{R}^d$ and define
$$
\rho_{\max}(T):=\sup \big\{ \rho>0\,:\, T\text{ is star-shaped with respect to a ball of radius }\rho\big\}.
$$
The \emph{shape-regularity constant of $T$} is defined by
\begin{equation}
\label{e:shapeRegularity}
\gamma(T):=\frac{\operatorname{diam}(T)}{\rho_{\max}(T)}.
\end{equation}
\end{definition}
%(Observe that $\gamma(T)=2$ when $T$ is a ball.)

%The reference element is usually taken to be a $d$-simplex. 
\begin{definition}[$C^r$ conforming mesh]
\label{d:Crtriang}
A conforming mesh is $C^r$ with constant $\Upsilon>0$ if $\gamma(\widehat{T})\leq \Upsilon$ and for each $T\in\mathcal{T}$ there is $h_T>0$ such that the  element map $\mapF_T$ can be written as $R_T \circ A_T$ where $A_T$ is an affine map and 
\begin{align}
&\| \partial A_T\|_{L^\infty} \le \Upsilon h_T, \quad 
\| (\partial A_T)^{-1}\|_{L^\infty}\leq \Upsilon h^{-1}_T,\quad \| (\partial R_T)^{-1}\|_{L^\infty}\leq \Upsilon,\nonumber\\
&
\| \partial^\alpha R_T\|_{L^\infty}\leq \Upsilon^{1+|\alpha|} \alpha! \quad\tfa\, |\alpha|\leq r.\label{e:mapsHighReg}
\end{align}
Let $h(\mathcal{T}):=\sup_{T\in\mathcal{T}}h_T$ be the \emph{width of the mesh $\mathcal{T}$}. Furthermore, a mesh is $C^\omega$ if~\eqref{e:mapsHighReg} holds for all $|\alpha|<\infty$.
\end{definition}

\bre 
Given a piecewise $C^{r}$ domain, a $C^r$ mesh is constructed in \cite[\S6]{Be:89}, building on the 2-d results of \cite{Sc:73, Zl:73} and the isoparametric elements in general dimension of \cite{Le:86}.
\ere

% \bre[Shape regularity]
% A collection of triangulations, $\mathscr{T}$, is \emph{shape regular} in the sense of \cite[Definition 11.2]{ErGu:21}, \cite[Definition 4.4.13]{BrSc:08} if
% \beqs
% \sup_{\mathcal{T}\in\mathscr{T}}\sup_{T\in \mathcal{T}}\frac{h_T}{r_{\max}(T)}<\infty,
% \eeqs
% where
% \beqs
% r_{\max}(T):=\sup \big\{r>0\,:\, \text{there is }x_0\in T\text{ with }B(x_0,r)\subset T\big\}.
% \eeqs
% Given $\Upsilon>0$, the collection of $C^1$ triangulations with this $\Upsilon$ is then shape regular. Indeed, the bounds on $\partial \mapF_T$ and $\partial (\mapF_T)^{-1}$ from \eqref{e:mapsHighReg} 
% and the fact that $\operatorname{diam}(\widehat{T})=1$
% imply that
% $$
% \Upsilon^{-2}h_T\leq \operatorname{diam}(T)\leq \Upsilon^2 h_T,
% $$
% and 
% $$
% \Upsilon^{-2}h_T r_{\max}(\widehat{T})\leq r_{\max}(T)\leq \Upsilon^2 h_T r_{\max}(\widehat{T})
% $$
% and hence, since $\rho_{\max}(\widehat{T})\leq r_{\max}(\widehat{T})$,
% % $$
% % %\Upsilon^{-2}\gamma(\widehat{T})\leq
% % \frac{h_T}{r_{\max}(T)}
% % \leq 
% % \frac{
% % \Upsilon^2\operatorname{diam}(T)
% % }{
% % \operatorname{diam}(\widehat{T}) 
% % }
% % \Upsilon^2\gamma(\widehat{T}).
% % $$
% $$
% \frac{h_T}{r_{\max}(T)}\leq \frac{\Upsilon^2}{r_{\max}(\widehat{T})}\leq \frac{\Upsilon^2}{\rho_{\max}(\widehat{T})}=\Upsilon^2\gamma(\widehat{T}).
% $$
% \ere

\begin{definition}[Affine-conforming mesh]
\label{d:AffineConforming}
A $C^r$ mesh (in the sense of Definition \ref{d:Crtriang}) is \emph{affine conforming} if whenever $F_{j_1,i_1}(T_1)\cap F_{j_2,i_2}(T_2)\neq \emptyset$ there is an affine isomorphism $\kappa:\mapF_{T_1}^{-1}(F_{j_1,i_1}(T_1))\to \mapF_{T_2}^{-1}(F_{j_2,i_2}(T_2))$  such that 
\beq\label{e:isomorphism}
\mapF_{T_1}|_{
\overline{
\mapF_{T_1}^{-1}(F_{j_1,i_1}(T_1))
}}
=\mapF_{T_2}|_{
\overline{
\mapF_{T_2}^{-1}(F_{j_2,i_2}(T_2))
}}\circ \kappa.
\eeq
\end{definition}

% \begin{definition}[Isometry-conforming triangulation]
% \label{d:isometryConforming}
% We call a $C^r$ triangulation (in the sense of Definition \ref{d:Crtriang}) \emph{isometry conforming} if whenever $F_{j_1,i_1}(T_1)\cap F_{j_2,i_2}(T_2)\neq \emptyset$ there is an affine isometry $\kappa:\mapF_{T_1}^{-1}(F_{j_1,i_1}(T_1))\to \mapF_{T_2}^{-1}(F_{j_2,i_2}(T_2))$  such that 
% \beq\label{e:isometry}
% \mapF_{T_1}|_{\mapF_{T_1}^{-1}(F_{j_1,\es{i}_1}(T_1))}=\mapF_{T_2}|_{\mapF_{T_2}^{-1}(F_{j_2,i_2}(T_2))}\circ \kappa.
% \eeq
% \end{definition}

% We use the adjective \emph{conforming} in Definitions \ref{d:AffineConforming} 
% %and~\ref{d:isometryConforming} 
% because the additional property~\eqref{e:isomorphism} 
% %and \eqref{e:isometry} 
% is used to generate an $H^1$-conforming polynomial approximant in
% %respectively Theorems~\ref{t:approxInterpolant} and 
% \ref{t:approxHighLowReg} below.

\bre
Without a condition on the element maps such as Definitions \ref{d:AffineConforming}, %and \ref{d:isometryConforming} 
the space of mapped polynomials on two adjacent mesh elements 
may not even intersect non-trivially when restricted to a common face. If the intersection is trivial, there cannot be any non-trivial, globally-$H^1$ mapped polynomials. 

Conditions such as Definitions \ref{d:AffineConforming} 
%and \ref{d:isometryConforming} 
are therefore common in the literature. 
%see, e.g., \cite[Page 627, Point (c)]{BaCrMaPi:91}, \cite[\S3.2 (iv)]{MeSa:20}. 
For example, the Lagrange $\mathbb{P}_{k,d}$ elements in \cite{ErGu:21} are affine conforming (in the sense of Definition \ref{d:AffineConforming}) by 
 \cite[Definition 8.1 and Exercise 20.1]{ErGu:21}. 
 
 \blueOld{In general, given a bounded, connected domain $\Omega$ with $C^{1}$ boundary, it is possible to construct an affine-conforming mesh that is $C^p$ for all $p$. Indeed, the Whitehead theorem~\cite{Wh:40} implies that there is a map $\Phi:P\to \Omega$ from a polyhedron (not necessarily simply connected) to $\Omega$ that is bilipschitz and such that there is a decomposition of $P$ into $d+1$ simplices, $\{\tilde{T}_i\}_{i=1}^N$ such that $\Phi|_{\overline{\tilde{T}_i}}$ is a diffeomorphsim. If desired, one then further decomposes $\{\tilde{T}_i\}$ into smaller simplices $\{T_j
 \}_{j=1}^M$, which can be mapped to $\widehat{T}$ using an affine map, $A_{T_j}$.   The maps $\mathscr{F}_{T_j}$ can then be chosen as $P\circ A_{T_j}^{-1}$. One can see that the constant $\Upsilon$ in this construction can be controlled in terms of only the map $P$ i.e. it need not increase upon further refinement.  
 }
\ere

% \bre\label{r:isomorphism}
% Recall that an affine isomorphism, $\kappa$, takes the form
% $
% \kappa(x)= x_0+Ax,
% $
% where $x_0\in\mathbb{R}^d$ and $A$ is an invertible matrix, while an
% affine isometry, $\kappa$, takes the form
% $
% \kappa(x)= x_0+Ox,
% $
% where $x_0\in\mathbb{R}^d$ and $O$ is an orthogonal matrix. We highlight that the proof of Theorem~\ref{t:approxHighLowReg} below does not apply when the triangulation is only affine conforming rather than isometry conforming (and the authors are not aware of a proof of a $p$-explicit approximation result that does).  This is because the operators $Q_{j,i}$ in Proposition \ref{p:toBoundaryCompatible} do not commute with affine isomorphisms (see~\eqref{e:whatIsQij} for the definition of $Q_{j,i}$ and \eqref{e:commute1} for the commutation). 
% \ere
\subsection{Quasi-interpolation in spaces of mapped polynomials}

\begin{definition}[Lagrange nodes of a simplex]
    For a simplex $T$ with vertices $\{x_i\}_{i=1}^m$, let 
    $$
    \mathcal{L}_p(T):=\bigg\{ \frac{1}{p}\sum_{i=1}^{n+1}\alpha_i x_i\,:\, \alpha \in\mathbb{N}^{m},\,|\alpha|=p\bigg\}
    $$
    denote the \emph{$p$-Lagrange nodes} of $T$. 
\end{definition}
\begin{definition}[Lagrange nodes of an affine-conforming mesh]
    For a simplicial affine-conforming mesh $\mathcal{T}$, we define the \emph{$p$-Lagrange nodes of the mesh} by 
    $$
    \mathcal{L}_p(\mathcal{T}):=\bigcup_{T\in\mathcal{T}}\mathscr{F}_{T}(\mathcal{L}_p(\widehat{T})).
    $$
    \end{definition}

    Observe that Lagrange nodes are invariant under affine isomorphisms; i.e., if $A$ is an affine isomorphism and $T$ is a simplex, then $\mathcal{L}_p(A(T))=A(\mathcal{L}_p(T))$. Together with Definition~\eqref{d:AffineConforming}, this implies that if $T_1,T_2\in\mathcal{T}$ share a $(d-j)$-dimensional face $F_{j,i_1}(T_1)=F_{j,i_2}(T_2)$, then $\mathscr{F}_{T_1}(L_p(\widehat{T}))\cap F_{j,i_1}(T_1)=\mathscr{F}_{T_2}(\mathcal{L}_p(\widehat{T}))\cap F_{j,i_2}(T_2)=\mathscr{F}_{T_1}(\mathcal{L}_p(\hat{T} \cap F_{j,i_1}^{-1}(T_1))) = \mathscr{F}_{T_2}(\mathcal{L}_p(\hat{T} \cap F_{j,i_2}^{-1}(T_2)))$. 

    \begin{definition}
\label{d:lagrangeBasisSimplex}
        For a simplex $T$, denoting $L_p(T)=\{v_i\}_{i=1}^N$, we define the \emph{Lagrange basis of degree $p$-polynomials for $T$}, $\mathscr{L}_p(T):=\{\phi_{i}\}_{i=1}^N$ by
        \begin{equation}
        \label{e:defineLagrange}
        \phi_i(v_j)=\delta_{ij},\qquad \phi_i\in\mathbb{P}^p.
        \end{equation}
            For $S\subset T$ a simplex, we also define the \emph{dual Lagrange basis of degree $p$ polynomials}, $\mathscr{L}_p^*(S,T):=\{\psi_{i,S}\}_{i=1}^N\subset \mathbb{P}^p(S) $
    $$
    \int_{S} \psi_{i,S} \phi_jd\textup{vol}_S=\delta_{ij}.
    $$
    \end{definition}
    Notice that~\eqref{e:defineLagrange} defines a unique polynomial of degree $p$ since the $p$-Lagrange nodes of a simplex are unisolvent for degree $p$ polynomials (see~\cite[Proposition 7.12]{ErGu:21}). This also implies that $\mathscr{L}_p(T)$ is, in fact, a basis for degree $p$ polynomials on $T$. 

        \begin{lemma}
        Let $\mathcal{T}$ be a simplicial affine conforming mesh. Then for all $v\in\mathcal{L}_p(\mathcal{T})$, there is a unique $\phi\in V_{\mathcal{T}}^p$ such that 
           \begin{equation}
    \phi(v)=1,\qquad \phi(v')=0,\quad v'\in L_p(\mathcal{T})\setminus\{v\}. 
    \end{equation}
    \end{lemma}
    \begin{proof}
    Define the function $f_v:L_p(\mathcal{T})\to \{0,1\}$ by 
    $$
    f_v(v)=1,\qquad f_v(v')=0,\, v'\neq v.
    $$
    For $T\in\mathcal{T}$, let $\hat{\phi}_T\in \mathscr{L}_p(\widehat{T})$ be the $p$-Lagrange polynomial on $\widehat{T}$ satisfying
        $$
        \hat{\phi}_{T}(w)= f_v(\mathscr{F}_T(w)) \,, \quad \tfa w \in \mathcal{L}_p(\hat{T})
        $$
        Then, define $\phi:\Omega\to \mathbb{R}$ by 
        $$
        \phi|_{T}=\hat{\phi}_T\circ\mathscr{F}_{T}^{-1}. 
        $$
        To check that $\phi\in V_{\mathcal{T}}^p$ it remains to check that $\phi\in H^1(\Omega)$. For this, it suffices to show that, if $T_1,T_2\in\mathcal{T}$ share a face $F_{j,i_1}(T_1)=F_{j,i_2}(T_2)$, then
        $$
\hat{\phi}_{T_1}\circ\mathscr{F}_{T_1}^{-1}|_{F_{j,i_1}}=\hat{\phi}_{T_2}\circ\mathscr{F}_{T_2}^{-1}|_{F_{j,i_2}}.
        $$
        This follows from~\eqref{e:isomorphism} together with the fact that Lagrange nodes are unisolvent on faces of $\widehat{T}$ and affine isomorphism invariant. 
    %Let $T \in \mathcal{T}$ and let $v_1(T),\ldots,v_N(T)$ be its Lagrange nodes.     
    \end{proof}
    
    We are now in a position to define the Lagrange basis of polynomials on an affine-conforming mesh.
    \begin{definition}
    \label{d:lagrangeBasisMesh}
    For a simplicial affine-conforming mesh $\mathscr{T}$ with $L_p(\mathcal{T})=\{v_i\}_{i=1}^N$, we define the \emph{Lagrange basis of degree $p$ mapped polynomials for $\mathcal{T}$}, $\mathscr{L}_p(\mathcal{T}):=\{\phi_i\}_{i=1}^N$ by
              \begin{equation}
    \phi(v_i)=\delta_{ij}. 
    \end{equation}
    \end{definition}

    In the next definition, for each Lagrange node $v$, we fix a mapped simplex $\sigma_v$ of dimension either $d$ or $d-1$ containing $v$, with the only restriction that if $v$ is in the boundary of a mesh element, then $\sigma_v$ is a $(d-1)$-dimensional face of an element $T \in \mathcal{T}$, and if $v \in \partial \Omega$, then $\sigma_v \subset \partial \Omega$.
    \begin{definition}[Mapped Scott-Zhang quasi-interpolant]\label{d:mapSZ}
    Let $\mathcal{T}$ be an affine-conforming mesh. For each $v\in \mathcal{L}_p(\mathcal{T})$, if $v\in T^\circ$ for some $T\in \mathcal{T}$, let $\sigma_v:=T$. If $v\in\partial T$ for some $T\in\mathcal{T}$ and $v\notin \partial\Omega$, choose $i$ such that $v\in F_{1,i}(T)$ and set $\sigma_v:=F_{1,i}(T)$. Otherwise, there is $T\in\mathcal{T}$ and $i$ such that $v\in F_{1,i}(T)$ and $F_{1,i}(T)\subset \partial\Omega$ and we set $\sigma_v:=F_{1,i}(T)$. In all of the above cases, we set $T_v:=T$.

    We say $\Pi_{\mathcal{T}}^p$ is a \emph{mapped Scott-Zhang quasi-interpolant} if there is a choice of $\sigma_v$ as above such that
    $$
    \Pi_{\mathcal{T}}^pu:=\sum_{v\in\mathcal{L}_p(\mathcal{T})}\left(\int_{\sigma_v} \psi_{\mathscr{F}_{T_v}^{-1}(v),\mathscr{F}_{T_v}^{-1}(\sigma_v)}\circ\mathscr{F}_{T_v}^{-1}(x)\, u\,|\det D\mathscr{F}_{T_v}^{-1}(x)|\,d\textup{vol}_{\sigma_v}(x)\right) \phi_{v}
    $$
    \end{definition}

    \begin{lemma}
    \label{l:scottZhangApprox}
        Given $p\geq 1$, $1\leq \ell\leq m\leq p+1$, $\Upsilon>0$ there exists $C>0$ such that the following is true.                For all  $C^{p+1}$ affine-conforming meshes, $\mathcal{T}$, with constant $\Upsilon$, any mapped Scott-Zhang quasi-interpolant $\Pi_{\mathcal{T}}^p$, and all $T\in\mathcal{T}$,
        $$
        |(I-\Pi_{\mathcal{T}}^p)u|^2_{H_k^\ell(T)}\leq C
        (h_{T}k)^{2(m-\ell)}k^{-2m}\sum_{T'\in\mathcal{T},\,T'\cap T\neq \emptyset}\|u\|_{H^m(T')}^2
        $$
    \end{lemma}
    \begin{proof}
    Fix $T\in\mathcal{T}$, let $\{T_i\}_{i=1}^N\subset \mathcal{T}$ be those mesh elements that intersect $T$ nontrivially, and define $\widetilde{T}=\cup_iT_i$. 
    
    There is a polynomial $\tilde{p}_T\in\mathbb{P}^{p}(\widehat{T})$ (e.g. the averaged Taylor polynomial for $u\circ \mathscr{F}_T$~\cite[Lemma 4.3.8]{BrSc:00}) such that,
    \beq\label{e:average}
    \|\tilde{p}_T-u\circ\mathscr{F}_T\|_{H^\ell(\widehat{T})}\leq C_{\Upsilon,\ell,m} |u\circ\mathscr{F}_T|_{H^m(\widehat{T})}.
    \eeq
    In particular, then $p_T:=\tilde{p}_T\circ\mathscr{F}_T^{-1}$ satisfies
    \begin{align*}
    |u-p_T|_{H^\ell(T)}&=\sum_{|\alpha|=\ell}\|\partial^\alpha [(\tilde{p}_T-u\circ\mathscr{F}_T)\circ \mathscr{F}_T^{-1}|_{L^2(T)}\\
    &\leq C_\Upsilon h_T^{\frac{d}{2}-\ell}\sum_{|\beta|\leq \ell}\|\partial^\beta (\tilde{p}_T-u\circ\mathscr{F}_T)\|_{L^2(\widehat{T})}\\
    &\leq C_{\Upsilon,\ell,m} h_T^{\frac{d}{2}-\ell}| u\circ\mathscr{F}_T|_{H^m(\widehat{T})}
     \leq C_{\Upsilon,\ell,m} h_T^{m-\ell}\| u\|_{H^m(T)},
    \end{align*}
    where we have used \eqref{e:mapsHighReg} and \eqref{e:average}.

    Now, observe that extending $p_T$ from $T$ to an element of $V_{\mathcal{T}}^p$
    $$
    \|(I-\Pi_{\mathcal{T}}^p)u\|_{H^\ell(T)}\leq \|u-p_T\|_{H^{\ell}(T)}+\|\Pi_{\mathcal{T}}^p(u-p_T)\|_{H^m(T)}.
    $$
By Definition \ref{d:mapSZ},
    \begin{align*}
    &\|\Pi_{\mathcal{T}}^p(u-p_T)\|_{H^\ell(T)}  
\\
&\qquad\leq\sum_{v\in\mathcal{L}_p(\mathcal{T})\cap T} \Big|\int_{\sigma_v} \psi_{\mathscr{F}_{T_v}^{-1}(v),\mathscr{F}_{T_v}^{-1}(\sigma_v)}\circ\mathscr{F}_{T_v}^{-1}(x)\, (u-p_T)\,|\det D\mathscr{F}_{T_v}^{-1}(x)|\,d\textup{vol}_{\sigma_v}(x)\Big|\|\phi_{v}\|_{H^\ell(T)}.
    \end{align*}
    First, we estimate 
    $$
    \|\phi_v\|_{H^\ell(T)}\leq  h_T^{\frac{d}{2}-\ell}\| \phi_{v}\circ\mathscr{F}_T \|_{H^\ell(\hat{T})}\leq C_{p,\ell,\Upsilon}h_T^{\frac{d}{2}-\ell}.
    $$
    Next, when $\sigma_v=T$, we estimate 
    \begin{align*}
    &\Big|\int_{T} \psi_{\mathscr{F}_{T}^{-1}(v),\mathscr{F}_{T}^{-1}(\sigma_v)}\circ\mathscr{F}_{T}^{-1}(x)\, (u-p_T)\,|\det D\mathscr{F}_{T}^{-1}(x)|\,d\textup{vol}_{T}(x)\Big|\\
    &\leq \|\det D_{\mathscr{F}_T^{-1}}\|_{L^\infty}^{1/2}\|(u-p_T)\|_{L^2(T)}\|\psi_{\mathscr{F}_T^{-1}(v),\widehat{T}}\circ \mathscr{F}_{T^{-1}}|\det D\mathscr{F}_{T}^{-1}|^{1/2}\|_{L^2}\\
    &\leq C_{\Upsilon,p} h_T^{-\frac{d}{2}}\|u-p_T\|_{L^2(T)}.
    \end{align*}
    Finally, when $\sigma_v$ has dimension $d-1$, adapting the proof of~\cite[Lemma 3.1]{CaDe:15} (see also~\cite{Ve:16}), we obtain
\begin{align*}
   \Big|\int_{\sigma_v} \psi_{\mathscr{F}_{T_v}^{-1}(v),\mathscr{F}_{T_v}^{-1}(\sigma_v)}\circ\mathscr{F}_{T_v}^{-1}(x)\, (u-p_T)\,|\det D\mathscr{F}_{T_v}^{-1}(x)|\,d\textup{vol}_{\sigma_v}(x)\Big|
   \leq C_{\Upsilon, p }h_{T_i}^{-\frac{d}{2}+m}\Big(\sum_{T_i}\|u\|^2_{H^m(T_i)}\Big)^{\frac{1}{2}}.
\end{align*}
 All together, using the fact that, for $T_1\cap T_2\neq \emptyset$, $h_{T_1}\leq C_{\Upsilon}h_{T_2}$, we obtain 
 $$
 \|(I-\Pi_{\mathcal{T}}^p)u\|^2_{H^\ell(T)}\leq C_{\Upsilon,p}h_T^{2(m-\ell)}\sum_{T_i}\|u\|_{H^m(T_i)}^2,
 $$
 which implies the result after multiplication by $k^{2(m-\ell)}$.
    \end{proof}

\subsection{Spaces of mapped polynomials are well-behaved}

We now show that the spaces of mapped polynomials defined above are well-behaved in the sense of Definition \ref{d:wellbehaved}. We highlight that once one has the order-$p$ approximation property (of Definition \ref{d:app}) for mapped polynomials, the proofs of the inverse inequalities of Definition \ref{d:ii} and the super-approximation property of Definition \ref{d:sap} are appropriate curved-element modifications of standard proofs in the literature.

\begin{theorem}
\label{t:itIsWellBehaved}
Let $p\geq 1$, $\Upsilon>0$. Then there is $C_0>0$ such that for all $\mathcal{T}$, $C^{p+1}$ simplicial affine conforming meshes with constant $\Upsilon$ with $h(\mathcal{T})\leq \Upsilon k^{-1}$ and $U(\mathcal{T},k^{-1})\leq \Upsilon$, the space $V^p_{\mathcal{T}}$ is a well-behaved finite-element space of order $p$ at frequency $k$ with constants $(C_0,2)$ in the sense of Definition \ref{d:wellbehaved}.
\end{theorem}

\begin{proof}
    To see that $V^p_{\mathcal{T}}$ 
    satisfies the order-$p$ approximation property of Definition \ref{d:app}, 
    we apply Lemma~\ref{l:scottZhangApprox} and sum over $T\in\mathcal{T}$ to arrive at the approximation estimate in Definition~\ref{d:app}. Here, we use the fact that for any $T\in\mathcal{T}$, there is $C_{\Upsilon,d}$ such that 
    $$
    \sup_{T\in\mathcal{T}}\#\{T'\in\mathcal{T}\,:\, T'\cap T\}\leq C_{\Upsilon,d}.
    $$
    The support property of the approximant (with $\kappa=2$) in Definition~\ref{d:app} follows from the fact that the approximation operator $\Pi_{\mathcal{T}}^pu|_{T}=0$ if $u|_{T'}=0$ for $T\cap T'\neq \emptyset$, $T'\in\mathcal{T}$.  

Next, we check that $V_{\mathcal{T}}^p$ satisfies the inverse inequalities Definition \ref{d:ii} --  We first claim that for $v\in V^p_{\mathcal{T}}$ and $T\in\mathcal{T}$, $r\in\{0,1\}$, and $\alpha\in\mathbb{N}^d$, $|\alpha|\geq 1$,
\begin{equation} 
\label{e:highNormInverse}
\|\partial^{\alpha}v\|_{L^2(T)}\leq C_{|\alpha|,\Upsilon,r}h_T^{r-|\alpha|}\sum_{|\gamma|=r}\|\partial^\gamma v\|_{L^2(T)}.
\end{equation}
To prove this, observe that, 
by the definition of $V^p_{\mathcal{T}}$ \eqref{e:polySpace} there is $u\in\mathbb{P}^p(\hat{T})$ such that $v|_{T}=u\circ \mathscr{F}_{T}^{-1}.$ Let also 
$$
M_{\hat{T}}:\hat{T}\to \hat{T}_0:=\Big\{x=(x_1,\dots,x_d)\in\mathbb{R}^d\,:\, x_i>0,\, \sum_{i=1}^d x_i<1\Big\}
$$
be a bijective affine map and note that $\|DM_{\hat{T}}\|+\|D (M_{\hat{T}})^{-1}\|\leq C_{\Upsilon}$. 

We estimate
\begin{align*}
\|\partial^\alpha v\|_{L^2(T)}&=\|\partial^\alpha (u\circ\mathscr{F}_{T}^{-1})\|_{L^2(T)}\\
&\leq \sum_{1\leq |\beta|\leq |\alpha|}C_{\alpha,\Upsilon}h_T^{-|\alpha|} \|(\partial^{\beta}u )\circ \mathscr{F}_{T}^{-1}\|_{L^2(T)}\\
&=\sum_{1\leq |\beta|\leq |\alpha|}C_{\alpha,\Upsilon}h_T^{-|\alpha|} \|\det  D\mathscr{F}_T\|_{L^\infty(\hat{T})}^{1/2}\|\partial^{\beta}u\|_{L^2(\hat{T})}\\
&\leq \sum_{1\leq |\beta|\leq |\alpha|}C_{\alpha,\Upsilon}h_T^{\frac{d}{2}-|\alpha|} \|\partial^{\beta}(u\circ M_{\hat T}^{-1})\|_{L^2(\hat{T}_0)}\\
&\leq C_{\alpha,r,\Upsilon}h_T^{\frac{d}{2}-|\alpha|} \sum_{|\gamma|=r}\| \partial^{\gamma}(u \circ M_{\hat T}^{-1})\|_{L^2(\hat{T}_0)},
\end{align*}
where in the last inequality we use the inverse inequality~\cite[Lemma 4.5.3]{BrSc:08} (applied with $h=1$ $\blueOld{T}=\hat{T}_0$, and $\mathcal{P}$ the collection of degree $p$ polynomials on $\hat{T}_0$) together with the fact that $u\circ M_{\hat T}^{-1}$ is a polynomial of degree $p$.
Finally, we estimate
\begin{align*}
    \sum_{|\gamma|=r}h_T^{\frac{d}{2}-|\alpha|} \|\partial^{\gamma}(u \circ M_{\hat T}^{-1})\|_{L^2(\hat{T}_0)}&\leq C_{\Upsilon}h_T^{-|\alpha|}\sum_{|\gamma|=r}h_T^{|\gamma|} \|\partial^{\gamma}(u\circ \mathscr{F}_{T}^{-1}) \|_{L^2(T)}=C_{\Upsilon}h_T^{r-|\alpha|} \sum_{|\gamma|=r}\|\partial^{\gamma}v \|_{L^2(T)},
\end{align*}
which implies~\eqref{e:highNormInverse}.
Notice that~\eqref{e:highNormInverse} with $r=0$ directly implies the first estimate in Definition~\ref{d:ii}.
To conclude that $V^p_{\mathcal{T}}$ satisfies the inverse inequalities of Definition \ref{d:ii}, we argue by duality. Let $\iota:\widetilde{V}^p_{\mathcal{T}}(T)\to L^2(T)$ denote the inclusion operator, where 
$$
\widetilde{V}^p_{\mathcal{T}}(T):=\big\{ u\in L^2(T)\,:\, v\circ\mathscr{F}_{T}\in \mathbb{P}^{p}(\hat{T})\big\}.
$$ 
Our goal is to estimate
$$
\|\iota \|_{H_k^{-j}(T)\to L^2(T)}\leq C_{j,\Upsilon}(h_Tk)^{-j},
$$
which is equivalent to the estimate
\begin{equation}
\label{e:dualityInverse}
\|\Pi\|_{L^2(T)\to H_k^j(T)}\leq C_{j,\Upsilon}(h_T k)^{-j},
\end{equation}
where $\iota^*=\Pi:L^2(T)\to \tilde{V}_{\mathcal{T}}^p(T)$ is the orthogonal projector.
This follows from the $L^2$ boundedness of $\Pi$ and the estimates~\eqref{e:highNormInverse}.

% \sum_{|\alpha|=1}\| \partial^{\alpha}v\|_{L^\infty(T)}\leq C_{|\alpha|,\Upsilon} h_T^{-\frac{d}{2}}\sum_{|\gamma|=1}\|\partial^{\gamma}v\|_{L^2(T)},\qquad \| v\|_{L^\infty(T)}\leq C_{|\alpha|,\Upsilon} h_T^{-\frac{d}{2}}\|v\|_{L^2(T)}
% For later use, we also record that~\eqref{e:highNormInverse} implies 
% $$
% $$

We now show that $V^p_{\mathcal{T}}$ has the super-approximation property (Definition \ref{d:sap}); we follow the argument in \cite{DeGuSc:11}, adjusted appropriately here to apply to mapped polynomials. 
Observe that for $|\alpha|=p+1$ and $v\in \mathcal{V}_{\mathcal{T}}^p$, there is $u\in\mathbb{P}^p$ such that 
\begin{align}\nonumber
\|\partial^\alpha v\|_{L^2(T)}&=\|\partial^\alpha (u\circ A_T^{-1}\circ R_T^{-1} )\|_{L^2(T)}\\ \nonumber
&\leq C\sum_{1\leq |\beta|\leq p}\| (\partial^\beta (u\circ A_T^{-1})\circ R_T^{-1}\|_{L^2(T)}\\ \label{e:pplus1}
&\leq C\sum_{1\leq |\beta|\leq p}\| (\partial^\beta (v\circ R_T))\circ R_{T}^{-1}\|_{L^2(T)}\leq C\sum_{1\leq |\beta|\leq p}\| \partial^\beta v\|_{L^2(T)};
\end{align}
i.e., we can control $p+1$ derivatives of mapped polynomials by $p$ derivatives (note that when $R_T=I$, i.e., when the element is not curved, $\partial^\alpha v=0$).
We now apply the approximation property (Definition~\ref{d:app}). Indeed, using the operator $\mathcal{C}_{\mathcal{T}}^p$ from~\cite[Theorem A.11]{GaSp:25b} there is $v_h\in V_{\mathcal{T}}^p$ supported in $U_1$ such that 
\begin{align*}
    &(h_Tk)^{-p}k^{p+1}\|\chi^2 u_h-v_h\|_{H_k^1(\blueOld{T})}\leq C\|\chi^2 u_h\|_{H^{p+1}(T)}\\
    &\leq C\Big(\|\chi^2 u_h\|_{L^2(T)}+\sum_{|\alpha|=p+1}\|\partial^{\alpha}(\chi^2 u_h)\|_{L^2(T)}\Big)\\
    &\leq C\bigg(\|\chi^2 u_h\|_{L^2(T)}+\sum_{j=2}^{p+1}|\chi^2|_{W^{j,\infty}(T)}\| u_h\|_{H^{p+1-j}(T)}+\sum_{\substack{0\leq |\gamma|\leq 1\\|\beta|\leq p+1-|\gamma|}}\|(\partial^\gamma \chi^2) \partial^\beta u_h\|_{L^2(T)}\bigg)
\end{align*}
where we have used the chain rule in the last inequality.
Let $\overline{\chi}$ be the average of $\chi$ on $T$ and observe that $\|\chi -\overline{\chi}\|_{L^\infty(T)}\leq CC_\dagger h_T d^{-1}$ (by \eqref{e:chiBound} with $n=1$). Then, by using inverse inequalities, introducing $\overline{\chi}$, using the bound on $\chi-\overline{\chi}$, and the bound \eqref{e:chiBound} with $n=1$, we have 
\begin{align*}    
    &(h_Tk)^{-p}k^{p+1}\|\chi^2 u_h-v_h\|_{H_k^1(\blueOld{T})}
    \\
    &\leq C\bigg(\|\chi^2 u_h\|_{L^2(T)}+C_{\dagger}^2\sum_{j=2}^{p+1}d^{-j}h_T^{-p-1+j}\|u_h\|_{L^2(T)}\\
&\hspace{4cm}+\sum_{\substack{0\leq |\gamma|\leq 1\\|\beta|\leq p+1-|\gamma|}}C_\dagger d^{-|\gamma|}\Big(h_Td^{-1}C_{\dagger}\|  \partial^\beta u_h\|_{L^2(T)}+\overline{\chi}\|\partial^\beta u\|_{L^2(T)}\Big)\bigg).
\end{align*}
Then, by using \eqref{e:pplus1} and \eqref{e:highNormInverse}, reintroducing $\chi$ via $\overline{\chi}= -(\chi-\overline{\chi}) +\chi$, and using the bound on $\chi-\overline{\chi}$, we obtain that
\begin{align*}
&(h_Tk)^{-p}k^{p+1}\|\chi^2 u_h-v_h\|_{H_k^1(\blueOld{T})}\\
    &\leq C\bigg(\|\chi^2 u_h\|_{L^2(T)}+C_{\dagger}^2\sum_{j=2}^{p+1}d^{-j}h_T^{-p-1+j}\|u_h\|_{L^2(T)}\\
    &\hspace{4cm}+\sum_{|\beta|\leq p}C_\dagger(1+d^{-1}) \Big(h_Td^{-1}C_{\dagger}\|\partial^\beta u_h\|_{L^2(T)}+\overline{\chi}\|  \partial^\beta u_h\|_{L^2(T)}\Big)\bigg)\\
        &\leq C\bigg(\|\chi^2 u_h\|_{L^2(T)}+C_{\dagger}^2\sum_{j=2}^{p+1}d^{-j}h_T^{-p-1+j}\|u_h\|_{L^2}\\
        &\qquad+  C_\dagger h_T^{1-p}(1+d^{-1}) \Big(d^{-1}C_\dagger\|u_h\|_{L^2(T)}+\overline{\chi}| u_h|_{H^1(T)}\Big)
     +C_\dagger^2(1+d^{-1})(h_Td^{-1}+1)\|u_h\|_{L^2(T)}\bigg)\\
 &\leq C\bigg(\|\chi^2 u_h\|_{L^2(T)}+C_{\dagger}^2\sum_{j=2}^{p+1}d^{-j}h_T^{-p-1+j}\|u_h\|_{L^2(T)}\\
 &\hspace{3cm}+   C_\dagger h_T^{1-p}(1+d^{-1}) 
\big(| (\chi-\bar{\chi})u_h|_{H^1(T)}+|\chi u_h|_{H^1(T)}\big)\\
&\hspace{5cm}+C_\dagger^2(1+d^{-1})\big((h_T^{1-p}+h_T)d^{-1}+1\big)\|u_h\|_{L^2(T)}\bigg)\\
         &\leq C\bigg(\|\chi^2 u_h\|_{L^2(T)}+C_{\dagger}^2\sum_{j=2}^{p+1}d^{-j}h_T^{-p-1+j}\|u_h\|_{L^2(T)}\\
         &\,+   C_\dagger h_T^{1-p}(1+d^{-1}) \big(C_\dagger d^{-1}\|u_h\|_{L^2}+|\chi u_h|_{H^1(T)}\big)+C_\dagger^2(1+d^{-1})\big((h_T^{1-p}+h_T)d^{-1}+1\big)\|u_h\|_{L^2(T)}\bigg)\\
        &\leq CC_{\dagger}^2(1+d^{-1})(h_T^{1-p}d^{-1}+1)\|u_h\|_{L^2(T)}+CC_\dagger h_T^{1-p}(1+d^{-1})|\chi u_h|_{H^1(T)}.
\end{align*}
Thus,
\begin{align*}
\|\chi^2 u_h-v_h\|_{H_k^1(\blueOld{T})}
&\leq  k^{-1}
CC_{\dagger}^2(1+d^{-1})(h_Td^{-1}+h_T^p)\|u_h\|_{L^2(T)}+C_\dagger h_T(1+d^{-1})|\chi u_h|_{H^1(T)}. 
\end{align*}
The estimate in Definition~\ref{d:sap} now follows since the parameter $d$ is assumed to be $<2$. 
    \end{proof}

\footnotesize{

\bibliographystyle{alpha}
\bibliography{biblio_combined_sncwadditions}

%\newpage
}

\end{document}